\documentclass[11pt, book]{memoir}

\settrims{0pt}{0pt} 
\settypeblocksize{*}{34.5pc}{*} 
\setlrmargins{*}{*}{1} 
\setulmarginsandblock{1in}{1in}{*} 
\setheadfoot{\onelineskip}{2\onelineskip} 
\setheaderspaces{*}{1.5\onelineskip}{*} 
\checkandfixthelayout

\chapterstyle{bianchi}
\newcommand{\titlefont}{\normalfont\Huge\bfseries}

\usepackage{mathtools}
\usepackage{amsthm}
\usepackage{amssymb}
\usepackage{accents}
\usepackage{newpxtext}
\usepackage[varg,bigdelims]{newpxmath}
\usepackage{eucal}
\usepackage[usenames,dvipsnames]{xcolor}
\usepackage{tikz}
\usepackage[siunitx]{circuitikz}
\usepackage{graphicx}
\usepackage{outline}
\usepackage{varwidth}
\usepackage[inline]{enumitem}
\usepackage{ifthen}
\usepackage{footnote}
\usepackage[utf8]{inputenc} 
\usepackage[bookmarks=true, colorlinks=true, linkcolor=blue!50!black,
citecolor=orange!50!black, urlcolor=orange!50!black, pdfencoding=unicode]{hyperref}
\usepackage{subfiles}
\usepackage[capitalize]{cleveref}
\usepackage[backend=biber, backref=true, maxbibnames = 10, style = alphabetic]{biblatex}
\usepackage{makeidx}
\usepackage[all]{xy}
\usepackage[framemethod=tikz]{mdframed}
\usepackage{todonotes}
\usepackage{tablefootnote}

  \makeindex 

  \hypersetup{final}

  \setlist{nosep}
  
	\makesavenoteenv{tabular}

  \usetikzlibrary{ 
  	cd,
  	math,
  	decorations.markings,
		decorations.pathreplacing,
  	positioning,
  	arrows.meta,
  	shapes,
		shadows,
		shadings,
  	calc,
  	fit,
  	quotes,
  	intersections,
    circuits,
    circuits.ee.IEC
  }
  
	\tikzcdset{arrow style=tikz, diagrams={>=To}}

\makeatletter
\AfterEndEnvironment{mdframed}{%
  \tfn@tablefootnoteprintout%
  \gdef\tfn@fnt{0}%
}
\makeatother

\pgfdeclarelayer{edgelayer}
\pgfdeclarelayer{nodelayer}
\pgfsetlayers{edgelayer,nodelayer,main}

\tikzset{->-/.style={decoration={
  markings,
  mark=at position .5 with {\arrow{>}}},postaction={decorate}}}

\tikzstyle{none}=[inner sep=0pt]
\tikzstyle{contact}=[circle, fill=black, draw, inner sep=1.5pt, anchor=center]
\tikzstyle{bdot}=[circle, fill=black, draw, inner sep=1.5pt, anchor=center]
\tikzstyle{wdot}=[circle, fill=white, draw, inner sep=1.5pt]
\tikzstyle{bwdot}=[circle, fill=white, draw, inner sep=3pt]
\tikzstyle{bamp}=[rounded rectangle, rounded rectangle west arc=0pt, draw, inner
sep = 2pt,minimum width=4ex,minimum height=2ex,fill=black,text=white]
\tikzstyle{bampop}=[rounded rectangle, rounded rectangle east arc=0pt, draw, inner
sep = 2pt,minimum width=4ex,minimum height=2ex,fill=black,text=white]
\tikzstyle{wamp}=[rounded rectangle, rounded rectangle west arc=0pt, draw, inner
sep = 2pt,minimum width=4ex,minimum height=2ex,fill=white]
\tikzstyle{wampop}=[rounded rectangle, rounded rectangle east arc=0pt, draw, inner
sep = 2pt,minimum width=4ex,minimum height=2ex,fill=white]

  \tikzset{
     tick/.style={postaction={
        decorate,
        decoration={markings, mark=at position 0.5 with
        {\draw[-] (0,.4ex) -- (0,-.4ex);}}}
     }
  } 
  
  \tikzset{
     functor/.style={->, thick, shorten <=4pt, shorten >=4pt}
  }

  \tikzset{
     function/.style={->, thin, shorten <=4pt, shorten >=4pt}
  }

  \tikzset{
     mapsto/.style={->, densely dashed, blue, shorten <=\short, shorten
     >=\short, >=stealth, thick},
     short/.store in=\short,
     short=0pt
  }
  
\tikzset{
	spider diagram/.style={
		every to/.style={out=0, in=180, draw, thick},
		thick
	},
	dot size/.store in=\dotsize,
	dot size = 5pt,
	leg length/.store in=\leglen,
	leg length = 15pt,
	baby/.style={dot size = 2pt, leg length = 6pt},
	young/.style={dot size = 3pt, leg length = 10pt},
	special spider/.code n args={4}{
		\pgfkeysalso{circle, draw, thick, inner sep=0, minimum width=\dotsize,
  		prefix after command={\pgfextra{\let\fixname\tikzlastnode}},
  		append after command={\pgfextra{
  			\ifnum #1=0{} \else {\foreach \i in {1,...,#1} {
					\tikzmath{\anglei={-90*(#1+1-2*\i)/#1};}
  				\draw [thick]
						(\fixname) .. controls 
						($(\fixname.center)-(\anglei:#3/3)$) and ($(\fixname.center)-(\anglei:#3*2/3)$) .. 
						({$(\fixname)-(\anglei:#3*2/3)$}-|{$(\fixname)-(#3,0)$}) coordinate (\fixname_in\i);
  			}}\fi
  			\ifnum #2=0{} \else {\foreach \i in {1,...,#2} {
					\tikzmath{\anglei={90*(#2+1-2*\i)/#2};}
  				\draw [thick]
						(\fixname.center) .. controls 
						($(\fixname.center)+(\anglei:#4/3)$) and ($(\fixname.center)+(\anglei:#4*2/3)$) .. 
						({$(\fixname.center)+(\anglei:#4*2/3)$}-|{$(\fixname.center)+(#4,0)$}) coordinate (\fixname_out\i);
  			}}\fi
  		}}
		}
	},
	spider/.code 2 args={
		\pgfkeysalso{special spider={#1}{#2}{\leglen}{\leglen}}
	}
}

  \tikzset{
     oriented WD/.style={
        every to/.style={out=0,in=180,draw},
        label/.style={
           font=\everymath\expandafter{\the\everymath\scriptstyle},
           inner sep=0pt,
           node distance=2pt and -2pt},
        semithick,
        node distance=1 and 1,
        decoration={markings, mark=at position \stringdecpos with \stringdec},
        ar/.style={postaction={decorate}},
        execute at begin picture={\tikzset{
           x=\bbx, y=\bby,
           every fit/.style={inner xsep=\bbx, inner ysep=\bby}}}
        },
     string decoration/.store in=\stringdec,
     string decoration={\arrow{stealth};},
     string decoration pos/.store in=\stringdecpos,
     string decoration pos=.7,
     bbx/.store in=\bbx,
     bbx = 1.5cm,
     bby/.store in=\bby,
     bby = 1.5ex,
     bb port sep/.store in=\bbportsep,
     bb port sep=1.5,
     bb port length/.store in=\bbportlen,
     bb port length=4pt,
     bb penetrate/.store in=\bbpenetrate,
     bb penetrate=0,
     bb min width/.store in=\bbminwidth,
     bb min width=1cm,
     bb rounded corners/.store in=\bbcorners,
     bb rounded corners=2pt,
     bb small/.style={bb port sep=1, bb port length=2.5pt, bbx=.4cm, bb min width=.4cm, 
bby=.7ex},
		 bb medium/.style={bb port sep=1, bb port length=2.5pt, bbx=.4cm, bb min width=.4cm, 
bby=.9ex},
     bb/.code 2 args={
        \pgfmathsetlengthmacro{\bbheight}{\bbportsep * (max(#1,#2)+1) * \bby}
        \pgfkeysalso{draw,minimum height=\bbheight,minimum width=\bbminwidth,outer 
sep=0pt,
           rounded corners=\bbcorners,thick,
           prefix after command={\pgfextra{\let\fixname\tikzlastnode}},
           append after command={\pgfextra{\draw
              \ifnum #1=0{} \else foreach \i in {1,...,#1} {
                 ($(\fixname.north west)!{\i/(#1+1)}!(\fixname.south west)$) +(-
\bbportlen,0) 
  coordinate (\fixname_in\i) -- +(\bbpenetrate,0) coordinate (\fixname_in\i')}\fi 
              \ifnum #2=0{} \else foreach \i in {1,...,#2} {
                 ($(\fixname.north east)!{\i/(#2+1)}!(\fixname.south east)$) +(-
\bbpenetrate,0) 
  coordinate (\fixname_out\i') -- +(\bbportlen,0) coordinate (\fixname_out\i)}\fi;
           }}}
     },
     bb name/.style={append after command={\pgfextra{\node[anchor=north] at 
(\fixname.north) {#1};}}}
  }

  \tikzset{
  	unoriented WD/.style={
  		every to/.style={draw},
  		shorten <=-\penetration, shorten >=-\penetration,
  		label distance=-2pt,
  		thick,
  		node distance=\spacing,
  		execute at begin picture={\tikzset{
  			x=\spacing, y=\spacing}}
  		},
  	pack size/.store in=\psize,
  	pack size = 8pt,
  	spacing/.store in=\spacing,
  	spacing = 8pt,
  	link size/.store in=\lsize,
  	link size = 2pt,
		penetration/.store in=\penetration,
		penetration = 2pt,
  	pack color/.store in=\pcolor,
  	pack color = blue,
  	pack inside color/.store in=\picolor,
  	pack inside color=blue!20,
  	pack outside color/.store in=\pocolor,
  	pack outside color=blue!50!black,
  	surround sep/.store in=\ssep,
  	surround sep=8pt,
  	link/.style={
  		circle, 
  		draw=black, 
  		fill=black,
  		inner sep=0pt, 
  		minimum size=\lsize
  	},
  	pack/.style={
  		circle, 
  		draw = \pocolor, 
  		fill = \picolor,
  		inner sep = .25*\psize,
  		minimum size = \psize
  	},
  	outer pack/.style={
  		ellipse, 
  		draw,
  		inner sep=\ssep,
  		color=\pocolor,
  	},
  	intermediate pack/.style={
  		ellipse,
  		dashed, 
  		draw,
  		inner sep=\ssep,
  		color=\pocolor,
  	},
  }

	\tikzcdset{
		bijection/.style={every arrow/.append style={dash}}
	}

\newcommand{\scalar}[1]
{
\begin{aligned}
    \resizebox{#1}{!}{
\begin{tikzpicture}
	\begin{pgfonlayer}{nodelayer}
		\node [style=none] (0) at (-.75, 0) {};
		\node [style=wamp] (1) at (0, 0) {$a$};
		\node [style=none] (2) at (.75, 0) {};
	\end{pgfonlayer}
	\begin{pgfonlayer}{edgelayer}
		\draw[line width=1.25pt] (0.center) to (1.center);
		\draw[line width=1.25pt] (1.center) to (2.center);
	\end{pgfonlayer}
      \end{tikzpicture}}
\end{aligned}
}
	
\newcommand{\mult}[1]
{
\begin{aligned}
    \resizebox{#1}{!}{
\begin{tikzpicture}
	\begin{pgfonlayer}{nodelayer}
		\node [style=none] (0) at (1, -0) {};
		\node [style=circ] (1) at (0.125, -0) {};
		\node [style=none] (2) at (-1, 0.5) {};
		\node [style=none] (3) at (-1, -0.5) {};
	\end{pgfonlayer}
	\begin{pgfonlayer}{edgelayer}
		\draw[line width=2pt] (0.center) to (1.center);
		\draw[line width=2pt] [in=0, out=120, looseness=1.20] (1.center) to (2.center);
		\draw[line width=2pt] [in=0, out=-120, looseness=1.20] (1.center) to (3.center);
	\end{pgfonlayer}
      \end{tikzpicture}}
\end{aligned}
}

\newcommand{\add}[1]
{
\begin{aligned}
    \resizebox{#1}{!}{
\begin{tikzpicture}
	\begin{pgfonlayer}{nodelayer}
		\node [style=none] (0) at (1, -0) {};
		\node [style=circw] (1) at (0.125, -0) {};
		\node [style=none] (2) at (-1, 0.5) {};
		\node [style=none] (3) at (-1, -0.5) {};
	\end{pgfonlayer}
	\begin{pgfonlayer}{edgelayer}
		\draw[line width=2pt] (0.center) to (1.center);
		\draw[line width=2pt] [in=0, out=120, looseness=1.20] (1.center) to (2.center);
		\draw[line width=2pt] [in=0, out=-120, looseness=1.20] (1.center) to (3.center);
	\end{pgfonlayer}
      \end{tikzpicture}}
\end{aligned}
}

\newcommand{\unit}[1]
{
  \begin{aligned}
    \resizebox{#1}{!}{
\begin{tikzpicture}
	\begin{pgfonlayer}{nodelayer}
		\node [style=none] (spaceup) at (0, 0.5) {};
		\node [style=none] (spacedown) at (0, -0.5) {};
		\node [style=none] (0) at (1, -0) {};
		\node [style=none] (1) at (-1, -0) {};
		\node [style=circ] (2) at (0, -0) {};
	\end{pgfonlayer}
	\begin{pgfonlayer}{edgelayer}
		\draw[line width=2pt] (0.center) to (2);
	\end{pgfonlayer}
      \end{tikzpicture}}
  \end{aligned}
}

\newcommand{\zero}[1]
{
  \begin{aligned}
    \resizebox{#1}{!}{
\begin{tikzpicture}
	\begin{pgfonlayer}{nodelayer}
		\node [style=none] (spaceup) at (0, 0.5) {};
		\node [style=none] (spacedown) at (0, -0.5) {};
		\node [style=none] (0) at (1, -0) {};
		\node [style=none] (1) at (-1, -0) {};
		\node [style=circw] (2) at (0, -0) {};
	\end{pgfonlayer}
	\begin{pgfonlayer}{edgelayer}
		\draw[line width=2pt] (0.center) to (2);
	\end{pgfonlayer}
      \end{tikzpicture}}
  \end{aligned}
}

\newcommand{\comult}[1]
{
\begin{aligned}
    \resizebox{#1}{!}{
\begin{tikzpicture}
	\begin{pgfonlayer}{nodelayer}
		\node [style=none] (0) at (-1, -0) {};
		\node [style=circ] (1) at (-0.125, -0) {};
		\node [style=none] (2) at (1, 0.5) {};
		\node [style=none] (3) at (1, -0.5) {};
	\end{pgfonlayer}
	\begin{pgfonlayer}{edgelayer}
		\draw[line width=2pt] (0.center) to (1.center);
		\draw[line width=2pt] [in=180, out=60, looseness=1.20] (1.center) to (2.center);
		\draw[line width=2pt] [in=180, out=-60, looseness=1.20] (1.center) to (3.center);
	\end{pgfonlayer}
      \end{tikzpicture}}
\end{aligned}
}

\newcommand{\counit}[1]
{
  \begin{aligned}
    \resizebox{#1}{!}{
\begin{tikzpicture}
	\begin{pgfonlayer}{nodelayer}
		\node [style=none] (spaceup) at (0, 0.5) {};
		\node [style=none] (spacedown) at (0, -0.5) {};
		\node [style=none] (0) at (-1, -0) {};
		\node [style=none] (1) at (1, -0) {};
		\node [style=circ] (2) at (0, -0) {};
	\end{pgfonlayer}
	\begin{pgfonlayer}{edgelayer}
		\draw[line width=2pt] (0.center) to (2);
	\end{pgfonlayer}
      \end{tikzpicture}}
  \end{aligned}
}

\newcommand{\coadd}[1]
{
\begin{aligned}
    \resizebox{#1}{!}{
\begin{tikzpicture}
	\begin{pgfonlayer}{nodelayer}
		\node [style=none] (0) at (-1, -0) {};
		\node [style=bwdot] (1) at (-0.125, -0) {};
		\node [style=none] (2) at (1, 0.5) {};
		\node [style=none] (3) at (1, -0.5) {};
	\end{pgfonlayer}
	\begin{pgfonlayer}{edgelayer}
		\draw[line width=2pt] (0.center) to (1.center);
		\draw[line width=2pt] [in=180, out=60, looseness=1.20] (1.center) to (2.center);
		\draw[line width=2pt] [in=180, out=-60, looseness=1.20] (1.center) to (3.center);
	\end{pgfonlayer}
      \end{tikzpicture}}
\end{aligned}
}

\newcommand{\cozero}[1]
{
  \begin{aligned}
    \resizebox{#1}{!}{
\begin{tikzpicture}
	\begin{pgfonlayer}{nodelayer}
		\node [style=none] (spaceup) at (0, 0.5) {};
		\node [style=none] (spacedown) at (0, -0.5) {};
		\node [style=none] (0) at (-1, -0) {};
		\node [style=none] (1) at (1, -0) {};
		\node [style=bwdot] (2) at (0, -0) {};
	\end{pgfonlayer}
	\begin{pgfonlayer}{edgelayer}
		\draw[line width=2pt] (0.center) to (2);
	\end{pgfonlayer}
      \end{tikzpicture}}
  \end{aligned}
}

\newcommand{\idone}[1]
{
  \begin{aligned}
    \resizebox{#1}{!}{
     \begin{tikzpicture}
	\begin{pgfonlayer}{nodelayer}
		\node [style=none] (0) at (-1, -0) {};
		\node [style=none] (1) at (1, -0) {};
		\node [style=none] (2) at (0, 0.5) {};
		\node [style=none] (3) at (0, -0.5) {};
	\end{pgfonlayer}
	\begin{pgfonlayer}{edgelayer}
		\draw[line width=2pt] (1.center) to (0.center);
	\end{pgfonlayer}
\end{tikzpicture} 
    }
  \end{aligned}
}

\newcommand{\swap}[1]
{
  \begin{aligned}
    \resizebox{#1}{!}{
\begin{tikzpicture}
	\begin{pgfonlayer}{nodelayer}
		\node [style=none] (2) at (-0.5, -0.5) {};
		\node [style=none] (3) at (-2, 0.5) {};
		\node [style=none] (4) at (-0.5, 0.5) {};
		\node [style=none] (5) at (-2, -0.5) {};
	\end{pgfonlayer}
	\begin{pgfonlayer}{edgelayer}
		\draw[line width=2pt] [in=180, out=0, looseness=1.00] (3.center) to (2.center);
		\draw[line width=2pt] [in=0, out=180, looseness=1.00] (4.center) to (5.center);
	\end{pgfonlayer}
\end{tikzpicture}
    }
  \end{aligned}
}
\newcommand{\assocl}[1]
{
  \begin{aligned}
    \resizebox{#1}{!}{
\begin{tikzpicture}
	\begin{pgfonlayer}{nodelayer}
		\node [style=circ] (0) at (0.125, -0) {};
		\node [style=none] (1) at (-1, 0.5) {};
		\node [style=none] (2) at (-1, -0.5) {};
		\node [style=none] (3) at (0, -1) {};
		\node [style=none] (4) at (2.25, -0.5) {};
		\node [style=none] (5) at (0.25, -0) {};
		\node [style=circ] (6) at (1.25, -0.5) {};
		\node [style=none] (7) at (-1, -1) {};
	\end{pgfonlayer}
	\begin{pgfonlayer}{edgelayer}
		\draw[line width=2pt] [in=0, out=120, looseness=1.20] (0.center) to (1.center);
		\draw[line width=2pt] [in=0, out=-120, looseness=1.20] (0.center) to (2.center);
		\draw[line width=2pt] (4.center) to (6);
		\draw[line width=2pt] [in=0, out=120, looseness=1.20] (6) to (5.center);
		\draw[line width=2pt] [in=0, out=-120, looseness=1.20] (6) to (3.center);
		\draw[line width=2pt] (3.center) to (7.center);
	\end{pgfonlayer}
      \end{tikzpicture}}
  \end{aligned}
}

\newcommand{\assocr}[1]
{
  \begin{aligned}
    \resizebox{#1}{!}{
\begin{tikzpicture}
	\begin{pgfonlayer}{nodelayer}
		\node [style=circ] (0) at (0.125, -0.5) {};
		\node [style=none] (1) at (-1, -1) {};
		\node [style=none] (2) at (-1, 0) {};
		\node [style=none] (3) at (0, 0.5) {};
		\node [style=none] (4) at (2.25, 0) {};
		\node [style=none] (5) at (0.25, -0.5) {};
		\node [style=circ] (6) at (1.25, 0) {};
		\node [style=none] (7) at (-1, 0.5) {};
	\end{pgfonlayer}
	\begin{pgfonlayer}{edgelayer}
		\draw[line width=2pt] [in=0, out=-120, looseness=1.20] (0.center) to (1.center);
		\draw[line width=2pt] [in=0, out=120, looseness=1.20] (0.center) to (2.center);
		\draw[line width=2pt] (4.center) to (6);
		\draw[line width=2pt] [in=0, out=-120, looseness=1.20] (6) to (5.center);
		\draw[line width=2pt] [in=0, out=120, looseness=1.20] (6) to (3.center);
		\draw[line width=2pt] (3.center) to (7.center);
	\end{pgfonlayer}
      \end{tikzpicture}}
  \end{aligned}
}

\newcommand{\unitl}[1]
{
  \begin{aligned}
    \resizebox{#1}{!}{
\begin{tikzpicture}
	\begin{pgfonlayer}{nodelayer}
		\node [style=none] (0) at (1, -0) {};
		\node [style=circ] (1) at (0.125, -0) {};
		\node [style=circ] (2) at (-1, 0.5) {};
		\node [style=none] (3) at (-1, -0.5) {};
		\node [style=none] (4) at (-2, -0.5) {};
	\end{pgfonlayer}
	\begin{pgfonlayer}{edgelayer}
		\draw[line width=2pt] (0.center) to (1.center);
		\draw[line width=2pt] [in=0, out=120, looseness=1.20] (1.center) to (2.center);
		\draw[line width=2pt] [in=0, out=-120, looseness=1.20] (1.center) to (3.center);
		\draw[line width=2pt] (4.center) to (3.center);
	\end{pgfonlayer}
\end{tikzpicture}
    }
  \end{aligned}
}

\newcommand{\commute}[1]
{
  \begin{aligned}
    \resizebox{#1}{!}{
\begin{tikzpicture}
	\begin{pgfonlayer}{nodelayer}
		\node [style=none] (0) at (1.25, -0) {};
		\node [style=circ] (1) at (0.375, -0) {};
		\node [style=none] (2) at (-0.5, -0.5) {};
		\node [style=none] (3) at (-2, 0.5) {};
		\node [style=none] (4) at (-0.5, 0.5) {};
		\node [style=none] (5) at (-2, -0.5) {};
	\end{pgfonlayer}
	\begin{pgfonlayer}{edgelayer}
		\draw[line width=2pt] (0.center) to (1.center);
		\draw[line width=2pt] [in=0, out=-120, looseness=1.20] (1.center) to (2.center);
		\draw[line width=2pt] [in=180, out=0, looseness=1.00] (3.center) to (2.center);
		\draw[line width=2pt] [in=0, out=120, looseness=1.20] (1.center) to (4.center);
		\draw[line width=2pt] [in=0, out=180, looseness=1.00] (4.center) to (5.center);
	\end{pgfonlayer}
\end{tikzpicture}
    }
  \end{aligned}
}

\newcommand{\frobs}[1]
{
  \begin{aligned}
    \resizebox{#1}{!}{
\begin{tikzpicture}
	\begin{pgfonlayer}{nodelayer}
		\node [style=none] (0) at (-1.5, 0.5) {};
		\node [style=circ] (1) at (-0.75, 0.5) {};
		\node [style=none] (2) at (0.25, -0) {};
		\node [style=none] (3) at (0.25, 1) {};
		\node [style=circ] (4) at (1, -0.5) {};
		\node [style=none] (5) at (0, -0) {};
		\node [style=none] (6) at (1.75, -0.5) {};
		\node [style=none] (7) at (0, -1) {};
		\node [style=none] (8) at (1.75, 1) {};
		\node [style=none] (9) at (-1.5, -1) {};
	\end{pgfonlayer}
	\begin{pgfonlayer}{edgelayer}
		\draw[line width=2pt] [in=180, out=-60, looseness=1.20] (1) to (2.center);
		\draw[line width=2pt] [in=180, out=60, looseness=1.20] (1) to (3.center);
		\draw[line width=2pt] (0.center) to (1);
		\draw[line width=2pt] (6.center) to (4);
		\draw[line width=2pt] [in=0, out=120, looseness=1.20] (4) to (5.center);
		\draw[line width=2pt] [in=0, out=-120, looseness=1.20] (4) to (7.center);
		\draw[line width=2pt] (3.center) to (8.center);
		\draw[line width=2pt] (7.center) to (9.center);
	\end{pgfonlayer}
\end{tikzpicture}
    }
  \end{aligned}
}

\newcommand{\frobx}[1]
{
  \begin{aligned}
    \resizebox{#1}{!}{
\begin{tikzpicture}
	\begin{pgfonlayer}{nodelayer}
		\node [style=circ] (0) at (-0.5, -0) {};
		\node [style=none] (1) at (-1.5, -0.5) {};
		\node [style=none] (2) at (-1.5, 0.5) {};
		\node [style=circ] (3) at (0.5, -0) {};
		\node [style=none] (4) at (1.5, -0.5) {};
		\node [style=none] (5) at (1.5, 0.5) {};
	\end{pgfonlayer}
	\begin{pgfonlayer}{edgelayer}
		\draw[line width=2pt] [in=0, out=-120, looseness=1.20] (0.center) to (1.center);
		\draw[line width=2pt] [in=0, out=120, looseness=1.20] (0.center) to (2.center);
		\draw[line width=2pt] [in=180, out=-60, looseness=1.20] (3) to (4.center);
		\draw[line width=2pt] [in=180, out=60, looseness=1.20] (3) to (5.center);
		\draw[line width=2pt] (0) to (3);
	\end{pgfonlayer}
\end{tikzpicture}
    }
  \end{aligned}
}

\newcommand{\frobz}[1]
{
  \begin{aligned}
    \resizebox{#1}{!}{
\begin{tikzpicture}
	\begin{pgfonlayer}{nodelayer}
		\node [style=none] (0) at (1.75, 0.5) {};
		\node [style=circ] (1) at (1, 0.5) {};
		\node [style=none] (2) at (0, -0) {};
		\node [style=none] (3) at (0, 1) {};
		\node [style=circ] (4) at (-0.75, -0.5) {};
		\node [style=none] (5) at (0.25, -0) {};
		\node [style=none] (6) at (-1.5, -0.5) {};
		\node [style=none] (7) at (0.25, -1) {};
		\node [style=none] (8) at (-1.5, 1) {};
		\node [style=none] (9) at (1.75, -1) {};
	\end{pgfonlayer}
	\begin{pgfonlayer}{edgelayer}
		\draw[line width=2pt] [in=0, out=-120, looseness=1.20] (1) to (2.center);
		\draw[line width=2pt] [in=0, out=120, looseness=1.20] (1) to (3.center);
		\draw[line width=2pt] (0.center) to (1);
		\draw[line width=2pt] (6.center) to (4);
		\draw[line width=2pt] [in=180, out=60, looseness=1.20] (4) to (5.center);
		\draw[line width=2pt] [in=180, out=-60, looseness=1.20] (4) to (7.center);
		\draw[line width=2pt] (3.center) to (8.center);
		\draw[line width=2pt] (7.center) to (9.center);
	\end{pgfonlayer}
\end{tikzpicture}
    }
  \end{aligned}
}

\newcommand{\spec}[1]
{
  \begin{aligned}
    \resizebox{#1}{!}{
\begin{tikzpicture}
	\begin{pgfonlayer}{nodelayer}
		\node [style=none] (0) at (1.75, -0) {};
		\node [style=circ] (1) at (0.75, -0) {};
		\node [style=none] (2) at (0, -0.5) {};
		\node [style=none] (3) at (0, 0.5) {};
		\node [style=circ] (4) at (-0.75, -0) {};
		\node [style=none] (5) at (0, -0.5) {};
		\node [style=none] (6) at (-1.75, -0) {};
		\node [style=none] (7) at (0, 0.5) {};
	\end{pgfonlayer}
	\begin{pgfonlayer}{edgelayer}
		\draw[line width=2pt] (0.center) to (1.center);
		\draw[line width=2pt] [in=0, out=-120, looseness=1.20] (1.center) to (2.center);
		\draw[line width=2pt] [in=0, out=120, looseness=1.20] (1.center) to (3.center);
		\draw[line width=2pt] (6.center) to (4);
		\draw[line width=2pt] [in=180, out=-60, looseness=1.20] (4) to (5.center);
		\draw[line width=2pt] [in=180, out=60, looseness=1.20] (4) to (7.center);
	\end{pgfonlayer}
\end{tikzpicture}
    }
  \end{aligned}
}

\newcommand{\boxofbullets}[4][.5]{
	\foreach \i [evaluate=\i as \x using {#1*(\i+(1-#2)/2)}] in {1,...,#2}{
		\node (pt_#4_\i) at ($(\x,0)+#3$) {$\bullet$};
	}
	\node[draw, rounded corners, inner ysep=0pt, fit={(pt_#4_1) (pt_#4_#2)}] (box_#4) {};
}

\colorlet{theoremcolor}{white!92!blue}
\colorlet{definitioncolor}{white!92!purple}
\colorlet{examplecolor}{white!93!green}

\mdfdefinestyle{theoremframe}{
    linewidth=0pt,
    backgroundcolor=theoremcolor,
    roundcorner=6pt,
    nobreak=true,
    leftmargin=0,
    innerleftmargin=0,
    rightmargin=0,
    innerrightmargin=0,
    }

\mdfdefinestyle{definitionframe}{
    linewidth=0pt,
    backgroundcolor=definitioncolor,
    roundcorner=6pt,
    leftmargin=0,
    innerleftmargin=0,
    rightmargin=0,
    innerrightmargin=0,
    }

\mdfdefinestyle{exampleframe}{
    linewidth=0pt,
    backgroundcolor=examplecolor,
    leftmargin=0,
    innerleftmargin=0,
    rightmargin=0,
    innerrightmargin=0,
    }

\newtheoremstyle{plain}
  {-\topsep}		
  {}			
  {\normalfont}		
  {}			
  {\bfseries}		
  {.}			
  {.5em}		
  {}			

  \theoremstyle{plain}
  \newmdtheoremenv[style=theoremframe]{theorem}[equation]{Theorem}
  \newmdtheoremenv[style=theoremframe]{proposition}[equation]{Proposition}
  \newmdtheoremenv[style=theoremframe]{corollary}[equation]{Corollary}
  \newmdtheoremenv[style=theoremframe]{lemma}[equation]{Lemma}
  
  \theoremstyle{plain}
  \newmdtheoremenv[style=definitionframe]{definition}[equation]{Definition}
  \newmdtheoremenv[style=definitionframe]{roughDef}[equation]{Rough Definition}
  \crefname{roughDef}{Definition}{Definitions}
  \newtheorem{construction}[equation]{Construction}

  \newtheorem*{axiom*}{Axiom}
  
  \theoremstyle{remark}
  \newtheorem{remark}[equation]{Remark}
  
\newcommand{\finishSolutionChapter}{
\clearpage
}

\makeatletter
\newcommand{\nolisttopbreak}{\nobreak\@afterheading}
\makeatother

\newcounter{solcounterlocal}[section]
\newcounter{solcounterglobal}

\newcommand{\sol}[4][noprint]{

\stepcounter{solcounterlocal}\stepcounter{solcounterglobal}

\noindent\ignorespacesafterend\emph{Solution to} \cref{#2}.%
\nopagebreak%
\ifthenelse{\equal{#1}{print}}{
\nopagebreak%
\begin{mdframed}[backgroundcolor=examplecolor,linewidth=0pt]%
#3%
\end{mdframed}%
\nopagebreak
}{}%
\nolisttopbreak
\begin{description}[leftmargin=2.5ex,itemindent=0pt,topsep=0ex,nosep]
\item\nopagebreak
#4
\end{description}
\bigskip
}

\newenvironment{altikz}{
\begin{aligned}
\begin{tikzpicture}
}
{
\end{tikzpicture}
\end{aligned}
}
		
\newmdtheoremenv[style=exampleframe]{example}[equation]{Example}

\newtheorem{exc-inner}[equation]{Exercise}
\newenvironment{exercise}[1][]{
  \pushQED{\qed}
  \begin{exc-inner}[#1]~
}{
  \popQED
  \end{exc-inner}
}
  \crefname{exercise}{Exercise}{Exercises}

\newcommand{\adj}[5][30pt]{
\begin{tikzcd}[ampersand replacement=\&, column sep=#1]
  #2\ar[r, bend left=15, shift left=2pt, "#3"]
  \ar[r, Rightarrow, shorten <=8pt, shorten >=8pt]\&
  #5\ar[l, bend left=15, shift left=2pt, "#4"]
\end{tikzcd}
}

\DeclareSymbolFont{stmry}{U}{stmry}{m}{n}
\DeclareMathSymbol\fatsemi\mathop{stmry}{"23}

\DeclareFontFamily{U}{mathx}{\hyphenchar\font45}
\DeclareFontShape{U}{mathx}{m}{n}{
      <5> <6> <7> <8> <9> <10>
      <10.95> <12> <14.4> <17.28> <20.74> <24.88>
      mathx10
      }{}
\DeclareSymbolFont{mathx}{U}{mathx}{m}{n}
\DeclareFontSubstitution{U}{mathx}{m}{n}
\DeclareMathAccent{\widecheck}{0}{mathx}{"71}

\renewcommand{\ss}{\subseteq}

\DeclarePairedDelimiter{\pair}{\langle}{\rangle}
\DeclarePairedDelimiter{\copair}{[}{]}
\DeclarePairedDelimiter{\floor}{\lfloor}{\rfloor}
\DeclarePairedDelimiter{\ceil}{\lceil}{\rceil}

\DeclarePairedDelimiter{\corners}{\ulcorner}{\urcorner}

\DeclareMathOperator{\Hom}{Hom}
\DeclareMathOperator{\Mor}{Mor}
\DeclareMathOperator{\dom}{dom}
\DeclareMathOperator{\cod}{cod}
\DeclareMathOperator*{\colim}{colim}
\DeclareMathOperator{\im}{im}
\DeclareMathOperator{\Ob}{Ob}

\DeclareMathOperator{\dju}{\sqcup}

\newcommand{\const}[1]{\mathtt{#1}}
\newcommand{\Set}[1]{\mathrm{#1}}
\newcommand{\cat}[1]{\mathcal{#1}}
\newcommand{\Cat}[1]{\mathbf{#1}}
\newcommand{\fun}[1]{\textit{#1}}
\newcommand{\Fun}[1]{\mathsf{#1}}

\newcommand{\id}{\mathrm{id}}
\newcommand{\cocolon}{:\!}

\newcommand{\too}{\longrightarrow}
\newcommand{\tto}{\rightrightarrows}
\newcommand{\To}[1]{\xrightarrow{#1}}
\newcommand{\Tto}[3][13pt]{\begin{tikzcd}[sep=#1, cramped, ampersand replacement=\&, text height=1ex, text depth=.3ex]\ar[r, shift left=2pt, "#2"]\ar[r, shift right=2pt, "#3"']\&{}\end{tikzcd}}
\newcommand{\Too}[1]{\xrightarrow{\;\;#1\;\;}}
\newcommand{\from}{\leftarrow}
\newcommand{\From}[1]{\xleftarrow{#1}}
\newcommand{\Fromm}[1]{\xleftarrow{\;\;#1\;\;}}
\newcommand{\surj}{\twoheadrightarrow}
\newcommand{\inj}{\rightarrowtail}

\newcommand{\tickar}{\begin{tikzcd}[baseline=-0.5ex,cramped,sep=small,ampersand 
replacement=\&]{}\ar[r,tick]\&{}\end{tikzcd}}
\newcommand{\imp}{\Rightarrow}

\renewcommand{\th}{\ensuremath{^\tn{th}}\ }

\newcommand{\op}{^\tn{op}}

\newcommand{\tn}[1]{\textnormal{#1}}
\newcommand{\ol}[1]{\overline{#1}}
\newcommand{\ul}[1]{\underline{#1}}

\newcommand{\LMO}[2][over]{\ifthenelse{\equal{#1}{over}}{\overset{#2}{\bullet}}{\underset{#2}{\bullet}}}
\newcommand{\LTO}[2][\bullet]{\overset{\tn{#2}}{#1}}

\newcommand{\NN}{\mathbb{N}}
\newcommand{\bb}{\mathbb{B}}
\newcommand{\BB}{\mathbb{B}}
\newcommand{\nn}{\NN}

\newcommand{\ZZ}{\mathbb{Z}}
\newcommand{\zz}{\mathbb{Z}}
\newcommand{\RR}{\mathbb{R}}
\newcommand{\rr}{\mathbb{R}}
\newcommand{\IR}{\mathbb{I}\hspace{.6pt}\mathbb{R}}

\newcommand{\singleton}{\{1\}}
\newcommand{\powset}{\Fun{P}}
\newcommand{\upset}{\Fun{U}}
\newcommand{\beh}{\Fun{B}}
\newcommand{\prt}[1]{\Fun{Prt}(#1)}

\newcommand{\upclose}{\mathop{\uparrow}}

\newcommand{\foo}{\const{foo}}

\newcommand{\inv}{^{-1}}

\newcommand{\smset}{\Cat{Set}}
\newcommand{\smcat}{\Cat{Cat}}

\newcommand{\Op}{\Cat{Op}}
\newcommand{\Shv}{\Cat{Shv}}
\newcommand{\true}{\const{true}}
\newcommand{\false}{\const{false}}
\newcommand{\Bool}{\Cat{Bool}}
\newcommand{\Cost}{\Cat{Cost}}

\newcommand{\inst}{\tn{-}\Cat{Inst}}
\newcommand{\mat}{\Cat{Mat}}
\newcommand{\corel}[1]{\Cat{Corel}_{#1}}
\newcommand{\rel}{\Cat{Rel}}
\newcommand{\cospan}[1]{\Cat{Cospan}_{#1}}
\newcommand{\finset}{\Cat{FinSet}}

  \newcommand{\Prof}{\Cat{Prof}}
  \newcommand{\Feas}{\Cat{Feas}}
  \newcommand{\Unit}[1]{\mathrm{U}_{#1}}
  \newcommand{\comp}[1]{\widehat{#1}}
  \newcommand{\conj}[1]{\widecheck{#1}}
  
\newcommand{\cp}{\mathbin{\fatsemi}}

\newcommand{\pgin}{\fun{in}}
\newcommand{\pgout}{\fun{out}}
\newcommand{\ord}[1]{\underline{{#1}}}
\newcommand{\free}{\Cat{Free}}
\newcommand{\expr}{\mathrm{Expr}}
\newcommand{\sfg}{\mathbf{SFG}}
\newcommand\addgen{\lower8pt\hbox{$\includegraphics[height=0.7cm]{pics/add.pdf}$}}
\newcommand\zerogen{\lower5pt\hbox{$\includegraphics[height=0.5cm]{pics/zero.pdf}$}}
\newcommand\delaygen{\lower6pt\hbox{$\includegraphics[height=0.6cm]{pics/delay.pdf}$}}
\newcommand\scalargen{\lower6pt\hbox{$\includegraphics[height=0.6cm]{pics/scalar.pdf}$}}
\newcommand\copyopgen{\lower8pt\hbox{$\includegraphics[height=0.7cm]{pics/copyop.pdf}$}}
\newcommand\discardopgen{\lower5pt\hbox{$\includegraphics[height=0.5cm]{pics/discardop.pdf}$}}
\newcommand\twist{\lower6pt\hbox{$\includegraphics[height=0.6cm]{pics/twist.pdf}$}}

\tikzstyle{none}=[inner sep=0pt]
\tikzstyle{circ}=[circle,fill=black,draw,inner sep=3pt]
\tikzstyle{circw}=[circle,fill=white,draw,inner sep=3pt,thick]

\newcommand{\oprdset}{\mathbf{Set}}
\newcommand{\oprdcospan}{\mathbf{Cospan}}
\newcommand{\light}{\texttt{light}}
\newcommand{\switch}{\texttt{switch}}
\newcommand{\battery}{\texttt{battery}}
\newcommand{\elec}{\Fun{Circ}}

\newcommand{\restrict}[2]{#1\big|_{#2}}
\newcommand{\Prop}{\const{Prop}}

\newcommand{\boxCD}[2][black]{\fcolorbox{#1}{white}{\begin{varwidth}{\textwidth}\centering #2\end{varwidth}}}

\newcommand{\?}{{\color{gray}{?}}}
\newcommand{\DNE}{{\color{gray}{\boxtimes}}}

\newcommand{\erase}[2][]{{\color{red}#1}}
\newcommand{\showhide}[2]{#1}

\setsecnumdepth{subsubsection}
\settocdepth{subsection}
\DeclareMathVersion{normal2}

\begin{document}

\frontmatter

\title{\titlefont Seven Sketches in Compositionality:\\\LARGE Real-world models of category theory}
\title{\titlefont Seven Sketches in Compositionality:\\\LARGE Category theory in the real world}
\title{\titlefont Seven Sketches in Compositionality:\\\LARGE Category theoretic foundations of real-world phenomena}
\title{\titlefont Seven Sketches in Compositionality:\\\LARGE Toward a category-theoretic foundation for science and engineering}
\title{\titlefont Seven Sketches in Compositionality:\\\LARGE Real-World Applications of Category Theory}
\title{\titlefont Seven Sketches in Compositionality:\\\medskip\huge An Invitation to Categorical Modeling}
\title{\titlefont Seven Sketches in Compositionality:\\\medskip\huge An Invitation to Applied Category Theory}

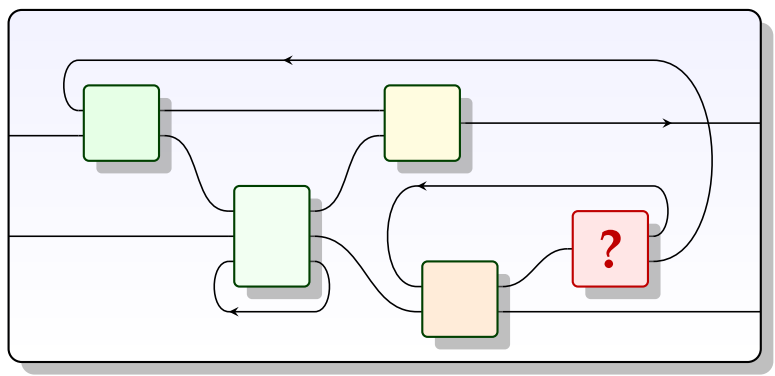
\author{\LARGE Brendan Fong \and \LARGE David I. Spivak}

\posttitle{
  \vspace{.8in}
  \[
  \begin{tikzpicture}[oriented WD, bb min width =1cm, bbx=1cm, bb port sep =1, bb port length=2pt, bby=1ex]
  	\node[coordinate] at (0,0) (ul) {};
  	\node[coordinate] at (8,-12) (lr) {};
    \node[bb={0}{0}, rounded corners=5pt, drop shadow, top color=blue!5, fit = (ul) (lr)] (Z) {};
  	\node[bb={2}{2}, green!25!black, drop shadow, fill=green!10, below right=2 and 0 of ul] (X11) {};
  	\node[bb={3}{3}, green!25!black, drop shadow, fill=green!5, below right=of X11] (X12) {};
  	\node[bb={2}{1}, green!25!black, drop shadow, fill=yellow!15, above right=of X12] (X13) {};
  	\node[bb={2}{2}, green!25!black, drop shadow, fill=orange!15, below right = -1 and 1.5 of X12] (X21) {};
  	\node[bb={1}{2}, red!75!black, drop shadow, fill=red!10, above right=-1 and 1 of X21] (X22) {?};
  	\draw (X21_out1) to (X22_in1);
  	\draw[ar] let \p1=(X22.north east), \p2=(X21.north west), \n1={\y1+\bby}, \n2=\bbportlen in
            (X22_out1) to[in=0] (\x1+\n2,\n1) -- (\x2-\n2,\n1) to[out=180] (X21_in1);
  	\draw (X11_out1) to (X13_in1);
  	\draw (X11_out2) to (X12_in1);
  	\draw (X12_out1) to (X13_in2);
  	\draw (Z.west|-X11_in2) to (X11_in2);	
  	\draw (Z.west|-X12_in2) to (X12_in2);
  	\draw (X12_out2) to (X21_in2);
  	\draw (X21_out2) to (Z.east|-X21_out2);
  	\draw[ar] let \p1=(X12.south east), \p2=(X12.south west), \n1={\y1-\bby}, \n2=\bbportlen in
  	  (X12_out3) to[in=0] (\x1+\n2,\n1) -- (\x2-\n2,\n1) to[out=180] (X12_in3);
  	\draw[ar] let \p1=(X22.north east), \p2=(X11.north west), \n1={\y2+\bby}, \n2=\bbportlen in
            (X22_out2) to[in=0] (\x1+\n2,\n1) -- (\x2-\n2,\n1) to[out=180] (X11_in1);
  	\draw[ar] (X13_out1) to (Z.east|-X13_out1);
  \end{tikzpicture}
  \]
  \vspace{.5in}
  \endgroup
}

\date{\vfill (Last updated: \today)}

\maketitle
\thispagestyle{empty}

\chapter*{Preface}

Category theory is becoming a central hub for all of pure mathematics.\index{category theory!as central hub of mathematics} It is unmatched in its ability to organize and layer abstractions, to find commonalities between structures of all sorts, and to facilitate communication between different mathematical communities.

But it has also been branching out into science, informatics, and industry. We believe that it has the potential to be a major cohesive force in the world, building rigorous bridges between disparate worlds, both theoretical and practical. The motto at MIT is \emph{mens et manus}, Latin for mind and hand. We believe that category theory---and pure math in general---has stayed in the realm of mind for too long; it is ripe to be brought to hand.

\section*{Purpose and audience}

The purpose of this book is to offer a self-contained tour of applied category theory. It is an invitation to discover advanced topics in category theory through concrete real-world examples. Rather than try to give a comprehensive treatment of these topics---which include adjoint functors, enriched categories, proarrow equipments, toposes, and much more---we merely provide a taste of each. We want to give readers some insight into how it feels to work with these structures as well as some ideas about how they might show up in practice.

The audience for this book is quite diverse: anyone who finds the above description intriguing. This could include a motivated high school student who hasn't seen calculus yet but has loved reading a weird book on mathematical logic they found at the library. Or a machine-learning researcher who wants to understand what vector spaces, design theory, and dynamical systems could possibly have in common. Or a pure mathematician who wants to imagine what sorts of applications their work might have. Or a recently-retired programmer who's always had an eerie feeling that category theory is what they've been looking for to tie it all together, but who's found the usual books on the subject impenetrable.

For example, we find it something of a travesty that in 2018 there is almost no introductory material available on monoidal categories. Even beautiful modern introductions to category theory, e.g.\ by Riehl or Leinster, do not include anything on this rather central topic. The only exceptions we can think of are \cite[Chapter 3]{Coecke.Kissinger:2017a} and \cite{coecke2010categories}, each of which has a very user-friendly introduction to monoidal categories; however, readers who are not drawn to physics may not think to look there.\index{category theory!books on}

The basic idea of monoidal categories is certainly not too abstract; modern human intuition seems to include a pre-theoretical understanding of monoidal categories that is just waiting to be formalized. Is there anyone who wouldn't correctly understand the basic idea being communicated in the following diagram?\index{monoidal category}\index{wiring diagram}\index{cooking}
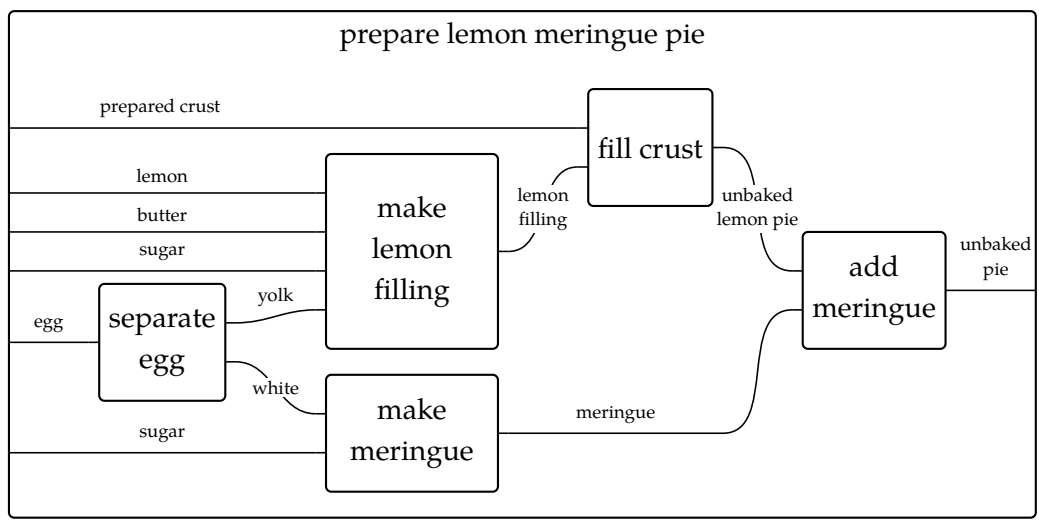
\[
\begin{tikzpicture}[oriented WD, align=center, bbx=1.2cm, bby=2ex]
	\node[bb={4}{1}, bb min width=.9in] (filling) {make\\lemon\\filling};
	\node[bb={2}{1}, bb min width=.9in, below=of filling] (meringue) {make\\meringue};
	\node at ($(filling.west)!.5!(meringue.west)$) (helper) {};
	\node[bb={1}{2}, left = of helper] (separate) {separate\\egg};
	\node[bb={2}{1}, above right = -2 and 1 of filling] (fill) {fill crust};
	\node[bb={2}{1}, below right = of fill] (finish) {add\\meringue};
	\node[bb={0}{0}, bb name=prepare lemon meringue pie, fit={(separate) (meringue) ($(fill.north)+(0,2)$) (finish)}] (pie) {};
\begin{scope}[font=\tiny]
	\draw (pie.west|-fill_in1) to node[pos=.25, above] {prepared crust} (fill_in1);
	\draw (pie.west|-filling_in1) to node[above] {lemon} (filling_in1);
	\draw (pie.west|-filling_in2) to node[above] {butter} (filling_in2);
	\draw (pie.west|-filling_in3) to node[above] {sugar} (filling_in3);
	\draw (pie.west|-separate_in1) to node[above] {egg} (separate_in1);
	\draw (pie.west|-meringue_in2) to node[above] {sugar} (meringue_in2);
	\draw (separate_out1) to node[above] {yolk} (filling_in4);
	\draw (separate_out2) to node[fill=white, inner sep=0.8pt] {white} (meringue_in1);
	\draw (filling_out1) to node[fill=white, inner sep=0.8pt] {lemon\\filling} (fill_in2);
	\draw (fill_out1) to node[fill=white, inner sep=0.8pt] {unbaked\\lemon pie} (finish_in1);
	\draw let \p1=(fill.east|-meringue_out1), \n1=\bbportlen in
		(meringue_out1) to node[above] {meringue} (\x1+\n1,\y1) to (finish_in2);
	\draw (finish_out1) to node[above] {unbaked\\pie} (finish_out1-|pie.east);
\end{scope}
\end{tikzpicture}
\]
Many applied category theory topics seem to take monoidal categories as their jumping-off point. So one aim of this book is to provide a reference---even if unconventional---for this important topic.

We hope this book inspires both new visions and new questions. We intend it
to be self-contained in the sense that it is approachable with minimal
prerequisites, but not in the sense that the complete story is told here. On the
contrary, we hope that readers use this as an invitation to further reading, to
orient themselves in what is becoming a large literature, and to discover new applications for themselves.

This book is, unashamedly, our take on the subject. While the abstract
structures we explore are important to any category theorist, the specific
topics have simply been chosen to our personal taste. Our examples are ones that
we find simple but powerful, concrete but representative, entertaining but in a
way that feels important and expansive at the same time. We hope our readers
will enjoy themselves and learn a lot in the process.

\section*{How to read this book}

The basic idea of category theory---which threads through every chapter---is that if one pays careful attention to structures and coherence, the resulting systems will be extremely reliable and interoperable. For example, a category involves several structures: a collection of objects, a collection of morphisms relating objects, and a formula for combining any chain of morphisms into a morphism. But these structures need to \emph{cohere} or work together in a simple commonsense way: a chain of chains is itself a long chain, so combining a chain of chains should be the same as combining the long chain. That's it!\index{morphism}

We will see structures and coherence come up in pretty much every definition we
give: ``here are some things and here are how they fit together.'' We ask the
reader to be on the lookout for structures and coherence as they read the book,
and to realize that as we layer abstraction upon abstraction, it is the coherence
that makes all the parts work together harmoniously in concert.\index{coherence}

Each chapter in this book is motivated by a real-world topic, such as electrical circuits, control theory, cascade failures, information integration, and hybrid systems. These motivations lead us into and through various sorts of category-theoretic concepts. We generally have one motivating idea and one category-theoretic purpose per chapter, and this forms the title of the chapter, e.g.\ Chapter 4 is ``Collaborative design: profunctors, categorification, and monoidal categories.''

In many math books, the difficulty is roughly a monotonically-increasing function of the page number. In this book, this occurs in each chapter, but not so much in the book as a whole. The chapters start out fairly easy and progress in difficulty.
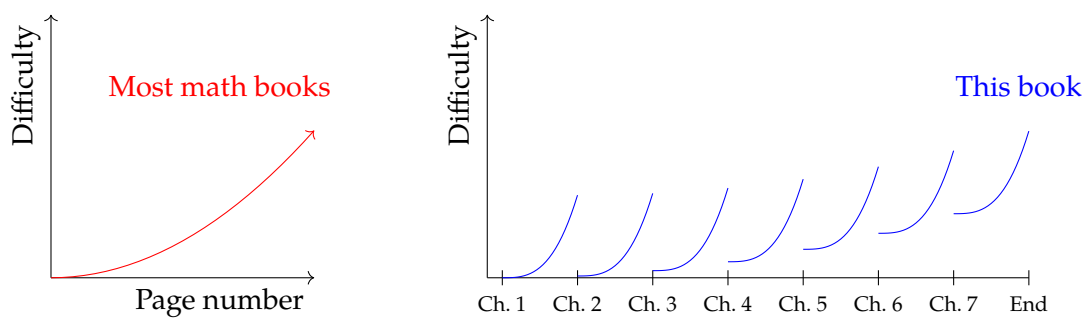
\[
\begin{tikzpicture}
  \draw[->] (0,0) -- (3.5,0) node[below left] (p) {Page number};
  \draw[->] (0,0) -- (0,3.5) node[above left, rotate=90] (d) {Difficulty};
  \draw[color=red, domain=0:3.5, ->] plot (\x,{.16*\x^2)});
  \node[red] at (p|-d) {Most math books};
  \draw (5.8,0) -- (13,0) node[below left] (p) {};
  \draw[->] (5.8,0) -- (5.8,3.5) node[above left, rotate=90] (d) {Difficulty};
  \foreach \i in {1,...,7} {
  	\draw (\i+5,.1) to (\i+5,-.1) node[below, font=\scriptsize] {Ch.\ \i};
    \draw[color=blue, domain=\i+5:\i+6] plot (\x,{1.1*(\x-\i-5)^3+((\i-1)/6.5)^2});
	}
 	\draw (13,.1) to (13,-.1) node[below, font=\scriptsize] {End};
  \node[blue] at (p|-d) {This book};
\end{tikzpicture}
\]
The upshot is that if you find the end of a chapter very difficult, hope is certainly not lost: you can start on the next one and make good progress. This format lends itself to giving you a first taste now, but also leaving open the opportunity for you to come back to the book at a later date and get more deeply into it. But by all means, if you have the gumption to work through each chapter to its end, we very much encourage that!

We include about 240 exercises throughout the text, with solutions in \cref{chap.solutions}. Usually these exercises are fairly straightforward; the only thing they demand is that the reader changes their mental state from passive to active, rereads the previous paragraphs with intent, and puts the pieces together. A reader becomes a \emph{student} when they work the exercises; until then they are more of a tourist, riding on a bus and listening off and on to the tour guide. Hey, there's nothing wrong with that, but we do encourage you to get off the bus and make direct contact with the native population and local architecture as often as you can.

\section*{Acknowledgments}

Thanks to Jared Briskman, James Brock, Ronnie Brown, Thrina Burana, David Chudzicki, Jonathan
Castello, Margo Crawford, Fred Eisele, David Ellerman, Cam Fulton, Bruno
Gavranovi\'c, Sebastian Galkin, John Garvin, Peter Gates, Juan Manuel Gimeno,
Alfredo G\'omez, Leo Gorodinski, Jason Grossman, Jason Hooper, Yuxi Liu, Jes\'us
L\'opez, MTM, Nicol\`o Martini, Martin MacKerel, Pete Morcos, Nelson Niu, James
Nolan, Dan Oneata, Paolo Perrone, Thomas Read, Rif A.  Saurous, Dan Schmidt,
Samantha Seaman, Marcello Seri, Robert Smart, Valter Sorana, Adam
Theriault-Shay, Emmy Trewartha, Sergey Tselovalnikov, Andrew Turner, Joan
Vazquez, Daniel Wang, Jerry Wedekind for helpful comments and conversations.

We also thank our sponsors at the AFOSR; this work was supported by grants
FA9550--14--1--0031 and FA9550--17--1--0058.

Finally, we extend a very special thanks to John Baez for running an
\href{https://forum.azimuthproject.org/categories/applied-category-theory-course}{online
course} on this material and generating tons of great feedback.

\section*{Personal note}

Our motivations to apply category theory outside of math are, perhaps naively,
grounded in the hope it can help bring humanity together to solve our big
problems. But category theory is a tool for thinking, and like any tool it can
be used for purposes we align with and those we don't. 

In this personal note, we ask that readers try to use what they learn in this
book to do something they would call ``good,'' in terms of contributing to the
society they'd want to live in. For example, if you're planning to study this
material with others, consider specifically inviting someone from an
under-represented minority---a group that is more highly represented in society
than in upper-level math classes---to your study group. As another example,
perhaps you can use the material in this book to design software that helps
people relate to and align with each other. What's the mathematics of a
well-functioning society?

The way we use our tools affects all our lives. Our society has seen the
results---both the wonders and the waste---resulting from rampant selfishness.
We would be honored if readers found ways to use category theory as part of an
effort to connect people, to create common ground, to explore the cross-cutting
categories in which life, society, and environment can be represented, and to end the
ignorance entailed by limiting ourselves to a singular ontological perspective
on anything.

If you do something of the sort, please let us and the community know about it.
\vspace{1cm}

\flushright
Brendan Fong and David I.\ Spivak

Cambridge MA, October 2018

%

\clearpage
\tableofcontents*

\mainmatter

\index{applied category theory|(}

\setcounter{chapter}{0}

\chapter[Generative effects: Orders and adjunctions]{Generative effects:\\Orders and Galois connections} %
\label{chap.preorders}

In this book, we explore a wide variety of situations---in the world of science, engineering, and commerce---where we see something we might call \emph{compositionality}. These are cases in which systems or relationships can be combined to form new systems or relationships. In each case we find category-theoretic constructs---developed for their use in pure math---which beautifully describe the compositionality of the situation.%
\index{compositionality}

This chapter, being the first of the book, must serve this goal in two capacities. First, it must provide motivating examples of compositionality, as well as the relevant categorical formulations. Second, it must provide the mathematical foundation for the rest of the book. Since we are starting with minimal assumptions about the reader's background, we must begin slowly and build up throughout the book. As a result, examples in the early chapters are necessarily simplified. However, we hope the reader will already begin to see the sort of structural approach to modeling that category theory brings to the fore.

\section{More than the sum of their parts}%
\label{sec.motivate_1}

We motivate this first chapter by noticing that while many real-world structures
are compositional, the results of observing them are often not. The reason is that observation is inherently ``lossy'': in order to extract information from something, one must drop the details. For example, one stores a real number by rounding it to some precision. But if the details are actually relevant in a given system operation, then the observed result of that operation will not be as expected. This is clear in the case of roundoff error, but it also shows up in non-numerical domains: observing a complex system is rarely enough to predict its behavior because the observation is lossy.

A central theme in category theory is the study of structures and
structure-preserving maps.%
\index{map!structure preserving} A map $f\colon X\to Y$ is a kind of observation of object $X$ via a specified relationship it has with another object, $Y$. For example, think of $X$ as the subject of an experiment and $Y$ as a meter connected to $X$, which allows us to extract certain features of $X$ by looking at the reaction of $Y$.

Asking which aspects of $X$ one wants
to preserve under the observation $f$ becomes the question ``what category are you working
in?.'' As an example, there are many functions $f$ from $\RR$ to $\RR$, and we can think of them as observations: rather than view $x$ ``directly'', we only observe $f(x)$. Out of all the functions $f\colon\rr\to\rr$, only
some of them preserve the order of numbers, only some of them preserve the distance between numbers, only
some of them preserve the sum of numbers, etc. Let's check in with an exercise; a solution can be found in \cref{chap.preorders}.

\begin{exercise}%
\label{exc.function_pres} %
\index{function}%
\index{function!structure preserving}
Some terminology: a function $f\colon \RR\to\RR$ is said to be
\begin{enumerate}[label=(\alph*)]
	\item \emph{order-preserving} if $x\leq y$ implies $f(x)\leq f(y)$, for all $x,y\in\RR$;%
	\footnote{We are often taught to view functions $f\colon\rr\to\rr$ as plots on an $(x,y)$-axis, where $x$ is the domain (independent) variable and $y$ is the codomain (dependent) variable. In this book, we do not adhere to that naming convention; e.g.\ in \cref{exc.function_pres}, both $x$ and $y$ are being ``plugged in'' as input to $f$. As an example consider the function $f(x)=x^2$. Then $f$ being order-preserving would say that for any $x,y\in\RR$, if $x\leq y$ then $x^2\leq y^2$; is that true?}
	\item \emph{metric-preserving} if $|x-y|=|f(x)-f(y)|$;
	\item \emph{addition-preserving} if $f(x+y)=f(x)+f(y)$.
\end{enumerate}
For each of the three properties defined above---call it \emph{foo}---find an $f$ that is \emph{foo}-preserving and
an example of an $f$ that is not \emph{foo}-preserving.
\end{exercise}

In category theory we want to keep control over which aspects of our systems are being preserved under various observations. As we said above, the less structure is preserved by our observation of a system, the more ``surprises'' occur when we observe its operations. One might call these surprises \emph{generative effects}.
\index{generative effect}

In using category theory
to explore generative effects, we follow the basic ideas from work by Adam \cite{Adam:2017a}. He goes much more deeply into the issue than we can here; see \cref{ch1.further_reading}. But as mentioned above, we must also use this chapter to give an order-theoretic warm-up for the full-fledged category theory to come.

\subsection{A first look at generative effects}%
\label{subsec.first_look_gen}

To explore the notion of a generative effect we need a sort of system, a sort of observation, and a system-level operation that is not preserved by the observation. Let's start with a simple example.

\paragraph{A simple system.}

Consider three points; we'll call them $\bullet$, $\circ$ and $\ast$. In this example, a \emph{system} will simply be a way of connecting these points
together. We might think of our points as sites on a power grid, with a system
describing connection by power lines, or as people susceptible to some disease,
with a system describing interactions that can lead to contagion. As an abstract example of a system,
there is a system where $\bullet$ and $\circ$ are connected, but neither are
connected to $\ast$. We shall draw this like so:%
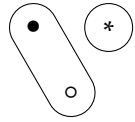
\index{connectedness}
\[

\]
Connections are symmetric, so if $a$ is connected to $b$, then
$b$ is connected to $a$. Connections are also transitive, meaning that if
$a$ is connected to $b$, and $b$ is connected to $c$, then $a$ is connected to
$c$; that is, all $a$, $b$, and $c$ are connected. Friendship is not transitive---my friend's friend is not necessarily my friend---but possible communication of a concept or a disease is.

Here we depict two more systems, one in which none of the points are connected, and one in which all three points are connected.
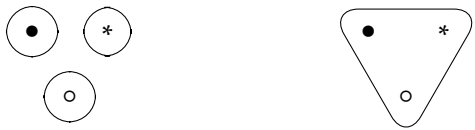
\[
%
\]
There are five systems in all, and we depict them just below.

Now that we have defined the sort of system we want to discuss, suppose that Alice is observing this system. Her observation of interest, which we call $\Phi$, extracts a single feature from a system, namely whether the point $\bullet$ is
connected to the point $\ast$; this is what she wants to know. Her observation of the system will be an assignment of either $\true$ or $\false$; she assigns $\true$ if $\bullet$ is connected to $\ast$, and $\false$
otherwise. So $\Phi$ assigns the value $\true$ to the following two systems:
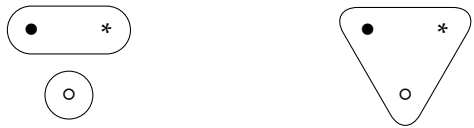
\[
%
\]
and $\Phi$ assigns the value $\false$ to the three remaining systems:
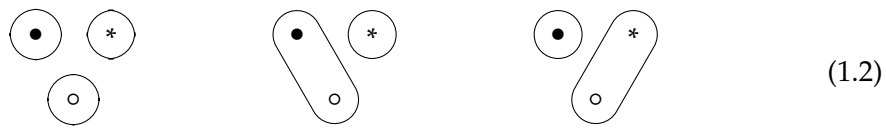
\begin{equation}%
\label{eqn.false_systems}
  \begin{aligned}
%
\end{aligned}
\end{equation}

The last piece of setup is to give a sort of operation that Alice wants to perform on the systems themselves. It's a very common operation---one that will come up many times throughout the book---called \emph{join}. If the reader has been following the story arc, the expectation here is that Alice's connectivity observation will not be compositional with respect to the operation of system joining; that is, there will be generative effects. Let's see what this means.

\paragraph{Joining our simple systems.}%
\index{join|(}

Joining two systems $A$ and $B$ is performed simply by combining their
connections.  That is, we shall say the \emph{join} of systems $A$ and $B$,
denote it $A \vee B$, has a connection between points $x$ and $y$ if there are
some points $z_1, \dots, z_n$ such that each of the following are true in at
least one of $A$ or $B$: $x$ is connected to $z_1$, $z_i$ is connected to
$z_{i+1}$, and $z_n$ is connected to $y$. In a three-point system, the above
definition is overkill, but we want to say something that works for systems with
any number of elements. The high-level way to say it is ``take the transitive
closure of the union of the connections in $A$ and $B$.'' In our three-element
system, it means for example that%
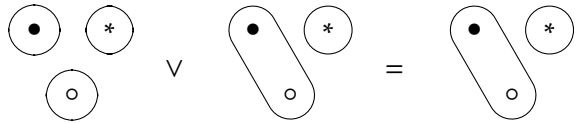
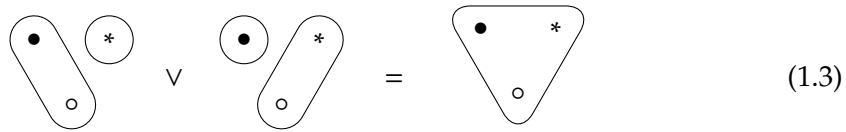
\index{union}
\[
\begin{aligned}

\end{aligned}
\end{equation}

\begin{exercise}%
\label{exc.joining_systems}
What is the result of joining the following two systems?
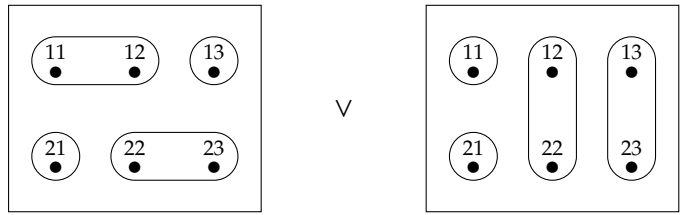
\[
%
\qedhere
\]
\end{exercise}

We are now ready to see the generative effect.%
\index{generative effect} We don't want to build it up too much---this example has been made as simple as possible---but we will see that Alice's observation fails to preserve the join operation. We've been denoting her observation---measuring whether $\bullet$ and $\ast$ are connected---by the symbol $\Phi$; it returns a boolean result, either $\true$ or $\false$.%
\index{booleans}

We see above in \cref{eqn.false_systems} that
$\Phi(\resizebox{!}{2ex}{
%
})
= \false$: in both cases $\bullet$ is not
connected to $\ast$. On the other hand, when we join these two systems as in \cref{eqn.generative}, we see that
$\Phi(\resizebox{!}{2ex}{
%
}) 
= \true$: in the joined system, $\bullet$ \emph{is} connected to $\ast$.%
\label{page.generativity}
The question that Alice is interested in, that of $\Phi$, is inherently lossy with respect to join, and there is no way to fix it without a more detailed observation, one that includes not only $\ast$ and $\bullet$ but also $\circ$. 

While this was a simple example, it should be noted that whether the potential for such effects exist---i.e.\ determining whether an observation is operation-preserving---can be incredibly important information to know. For example, Alice could be in charge of putting together the views of two
local authorities regarding possible contagion between an infected person $\bullet$ and
a vulnerable person $\ast$. Alice has noticed that if they separately extract information from their raw data and combine the results, it gives a different answer than if they combine their raw data and extract information from it.

\index{join|)}
\subsection{Ordering systems}

Category theory is all about organizing and layering structures. In this section we will explain how the operation of joining systems can be derived from a more basic structure: order. We will see that while joining is not preserved by Alice's connectivity observation $\Phi$, order is.%
\index{joins!preservation of}%
\index{order!preservation of}%
\index{order|seealso {preorder}}

To begin, we note that
the systems themselves are ordered in a hierarchy. Given systems $A$ and $B$, we
say that $A \le B$ if, whenever $x$ is connected to $y$ in $A$, then $x$ is
connected to $y$ in $B$. For example,
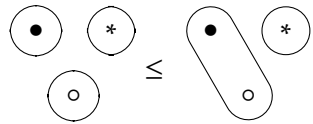
\[
\begin{aligned}

\end{aligned}
\]
This notion of $\leq$ leads to the following diagram:
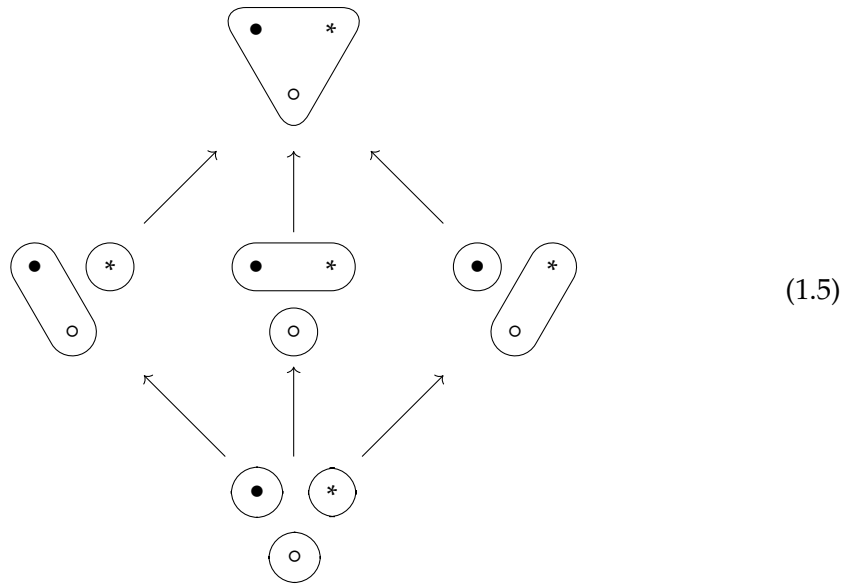
\begin{equation}%
\label{eqn.parts_of_3}
%
    };
    \draw[->] (none) -- (ab);
    \draw[->] (none) -- (ac);
    \draw[->] (none) -- (bc);
    \draw[->] (ab) -- (all);
    \draw[->] (ac) -- (all);
    \draw[->] (bc) -- (all);
\end{tikzpicture}
\end{equation}
where an arrow from system $A$ to system $B$ means $A \le B$. Such diagrams are known as \emph{Hasse diagrams}.%
\index{Hasse diagram!for preorders}
\index{Hasse diagram}

As we were saying above, the notion of join is derived from this order. Indeed for any two systems $A$ and $B$ in the Hasse diagram \eqref{eqn.parts_of_3}, the joined system $A \vee B$ is the smallest system that is bigger than both $A$ and $B$. That is, $A\leq (A\vee B)$ and $B\leq (A\vee B)$, and for any $C$, if  $A \le C$ and $B \le C$ then $(A\vee B)\leq C$. Let's walk through this with an exercise.

\begin{exercise}%
\label{exc.partitions_practice}
\begin{enumerate}
	\item Write down all the partitions of a two element set $\{\bullet,\ast\}$, order them as above, and draw the Hasse diagram.
	\item Now do the same thing for a four element-set, say $\{1,2,3,4\}$. There should be 15 partitions.
\end{enumerate}
Choose any two systems in your 15-element Hasse diagram, call them $A$ and $B$. 
\begin{enumerate}[resume]
	\item What is $A\vee B$, using the definition given in the paragraph above \cref{eqn.generative}?
	\item Is it true that $A\leq (A\vee B)$ and $B\leq (A\vee B)$?
	\item What are all the systems $C$ for which both $A\leq C$ and $B\leq C$.
	\item Is it true that in each case, $(A\vee B)\leq C$?
	\qedhere
\qedhere
\end{enumerate}
\end{exercise}

The set $\BB=\{\true, \false\}$ of booleans also has an order, $\false \le \true$:
\[
\begin{tikzpicture}
\node (t) {$\true$};
\node at ($(t)+(0,-1)$) (f) {$\false$};
\draw[->] (f) -- (t);
\end{tikzpicture}
\]
Thus $\false\leq\false$, $\false\leq\true$, and $\true\leq\true$, but
$\true\not\leq\false$. In other words, $A\leq B$ if $A$ implies $B$.%
\footnote{In mathematical logic, $\false$ implies $\true$ but $\true$ does not imply $\false$. That is ``$P$ implies $Q$'' means, ``if $P$ is true, then $Q$ is true too, but if $P$ is not true, I'm making no claims.''%
\index{logic!implication in}}
\index{booleans!as set}

For any $A,B$ in $\BB$, we can again write $A \vee B$ to mean the least element that is greater than both $A$ and $B$.
\begin{exercise}%
\label{exc.boolean_vee_practice}
Using the order $\false\leq\true$ on $\BB=\{\true,\false\}$, what is:
\begin{enumerate}
	\item $\true \vee \false$?
	\item $\false \vee \true$?
	\item $\true \vee \true$?
	\item $\false \vee \false$?
	\qedhere
\qedhere
\end{enumerate}
\end{exercise}

Let's return to our systems with $\bullet$, $\circ$, and $\ast$, and Alice's
``$\bullet$ is connected to $\ast$'' function, which we called $\Phi$. It takes
any such system and returns either $\true$ or $\false$. Note that the map $\Phi$
preserves the $\leq$ order: if $A\leq B$ and there is a connection between
$\bullet$ and $\ast$ in $A$, then there is such a connection in $B$ too. The
possibility of a generative effect is captured in the inequality
\begin{equation}%
\label{eqn.generative_def}
\Phi(A) \vee \Phi(B) \leq \Phi(A\vee B).
\end{equation}
We saw on page~\pageref{page.generativity} that this can be a strict inequality: we showed two systems $A$ and $B$
with $\Phi(A)=\Phi(B)=\false$, so $\Phi(A)\vee\Phi(B)=\false$, but where
$\Phi(A\vee B)=\true$. In this case, a generative effect exists.%
\index{generative effect}

These ideas capture the most basic ideas in category theory. Most directly, we have seen that the map $\Phi$ preserves some structure but not others: it preserves order but not join. In fact, we have seen here hints of more complex notions from category theory, without making them explicit; these include the notions of category, functor, colimit, and adjunction. In this
chapter we will explore these ideas in the elementary setting of ordered sets.


\section{What is order?} %
\label{sec.preorders}

Above we informally spoke of two different ordered sets: the order on system
connectivity and the order on booleans $\false\leq\true$. Then we related these
two ordered sets by means of Alice's observation $\Phi$. Before continuing, we need
to make such ideas more precise. We begin in \cref{sec.sets_and_rels} with a
review of sets and relations. In \cref{subsec.def_preorder} we will give the
definition of a preorder---short for preordered set---and a good number
of examples.

\subsection{Review of sets, relations, and functions}%
\label{sec.sets_and_rels}%
\index{set|(}

We will not give a definition of \emph{set} here, but informally we will think of a set as
a collection of things, known as elements. These things could be all the leaves
on a certain tree, or the names of your favorite fruits, or simply some symbols $a$,
$b$, $c$. For example, we write $A=\{h,1\}$ to denote the set, called $A$, that
contains exactly two elements, one called $h$ and one called $1$. The set $\{h,h,1,h,1\}$ is exactly the same as $A$ because they both contain the same elements, $h$ and $1$, and repeating an element more than once in the notation doesn't change the set.%
\footnote{If you want a notion where ``$h,1$'' is different than ``$h,h,1,h,1$'', you can use something called \emph{bags}, where the number of times an element is listed matters, or \emph{lists}, where order also matters. All of these are important concepts in applied category theory, but sets will come up the most for us.}
For an arbitrary set $X$, we write $x \in X$ if $x$ is
element of $X$; so we have $h\in A$ and $1\in A$, but $0\not\in A$.

\begin{example}
Here are some important sets from mathematics---and the notation we will use---that will appear again in this book.%
\index{notation!for common sets}
\begin{itemize}
	\item $\varnothing$ denotes the empty set; it has no elements.%
\index{set!empty}
	\item $\{1\}$ denotes a set with one element; it has one element, $1$.%
\index{set!one element}
	\item $\bb$ denotes the set of \emph{booleans}; it has two elements, $\true$ and $\false$.%
\index{set!booleans as}%
\index{booleans!as set}
	\item $\nn$ denotes the set of \emph{natural numbers}; it has elements $0,1,2,3,\ldots,90^{717},\ldots$.%
\index{set!natural numbers as}
\index{natural numbers as!as set}
	\item $\ord{n}$, for any $n\in\nn$, denotes the \emph{$n^{\text{th}}$ ordinal}; it has $n$ elements $1,2,\ldots,n$. For example, $\ord{0}=\varnothing$, $\ord{1}=\{1\}$, and $\ord{5}=\{1,2,3,4,5\}$.%
\index{set!$n\th$ ordinal as}
	\item $\zz$, the set of \emph{integers}; it has elements $\ldots,-2,-1,0,1,2,\ldots,90^{717},\ldots$.%
\index{set!integers as}
	\item $\rr$, the set of \emph{real numbers}; it has elements like $\pi, 3.14, 5*\sqrt{2}, e, e^2, -1457, 90^{717}$, etc.%
\index{set!real numbers as}%
\index{real numbers!as set}
\end{itemize}
\end{example}

Given sets $X$ and $Y$, we say that $X$ is a \emph{subset} of $Y$, and write
$X\ss Y$, if every element in $X$ is also in $Y$. For example $\{h\}\ss A$. Note
that the empty set $\varnothing\coloneqq\{\}$ is a subset of every other set.%
	\footnote{When we write $Z\coloneqq \foo$, it means ``assign the meaning
	$\foo$ to variable $Z$'', whereas $Z=\foo$ means simply that $Z$ is
	equal to $\foo$, perhaps as discovered via some calculation. In
	particular, $Z\coloneqq \foo$ implies $Z=\foo$ but not vice versa;
	indeed it \emph{would not} be proper to write $3+2\coloneqq 5$ or
	$\{\}\coloneqq \varnothing$.
}
 Given a set $Y$ and a property $P$ that is either true or false for each element of $Y$, we write $\{y\in Y\mid P(y)\}$ to mean the subset of those $y$'s that satisfy $P$.%
\index{subset}%
\index{subset|seealso {power set}}%
\index{notation!set builder}

\begin{exercise}%
\label{exc.comprehension_comprehension}
\begin{enumerate}
	\item Is it true that $\nn=\{n\in\zz\mid n\geq 0\}$?
	\item Is it true that $\nn=\{n\in\zz\mid n\geq 1\}$?
	\item Is it true that $\varnothing=\{n\in\zz\mid 1<n<2\}$?
\qedhere
\end{enumerate}
\end{exercise}

If both $X_1$ and $X_2$ are subsets of $Y$, their \emph{union}, denoted $X_1\cup X_2$, is also a subset of $Y$, namely the one containing the elements in $X_1$ and the elements in $X_2$ but no more. For example if $Y=\{1,2,3,4\}$ and $X_1=\{1,2\}$ and $X_2=\{2,4\}$, then $X_1\cup X_2=\{1,2,4\}$. Note that $\varnothing\cup X=X$ for any $X\ss Y$.%
\index{union}

Similarly, if both $X_1$ and $X_2$ are subsets of $Y$, then their \emph{intersection}, denoted $X_1\cap X_2$, is also a subset of $Y$, namely the one containing all the elements of $Y$ that are both in $X_1$ and in $X_2$, and no others. So $\{1,2,3\}\cap\{2,5\}=\{2\}$.%
\index{intersection}

What if we need to union or intersect a lot of subsets? For example, consider the sets $X_0=\varnothing$, $X_1=\{1\}$, $X_2=\{1,2\}$, etc. as subsets of $\nn$, and we want to know what the union of all of them is. This union is written $\bigcup_{n\in\nn}X_n$, and it is the subset of $\nn$ that contains every element of every $X_n$, but no others. Namely, $\bigcup_{n\in\nn}X_n=\{n\in \nn\mid n\geq 1\}$. Similarly one can write $\bigcap_{n\in\nn}X_n$ for the intersection of all of them, which will be empty in the above case.

Given two sets $X$ and $Y$, the \emph{product} $X \times Y$ of
$X$ and $Y$ is the set of pairs $(x,y)$, where $x \in X$ and $y \in Y$.%
\index{product!of sets}

Finally, we may want to take a \emph{disjoint} union of two sets, even if they
have elements in common. Given two sets $X$ and $Y$, their \emph{disjoint union}
$X \dju Y$ is the set of pairs of the form $(x,1)$ or $(y,2)$, where $x \in X$
and $y \in Y$.
\index{union!disjoint}

\begin{exercise}%
\label{exc.subsets_products}
	Let $A\coloneqq\{h,1\}$ and $B\coloneqq\{1,2,3\}$.
	\begin{enumerate}
		\item There are eight subsets of $B$; write them out.
		\item Take any two nonempty subsets of $B$ and write out their union.
		\item There are six elements in $A\times B$; write them out.
	\qedhere
		\item There are five elements of $A \dju B$; write them out.
		\item If we consider $A$ and $B$ as subsets of the set
	$\{h,1,2,3\}$, there are four elements of $A \cup B$; write them out.
	\end{enumerate}
\end{exercise}
\index{set|)}

Relationships between different sets---for example between the set of trees in your neighborhood and the set of your favorite fruits---are captured using subsets and product sets.

\begin{definition}%
\index{relation}
Let $X$ and $Y$ be sets. A \emph{relation between $X$ and $Y$} is a subset $R\subseteq X
\times Y$. A \emph{binary relation on $X$} is a relation between $X$ and $X$, i.e.\ a subset $R\ss X\times X$.%
\index{relation!binary}%
\index{binary relation|see {relation, binary}}
\end{definition}

It is convenient to use something called \emph{infix notation} for binary relations $R\ss A\times A$. This means one picks a symbol, say $\star$, and writes $a\star b$ to mean $(a,b)\in R$.
\index{infix notation}%
\index{notation!infix|see {infix}}

\begin{example}%
\index{relation!binary}
There is a binary relation on $\RR$ with infix notation $\leq$. Rather than writing $(5,6)\in R$, we write $5\leq 6$.

Other examples of infix notation for relations are $=$, $\approx$, $<$, $>$. In number theory, they are interested in whether one number divides without remainder into another number; this relation is denoted with infix notation $|$, so $5|10$.
\end{example}%
\index{divides relation}%
\index{infix notation}

\paragraph{Partitions and equivalence relations.}%
\index{equivalence relation!and partition|(}

We can now define partitions more formally.

\begin{definition}%
\label{def.partition}%
\index{partition}
If $A$ is a set, a \emph{partition} of $A$ consists of a set $P$ and, for each $p\in P$, a nonempty subset $A_p\ss A$, such that 
\begin{equation}%
\label{eqn.condition_part}
  A=\bigcup_{p\in P}A_p
  \qquad\text{and}\qquad
  \text{if }p\neq q\text{ then }A_p\cap A_q=\varnothing.
\end{equation}
We may denote the partition by $\{A_p\}_{p\in P}$. We refer to $P$ as the set of \emph{part labels} and if $p\in P$ is a part label, we refer to $A_p$ as the \emph{$p^{\tn{th}}$ part}. The condition \eqref{eqn.condition_part} says that each element $a\in A$ is in exactly one part.

We consider two different partitions $\{A_p\}_{p\in P}$ and $\{A'_{p'}\}_{p'\in P'}$ of $A$ to be the same if for each $p\in P$ there exists a $p'\in P'$ with $A_p=A'_{p'}$. In other words, if two ways to divide $A$ into parts are exactly the same---the only change is in the labels---then we don't make a distinction between them.%
\index{partition!label irrelevance of}
\end{definition}

\begin{exercise}%
\label{exc.proof_re_parts}
Suppose that $A$ is a set and $\{A_p\}_{p\in P}$ and $\{A'_{p'}\}_{p'\in P'}$ are two partitions of $A$ such that for each $p\in P$ there exists a $p'\in P'$ with $A_p=A'_{p'}$.
\begin{enumerate}
	\item Show that for each $p\in P$ there is at most one $p'\in P'$ such that $A_p=A'_{p'}$
	\item Show that for each $p'\in P'$ there is a $p\in P$ such that $A_p=A'_{p'}$.
	\qedhere
\qedhere
\end{enumerate}
\end{exercise}

\begin{exercise}%
\label{exc.equiv_rel_practice}
Consider the partition shown below:
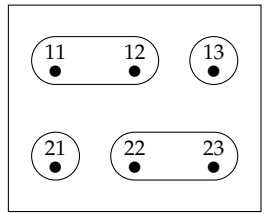
\[
\begin{tikzpicture}[x=1cm]
	\node (a11) {$\LMO{11}$};
	\node[right=.5 of a11] (a12) {$\LMO{12}$};
	\node[right=.5 of a12] (a13) {$\LMO{13}$};
	\node[below=.5 of a11] (a21) {$\LMO{21}$};
	\node[below=.5 of a12] (a22) {$\LMO{22}$};
	\node[below=.5 of a13] (a23) {$\LMO{23}$};
	\draw [rounded corners=9pt] 
     ($(a11)+(135:.45)$) --
     ($(a11)+(-135:.45)$) --
     ($(a12)+(-45:.45)$) --
     ($(a12)+(45:.45)$) --     
     cycle;
	\draw [rounded corners=9pt] 
     ($(a22)+(135:.45)$) --
     ($(a22)+(-135:.45)$) --
     ($(a23)+(-45:.45)$) --
     ($(a23)+(45:.45)$) --     
     cycle;
	\draw [rounded corners=9pt] 
     ($(a13)+(135:.45)$) --
     ($(a13)+(-135:.45)$) --
     ($(a13)+(-45:.45)$) --
     ($(a13)+(45:.45)$) --     
     cycle;
	\draw [rounded corners=9pt] 
     ($(a21)+(135:.45)$) --
     ($(a21)+(-135:.45)$) --
     ($(a21)+(-45:.45)$) --
     ($(a21)+(45:.45)$) --     
     cycle;
	\node[inner sep=10pt, draw, fit=(a11) (a23)] (a) {};
\end{tikzpicture}
\]
For any two elements $a,b\in\{11,12,13,21,22,23\}$, let's allow ourselves to write a twiddle symbol $a\sim b$ between them if $a$ and $b$ are both in the same part. Write down every pair of elements $(a,b)$ that are in the same part. There should be 10.%
\footnote{Hint: whenever someone speaks of ``two elements $a,b$ in a set $A$'', the two elements may be the same!}%
\index{partition!part of}
\end{exercise}

We will see in \cref{prop.equivalence_partition} that there is a strong relationship between partitions and something called equivalence relations, which we define next.

\begin{definition}%
\label{def.equivalence_relation}%
\index{equivalence relation}%
\index{relation!equivalence|see {equivalence relation}}
Let $A$ be a set. An \emph{equivalence relation} on $A$ is a binary relation, let's give it infix notation $\sim$,%
\index{equivalence relation!as binary relation}
 satisfying the following three properties:
\begin{enumerate}[label=(\alph*)]
	\item $a\sim a$, for all $a\in A$,
	\item $a\sim b$ iff\footnote{`Iff' is short for `if and only if'.} $b\sim a$, for all $a,b\in A$, and
	\item if $a\sim b$ and $b\sim c$ then $a\sim c$, for all $a,b,c\in A$.
\end{enumerate}
These properties are called \emph{reflexivity}, \emph{symmetry}, and \emph{transitivity}, respectively.
\index{reflexivity}%
\index{symmetry}%
\index{transitivity}%
\index{infix notation}
\end{definition}

\begin{proposition}%
\label{prop.equivalence_partition}%
\index{partition!associated equivalence relation of}
Let $A$ be a set. There is a one-to-one correspondence between the ways to partition $A$ and the equivalence relations on $A$.
\end{proposition}
\begin{proof}
We first show that every partition gives rise to an equivalence relation, and then that every equivalence relation gives rise to a partition. Our two constructions will be mutually inverse, proving the proposition.

Suppose we are given a partition $\{A_p\}_{p\in P}$; we define a relation $\sim$ and show it is an equivalence relation. Define $a\sim b$ to mean that $a$ and $b$ are in the same part: there is some $p\in P$ such that $a\in A_p$ and $b\in A_p$. It is obvious that $a$ is in the same part as itself. Similarly, it is obvious that if $a$ is in the same part as $b$ then $b$ is in the same part as $a$, and that if further $b$ is in the same part as $c$ then $a$ is in the same part as $c$. Thus $\sim$ is an equivalence relation as defined in \cref{def.equivalence_relation}.

Suppose given an equivalence relation $\sim$; we will form a partition on $A$ by saying what the parts are. Say that a subset $X\ss A$ is $(\sim)$-closed if, for every $x\in X$ and $x'\sim x$, we have $x'\in X$. Say that a subset $X\ss A$ is $(\sim)$-connected if it is nonempty and $x\sim y$ for every $x,y\in X$. Then the parts corresponding to $\sim$ are exactly the $(\sim)$-closed, $(\sim)$-connected subsets. It is not hard to check that these indeed form a partition.%
\index{connected}
\end{proof}

\begin{exercise}%
\label{exc.equiv_part_proof}
Let's complete the ``it's not hard to check'' part in the proof of \cref{prop.equivalence_partition}. Suppose that $\sim$ is
an equivalence relation on a set $A$, and let $P$ be the set of $(\sim)$-closed and
$(\sim)$-connected subsets $\{A_p\}_{p\in P}$.
\begin{enumerate}
	\item Show that each part $A_p$ is nonempty.
	\item Show that if $p\neq q$, i.e.\ if $A_p$ and $A_q$ are not exactly the same set, then $A_p\cap A_q=\varnothing$.
	\item Show that $A=\bigcup_{p\in P}A_p$.
	\qedhere
\qedhere
\end{enumerate}
\end{exercise}

\begin{definition}
\label{def.quotient}
\index{quotient}
Given a set $A$ and an equivalence relation $\sim$ on $A$, we say that the
\emph{quotient} $A/\sim$ of $A$ under $\sim$ is the set of parts of the
corresponding partition.
\end{definition}

\index{equivalence relation!and partition|)}

\paragraph{Functions.} The most frequently used sort of relation between sets is
that of functions.%
\index{relation!function as}

\begin{definition}%
\label{def.function}%
\index{function}
Let $S$ and $T$ be sets. A \emph{function from $S$ to $T$} is a subset $F\ss
S\times T$ such that for all $s\in S$ there exists a unique $t\in T$ with
$(s,t)\in F$.%
\index{function!as relation}

The function $F$ is often denoted $F\colon S\to T$. From now on, we write $F(s)=t$, or sometimes $s\mapsto t$, to mean $(s,t)\in F$. For any $t\in T$, the \emph{preimage of $t$ along $F$} is the subset $\{s\in S\mid F(s)=t\}$.%
\index{preimage}

A function is called \emph{surjective}%
\index{function!surjective}, or a
\emph{surjection}, if
for all $t\in T$
there exists $s\in S$ with $F(s)=t$. A function is called
\emph{injective}%
\index{function!injective}, or an \emph{injection}, if
for all $t\in T$ and $s_1,s_2\in S$ with
$F(s_1)=t$ and $F(s_2)=t$, we have $s_1=s_2$. A function is called
\emph{bijective}%
\index{function!bijective} if it is both surjective and
injective.%
\index{function!injective}%
\index{function!bijective}%
\index{function!surjective}
\end{definition}

We use various decorations on arrows, $\to, \surj, \inj, \To{\cong}$ to denote these special sorts of functions. Here is a table with the name, arrow decoration, and an example of each sort of function:
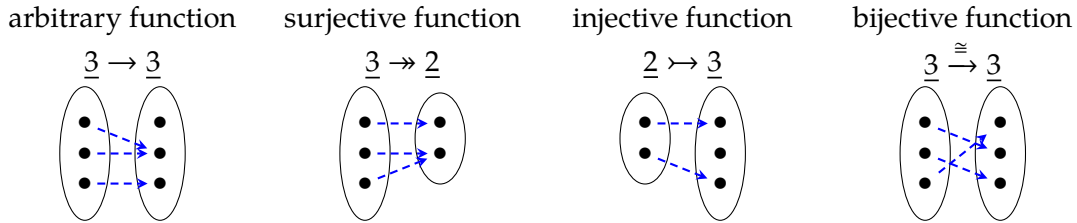
\[
\begin{tabular}{>{\centering\arraybackslash}m{1.3in}>{\centering\arraybackslash}m{1.3in}>{\centering\arraybackslash}m{1.3in}>{\centering\arraybackslash}m{1.3in}}
	arbitrary function&
	surjective function&
	injective function&
	bijective function\\
	$\ord{3}\to \ord{3}$&
	$\ord{3}\surj\ord{2}$&
	$\ord{2}\inj \ord{3}$&
	$\ord{3}\To{\cong}\ord{3}$\\
\end{tabular}
\]

\begin{example}%
\label{ex.identity}%
\index{function!identity}%
\index{identity!function}
An important but very simple sort of function is the \emph{identity function} on a set $X$, denoted $\id_X$. It is the bijective function $\id_X(x)=x$.
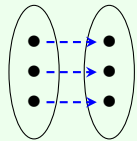
\[
%
\qedhere
\]
\end{example}

For notational consistency with \cref{def.function}, the arrows in
\cref{ex.identity} might be drawn as $\mapsto$ rather than
$\color{blue}\dashrightarrow$. The $\color{blue}\dashrightarrow$-style arrows
were drawn because we thought it was prettier, i.e.\ easier on the eye. Beauty
is important too; an imbalanced preference for strict correctness over beauty
becomes \emph{pedantry}. But outside of pictures, we will be careful.

\begin{exercise}%
\label{exc.inj_surj_fun}
In the following, do not use any examples already drawn above.
\begin{enumerate}
	\item Find two sets $A$ and $B$ and a function $f\colon A\to B$ that is injective but not surjective.
	\item Find two sets $A$ and $B$ and a function $f\colon A\to B$ that is surjective but not injective.
\end{enumerate}
Now consider the four relations shown here:
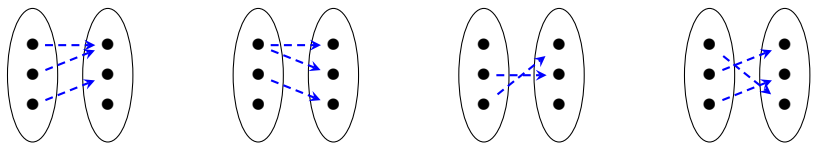
\[
\begin{tikzpicture}[short=-2pt]
	\foreach \f in {0,...,3} {
		\foreach \x in {0,...,2} 
			{\draw (0+3*\f,.4-.4*\x) node (X\f\x) {$\bullet$};}
		\node[draw, ellipse, inner sep=0pt, fit=(X\f0) (X\f2)] (X\f) {};
		\foreach \x in {0,...,2} 
			{\draw (1+3*\f,.4-.4*\x) node (Y\f\x) {$\bullet$};}
		\node[draw, ellipse, inner sep=0pt, fit=(Y\f0) (Y\f2)] (Y\f) {};
	}
		\draw[mapsto] (X00) to (Y00);
		\draw[mapsto] (X01) to (Y00);
		\draw[mapsto] (X02) to (Y01);
		\draw[mapsto] (X10) to (Y10);
		\draw[mapsto] (X10) to (Y11);
		\draw[mapsto] (X11) to (Y12);
		\draw[mapsto] (X21) to (Y21);
		\draw[mapsto] (X22) to (Y20);
		\draw[mapsto] (X30) to (Y32);
		\draw[mapsto] (X31) to (Y30);
		\draw[mapsto] (X32) to (Y31);
\end{tikzpicture}
\]
For each relation, answer the following two questions.
\begin{enumerate}[resume]
	\item Is it a function?
	\item If not, why not? If so, is it injective, surjective, both (i.e.\ bijective), or neither?
\qedhere
\end{enumerate}
\end{exercise}

\begin{exercise}%
\label{exc.map_to_empty}
Suppose that $A$ is a set and $f\colon A\to\varnothing$ is a function to the empty set. Show that $A$ is empty.%
\index{set!empty}
\end{exercise}

\begin{example}%
\label{ex.partition_and_surjections}%
\index{partition!as
  surjection}
A partition on a set $A$ can also be understood in terms of surjective functions
out of $A$. Given a surjective function $f\colon A\surj P$, where $P$ is any
other set, the preimages $f\inv(p)\ss A$, one for each element $p\in
P$, form a partition of $A$. Here is an example.%
\index{preimage}

Consider the partition of $S\coloneqq\{11, 12, 13, 21, 22, 23\}$ shown below:
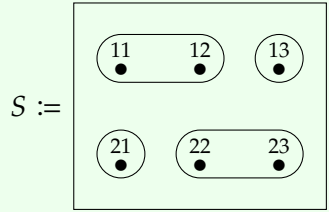
\[
\begin{tikzpicture}[x=1cm]
	\node (a11) {$\LMO{11}$};
	\node[right=.5 of a11] (a12) {$\LMO{12}$};
	\node[right=.5 of a12] (a13) {$\LMO{13}$};
	\node[below=.5 of a11] (a21) {$\LMO{21}$};
	\node[below=.5 of a12] (a22) {$\LMO{22}$};
	\node[below=.5 of a13] (a23) {$\LMO{23}$};
	\draw [rounded corners=9pt] 
     ($(a11)+(135:.45)$) --
     ($(a11)+(-135:.45)$) --
     ($(a12)+(-45:.45)$) --
     ($(a12)+(45:.45)$) --     
     cycle;
	\draw [rounded corners=9pt] 
     ($(a22)+(135:.45)$) --
     ($(a22)+(-135:.45)$) --
     ($(a23)+(-45:.45)$) --
     ($(a23)+(45:.45)$) --     
     cycle;
	\draw [rounded corners=9pt] 
     ($(a13)+(135:.45)$) --
     ($(a13)+(-135:.45)$) --
     ($(a13)+(-45:.45)$) --
     ($(a13)+(45:.45)$) --     
     cycle;
	\draw [rounded corners=9pt] 
     ($(a21)+(135:.45)$) --
     ($(a21)+(-135:.45)$) --
     ($(a21)+(-45:.45)$) --
     ($(a21)+(45:.45)$) --     
     cycle;
	\node[inner sep=10pt, draw, fit=(a11) (a23)] (a) {};
	\node[left=0pt of a] {$S\coloneqq$};
\end{tikzpicture}
\]
It has been partitioned into four parts, so let $P=\{a,b,c,d\}$ and let $f\colon S\surj P$ be given by
\[
  f(11)=a, \quad f(12)=a,\quad f(13)=b,\quad f(21)=c,\quad f(22)=d, \quad f(23)=d
\]
\end{example}

\begin{exercise}%
\label{exc.part_surj}
Write down a surjection corresponding to each of the five partitions in \cref{eqn.parts_of_3}.
\end{exercise}

\begin{definition}%
\label{def.composite_fn}%
\index{function!composite}
If $F\colon X\to Y$ is a function and $G\colon Y\to Z$ is a function, their
\emph{composite} is the function $X\to Z$ defined to be $G(F(x))$ for any $x\in
X$. It is often denoted $G\circ F$, but we prefer to denote it $F\cp G$. It takes any element $x\in X$, evaluates $F$ to get an element $F(x)\in Y$ and then evaluates $G$ to get an element $G(F(x))$.
\end{definition}

\begin{example}
If $X$ is any set and $x\in X$ is any element, we can think of $x$ as a function $\singleton\to X$, namely the function sending $1$ to $x$. For example, the three functions $\singleton\to\{1,2,3\}$ shown below correspond to the three elements of $\{1,2,3\}$:
\[
\begin{tikzpicture}[short=-2pt]
	\foreach \f in {1,...,3} {
		\draw (0+3*\f,0) node (X\f0) {$\bullet$};
		\node[draw, ellipse, inner sep=0pt, fit=(X\f0)] (X\f) {};
		\foreach \x in {1,...,3} {
			\draw (1+3*\f,.8-.4*\x) node[label={[below right=0]:\tiny\x}] (Y\f\x) {$\bullet$};
		}
		\draw[mapsto] (X\f0) to (Y\f\f);
		\node[draw, ellipse, inner sep=0pt, fit=(Y\f1) (Y\f3)] (Y\f) {};
	}
\end{tikzpicture}
\]
Suppose given a function $F\colon X\to Y$ and an element of $X$, thought of as a function $x\colon\singleton\to X$. Then evaluating $F$ at $x$ is given by the composite, $F(x)=x\cp F$.
\end{example}

\subsection{Preorders}%
\label{subsec.def_preorder}%
\index{preorder|(}

In \cref{sec.motivate_1}, we several times used the symbol $\leq$ to denote a sort of order. Here is a formal definition of what it means for a set to have an order.

\begin{definition}%
\label{def.preorder}%
\index{preorder}%
\index{relation!preorder|see {preorder}}
A \emph{preorder relation} on a set $X$ is a binary relation on $X$, here denoted with infix notation $\leq$, such that %
\index{preorder relation!as binary relation}
\begin{enumerate}[label=(\alph*)]
\item $x \le x$; and 
\item if $x \le y$ and $y \le z$, then $x \le z$.
\end{enumerate}
The first condition is called \emph{reflexivity}%
\index{reflexivity} and the second is called \emph{transitivity}%
\index{transitivity}. If $x\leq y$ and $y\leq x$, we write $x\cong y$ and say $x$ and $y$ are \emph{equivalent}. We call a pair $(X,\le)$ consisting of a set equipped with a preorder relation a \emph{preorder}.%
\end{definition}%
\index{equivalence relation!generated by a preorder}%
\index{infix notation}

\begin{remark}
  Observe that reflexivity and transitivity are familiar from \cref{def.equivalence_relation}: preorders are just equivalence relations without the symmetry condition.
\end{remark}%
\index{equivalence relation!as symmetric preorder}%
\index{symmetry}

\begin{example}[Discrete preorders]%
\label{ex.disc_preorder}%
\index{preorder!discrete}
Every set $X$ can be considered as a discrete preorder $(X,=)$. This means that the only order relations on $X$ are of the form $x \le x$; if $x \ne y$ then neither $x \le y$ or $y \le x$ hold.

We depict discrete preorders as simply a collection of points:
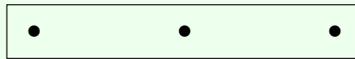
\[
\begin{tikzpicture}
  	\node (0) at (0,0) {$\bullet$};
  	\node (1) at (2,0) {$\bullet$};
  	\node (2) at (4,0) {$\bullet$};
  	\node[draw, fit=(0) (1) (2)] (A) {};
\end{tikzpicture}
\qedhere
\]
\end{example}

\begin{example}[Codiscrete preorders]%
\label{ex.codisc_preorder}%
\index{preorder!codiscrete}
From every set we may also construct its codiscrete preorder $(X,\le)$ by
equipping it with the total binary relation $X\times X \subseteq X \times X$.
This is a very trivial structure: it means that for \emph{all} $x$ and $y$ in
$X$ we have $x \le y$ (and hence also $y \le x$).
\end{example}

\begin{example}[Booleans]%
\label{ex.boolean_order}%
\index{booleans!as preorder}
  The booleans $\bb =\{\false,\true\}$ form a preorder with $\false \le \true$.
\[
\begin{tikzpicture}
  	\node (0) at (0,0) {$\false$};
  	\node (1) at (0,1) {$\true$};
  	\node[draw, fit=(0) (1)] (A) {};
	\draw[->] (0) to (1);
\end{tikzpicture}
\qedhere
\]
\end{example}

\begin{remark}[Partial orders are skeletal preorders] %
\label{rem.partial_orders}
\label{rem.skeletal_preorder}%
\index{preorder!skeletal}%
\index{preorder!partial order as}
A preorder is a \emph{partial order} if we additionally have that
\begin{enumerate}
  \item[(c)] $x \cong y$ implies $x =y$.  
\end{enumerate}

In category theory terminology, the requirement that $x \cong y$ implies $x =y$
is known as \emph{skeletality}, so partial orders are \emph{skeletal
preorders}.%
\index{skeleton}
For short, we also use the term \emph{poset}, a contraction of partially ordered
set.%
\index{poset|see {partial order}}%
\index{partial order}%
\index{partial order|seealso {preorder, skeletal}}    

The difference between preorders and partial orders is rather minor. A partial order already is a preorder, and every preorder can be made into a partial order by equating any two elements $x,y$ for which $x\cong y$, i.e.\ for which $x\leq y$ and $y\leq x$.

For example, any discrete preorder is already a partial order, while any
codiscrete preorder simply becomes the unique partial order on a one element
set.
\end{remark}

We have already introduced a few examples of preorders using Hasse diagrams. It
will be convenient to continue to do this, so let us be a bit more formal about
what we mean. First, we need to define a graph.

\begin{definition} %
\label{def.graph}%
\index{graph}%
\index{graph!vertex}%
\index{graph!arrow}
A \emph{graph} $G=(V,A,s,t)$ consists of a set $V$ whose elements are called
\emph{vertices}, a set $A$ whose elements are called \emph{arrows}, and two
functions $s,t\colon A \to V$ known as the \emph{source} and \emph{target}
functions respectively. Given $a \in A$ with $s(a)=v$ and $t(a)=w$, we say that
$a$ is an arrow from $v$ to $w$.

By a \emph{path} in $G$ we mean any sequence of
arrows such that the target of one arrow is the source of the next. This
includes sequences of length 1, which are just arrows $a\in A$ in $G$, and
sequences of length 0, which just start and end at the same vertex $v$, without
traversing any arrows. %
\index{graph!paths in}
\end{definition}

\begin{example}%
\label{ex.my_first_graph}
Here is a picture of a graph:
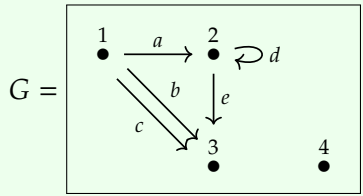
\[
G=\fbox{
\begin{tikzcd}[ampersand replacement=\&]
	\LMO{1}\ar[r, "a"]\ar[dr, shift left, "b"]\ar[dr, shift right, "c"']\&
	\LMO{2}\ar[d, "e"]\ar[loop right, "d"]\\
	\&\LMO{3}\&\LMO{4}
\end{tikzcd}
}
\]
It has $V=\{1,2,3,4\}$ and $A=\{a,b,c,d,e\}$. The source and target functions, $s,t\colon A\to V$ are given by the following partially-filled-in tables (see \cref{exc.understanding_graphs}):
\[
\begin{array}{l || l | l}
	\textbf{arrow }a&\textbf{source }s(a)\in V&\textbf{target }t(a)\in V\\\hline
	a&1&\?\\
	b&1&3\\
	c&\?&\?\\
	d&\?&\?\\
	e&\?&\?
\end{array}
\]

There is one path from 2 to 3, namely the arrow $e$ is a path of length 1. There are no paths from $4$ to 3, but there is one path from $4$ to $4$, namely the path of length 0. There are infinitely many paths $1\to 2$ because one can loop and loop and loop through $d$ as many times as one pleases.
\end{example}

\begin{exercise}%
\label{exc.understanding_graphs}
Fill in the table from \cref{ex.my_first_graph}.
\end{exercise}

\begin{remark}%
\label{rem.Hasse}%
\index{Hasse diagram}%
\index{preorder!presentation of}%
\index{presentation!of preorder}
From every graph we can get a preorder. Indeed, a Hasse diagram is a graph
$G=(V,A,s,t)$ that gives a \emph{presentation} of a preorder $(P,\leq)$. The
elements of $P$ are the vertices $V$ in $G$, and the order $\leq$ is given by $v\leq
w$ iff\footnote{\index{iff}The word `iff' is a common mathematical shorthand
for the phrase ``if and only if'', and we use it to connect two statements that
each imply the other, and hence are logically equivalent.} there is a path $v\to w$. For any vertex $v$, there is always a path $v\to v$,
and this translates into the reflexivity law from \cref{def.preorder}. The fact
that paths $u\to v$ and $v\to w$ can be concatenated to a path $u\to w$
translates into the transitivity law.
\end{remark}

\begin{exercise}%
\label{exc.graph_to_hasse}
	What preorder relation $(P,\leq)$ is depicted by the graph $G$ in \cref{ex.my_first_graph}? That is, what are the elements of $P$ and write down every pair $(p_1,p_2)$ for which $p_1\leq p_2$.
\end{exercise}

\begin{exercise}%
\label{exc.points_as_hasse}
Does a collection of points, like the one in \cref{ex.disc_preorder}, count as a Hasse diagram? 
\end{exercise}

\begin{exercise}%
\label{exc.order_parts_practice}
Let $X$ be the set of partitions of $\{\bullet,\circ,\ast\}$; it has five elements and an order by coarseness, as shown in the Hasse diagram \cref{eqn.parts_of_3}. Write down every pair $(x,y)$ of elements in $X$ such that $x\leq y$. There should be 12.
\end{exercise}

\begin{remark}
In \cref{rem.partial_orders} we discussed partial orders---preorders with the property that whenever two elements are equivalent, they are the same---and then said that this property is fairly inconsequential: any preorder can be converted to a partial order that's ``equivalent'' category-theoretically. A partial order is like a preorder with a fancy haircut: some mathematicians might not even notice it.

However, there are other types of preorders that are more special and noticeable. For example, a \emph{total order} has the following additional property:
\begin{enumerate}
\item[(d)] for all $x,y$, either $x\leq y$ or $y\leq x$.  
\end{enumerate}
\index{total order}
We say two elements $x,y$ of a preorder are \emph{comparable} if either $x\leq y$ or $y\leq x$, so a total order is a preorder where \emph{every} two elements are comparable.%
\index{comparable}
\end{remark}

\begin{exercise}%
\label{exc.discrete_preorder_comparable}
Is it correct to say that a discrete preorder is one where \emph{no} two elements are comparable?
\end{exercise}

\begin{example}[Natural numbers]%
\label{ex.orders_on_N}%
\index{natural numbers}
The natural numbers $\nn\coloneqq\{0,1,2,3,\ldots\}$ are a preorder with the order given by the usual size
ordering, e.g.\ $0\leq 1$ and $5\leq 100$. This is a total order: either $m\leq n$ or $n\leq m$ for all $m,n$. One can see that its Hasse diagram looks like a line:
\[
\begin{tikzcd}
	\LMO{0}\ar[r]&\LMO{1}\ar[r]&\LMO{2}\ar[r]&\LMO{3}\ar[r]&\cdots
\end{tikzcd}
\]
What made \cref{eqn.parts_of_3} not look like a line is that there are non-comparable elements $a$ and $b$---namely all those in the middle row---which satisfy neither $a\leq b$ nor $b\leq a$.

Note that for any set $S$, there are many different ways of assigning an order to $S$.
Indeed, for the set $\nn$, we could also use the discrete ordering: only write $n\leq m$ if $n=m$. Another ordering is the reverse ordering, like $5\leq 3$ and $3\leq 2$, like how golf is scored (5 is worse than 3).

Yet another ordering on $\nn$ is given by division: we say that $n \le m$ if $n$ divides into $m$ without remainder.
In this ordering $2 \le 4$, for example, but $2\not\leq 3$, since there is a remainder when $2$
is divided into $3$.
\end{example}%
\index{divides relation!as preorder}%
\index{relation!divides|see {divides relation}}

\begin{exercise}%
\label{exc.total_order1}
Write down the numbers $1,2,\ldots,10$ and draw an arrow $a\to b$ if $a$ divides perfectly into $b$. Is it a total order?
\end{exercise}

\begin{example}[Real numbers]%
\index{real numbers}
  The real numbers $\rr$ also form a preorder with the ``usual ordering'', e.g.\ $-500\leq -499\leq 0\leq \sqrt{2}\leq 100/3$.
\end{example}

\begin{exercise}%
\label{exc.total_order2}%
\index{order!total}
Is the usual $\leq$ ordering on the set $\rr$ of real numbers a total order?
\end{exercise}

\begin{example}[Partition from preorder]%
\label{ex:part_from_preorder}
\index{partition!from preorder}%
\index{equivalence relation!and partition}
Given a preorder, i.e.\ a pre-ordered set $(P,\leq)$, we defined the notion of equivalence of elements, denoted $x\cong y$, to mean $x\leq y$ and $y\leq x$. This is an equivalence relation, so it induces a partition on $P$. (The phrase ``A induces B'' means that we have an automatic way to turn an A into a B. In this case, we're saying that we have an automatic way to turn equivalence relations into partitions, which we do; see \cref{prop.equivalence_partition}.)

For example, the preorder whose Hasse diagram is drawn on the left corresponds to the partition drawn on the right.%
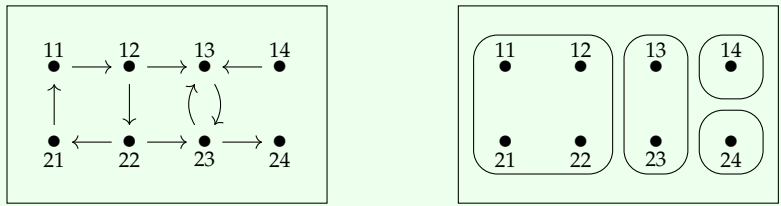
\index{Hasse diagram}
\[

\]
\end{example}

\begin{example}[Power set]%
\label{ex.powerset}%
\index{power set}
Given a set $X$, the set of subsets of $X$ is known as the \emph{power set} of
$X$; we denote it $\powset(X)$. The
power set can naturally be given an order by inclusion of subsets (and from now
on, whenever we speak of the power set as an ordered set, this is the order we
mean).

For example, taking $X =\{0,1,2\}$, we depict $\powset(X)$ as%
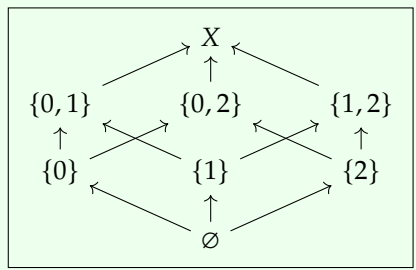
\index{Hasse diagram}
\[
%
\]
See the cube? The Hasse diagram for the power set of a finite set, say
$\powset\{1,2,\ldots,n\}$,\footnote{Note that we omit the parentheses here,
  writing $\powset X$ instead of $\powset(X)$; throughout this book we will
omit parentheses if we judge the presentation is cleaner and it is unlikely to cause confusion.} 
 always looks like a cube of dimension $n$.
\end{example}

\begin{exercise}%
\label{exc.powerset_Hasse}%
\index{Hasse diagram}
Draw the Hasse diagrams for $\powset(\varnothing)$, $\powset\{1\}$, and $\powset\{1,2\}$.
\end{exercise}

\begin{example}[Partitions]%
\label{ex.partitions}%
\index{preorder!of partitions}
We talked about getting a partition from a preorder; now let's think about how we might order the set $\prt{A}$ of \emph{all partitions} of $A$, for some set $A$. In fact, we have done this before in \cref{eqn.parts_of_3}.
Namely, we order on partitions by fineness: a partition $P$ is \emph{finer} than a
partition $Q$ if, for every part $p\in P$ there is a part $q\in Q$ such that $A_p\ss A_q$. We could also say that $Q$ is \emph{coarser} than $P$.

Recall from \cref{ex.partition_and_surjections} that partitions on $A$ can be thought of as surjective functions out of $A$. Then $f\colon A\surj P$ is finer than $g\colon A\surj Q$ if there is a function $h\colon P\to Q$ such that $f\cp h=g$.
\end{example}

\begin{exercise}%
\label{exc.partitions_and_functions}
For any set $S$ there is a coarsest partition, having just one part. What surjective function does it correspond to?

There is also a finest partition, where everything is in its own partition. What surjective function does it correspond to?
\end{exercise}

\begin{example}[Upper sets]%
\index{upper set}%
\label{ex.upper_set}
  Given a preorder $(P,\le)$, an \emph{upper set} in $P$ is a subset $U$ of $P$ satisfying the condition that
  if $p \in U$ and $p \le q$, then $q \in U$. ``If $p$ is an element then so is anything bigger.''
  Write $\upset(P)$ for the set of
  upper sets in $P$. We can give the set $\upset$ an order by letting $U \le V$ if $U$ is
  contained in $V$.

  For example, if $(\bb,\leq)$ is the booleans (\cref{ex.boolean_order}), then its preorder of uppersets $\upset(\bb)$ is%
\index{Hasse diagram}
\[
\begin{tikzpicture}
  \node (0) at (0,0) {$\varnothing$};
	\node (1) at (0,1) {$\{\true\}$};
	\node (2) at (0,2) {$\{\true,\false\}$};
  \node[draw, fit=(0) (1) (2)] (A) {};
	\draw[->] (0) to (1);
	\draw[->] (1) to (2);
\end{tikzpicture}
\]
The subset $\{\false\}\ss\bb$ is not an upper set, because $\false\leq\true$ and $\true \notin \{\false\}$.
\end{example}

\begin{exercise}%
\label{exc.uppersets_on_discrete}
  Prove that the preorder of upper sets on a discrete preorder (see \cref{ex.disc_preorder}) on a set $X$ is simply the power
  set $\powset(X)$.%
\index{power set}
\end{exercise}

\begin{example}[Product preorder]%
\index{product!of preorders}%
\label{ex.product_preorder}
  Given preorders $(P,\le)$ and $(Q,\le)$, we may define a preorder structure on
  the product set $P\times Q$ by setting $(p,q) \le (p',q')$ if and only if $p
  \le p'$ and $q \le q'$. We call this the \emph{product preorder}. This is a basic
  example of a more general construction known as the product of categories.
\end{example}

\begin{exercise}%
\label{exc.product_preorder}
Draw the Hasse diagram for the product of the two preorders drawn below:%
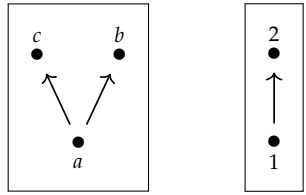
\index{Hasse diagram}
\[
\boxCD{
\begin{tikzcd}[column sep=0, ampersand replacement=\&]
	\LMO{c}\&\&\LMO{b}\\
	\&\LMO[under]{a}\ar[ul]\ar[ur]
\end{tikzcd}
}
\hspace{.5in}
\boxCD{
\begin{tikzcd}
	\LMO{2}\\
	\LMO[under]{1}\ar[u]
\end{tikzcd}
}
\]
For bonus points, compute the upper set preorder on the result.
\end{exercise}

\begin{example}[Opposite preorder]%
\label{ex.opposite}%
\index{opposite!preorder}
  Given a preorder $(P,\le)$, we may define the opposite preorder $(P,\leq\op)$ to have the same set
  of elements, but with $p\leq\op q$ if and only if $q \le p$.
\end{example}

\index{preorder|)}
\subsection{Monotone maps}%
\index{map!monotone|(}%
\index{map!order-preserving|see {map, monotone}}

We have said that the categorical perspective emphasizes relationships between
things. For example, a preorder is a setting---or world---in which we have one sort of
relationship, $\leq$, and any two objects may be, or may not be, so-related.
Jumping up a level, the categorical perspective emphasizes that preorders
themselves---each a miniature world composed of many relationships---can be related to one another.%
\index{level shift}

The most important sort of relationship between
preorders is called a \emph{monotone map}. These are functions that preserve preorder
relations---in some sense mappings that respect $\leq$---and are hence considered the right notion of \emph{structure-preserving map} for preorders.%
\index{map!structure preserving}

\begin{definition}%
\index{map!monotone}%
\index{monotone map|seealso {map, monotone}}%
\index{map!monotone|seealso {monotone map}}
A \emph{monotone map} between
preorders $(A,\leq_A)$ and $(B,\leq_B)$ is a function $f\colon A \to B$ such that, for all elements $x,y\in A$, if $x \leq_A
y$ then $f(x) \leq_B f(y)$.
\end{definition}

A monotone map $A\to B$ between two preorders associates to each element of preorder $A$ an element of the preorder $B$. We depict this by drawing a dotted arrow from each
element $x\in A$ to its image $f(x)\in B$. Note that the order must be preserved in order to count as a valid monotone map, so if element $x$ is above
element $y$ in the lefthand preorder $A$, then the image $f(x)$ will be
above the image $f(y)$ in the righthand preorder.
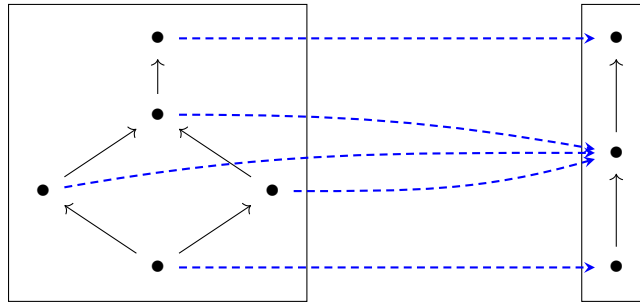
\[
\begin{tikzpicture}[x=.6in, y=.4in, inner sep=5pt]
  	\node (a) at (0,4) {$\bullet$};
  	\node (b) at (0,3) {$\bullet$};
  	\node (c) at (-1,2) {$\bullet$};
  	\node (d) at (1,2) {$\bullet$};
  	\node (e) at (0,1) {$\bullet$};
  	\node[draw, fit=(a) (b) (c) (d) (e)] (A) {};
  	\draw[->] (b) to (a);
  	\draw[->] (c) to (b);
  	\draw[->] (d) to (b);
  	\draw[->] (e) to (c);
  	\draw[->] (e) to (d);
  	\node (B0) at (4,4) {$\bullet$};
  	\node (B1) at (4,2.5) {$\bullet$};
  	\node (B2) at (4,1) {$\bullet$};
  	\draw[->] (B2) to (B1);
  	\draw[->] (B1) to (B0);
  	\node[draw, fit=(B0) (B1) (B2)] (B) {};
	\begin{scope}[mapsto]
    \draw (a) -- (B0);
  	\draw (b) to[out=0,in=170] (B1);
  	\draw (c) to[out=10,in=180] (B1);
  	\draw (d) to[out=0,in=196] (B1);
  	\draw (e) -- (B2);
	\end{scope}
\end{tikzpicture}
\]

\begin{example}
Let $\bb$ and $\nn$ be the preorders of booleans from \cref{ex.boolean_order} and $\nn$ the preorder of natural numbers from \cref{ex.orders_on_N}. The map $\bb \to \nn$ sending $\false$ to $17$ and $\true$ to $24$ is a monotone map, because it preserves order.
\[
\begin{tikzcd}[column sep=small]
	&&&&
	\false\ar[r]\ar[dl, dashed, blue]&
	\true\ar[drr, dashed, blue]\\
  0\ar[r]&
  1\ar[r]&
  \cdots\ar[r]&
  17\ar[r]&
  18\ar[r]&
  \cdots\ar[r]&
  23\ar[r]&
  24\ar[r]&
  \cdots
\end{tikzcd}
\]
\end{example}

\begin{example}[The tree of life]%
\index{tree of life}
Consider the set of all animal classifications, for example `tiger', `mammal',
`sapiens', `carnivore', etc.. These are ordered by specificity: since `tiger' is a type
of `mammal', we write tiger $\le$ mammal. The result is a preorder, which in fact forms a tree, often called the tree of life. At the top of the following diagram we see a small part of it:
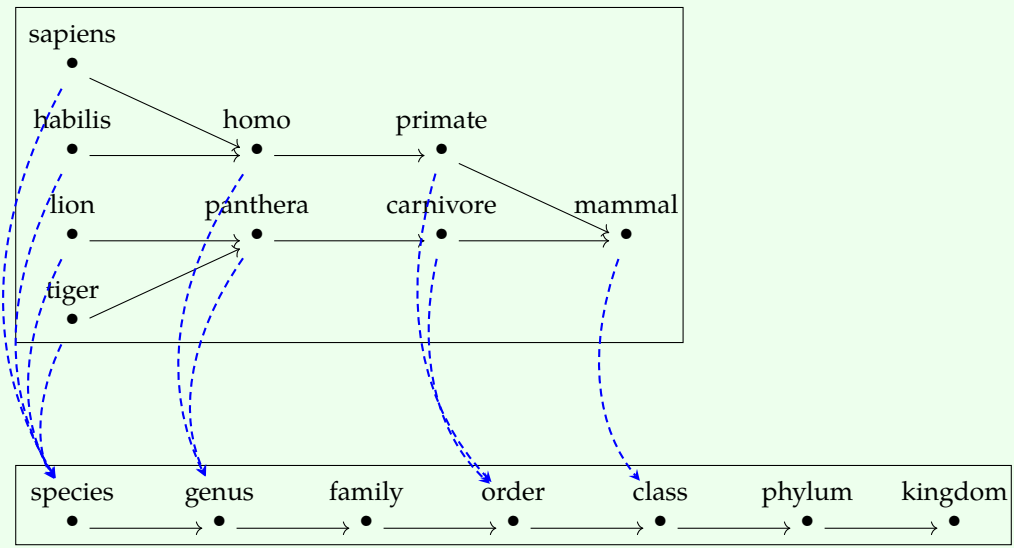
\[

\]
At the bottom we see the hierarchical structure as a preorder. The dashed arrows show a monotone map, call it $F$, from the classifications to the hierarchy. It is monotone because it preserves order: whenever there is a path $x\to y$ upstairs, there is a path $F(x)\to F(y)$ downstairs.
\end{example}

\begin{example}%
\label{ex.card_as_monotone}
  Given a finite set $X$, recall the power set $\powset(X)$ and its natural
  order relation from \cref{ex.powerset}. The map $\lvert\cdot\rvert\colon \powset(X) \to \nn$ sending
  each subset $S$ to its number of elements $\lvert S\rvert$, also called its
  \emph{cardinality}%
\index{cardinality}, is a monotone map.%
\index{power set}
\end{example}

\begin{exercise}%
\label{exc.card_as_monotone}
Let $X=\{0,1,2\}$.
\begin{enumerate}
	\item Draw the Hasse diagram for $\powset(X)$.
	\item Draw the Hasse diagram for the preorder $0\leq 1\leq 2\leq 3$.
	\item Draw the cardinality map $\lvert \cdot \rvert$ from \cref{ex.card_as_monotone} as dashed lines between them.
\qedhere
\end{enumerate}
\end{exercise}

\begin{example}%
\label{ex.upset_monotone_powset}
  Recall the notion of upper set from \cref{ex.upper_set}. Given a preorder
  $(P,\le)$, the map $\upset(P) \to \powset(P)$ sending each upper set of
  $(P,\le)$ to itself---considered as a subset of $P$---is a monotone map.
\end{example}

\begin{exercise}%
\label{exc.upset_monotone_powset}
Consider the preorder $\bb$. The Hasse diagram for $\upset(\bb)$ was drawn in \cref{ex.upper_set}, and you drew the Hasse diagram for $\powset(\bb)$ in \cref{exc.powerset_Hasse}. Now draw the monotone map between them, as described in \cref{ex.upset_monotone_powset}.
\end{exercise}

\begin{exercise}%
\label{exc.upper_set}
  Let $(P,\leq)$ be a preorder, and recall the notion of opposite preorder from \cref{ex.opposite}.
\begin{enumerate}
	\item Show that the set $\uparrow p\coloneqq\{p'\in P\mid p\leq p'\}$ is an upper set, for any $p\in P$.
	\item Show that this construction defines a monotone map $\uparrow\colon P\op\to\upset(P)$.
	\item Show that if $p \le p'$ in $P$ if and only if $\upclose(p')
	\subseteq \upclose(p)$. 
	\item Draw a picture of the map $\uparrow$ in the case where $P$ is the preorder $(b\geq a\leq c)$ from \cref{ex.product_preorder}.
\end{enumerate}
This is known as the \emph{Yoneda lemma} for preorders.%
\index{Yoneda lemma!for preorders} The if and only if
condition proved in part 3 implies that, up to equivalence, to know an element
is the same as knowing its upper set---that is, knowing its web of
relationships with the other elements of the preorder. The general Yoneda lemma
is a powerful tool in category theory, and a fascinating philosophical idea
besides.
\end{exercise}

\begin{exercise}%
\label{exc.monotone_from_discrete}
As you yourself well know, a monotone map $f\colon\mbox{$(P,\leq_P)$}\to\mbox{$(Q,\leq_Q)$}$ consists of a function $f\colon P\to Q$ that satisfies a ``monotonicity'' property. Show that when $(P,\leq_P)$ is a discrete preorder, then \emph{every} function $P\to Q$ satisfies the monotonicity property, regardless of the order $\leq_Q$.%
\index{preorder!discrete}
\end{exercise}

\begin{example}%
\label{ex.ptnmonotone}%
\index{partition}
Recall from \cref{ex.partitions} that given a set $X$ we define $\prt X$ to be
the set of partitions on $X$, and that a partition may be defined using a
surjective function $s\colon X \surj P$ for some set $P$.

Any surjective function $f\colon X \surj Y$ induces a monotone map
$f^\ast\colon \prt Y \to \prt X$, going ``backwards.'' It is defined by sending
a partition $s\colon Y \surj P$ to the composite $f\cp s\colon X \surj P$.%
\tablefootnote{We will later see that any function $f\colon X\to Y$, not necessarily surjective, induces a monotone map $f^*\colon\prt Y\to \prt X$, but it involves an extra step. See \cref{subsec.back_to_parts}.}
\end{example}

\begin{exercise}%
\label{exc.f^*_partitions}
Choose two sets $X$ and $Y$ with at least three elements each and choose a
surjective, non-identity function $f\colon X\surj Y$ between them. Write down two
different partitions $P$ and $Q$ of $Y$, and then find $f^*(P)$ and $f^*(Q)$.
\end{exercise}

The following proposition, \cref{prop.id_and_comp_monotones}, is straightforward to check. Recall the definition of the identity function from \cref{ex.identity} and the definition of composition from \cref{def.composite_fn}.%
\index{identity!function|seealso {function, identity}}

\begin{proposition}%
\label{prop.id_and_comp_monotones}%
\index{category!of preorders}
For any preorder $(P,\leq_P)$, the identity function is monotone.

If $(Q,\leq_Q)$ and $(R,\leq_R)$ are preorders and $f\colon P\to Q$ and $g\colon Q\to R$ are monotone, then $(f\cp g)\colon P\to R$ is also monotone.
\end{proposition}

\begin{exercise}%
\label{exc.check_id_comp_monotone}
Check the two claims made in \cref{prop.id_and_comp_monotones}.
\end{exercise}

\begin{example}%
\label{ex.dagger_preorder}%
\index{preorder!dagger|see {dagger}}%
\index{dagger}
Recall again the definition of opposite preorder from \cref{ex.opposite}. The
identity function $\id_P\colon P \to P$ is a monotone map $(P,\le) \to (P,
\leq\op)$ if and only if for all $p,q \in P$ we have $q \le p$ whenever $p\le
q$. For historical reasons connected to linear algebra, when this is true, we
call $(P,\le)$ a \emph{dagger preorder}. 

But in fact, we have seen dagger preorders before in another guise. Indeed, if $(P,\leq)$ is a dagger preorder, then the relation $\leq$ is symmetric: $p \le q$ if and only if
$q \le p$, and it is also reflexive and transitive by definition of preorder. So in fact $\leq$ is an equivalence relation (\cref{def.equivalence_relation}).
\end{example}%
\index{symmetry!and dagger}

\begin{exercise}%
\label{exc.skeletal_dagger_preorder}%
\index{preorder!skeletal}
Recall the notion of skeletal preorders (\cref{rem.skeletal_preorder}) and discrete preorders
(\cref{ex.disc_preorder}). Show that a skeletal dagger preorder is just a discrete preorder, and hence can be identified with a set.%
\end{exercise}
\begin{remark}
  We say that an $A$ ``can be identified with'' a $B$ when any $A$ gives us a unique $B$ and any $B$ gives us a unique $A$, and both round-trips---from an $A$ to a $B$ and back to an $A$, or from a $B$ to an $A$ and back to a $B$---return us where we started. For example, any discrete preorder $(P,\leq)$ has an underlying set $P$, and any set $P$ can be made into a discrete preorder ($p_1\leq p_2$ iff $p_1=p_2$), and the round-trips return us where we started. So what's the difference? It's like the notion of \emph{object-permanence} from child development jargon: we can recognize ``the same chair, just moved from one room to another.'' A chair in the room of sets can be moved to a chair in the room of preorders. The lighting is different but the chair is the same.

Eventually, we will be able to understand this notion in terms of \emph{equivalence of categories}, which are related to isomorphisms, which we will explore next in \cref{def.preorder_isomorphism}.
\end{remark}

\begin{definition}%
\label{def.preorder_isomorphism}%
\index{isomorphism!of preorders}
Let $(P,\leq_P)$ and $(Q,\leq_Q)$ be preorders. A monotone function $f\colon P\to Q$ is called an \emph{isomorphism} if there exists a monotone function $g\colon Q\to P$ such that $f\cp g=\id_P$ and $g\cp f=\id_Q$. This means that for any $p\in P$ and $q\in Q$, we have
\[
	p= g(f(p))
	\qquad\text{and}\qquad
	q=f(g(q)).
\]
We refer to $g$ as the \emph{inverse} of $f$, and vice versa: $f$ is the inverse of $g$.

If there is an isomorphism $P\to Q$, we say that $P$ and $Q$ are \emph{isomorphic}.
\end{definition}

An isomorphism between preorders is basically just a relabeling of the elements.

\begin{example}
Here are the Hasse diagrams for three preorders $P$, $Q$, and $R$, all of which are isomorphic:
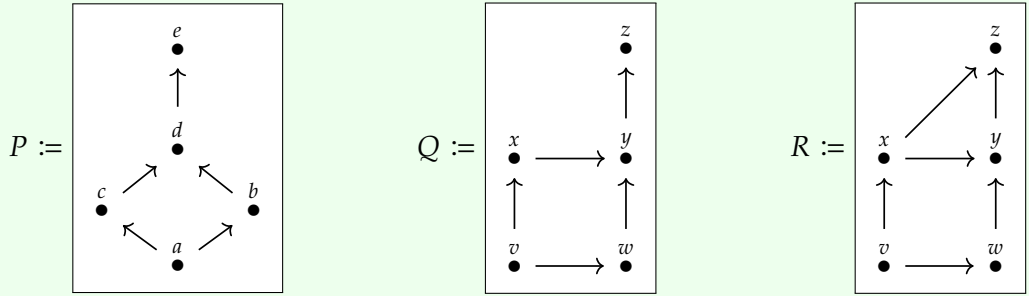
\[
P\coloneqq\boxCD{
\begin{tikzcd}[ampersand replacement=\&, row sep=0, column sep=small]
	\&\LMO{e}\\[15pt]
	\&\LMO{d}\ar[u]\\
	\LMO{c}\ar[ur]\&\&\LMO{b}\ar[ul]\\
	\&\LMO{a}\ar[ur]\ar[ul]
\end{tikzcd}
}
\hspace{.7in}
Q\coloneqq\boxCD{
\begin{tikzcd}[ampersand replacement=\&]
	\&\LMO{z}\\
	\LMO{x}\ar[r]\&\LMO{y}\ar[u]\\
	\LMO{v}\ar[u]\ar[r]\&\LMO{w}\ar[u]
\end{tikzcd}
}
\hspace{.7in}
R\coloneqq\boxCD{
\begin{tikzcd}[ampersand replacement=\&]
	\&\LMO{z}\\
	\LMO{x}\ar[ur]\ar[r]\&\LMO{y}\ar[u]\\
	\LMO{v}\ar[u]\ar[r]\&\LMO{w}\ar[u]
\end{tikzcd}
}
\]
The map $f\colon P\to Q$ given by $f(a)=v$, $f(b)=w$, $f(c)=x$, $f(d)=y$, and $f(e)=z$ has an inverse.

In fact $Q$ and $R$ are the same preorder. One may be confused by the fact that there is an arrow $x\to z$ in the Hasse diagram for $R$ and not one in $Q$, but in fact this arrow is superfluous. By the transitivity property of preorders (\cref{def.preorder}), since $x\leq y$ and $y\leq z$, we must have $x\leq z$, whether it is drawn or not. Similarly, we could have drawn an arrow $v\to y$ in either $Q$ or $R$ and it would not have changed the preorder.
\end{example}

Recall the preorder $\bb=\{\false,\true\}$, where $\false\leq\true$. As simple as this preorder is, it is also one of the most important.%
\index{booleans!as preorder}

\begin{exercise}%
\label{exc.weird_Phi}
  Show that the map $\Phi$ from \cref{subsec.first_look_gen}, which was roughly
  given by `Is $\bullet$ connected to $\ast$?'  is a monotone map
  $\prt{\{\ast,\bullet,\circ\}} \to\bb$; see also \cref{eqn.parts_of_3}.
\end{exercise}

\begin{proposition}%
\label{prop.upperset_presheaf}%
\index{upper set}
  Let $P$ be a preorder. Monotone maps $P \to \bb$ are in one-to-one
  correspondence with upper sets of $P$.
\end{proposition}
\begin{proof}
  Let $f\colon P \to \bb$ be a monotone map. We will show that the subset $f\inv(\true)\ss P$ is an upper set. Suppose
$p \in f\inv(\true)$ and $p \le q$; then $\true = f(p) \le f(q)$. But in $\BB$, if $\true\leq f(q)$ then $\true=f(q)$. This implies $q \in f\inv(\true)$ and thus shows that
$f\inv(\true)$ is an upper set.

Conversely, if $U$ is an upper set in $P$, define $f_U\colon P \to \bb$ such
that $f_U(p) =\true$ when $p \in U$, and $f_U(p) = \false$ when $p\not\in U$. This is a monotone
map, because if $p \le q$, then either $p \in U$, so $q \in U$ and $f(p) = \true= f(q)$,
or $p \notin U$, so $f(p) =\false \le f(q)$.

These two constructions are mutually inverse, and hence prove the proposition.
\end{proof}

\begin{exercise}[Pullback map]%
\index{pullback!along a map}%
\label{exc.pullback_upset}
  Let $P$ and $Q$ be preorders, and $f\colon P \to Q$ be a monotone map. Then we
  can define a monotone map $f^\ast\colon\upset(Q) \to \upset(P)$ sending an
  upper set $U \subseteq Q$ to the upper set $f\inv (U) \subseteq P$. We call this the
  \emph{pullback along $f$}.

  Viewing upper sets as a monotone maps to $\bb$ as in \cref{prop.upperset_presheaf}, the pullback can be understood in terms of composition. Indeed, show
  that the $f^\ast$ is defined by taking $u\colon Q \to \bb$ to $(f\cp u)
  \colon P \to \bb$. 
\end{exercise}

\index{map!monotone|)}

\section{Meets and joins}%
\label{sec.meets_joins}%
\index{join|(}%
\index{meet|(}

As we have said, a preorder is a set $P$ endowed with an order $\leq$ relating the
elements. With respect to this order, certain elements of $P$ may have distinctive characterizations, either absolutely or in relation to other elements. We have discussed joins before, but we discuss them again
now that we have built up some formalism.

\subsection{Definition and basic examples}%
\index{real numbers!as preorder}
Consider the preorder $(\RR,\leq)$ of real numbers ordered in the usual way. The
subset $\NN\ss\RR$ has many lower bounds, namely $-1.5$ is a lower bound: every element of $\nn$ is bigger than $-1.5$. But within all lower bounds for $\nn\ss\rr$, one is distinctive: a \emph{greatest lower bound}---also called a \emph{meet}---namely 0. It is a lower bound, and there is no lower bound for $\nn$ that is above it. However, the set $\nn\ss\rr$ has no upper bound, and certainly no least upper bound---which would be called a \emph{join}. On the other hand, the set
\[
  \left\{\frac{1}{n+1}\;\;\middle|\;\; n\in\NN\right\}
  =
  \left\{\,1,\,\frac12,\,\frac13,\,\frac14,\,\ldots\,\right\}\ss\rr
\]
has both a greatest lower bound (meet), namely 0, and a least upper bound (join), namely 1.

These notions will have correlates in category theory, called limits and
colimits, which we will discuss in \cref{chap.databases}. More generally, we
say these distinctive characterizations are \emph{universal properties}, since,
for example, a greatest lower bound is greatest among \emph{all} lower bounds.
For now, however, we simply want to make the definition of greatest lower
bounds and least upper bounds, called meets and joins, precise.%
\index{universal property}%
\index{least upper bound}%
\index{greatest lower bound}

\begin{exercise}%
\label{exc.0_glb}
\begin{enumerate}
	\item Why is $0$ a lower bound for $\{\frac{1}{n+1}\mid n\in\NN\}\ss\rr$?
	\item Why is $0$ a \emph{greatest} lower bound (meet)?
\qedhere
\end{enumerate}
\end{exercise}

\begin{definition}%
\label{def.meets_joins}%
\index{meet}%
\index{join}
  Let $(P,\le)$ be a preorder, and let $A\ss P$ be a subset. We say that an
  element $p \in P$ is a \emph{meet} of $A$ if 
  \begin{enumerate}[label=(\alph*)]
    \item for all $a \in A$, we have $p \le a$, and 
    \item for all $q$ such that $q \le a$ for all $a \in A$, we have that $q \le p$. 
  \end{enumerate} 
  We write $p = \bigwedge A$, $p=\bigwedge_{a\in A}a$, or, if the dummy variable $a$ is clear from context, just $p=\bigwedge_A a$. If $A$ just consists
  of two elements, say $A=\{a, b\}$, we can denote $\bigwedge A$ simply by $a
  \wedge b$.%

  Similarly, we say that $p$ is a \emph{join} of $A$ if
  \begin{enumerate}[label=(\alph*)]
  	\item for all $a \in A$ we have $a \le p$, and
		\item for all $q$ such that $a \le q$ for all $a
\in A$, we have that $p \le q$.
	\end{enumerate}
	We write $p = \bigvee A$ or $p=\bigvee_{a\in A}a$, or when $A =\{a,b\}$ we may simply write $p=a\vee b$.
\end{definition}

\begin{remark}%
\label{rem.meet_abuse}
In \cref{def.meets_joins}, we committed a seemingly egregious abuse of notation. We will see next in \cref{ex.two_meets} that there could be two different meets of $A\ss P$, say $p=\bigwedge A$ and $q=\bigwedge A$ with $p\neq q$, which does not make sense if $p\neq q$!

But in fact, as we use the symbol $\bigwedge A$, this abuse won't matter
because any two meets $p,q$ are automatically isomorphic: the very definition
of meet forces both $p\leq q$ and $q\leq p$, and thus we have $p\cong q$. So
for any $x\in P$, we have $p\leq x$ iff $q\leq x$ and $x\leq p$ iff $x\leq q$.
Thus as long as we are only interested in elements of $P$ based on their
relationships to other elements (and in category theory, this is the case: we
should only care about things based on how they interact with other things,
rather than on some sort of ``internal essence''), the distinction between $p$
and $q$ will never matter.
    
This foreshadows a major theme of---as well as standard abuse of notation
in---category theory, where any two things defined by the same universal
property are automatically equivalent in a way known as `unique up to unique
isomorphism'%
\index{unique up to unique isomorphism}; this means that we
generally do not run into trouble if we pretend they are equal. We'll pick up
this theme of `the' vs `a' again in \cref{rem.the_vs_a}.
\end{remark}

\begin{example}[Meets or joins may not exist]
Note that, in an arbitrary preorder $(P,\leq)$, a subset $A$ need not have a
meet or a join. Consider the three element set $P=\{p,q,r\}$ with the discrete ordering. The set $A=\{p,q\}$ does not have a join in $P$ because if $x$ was a join, we would need $p\leq x$ and $q\leq x$, and there is no such element $x$.
\end{example}

\begin{example}[Multiple meets or joins may exist]%
\label{ex.two_meets}
It may also be the case that a subset $A$ has more than one meet or
join. Here is an example.
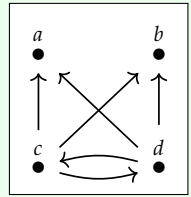
\[
\boxCD{
\begin{tikzcd}[ampersand replacement=\&, row sep=0, column sep=30pt]
  \LMO{a}\&\LMO{b}\\[20pt]
  \LMO{c} \ar[u]\ar[ur] \ar[bend right=15, r]\&\LMO{d}\ar[ul]\ar[u] \ar[bend
	right=15, l]
\end{tikzcd}
}
\]
Let $A$ be the subset $\{a,b\}$ in the preorder specified by this Hasse
diagram. Then both $c$ and $d$ are meets of $A$: any element less than both
$a$ and $b$ is also less than $c$, and also less than $d$. Note that, as in \cref{rem.meet_abuse}, $c\leq d$ and $d\leq c$, so $c \cong d$. Such will always the case when there is more than one meet: any two meets of the same subset will be isomorphic.
\end{example}

\begin{exercise}%
\label{exc.wedge_element}
Let $(P,\leq)$ be a preorder and $p\in P$ an element. Consider the set $A=\{p\}$ with one element.
\begin{enumerate}
	\item Show that $\bigwedge A\cong p$.
	\item Show that if $P$ is in fact a partial order, then $\bigwedge A=p$.
	\item Are the analogous facts true when $\bigwedge$ is replaced by $\bigvee$?
\qedhere
\end{enumerate}
\end{exercise}

\begin{example}
In any partial order $P$, we have $p \vee p =p \wedge p = p$. The reason is that
our notation says $p\vee p$ means $\bigvee\{p,p\}$. But $\{p,p\}=\{p\}$ (see
\cref{sec.sets_and_rels}), so $p\vee p=p$ by \cref{exc.wedge_element}.
\end{example}

\begin{example}%
\index{intersection!as meet}%
\index{union!as join}
  In a power set $\powset(X)$, the meet of a collection of subsets, say $A,B\ss X$ is their intersection $A\wedge B=A\cap B$,
  while the join is their union, $A\vee B=A\cup B$.
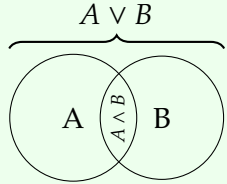
  \[
  \begin{tikzpicture}[decoration={brace, amplitude=5pt}]
  	\node[circle, draw, inner sep=13pt] (a) {A};
  	\node[circle, draw, inner sep=13pt, right=10pt of a.center] (b) {B};
		\coordinate (helper) at ($(a.north)+(0,2pt)$);
		\draw[decorate, thick] (a.west|-helper) to node[above=5pt] {$A\vee B$} (b.east|-helper);
		\node[rotate=90, font=\tiny] at ($(a)!.5!(b)$) {$A\wedge B$};
  \end{tikzpicture}
  \]
  Perhaps this justifies the terminology: the joining of two sets is their union, the meeting of two sets is their intersection.
\end{example}

\begin{example} %
\label{ex.bool_meet_join}%
\index{booleans!meets and joins in}
  In the booleans $\bb =\{\false,\true\}$ (\cref{ex.boolean_order}), the meet of any two elements is given by
  AND and the join of any two elements is given by OR (recall \cref{exc.boolean_vee_practice}).
\end{example}

\begin{example} %
\label{ex.join_is_supremum}%
\index{supremum}%
\index{infimum}
In a total order, the meet of a set is its infimum, while the join of a set is
its supremum. Note that $\bb$ is a total order, and this generalizes
\cref{ex.bool_meet_join}.
\end{example}

\begin{exercise}%
\label{exc.division_meet}%
\index{divides relation}%
\index{least common multiple}%
\index{greatest common divisor}
Recall the division ordering on $\nn$ from \cref{ex.orders_on_N}: we write $n
\vert m$ if $n$ divides perfectly into $m$. The meet of any two numbers in this
preorder has a common name, that you may have learned when you were around 10
years old; what is it? Similarly the join of any two numbers has a common name;
what is it?
\end{exercise}

\begin{proposition}%
\label{prop.containment_meet_join}
Suppose $(P,\leq)$ is a preorder and $A\ss B\ss P$ are subsets that have meets. Then $\bigwedge B\leq \bigwedge A$.

Similarly, if $A$ and $B$ have joins, then $\bigvee A\leq \bigvee B$.
\end{proposition}
\begin{proof}
Let $m=\bigwedge A$ and $n=\bigwedge B$. Then for any $a\in A$ we also have $a\in B$, so $n\leq a$ because $n$ is a lower bound for $B$. Thus $n$ is also a lower bound for $A$ and hence $n\leq m$, because $m$ is $A$'s greatest lower bound. The second claim is proved similarly.
\end{proof}

\index{meet|)}%
\index{join|)}

\subsection{Back to observations and generative effects}

In the thesis \cite{Adam:2017a}, Adam thinks of monotone maps as observations. A monotone map $\Phi\colon P\to Q$ is a phenomenon (we might say ``feature'') of $P$ as observed by $Q$. He defines the \emph{generative effect} of such a map $\Phi$ to be its failure to preserve joins (or more generally, for categories, its failure to preserve colimits).%
\index{generative effect}

\begin{definition}%
\index{meets!preservation of}
  We say that a monotone map $f\colon P \to Q$ \emph{preserves meets} if
$f(a \wedge b) \cong f(a) \wedge f(b)$ for all $a,b\in P$. We similarly say $f$
\emph{preserves joins} if $f(a \vee b) \cong f(a) \vee f(b)$ for all $a,b\in P$.
\end{definition}

\begin{definition}%
\index{generative effect}%
\label{def.gen_effect}
  We say that a monotone map $f\colon P \to Q$ \emph{has a generative
  effect} if there exist elements $a,b\in P$ such that
  \[
    f(a) \vee f(b) \ncong f(a \vee b).
  \]
\end{definition}

In \cref{def.gen_effect}, if we think of $\Phi$ as a observation or measurement of the systems $a$ and $b$, then
the left hand side $f(a) \vee f(b)$ may be interpreted as the combination of the observation of $a$ with
the observation of $b$. On the other hand, the right hand side $f(a \vee b)$ is the observation of the
combined system $a \vee b$. The inequality implies that we see something when we observe the combined system that we could not expect by merely combining our observations of the pieces. That is, that there are generative effects from the
interconnection of the two systems.

\begin{exercise}%
\label{exc.more_stuff}
In \cref{def.gen_effect}, we defined generativity of $f$ as the inequality $f(a \vee b)\neq f(a)\vee f(b)$, but in the subsequent text we seemed to imply there would be not just a difference, but \emph{more stuff} in $f(a \vee b)$ than in $f(a)\vee f(b)$.

Prove that for any monotone map $f\colon P\to Q$, if $a,b\in P$ have a join and $f(a),f(b)\in Q$ have a join, then indeed $f(a)\vee f(b)\leq f(a \vee b)$.
\end{exercise}

In his work on generative effects, Adam restricts his attention to generative maps that
preserve meets (but do not preserve joins). The preservation of meets
implies that the map $\Phi$ behaves well when restricting to subsystems, even
though it can throw up surprises when joining systems.

This discussion naturally leads into Galois connections, which are pairs of monotone maps between preorders, one of which preserves all joins and the other of which preserves all meets.

\section{Galois connections} %
\label{sec.galois_connections}%
\index{Galois connection|(}%
\index{adjunction!of preorders|(}

The preservation of meets and joins, and in particular issues concerning generative effects, is tightly related to the theory of \emph{Galois connections}, which is a special case of a more general theory we will discuss later, namely that of \emph{adjunctions}. We will use some adjunction terminology when describing Galois connections.

\subsection{Definition and examples of Galois connections}%
\label{subsec.def_ex_Galois}
Galois connections between preorders were first considered by \'Evariste
Galois---who didn't call them by that name---in the context of a connection he
found between ``field extensions'' and ``automorphism groups.'' We will not
discuss this further, but the idea is that given two preorders $P$ and $Q$, a
Galois connection is a pair of maps back and forth---from $P$ to $Q$ and from
$Q$ to $P$---with certain properties, which make it like a relaxed version of
isomorphisms. To be a bit more precise, preorder isomorphisms are examples of
Galois connections, but Galois connections need not be preorder isomorphisms.%
\index{isomorphism!adjunction as relaxed version of}

\begin{definition}%
\label{def.galois}%
\index{Galois connection}%
\index{adjunction!Galois connection as}
A \emph{Galois connection} between preorders $P$ and $Q$ is a pair of monotone maps
$f\colon P \to Q$ and $g\colon Q \to P$ such that 
\begin{equation}%
\label{eqn.galois_connection}
  f(p) \le q \quad\text{ if and only if }\quad p \le g(q).
\end{equation}
We say that $f$ is the \emph{left adjoint} and $g$ is the \emph{right adjoint}
of the Galois connection.
\end{definition}

\begin{example}%
\label{ex.adjoint_to_3times}%
\index{ceiling function}%
\index{floor function}%
\index{adjunction!examples of}
Consider the map $(3\times-)\colon \zz \to \rr$ which sends $x \in \zz$ to $3x$, which we can consider as a real number $3x\in\zz\ss\rr$. Let's find a left adjoint for the map $(3\times-)$.

Write $\ceil{z}$ for the smallest natural number above $z \in
\rr$, and write $\floor{z}$ for the largest integer below $z \in
\rr$, e.g.\ $\ceil{3.14}=4$ and $\floor{3.14}=3$.%
\footnote{By ``above'' and ``below,'' we mean \emph{greater than or equal to} or \emph{less than or equal to}; the latter being a mouthful. Anyway, $\floor{3}=3=\ceil{3}$.}
As the left adjoint $\rr\to\zz$, let's see if $\ceil{-/3}$ works.

It is easily checked that  
\[
  \ceil{x/3} \le y \mbox{ if and only if } x \le 3y.
\]
Success! Thus we have a Galois connection between $\ceil{-/3}$ and
$(3\times-)$. 
\end{example}

\begin{exercise}%
\label{exc.right_adj_3times}
In \cref{ex.adjoint_to_3times} we found a left adjoint for the monotone map $(3\times-)\colon \zz \to \rr$. Now find a right adjoint for the same map, and show it is correct.
\end{exercise}

\begin{exercise}%
\label{exc.galois_linear_ord}
Consider the preorder $P=Q=\ord{3}$.
\begin{enumerate}
	\item Let $f,g$ be the monotone maps shown below:
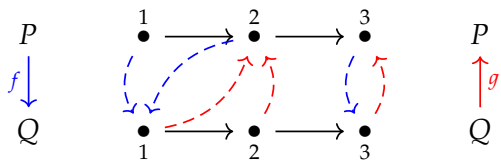
\[
\begin{tikzcd}
	P\ar[d, blue, "f"']&\LMO{1}\ar[r]\ar[d, ->, dashed, bend right, blue]&\LMO{2}\ar[r]\ar[dl, ->, dashed, bend right, blue]&\LMO{3}\ar[d, ->, dashed, bend right, blue]&P\\
	Q&\LMO[under]{1}\ar[r]\ar[ur, ->, dashed, bend right, red]&\LMO[under]{2}\ar[r]\ar[u, ->, dashed, bend right, red]&\LMO[under]{3}\ar[u, ->, dashed, bend right, red]&Q\ar[u, red, "g"']
\end{tikzcd}
\]
Is it the case that $f$ is left adjoint to $g$? Check that for each $1\leq p,q\leq 3$, one has $f(p)\leq q$ iff $p\leq g(q)$.
	\item Let $f,g$ be the monotone maps shown below:
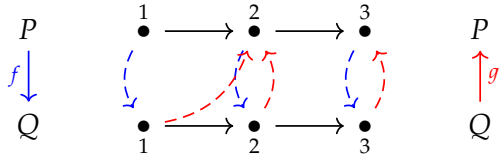
\[
\begin{tikzcd}
	P\ar[d, blue, "f"']&\LMO{1}\ar[r]\ar[d, ->, dashed, bend right, blue]&\LMO{2}\ar[r]\ar[d, ->, dashed, bend right, blue]&\LMO{3}\ar[d, ->, dashed, bend right, blue]&P\\
	Q&\LMO[under]{1}\ar[r]\ar[ur, ->, dashed, bend right, red]&\LMO[under]{2}\ar[r]\ar[u, ->, dashed, bend right, red]&\LMO[under]{3}\ar[u, ->, dashed, bend right, red]&Q\ar[u, red, "g"']
\end{tikzcd}
\]
Is it the case that $f$ is left adjoint to $g$?
\qedhere
\end{enumerate}
\end{exercise}

\begin{remark}%
\index{total order}
The pictures in \cref{exc.galois_linear_ord} suggest the following idea. If $P$
and $Q$ are total orders and $f\colon P\to Q$ and $g\colon Q\to P$ are drawn
with arrows bending counterclockwise, then $f$ is left adjoint to $g$ iff the
arrows do not cross. With a little bit of thought, this can be formalised. We
think this is a pretty neat way of visualizing Galois connections between total
orders!
\end{remark}

%
%
%
%
%

\begin{exercise}%
\label{exc.extra_right_adj_ceil}
\begin{enumerate}
	\item Does $\ceil{-/3}$ have a left adjoint $L\colon \zz\to\rr$?
	\item If not, why? If so, does its left adjoint have a left adjoint?
	\qedhere
\end{enumerate}
\end{exercise}

\subsection{Back to partitions}%
\label{subsec.back_to_parts}%
\index{partition|(}
Recall from \cref{ex.partitions} that we can understand the set $\prt{S}$ of partitions on a set $S$
in terms of surjective functions out of $S$.

Suppose we are given any function $g\colon S\to T$. We will show that this
function $g$ induces a Galois connection 
$g_!\colon\prt{S}\leftrightarrows\prt{T}\cocolon g^*$,
between preorder of $S$-partitions and the preorder of $T$-partitions. The way
you might explain it to a seasoned category theorist is:
\begin{quote}%
\index{adjunction!examples of}
The left adjoint is given by taking any surjection out of $S$ and pushing out along $g$ to get a surjection out of $T$. The right adjoint is given by taking any surjection out of $T$, composing with $g$ to get a function out of $S$, and then taking the epi-mono factorization to get a surjection out of $S$.
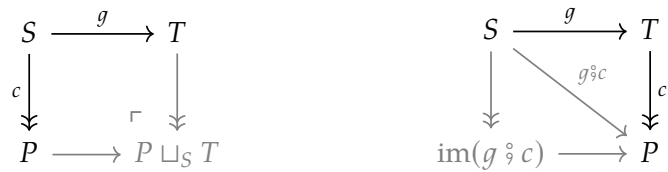
\[
\begin{tikzcd}[row sep=large]
	S\ar[r, "g"]\ar[d, two heads, "c"']\ar[dr, phantom, very near end, gray,  "\ulcorner"]&T\ar[d, two heads, gray]\\
	P\ar[r, gray]&{\color{gray}P\sqcup_ST}
\end{tikzcd}
\hspace{1in}
\begin{tikzcd}[row sep=large]
	S\ar[r, "g"]\ar[d, gray, two heads]\ar[dr, gray, "g\cp c"]&T\ar[d, two heads, "c"]\\
	{\color{gray}\mathrm{im}(g\cp c)}\ar[r, gray]&P
\end{tikzcd}
\]
\end{quote}
By the end of this book, the reader will understand pushouts and epi-mono factorizations, so he or she will be able to make sense of the above statement. But for now we will explain the process in more down-to-earth terms.%
\index{pushout}%
\index{epi-mono factorization}

Start with $g\colon S\to T$; we first want to understand $g_!\colon\prt S\to \prt T$. So start with a partition $\sim_S$ of $S$. To begin the process of obtaining a partition $\sim_T$ on $T$, say that two elements $t_1,t_2\in T$ are in the same part, $t_1\sim_T t_2$, if there exist $s_1,s_2\in S$ with such that $s_1\sim_Ss_2$ and $g(s_1)=t_1$ and $g(s_2)=t_2$. However, the result of doing so will not necessarily be transitive---you may get $t_1\sim_T t_2$ and $t_2\sim_T t_3$ without $t_1\sim_T^? t_3$---and partitions must be transitive. So complete the process by just adding in the missing pieces (take the transitive closure).%
\index{transitive closure} The result is $g_!(\sim_S)\coloneqq\sim_T$.

Again starting with $g$, we want to get the right adjoint $g^*\colon\prt T\to\prt S$. So start with a partition $\sim_T$ of $T$. Get a partition $\sim_S$ on $S$ by saying that $s_1\sim_Ss_2$ iff $g(s_1)\sim_T g(s_2)$. The result is $g^*(\sim_T)\coloneqq \sim_S$.%
\index{partition!pullback of}

\begin{example}%
\label{ex.pushforward_part}%
\index{partition!pushforward of}
Let $S=\{1,2,3,4\}$, $T=\{12,3,4\}$, and $g\colon S\to T$ by $g(1)\coloneqq g(2)\coloneqq 12$, $g(3)\coloneqq 3$, and $g(4)\coloneqq4$. The partition shown left below is translated by $g_!$ to the partition shown on the right.
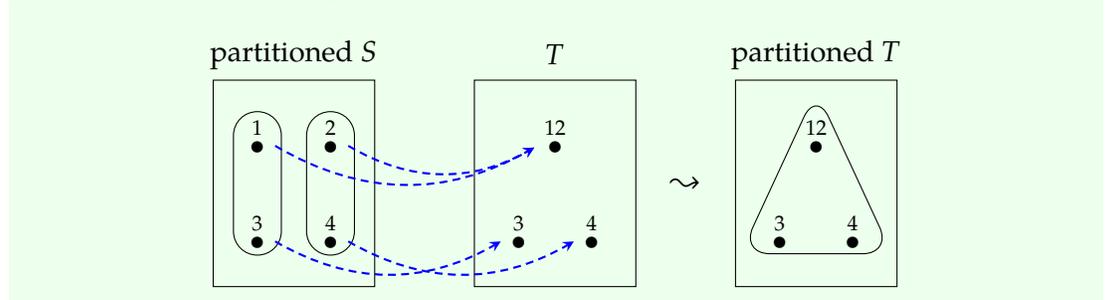
\[

\qedhere
\]
\end{example}

\begin{exercise}%
\label{exc.g_left_part}
There are 15 different partitions of a set with four elements. Choose 6 different ones and for each one, call it $c\colon S\surj P$, find $g_!(c)$, where $S$, $T$, and $g\colon S\to T$ are the same as they were in \cref{ex.pushforward_part}.
\end{exercise}

\begin{example}
Let $S$, $T$ be as below, and let $g\colon S\to T$ be the function shown in blue. Here is a picture of how $g^*$ takes a partition on $T$ and ``pulls it back'' to a partition on $S$:
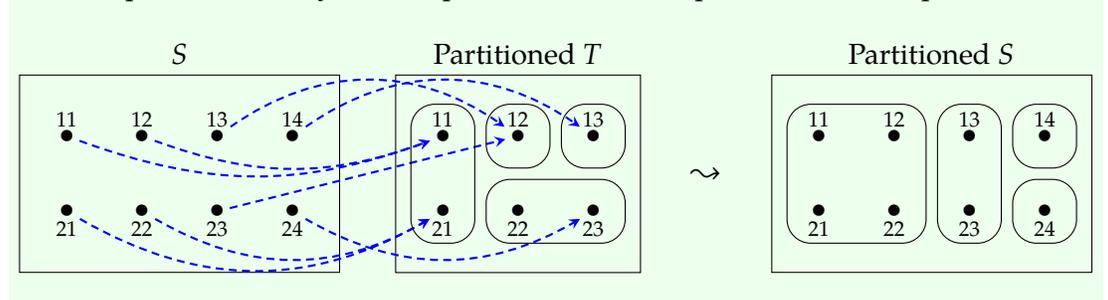
\[
%
\]
\end{example}

\begin{exercise}%
\label{exc.pullback_parts}
There are five partitions possible on a set with three elements, say $T=\{12,3,4\}$. Using the same $S$ and $g\colon S\to T$ as in \cref{ex.pushforward_part}, determine the partition $g^*(c)$ on $S$ for each of the five partitions $c\colon T\surj P$.
\end{exercise}

To check that for any function $g\colon S\to T$, the monotone map
$g_!\colon\prt S\to \prt T$ really is left adjoint to $g^*\colon\prt T\to \prt S$ would take too much time for this sketch. But the following exercise gives some evidence.

\begin{exercise}%
\label{exc.parts_adjunction}
Let $S$, $T$, and $g\colon S\to T$ be as in \cref{ex.pushforward_part}.
\begin{enumerate}
	\item Choose a nontrivial partition $c\colon S\surj P$ and let $g_!(c)$ be its push forward partition on $T$.
	\item Choose any coarser partition $d\colon T\surj P'$, i.e.\ where $g_!(c)\leq d$.
	\item Choose any non-coarser partition $e\colon T\surj Q$, i.e.\ where $g_!(c)\not\leq e$. (If you can't do this, revise your answer for \#1.)
	\item Find $g^*(d)$ and $g^*(e)$.
	\item The adjunction formula \cref{eqn.galois_connection} in this case says that since $g_!(c)\leq d$ and $g_!(c)\not\leq e$, we should have $c\leq g^*(d)$ and $c\not\leq g^*(e)$. Show that this is true.
\qedhere
\end{enumerate}
\end{exercise}
\index{partition|)}

\subsection{Basic theory of Galois connections}

\begin{proposition}%
\label{prop.galois_monad_comonad}%
\index{Galois connection}
Suppose that $f\colon P\to Q$ and $g\colon Q\to P$ are monotone maps. The following are equivalent
\begin{enumerate}[label=(\alph*)]
	\item $f$ and $g$ form a Galois connection where $f$ is left adjoint to $g$,
	\item for every $p\in P$ and $q\in Q$ we have
  \begin{equation}%
\label{eqn.preorder_monad_comonad}
  	p\leq g(f(p))
  	\qquad\text{and}\qquad
  	f(g(q))\leq q.
  \end{equation}
\end{enumerate}
\end{proposition}
\begin{proof}
Suppose $f$ is left adjoint to $g$. Take any $p\in P$, and let $q\coloneqq
f(p)$. By reflexivity, we have $f(p)\leq q$, so by the \cref{def.galois} of
Galois connection we have $p\leq g(q)$, but this means $p\leq g(f(p))$. The
proof that $f(g(q))\leq q$ is similar.

Now suppose that \cref{eqn.preorder_monad_comonad} holds for all $p\in P$ and $q\in Q$. We want show that $f(p)\leq q$ iff $p\leq g(q)$. Suppose $f(p)\leq q$; then since $g$ is monotonic, $g(f(p))\leq g(q)$, but $p\leq g(f(p))$ so $p\leq g(q)$. The proof that $p\leq g(q)$ implies $f(p)\leq q$ is similar.
\end{proof}

\begin{exercise}%
\label{exc.proof_galois_monad}
Complete the proof of \cref{prop.galois_monad_comonad} by showing that
\begin{enumerate}
	\item if $f$ is left adjoint to $g$ then for any $q\in Q$, we have $f(g(q))\leq q$, and 
	\item if \cref{eqn.preorder_monad_comonad} holds, then holds $p\leq g(q)$ iff $f(p)\leq q$ holds, for all $p\in P$ and $q\in Q$.
\qedhere
\end{enumerate}
\end{exercise}

If we replace $\leq$ with $=$ in \cref{eqn.preorder_monad_comonad}, we get back the definition of isomorphism (\cref{def.preorder_isomorphism}); this is why we said at the beginning of \cref{subsec.def_ex_Galois} that Galois connections are a kind of relaxed version of isomorphisms.

\begin{exercise}%
\label{exc.uniqueness_of_adjoints}
 	\begin{enumerate}
		\item Show that if $f\colon P\to Q$ has a right adjoint $g$, then it is unique up to isomorphism. That means, for any other right adjoint $g'$, we have $g(q)\cong g'(q)$ for all $q\in Q$.
  	\item Is the same true for left adjoints? That is, if $h\colon P\to Q$ has a left adjoint, is it necessarily unique up to isomorphism?
	\qedhere
	\end{enumerate}
\end{exercise}

\begin{proposition}[Right adjoints preserve meets]%
\label{prop.right_adj_meets}%
\index{meets!preservation of}%
\index{joins!preservation of}%
\index{adjunction!preservation of meets and joins}
  Let $f\colon P\to Q$ be left adjoint to $g\colon Q\to P$. Suppose $A\ss Q$ any subset, and let $g(A)\coloneqq\{g(a)\mid a\in A\}$ be its image. Then if $A$ has a meet $\bigwedge A\in Q$ then $g(A)$ has a meet $\bigwedge g(A)$ in $P$, and we have
  \[g\left(\bigwedge A\right)\cong\bigwedge g(A).\]
That is, right adjoints preserve meets. Similarly, left adjoints preserve joins: if $A\ss P$ is any subset that has a join $\bigvee A\in P$, then $f(A)$ has a join $\bigvee f(A)$ in $Q$, and we have
  \[f\left(\bigvee A\right)\cong\bigvee f(A).\]  
\end{proposition}
\begin{proof}
  Let $f\colon P \to Q$ and $g\colon Q \to P$ be adjoint monotone maps, with $g$
  right adjoint to $f$. Let $A\ss Q$ be any subset and let $m\coloneqq\bigwedge A$ be
  its meet. Then since $g$ is monotone $g(m) \le g(a)$ for all $a \in A$, so $g(m)$ is a
  lower bound for the set $g(A)$. We will be done if we can show $g(m)$ is a greatest 
  lower bound. 

  So take any other lower bound $b$ for $g(A)$; that is suppose that for all $a\in A$, we have $b\leq g(a)$ and we want to show $b\leq g(m)$. Then by definition of $g$ being a right adjoint (\cref{def.galois}), we also have $f(b)\leq a$. This means that $f(b)$ is a
  lower bound for $A$ in $Q$. Since the meet $m$ is the greatest lower bound, we
  have $f(b)\leq m$. Once again using the Galois connection, $b\leq g(m)$, proving that $g(m)$ is indeed the greatest lower bound for $g(A)$, as desired. 
  
  The second claim is proved similarly; see \cref{exc.mimic_proof_joins}.
\end{proof}

\begin{exercise}%
\label{exc.mimic_proof_joins}
  Complete the proof of \cref{prop.right_adj_meets} by showing that left adjoints preserve joins.
\end{exercise}

Since left adjoints preserve joins, we know that they cannot have generative
effects. In fact, we will see in \cref{prop.adjoint_functor_thm} that a monotone map does not have generative effects---i.e.\ it preserves joins---if and only if it is a left adjoint to some other monotone.

\begin{example}%
\label{ex.right_adj_not_joins}
Right adjoints need not preserve joins. Here is an example:
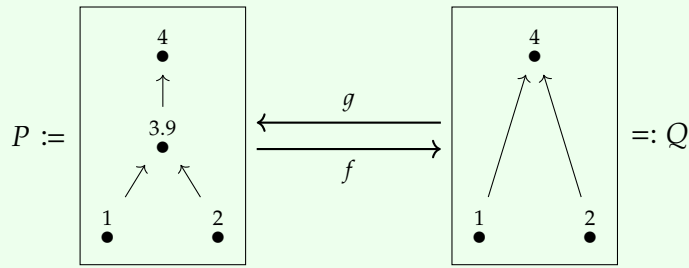
\[
\begin{tikzpicture}[y=1.2cm]
	\node (b1) {$\LMO{1}$};
	\node[right= of b1] (b2) {$\LMO{2}$};
	\node at ($(b1)!.5!(b2)+(0,1)$) (b3) {$\LMO{3.9}$};
	\node at ($(b3)+(0,1)$) (b4) {$\LMO{4}$};
	\draw[->] (b1) -- (b3);
	\draw[->] (b2) -- (b3);
	\draw[->] (b3) -- (b4);
	\node[fit=(b1) (b2) (b4), draw] (b) {};
	\node[left=1pt of b] {$P\coloneqq$};
	\node[right=3 of b2] (a1) {$\LMO{1}$};
	\node[right= of a1] (a2) {$\LMO{2}$};
	\node at ($(a1)!.5!(a2)+(0,2)$) (a4) {$\LMO{4}$};
	\draw[->] (a1) -- (a4);
	\draw[->] (a2) -- (a4);
	\node[fit=(a1) (a2) (a4), draw] (a) {};
	\node[right=1pt of a] {$=:Q$};
	\draw[functor] ($(a.west)+(0,5pt)$) to node[above, font=\footnotesize] {$g$} ($(b.east)+(0,5pt)$);
	\draw[functor] ($(b.east)+(0,-5pt)$) to node[below, font=\footnotesize] {$f$} ($(a.west)+(0,-5pt)$);
\end{tikzpicture}
\]
Let $g$ be the map that preserves labels, and let $f$ be the map that preserves labels as far as possible but with $f(3.9)\coloneqq4$. Both are $f$ and $g$ monotonic, and one can check that $g$ is right adjoint to $f$ (see \cref{exc.g_really_is_right_adj}). But $g$ does not preserve joins because $1\vee 2=4$ holds in $Q$, whereas $g(1)\vee g(2)=1\vee 2=3.9\neq 4=g(4)$ in $P$.
\end{example}

\begin{exercise}%
\label{exc.g_really_is_right_adj}
To be sure that $g$ really is right adjoint to $f$ in \cref{ex.right_adj_not_joins}, there are twelve tiny things to check; do so. That is, for every $p\in P$ and $q\in Q$, check that $f(p)\leq q$ iff $p\leq g(q)$.
\end{exercise}


\begin{theorem}[Adjoint functor theorem for
  preorders]%
\label{prop.adjoint_functor_thm}%
\index{adjoint functor theorem}
Suppose $Q$ is a preorder that has all meets and let $P$ be any preorder. A monotone map $g\colon Q \to P$
preserves meets if and only if it is a right adjoint. 

Similarly, if $P$ has all joins and $Q$ is any preorder, a monotone map $f\colon P \to Q$ preserves joins if
and only if it is a left adjoint.
\end{theorem}
\begin{proof}
We will only prove the claim about meets; the claim about joins follows similarly.

We proved one direction in \cref{prop.right_adj_meets}, namely that right adjoints preserve meets. For the other, suppose that $g$ is a
monotone map that preserves meets; we shall construct a left adjoint $f$. We define our candidate $f\colon P \to Q$ on any $p\in P$ by
\begin{equation}%
\label{eqn.def_f_random596}
  f(p)\coloneqq\bigwedge \{q \in Q \mid p\leq g(q)\};
\end{equation}
this meet is well defined because $Q$ has all meets, but for $f$ to really be a candidate, we need to show it is monotone. So suppose that $p\leq p'$. Then $\{q' \in Q \mid p'\leq g(q')\}\ss\{q \in Q \mid p\leq g(q)\}$. By \cref{prop.containment_meet_join}, this implies $f(p)\leq f(p')$. Thus $f$ is monotone.

By \cref{prop.right_adj_meets}, it suffices to show that $p_0\leq g(f(p_0))$ and that $f(g(q_0))\leq q_0$ for all $p_0\in P$ and $q_0\in Q$. For the first, we have
\[
  p_0\leq
  \bigwedge\{g(q)\in P\mid p_0\leq g(q)\}\cong
  g\left(\bigwedge\{q\in Q\mid p_0\leq g(q)\}\right)=
  g(f(p_0)),
\]
where the first inequality follows from the fact that if $p_0$ is below every element of a set, then it is below their meet, and the isomorphism is by definition of $g$ preserving meets. For the second, we have
\[
	f(g(q_0))=\bigwedge\{q\in Q\mid g(q_0)\leq g(q)\}\leq\bigwedge\{q_0\}=q_0,
\]
where the first inequality follows from \cref{prop.containment_meet_join} since $\{q_0\}\ss\{q\in Q\mid g(q_0)\leq g(q)\}$, and the fact that $\bigwedge\{q_0\}=q_0$.
\end{proof}

\begin{example}%
\index{pullback!along a map}
  Let $f\colon A \to B$ be a function between sets. We can imagine $A$ as a set
  of apples, $B$ as a set of buckets, and $f$ as putting each apple in a bucket.
  
  Then we have the monotone map $f^\ast\colon \powset(Y) \to \powset(X)$ that category theorists call
  ``pullback along $f$.'' This map takes a subset $B' \subseteq B$ to its preimage
  $f\inv (B') \subseteq A$: that is, it takes a collection $B'$ of buckets, and tells
  you all the apples that they contain in total. This operation is monotonic (more buckets means more apples) and it has both a left and a right adjoint. %
\index{preimage}
  
  The left adjoint $f_!(A)$ is given by the direct image: it maps a subset $A' \subseteq A$ to
  \[
  f_!(A') \coloneqq \{b \in B\mid \mbox{there exists }a \in A' \mbox{ such that }
  f(a) = b\}
  \]
  This map takes a set $A'$ of apples, and tells you all the buckets that
  contain at least one of those apples.

  The right adjoint $f_*$ maps a subset $A' \subseteq A$ to 
  \[
  f_*(A')\coloneqq\{b \in B\mid \mbox{for all }a \mbox{ such that } f(a)=b,\mbox{ we have } a \in A'\}
  \]
  This map takes a set $A'$ of apples, and tells you all the buckets $b$ that are all-$A'$: all the apples in $b$ are from the chosen subset $A'$. Note that if a bucket doesn't contain any apples at
  all, then vacuously all its apples are from $A'$, so empty buckets count as far as $f_*$ is concerned.
 
  Notice that all three of these operations turn out to be interesting: start with a set $B'$ of buckets and return all the apples in them, or start with a set $A'$ of apples and either find the buckets that contain at least one apple from $A'$, or the buckets whose only apples are from $A'$. But we did not invent these mappings $f^*$, $f_!$, and $f_*$: they were \emph{induced} by the function $f$. They were automatic. It is one of the pleasures of category theory that adjoints so often turn out to have interesting semantic interpretations.
\end{example}

\begin{exercise}%
\label{exc.subsets_!*}
Choose sets $X$ and $Y$ with between two and four elements each, and choose a function $f\colon X\to Y$.
\begin{enumerate}
	\item Choose two different subsets $B_1,B_2\ss Y$ and find $f^*(B_1)$ and $f^*(B_2)$.
	\item Choose two different subsets $A_1,A_2\ss X$ and find $f_!(A_1)$ and $f_!(A_2)$.
	\item With the same $A_1,A_2\ss X$, find $f_*(A_1)$ and $f_*(A_2)$.\qedhere
\qedhere
\end{enumerate}
\end{exercise}

\subsection{Closure operators}%
\label{subsec.closure_operator}%
\index{closure operator|(}

Given a Galois connection with $f\colon P\to Q$ left adjoint to $g\colon Q\to P$, we may compose $f$ and
$g$ to arrive at a monotone map $f\cp g\colon P \to P$ from preorder $P$ to itself.
This monotone map has the property that $p \le (f\cp g)(p)$, and that $(f\cp g\cp f \cp g)(p) \cong (f\cp g)(p)$ for any $p\in P$.
This is an example of a \emph{closure operator}.%
\footnote{The other composite $g\cp f$ satisfies the dual properties: $(g\cp f)(q)\leq q$ and $(g\cp f\cp g\cp f)(q)\cong(g \cp f)(q)$ for all $q\in Q$. This is called an \emph{interior operator}, though we will not discuss this concept further.%
\index{dual|seealso {properties, dual}}%
\index{properties!dual}}

\begin{exercise}%
\label{exc.closure}
Suppose that $f$ is left adjoint to $g$. Use \cref{prop.galois_monad_comonad} to show the following.
\begin{enumerate}
	\item $p \le (f\cp g)(p)$.
	\item $(f\cp g\cp f \cp g)(p) \cong (f\cp g)(p)$. To prove this, show inequalities in both directions, $\leq$ and $\geq$.
\qedhere
\end{enumerate}
\end{exercise}

\begin{definition}%
\index{closure operator}
A \emph{closure operator} $j\colon P \to P$ on a preorder $P$ is a monotone map such that for all $p \in
P$ we have
\begin{enumerate}[label=(\alph*)]
\item $p \le j(p)$;
\item $j(j(p)) \cong j(p)$.
\end{enumerate}
\end{definition}

\begin{example}%
\label{ex.rewrite_systems}%
\index{program semantics}
Here is an example of closure operators from computation, very roughly presented. Imagine computation as a
process of rewriting input expressions to output expressions. For example, a computer
can rewrite the expression \texttt{7+2+3} as the expression \texttt{12}. The set of arithmetic expressions
has a partial order according to whether one expression can be rewritten as another.

We might think of a computer program, then, as a method of taking an expression and reducing it to
another expression. So it is a map $j\colon \texttt{exp} \to \texttt{exp}$.
It furthermore is desirable to require that this computer program is a closure
operator. Monotonicity means that if an expression $x$ can be rewritten into expression
$y$, then the reduction $j(x)$ can be rewritten into $j(y)$. Moreover, the requirement $x \le j(x)$ implies that $j$
can only turn one expression into another if doing so is a permissible rewrite. The requirement
$j(j(x))=j(x)$ implies if you try to reduce an expression that has already been reduced,
the computer program leaves it as is. These properties provide useful structure
in the analysis of program semantics.
\end{example}

\begin{example}[Adjunctions from closure operators]%
\index{adjunction!from
  closure operator}
Just as every adjunction gives rise to a closure operator, from every closure
operator we may construct an adjunction.

Let $P$ be a preorder and let
$j\colon P\to P$ be a closure operator. We can define a preorder $\Set{fix}_j$ to have
elements the fixed points of $j$; that is,
\[
\Set{fix}_j\coloneqq\{p \in P\mid j(p)\cong p\}.
\]
This is a subset of $P$, and inherits an order as a result; hence $\Set{fix}_j$ is a sub-preorder of $P$. Note that $j(p)$ is a fixed point for all $p\in P$, since $j(j(p))\cong j(p)$.

We define an adjunction with left adjoint $j\colon P \to
\Set{fix}_j$ sending $p$ to $j(p)$, and right adjoint $g\colon \Set{fix}_j
\to P$ simply the inclusion of the sub-preorder. To see it's really an adjunction, we need to see that for any $p \in P$ and $q \in \Set{fix}_j$, we have $j(p) \le q$ if and only if $p \le q$. Let's check it. Since $p \le j(p)$, we have that $j(p) \le q$ implies $p \le q$ by transitivity.
Conversely, since $q$ is a fixed point, $p \le q$ implies $j(p) \le j(q)\cong q$.
\end{example}

\begin{example}%
\label{ex.modal_operator}%
\index{logic}
Another example of closure operators comes from logic. This will be discussed in the final chapter of the book, in particular \cref{subsec.modalities}, but we will give a quick overview here. In essence, logic is the study of when one formal statement---or proposition---implies another. For example, if $n$ is prime then $n$ is not a multiple of $6$, or if it is raining then the ground is getting wetter. Here ``$n$ is prime'', ``$n$ is not a multiple of $6$'', ``it is raining'', and ``the ground is getting wetter'' are propositions, and we gave two implications.

Take the set of all propositions, and order them by $p\leq q$ iff $p$ implies $q$, denoted $p\imp q$. Since $p\imp p$ and since whenever $p\imp q$ and $q\imp r$, we also have $p\imp r$, this is indeed a preorder.

A closure operator on it is often called a \emph{modal operator}. It is a function $j$
from propositions to propositions, for which $p\imp j(p)$ and $j(j(p))= j(p)$. An
example of a $j$ is ``assuming Bob is in San Diego....'' Think of this as a proposition $B$; so ``assuming Bob is in San Diego, $p$'' means $B\imp p$. Let's see why $B\imp -$ is a closure operator.

If `$p$' is true then
``assuming Bob is in San Diego, $p$'' is still true. Suppose that ``assuming Bob
is in San Diego it is the case that, assuming Bob is in San Diego, $p$' is true.''
It follows that ``assuming Bob is in San Diego, $p$'' is true. So we have seen, at least informally, that ``assuming Bob is in San Diego...'' is a closure operator.
\end{example}
\index{closure operator|)}

\subsection{Level shifting} %
\label{ssec.level_shift}%
\index{level shift}\index{level shift|seealso {primordial ooze}}

The last thing we want to discuss in this chapter is a phenomenon that happens often in category theory, something we might informally call ``level-shifting.'' It is easier to give an example of this than to explain it directly.

Given any set $S$, there is a set $\Cat{Rel}(S)$ of binary relations on $S$. An element $R\in\Cat{Rel}(S)$ is formally a subset $R\ss S\times S$. The set $\Cat{Rel}(S)$ can be given an order via the subset relation, $R\ss R'$, i.e.\ if whenever $R(s_1,s_2)$ holds then so does $R'(s_1,s_2)$.%
\index{relation!binary}

For example, the Hasse diagram for $\Cat{Rel}(\{1\})$ is:
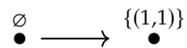
\[
\begin{tikzcd}
	\LMO{\varnothing}\ar[r]& 	\LMO{\{(1,1)\}}
\end{tikzcd}
\]
\index{Hasse diagram}

\begin{exercise}%
\label{exc.Hasse_binary12}
Draw the Hasse diagram for the preorder $\Cat{Rel}(\{1,2\})$ of all binary relations on the set $\{1,2\}$.
\end{exercise}

For any set $S$, there is also a set $\Cat{Pos}(S)$,
consisting of all the preorder relations on $S$. In fact there is a preorder
structure $\sqsubseteq$%
\label{page.order_preorder} on $\Cat{Pos}(S)$, again given by inclusion: $\leq$
is below $\leq'$ (we'll write $\leq\sqsubseteq\leq')$ if $a\leq b$ implies $a\leq' b$ for every $a,b\in S$. A preorder of preorder structures? That's what we mean by a level shift.

Every preorder relation is---in particular---a relation, so we have an inclusion
$\Cat{Pos}(S)\to\Cat{Rel}(S)$. This is the right adjoint of a Galois connection.%
\index{adjunction!examples of}
Its left adjoint is a monotone map
$\mathrm{Cl}\colon\Cat{Rel}(S)\to\Cat{Pos}(S)$ given by taking any relation $R$, writing it in infix notation using $\leq$, and
taking the reflexive and transitive closure, i.e.\ adding $s\leq s$ for every
$s$ and adding $s\leq u$ whenever $s\leq t$ and $t\leq u$.%
\index{relation!free preorder on}

\begin{exercise}%
\label{exc.understand_adj}%
\index{adjunction}
	Let $S=\{1,2,3\}$. Let's try to understand the adjunction discussed above.
	\begin{enumerate}
		\item Come up with any preorder relation $\leq$ on $S$, and define $U(\leq)$ to be the subset $U(\leq)\coloneqq\{(s_1,s_2)\mid s_1\leq s_2\}\ss S\times S$, i.e.\ $U(\leq)$ is the image of $\leq$ under the inclusion $\Cat{Pos}(S)\to\Cat{Rel}(S)$, the relation `underlying' the preorder.
		\item Come up with any two binary relations $Q\ss S\times S$ and $Q'\ss S\times S$ such that $Q\ss U(\leq)$ but $Q'\not\ss U(\leq)$. Note that your choice of $Q,Q'$ do not have to come from preorders.
	\end{enumerate}
	We now want to check that in this case, the closure operation $\Fun{Cl}$ is really left adjoint to the `underlying relation' map $U$.
	\begin{enumerate}[resume]
		\item Concretely (without using the assertion that there is some sort of adjunction), show that $\Fun{Cl}(Q)\sqsubseteq\;\leq$, where $\sqsubseteq$ is the order on $\Cat{Pos}(S)$, defined immediately above this exercise.
		\item Concretely show that $\Fun{Cl}(Q')\not\sqsubseteq\;\leq$.
	\qedhere
\end{enumerate}
\end{exercise}

\index{Galois connection|)}%
\index{adjunction!of preorders|)}
\section{Summary and further reading}%
\label{ch1.further_reading}

In this first chapter, we set the stage for category theory by introducing
one of the simplest interesting sorts of example: preorders. From this seemingly
simple structure, a bunch of further structure emerges:
monotone maps, meets, joins, and more. In terms of modeling real world phenomena, we
thought of preorders as the states of a system, and monotone maps as describing a
way to use one system to observe another. From this point of view, generative
effects occur when observations of the whole cannot be deduced by combining observations of the parts.

In the final section we introduced Galois connections. A Galois connection, or adjunction, is a pair of maps that are like inverses, but allowed to be more ``relaxed'' by getting the orders involved. Perhaps
surprisingly, it turns out adjunctions are closely related to joins and meets:
if a preorder $P$ has all joins, then a monotone map out of $P$ is a left adjoint if and only if it preserves joins; similarly for meets and right adjoints.

The next two chapters build significantly on this material, but in two different
directions. \cref{chap.resource_theory} adds a new operation on the underlying
set: it introduces the idea of a monoidal structure on preorders. This allows us to
construct an element $a \otimes b$ of a preorder $P$ from any elements $a, b \in
P$, in a way that respects the order. On the other hand,
\cref{chap.databases} adds new structure on the order itself: it introduces the
idea of a morphism, which describes not only whether $a \le b$, but gives a name
$f$ for how $a$ relates to $b$. This structure is known as a category.
These generalizations are both fundamental to the story of compositionality, and
in \cref{chap.codesign} we'll see them meet in the concept of a monoidal
category.  The lessons we have learned in this chapter will illuminate the
more highly-structured generalizations in the chapters to come.  Indeed, it is a
useful principle in studying category theory to try to understand concepts first
in the setting of preorders---where often much of the complexity is stripped away
and one can develop some intuition---before considering the general case.

But perhaps you might be interested in exploring some ideas in this chapter in
other directions. While we won't return to them in this book, we learned about
generative effects from Elie Adam's thesis \cite{Adam:2017a}, and a much richer
treatment of generative effect can be found there. In particular, he discusses abelian categories and cohomology, providing a way to
detect generative effects in quite a general setting.%
\index{generative effect}

Another important application of preorders, monotone maps, and Galois connections
is to the analysis of programming languages. In this setting, preorders describe
the possible states of a computer, and monotone maps describe the action of
programs, or relationships between different ways of modeling computation
states. Galois connections are useful for showing how different models may be closely
related, and for transporting program analysis from one framework to another.
For more detail on this, see Chapter 4 of the textbook
\cite{Nielson:1999:PPA:555142}.%
\index{program semantics}

\setcounter{chapter}{1}
\chapter[Resources: monoidal preorders and enrichment]{Resource theories:\\Monoidal preorders and enrichment} %
\label{chap.resource_theory}


\section{Getting from $a$ to $b$}%
\index{resource!theory|(}

You can't make an omelette without breaking an egg. To obtain the things we want requires resources, and the process of transforming what we have into what we want is often an intricate one. In this chapter, we will discuss how monoidal preorders can help us think about this matter.

Consider the following three questions you might ask yourself:
\begin{itemize}
	\item Given what I have, is it \emph{possible} to get what I want?
	\item Given what I have, what is the \emph{minimum cost} to get what I want?
	\item Given what I have, what is the \emph{set of ways} to get what I want?
\end{itemize}
These questions are about resources---those you have and those you want---but perhaps more importantly, they are about moving from have to want: possibility of, cost of, and ways to.

Such questions come up not only in our lives, but also in science and industry. In chemistry, one asks whether a certain set of compounds can be transformed into another set, how much energy such a reaction will require, or what methods exist for making it happen. In manufacturing, one asks similar questions.
\index{chemistry}
\index{manufacturing}

From an external point of view, both a chemist and an industrial firm might be regarded as store-houses of information on the above subjects. The chemist knows which compounds she can make given other ones, and how to do so; the firm has stored knowledge of the same sort. The research work of the chemist and the firm is to use what they know in order to derive---or discover---new knowledge.

This is roughly the first goal of this chapter: to discuss a formalism for expressing recipes---methods for transforming one set of resources into another---and for deriving new recipes from old. The idea here is not complicated, neither in life nor in our mathematical formalism. The value added then is to simply see how it works, so we can build on it within the book, and so others can build on it in their own work.%
\index{recipes}

We briefly discuss the categorical approach to this idea---namely that of
\emph{monoidal preorders}---for building new recipes from old. The following
\emph{wiring diagram} shows, assuming one knows how to implement each of the
interior boxes, how to implement the preparation of a lemon meringue
pie:%
\index{pie!lemon meringue}%
\index{wiring diagram}
\begin{equation}%
\label{eqn.how_to_bake_pie}

\end{equation}
The wires show resources: we start with prepared crust, lemon, butter, sugar, and egg resources, and we end up with an unbaked pie resource. We could take this whole method and combine it with others, e.g.\ baking the pie:
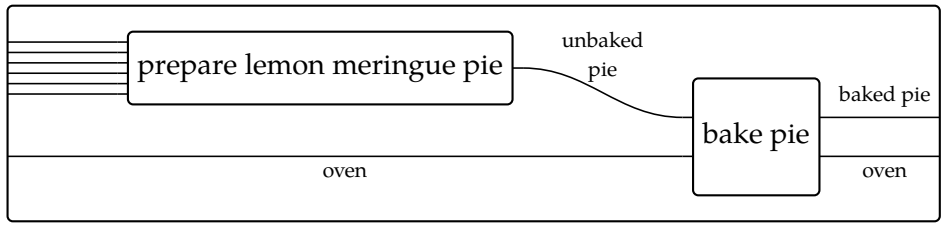
\[
%
\]
In the above example we see that resources are not always consumed when they are used. For example, we use an oven to convert---or catalyze the transformation of---an unbaked pie into a baked pie, and we get the oven back after we are done. It's a nice feature of ovens! To use economic terms, the oven is a ``means of production'' for pies.

String diagrams are important mathematical objects that will come up repeatedly in this book. They were invented in the mathematical context---more specifically in the context of monoidal categories---by Joyal and Street \cite{Joyal.Street:1993a}, but they have been used less formally by engineers and scientists in various contexts for a long time.

As we said above, our first goal in this chapter is to use monoidal preorders, and
the corresponding wiring diagrams, as a formal language for recipes from
old. Our second goal is to discuss something called $\mathcal{V}$-categories for
various monoidal preorders $\cat{V}$.

A $\cat{V}$-category is a set of objects,
which one may think of as points on a map, where $\cat{V}$ somehow ``structures the question'' of getting from point $a$ to point $b$. The examples of monoidal preorders $\cat{V}$ that we will be most interested in
are called $\Bool$ and $\Cost$. Roughly speaking, a $\Bool$-category is a set of
points where the question of getting from point $a$ to point $b$ has a $\true$ /
$\false$ answer. A $\Cost$-category is a set of points where the question of
getting from $a$ to $b$ has an answer $d\in[0,\infty]$, a cost.

This story works in more generality than monoidal preorders. Indeed, in
\cref{chap.codesign} we will discuss something called a monoidal category, a notion which
generalizes monoidal preorders, and we will generalize the definition of $\cat{V}$-category
accordingly. In this more general setting, $\cat{V}$-categories can also
address our third question above, describing \emph{methods} of getting between points. For example a $\smset$-category is a set of points where the question of getting from point $a$ to point $b$ has a set of answers (elements of which might be called methods).

We will begin in \cref{sec.sym_mon_preorders} by defining symmetric monoidal
preorders, giving a few preliminary examples, and discussing wiring diagrams. We
then give many more examples of symmetric monoidal preorders, including both some
real-world examples, in the form of resource theories, and some mathematical
examples that will come up again throughout the book. In
\cref{sec.enrichment} we discuss enrichment and $\cat{V}$-categories---how a
monoidal preorder $\cat{V}$ can ``structure the question'' of getting from $a$ to $b$---and then give some important constructions on $\cat{V}$-categories
(\cref{sec.vcat_constructions}), and analyze them using a sort of matrix multiplication technique (\cref{sec.quantales}).

\section{Symmetric monoidal preorders}%
\label{sec.sym_mon_preorders}%
\index{monoidal preorder|(}%
\index{preorder!monoidal|see {monoidal preorder}}

In \cref{subsec.def_preorder} we introduced preorders. The notation for a preorder, namely $(X,\leq)$, refers to two pieces of structure: a set called $X$ and a relation called $\leq$ that is reflexive and transitive.

We want to add to the concept of preorders a way of combining elements in $X$, an operation
taking two elements and adding or multiplying them together. However, the
operation does not have to literally be addition or multiplication; it only
needs to satisfy some of the properties one expects from them.

\subsection{Definition and first examples}

We begin with a formal definition of symmetric monoidal preorders.

\begin{definition}%
\label{def.symm_mon_structure}%
\index{preorder!symmetric monoidal|see {monoidal preorder}}%
\index{monoidal structure}
A  \emph{symmetric monoidal structure} on a preorder $(X,\leq)$ consists of two constituents:
\begin{enumerate}[label=(\roman*)]
	\item an element $I\in X$, called the \emph{monoidal unit}, and%
\index{monoidal unit}
	\item a function $\otimes \colon X\times X\to X$, called the \emph{monoidal product}.%
\index{monoidal product}%
\end{enumerate}
These constituents must satisfy the following properties, where we write $\otimes(x_1,x_2) = x_1 \otimes x_2$:%
\index{unit!monoidal}
\begin{enumerate}[label=(\alph*)]
	\item for all $x_1,x_2,y_1,y_2\in X$, if $x_1\leq y_1$ and $x_2\leq y_2$, then $x_1\otimes x_2\leq y_1\otimes y_2$,
	\item for all $x\in X$, the equations $I\otimes x= x$ and $x\otimes I= x$ hold,
	\item for all $x,y,z\in X$, the equation $(x\otimes y)\otimes z= x\otimes (y\otimes z)$ holds, and
	\item for all $x,y\in X$, the equation $x\otimes y = y\otimes x$ holds.%
\end{enumerate}
We call these conditions \emph{monotonicity}, \emph{unitality},
\emph{associativity}, and \emph{symmetry} respectively.
A preorder equipped with a symmetric monoidal structure, $(X,\leq,I,\otimes)$, is called a \emph{symmetric monoidal preorder}.%
\index{associativity!of monoidal product}%
\index{symmetry}%
\index{unitality!of monoidal product}%
\index{symmetry!of monoidal product}
\end{definition}

Anyone can propose a set $X$, an order $\leq$ on $X$, an element $I$ in $X$, and a binary operation $\otimes$ on $X$ and ask whether $(X,\leq,I,\otimes)$ is a symmetric monoidal preorder. And it will indeed be one, as long as it satisfies rules a, b, c, and d of \cref{def.symm_mon_structure}.

\begin{remark}
It is often useful to replace $=$ with $\cong$ throughout
\cref{def.symm_mon_structure}. The result is a perfectly good notion, called a
\emph{weak monoidal structure}. The reason we chose equality is that it makes
equations look simpler, which we hope aids first-time readers.%
\index{monoidal structure!weak}
\end{remark}

The notation for the monoidal unit and the monoidal product may vary: monoidal
units we have seen include $I$ (as in the definition), $0$, $1$, $\true$,
$\false$, $\{*\}$, and more. Monoidal products we have seen include $\otimes$
(as in the definition), $+$, $*$, $\wedge$, $\vee$, and $\times$. The
\emph{preferred notation} in a given setting is whatever best helps our brains
remember what we're trying to do; the names $I$ and $\otimes$ are just defaults.%
\index{notation!for monoidal structures}

\begin{example}%
\label{ex.real_nums_preorder}%
\index{natural numbers}
There is a well-known preorder structure, denoted $\leq$, on the set $\RR$ of real numbers; e.g.\ $-5\leq \sqrt{2}$. We propose $0$ as a monoidal unit and $+\colon\RR\times\RR\to\RR$ as a monoidal product. Does $(\RR,\leq,0,+)$ satisfy the conditions of \cref{def.symm_mon_structure}?

If $x_1\leq y_1$ and $x_2\leq y_2$, it is true that $x_1+x_2\leq y_1+y_2$. It is also true that $0+x=x$ and $x+0=x$, that $(x+y)+z=x+(y+z)$, and that $x+y=y+x$. Thus $(\RR,\leq,0,+)$ satisfies the conditions of being a symmetric monoidal preorder.
\end{example}

\begin{exercise}%
\label{exc.monoidal_reals} %
\index{real numbers}
Consider again the preorder $(\RR,\leq)$ from \cref{ex.real_nums_preorder}. Someone proposes $1$ as a monoidal unit and $*$ (usual multiplication) as a monoidal product. But an expert walks by and says ``that won't work.'' Figure out why, or prove the expert wrong!
\end{exercise}

\begin{example}%
\index{monoid}%
\label{ex.monoid}
A \emph{monoid} consists of a set $M$, a function $*\colon M\times M\to M$ called the \emph{monoid multiplication}, and an element $e\in M$ called the \emph{monoid unit}, such that, when you write $*(m,n)$ as $m*n$, i.e.\ using infix notation%
\index{infix notation}, the equations
\begin{equation} %
\label{eqn.monoid}
	m*e=m,\qquad e*m=m,\qquad (m*n)*p=m*(n*p)
\end{equation}
hold for all $m,n,p\in M$. It is called \emph{commutative} if also $m*n=n*m$.

Every set $S$ determines a discrete preorder $\Cat{Disc}_S$ (where $m\leq n$ iff
$m=n$; see \cref{ex.disc_preorder}), and it is easy to check that if $(M,e,*)$
is a commutative monoid then $(\Cat{Disc}_M,=,e,*)$ is a symmetric monoidal preorder.
\end{example}%
\index{infix notation}

\begin{exercise} %
\label{exc.disc_mon_preorder}
  We said it was easy to check that if $(M,*,e)$ is a commutative monoid then
  $(\Cat{Disc}_M,=,*, e)$ is a symmetric monoidal preorder. Are we telling the truth?
\end{exercise}

\begin{example}%
\label{ex.nonexample_poker}%
\index{poker}
Here is a non-example for people who know the game ``standard poker.'' Let $H$
be the set of all poker hands, where a hand means a choice of five cards from
the standard 52-card deck. As an order, put $h\leq h'$ if $h'$ beats or
equals $h$ in poker.

One could propose a monoidal product $\otimes\colon H\times H\to H$ by assigning
$h_1\otimes h_2$ to be ``the best hand one can form out of the ten cards in
$h_1$ and $h_2$.'' If some cards are in both $h_1$ and $h_2$, just throw the
duplicates away. So for example $\{2\heartsuit,\; 3\heartsuit,\; 4\heartsuit,\;
6\spadesuit,\; 7\spadesuit\}\otimes\{2\heartsuit,\; 5\heartsuit,\;
6\heartsuit,\; 6\spadesuit,\; 7\spadesuit\}=\{2\heartsuit,\; 3\heartsuit,\;
4\heartsuit,\; 5\heartsuit,\; 6\heartsuit\}$, because the latter is the best
hand you can make with the former two.

This proposal for a monoidal structure will fail the condition (a) of \cref{def.symm_mon_structure}: it could be the case that $h_1\leq i_1$ and $h_2\leq i_2$, and yet \emph{not} be the case that $h_1\otimes h_2\leq i_1\otimes i_2$. For example, consider this case:
\begin{align*}
  h_1&\coloneqq\{2\heartsuit,\; 3\heartsuit,\; 10\spadesuit,\; \mathrm{J}\spadesuit,\; \mathrm{Q}\spadesuit\}
  &
	i_1&\coloneqq\{4\clubsuit,\; 4\spadesuit,\; 6\heartsuit,\; 6\diamondsuit,\; 10\diamondsuit\}
	\\
	h_2&\coloneqq\{2\diamondsuit,\; 3\diamondsuit,\; 4\diamondsuit,\; \mathrm{K}\spadesuit,\; \mathrm{A}\spadesuit\}
	&	
	i_2&\coloneqq\{5\spadesuit,\; 5\heartsuit,\; 7\heartsuit,\; J\diamondsuit,\; Q\diamondsuit\}.
\end{align*}
Here, $h_1\leq i_1$ and $h_2\leq i_2$, but $h_1\otimes h_2=\{10\spadesuit,\; \mathrm{J}\spadesuit,\; \mathrm{Q}\spadesuit,\; \mathrm{K}\spadesuit,\; \mathrm{A}\spadesuit\}$ is the best possible hand and beats $i_1\otimes i_2=\{5\spadesuit,\; 5\heartsuit,\; 6\heartsuit,\; 6\diamondsuit,\; Q\diamondsuit\}$.
\end{example}

Subsections \ref{subsec.SMPs_science} and \ref{subsec.SMPs_pure_math}
are dedicated to examples of symmetric monoidal preorders. Some are aligned
with the notion of resource theories, others come from pure math. When
discussing the former, we will use wiring diagrams, so here is a quick primer.

\subsection{Introducing wiring diagrams}%
\label{ssec.wirdia2}%
\index{wiring diagram|(}

Wiring diagrams are visual representations for building new relationships from old. In a preorder without a monoidal structure, the only sort of relationship between objects is $\leq$, and the only way you build a new $\leq$ relationship from old ones is by chaining them together. We denote the relationship $x\leq y$ by \begin{equation}%
\label{eqn.string_x_leq_y}
\begin{aligned}

\end{aligned}
\end{equation}
We can chain some number of these $\leq$-relationships---say 0, 1, 2, or 3 of them---together in series as shown here
\begin{equation}%
\label{eqn.strings2}
\begin{aligned}
%
\end{aligned}
\end{equation}

If we add a symmetric monoidal structure, we can combine relationships not only in series but also in parallel. Here is an example:
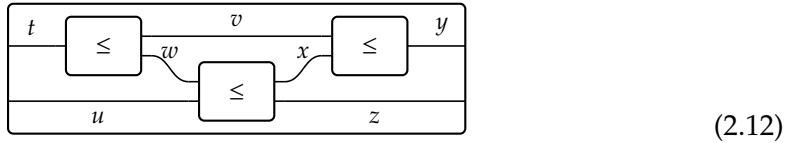
\begin{equation}%
\label{eqn.random334_string_diag}
%
\end{equation}

\paragraph{Different styles of wiring diagrams}%
\index{wiring diagram!styles of}%
\index{wiring diagram!icon of}
In fact, we will see later that there are many styles of wiring diagrams. When we are dealing with preorders, the sort of wiring diagram we can draw is that with single-input, single-output boxes connected in series. When we are dealing with symmetric monoidal preorders, we can have more complex boxes and more complex wiring diagrams, including parallel composition. Later we will see that for other sorts of categorical structures, there are other styles of wiring diagrams:
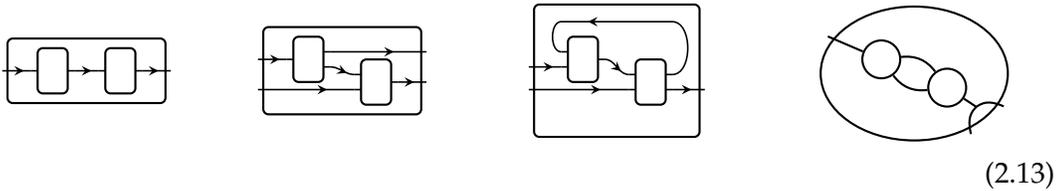
\begin{equation}%
\label{eqn.styles_of_WD}

\end{equation}

\paragraph{Wiring diagrams for symmetric monoidal preorders}%
\index{wiring
diagram!for monoidal preorders}

The style of wiring diagram that makes sense in any symmetric monoidal preorder is that shown in \cref{eqn.random334_string_diag}: boxes can have multiple inputs and outputs, and they may be arranged in series and parallel. Symmetric monoidal preorders and their wiring diagrams are tightly coupled with each other. How so?

The answer is that a monoidal preorder $(X,\leq,I,\otimes)$ has some notion of element ($x\in X$), relationship ($\leq$), and combination (using transitivity and $\otimes$), and so do wiring diagrams: the wires represent elements, the boxes represent relationships, and the wiring diagrams themselves show how relationships can be combined. We call boxes and wires \emph{icons}; we will encounter several more icons in this chapter, and throughout the book.%
\index{icon}

To get a bit more rigorous about the connection, let's start with a monoidal
preorder $(X,\leq,I,\otimes)$ as in \cref{def.symm_mon_structure}. Wiring
diagrams have wires on the left and the right. Each element $x\in X$ can be made
the label of a wire. Note that given two objects $x,y$, we can either draw two
wires in parallel---one labeled $x$ and one labeled $y$---or we can draw one
wire labeled $x\otimes y$. 
\[
\begin{tikzpicture}[y=1ex,decoration=brace]
	\draw (4,1) to node[above] {$x$} +(1,0);
	\draw (4,-1) to node[below] {$y$} +(1,0);
	\draw (7, 0) to node[below] {$x\otimes y$} +(1,0);
\end{tikzpicture}
\]
We consider wires in parallel to represent the monoidal product of their labels,
so we consider both cases above to represent the element $x \otimes y$. Note
also that a wire labeled $I$ or an absence of wires:
\[
\begin{tikzpicture}
	\draw (0,0) to node[below] {$I$} +(1,0);
	\node[gray] at (4,0) {\emph{nothing}};
\end{tikzpicture}
\]
both represent the monoidal unit $I$; another way of thinking of this is that
the unit is the empty monoidal product.%
\index{monoidal unit!drawn as nothing}%
\index{unit!monoidal|seealso {monoidal unit}}

A wiring diagram runs between a set of parallel wires on the left and a set of
parallel wires on the right. We say that a wiring diagram is \emph{valid} if the
monoidal product of the elements on the left is less than the monoidal product
of those on the right. For example, if we have the inequality $x\leq y$, the the
diagram that is a box with a wire labeled $x$ on the left and a wire labeled $y$
on the right is valid; see the first box below: 
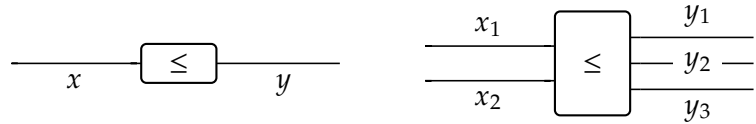
\[

\]
The validity of the second box corresponds to the inequality $x_1\otimes x_2\leq
y_1\otimes y_2\otimes y_3$. Before going on to the properties from
\cref{def.symm_mon_structure}, let us pause for an example of what we've
discussed so far.

\begin{example}%
\index{real numbers}
Recall the symmetric monoidal preorder $(\RR,\leq,0,+)$ from \cref{ex.real_nums_preorder}. The wiring diagrams for it allow wires labeled by real numbers. Drawing wires in parallel corresponds to adding their labels, and the wire labeled $0$ is equivalent to no wires at all.
\[
%
\]
And here we express a couple facts about $(\RR,\leq,0,+)$ in this language: $4\leq 7$ and $2+5\leq -1+5+3$.
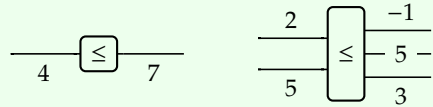
\[
%
\qedhere
\]
\end{example}

We now return to how the properties of symmetric monoidal preorders correspond
to properties of this sort of wiring diagram. Let's first talk about the order
structure: conditions (a)---reflexivity---and (b)---transitivity---from
\cref{def.preorder}. Reflexivity says that $x \le x$, this means the diagram
just consisting of a wire
\[
  %
\]
is always valid. Transitivity allows us to connect facts together: it says that
if $x \le y$ and $y \le z$, then $x \le z$. This means that if the diagrams
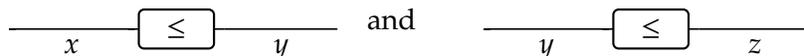
\[
  %
\]
are valid, we can put them together and obtain the valid diagram
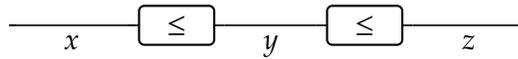
\[
  %
\]

Next let's talk about the properties (a)--(d) from the definition of symmetric monoidal structure
(\cref{def.symm_mon_structure}). Property (a) says that
if $x_1\leq y_1$ and $x_2\leq y_2$ then $x_1\otimes x_2\leq y_1\otimes y_2$.
This corresponds to the idea that stacking any two valid boxes in
parallel is still valid:
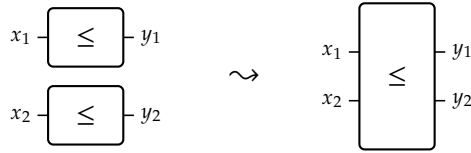
\[
%
	
\]
Condition (b), that $I\otimes x=x$ and $x\otimes I=x$, says we don't need to
worry about $I$ or blank space; in particular diagrams such as the following are
valid:
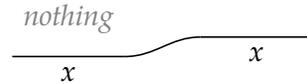
\[
%
\]
Condition (c), that $(x\otimes y)\otimes z=x\otimes(y\otimes z)$ says that we
don't have to worry about whether we build up diagrams from the top or from the
bottom
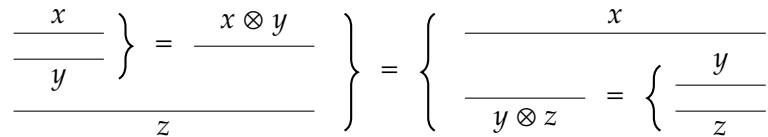
\[
%
\]
But this looks much harder than it is: the associative property should be
thought of as saying that there is no distinction between the stuff on the very
left above and the stuff on the very right, i.e.\ %
\index{associativity!in wiring diagrams}
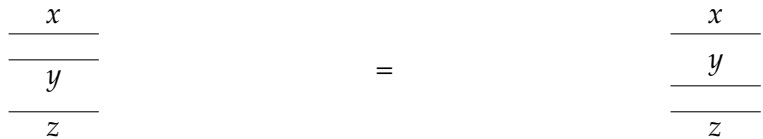
\[
%
\]
and indeed a diagram that moves from one to the other is valid.

Finally, the symmetry condition (d), that $x\otimes y = y\otimes x$, says that a
diagram is valid even if its wires cross:
\]
One may regard the pair of crossing wires as another icon in our iconography, in addition to the boxes and wires we already have.%
\index{wiring diagram!icon of}%
\index{icon!crossing wires}%
\index{symmetry!in wiring diagrams}

\paragraph{Wiring diagrams as graphical proofs}%
\index{wiring diagram!as graphical proof}

Given a monoidal preorder $\cat{X}=(X,\leq, I, \otimes)$, a wiring diagram is a graphical proof of something about $\cat{X}$. Each box in the diagram has a left side and a right side, say $x$ and $y$, and represents the assertion that $x\leq y$.
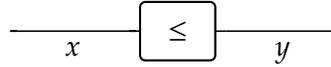
\[
%
\]
A wiring diagram is a bunch of interior boxes connected together inside an exterior box. It represents a graphical proof that says: if all of the interior assertions are correct, then so is the exterior assertion.
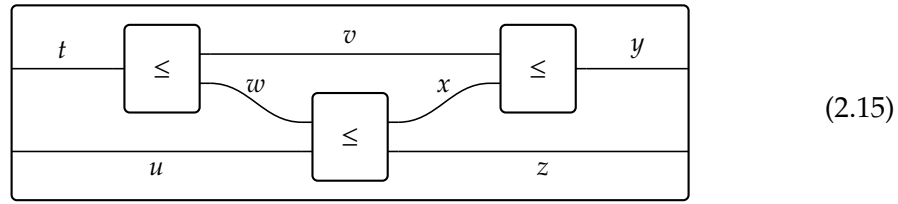
\begin{equation}%
\label{eqn.random5739_string_diag}
  \begin{aligned}
%
\end{aligned}
\end{equation}
The inner boxes in \cref{eqn.random5739_string_diag} translate into the assertions:
\begin{equation}%
\label{eqn.some_random334_assertions}
	t\leq v+w\qquad w+u\leq x+z\qquad v+x\leq y
\end{equation}
and the outer box translates into the assertion:
\begin{equation}%
\label{eqn.random334_conclusion}
	t+u\leq y+z.
\end{equation}
The whole wiring diagram \ref{eqn.random5739_string_diag} says ``if you know
that the assertions in \ref{eqn.some_random334_assertions} are true, then I am a
proof that the assertion in \ref{eqn.random334_conclusion} is also true.'' What
exactly is the proof that diagram \ref{eqn.random5739_string_diag} represents? It is
the proof
\begin{equation}%
\label{eqn.almost_proof}
	t+u\;\leq\; v+w+u\;\leq\;v+x+z\;\leq\; y+z.
\end{equation}
Indeed, each inequality here is a vertical slice of the diagram
\ref{eqn.random5739_string_diag}, and the transitivity of these inequalities is
expressed by connecting these vertical slices together.

\begin{example}%
\index{pie!lemon meringue}
Recall the lemon meringue pie wiring diagram from \cref{eqn.how_to_bake_pie}. It has five interior boxes, such as ``separate egg'' and ``fill crust,'' and it has one exterior box called ``prepare lemon meringue pie.'' Each box is the assertion that, given the collection of resources on the left, say an egg, you can transform it into the collection of resources on the right, say an egg white and an egg yolk. The whole string diagram is a proof that if each of the interior assertions is true---i.e.\ you really do know how to separate eggs, make lemon filling, make meringue, fill crust, and add meringue---then the exterior assertion is true: you can prepare a lemon meringue pie. 
\end{example}

\begin{exercise} %
\label{exc.proof_by_wiring}
The string of inequalities in \cref{eqn.almost_proof} is not quite a proof, because technically there is no such thing as $v+w+u$, for example. Instead, there is $(v+w)+u$ and $v+(w+u)$, and so on.
\begin{enumerate}
	\item Formally prove, using only the rules of symmetric monoidal preorders (\cref{def.symm_mon_structure}), that given the assertions in \cref{eqn.some_random334_assertions}, the conclusion in \cref{eqn.random334_conclusion} follows.
	\item Reflexivity and transitivity should show up in your proof. Make sure you are explicit about where they do.
	\item How can you look at the wiring diagram
	\cref{eqn.random334_string_diag} and know that the symmetry axiom
	(\cref{def.symm_mon_structure}(d)) does not need to be invoked?
\qedhere
\end{enumerate}
\end{exercise}

\index{wiring diagram|)}

We next discuss some examples of symmetric monoidal preorders. We begin in
\cref{subsec.SMPs_science} with some more concrete examples, from science, commerce, and informatics. Then in \cref{subsec.SMPs_pure_math} we discuss some examples arising from pure math, some of which will get a good deal of use later on, e.g.\ in \cref{chap.codesign}.

\subsection{Applied examples}%
\label{subsec.SMPs_science}

Resource theories are studies of how resources are exchanged in a given arena. For example, in social resource theory one studies a marketplace where combinations of goods can be traded for---as well as converted into---other combinations of goods.

Whereas marketplaces are very dynamic, and an apple might be tradable for an orange on Sunday but not on Monday, what we mean by resource theory in this chapter is a static notion: deciding ``what buys what,'' once and for all.%
\footnote{Using some sort of temporal theory, e.g.\ the one presented in \cref{chap.temporal_topos}, one could take the notion here and have it change in time.}
This sort of static notion of conversion might occur in chemistry: the chemical reactions that are possible one day will quite likely be possible on a different day as well. Manufacturing may be somewhere in between: the set of production techniques---whereby a company can convert one set of resources into another---do not change much from day to day.

We learned about resource theories from
\cite{Coecke.Fritz.Spekkens:2016a,Fritz:2017a}, who go much further than we
will; see \cref{sec.resource_theory_further_reading} for more information.%
\index{resource!theory} In
this section we will focus only on the main idea. While there are many beautiful
mathematical examples of symmetric monoidal preorders, as we will see in
\cref{subsec.SMPs_pure_math}, there are also ad hoc examples coming from life
experience. In the next chapter, on databases, we will see the same theme: while
there are some beautiful mathematical categories out there, database schemas are
\emph{ad hoc} organizational patterns of information. Describing something as
``ad hoc'' is often considered derogatory, but it just means ``formed, arranged,
or done for a particular purpose only.'' There is nothing wrong with doing
things for a particular purpose; it's common outside of pure math and pure art. Let's get
to it.

\paragraph{Chemistry}%
\label{subsubsec.chemistry}%
\index{chemistry!monoidal preorder of}

In high school chemistry, we work with chemical equations, where material collections such as
\[\mathrm{H_2O},\quad \mathrm{NaCl},\quad \mathrm{2NaOH},\quad \mathrm{CH_4+3O_2}\]
are put together in the form of reaction equations, such as
\[\mathrm{2H_2O+2Na\to 2NaOH+H_2}.\]
The collection on the left, $\mathrm{2H_2O+2Na}$ is called the \emph{reactant}, and the collection on the right, $\mathrm{2NaOH+H_2}$ is called the \emph{product}.

We can consider reaction equations such as the one above as taking place inside a single symmetric monoidal preorder $(\Set{Mat},\to,0,+)$. Here $\Set{Mat}$ is the set of all collections of atoms and molecules, sometimes called \emph{materials}. So we have $\mathrm{NaCl}\in\Set{Mat}$ and $\mathrm{4H_2O+6Ne}\in\Set{Mat}$.

The set $\Set{Mat}$ has a preorder structure denoted by the $\to$ symbol, which is the preferred symbol in the setting of chemistry. To be clear, $\to$ is taking the place of the order relation $\leq$ from \cref{def.symm_mon_structure}. The $+$ symbol is the preferred notation for the monoidal product in the chemistry setting, taking the place of $\otimes$. While it does not come up in practice, we use $0$ to denote the monoidal unit.

\begin{exercise} %
\label{exc.reaction_equations}
Here is an exercise for people familiar with reaction equations: check that
conditions (a), (b), (c), and (d) of \cref{def.symm_mon_structure} hold.
\end{exercise}

\index{chemistry!catalysis}
An important notion in chemistry is that of catalysis: one compound \emph{catalyzes} a certain reaction. For example, one might have the following set of reactions:
\begin{equation}%
\label{eqn.three_reactions}
  y+k\to y'+k'\qquad x+y'\to z'\qquad z'+k'\to z+k
\end{equation}
Using the laws of monoidal preorders, we obtain the composed reaction
\begin{equation}%
\label{eqn.result_reaction}
	x+y+k\to x+y'+k'\to z'+k'\to z+k.
\end{equation}
Here $k$ is the catalyst because it is found both in the reactant and the product of the reaction. It is said to catalyze the reaction $x+y\to z$. The idea is that the reaction $x+y\to z$ cannot take place given the reactions in \cref{eqn.three_reactions}. But if $k$ is present, meaning if we add $k$ to both sides, the resulting reaction can take place.

The wiring diagram for the reaction in \cref{eqn.result_reaction} is shown in \cref{eqn.catalyst}. The three interior boxes correspond to the three reactions given in \cref{eqn.three_reactions}, and the exterior box corresponds to the composite reaction $x+y+k\to z+k$.
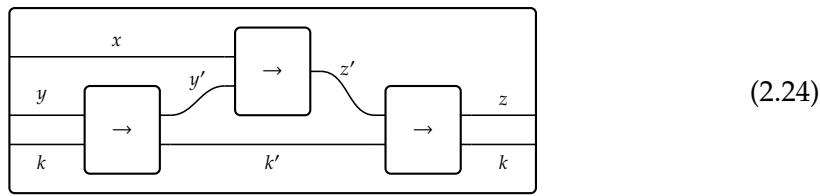
\begin{equation}%
\label{eqn.catalyst}
\begin{tikzpicture}[oriented WD, bbx=1cm, font=\tiny, baseline=(a)]
	\node[bb={2}{2}] (a1) {$\to$};
	\node[bb={2}{1}, above right=-1.5 and 1 of a1] (a2) {$\to$};
	\node[bb={2}{2}, below right=-1.5 and 1 of a2] (a3) {$\to$};
	\node[bb={0}{0}, fit=(a1) (a2) (a3)] (a) {};
	\draw (a.west|-a1_in1) to node[above] {$y$} (a1_in1);
	\draw (a.west|-a1_in2) to node[below] {$k$} (a1_in2);
	\draw (a.west|-a2_in1) to node[above] {$x$} (a2_in1);
	\draw (a1_out1) to node[above] {$y'$} (a2_in2);
	\draw (a1_out2) to node[below] {$k'$} (a3_in2);
	\draw (a2_out1) to node[above=3pt] {$z'$} (a3_in1);
	\draw (a3_out1) to node[above] {$z$} (a3_out1-|a.east);
	\draw (a3_out2) to node[below] {$k$} (a3_out2-|a.east);
\end{tikzpicture}
\end{equation}

\paragraph{Manufacturing}%
\label{subsubsec.manufacturing}
\index{manufacturing!monoidal preorder of}

Whether we are talking about baking pies, building smart phones, or following pharmaceutical recipes, manufacturing firms need to store basic recipes, and build new recipes by combining simpler recipes in schemes like the one shown in \cref{eqn.how_to_bake_pie} or \cref{eqn.catalyst}.

The basic idea in manufacturing is exactly the same as that for chemistry, except there is an important assumption we can make in manufacturing that does not hold for chemical reactions:
\begin{center}
\emph{You can trash anything you want, and it disappears from view.}
\end{center}
This simple assumption has caused the world some significant problems, but it is
still in effect. In our meringue pie example, we can ask: ``what happened to the
egg shell, or the paper wrapping the stick of butter''? The answer is they were
trashed, i.e.\ thrown in the garbage bin. It would certainly clutter our diagram
and our thinking if we had to carry these resources through the diagram:
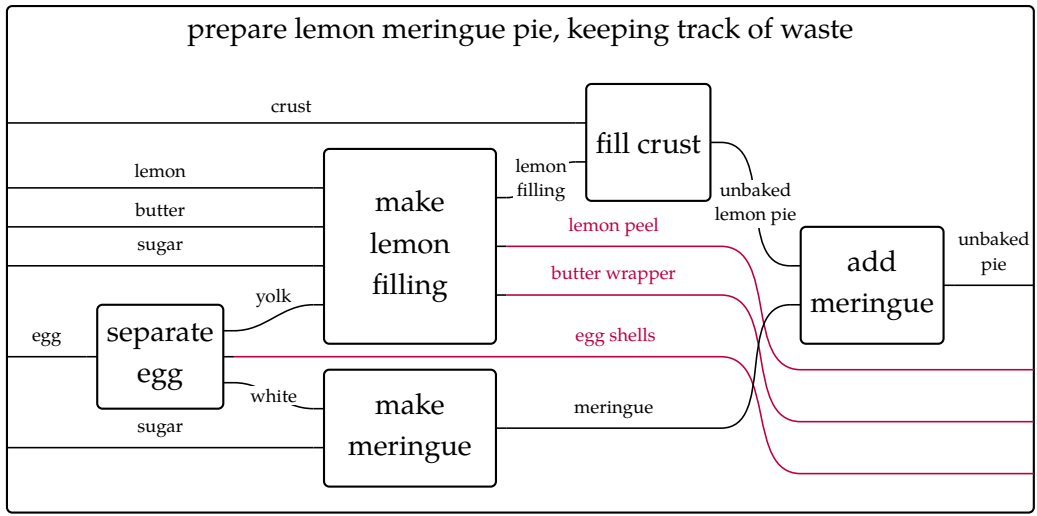
\[

\]

Instead, in our daily lives and in manufacturing, we do not have to hold on to something if we don't need it; we can just discard it. In terms of wiring diagrams, this can be shown using a new icon $%
$, as follows:%
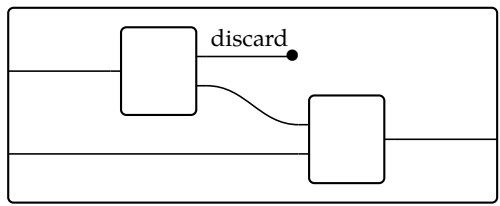
\index{icon!discard}%
\index{manufacturing!discard operation in}
\begin{equation}%
\label{eqn.discard}
%
\end{equation}

To model this concept of waste using monoidal categories, one just adds an
additional axiom to (a), (b), (c), and (d) from \cref{def.symm_mon_structure}:
\bigskip
\begin{enumerate}[label=(e)]
	\item $x\leq I$ for all $x\in X$.\hfill{(discard axiom)}%
\label{page.discard_axiom}
\end{enumerate}
\bigskip
It says that every $x$ can be converted into the monoidal unit $I$. In the notation of the chemistry section, we would write instead $x\to 0$: any $x$ yields nothing. But this is certainly not accepted in the chemistry setting. For example,
\[\mathrm{H_2O+NaCl\to^?H_2O}\]
is certainly not a legal chemical equation. It is easy to throw things away in
manufacturing, because we assume that we have access to---the ability to grab onto and directly manipulate---each item produced. In
chemistry, when you have $10^{23}$ of substance $A$ dissolved in something else,
you cannot just simply discard $A$. So axiom (e) is valid in manufacturing but
not in chemistry. 

Recall that in \cref{ssec.wirdia2} we said that there were many different styles
of wiring diagrams. Now we're saying that adding the discard axiom changes the
wiring diagram style, in that it adds this new discard icon that allows wires to terminate, as shown in
\cref{eqn.discard}. In informatics, we will change the wiring diagram style yet
again.%
\index{wiring diagram!icon of}

\paragraph{Informatics}%
\label{subsubsec.informatics}
\index{informatics!monoidal preorder of}
A major difference between information and a physical object is that information
can be copied. Whereas one cup of butter never becomes two, it is easy for a
single email to be sent to two different people. It is much easier to copy a
music file than it is to copy a CD. Here we do not mean ``copy the information
from one compact disc onto another''---of course that's easy---instead, we mean
that it's quite difficult to copy the physical disc, thereby forming a second
physical disc! In diagrams, the distinction is between the relation
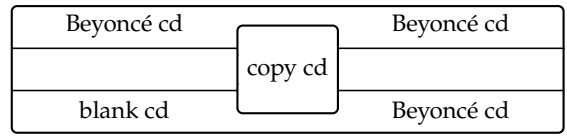
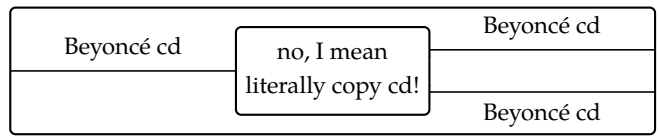
\[

\]%
\index{Beyonc\'e}
and the relation
\[
%
\]
The former is possible, the latter is magic.

Of course material objects can sometimes be copied; cell mitosis is a case in point. But 
this is a remarkable biological process, certainly not something that is expected for
ordinary material objects. In the physical world, we would make mitosis a box transforming one cell into two. But in (classical, not quantum) information, everything can be copied, so we add a new icon to our repertoire.%
\index{icon!copy}

Namely, in wiring diagram notation, copying information appears as a new icon, $%
$, allowing us to split wires:%
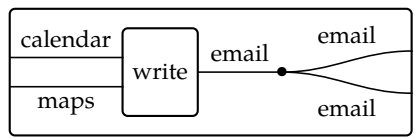
\index{notation|seealso {icon}}
\[
%
\]
Now with two copies of the email, we can send one to Alice and one to Bob.
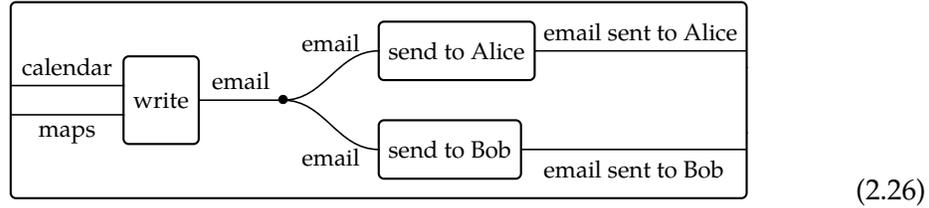
\begin{equation}%
\label{eqn.copy_pic}
%
\end{equation}

Information can also be discarded, at least in the conventional way of thinking,
so in addition to axioms (a) to (d) from \cref{def.symm_mon_structure}, we can
keep axiom (e) from page~\pageref{page.discard_axiom} and add a new \emph{copy axiom}:
\bigskip
\begin{enumerate}[label=(f)]
	\item $x\leq x+x$ for all $x\in X$.\hfill{(copy axiom)}%
\label{page.copy_axiom}
\end{enumerate}
\bigskip
allowing us to make mathematical sense of diagrams like \cref{eqn.copy_pic}.%
\index{informatics!duplication in}%
\index{informatics!discarding in}

Now that we have examples of monoidal preorders under our belts, let's discuss some
nice mathematical examples.

\index{resource!theory|)}
\subsection{Abstract examples}%
\label{subsec.SMPs_pure_math}

In this section we discuss several mathematical examples of symmetric monoidal structures on preorders.

\paragraph{The Booleans}

The simplest nontrivial preorder is the booleans: $\BB=\{\true,\false\}$ with $\false\leq\true$. There are two different symmetric monoidal structures on it.

\begin{example}[Booleans with AND]%
\label{ex.Bool}%
\index{booleans!usual monoidal structure}
We can define a monoidal structure on $\BB$ by letting the monoidal unit be $\true$ and the monoidal product be $\wedge$ (AND). If one thinks of $\false=0$ and $\true=1$, then $\wedge$ corresponds to the usual multiplication operation $*$. That is, with this correspondence, the two tables below match up:
\begin{equation}%
\label{eqn.wedge_mult_table}
\begin{array}{c | c c}
	\wedge&\false&\true\\\hline
	\false&\false&\false\\
	\true&\false&\true
\end{array}
\hspace{1in}
\begin{array}{c | c c}
	*&0&1\\\hline
	0&0&0\\
	1&0&1
\end{array}
\end{equation}

One can check that all the properties in \cref{def.symm_mon_structure} hold, so we have a monoidal preorder which we denote $\Bool\coloneqq(\BB,\leq,\true,\wedge)$.
\end{example}

$\Bool$ will be important when we get to the notion of enrichment. Enriching in a monoidal preorder $\cat{V}=(V,\leq,I,\otimes)$ means ``letting $\cat{V}$ structure the question of getting from $a$ to $b$.'' All of the structures of a monoidal preorder---i.e.\ the set $V$, the ordering relation $\leq$, the monoidal unit $I$, and the monoidal product $\otimes$---play a role in how enrichment works.

For example, let's look at the case of $\Bool=(\BB,\leq,\true,\wedge)$. The fact that its underlying set is $\BB=\{\false,\true\}$ will translate into saying that ``getting from $a$ to $b$ is a $\true/\false$ question.'' The fact that $\true$ is the monoidal unit will translate into saying ``you can always get from $a$ to $a$.'' The fact that $\wedge$ is the monoidal product will translate into saying ``if you can get from $a$ to $b$ AND you can get from $b$ to $c$ then you can get from $a$ to $c$.'' Finally, the ``if-then'' form of the previous sentence is coming from the order relation $\leq$. We will make this more precise in \cref{sec.enrichment}.

We will be able to play the same game with other monoidal preorders, as we will see after we define a monoidal preorder called $\Cost$ in \cref{ex.Lawveres_base}.

\paragraph{Some other monoidal preorders}

It is a bit imprecise to call $\Bool$ ``the'' boolean monoidal preorder, because
there is another monoidal structure on $(\BB,\leq)$, which we describe in
\cref{exc.boolean_mon_pos_II}. The first structure, however, seems to be more useful in practice than the second.

\begin{exercise}%
\label{exc.boolean_mon_pos_II}%
\index{booleans!alternative
  monoidal structure}
Let $(\BB,\leq)$ be as above, but now consider the monoidal product to be $\vee$ (OR). 
\[
\begin{array}{c | c c}
	\vee&\false&\true\\\hline
	\false&\false&\true\\
	\true&\true&\true
\end{array}
\hspace{1in}
\begin{array}{c | c c}
	\max&0&1\\\hline
	0&0&1\\
	1&1&1
\end{array}
\]
What must the monoidal unit be in order to satisfy the conditions of \cref{def.symm_mon_structure}? 

Does it satisfy the rest of the conditions?
\end{exercise}

In \cref{ex.nats_with_plus,exc.monoidal_naturals} we give two different monoidal structures on the preorder $(\NN,\leq)$ of natural numbers, where $\leq$ is the usual ordering ($0\leq 1$ and $5\leq 16$). 

\begin{example}[Natural numbers with
  addition]%
\label{ex.nats_with_plus}%
\index{natural numbers}
There is a monoidal structure on $(\NN,\leq)$ where the monoidal unit is $0$ and the monoidal product is $+$, i.e.\ $6+4=10$. It is easy to check that $x_1\leq y_1$ and $x_2\leq y_2$ implies $x_1+x_2\leq y_1+y_2$, as well as all the other conditions of \cref{def.symm_mon_structure}.
\end{example}

\begin{exercise} %
\label{exc.monoidal_naturals}
Show there is a monoidal structure on $(\NN,\leq)$ where the monoidal product is $*$, i.e.\ $6*4=24$. What should the monoidal unit be?
\end{exercise}

\begin{example}[Divisibility and multiplication]
Recall from \cref{ex.orders_on_N} that there is a ``divisibility'' order on $\NN$: we write $m|n$ to mean that $m$ divides into $n$ without remainder. So $1\vert m$ for all $m$ and $4\vert 12$.

There is a monoidal structure on $(\NN,\vert\;)$, where the monoidal unit is $1$ and the monoidal product is $*$, i.e.\ $6*4=24$. Then if $x_1\vert y_1$ and $x_2\vert y_2$, then $(x_1*x_2)\vert (y_1*y_2)$. Indeed, if there is some $p_1,p_2\in\NN$ such that $x_1*p_1=y_1$ and $x_2*p_2=y_2$, then $(p_1*p_2)*(x_1*x_2)=y_1*y_2$.
\end{example}

\begin{exercise} %
\label{exc.monoidal_naturals2}
Again taking the divisibility order $(\NN,\vert\;)$. Someone proposes $0$ as the monoidal unit and $+$ as the monoidal product. Does that proposal satisfy the conditions of \cref{def.symm_mon_structure}? Why or why not?
\end{exercise}

\begin{exercise}%
\label{exc.no_maybe_yes}%
\index{Hasse diagram}
Consider the preorder $(P,\leq)$ with Hasse diagram \fbox{$\const{no}\to\const{maybe}\to\const{yes}$}. We propose a monoidal structure with $\const{yes}$ as the monoidal unit and ``min'' as the monoidal product.
\begin{enumerate}
	\item Make sense of ``$\min$'' by filling in the multiplication table with elements of $P$.
\[
\begin{array}{c|ccc}
	\min&\const{no}&\const{maybe}&\const{yes}\\\hline
	\const{no}&\?&\?&\?\\
	\const{maybe}&\?&\?&\?\\
	\const{yes}&\?&\?&\?
\end{array}
\]
	\item Check the axioms of \cref{def.symm_mon_structure} hold for $\Cat{NMY}\coloneqq(P,\leq,\const{yes},\min)$, given your definition of $\min$. If not, change your definition of $\min$.
\qedhere
\end{enumerate}
\end{exercise}

\begin{exercise}%
\label{exc.powerset_symm_mon_pos}%
\index{power set}%
\index{power set|seealso {preorder, of subobjects}}
\index{relation!subset}
Let $S$ be a set and let $\powset(S)$ be its power set, the set of all subsets of
$S$, including the empty subset, $\varnothing\ss S$, and the ``everything''
subset, $S\ss S$. We can give $\powset(S)$ an order: $A\leq B$ is given by the
subset relation $A\ss B$, as discussed in \cref{ex.powerset}. We propose a
symmetric monoidal structure on $\powset(S)$ with monoidal unit $S$ and monoidal product given by intersection $A\cap B$.%
\index{intersection}

Does it satisfy the conditions of \cref{def.symm_mon_structure}?
\end{exercise}

\begin{exercise}%
\label{exc.propositions_preorder}%
\index{proposition}%
\index{proposition|seealso {logic}}
Let $\Prop^{\NN}$ denote the set of all mathematical statements one can make
about a natural number, where we consider two statements to be the same if one
is true if and only if the other is true. For example ``$n$ is prime'' is an
element of $\Prop^\NN$, and so are ``$n=2$'' and ``$n\geq 11$.'' The statements
``$n+2=5$'' and ``$n$ is the least odd prime'' are considered the same. Given
$P,Q\in\Prop^\NN$, we say $P\leq Q$ if for all $n\in\NN$, whenever $P(n)$ is
true, so is $Q(n)$.

Define a monoidal unit and a monoidal product on $\Prop^\NN$ that satisfy the
conditions of \cref{def.symm_mon_structure}.
\end{exercise}

\paragraph{The monoidal preorder $\Cost$}

As we said above, when we enrich in monoidal preorders we see them as different ways to structure the question of ``getting from here to there.'' We will explain this in more detail in \cref{sec.enrichment}. The following monoidal preorder will eventually structure a notion of distance or cost for getting from here to there.

\begin{example}[Lawvere's monoidal
  preorder, $\Cost$]%
\label{ex.Lawveres_base}%
\index{Cost@$\Cost$}
Let $[0,\infty]$ denote the set of nonnegative real numbers---such as $0$, $1$, $15.33\ol{3}$, and $2\pi$---together with $\infty$. Consider the preorder $([0,\infty],\geq)$, with the usual notion of $\geq$, where of course $\infty\geq x$ for all $x\in[0,\infty]$.

There is a monoidal structure on this preorder, where the monoidal unit is $0$ and the monoidal product is $+$. In particular, $x+\infty=\infty$ for any $x\in[0,\infty]$. Let's call this monoidal preorder
\[
\Cost\coloneqq([0,\infty],\geq,0,+),
\]
because we can think of the elements of $[0,\infty]$ as costs. In terms of structuring
``getting from here to there,'' $\Cost$ seems to say ``getting
from $a$ to $b$ is a question of cost.'' The monoidal unit being $0$ will
translate into saying that you can always get from $a$ to $a$ at no cost. The
monoidal product being $+$ will translate into saying that the cost of getting
from $a$ to $c$ is at most the cost of getting from $a$ to $b$ \emph{plus} the cost of getting from $b$ to $c$. Finally, the ``at most'' in the previous sentence is coming from the $\geq$.
\end{example}

\paragraph{The opposite of a monoidal preorder}

One can take the opposite of any preorder, just flip the order: $(X,\leq)\op\coloneqq(X,\geq)$; see \cref{ex.opposite}. \cref{prop.opposite_monoidal_preorder} says that if the preorder had a symmetric monoidal structure, so does its opposite.

\begin{proposition}%
\label{prop.opposite_monoidal_preorder}%
\index{monoidal preorder!opposite of}
  Suppose $\cat{X}=(X,\leq)$ is a preorder and $\cat{X}\op=(X,\geq)$ is its opposite. If $(X,\leq,I,\otimes)$ is a symmetric monoidal preorder then so is its opposite, $(X,\geq,I,\otimes)$.
\end{proposition}
\begin{proof}
  Let's first check monotonicity. Suppose $x_1\geq y_1$ and $x_2\geq y_2$ in
  $\cat{X}\op$; we need to show that $x_1\otimes x_2\geq y_1\otimes y_2$. But by
  definition of opposite order, we have $y_1\leq x_1$ and $y_2\leq x_2$ in
  $\cat{X}$, and thus $y_1\otimes y_2\leq x_1\otimes x_2$ in $\cat{X}$. Thus
  indeed $x_1\otimes x_2\geq y_1\otimes y_2$ in $\cat{X}\op$. The other three
  conditions are even easier; see \cref{exc.complete_op_proof}.
\end{proof}

\begin{exercise}%
\label{exc.complete_op_proof}
  Complete the proof of \cref{prop.opposite_monoidal_preorder} by proving that the three remaining conditions of \cref{def.symm_mon_structure} are satisfied.
\end{exercise}

\begin{exercise} %
\label{exc.costop}
  Since $\Cost$ is a symmetric monoidal preorder, \cref{prop.opposite_monoidal_preorder} says that $\Cost\op$ is too.
  \begin{enumerate}
    \item What is $\Cost\op$ as a preorder?
    \item What is its monoidal unit?
    \item What is its monoidal product?
  \qedhere
\end{enumerate}
\end{exercise}

\subsection{Monoidal monotone maps}%
\index{monoidal monotone|(}
Recall from \cref{ex:part_from_preorder} that for any preorder $(X,\leq)$, there is
an induced equivalence relation $\cong$ on $X$, where $x\cong x'$ iff both
$x\leq x'$ and $x'\leq x$.%
\index{equivalence relation!generated by a preorder}

\begin{definition}%
\label{def.monoidal_functor}
Let $\cat{P}=(P,\leq_P, I_P, \otimes_P)$ and $\cat{Q}=(Q,\leq_Q, I_Q,\otimes_Q)$
be monoidal preorders. A \emph{monoidal monotone} from $\cat{P}$ to $\cat{Q}$ is a
monotone map $f\colon(P,\leq_P)\to (Q,\leq_Q)$, satisfying two conditions:
\begin{enumerate}[label=(\alph*)]
	\item $I_Q\leq_Q f(I_P)$, and
	\item $f(p_1)\otimes_Qf(p_2)\leq_Q f(p_1\otimes_Pp_2)$
\end{enumerate}
for all $p_1,p_2\in P$.

There are strengthenings of these conditions that are also important. If $f$
satisfies the following conditions, it is called a \emph{strong monoidal
monotone}:
\begin{enumerate}[label=(\alph*')]
	\item $I_Q\cong f(I_P)$, and
	\item $f(p_1)\otimes_Qf(p_2)\cong f(p_1\otimes_Pp_2)$;
\end{enumerate}
and if it satisfies the following conditions it is called a \emph{strict
monoidal monotone}:
\begin{enumerate}[label=(\alph*'')]
	\item $I_Q= f(I_P)$, and
	\item $f(p_1)\otimes_Qf(p_2)= f(p_1\otimes_Pp_2)$.
\end{enumerate}
\end{definition}

Monoidal monotones are examples of \emph{monoidal functors}%
\index{monoidal functor!monoidal monotone as}, which we will see various incarnations of throughout the book; see \cref{roughdef.monoidal_functor}. What we call monoidal monotones could also be called \emph{lax monoidal monotones}, and there is a dual notion of \emph{oplax monoidal monotones}, where the inequalities in (a) and (b) are reversed; we will not use oplaxity in this book.%
\index{dual!of lax monoidal monotone}

\label{def.equiv_iso_mon_preorders}
%

\begin{example}
There is a monoidal monotone $i\colon(\NN,\leq,0,+)\to(\RR,\leq,0,+)$, where $i(n)=n$ for all $n\in\NN$. It is clearly monotonic, $m\leq n$ implies $i(m)\leq i(n)$. It is even strict monoidal because $i(0)=0$ and $i(m+n)=i(m)+i(n)$.

There is also a monoidal monotone $f\colon(\RR,\leq,0,+)\to(\NN,\leq,0,+)$ going
the other way. Here $f(x)\coloneqq\floor{x}$ is the floor function, e.g.
$f(3.14)=3$. It is monotonic because $x\leq y$ implies $f(x)\leq f(y)$. Also
$f(0)=0$ and $f(x)+f(y)\leq f(x+y)$, so it is a monoidal monotone. But it is not
strict or even strong because $f(0.5)+f(0.5)\neq f(0.5+0.5)$.
\end{example}

Recall $\Bool=(\BB,\leq,\true,\wedge)$ from \cref{ex.Bool} and
$\Cost=([0,\infty],\ge,0,+)$ from \cref{ex.Lawveres_base}. There is a monoidal
monotone $g\colon\Bool\to\Cost$, given by $g(\false)\coloneqq \infty$ and
$g(\true)\coloneqq 0$.
\begin{exercise}%
\label{exc.bool_to_cost}~
\begin{enumerate}
	\item Check that the map $g\colon(\BB,\leq,\true,\wedge)\to([0,\infty],\geq,0,+)$ presented above indeed
  \begin{itemize}
  	\item is monotonic,
  	\item satisfies condition (a) of \cref{def.monoidal_functor}, and
  	\item satisfies condition (b) of \cref{def.monoidal_functor}.
  \end{itemize}
  \item Is $g$ strict?
  \qedhere
\qedhere
\end{enumerate}
\end{exercise}

\begin{exercise}%
\label{exc.bool_to_cost_inverses}
Let $\Bool$ and $\Cost$ be as above, and consider the following quasi-inverse functions $d,u\colon[0,\infty]\to\BB$ defined as follows:
\[
  d(x)\coloneqq
  \begin{cases}
  	\false&\tn{ if }x>0\\
		\true&\tn{ if } x=0
	\end{cases}
\hspace{1in}
  u(x)\coloneqq
  \begin{cases}
  	\false&\tn{ if }x=\infty\\
		\true&\tn{ if } x<\infty
	\end{cases}
\]
\begin{enumerate}
	\item Is $d$ monotonic?
	\item Does $d$ satisfy conditions (a) and (b) of \cref{def.monoidal_functor}?
	\item Is $d$ strict?
	\item Is $u$ monotonic?
	\item Does $u$ satisfy conditions (a) and (b) of \cref{def.monoidal_functor}?
	\item Is $u$ strict?
	\qedhere
\qedhere
\end{enumerate}
\end{exercise}

\begin{exercise}%
\label{exc.natural_monotone}
\begin{enumerate}
	\item Is $(\NN,\leq,1,*)$ a monoidal preorder, where $*$ is the usual multiplication of natural numbers?
	\item If not, why not? If so, does there exist a monoidal monotone $(\NN,\leq,0,+)\to(\NN,\leq,1,*)$? If not; why not? If so, find it.
	\item Is $(\zz,\leq,*,1)$ a monoidal preorder?
\qedhere
\end{enumerate}
\end{exercise}

\index{monoidal preorder|)}%
\index{monoidal monotone|)}

\section{Enrichment}%
\label{sec.enrichment}%
\index{enrichment|(}

In this section we will introduce $\cat{V}$-categories, where $\cat{V}$ is a
symmetric monoidal preorder. We will see that $\Bool$-categories are preorders, and
that $\Cost$-categories are a nice generalization of the notion of metric space.

\subsection{$\cat{V}$-categories}%
\label{subsec.V_cats}%
\index{V-category@$\cat{V}$-category|see {enrichment}}

While $\cat{V}$-categories can be defined even when $\cat{V}$ is not symmetric,
i.e.\ just obeys conditions (a)--(c) of \cref{def.symm_mon_structure}, certain things don't work quite right. For example, we will see later in \cref{exc.check_enriched_prod} that the symmetry condition is necessary in order for products of $\cat{V}$-categories to exist. Anyway, here's the definition.

\begin{definition}%
\label{def.cat_enriched_mpos}%
\index{enrichment}
Let $\cat{V}=(V,\leq,I,\otimes)$ be a symmetric monoidal preorder. A \emph{$\cat{V}$-category} $\cat{X}$ consists of two constituents, satisfying two properties. To specify $\cat{X}$, 
\begin{enumerate}[label=(\roman*)]
	\item one specifies a set $\Ob(\cat{X})$, elements of which are called \emph{objects};
	\item for every two objects $x,y$, one specifies an element $\cat{X}(x,y)\in V$, called the \emph{hom-object}.%
	\tablefootnote{The word ``hom'' is short for \emph{homomorphism} and reflects the origins of this subject. A more descriptive name for $\cat{X}(x,y)$ might be \emph{mapping object}, but we use ``hom'' mainly because it is an important jargon word to know in the field.}
\end{enumerate}
\index{hom object}
\index{mapping object!see {hom object}}
\index{closed category!hom object in|see {hom object}}
The above constituents are required to satisfy two properties:
\begin{enumerate}[label=(\alph*)]
	\item for every object $x\in\Ob(\cat{X})$ we have $I\leq\cat{X}(x,x)$, and
	\item for every three objects $x,y,z\in\Ob(\cat{X})$, we have $\cat{X}(x,y)\otimes\cat{X}(y,z)\leq\cat{X}(x,z)$.
\end{enumerate}
We call $\cat{V}$ the \emph{base of the enrichment} for $\cat{X}$ or say that
$\cat{X}$ is \emph{enriched} in $\cat{V}$.%
\index{enrichment!base of}
\end{definition}

\begin{example} %
\label{ex.preorder_as_boolcat}%
\index{preorder!as $\Bool$-category}%
\index{enriched category!preorders as}
As we shall see in the next subsection, from every preorder we can construct
a $\Bool$-category, and vice versa. So, to get a feel for
$\cat{V}$-categories, let us consider the preorder generated by the Hasse diagram:
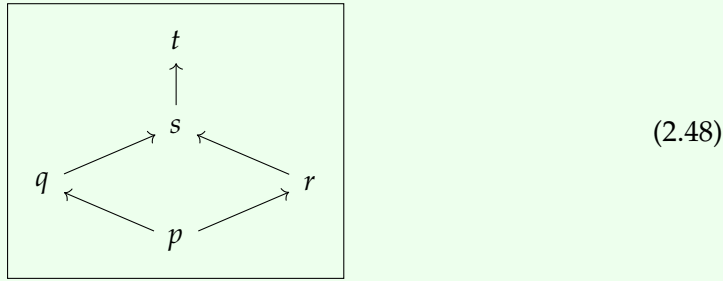
\begin{equation}%
\label{eqn.random_hasse91}
\begin{tikzpicture}[x=.7in, y=.3in, inner sep=5pt, baseline=(A.center)]
  	\node (a) at (0,4.5) {$t$};
  	\node (b) at (0,3) {$s$};
  	\node (c) at (-1,2) {$q$};
  	\node (d) at (1,2) {$r$};
  	\node (e) at (0,1) {$p$};
  	\node[draw, fit=(a) (b) (c) (d) (e)] (A) {};
  	\draw[->] (b) to (a);
  	\draw[->] (c) to (b);
  	\draw[->] (d) to (b);
  	\draw[->] (e) to (c);
  	\draw[->] (e) to (d);
\end{tikzpicture}
\end{equation}

How does this correspond to a $\Bool$-category $\cat{X}$? Well, the objects of
$\cat{X}$ are simply the elements of the preorder, i.e.\ $\Ob(\cat{X})=\{p,q,r,s,t\}$. Next, for
every pair of objects $(x,y)$ we need an element of $\BB=\{\false,\true\}$:
simply take $\true$ if $x \le y$, and $\false$ if otherwise. So for example,
since $s\leq t$ and $t\not\leq s$, we have $\cat{X}(s,t)=\true$ and
$\cat{X}(t,s)=\false$. Recalling from \cref{ex.Bool} that the
monoidal unit $I$ of $\Bool$ is $\true$, it's straightforward to check that
this obeys both (a) and (b), so we have a $\Bool$-category.

In general, it's sometimes convenient to represent a $\cat{V}$-category $\cat{X}$ with a
square matrix. The rows and columns of the matrix correspond to the objects of
$\cat{X}$, and the $(x,y)$th entry is simply the hom-object $\cat{X}(x,y)$. So,
for example, the above preorder in \cref{eqn.random_hasse91} can be
represented by the matrix%
\index{hom object}
\[
\begin{array}{c|ccccc}
  \rotatebox{45}{$\scriptstyle\cdot\leq\cdot$}&p&q&r&s&t\\\hline
  p&\true&\true&\true&\true&\true\\
  q&\false&\true&\false&\true&\true\\
  r&\false&\false&\true&\true&\true\\
  s&\false&\false&\false&\true&\true\\
  t&\false&\false&\false&\false&\true
\end{array}
\qedhere
\]
\end{example}

\subsection{Preorders as $\Bool$-categories}%
\label{subsec.preorders_Bool_enriched}

Our colleague Peter Gates has called category theory ``a primordial ooze,'' because so much of it can be defined in terms of other parts of it. There is nowhere to rightly call the beginning, because that beginning can be defined in terms of something else. So be it; this is part of the fun.
\index{primordial ooze|seealso {ooze, primordial}}%
\index{ooze!primordial|see {primordial ooze}}

\begin{theorem}%
\label{thm.preorder_is_bool_cat}%
\index{preorder!as $\Bool$-category}%
\index{booleans!as base of enrichment for preorders}
There is a one-to-one correspondence between preorders and $\Bool$-categories.
\end{theorem}

Here we find ourselves in the ooze, because we are saying that preorders are the same as $\Bool$-categories, whereas $\Bool$ is itself a preorder. ``So then $\Bool$ is like... enriched in itself?'' Yes, every preorder, including $\Bool$, is enriched in $\Bool$, as we will now see.

\begin{proof}[Proof of \cref{thm.preorder_is_bool_cat}]
Let's check that we can construct a preorder from any $\Bool$-category. Since $\BB=\{\false,\true\}$, \cref{def.cat_enriched_mpos} says a
$\Bool$-category consists of two things:
\begin{enumerate}[label=(\roman*)]
	\item a set $\Ob(\cat{X})$, and
	\item for every $x,y\in \Ob(\cat{X})$ an element $\cat{X}(x,y)\in\BB$, i.e.\ either $\cat{X}(x,y)=\true$ or $\cat{X}(x,y)=\false$. 
\end{enumerate}
We will use these two things to begin forming a preorder whose elements are the objects of $\cat{X}$. So let's call the preorder $(X,\leq)$, and let $X\coloneqq\Ob(\cat{X})$. For the $\leq$ relation, let's declare $x\leq y$ iff $\cat{X}(x,y)=\true$. We have the makings of a preorder, but for it to work, the $\leq$ relation must be reflexive and transitive. Let's see if we get these from the properties guaranteed by \cref{def.cat_enriched_mpos}:
\begin{enumerate}[label=(\alph*)]
	\item for every element $x\in X$ we have $\true\leq\cat{X}(x,x)$,
	\item for every three elements $x,y,z\in X$ we have $\cat{X}(x,y)\wedge\cat{X}(y,z)\leq\cat{X}(x,z)$.
\end{enumerate}
For $b\in\Bool$, if $\true\leq b$ then $b=\true$, so the first statement says
$\cat{X}(x,x)=\true$, which means $x\leq x$. For the second statement, one can
consult \cref{eqn.wedge_mult_table}. Since $\false\leq b$ for all $b\in\BB$, the
only way statement (b) has any force is if $\cat{X}(x,y)=\true$ and $\cat{X}(y,z)=\true$, in which case it forces $\cat{X}(x,z)=\true$. This condition exactly translates as saying that $x\leq y$ and $y\leq z$ implies $x\leq z$. Thus we have obtained reflexivity and transitivity from the two axioms of $\Bool$-categories. 

In \cref{ex.preorder_as_boolcat}, we constructed a $\Bool$-category from a preorder.
We leave it to the reader to generalize this example and show that the two
constructions are inverses; see \cref{ex.preorders_as_boolcats}.
\end{proof}

\begin{exercise}%
\label{ex.preorders_as_boolcats}
  \begin{enumerate}
    \item Start with a preorder $(P,\le)$, and use it to define a $\Bool$-category as we did
      in \cref{ex.preorder_as_boolcat}. In the proof of \cref{thm.preorder_is_bool_cat} we
      showed how to turn that $\Bool$-category back into a preorder. Show that doing so,
      you get the preorder you started with. 

    \item Similarly, show that if you turn a $\Bool$-category into a preorder using the
      above proof, and then turn the preorder back into a $\Bool$-category using your
      method, you get the $\Bool$-category you started with.
      \qedhere
  \end{enumerate}
\end{exercise}

We now discuss a beautiful application of the notion of enriched categories: metric spaces.

\subsection{Lawvere metric spaces} %
\label{subsec.Lawv_metric_spaces}%
\index{Lawvere metric space|(}%
\index{Lawvere metric space!as $\Cost$-category}

Metric spaces offer a precise way to describe spaces of points, each pair of
which is separated by some distance. Here is the usual definition:

\begin{definition}%
\label{def.ord_metric_space}%
\index{metric space!ordinary}
\index{metric space!extended}
A \emph{metric space} $(X,d)$ consists of:
\begin{enumerate}[label=(\roman*)]
	\item a set $X$, elements of which are called \emph{points}, and
	\item a function $d\colon X\times X\to\RR_{\ge 0}$, where $d(x,y)$ is called the \emph{distance between $x$ and $y$}.
\end{enumerate}	
These constituents must satisfy four properties:
\begin{enumerate}[label=(\alph*)]
	\item for every $x\in X$, we have $d(x,x)=0$,
	\item for every $x,y\in X$, if $d(x,y)=0$ then $x=y$,
	\item for every $x,y\in X$, we have $d(x,y)=d(y,x)$, and
	\item for every $x,y,z\in X$, we have $d(x,y)+d(y,z)\geq d(x,z)$.
\end{enumerate}
The fourth property is called the \emph{triangle inequality}.

If we ask instead in (ii) for a function $d\colon X \times X \to [0,\infty]= \RR_{\ge 0}
\cup \{\infty\}$, we call $(X,d)$ an \emph{extended} metric space.
\end{definition}
\index{triangle inequality}

The triangle inequality says that when plotting a route from $x$ to $z$, the distance is always at most what you get by choosing an intermediate point $y$ and going $x\to y\to z$.
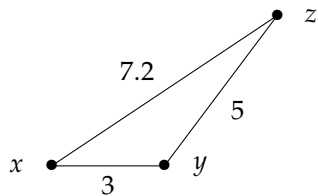
\[
\begin{tikzpicture}[font=\small, x=.5cm, y=.5cm]
	\node (A) {$\bullet$};
	\node at ($(A.center)+(3,0)$) (B) {$\bullet$};
	\node at ($(B.center)+(3,4)$) (C) {$\bullet$};
	\node[left=0 of A] {$x$};	
	\node[right=0 of B] {$y$};	
	\node[right=0 of C] {$z$};
	\draw (A.center) to node[below] {$3$} (B.center);
	\draw (B.center) to node[below right] {$5$} (C.center);
	\draw (A.center) to node[above left] {$7.2$} (C.center);
\end{tikzpicture}
\]
It can be invoked three different ways in the above picture: $3+5\geq 7.2$, but also $5+7.2\geq 3$ and $3+7.2\geq 5$. Oh yeah, and $5+3\geq 7.2$, $7.2+5\geq 3$ and $7.2+3\geq 5$.

The triangle inequality wonderfully captures something about distance, as does
the fact that $d(x,x)=0$ for any $x$. However, the other two conditions are not
quite as general as we would like. Indeed, there are many examples of things
that ``should'' be metric spaces, but which do not satisfy conditions (b) or
(c) of \cref{def.ord_metric_space}.

For example, what if we take $X$ to be places in your neighborhood, but instead of measuring distance, you want $d(x,y)$ to measure \emph{effort} to get from $x$ to $y$. Then if there are any hills, the symmetry axiom, $d(x,y)=^? d(y,x)$, fails: it's easier to get from $x$ downhill to $y$ then to go from $y$ uphill to $x$.%
\index{effort!as metric}%
\index{symmetry!lack of for effort}

Another way to find a model that breaks the symmetry axiom is to imagine that the elements of $X$ are not points, but whole regions such as the US, Spain, and Boston. Say that the distance from region $A$ to region $B$ is understood using the setup ``I will put you in an arbitrary part of $A$ and you just have to get anywhere in $B$; what is the distance in the worst-case scenario?'' So $d(\mathrm{US}, \mathrm{Spain})$ is the distance from somewhere in the western US to the western tip of Spain: you just have to get into Spain, but you start in the worst possible part of the US for doing so.%
\index{Lawvere metric space!of regions}

\begin{exercise}%
\label{ex.regions_of_world}
Which distance is bigger under the above description, $d(\mathrm{Spain}, \mathrm{US})$ or $d(\mathrm{US}, \mathrm{Spain})$?
\end{exercise}

This notion of distance, which is strongly related to something called \emph{Hausdorff distance},%
\footnote{
The Hausdorff distance gives a metric on the set of all subsets $U\ss X$ of a given metric space $(X,d)$. One first defines
\[d_L(U,V)\coloneqq \sup_{u\in U}\underset{v \in V}{\inf\vphantom{p}}d(u,v),\]
and this is exactly the formula we intend above; the result will be a Lawvere
metric space. However, if one wants the Hausdorff distance to define a
(symmetric) metric, as in \cref{def.ord_metric_space}, one must take the above
formula and symmetrize it: $d(U,V)\coloneqq\max(d_L(U,V),d_L(V,U))$. We happen
to see the unsymmetrized notion as more interesting.
}%
\index{Hausdorff distance}
will again satisfy the triangle inequality, but it violates the symmetry condition. It also violates another condition, because $d(\mathrm{Boston},\mathrm{US})=0$. No matter where you are in Boston, the distance to the nearest point of the US is 0. On the other hand, $d(\mathrm{US},\mathrm{Boston})\neq 0$.

Finally, one can imagine a use for distances that are not finite. In terms of my effort, the distance from here to Pluto is $\infty$, and it would not be any better if Pluto was still a planet. Similarly, in terms of Hausdorff distance, discussed above, the distance between two regions is often infinite, e.g.\ the distance between $\{r\in\RR\mid r<0\}$ and $\{0\}$ as subsets of $(\RR,d)$ is infinite.

When we drop conditions (b) and (c) and allow for infinite distances, we get the following relaxed notion of metric space, first proposed by Lawvere. Recall the symmetric monoidal preorder $\Cost=([0,\infty],\geq,0,+)$ from \cref{ex.Lawveres_base}.

\begin{definition} %
\label{def.Lawvere_metric_space}%
\index{metric space!as $\Cost$-category}%
\index{enriched category!metric space as}
A \emph{Lawvere metric space} is a $\Cost$-category.
\end{definition}

This is a very compact definition, but it packs a punch. Let's work out what it means, by relating it to the usual definition of metric space. By \cref{def.cat_enriched_mpos}, a $\Cost$-category $\cat{X}$  consists of:
\begin{enumerate}[label=(\roman*)]
	\item a set $\Ob(\cat{X})$,
	\item for every $x,y\in \Ob(\cat{X})$ an element $\cat{X}(x,y)\in[0,\infty]$. 
\end{enumerate}
Here the set $\Ob(\cat{X})$ is playing the role of the set of points, and $\cat{X}(x,y)\in[0,\infty]$ is playing the role of distance, so let's write a little translator:
\[X\coloneqq\Ob(\cat{X})\qquad d(x,y)\coloneqq\cat{X}(x,y).\]
The properties of a category enriched in $\Cost$ are:
\begin{enumerate}[label=(\alph*)]
	\item $0\geq d(x,x)$ for all $x\in X$, and
	\item $d(x,y)+d(y,z)\geq d(x,z)$ for all $x,y,z\in X$.
\end{enumerate}
Since $d(x,x)\in[0,\infty]$, if $0\geq d(x,x)$ then $d(x,x)=0$. So the first
condition is equivalent to the first condition from \cref{def.ord_metric_space},
namely $d(x,x)=0$. The second condition is the triangle inequality.

\begin{example}%
\label{ex.reals_as_metric_space}%
\index{real numbers!as metric
  space}
The set $\RR$ of real numbers can be given a metric space structure, and hence a Lawvere metric space structure. Namely $d(x,y)\coloneqq|y-x|$, the absolute value of the difference. So $d(3,7)=4$.
\end{example}

\begin{exercise}%
\label{exc.finite_Lawvere}
Consider the symmetric monoidal preorder $(\RR_{\geq0},\geq,0,+)$, which is almost the same as $\Cost$, except it does not include $\infty$. How would you characterize the difference between a Lawvere metric space and a $(\RR_{\geq0},\geq,0,+)$-category in the sense of \cref{def.cat_enriched_mpos}?
\end{exercise}

\paragraph{Presenting metric spaces with weighted graphs}%
\index{metric space!presentation of}%
\index{presentation!of metric space}%
\index{Hasse diagram!weighted graph as}%
\index{Hasse diagram!for metric spaces}
Just as one can convert a Hasse diagram into a preorder, one can convert any
weighted graph---a graph whose edges are labeled with numbers $w\geq 0$---into a
Lawvere metric space. In fact, we shall consider these as graphs labelled with
elements of $[0,\infty]$, and more precisely call them $\Cost$-weighted
graphs.\footnote{This generalizes Hasse diagrams, which we could
call $\Bool$-weighted graphs---the edges of a Hasse diagram are thought of as
weighted with $\true$; we simply ignore any edges that are weighted with
$\false$, and neglect to even draw them!}%
\index{graph!weighted} 

One might think of a $\Cost$-weighted graph as describing a city with some
one-way roads (a two-way road is modeled as two one-way roads), each having some
effort-to-traverse, which for simplicity we just call length.  For example,
consider the following weighted graphs:
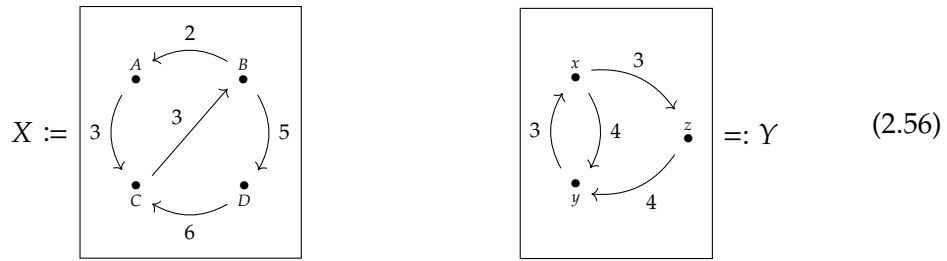
\begin{equation}%
\label{eqn.cities_distances}

\end{equation}
Given a weighted graph, one forms a metric $d_X$ on its set $X$ of vertices by setting $d(p,q)$ to be the length of the shortest path from $p$ to $q$. For example, here is the the table of distances for $Y$
\begin{equation}%
\label{eqn.distances_Y}
\begin{array}{c|ccc}
d(\nearrow)&x&y&z\\\hline
x&0&4&3\\
y&3&0&6\\
z&7&4&0
\end{array}
\end{equation}

\begin{exercise}%
\label{exc.distance_matrix_X}%
\index{matrix!of distances}
Fill out the following table of distances in the weighted graph $X$ from \cref{eqn.cities_distances}
\[
\begin{array}{c|cccc}
  d(\nearrow)&A&B&C&D\\\hline
  A&0&\?&\?&\?\\
  B&2&\?&5&\?\\
  C&\?&\?&\?&\?\\
  D&\?&\?&\?&\?
\end{array}
\qedhere
\]
\end{exercise}

Above we converted a weighted graph $G$, e.g.\ as shown in \cref{eqn.cities_distances}, into a table of distances, but this takes a bit of thinking. There is a more direct construction for taking $G$ and getting a square matrix $M_G$, whose rows and columns are indexed by the vertices of $G$. To do so, set $M_G$ to be $0$ along the diagonal, to be $\infty$ wherever an edge is missing, and to be the edge weight if there is an edge.

For example, the matrix associated to $Y$ in \cref{eqn.cities_distances} would be
\begin{equation}%
\label{eqn.matrix_Y}
M_Y\coloneqq
\begin{array}{c|ccc}
  \nearrow&x&y&z\\\hline
  x & 0 & 4 & 3\\
  y & 3 & 0 & \infty\\
  z & \infty & 4 & 0
\end{array}
\end{equation}
As soon as you see how we did this, you'll understand that it takes no thinking
to turn a weighted graph $G$ into a matrix $M_G$ in this way. We will see later
in \cref{subsubsec.nav_matrix_mult} that the more difficult ``distance
matrices'' $d_Y$, such as \cref{eqn.distances_Y}, can be obtained from the easy
graph matrices $M_Y$, such as \cref{eqn.matrix_Y}, by repeating a certain sort of ``matrix multiplication.''

\begin{exercise}%
\label{exc.adjacency_matrix_X}%
\index{matrix}
Fill out the matrix $M_X$ associated to the graph $X$ in \cref{eqn.cities_distances}:
\[
M_X=
\begin{array}{c|cccc}
  \nearrow&A&B&C&D\\\hline
  A&0&\?&\?&\?\\
  B&2&0&\infty&\?\\
  C&\?&\?&\?&\?\\
  D&\?&\?&\?&\?
\end{array}
\qedhere
\]
\end{exercise}

\subsection{$\cat{V}$-variations on preorders and metric spaces}%
\label{subsec.variations_quantale}

We have told the story of $\Bool$ and $\Cost$. But in \cref{subsec.SMPs_pure_math} we gave examples of many other monoidal preorders, and each one serves as the base of enrichment for a kind of enriched category. Which of them are useful? Something only becomes useful when someone finds a use for it. We will find uses for some and not others, though we encourage readers to think about what it would mean to enrich in the various monoidal categories discussed above; maybe they can find a use we have not explored.

\begin{exercise}%
\label{exc.NMY_cats}
Recall the monoidal preorder $\Cat{NMY}\coloneqq(P,\leq,\const{yes},\min)$ from
\cref{exc.no_maybe_yes}. Interpret what a $\Cat{NMY}$-category is.
\end{exercise}

In the next two exercises, we use $\cat{V}$-weighted graphs to construct
$\cat{V}$-categories. This is possible because we will use preorders that, like
$\Bool$ and $\Cost$, have joins.

\begin{exercise}%
\label{ex.modes_of_transport}%
\index{modes of transport}
Let $M$ be a set and let $\cat{M}\coloneqq(\powset(M),\ss,M,\cap)$ be the monoidal preorder whose elements are subsets of $M$.

Someone gives the following interpretation, ``for any set $M$, imagine it as the set of modes of transportation (e.g.\ car, boat, foot). Then an $\cat{M}$-category $\cat{X}$ tells you all the modes that will get you from $a$ all the way to $b$, for any two points $a,b\in\Ob(\cat{X})$.''
\begin{enumerate}
	\item Draw a graph with four vertices and four or five edges, each labeled with a subset of $M=\{\mathrm{car}, \mathrm{boat}, \mathrm{foot}\}$.	
	\item From this graph is it possible to construct an $\cat{M}$-category,
	where the hom-object from $x$ to $y$ is computed as follows: for each
	path $p$ from $x$ to $y$, take the intersection of the sets
	labelling the edges in $p$. Then, take the union of the these sets over
	all paths $p$ from $x$ to $y$. Write out the corresponding four-by-four
	matrix of hom-objects, and convince yourself that this is indeed an
	$\cat{M}$-category.%
\index{hom object}%
\index{union}
	\item Does the person's interpretation look right, or is it subtly mistaken somehow?
	\qedhere
\end{enumerate}
\end{exercise}

\begin{exercise}%
\label{exc.weight_limit}
Consider the monoidal preorder $\cat{W}\coloneqq(\NN\cup\{\infty\}, \leq, \infty,
\min)$.
\begin{enumerate}
	\item Draw a small graph labeled by elements of $\NN\cup\{\infty\}$. 
	\item Write out the matrix whose rows and columns are indexed by the
	nodes in the graph, and whose $(x,y)$th entry is given by the
	\emph{maximum} over all paths $p$ from $x$ to $y$ of the \emph{minimum}
	edge label in $p$.  
	\item Prove that this matrix is the matrix of hom-objects for a
	$\cat{W}$-category. This will give you a feel for how $\cat{W}$ works.%
\index{hom object!matrix of}
	\item Make up an interpretation, like that in
	\cref{ex.modes_of_transport}, for how to imagine enrichment in
	$\cat{W}$.
	\qedhere
\end{enumerate}
\erase{Weight-limit, e.g.\ for trucking.}
\end{exercise}

\index{Lawvere metric space|)}

\section{Constructions on $\cat{V}$-categories} %
\label{sec.vcat_constructions}

Now that we have a good intuition for what $\cat{V}$-categories are, we give
three examples of what can be done with $\cat{V}$-categories. The first
(\cref{subsec.change_base_mon_fun}) is known as change of base. This allows us
to use a monoidal monotone $f\colon \cat{V} \to \cat{W}$ to construct
$\cat{W}$-categories from $\cat{V}$-categories. The second construction
(\cref{subsec.enriched_functors}), that of $\cat{V}$-functors, allows us to
complete the analogy: a preorder is to a $\Bool$-category as a monotone map is to
what? The third construction (\cref{subsec.enriched_functors}) is known as a
$\cat{V}$-product, and gives us a way of combining two $\cat{V}$-categories.

\subsection{Changing the base of enrichment}%
\label{subsec.change_base_mon_fun}%
\index{enrichment!change of base}
Any monoidal monotone $\cat{V}\to\cat{W}$ between symmetric monoidal preorders lets us convert $\cat{V}$-categories into $\cat{W}$-categories.

\begin{construction}%
\label{const.mon_fun_base_change}%
\index{change of base|see {enrichment, change of base}}%
\index{enrichment!change of base}
Let $f\colon\cat{V}\to\cat{W}$ be a monoidal monotone. Given a $\cat{V}$-category $\cat{C}$, one forms the associated $\cat{W}$-category, say $\cat{C}_f$ as follows.
\begin{enumerate}[label=(\roman*)]
	\item We take the same objects: $\Ob(\cat{C}_f)\coloneqq\Ob(\cat{C})$.
	\item For any $c,d\in\Ob(\cat{C})$, put $\cat{C}_f(c,d)\coloneqq f(\cat{C}(c,d))$.
\end{enumerate}
\end{construction}

This construction $\cat{C}_f$ does indeed obey the definition of a
$\cat{W}$-category, as can be seen by applying
\cref{def.monoidal_functor} (of monoidal monotone) and \cref{def.cat_enriched_mpos}
(of $\cat{V}$-category):
\begin{enumerate}[label=(\alph*)]
	\item for every $c\in\cat{C}$, we have
	  \begin{align*}
	    I_W &\leq f(I_V) \tag{$f$ is monoidal monotone}\\
	    &\leq f(\cat{C}(c,c)) \tag{$\cat{C}$ is $\cat{V}$-category}\\
	    &=\cat{C}_f(c,c) \tag{definition of $\cat{C}_f$}
	  \end{align*}
	\item for every $c,d,e\in\Ob(\cat{C})$ we have
	\begin{align*}
		\cat{C}_f(c,d)\otimes_W\cat{C}_f(d,e)&=f(\cat{C}(c,d))\otimes_Wf(\cat{C}(d,e))
		\tag{definition of $\cat{C}_f$}\\
		&\leq f\big(\cat{C}(c,d)\otimes_V\cat{C}(d,e)\big)
		\tag{$f$ is monoidal monotone}\\
		&\leq f(\cat{C}(c,e))
		\tag{$\cat{C}$ is $\cat{V}$-category}\\
		&=\cat{C}_f(c,e)
		\tag{definition of $\cat{C}_f$}
	\end{align*}
\end{enumerate}

\begin{example}
As an example, consider the function $f\colon[0,\infty]\to\{\true,\false\}$ given by
\begin{equation}%
\label{eqn.Lawv_to_Bool1}
	f(x)\coloneqq
	\begin{cases}
		\true&\tn{ if }x=0\\
		\false&\tn{ if }x>0
	\end{cases}
\end{equation}
It is easy to check that $f$ is monotonic and that $f$ preserves the monoidal
product and monoidal unit; that is, it's easy to show that $f$ is a monoidal
monotone. (Recall \cref{exc.bool_to_cost_inverses}.)

Thus $f$ lets us convert Lawvere metric spaces into preorders.%
\index{Lawvere metric space}
\end{example}

\begin{exercise} %
\label{exc.regions_preorder}
Recall the ``regions of the world'' Lawvere metric space from
\cref{ex.regions_of_world} and the text above it. We just learned that, using
the monoidal monotone $f$ in \cref{eqn.Lawv_to_Bool1},  we can convert it to a
preorder. Draw the Hasse diagram for the preorder corresponding to the regions:
US, Spain, and Boston. How could you interpret this preorder relation?
\end{exercise}

\begin{exercise} %
\label{exc.metric_to_poset}
\begin{enumerate}
	\item Find another monoidal monotone $g\colon\Cost\to\Bool$ different from the one defined in \cref{eqn.Lawv_to_Bool1}.
	\item Using \cref{const.mon_fun_base_change}, both your monoidal
	monotone $g$ and the monoidal monotone $f$ in \cref{eqn.Lawv_to_Bool1}
	can be used to convert a Lawvere metric space into a preorder. Find a
	Lawvere metric space $\cat{X}$ on which they give different answers,
	$\cat{X}_f\neq\cat{X}_g$.
	\qedhere
\end{enumerate}
\end{exercise}

\subsection{Enriched functors} %
\label{subsec.enriched_functors}%
\index{functor!enriched|see {enriched functor}}%
\index{enriched functor}
The notion of functor provides the most important type of relationship between categories.
\begin{definition}%
\label{def.enriched_functor}
Let $\cat{X}$ and $\cat{Y}$ be $\cat{V}$-categories. A \emph{$\cat{V}$-functor from $\cat{X}$ to $\cat{Y}$}, denoted $F\colon\cat{X}\to\cat{Y}$, consists of one constituent:
\begin{enumerate}[label=(\roman*)]
	\item a function $F\colon\Ob(\cat{X})\to\Ob(\cat{Y})$
\end{enumerate}
subject to one constraint
\begin{enumerate}[label=(\alph*)]
	\item for all $x_1,x_2\in\Ob(\cat{X})$, one has $\cat{X}(x_1,x_2)\leq\cat{Y}(F(x_1),F(x_2))$.
\end{enumerate}
\end{definition}

\begin{example}%
\label{ex.bool_functors_monotone}%
\index{monotone map!as
  $\Bool$-functor}
For example, we have said several times---e.g.\ in \cref{thm.preorder_is_bool_cat}---that preorders are $\Bool$-categories, where $\cat{X}(x_1,x_2)=\true$ is denoted $x_1\leq x_2$. One would hope that monotone maps between preorders would correspond exactly to $\Bool$-functors, and that's true. A monotone map $(X,\leq_X)\to(Y,\leq_Y)$ is a function $F\colon X\to Y$ such that for every $x_1,x_2\in X$, if $x_1\leq_X x_2$ then $F(x_1)\leq_Y F(x_2)$. In other words, we have
\[\cat{X}(x_1,x_2)\leq\cat{Y}(F(x_1),F(x_2)),\]
where the above $\leq$ takes place in the enriching category $\cat{V}=\Bool$; this is exactly the condition from \cref{def.enriched_functor}.
\end{example}

\begin{remark}%
\label{rem.preorder_boolcats}%
\index{equivalence of categories}
In fact, we have what is called an \emph{equivalence} of categories between the
category of preorders and the category of $\Bool$-categories. In the next chapter
we will develop the ideas necessary to state what this means precisely
(\cref{rem.preorder_boolcats2}).
\end{remark}

\begin{example}%
\index{metric space}
Lawvere metric spaces are $\Cost$-categories. The definition of $\Cost$-functor
should hopefully return a nice notion---a ``friend''---from the theory of metric
spaces, and it does: it recovers the notion of Lipschitz function. A Lipschitz
(or more precisely, 1-Lipschitz) function is one under which the distance between any pair of points does not increase. That is, given Lawvere metric spaces $(X,d_X)$ and $(Y,d_Y)$, a $\Cost$-functor between them is a function $F\colon X\to Y$ such that for every $x_1,x_2\in X$ we have $d_X(x_1,x_2)\geq d_Y(F(x_1),F(x_2))$.
\end{example}

\begin{exercise}%
\label{def.enriched_op}%
\index{opposite!$\cat{V}$-category}
The concepts of opposite, dagger, and skeleton (see \cref{ex.opposite,ex.dagger_preorder,rem.skeletal_preorder}) extend from preorders to
$\cat{V}$-categories. The \emph{opposite} of a $\cat{V}$-category $\cat{X}$ is denoted $\cat{X}\op$
and is defined by
\begin{enumerate}[label=(\roman*)]
	\item $\Ob(\cat{X}\op)\coloneqq\Ob(\cat{X})$, and
	\item for all $x,y\in \cat{X}$, we have $\cat{X}\op(x,y)\coloneqq\cat{X}(y,x)$.
\end{enumerate}
A $\cat{V}$-category $\cat{X}$ is a \emph{dagger}%
\index{dagger}
$\cat{V}$-category if the identity function is a $\cat{V}$-functor
$\dagger\colon\cat{X}\to\cat{X}\op$.  And a  \emph{skeletal}%
\index{skeleton}
$\cat{V}$-category is one in which if $I \le \cat{X}(x,y)$ and $I \le
\cat{X}(y,x)$, then $x=y$.

Recall that an extended metric space $(X,d)$ is a Lawvere metric space with two
extra properties; see properties (b) and (c) in \cref{def.ord_metric_space}.
\begin{enumerate}
	\item Show that a skeletal dagger $\Cost$-category is an extended metric space.
	\item Use \cref{exc.skeletal_dagger_preorder} to make sense of the following analogy: ``preorders are to sets as
	Lawvere metric spaces are to extended metric spaces.''
	\qedhere
\end{enumerate}
%
\end{exercise}

\subsection{Product $\cat{V}$-categories}%
\label{subsec.enriched_products}
If $\cat{V}=(V,\leq,I,\otimes)$ is a symmetric monoidal preorder and $\cat{X}$ and $\cat{Y}$ are $\cat{V}$-categories, then we can define their $\cat{V}$-product, which is a new $\cat{V}$-category.

\begin{definition}%
\label{def.enriched_prod}%
\index{product!of $\cat{V}$-categories}
Let $\cat{X}$ and $\cat{Y}$ be $\cat{V}$-categories. Define their \emph{$\cat{V}$-product}, or simply \emph{product}, to be the $\cat{V}$-category $\cat{X}\times\cat{Y}$ with
\begin{enumerate}[label=(\roman*)]
	\item $\Ob(\cat{X}\times\cat{Y})\coloneqq\Ob(\cat{X})\times\Ob(\cat{Y})$,
	\item $(\cat{X}\times\cat{Y})\big((x,y),(x',y')\big)\coloneqq\cat{X}(x,x')\otimes\cat{Y}(y,y')$,
\end{enumerate}
for two objects $(x,y)$ and $(x',y')$ in $\Ob(\cat{X}\times\cat{Y})$. 
\end{definition}
Product $\cat{V}$-categories are indeed $\cat{V}$-categories
(\cref{def.cat_enriched_mpos}); see \cref{exc.check_enriched_prod}.

\begin{exercise}%
\label{exc.check_enriched_prod}
Let $\cat{X}\times\cat{Y}$ be the $\cat{V}$-product of $\cat{V}$-categories as in \cref{def.enriched_prod}.
\begin{enumerate}
	\item Check that for every object $(x,y)\in\Ob(\cat{X}\times\cat{Y})$ we
	have $I\leq(\cat{X}\times\cat{Y})\big((x,y),(x,y)\big)$.
	\item Check that for every three objects $(x_1,y_1)$, $(x_2,y_2),$ and $(x_3,y_3)$, we have
	\[
		(\cat{X}\times\cat{Y})\big((x_1,y_1),(x_2,y_2)\big)
		\otimes(\cat{X}\times\cat{Y})\big((x_2,y_2),(x_3,y_3)\big)
		\leq
		(\cat{X}\times\cat{Y})\big((x_1,y_1),(x_3,y_3)\big).
	\]
	\item We said at the start of \cref{subsec.V_cats} that the symmetry of $\cat{V}$ (condition (d) of \cref{def.symm_mon_structure}) would be required here. Point out exactly where that condition is used.%
\index{symmetry!as required for enriched products}
	\qedhere
\end{enumerate}
\end{exercise}

When taking the product of two preorders $(P,\leq_P)\times(Q,\leq_Q)$, as first described in \cref{ex.product_preorder}, we say that $(p_1,q_1)\leq (p_2,q_2)$ iff both $p_1\leq p_2$ AND $q_1\leq q_2$; the AND is the monoidal product $\otimes$ from of $\Bool$. Thus the product of preorders is an example of a $\Bool$-product.

\begin{example}%
\index{weighted graph|see {graph, weighted}}
Let $\cat{X}$ and $\cat{Y}$ be the Lawvere metric spaces (i.e.\ $\Cost$-categories) defined by the following weighted graphs:
\begin{equation}%
\label{eqn.weighted_graphs_XY}

\end{equation}
Their product is defined by taking the product of their sets of objects, so there are six objects in $\cat{X}\times\cat{Y}$. And the distance $d_{X\times Y}((x,y),(x',y'))$ between any two points is given by the sum $d_X(x,x')+d_Y(y,y')$.

Examine the following graph, and make sure you understand how easy it is to derive from the weighted graphs for $\cat{X}$ and $\cat{Y}$ in \cref{eqn.weighted_graphs_XY}:
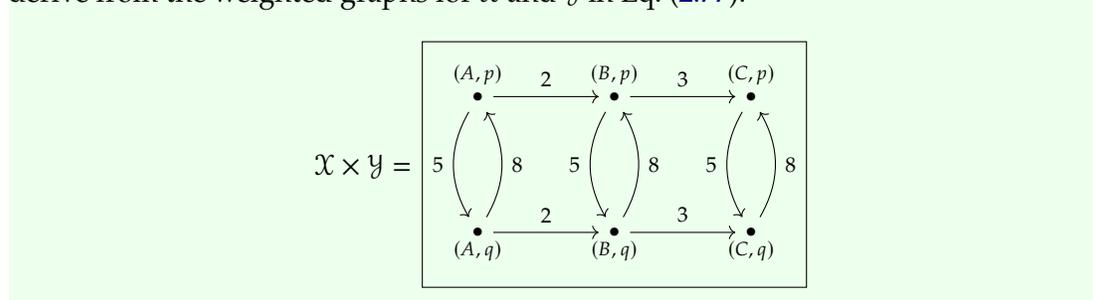
\[
%
\qedhere
\]
\end{example}

\begin{exercise}%
\label{exc.not_sqrt40}
Consider $\RR$ as a Lawvere metric space, i.e.\ as a $\Cost$-category (see \cref{ex.reals_as_metric_space}). Form the $\Cost$-product $\RR\times\RR$. What is the distance from $(5,6)$ to $(-1,4)$? Hint: apply \cref{def.enriched_prod}; the answer is not $\sqrt{40}$.
\end{exercise}

In terms of matrices, $\cat{V}$-products are also quite straightforward. They generalize what is known as the Kronecker product of matrices. The matrices for $\cat{X}$ and $\cat{Y}$ in \cref{eqn.weighted_graphs_XY} are shown below
\[
\begin{array}{c|ccc}
  \cat{X}&A&B&C\\\hline
  A & 0 & 2 & 5\\
  B & \infty  & 0 & 3\\
  C & \infty & \infty & 0
\end{array}
\hspace{1in}
\begin{array}{c|cc}
  \cat{Y}&p&q\\\hline
  p & 0 & 5\\
  q & 8 & 0
\end{array}
\]
and their product is as follows:
\[
\begin{array}{c||ccc| ccc}
	\cat{X}\times\cat{Y}&(A,p)&(B,p)&(C,p)&(A,q)&(B,q)&(C,q)\\\hline\hline
	(A,p)& 0 & 2 & 5 & 5&7&10\\
	(B,p)& \infty & 0 & 3 & \infty & 5 & 8\\
	(C,p)& \infty &\infty & 0& \infty&\infty &5\\\hline
	(A,q)& 8 & 10 & 13 & 0 & 2 & 5\\
	(B,q)& \infty & 8 & 11 & \infty & 0 & 3\\
	(C,q)& \infty & \infty & 8 & \infty &\infty & 0
\end{array}
\]
We have drawn the product matrix as a block matrix, where there is one block---shaped like $\cat{X}$---for every entry of $\cat{Y}$. Make sure you can see each block as the $\cat{X}$-matrix shifted by an entry in $\cat{Y}$. This comes directly from the formula from \cref{def.enriched_prod} and the fact that the monoidal product in $\Cost$ is $+$.

\section[Computing presented $\cat{V}$-categories with matrix mult.]{Computing presented $\cat{V}$-categories with matrix multiplication}
\label{sec.quantales}%
\index{quantale}
In \cref{subsec.Lawv_metric_spaces} we promised a straightforward way to
construct the matrix representation of a $\Cost$-category from a
$\Cost$-weighted graph.  To do this, we use a generalized matrix multiplication.
We shall show that this works, not just for $\Cost$, but also for $\Bool$, and
many other monoidal preorders. The property required of the preorder is that
of being a unital, commutative quantale. These are preorders with all joins,
plus one additional ingredient, being \emph{monoidal closed}, which we define
next, in \cref{subsec.mon_closed_preorder}. The definition of a quantale will be
given in \cref{subsec.comm_quantales}.

\index{enrichment|)}
\subsection{Monoidal closed preorders}%
\label{subsec.mon_closed_preorder}

The definition of $\cat{V}$-category makes sense for any symmetric monoidal
preorder $\cat{V}$. But that does not mean that any base of enrichment $\cat{V}$
is as useful as any other. In this section we define closed monoidal categories,
which in particular enrich themselves! ``Before you can really enrich others,
you should really enrich yourself.''

\begin{definition}%
\label{def.monoidal_closed}%
\index{closed category!monoidal|see {monoidal closed category}}%
\index{monoidal closed category}
A symmetric monoidal preorder $\cat{V}=(V,\leq,I,\otimes)$ is called \emph{symmetric monoidal closed} (or just \emph{closed}) if, for every two elements $v,w\in V$, there is an element $v\multimap w$ in $\cat{V}$, called the \emph{hom-element}, with the property
\begin{equation}%
\label{eqn.monoidal_closed_adj}
	(a\otimes v)\leq w\quad\text{iff}\quad a\leq (v\multimap w).
\end{equation}
for all $a,v,w\in V$.
\end{definition}

\begin{remark}%
\label{rem.language_one_time}
The term `closed' refers to the fact that a hom-element can be constructed for any two elements, so the preorder can be seen as closed under the operation of ``taking
homs.'' In later chapters we'll meet the closely-related concepts of compact closed
categories (\cref{def.compact_closed}) and cartesian closed categories
(\cref{subsec.set_like}) that make this idea more precise. See especially \cref{ex.ccposet_quantale}.

One can consider the hom-element $v\multimap w$ as a kind of ``single-use
$v$-to-$w$ converter.'' So \cref{eqn.monoidal_closed_adj} says that $a$ and $v$
are enough to get $w$ if and only if $a$ is enough to get a single-use $v$-to-$w$ converter.
\end{remark}

\begin{exercise} %
\label{exc.closure_is_adj} %
\index{adjunction}
Condition \cref{eqn.monoidal_closed_adj} says precisely that there is a Galois
connection in the sense of \cref{def.galois}. Let's prove this fact. In
particular, we'll prove that a monoidal preorder is monoidal closed iff, given
any $v\in V$, the map $(-\otimes v)\colon V \to V$ given by multiplying with $v$
has a right adjoint. We write this right adjoint $(v \multimap -)\colon V \to V$.
\begin{enumerate}
\item Using \cref{def.symm_mon_structure}, show that $(-\otimes v)$ is monotone.
\item Supposing that $\cat{V}$ is closed, show that for all $v,w \in V$ we have $\big((v \multimap w) \otimes v\big) \le w$.
\item Using 2., show that $(v \multimap -)$ is monotone.
\item Conclude that a symmetric monoidal preorder is closed if and only if the
monotone map $(-\otimes v)$ has a right adjoint. \qedhere
\end{enumerate}
\end{exercise}

\begin{example}%
\label{ex.Lawv_closed}%
\index{Cost@$\Cost$}%
\index{monoidal closed category!Cost as@$\Cost$|see {Cost}}
The monoidal preorder $\Cost=([0,\infty],\geq,0,+)$ is monoidal closed. Indeed, for any $x,y\in[0,\infty]$, define $x\multimap y\coloneqq\max(0,y-x)$. Then, for any $a,x,y\in[0,\infty]$, we have
\[
	a+x\geq y
	\quad\text{iff}\quad
	a\geq y-x
	\quad\text{iff}\quad
	\max(0,a)\geq\max(0,y-x)
	\quad\text{iff}\quad
	a\geq (x\multimap y)	
\]
so $\multimap$ satisfies the condition of \cref{eqn.monoidal_closed_adj}. 

Note that we have not considered subtraction in $\Cost$ before; we can in fact
use monoidal closure to \emph{define} subtraction in terms of the order and
monoidal structure!
\end{example}

\begin{exercise}%
\label{exc.Bool_monoidal_closed}%
\index{booleans!as monoidal closed}%
\index{monoidal closed category!booleans as}
Show that $\Bool=(\BB,\leq,\true,\wedge)$ is monoidal closed.
\end{exercise}

\begin{example}
A non-example is $(\BB,\leq,\false,\vee)$. Indeed, suppose we had a $\multimap$
operator as in \cref{def.monoidal_closed}. Note that $\false\leq p\multimap q$,
for any $p,q$ no matter what $\multimap$ is, because $\false$ is less than
everything. But using $a=\false$, $p=\true$, and $q=\false$, we then get
a contradiction: $(a\vee p)\not\leq q$ and yet $a\leq(p\multimap q)$.
\end{example}

\begin{example}%
\index{chemistry}
We started this chapter talking about resource theories. What does the closed structure look like from that perspective? For example, in chemistry it would say that for every two material collections $c,d$ one can form a material collection $c\multimap d$ with the property that for any $a$, one has
\[a+c\to d\quad\text{if and only if}\quad a\to(c\multimap d).\]
Or more down to earth, since we have the reaction $\mathrm{2H_2O+2Na\to2NaOH+H_2}$, we must also have
\[\mathrm{2H_2O\to(2Na\multimap (2NaOH+H_2))}\]
So from just two molecules of water, you can form a certain substance, and not
many substances fit the bill---our preorder $\Set{Mat}$ of chemical materials is not closed. 

But it is not so far-fetched: this hypothetical new substance $\mathrm{(2Na\multimap (2NaOH+H_2))}$ is not really a substance, but a potential reaction: namely that of converting a sodium to sodium-hydroxide-plus-hydrogen. Two molecules of water unlock that potential.
%
\end{example}

\begin{proposition}%
\label{prop.properties_closed_mon_preorders}
Suppose $\cat{V}=(V,\leq,I,\otimes,\multimap)$ is a symmetric monoidal preorder that is closed. Then
\begin{enumerate}[label=(\alph*)]
	\item For every $v\in V$, the monotone map $-\otimes v\colon (V,\leq)\to(V,\leq)$ is left adjoint to $v\multimap-\colon (V,\leq)\to(V,\leq)$. 
	\item For any element $v\in V$ and set of elements $A\ss V$, if the join $\bigvee_{a\in A}a$ exists then so does $\bigvee_{a\in A}v\otimes a$ and we have 
	\begin{equation}%
\label{eqn.monprod_distributes_joins}
		\left(v\otimes\bigvee_{a\in A}a\right)\cong\bigvee_{a\in A}(v\otimes a).
	\end{equation}
	\item For any $v,w\in V$, we have $v\otimes(v\multimap w)\leq w$.
	\item For any $v\in V$, we have $v\cong(I\multimap v)$.
	\item For any $u,v,w\in V$, we have $(u\multimap v)\otimes(v\multimap w)\leq(u\multimap w)$.
\end{enumerate}
\end{proposition}
\begin{proof}
We go through the claims in order.
\begin{enumerate}[label=(\alph*)]
	\item The definition of $(-\otimes v)$ being left adjoint to
	$(v\multimap -)$ is exactly the condition
	\cref{eqn.monoidal_closed_adj}; see \cref{def.galois} and
	\cref{exc.closure_is_adj}.
	\item This follows from (a), using the fact that left adjoints preserve
	  joins (\cref{prop.right_adj_meets}).
	\item This follows from (a), using the equivalent characterisation of
	  Galois connection in \cref{prop.galois_monad_comonad}. More
	  concretely, from reflexivity $(v\multimap w)\leq(v\multimap w)$, we
	  obtain $(v\multimap w)\otimes v\leq w$ \cref{eqn.monoidal_closed_adj},
	  and we are done by symmetry, which says $v\otimes(v\multimap
	  w)=(v\multimap w)\otimes v$. 
	\item Since $v\otimes I=v\leq v$, \cref{eqn.monoidal_closed_adj} says
	  $v\leq (I\multimap v)$. For the other direction, we have $(I\multimap
	  v)=I\otimes(I\multimap v)\leq v$ by (c).
	\item To obtain this inequality, we just need $u\otimes(u\multimap
	  v)\otimes(v\multimap w)\leq w$. But this follows by two applications
	  of (c).
	  \qedhere
\end{enumerate}
\end{proof}

One might read (c) as saying ``if I have a $v$ and a single-use $v$-to-$w$
converter, I can have a $w$.'' One might read (d) as saying ``having a $v$ is
the same as having a single-use nothing-to-$v$ converter.'' And one might read
(e) as saying ``if I have a single-use $u$-to-$v$ converter and a single-use
$v$-to-$w$ converter, I can get a single-use $u$-to-$w$ converter.

\begin{remark}%
\label{rem.quantale_enriches_itself}%
\index{quantales!as self-enriched}
We can consider $\cat{V}$ to be enriched in itself. That is, for every
$v,w\in\Ob(\cat{V})$, we can define $\cat{V}(v,w)\coloneqq (v\multimap
w)\in\cat{V}$. For this to really be an enrichment, we just need to check the
two conditions of \cref{def.cat_enriched_mpos}. The first condition
$I\leq\cat{X}(x,x)=(x\multimap x)$ is satisfied because $I\otimes x\leq x$. The
second condition is satisfied by \cref{prop.properties_closed_mon_preorders}(e).
\end{remark}

\erase{
\begin{exercise}
The symmetric monoidal preorder $\cat{N}\coloneqq(\NN,\vert, 1,*)$ appears to provide a counterexample to \cref{cor.closed_distrib_joins}. Find the flaw in the logic below:
\begin{itemize}
	\item \cref{cor.closed_distrib_joins} must be wrong.
	\begin{itemize}[label={$\circ$}]
		\item $\cat{N}$ is not monoidal closed (see \cref{def.monoidal_closed}) because there is no $\multimap$ that can satisfy \cref{eqn.monoidal_closed_adj}.
		\begin{itemize}[label={$\blacktriangleright$}]
			\item Note that $1\vert n$ for all $n\in\NN$.
			\item Take $a=1$, $v=2$, and $w=3$. Then $1*2\vert 3$ is false, but $1\vert(2\multimap 3)$ is true. So \cref{eqn.monoidal_closed_adj} does not hold.
		\end{itemize}	
		\item $\cat{N}$ has all joins and $*$ distributes over them, i.e.\ \cref{eqn.monprod_distributes_joins} holds for all $A\ss\NN$ and $n\in \NN$.
		\begin{itemize}[label={$\blacktriangleright$}]
			\item $\cat{N}$ has all joins.
				\begin{itemize}[label={$\triangleright$}]
					\item For pairs, $a\vee b$ is given by least common multiple, e.g.\ $4\vee 6=\mathrm{lcm}(4,6)=12$.
					\item In the case $A=\varnothing$, we have $\bigvee\varnothing=1$,
					\item $\cat{N}$ has all finite joins by repeated applications of lcm.
					\item $\bigvee A=0$ whenever $A$ is infinite: $a\vert 0$ for every $a\in A$, and there is no other such $b\in\NN$.
				\end{itemize}
			\item \cref{eqn.monprod_distributes_joins}, which says $\bigvee_{a\in A}(n*a)=n*\bigvee_{a\in A}a$ for any $A\ss\NN$ and $n\in\NN$, holds. For example, $n*\mathrm{lcm}(a,b)=\mathrm{lcm}(n*a,n*b)$.
		\end{itemize}	
	\end{itemize}
\end{itemize}
\end{exercise}
}

\subsection{Quantales}%
\label{subsec.comm_quantales}%
\index{quantale|(}
To perform matrix multiplication over a monoidal preorder, we need one more
thing: joins. These were first defined in \cref{def.meets_joins}.%
\index{joins!required in a quantale}

\begin{definition}%
\label{def.quantale}%
\index{quantale!commutative}
A \emph{unital commutative quantale} is a symmetric monoidal closed preorder $\cat{V}=(V,\leq,I,\otimes,\multimap)$ that has all joins: $\bigvee A$ exists for every $A\ss V$. In particular, we often denote the empty join by $0\coloneqq\bigvee\varnothing$.
\end{definition}

Whenever we speak of quantales in this book, we mean unital commutative
quantales. We will try to remind the reader of that. There are also very interesting
applications of noncommutative quantales; see
\cref{sec.resource_theory_further_reading}.

\begin{example}%
\index{preorder!Cost@$\Cost$|see {Cost}}%
\index{quantale!Cost@$\Cost$ as|see {Cost}}%
\index{Cost@$\Cost$}
  In \cref{ex.Lawv_closed}, we saw that $\Cost$ is monoidal closed. To check
  whether $\Cost$ is a quantale, we take an arbitrary set of elements
  $A\ss[0,\infty]$ and ask if it has a join $\bigvee A$. To be a join, it needs
  to satisfy two properties:
\begin{enumerate}[label=\alph*.]
	\item $a\geq\bigvee A$ for all $a\in A$, and
	\item if $b\in[0,\infty]$ is any element such that $a\geq b$ for all $a\in A$, then $\bigvee A\geq b$.
\end{enumerate}
In fact we can define such a join: it is typically called the \emph{infimum}, or greatest lower bound, of $A$.%
\tablefootnote{
Here, by the infimum of a subset $A\ss[0,\infty]$, we mean infimum in the usual order on $[0,\infty]$: the largest number that is $\leq$ everything in $A$. For example, the infimum of $\{3.1, 3.01, 3.001, \ldots\}$ is $3$. But note that this is the \emph{supremum} in the reversed, $\geq$, order of $\Cost$.
}
For example, if $A=\{2,3\}$ then $\bigvee A=2$. We have joins for infinite sets too: if $B=\{2.5, 2.05, 2.005, \ldots\}$, its infimum is $2$. Finally, in order to say that $([0,\infty],\geq)$ has all joins, we need a join to exist for the empty set $A=\varnothing$ too. The first condition becomes vacuous---there are no $a$'s in $A$---but the second condition says that for any $b\in [0,\infty]$ we have $\bigvee\varnothing\geq b$; this means $\bigvee\varnothing=\infty$.

Thus indeed $([0,\infty],\geq)$ has all joins, so $\Cost$ is a quantale.
\end{example}

\begin{exercise}%
\label{exc.matrix_mult1}
\begin{enumerate}
	\item What is $\bigvee\varnothing$, which we generally denote $0$, in the case 
  \begin{enumerate}[label=\alph*.]
  	\item $\cat{V}=\Bool=(\BB,\leq,\true,\wedge)$?
  	\item $\cat{V}=\Cost=([0,\infty],\geq,0,+)$?
  \end{enumerate}
  \item What is the join $x\vee y$ in the case
  \begin{enumerate}[label=\alph*.]
  	\item $\cat{V}=\Bool$, and $x,y\in\BB$ are booleans?
  	\item $\cat{V}=\Cost$, and $x,y\in[0,\infty]$ are distances?
	  \qedhere
  \end{enumerate}  
\qedhere
\end{enumerate}
\end{exercise}

\begin{exercise}%
\label{exc.Bool_quantale}%
\index{booleans!as quantale}
Show that $\Bool=(\BB,\leq,\true,\wedge)$ is a quantale.
\end{exercise}

\begin{exercise} %
\label{exc.powerset_quantale}
Let $S$ be a set and recall the power set monoidal preorder $(\powset(S),\ss,S,\cap)$ from \cref{exc.powerset_symm_mon_pos}. Is it a quantale?
\end{exercise}

\begin{remark}%
\label{rem.personify_navigator}%
\index{navigator}
One can personify the notion of unital, commutative quantale as a kind of
navigator. A navigator is someone who understands ``getting from one place to
another.'' Different navigators may care about or understand different
aspects---whether one can get from $A$ to $B$, how much time it will take, what
modes of travel will work, etc.---but they certainly have some commonalities.
Most importantly, a navigator needs to be able to read a map: given routes $A$
to $B$ and $B$ to $C$, they understand how to get a route $A$ to $C$. And they
know how to search over the space of way-points to get from $A$ to $C$. These
will correspond to the monoidal product and the join operations, respectively.
\end{remark}

\begin{proposition}
Let $\cat{P}=(P,\leq)$ be a preorder. It has all joins iff it has all meets.
\end{proposition}%
\index{dual!as opposite}
\begin{proof}
The joins (resp.\ meets) in $\cat{P}$ are the meets (resp.\ joins) in $\cat{P}\op$, so the two claims are dual: it suffices to show that if $\cat{P}$ has all joins then it has all meets.

Suppose $\cat{P}$ has all joins and suppose that $A\ss\cat{P}$ is a subset for which we want the meet. Consider the set $M_A\coloneqq\{p\in P\mid p\leq a\text{ for all }a\in A\}$ of elements below everything in $A$. Let $m_A\coloneqq\bigvee_{p\in M_A}p$ be their join. We claim that $m_A$ is a meet for $A$.

We first need to know that for any $a\in A$ we have $m_A\leq a$, but this is by definition of join: since all $p\in M_A$ satisfy $p\leq a$, so does their join $m_A\leq a$. We second need to know that for any $m'\in P$ with $m'\leq a$ for all $a\in A$, we have $m'\leq m$. But every such $m'$ is actually an element of $M_A$ and $m$ is their join, so $m'\leq m$. This completes the proof.
\end{proof}

In particular, a quantale has all meets and all joins, even though we only define it to have all joins.

\begin{remark}
%
\index{Hausdorff distance}
The notion of Hausdorff distance can be generalized, allowing the role of $\Cost$ to be taken by
any quantale $\cat{V}$. If $\cat{X}$ is a $\cat{V}$-category with objects $X$, and $U\ss X$
and $V\ss X$, we can generalize the usual Hausdorff distance, on the left
below, to the formula on the right:
\[
  d(U,V)\coloneqq\sup_{u\in U}\underset{v \in V}{\inf\vphantom{p}}d(u,v)
  \qquad\qquad
  \cat{X}(U,V)\coloneqq\bigwedge_{u\in U}\bigvee_{v\in V}\cat{X}(u,v).
\]
For example, if $\cat{V}=\Bool$, the Hausdorff distance between sub-preorders
$U$ and $V$ answers the question ``can I get into $V$ from every $u\in U$,''
i.e.\ $\forall_{u\in U}\ldotp\exists_{v\in V}\ldotp u\leq v$. Or for another example,
use $\cat{V}=\powset(M)$ with its interpretation as modes of transportation, as
in \cref{ex.modes_of_transport}. Then the Hausdorff distance
$d(U,V)\in\powset(M)$ tells us those modes of transportation that will get us
into $V$ from every point in $U$.
\end{remark}

\begin{proposition}%
\label{cor.closed_distrib_joins}%
\index{join}%
\index{adjoint functor theorem}
Suppose $\cat{V}=(V,\leq,I,\otimes)$ is any symmetric monoidal preorder that has all joins. Then $\cat{V}$ is closed---i.e.\ it has a $\multimap$ operation and hence is a quantale---if and only if $\otimes$ distributes over joins; i.e.\ if \cref{eqn.monprod_distributes_joins} holds for all $v\in V$ and $A\ss V$.
\end{proposition}
\begin{proof}
We showed one direction in \cref{prop.properties_closed_mon_preorders}(b): if
$\cat{V}$ is monoidal closed then \cref{eqn.monprod_distributes_joins} holds. We
need to show that \cref{eqn.monprod_distributes_joins} holds then $-\otimes
v\colon V\to V$ has a right adjoint $v\multimap-$. This is just the adjoint
functor theorem, \cref{prop.adjoint_functor_thm}. It says we can define
$v\multimap w$ to be
\[
	v\multimap w\coloneqq\bigvee_{\{a\in V\mid a\otimes v\leq w\}}a.
	\qedhere
\]
\end{proof}

\subsection{Matrix multiplication in a quantale}%
\label{subsubsec.nav_matrix_mult}
\index{quantale!matrix multiplication in}
A quantale $\cat{V}=(V,\leq,I,\otimes,\multimap)$, as defined in \cref{def.monoidal_closed}, provides what is necessary to perform matrix multiplication.%
\footnote{This works for noncommutative quantales as well.}
The usual formula for matrix multiplication is:
\begin{equation}%
\label{eqn.standard_matrix_mult}
(M*N)(i,k)=\sum_{j}M(i,j)*N(j,k).
\end{equation}
We will get a formula where joins stand in for the sum operation $\sum$, and $\otimes$ stands in for the product operation $*$. Recall our convention of writing $0\coloneqq\bigvee \varnothing$.

\begin{definition}%
\index{matrix}
Let $\cat{V}=(V,\leq,\otimes,I)$ be a quantale. Given sets $X$ and $Y$, a \emph{matrix with entries in $\cat{V}$}, or simply a \emph{$\cat{V}$-matrix}, is a function $M\colon X\times Y\to V$. For any $x\in X$ and $y\in Y$, we call $M(x,y)$ the \emph{$(x,y)$-entry}.
\end{definition}

\index{matrices!multiplication of}
Here is how you multiply $\cat{V}$-matrices $M\colon X\times Y\to V$ and $N\colon Y\times Z\to V$. Their product is defined to be the matrix $(M*N)\colon X\times Z\to V$, whose entries are given by the formula
\begin{equation}%
\label{eqn.quantale_matrix_mult}
	(M*N)(x,z)\coloneqq\bigvee_{y\in Y}M(x,y)\otimes N(y,z).
\end{equation}
Note how similar this is to \cref{eqn.standard_matrix_mult}.

\begin{example}
Let $\cat{V}=\Bool$. Here is an example of matrix multiplication $M*N$. Here $X=\{1,2,3\}$, $Y=\{1,2\}$, and $Z=\{1,2,3\}$, matrices $M\colon X\times Y\to\BB$ and $N\colon Y\times Z\to\BB$ are shown to the left below, and their product is shown to the right:
\[
\left(
\begin{array}{cc}
	\false&\false\\
	\false&\true\\
	\true&\true
\end{array}
\right)
*
\left(
\begin{array}{ccc}
	\true&\true&\false\\
	\true&\false&\true
\end{array}
\right)
=
\left(
\begin{array}{ccc}
	\false&\false&\false\\
	\true&\false&\true\\
	\true&\true&\true
\end{array}
\right)
\qedhere
\]
\end{example}

The identity $V$-matrix on a set $X$ is $I_X\colon X\times X\to V$ given by
\[
	I_X(x,y)\coloneqq
	\begin{cases}
		I&\text{ if }x=y\\
		0&\text{ if }x\neq y.
	\end{cases}
\]%
\index{matrix!identity}%
\index{identity!matrix}

\begin{exercise} %
\label{exc.identity_matrices}
Write down the $2\times 2$-identity matrix for each of the quantales
$(\nn,\leq,1,*)$, $\Bool=(\bb,\leq,\true,\wedge)$, and
$\Cost=([0,\infty],\geq,0,+)$.
\end{exercise}

\begin{exercise} %
\label{exc.matrix_mult2}
Let $\cat{V}=(V,\leq,I,\otimes,\multimap)$ be a quantale. Use \cref{eqn.quantale_matrix_mult} and \cref{prop.properties_closed_mon_preorders} to prove the following.
\begin{enumerate}
  \item Prove the \emph{identity law}: for any sets $X$ and $Y$ and $V$-matrix $M\colon X\times Y\to V$, one has $I_X*M=M$.
  \item Prove the \emph{associative law}: for any matrices $M\colon W\times X\to V$, $N\colon X\times Y\to V$, and $P\colon Y\times Z\to V$, one has $(M*N)*P=M*(N*P)$.%
\index{associativity!of quantale matrix multiplication}
\qedhere
\end{enumerate}
\end{exercise}

Recall the weighted graph $Y$ from \cref{eqn.cities_distances}. One can read off the associated matrix $M_Y$, and one can calculate the associated metric $d_Y$:
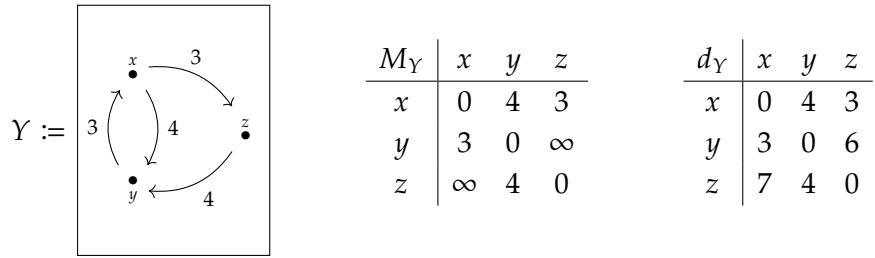
\[
\begin{tikzpicture}[font=\scriptsize, x=1cm, baseline=(Y)]
	\node (x) {$\LMO{x}$};
	\node[below=1 of x] (y) {$\LMO[under]{y}$};
	\node at ($(x)!.5!(y)+(1.5cm,0)$) (z) {$\LMO{z}$};
	\draw[bend left,->] (y) to node[left] {3} (x);
	\draw[bend left,->] (x) to node[right] {4} (y);	
	\draw[bend left,->] (x) to node[above] {3} (z);
	\draw[bend left,->] (z) to node[below right=-1pt and -1pt] {4} (y);
	\node[draw, inner sep=15pt, fit=(x) (y) (z.west)] (Y) {};
	\node[left=0 of Y, font=\normalsize] {$Y\coloneqq$};
\end{tikzpicture}
\hspace{.5in}
\begin{array}{c|ccc}
M_Y&x&y&z\\\hline
x & 0 & 4 & 3\\
y & 3 & 0 & \infty\\
z & \infty & 4 & 0
\end{array}
\hspace{.5in}
\begin{array}{c|ccc}
d_Y&x&y&z\\\hline
x&0&4&3\\
y&3&0&6\\
z&7&4&0
\end{array}
\]
Here we fully explain how to compute $d_Y$ using only $M_Y$.

The matrix $M_Y$ can be thought of as recording the length of paths that
traverse either $0$ or $1$ edges: the diagonals being $0$ mean we can get from
$x$ to $x$ without traversing any edges. When we can get from $x$ to $y$ in one edge
we record its length in $M_Y$, otherwise we use $\infty$.

When we multiply $M_Y$ by itself using the formula
\cref{eqn.quantale_matrix_mult}, the result $M_Y^2$ tells us the length of the
shortest path traversing 2 edges or fewer. Similarly $M_Y^3$ tells us about the
shortest path traversing 3 edges or fewer:
\[
M_Y^2=
\begin{array}{c|ccc}
\nearrow&x&y&z\\\hline
x & 0 & 4 & 3\\
y & 3 & 0 & 6\\
z & 7 & 4 & 0
\end{array}
\hspace{1in}
M_Y^3=
\begin{array}{c|ccc}
\nearrow&x&y&z\\\hline
x & 0 & 4 & 3\\
y & 3 & 0 & 6\\
z & 7 & 4 & 0
\end{array}
\]
One sees that the powers stabilize: $M_Y^2=M_Y^3$; as soon as that happens one
has the matrix of distances, $d_Y$. Indeed $M_Y^n$ records the lengths of the
shortest path traverse $n$ edges or fewer, and the powers will always stabilize
if the set of vertices is finite, since the shortest path from one vertex to
another will never visit a given vertex more than once.\footnote{The method
works even in the infinite case: one takes the infimum of all powers $M_Y^n$.
The result always defines a Lawvere metric space.}

\begin{exercise} %
\label{exc.computing_distances}
Recall from \cref{exc.adjacency_matrix_X} the matrix $M_X$, for $X$ as in
\cref{eqn.cities_distances}. Calculate $M_X^2$, $M_X^3$, and $M_X^4$.  Check
that $M_X^4$ is what you got for the distance matrix in
\cref{exc.distance_matrix_X}.
\end{exercise}

This procedure gives an algorithm for computing the $\cat{V}$-category presented
by any $\cat{V}$-weighted graph using matrix multiplication.

\index{quantale|)}
\section{Summary and further reading}%
\label{sec.resource_theory_further_reading}

In this chapter we thought of elements of preorders as describing resources, with the
order detailing whether one resource could be obtained from another. This
naturally led to the question of how to describe what could be built from a pair
of resources, which led us to consider monoid structures on preorders. More
abstractly, these monoidal preorders were seen to be examples of enriched
categories, or $\cat{V}$-categories, over the symmetric monoidal preorder $\Bool$.
Changing $\Bool$ to the symmetric monoidal preorder $\Cost$, we arrived upon
Lawvere metric spaces, a slight generalization of the usual notion of metric
space. In terms of resources, $\Cost$-categories tell us the cost of obtaining
one resource from another.

At this point, we sought to get a better feel for $\cat{V}$-categories in two
ways. First, we introduced various important constructions: base change, functors,
products. Second, we looked at how to present $\cat{V}$-categories using
labelled graphs; here, perhaps surprisingly, we saw that matrix multiplication
gives an algorithm to compute the hom-objects from a labelled graph.

Resource theories are discussed in much more detail in
\cite{Coecke.Fritz.Spekkens:2016a,Fritz:2017a}. The authors provide many more examples
of resource theories in science, including in thermodynamics, Shannon's theory
of communication channels, and quantum entanglement. They also discuss more of
the numerical theory than we did, including calculating the asymptotic rate of
conversion from one resource into another.

Enrichment is a fundamental notion in category theory, and we will we return to
it in \cref{chap.codesign}, generalizing the definition so that categories,
rather than mere preorders, can serve as bases of enrichment. In this more general
setting we can still perform the constructions we introduced in
\cref{sec.vcat_constructions}---base change, functors, products---and many
others; the authoratitive, but by no means easy, reference on this is the book
by Kelly \cite{Kelly:1982a}.  

While preorders were familiar before category theory came along, Lawvere metric spaces are a
beautiful generalization of the previous notion of (symmetric) metric space, that is due to, well, Lawvere. A deeper exploration than the taste
we gave here can be found in his classic paper \cite{Lawvere:1973a}, where he also discusses ideas like Cauchy completeness in category-theoretic terms, and which hence generalize to other categorical settings.

We observed that while any symmetric monoidal preorder can serve as a base for
enrichment, certain preorders---quantales---are better than others. Quantales are
well known for links to other parts of mathematics too. The word quantale is in
fact a portmanteau of `quantum locale', where quantum refers to quantum physics,
and locale is a fundamental structure in topology. For a book-length
introduction of quantales and their applications, one might check
\cite{Rosenthal:1990a}. %
\index{quantale} The notion of cartesian closed categories, later generalized to monoidal closed categories, is due to Ronnie Brown \cite{brown1961some}. %
\index{closed category!monoidal}%
\index{closed category!cartesian}

Note that while we have only considered commutative quantales, the
noncommutative variety also arise naturally. For example, the power set of any
monoid forms a quantale that is commutative iff the monoid is. Another example
is the set of all binary relations on a set $X$, where multiplication is
relational composition; this is non-commutative. Such noncommutative quantales
have application to concurrency theory, and in particular process semantics and
automata; see \cite{Abramsky.Vickers:1993a} for details.

\setcounter{chapter}{2}
\chapter[Databases: Categories, functors, and (co)limits]{Databases:\\Categories, functors, and universal constructions}%
\label{chap.databases}


\section{What is a database?}%
\label{sec.C2_motivation}
\index{database|(}

Integrating data from disparate sources is a major problem in industry today. A
study in 2008 \cite{bernstein.Hass:2008a} showed that data integration accounts for 40\% of IT (information technology) budgets, and that the market for
data integration software was \$2.5 billion in 2007 and increasing at a rate of
more than 8\% per year. In other words, it is a major problem; but what is it?


\paragraph{A database is a system of interlocking tables.}%
\index{database!as interlocking tables}

Data becomes information when it is stored \emph{in} a given \emph{formation}. That is, the numbers and letters don't mean anything until they are organized, often into a system of interlocking tables. An organized system of interlocking tables is called a database. Here is a favorite example:

\begin{equation}%
\label{eqn.fav_ex_db}
\begin{tabular}{ c | c  c  c}
  \textbf{Employee}&\textbf{FName}&\textbf{WorksIn}&\textbf{Mngr}\\\hline
  1&Alan&101&2\\
  2&Ruth&101&2\\
  3&Kris&102&3
\end{tabular}
\hspace{.6in}
\begin{tabular}{ c | c  c}
  \textbf{Department}&\textbf{DName}&\textbf{Secr}\\\hline
  101&Sales&1\\
  102&IT&3\\~
\end{tabular}
\end{equation}

These two tables interlock by use of a special left-hand column, demarcated by a vertical line; it is called the ID column. The ID column of the first table is called `Employee,' and the ID column of the second table is called `Department.' The entries in the ID column---e.g.\ 1, 2, 3 or 101, 102---are like row labels; they indicate a whole row of the table they're in. Thus each row label must be unique (no two rows in a table can have the same label), so that it can unambiguously specify its row.%
\index{database!ID columns of}

Each table's ID column, and the set of unique identifiers found therein, is what allows for the interlocking mentioned above. Indeed, other entries in various tables can reference rows in a given table by use of its ID column. For example, each entry in the WorksIn column references a department for each employee; each entry in the Mngr (manager) column references an employee for each employee, and each entry in the Secr (secretary) column references an employee for each department. Managing all this cross-referencing is the purpose of databases.

Looking back at \cref{eqn.fav_ex_db}, one might notice that every non-ID column, found in either table, is a reference to a label of some sort. Some of these, namely WorksIn, Mngr, and Secr, are \emph{internal references}, often called \emph{foreign keys}; they refer to rows (keys) in the ID column of some (foreign) table.%
\index{foreign key|see {database, foreign key}}
Others, namely FName and DName, are \emph{external references}; they refer to strings or integers, which can also be thought of as labels, whose meaning is known more broadly. Internal reference labels can be changed as long as the change is consistent---1 could be replaced by 1001 everywhere without changing the meaning---whereas external reference labels certainly cannot! Changing Ruth to Bruce everywhere would change how people understood the data.

The reference structure for a given database---i.e.\ how tables interlock via foreign keys---tells us something about what information was intended to be stored in it. One may visualize the reference structure for \cref{eqn.fav_ex_db} graphically as follows:
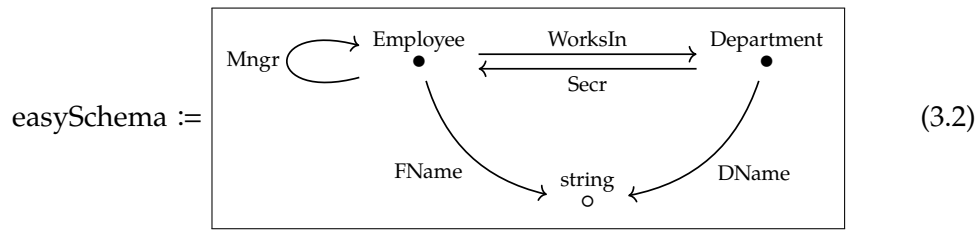
\begin{equation}%
\label{eqn.free_schema}
\text{easySchema}\coloneqq
\boxCD{
\begin{tikzcd}[row sep=large, ampersand replacement=\&]
  	\LTO{Employee}\ar[rr, shift left, "\text{WorksIn}"]\ar[dr, bend right, "\text{FName}"']\ar[loop left, "\text{Mngr}"]\&\&
  	\LTO{Department}\ar[ll, shift left, "\text{Secr}"]\ar[dl, bend left, "\text{DName}"]\\
  	\&\LTO[\circ]{string}
\end{tikzcd}
}
\end{equation}
This is a kind of ``Hasse diagram for a database,'' much like the Hasse diagrams for preorders in \cref{rem.Hasse}. How should you read it?
\index{database schema}

The two tables from \cref{eqn.fav_ex_db} are represented in the graph \eqref{eqn.free_schema} by the
two black nodes, which are given the same name as the ID columns: Employee and Department. There is
another node---drawn white rather than black---which represents the external reference type of strings, like ``Alan,'' ``Alpha,'' and ``Sales". The arrows in the diagram represent non-ID columns of the tables; they point in the direction of reference: WorksIn refers an employee to a department.

\begin{exercise}%
\label{exc.fks_arrows}%
\index{database!foreign key}
	Count the number of non-ID columns in \cref{eqn.fav_ex_db}. Count the number of arrows (foreign keys) in \cref{eqn.free_schema}. They should be the same number in this case; is this a coincidence?
\end{exercise}

\index{database schema!free}%
\index{free!schema}%
\index{Hasse diagram!database schema as}
A Hasse-style diagram like the one in \cref{eqn.free_schema} can be called a \emph{database schema}; it represents how the information is being organized, the formation in which the data is kept. One may add rules, sometimes called `business rules' to the schema, in order to ensure the integrity of the data. If these rules are violated, one knows that data being entered does not conform to  the way the database designers intended. For example, the designers may enforce rules saying 
\begin{itemize}
	\item every department's secretary must work in that department;
	\item every employee's manager must work in the same department as the employee.
\end{itemize}
Doing so changes the schema, say from `easySchema' \eqref{eqn.free_schema} to `mySchema' below.
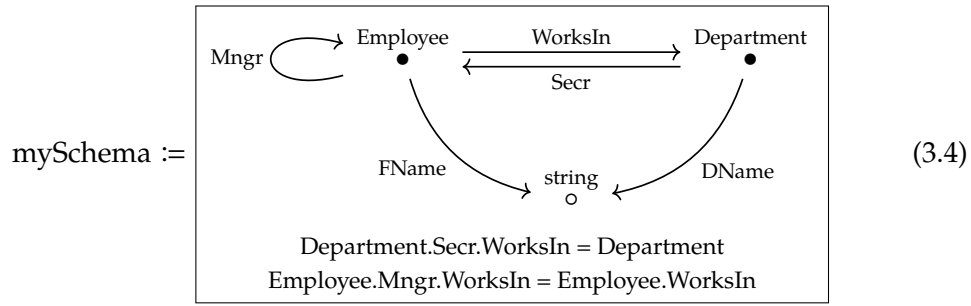
\begin{equation}%
\label{eqn.mySchema}
\text{mySchema}\coloneqq
\boxCD{
\begin{tikzcd}[row sep=large, ampersand replacement = \&]
 	\LTO{Employee}\ar[rr, shift left, "\text{WorksIn}"]\ar[dr, bend right, "\text{FName}"']\ar[loop left, "\text{Mngr}"]\&\&
  \LTO{Department}\ar[ll, shift left, "\text{Secr}"]\ar[dl, bend left, "\text{DName}"]\\
  \&\LTO[\circ]{string}
\end{tikzcd}
\\~\\\footnotesize
  Department.Secr.WorksIn = Department\\
	Employee.Mngr.WorksIn = Employee.WorksIn  
}
\end{equation}
In other words, the difference is that $\mathrm{easySchema}$ plus constraints equals $\mathrm{mySchema}$.%
\index{database!constraints}%
\index{database schema!as category presentation}

We will soon see that database schemas are categories $\cat{C}$, that the data itself is given by a `set-valued' functor $\cat{C}\to\smset$, and that databases can be mapped to each other via functors $\cat{C}\to\cat{D}$. In other words, there is a relatively large overlap between database theory and category theory. This has been worked out in a number of papers; see \cref{sec.ch2_further_reading}. It has also been implemented in working software, called FQL, which stands for \emph{functorial query language}. Here is example FQL code for the schema shown above:

\footnotesize
\begin{verbatim}
   schema mySchema = { 
      nodes
         Employee, Department;
      attributes
         DName : Department -> string,
         FName : Employee   -> string;
      arrows
         Mngr    : Employee   -> Employee,
         WorksIn : Employee   -> Department,
         Secr    : Department -> Employee;
      equations  
         Department.Secr.WorksIn = Department,
         Employee.Mngr.WorksIn   = Employee.WorksIn;
   }
\end{verbatim}
\normalsize
\index{functorial query language, FQL}


\paragraph{Communication between databases.}%
\index{database!communication between}

We have said that databases are designed to store information about something. But different people or organizations might view the same sort of thing in different ways. For example, one bank stores its financial records according to European standards and another does so according to Japanese standards. If these two banks merge into one, they will need to be able to share their data despite differences in the shape of their database schemas.

Such problems are huge and intricate in general, because databases often comprise hundreds or thousands of interlocking tables. Moreover, these problems occur more frequently than just when companies want to merge. It is quite common that a given company moves data between databases on a daily basis. The reason is that different ways of organizing information are convenient for different purposes. Just like we pack our clothes in a suitcase when traveling but use a closet at home, there is generally not one best way to organize anything.

Category theory provides a mathematical approach for translating between these
different organizational forms. That is, it formalizes a sort of automated
reorganizing process called \emph{data migration}, which takes data that fits
snugly in one schema and moves it into another.%
\index{database!data migration}

Here is a simple case. Imagine an airline company has two different databases, perhaps created at different times, that hold roughly the same data. 
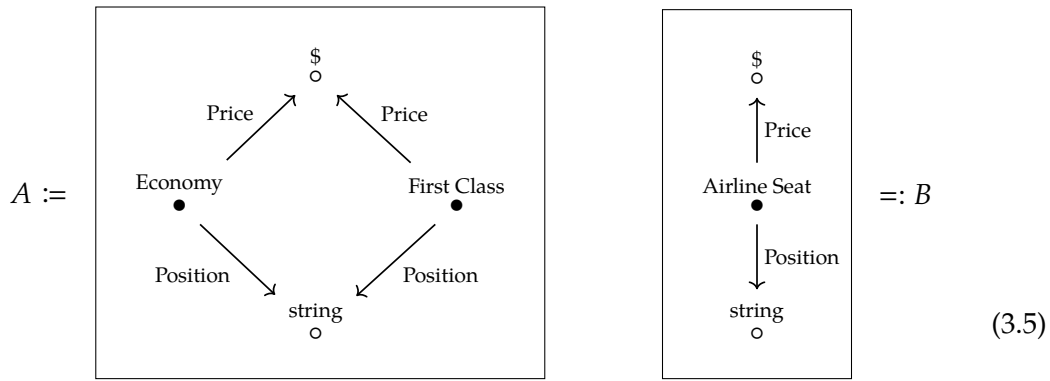
\begin{equation}%
\label{eqn.airline_schemas}
\begin{tikzpicture}[commutative diagrams/every diagram, inner sep=10pt, baseline=(A)]
	\matrix[matrix of math nodes, name=A, row sep=25pt, column sep=15pt, commutative diagrams/every cell] {
		&
			|(AD)|\LTO[\circ]{\$}
		\\
			|(AE)|\LTO{Economy}
		&&
			|(AF)|\LTO{First Class}
		\\
		&
			|(AS)|\LTO[\circ]{string}
		\\
	};
	\path[commutative diagrams/.cd, every arrow, every label, font=\scriptsize]
		(AE) edge["Price"]     (AD)
				 edge["Position"'] (AS)
		(AF) edge["Price"']    (AD)
				 edge["Position"]  (AS);
	\node[draw, fit=(AE) (AD) (AF) (AS)] (A box) {};
	\node[left=0 of A box] {$A\coloneqq$};
	\matrix[matrix of math nodes, name=B, row sep=25pt, commutative diagrams/every cell, matrix anchor=south, right=3 of A.south east] {
			|(BD)|\LTO[\circ]{\$}
		\\
			|(BAS)|\LTO{Airline Seat}
		\\
			|(BS)|\LTO[\circ]{string}
		\\
	};
	\path[commutative diagrams/.cd, every arrow, every label, font=\scriptsize]
		(BAS) edge["Price"']    (BD)
				 edge["Position"] (BS);
	\node[draw, fit=(BAS) (BD) (BS)] (B box) {};				
	\node[right=0 of B box] {$=:B$};
%
\end{tikzpicture}
\end{equation}%
\index{database schema!mapping between}
Schema $A$ has more detail than schema $B$---an airline seat may be in first class or economy---but they are roughly the same. We will see that they can be connected by a functor, and that data conforming to $A$ can be migrated through this functor to schema $B$ and vice versa.

The statistics at the beginning of this section show that this sort of problem---when occurring at enterprise scale---continues to prove difficult and expensive. If one attempts to move data from a source schema to a target schema, the migrated data could fail to fit into the target schema or fail to satisfy some of its constraints. This happens surprisingly often in the world of business: a night may be spent moving data, and the next morning it is found to have arrived broken and unsuitable for further use. In fact, it is believed that over half of database migration projects fail.

In this chapter, we will discuss a category-theoretic method for migrating data. Using categories and functors, one can prove up front that a given data migration will not fail, i.e.\ that the result is guaranteed to fit into the target schema and satisfy all its constraints.

The material in this chapter gets to the heart of category theory: in particular, we discuss categories, functors, natural transformations, adjunctions, limits, and colimits. In fact, many of these ideas have been present in the discussion above:
\begin{itemize}
	\item The schema pictures, e.g.\ \cref{eqn.mySchema} depict categories $\cat{C}$.
	\item The instances, e.g.\ \cref{eqn.fav_ex_db} are functors from $\cat{C}$ to a certain category called $\smset$.
	\item The implicit mapping in \cref{eqn.airline_schemas}, which takes economy and first class seats in $A$ to airline seats in $B$, constitutes a functor $A\to B$.
	\item The notion of data migration for moving data between schemas is formalized by  adjoint functors.
\end{itemize}

We begin in \cref{sec.categories} with the definition of categories and a bunch of different sorts of examples. In \cref{sec.cat_fun_nt_db} we bring back databases, in particular their instances and the maps between them, by discussing functors and natural transformations. In \cref{sec.adjunctions_mig} we discuss data migration by way of adjunctions, which generalize the Galois connections we introduced in \cref{sec.galois_connections}. Finally in \cref{sec.bonus_lims_colims} we give a bonus section on limits and colimits.%
\footnote{By ``bonus,'' we mean that although not strictly essential
to the understanding of this particular chapter, limits and colimits will show up throughout the book and throughout
one's interaction with category theory, and we think the reader will especially
benefit from this material in the long run.}

\index{database|)}

\section{Categories}%
\label{sec.categories}%
\index{category|(}

A category $\cat{C}$ consists of four pieces of data---objects, morphisms,
identities, and a composition rule---satisfying two properties.

\begin{definition}%
\label{def.category}%
\index{category}
To specify a \emph{category} $\cat{C}$:
\begin{enumerate}[label=(\roman*)]
	\item one specifies a collection%
	\tablefootnote{Here, a \emph{collection} can be thought of as a bunch of things, just like a set,
	but that may be too large to formally be a set. An example is the
	collection of all sets, which would run afoul of Russell's paradox if it
	were itself a set.}
	$\Ob(\cat{C})$, elements of which are called
	\emph{objects}.%
\index{collection}%
\index{category!object in}
	\item for every two objects $c,d$, one specifies a set $\cat{C}(c,d)$,%
	\tablefootnote{This set $\cat{C}(c,d)$ is often denoted
	$\Hom_{\cat{C}}(c,d)$, and called the ``hom-set from $c$ to $d$.'' The
	word ``hom'' stands for homomorphism, of which the word ``morphism'' is
	a shortened version.%
\index{hom-set}} 
	elements of which are called \emph{morphisms} from
	$c$ to $d$.%
\index{category!morphism in}
	\item for every object $c\in \Ob(\cat{C})$, one specifies a morphism $\id_c\in
	\cat{C}(c,c)$, called the \emph{identity morphism} on
	$c$.%
\index{identity!morphism}%
\index{morphism!identity}
	\item for every three objects $c,d,e\in\Ob(\cat{C})$ and morphisms $f\in
	\cat{C}(c,d)$ and $g\in \cat{C}(d,e)$, one specifies a morphism $f\cp g\in \cat{C}(c,e)$, called \emph{the composite of $f$ and $g$}.%
\index{category!composition in}%
\index{category!identity in}%
\index{composition!in a category|see {category, composition in}}
\end{enumerate}
We will sometimes write an object $c \in \cat{C}$, instead of $c \in
\Ob(\cat{C})$. It will also be convenient to denote elements $f\in\cat{C}(c,d)$
as $f\colon c\to d$. Here, $c$ is called the \emph{domain} of $f$, and $d$ is
called the \emph{codomain} of $f$.%
\index{domain|see {morphism, domain}}%
\index{codomain|see {morphism, codomain}}%
\index{morphism!domain}%
\index{morphism!codomain}

These constituents are required to satisfy two conditions:
\begin{enumerate}[label=(\alph*)]
	\item \emph{unitality}: for any morphism $f\colon c\to d$, composing with the identities at $c$ or $d$ does nothing: $\id_c\cp f=f$ and $f\cp \id_d=f$.
	\item \emph{associativity}: for any three morphisms $f\colon c_0\to
	c_1$, $g\colon c_1\to c_2$, and $h\colon c_2\to c_3$, the following are
	equal: $(f\cp g)\cp h=f\cp (g\cp h)$. We write this composite simply as $f\cp g\cp h$.
\end{enumerate}
\end{definition}%
\index{unitality}%
\index{associativity!of morphism composition}

Our next goal is to give lots of examples of categories. Our first source of
examples is that of free and finitely-presented categories, which generalize the
notion of Hasse diagram from \cref{rem.Hasse}.

\subsection{Free categories}%
\label{subsubsec.path_cats}%
Recall from \cref{def.graph} that a graph consists of two types of thing: vertices and arrows. From there one can define
paths, which are just head-to-tail sequences of arrows. Every path $p$ has a start
vertex and an end vertex; if $p$ goes from $v$ to $w$, we write $p\colon v\to w$. To every vertex $v$, there is a trivial path, containing no arrows,%
\index{trivial path}
starting and ending at $v$; we often denote it by $\id_v$ or simply by $v$. We may also concatenate paths: given $p\colon v\to w$ and $q\colon w\to x$, their concatenation is denoted $p\cp q$, and it goes $v\to x$.

In \cref{chap.preorders}, we used graphs to depict preorders $(V,\leq)$: the vertices form the elements of the preorder, and we say that $v\leq w$ if there is a path from $v$ to $w$ in $G$.
We will now use graphs in a very similar way to depict certain categories, known
as \emph{free categories}. Then we will explain a strong relationship between
preorders and categories in \cref{subsubsec.pos_free_spectrum}.

\begin{definition}%
\label{def.free_category}%
\index{category!free}%
\index{free!category}%
\index{graph!free category on}
For any graph $G = (V,A,s,t)$, we can define a category $\free(G)$,
called the \emph{free category on $G$}, whose objects are the vertices $V$ and
whose morphisms from $c$ to $d$ are the paths from $c$ to $d$. The identity
morphism on an object $c$ is simply the trivial path at $c$. Composition is given by concatenation of paths.%
\index{identity!morphism}%
\index{morphism!in free category}
\end{definition}

For example, we define $\Cat{2}$ to be the free category generated by the graph shown below:
\begin{equation}%
\label{eqn.graphs_1_2}
\Cat{2}\coloneqq\free\left(\;\raisebox{-.05in}{\fbox{
\begin{tikzcd}[ampersand replacement=\&]
	\LMO{v_1}\ar[r, "f_1"]\&\LMO{v_2}
\end{tikzcd}
}}
\;\right)
\end{equation}
It has two objects $v_1$ and $v_2$,
and three morphisms: $\id_{v_1}\colon v_1 \to v_1$, $f_1\colon v_1 \to v_2$,
and $\id_{v_2}\colon v_2 \to v_2$. Here $\id_{v_1}$ is the path of length 0
starting and ending at $v_1$, $f_1$ is the path of length 1 consisting
of just the arrow $f_1$, and $\id_{v_2}$ is the length 0 path at $v_2$. As our
notation suggests, $\id_{v_1}$ is the identity morphism for the object $v_1$, and
similarly $\id_{v_2}$ for $v_2$. As composition is given by concatenation, we
have, for example $\id_{v_1}\cp f_1 =f_1$, $\id_{v_2}\cp \id_{v_2}=\id_{v_2}$, and so
on.

From now on, we may elide the difference between a graph and the corresponding free category $\free(G)$, at least when the one we mean is clear enough from context.

\begin{exercise}%
\label{exc.free_cat}
For $\free(G)$ to really be a category, we must check that this data we
specified obeys the unitality and associativity properties. Check that these
are obeyed for any graph $G$.%
\index{associativity}
\end{exercise}

\begin{exercise}%
\label{exc.free_cat2}
The free category on the graph shown here:%
\footnote{As mentioned above, we elide the difference between the graph and the corresponding free category.}
\begin{equation}%
\label{eqn.graphs_rand9851}
\Cat{3}\coloneqq{\color{black!20!white}\mathbf{Free}\bigg(}\;\raisebox{-.05in}{\fbox{
\begin{tikzcd}[ampersand replacement=\&]
	\LMO{v_1}\ar[r, "f_1"]\&\LMO{v_2}\ar[r, "f_2"]\&\LMO{v_3}
\end{tikzcd}
}}
{\color{black!20!white}\;\bigg)}
\end{equation}
has three objects and six morphisms: the three vertices and six paths in the graph.

Create six names, one for each of the six morphisms in $\Cat{3}$. Write down a six-by-six table, label the rows and columns by the six names you chose.
\begin{enumerate}
	\item Fill out the table by writing the name of the composite in each cell, when there is a composite.
	\item Where are the identities?
\qedhere
\end{enumerate}
\end{exercise}

\begin{exercise}%
\label{exc.Cat_n}%
\index{ordinals!as categories}
Let's make some definitions, based on the pattern above:
\begin{enumerate}
	\item What is the category $\Cat{1}$? That is, what are its objects and morphisms?
	\item What is the category $\Cat{0}$?
	\item What is the formula for the number of morphisms in $\Cat{n}$ for arbitrary $n\in\NN$?
\qedhere
\end{enumerate}
\end{exercise}

\begin{example}[Natural numbers as a free category]
\label{ex.monoid_nats}%
\index{natural numbers!as free category}
Consider the following graph:
\begin{equation}%
\label{eqn.loop_graph}
\fbox{
\begin{tikzcd}[ampersand replacement=\&]
	\LMO[under]{z}\ar[loop above, "s"]
\end{tikzcd}
}
\end{equation}
It has only one vertex and one arrow, but it has infinitely many paths. Indeed, it
has a unique path of length $n$ for every natural number $n\in\NN$. That is,
$\Set{Path}=\{z, s, (s\cp s), (s\cp s\cp s), \ldots\}$, where we write $z$ for the
length 0 path on $z$; it represents the morphism $\id_z$. There is a one-to-one correspondence between $\Set{Path}$
and the natural numbers, $\NN=\{0,1,2,3,\ldots\}$.

This is an example of a category with one object. A category with one object is called a
\emph{monoid}%
\index{monoid}, a notion we first discussed in \cref{ex.monoid}. There we said that a monoid is a tuple $(M,*,e)$ where $*\colon M\times M\to M$ is a function and $e\in M$ is an element, and $m*1=m=1*m$ and $(m*n)*p=m*(n*p)$.%
\index{monoid!as one-object category}

The two notions may superficially look different, but it is easy to describe the connection. Given a category $\cat{C}$ with one object, say $\bullet$, let $M\coloneqq\cat{C}(\bullet,\bullet)$, let $e=\id_{\bullet}$, and let $*\colon\cat{C}(\bullet,\bullet)\times\cat{C}(\bullet,\bullet)\to\cat{C}(\bullet,\bullet)$ be the composition operation $*=\cp$. The associativity and unitality requirements for the monoid will be satisfied because $\cat{C}$ is a category.%
\index{associativity}%
\index{unitality}
\end{example}

\begin{exercise}%
\label{exc.nat_comp}
In \cref{ex.monoid_nats} we identified the paths of the loop graph \eqref{eqn.loop_graph} with numbers $n\in\NN$. Paths can be concatenated. Given numbers $m,n\in\NN$, what number corresponds to the concatenation of their associated paths?
\end{exercise}

\subsection{Presenting categories via path equations}
\label{subsec.presenting_cats}
\index{category!presentation of}

So for any graph $G$, there is a free category on $G$. But we don't have to stop
there: we can add equations between paths in the graph, and still get a
category. We are only allowed to equate two paths $p$ and $q$ when they are \emph{parallel}, meaning they have the same source vertex and the same target vertex.

A finite graph with path equations is called a \emph{finite presentation} for a
category, and the category that results is known as a \emph{finitely-presented category}.
Here are two examples:
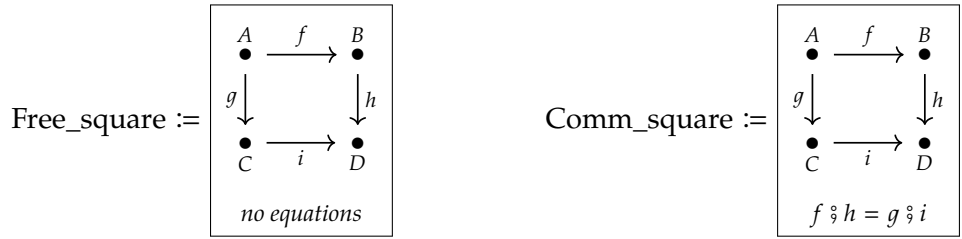
\[
\mathrm{Free\_square}\coloneqq\boxCD{
\begin{tikzcd}[ampersand replacement=\&]
	\LMO{A}\ar[r, "f"]\ar[d, "g"']\&\LMO{B}\ar[d, "h"]\\
	\LMO[under]{C}\ar[r, "i"']\&\LMO[under]{D}
\end{tikzcd}
  \\~\\\footnotesize
  \textit{no equations}
}
\hspace{.8in}
\mathrm{Comm\_square}\coloneqq
\boxCD{
\begin{tikzcd}[ampersand replacement=\&]
	\LMO{A}\ar[r, "f"]\ar[d, "g"']\&\LMO{B}\ar[d, "h"]\\
	\LMO[under]{C}\ar[r, "i"']\&\LMO[under]{D}
\end{tikzcd}
  \\~\\\footnotesize
  $f\cp h=g\cp i$
}
\]
Both of these are presentations of categories: in the left-hand one, there are no equations so it presents a free category, as discussed in \cref{subsubsec.path_cats}. The free square category has ten morphisms, because every path is a unique morphism.

\begin{exercise}%
\label{exc.label_free_square}~
\begin{enumerate}
	\item Write down the ten paths in the free square category above.
	\item Name two different paths that are parallel.
	\item Name two different paths that are not parallel.
	\qedhere
\end{enumerate}
\end{exercise}

On the other hand, the category presented on the right has only nine morphisms,
because $f\cp h$ and $g\cp i$ are made equal. This category is called the
``commutative square.''%
\index{commutative square} Its morphisms are
\[
\{A, B, C, D, f, g, h, i, f\cp h\}
\]
One might say ``the missing one is $g\cp i$,'' but that is not quite right: $g\cp i$ is there too, because it is equal to $f\cp h$. As usual, $A$ denotes $\id_A$, etc.

\begin{exercise}%
\label{exc.cat_gens_rels}
Write down all the morphisms in the category presented by the following diagram:
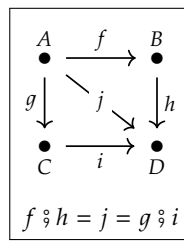
\[
\boxCD{
\begin{tikzcd}[ampersand replacement=\&]
	\LMO{A}\ar[r, "f"]\ar[d, "g"']\ar[dr, "j" description]\&\LMO{B}\ar[d, "h"]\\
	\LMO[under]{C}\ar[r, "i"']\&\LMO[under]{D}
\end{tikzcd}
\\~\\\footnotesize
  $f\cp h=j=g\cp i$
}
\]
\end{exercise}

\begin{example} %
\label{ex.group_of_order_2}%
\index{group}
We should also be aware that enforcing an equation between two morphisms often
implies additional equations. Here are two more examples of presentations, in
which this phenomenon occurs:
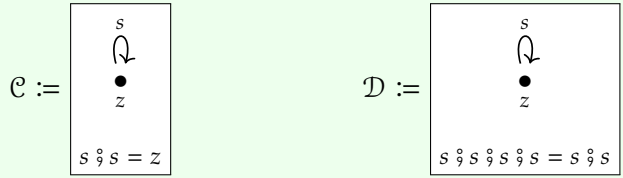
\[
\cat{C}\coloneqq\boxCD{\begin{tikzcd}[ampersand replacement=\&]
	\LMO[under]{z}\ar[loop above, "s"]
\end{tikzcd}
\\~\\\footnotesize
$s\cp s=z$
}
\hspace{1in}
\cat{D}\coloneqq\boxCD{\begin{tikzcd}[ampersand replacement=\&]
	\LMO[under]{z}\ar[loop above, "s"]
\end{tikzcd}
\\~\\\footnotesize
$s\cp s\cp s\cp s=s\cp s$
}
\]
In $\cat{C}$ we have the equation $s\cp s = z$. But this implies $s \cp s \cp s
= z \cp s = s$! And similarly we have $s \cp s \cp s \cp s = z \cp z = z$. The
set of morphisms in $\cat{C}$ is in fact merely $\{z,s\}$, with composition
described by $s\cp s= z \cp z =z$, and $z \cp s = s \cp z =s$. In group theory, one would speak of a group called $\zz/2\zz$.%
\index{presentation!of monoid}
\end{example}

\begin{exercise}%
\label{exc.group2}
Write down all the morphisms in the category $\cat{D}$ from \cref{ex.group_of_order_2}.
\end{exercise}

\begin{remark}%
\label{rem.db_schemas_are_cats}%
\index{database schema!as category presentation}
We can now see that the schemas in \cref{sec.C2_motivation}, e.g.\ \cref{eqn.free_schema,eqn.mySchema} are finite presentations of categories. We will come back to this idea in \cref{sec.cat_fun_nt_db}.
\end{remark}

\subsection{Preorders and free categories: two ends of a spectrum}%
\label{subsubsec.pos_free_spectrum}

Now that we have used graphs to depict preorders in \cref{chap.preorders} and categories above, one may want
to know the relationship between these two uses. The main idea we want to explain now is that
\begin{quote}
``A preorder is a category where every two parallel arrows are the same.''
\end{quote}
Thus any preorder can be regarded as a category, and any category can be somehow ``crushed down'' into a preorder. Let's discuss these ideas.

\paragraph{Preorders as categories.}%
\index{preorder!as category}

Suppose $(P,\leq)$ is a preorder. It specifies a category $\cat{P}$ as follows. The
objects of $\cat{P}$ are precisely the elements of $P$; that is,
$\Ob(\cat{P})=P$. As for morphisms, $\cat{P}$ has exactly one morphism $p\to q$ if
$p\leq q$ and no morphisms $p\to q$ if $p\not\leq q$. The fact that $\leq$ is reflexive ensures that every object has an identity, and the fact that $\leq$ is transitive ensures that morphisms can be composed. We call $\cat{P}$ the \emph{category corresponding to the preorder $(P,\leq)$}.%
\index{reflexivity!as identity in a preorder}%
\index{transitivity!as composition in a preorder}

In fact, a Hasse diagram for a preorder can be thought of a presentation of
a category where, for all vertices $p$ and $q$, every two paths from $p\to q$ are
declared equal. For example, in \cref{eqn.parts_of_3} we saw a Hasse diagram
that was like the graph on the left:
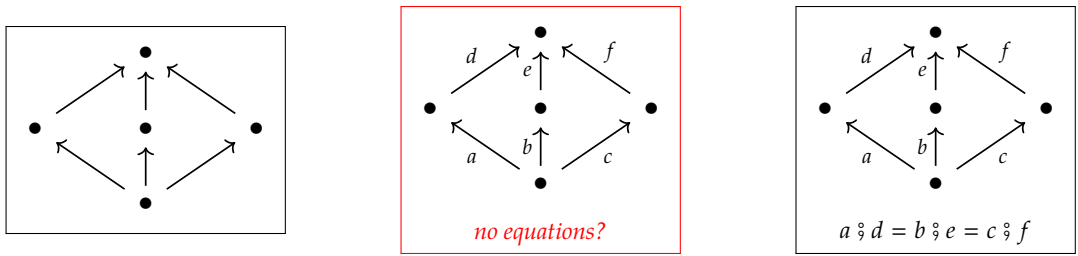
\[
\boxCD{
\begin{tikzcd}[row sep=15pt, ampersand replacement=\&]
	\&\bullet\\
	\bullet\ar[ur]\&\bullet\ar[u]\&\bullet\ar[ul]\\
	\&\bullet\ar[lu]\ar[u]\ar[ru]
\end{tikzcd}
  \\~\\\footnotesize
}
\hspace{.6in}
\boxCD[red]{
\begin{tikzcd}[row sep=15pt, ampersand replacement=\&]
	\&\bullet\\
	\bullet\ar[ur,"d"]\&\bullet\ar[u,"e"]\&\bullet\ar[ul,"f"']\\
	\&\bullet\ar[lu, "a"]\ar[u,"b"]\ar[ru,"c"']
\end{tikzcd}
  \\~\\\footnotesize
  \textit{\color{red}no equations?}
}
\hspace{.6in}
\boxCD{
\begin{tikzcd}[row sep=15pt, ampersand replacement=\&]
	\&\bullet\\
	\bullet\ar[ur,"d"]\&\bullet\ar[u,"e"]\&\bullet\ar[ul,"f"']\\
	\&\bullet\ar[lu, "a"]\ar[u,"b"]\ar[ru,"c"']
\end{tikzcd}
\\~\\\footnotesize
  $a\cp d = b\cp e = c\cp f$
}
\]
The Hasse diagram (left) might look the most like the free category presentation
(middle) which has no equations, but that is not correct. The free category has
three morphisms (paths) from bottom object to top object, whereas preorders are
categories with \emph{at most one} morphism between two given objects. Instead,
the diagram on the right, with these paths from bottom to top made equal, is the
correct presentation for the preorder on the left.

\begin{exercise}%
\label{exc.graph_to_preorder}
What equations would you need to add to the graphs below in order to present the associated preorders?
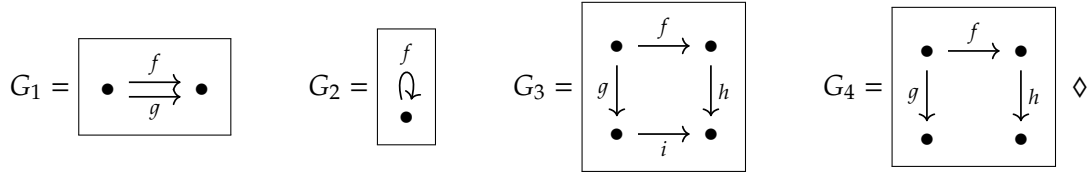
\[
G_1=\boxCD{
\begin{tikzcd}[ampersand replacement=\&, column sep=20pt]
	\bullet\ar[r, shift left, "f"]\ar[r, shift right, "g"']\&\bullet
\end{tikzcd}
}
\hspace{.4in}
G_2=\boxCD{
\begin{tikzcd}[ampersand replacement=\&]
	\bullet\ar[loop above, "f"]
\end{tikzcd}
}
\hspace{.4in}
G_3=\boxCD{
\begin{tikzcd}[ampersand replacement=\&, column sep=20pt]
	\bullet\ar[r, "f"]\ar[d, "g"']\&\bullet\ar[d, "h"]\\
	\bullet\ar[r, "i"']\&\bullet
\end{tikzcd}
}
\hspace{.4in}
G_4=\boxCD{
\begin{tikzcd}[ampersand replacement=\&, column sep=20pt]
	\bullet\ar[r, "f"]\ar[d, "g"']\&\bullet\ar[d, "h"]\\
	\bullet\&\bullet
\end{tikzcd}
}
\qedhere
\]
\end{exercise}

\paragraph{The preorder reflection of a category.}%
\index{category!preorder reflection of}

Given any category $\cat{C}$, one can obtain a preorder $(C,\leq)$ from it by destroying the distinction between any two parallel morphisms. That is, let $C\coloneqq\Ob(\cat{C})$, and put $c_1\leq c_2$ iff $\cat{C}(c_1,c_2)\neq\varnothing$. If there is one, or two, or fifty, or infinitely many morphisms $c_1\to c_2$ in $\cat{C}$, the preorder reflection does not see the difference. But it does see the difference between some morphisms and no morphisms.%
\index{morphism!inequality as mere existence of}

\begin{exercise}%
\label{exc.preorder_refl_N}
What is the preorder reflection of the category $\NN$ from \cref{ex.monoid_nats}?
\end{exercise}

We have only discussed adjoint functors between preorders, but soon we will discuss adjoints in general. Here is a statement you might not understand exactly, but it's true; you can ask a category theory expert about it and they should be able to explain it to you:
\begin{quote}
Considering a preorder as a category is right adjoint to turning a category into a preorder by preorder reflection.
\end{quote}%
\index{adjunction!examples of}

\begin{remark}[Ends of a spectrum]
The main point of this subsection is that both preorders and free categories are specified by a graph without path equations, but they denote opposite ends of a spectrum. In both cases, the vertices of the graph become the objects of a category and the paths become morphisms. But in the case of free categories, there are no equations so each path becomes a different morphism. In the case of preorders, all parallel paths become the same morphism. Every category presentation, i.e.\ graph with some equations, lies somewhere in between the free category (no equations) and its preorder reflection (all possible equations).
\end{remark}

\subsection{Important categories in mathematics} %
\label{subsec.important_cats}
We have been talking about category presentations, but there are categories that are best understood directly, not by way of presentations. Recall the definition of category from \cref{def.category}. The most important category in mathematics is the category of sets.

\begin{definition}%
\label{def.category_of_sets}%
\index{sets!category of}%
\index{category!of sets}
The \emph{category of sets}, denoted $\smset$, is defined as follows.
\begin{enumerate}[label=(\roman*)]
	\item $\Ob(\smset)$ is the collection of all sets.
	\item If $S$ and $T$ are sets, then $\smset(S,T)=\{f\colon S\to T\mid f\text{ is a function}\}$.
	\item For each set $S$, the identity morphism is the function $\id_S\colon S\to S$ given by $\id_S(s)\coloneqq s$ for each $s\in S$.
	\item Given $f\colon S\to T$ and $g\colon T\to U$, their composite is the function $f\cp g\colon S\to U$ given by $(f\cp g)(s)\coloneqq g(f(s))$.
\end{enumerate}
These definitions satisfy the unitality and associativity conditions, so $\smset$ is indeed a category.%
\index{associativity!of function composition}%
\index{unitality!of identity functions}
\end{definition}

Closely related is the category $\finset$. This is the category whose
objects are finite sets and whose morphisms are functions between them.
\index{category!of finite sets}

\begin{exercise}%
\label{exc.exponential_practice}
	Let $\ord{2}=\{1,2\}$ and $\ord{3}=\{1,2,3\}$. These are objects in the
	category $\smset$ discussed in \cref{def.category_of_sets}. Write down all the elements of the set $\smset(\ord{2},\ord{3})$; there should be nine.
\end{exercise}

\begin{remark} %
\label{rem.cats_and_vcats1}%
\index{enriched category!vs category}
You may have wondered what categories have to do with
$\cat{V}$-categories (\cref{def.cat_enriched_mpos}); perhaps you think the
definitions hardly look alike. Despite the term `enriched category',
$\cat{V}$-categories are not categories with extra structure. While some sorts
of $\cat{V}$-categories, such as $\Bool$-categories, i.e.\ preorders, can naturally be
seen as categories, other sorts, such as $\Cost$-categories, cannot.

The reason for the importance of $\smset$ is that, if we generalize the
definition of enriched category (\cref{def.cat_enriched_mpos}), we find that
categories in the sense of \cref{def.category} are exactly $\smset$-categories---so categories are
$\cat{V}$-categories for a very special choice of $\cat{V}$. We'll come back to
this in \cref{subsec.SMC_enrichment}.  For now, we simply remark that just like
a deep understanding of the category $\Cost$---for example, knowing that it is a
quantale---yields insight into Lawvere metric spaces, so the study of $\smset$
yields insights into categories. 
\end{remark}

There are many other categories that mathematicians care about:
\begin{itemize}%
\index{category!examples of}
	\item $\Cat{Top}$: the category of topological spaces (neighborhood)
	\item $\Cat{Grph}$: the category of graphs (connection)
	\item $\Cat{Meas}$: the category of measure spaces (amount)
	\item $\Cat{Mon}$: the category of monoids (action)
	\item $\Cat{Grp}$: the category of groups (reversible action, symmetry)
	\item $\Cat{Cat}$: the category of categories (action in context, structure)
\end{itemize}
But in fact, this does not at all do justice to the diversity of categories mathematicians
think about. They work with whatever category they find fits their purpose at
the time, like `the category of connected Riemannian manifolds of dimension at most
4'.

Here is one more source of examples: take any category you already have and reverse all its morphisms; the result is again a category.
\begin{example}%
\label{def.opposite_cat}%
\index{category!opposite of}%
\index{opposite!category|see {category, opposite}}
Let $\cat{C}$ be a category. Its \emph{opposite}, denoted $\cat{C}\op$, is the category with the same objects, $\Ob(\cat{C}\op)\coloneqq\Ob(\cat{C})$, and for any two objects $c,d\in\Ob(\cat{C})$, one has $\cat{C}\op(c,d)\coloneqq\cat{C}(d,c)$. Identities and composition are as in $\cat{C}$.
\end{example}

\subsection{Isomorphisms in a category}%
\index{isomorphism|(}

The previous sections have all been about examples of categories: free categories, presented categories, and important categories in math. In this section, we briefly switch gears and talk about an important concept in category theory, namely the concept of isomorphism.

In a category, there is often the idea that two objects are interchangeable. For
example, in the category $\smset$, one can exchange the set $\{\blacksquare,\square\}$ for the set $\{0,1\}$
and everything will be the same, other than the names for the elements. Similarly, if one has a preorder with elements $a,b$, such that $a \le b$ and $b \le a$, i.e.\ $a\cong b$, then $a$ and $b$ are essentially the same. How so? Well they act the same, in that for any other object $c$, we know that $c \le a$ iff $c
\le b$, and $c\ge a$ iff $c\geq b$. The notion of isomorphism formalizes this notion of interchangeability.

\begin{definition}%
\index{isomorphism}
An \emph{isomorphism} is a morphism $f\colon A \to B$ such that there exists a
morphism $g\colon B \to A$ satisfying $f\cp g=\id_A$ and $g\cp f=\id_B$. In this case
we call $f$ and $g$ \emph{inverses}, and we often write $g=f\inv$, or
equivalently $f=g\inv$. We also say that $A$ and $B$ are \emph{isomorphic}
objects.
\end{definition}%
\index{morphism!invertible}

\begin{example}%
\label{ex.simple_iso}
The set $A\coloneqq\{a,b,c\}$ and the set $\ord{3}=\{1,2,3\}$ are isomorphic; that is, there exists an isomorphism $f\colon A\to \ord{3}$ given by $f(a)=2$, $f(b)=1$, $f(c)=3$. The isomorphisms in the category $\smset$ are the bijections.%
\index{isomorphism!bijection as}
\end{example}

Recall that the cardinality of a finite set is the number of elements in it. This can be understood in terms of isomorphisms in $\finset$. Namely, for any finite set $A\in\finset$, its cardinality is the number $n\in\nn$ such that there exists an isomorphism $A\cong\ord{n}$.%
\index{cardinality!and isomorphisms} Georg Cantor defined the cardinality of any set $X$ to be its isomorphism class, meaning the equivalence class consisting of all sets that are isomorphic to $X$. 

\begin{exercise}%
\label{exc.iso_practice}
\begin{enumerate}
	\item What is the inverse $f\inv\colon\ord{3}\to A$ of the function $f$ given in \cref{ex.simple_iso}?
	\item How many distinct isomorphisms are there $A\to\ord{3}$?
	\qedhere
\end{enumerate}
\end{exercise}

\begin{exercise}%
\label{exc.id_iso}
Show that in any given category $\cat{C}$, for any given object $c\in\cat{C}$, the identity $\id_c$ is an isomorphism.
\end{exercise}

\begin{exercise}%
\label{exc.monoid_group}%
\index{group}%
\index{monoid!group as}
Recall Examples \ref{ex.monoid_nats} and \ref{ex.group_of_order_2}. A monoid in
which every morphism is an isomorphism is known as a \emph{group}. 
\begin{enumerate}
  \item Is the monoid in \cref{ex.monoid_nats} a group?
  \item What about the monoid $\cat{C}$ in \cref{ex.group_of_order_2}?
  \qedhere
\end{enumerate}
\end{exercise}

\begin{exercise}%
\index{free!category}%
\label{exc.iso_free_cat}
Let $G$ be a graph, and let $\free(G)$ be the corresponding free category. Somebody tells you that the only isomorphisms in $\free(G)$ are the identity morphisms. Is that person correct? Why or why not?
\end{exercise}

\begin{example}
In this example, we will see that it is possible for $g$ and $f$ to be almost---but not quite---inverses, in a certain sense.

Consider the functions $f\colon\ord{2}\to\ord{3}$ and $g\colon\ord{3}\to\ord{2}$ drawn below:
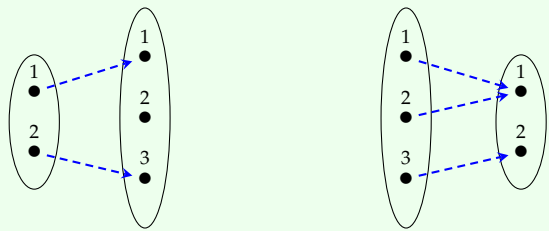
\[
\begin{tikzpicture}[y=.35cm, short=-2pt]
	\node[label={[above=-5pt, font=\tiny]:$1$}] (A1) {$\bullet$};
	\node[below=1 of A1, label={[above=-5pt, font=\tiny]:$2$}] (A2) {$\bullet$};
	\node[ellipse, draw, inner sep=0pt, fit=(A1) (A2)] (A) {};
	\node[above right=0 and 1 of A1, label={[above=-5pt, font=\tiny]:$1$}] (B1) {$\bullet$};
	\node[below=1 of B1, label={[above=-5pt, font=\tiny]:$2$}] (B2) {$\bullet$};
	\node[below=1 of B2, label={[above=-5pt, font=\tiny]:$3$}] (B3) {$\bullet$};
	\node[ellipse, draw, inner sep=0pt, fit=(B1) (B3)] (B) {};
	\node[right=3 of B1, label={[above=-5pt, font=\tiny]:$1$}] (C1) {$\bullet$};
	\node[below=1 of C1, label={[above=-5pt, font=\tiny]:$2$}] (C2) {$\bullet$};
	\node[below=1 of C2, label={[above=-5pt, font=\tiny]:$3$}] (C3) {$\bullet$};
	\node[ellipse, draw, inner sep=0pt, fit=(C1) (C3)] (C) {};
	\node[right=6 of A1, label={[above=-5pt, font=\tiny]:$1$}] (D1) {$\bullet$};
	\node[below=1 of D1, label={[above=-5pt, font=\tiny]:$2$}] (D2) {$\bullet$};
	\node[ellipse, draw, inner sep=0pt, fit=(D1) (D2)] (D) {};
	\draw[mapsto] (A1) to (B1);
	\draw[mapsto] (A2) to (B3);
	\draw[mapsto] (C1) to (D1);
	\draw[mapsto] (C2) to (D1);
	\draw[mapsto] (C3) to (D2);
\end{tikzpicture}
\]
Then the reader should be able to instantly check that $f\cp g=\id_{\ord{2}}$ but
$g\cp f\neq\id_{\ord{3}}$. Thus $f$ and $g$ are not inverses and hence not
isomorphisms. We won't need this terminology, but category theorists would say that $f$ and $g$ form a \emph{retraction}.%
\index{retraction}
\end{example}

\index{isomorphism|)}%
\index{category|)}

\section{Functors, natural transformations, and databases}%
\label{sec.cat_fun_nt_db}

In \cref{sec.C2_motivation} we showed some database schemas: graphs with path equations. Then in \cref{subsec.presenting_cats} we said that graphs with path equations correspond to finitely-presented categories. Now we want to explain what the data in a database is, as a way to introduce functors. To do so, we begin by noticing that sets and functions---the objects and morphisms in the category $\smset$---can be captured by particularly simple databases.

\subsection{Sets and functions as databases}
\index{database!schema|see {database schema}}%
\index{category!as database schema}%
\index{schema|see database schema}

The first observation is that any set can be understood as a table with only one column: the ID column.
\[
\begin{tabular}{ c |}
  \textbf{Planet of Sol}\\\hline
  Mercury\\
	Venus\\
	Earth\\
	Mars\\
	Jupiter\\
	Saturn\\
	Uranus\\
	Neptune
\end{tabular}
\hspace{.7in}
\begin{tabular}{ c |}
  \textbf{Prime number}\\\hline
  2\\
	3\\
	5\\
	7\\
	11\\
	13\\
	17\\
	$\vdots$
\end{tabular}
\hspace{.7in}
\begin{tabular}{ c |}
  \textbf{Flying pig}\\\hline
  ~\\
	~\\
	~\\
	~\\
	~\\
	~\\
	~\\
	~
\end{tabular}
\]
Rather than put the elements of the set between braces, e.g.\ $\{2,3,5,7,11,\ldots\}$, we write them down as rows in a table.

In databases, single-column tables are often called controlled vocabularies, or master data. Now to be honest, we can only write out every single entry in a table when its set of rows is finite. A database practitioner might find the idea of our prime number table a bit unrealistic. But we're mathematicians, so since the idea makes perfect sense abstractly, we will continue to think of sets as one-column tables.

The above databases have schemas consisting of just one vertex:
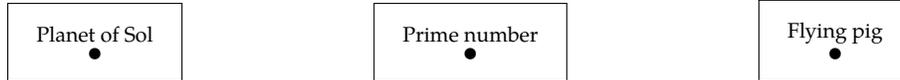
\[
\fbox{
\begin{tikzcd}[ampersand replacement=\&, column sep=50pt]
	\LTO{Planet of Sol}
\end{tikzcd}
}
\hspace{1in}
\fbox{
\begin{tikzcd}[ampersand replacement=\&, column sep=50pt]
	\LTO{Prime number}
\end{tikzcd}
}
\hspace{1in}
\fbox{
\begin{tikzcd}[ampersand replacement=\&, column sep=50pt]
	\LTO{Flying pig}
\end{tikzcd}
}
\]
Obviously, there's really not much difference between these schemas, other than the label of the unique vertex. So we could say ``sets are databases whose schema consists of a single vertex.'' Let's move on to functions.

A function $f\colon A\to B$ can almost be depicted as a two-column table
\[
\begin{tabular}{ c | c}
  \textbf{Beatle}&\textbf{Played}\\\hline
  George&Lead guitar\\
  John&Rhythm guitar\\
  Paul&Bass guitar\\
  Ringo&Drums
\end{tabular}
\]
except it is unclear whether the elements of the right-hand column exhaust all of $B$. What if there are rock-and-roll instruments out there that none of the Beatles played? So a function $f\colon A\to B$ requires two tables, one for $A$ and its $f$ column, and one for $B$:
\[
\begin{tabular}{ c | c}
  \textbf{Beatle}&\textbf{Played}\\\hline
  George&Lead guitar\\
  John&Rhythm guitar\\
  Paul&Bass guitar\\
  Ringo&Drums\\
  ~
\end{tabular}
\hspace{1in}
\begin{tabular}{ c |}
  \textbf{Rock-and-roll instrument}\\\hline
  Bass guitar\\
  Drums\\
  Keyboard\\
  Lead guitar\\
  Rhythm guitar
 \end{tabular}
\]
Thus the database schema for any function is just a labeled version of $\Cat{2}$:%
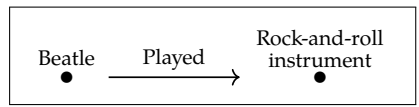
\index{function!as database instance}
\[
\fbox{
\begin{tikzcd}[ampersand replacement=\&, column sep=50pt]
	\LTO{Beatle}\ar[r, "\text{Played}"]\&\LTO{\parbox{.7in}{\centering Rock-and-roll\\\vspace{-.1in}instrument}}
\end{tikzcd}
}
\]
The lesson is that an instance of a database takes a presentation of a category,
and turns every vertex into a set, and every arrow into a function. As such, it
describes a map from the presented category to the category $\smset$. In
\cref{subsec.enriched_functors} we saw that maps of $\cat{V}$-categories are
known as $\cat{V}$-functors. Similarly, we call maps of plain old categories,
functors.

\subsection{Functors}%
\label{sec.functors}%
\index{functor|(}

A functor is a mapping between categories. It sends objects to objects and morphisms to morphisms, all while preserving identities and composition. Here is the formal definition.

\begin{definition}%
\index{functor}
Let $\cat{C}$ and $\cat{D}$ be categories. To specify a \emph{functor from $\cat{C}$ to $\cat{D}$}, denoted $F\colon\cat{C}\to\cat{D}$, 
\begin{enumerate}[label=(\roman*)]
	\item for every object $c\in\Ob(\cat{C})$, one specifies an object $F(c)\in\Ob(\cat{D})$;
	\item for every morphism $f\colon c_1\to c_2$ in $\cat{C}$, one specifies a morphism $F(f)\colon F(c_1)\to F(c_2)$ in $\cat{D}$.
\end{enumerate}
The above constituents must satisfy two properties:
\begin{enumerate}[label=(\alph*)]
	\item for every object $c\in\Ob(\cat{C})$, we have $F(\id_c)=\id_{F(c)}$.
	\item for every three objects $c_1,c_2,c_3\in\Ob(\cat{C})$ and two morphisms $f\in\cat{C}(c_1,c_2)$, $g\in\cat{C}(c_2,c_3)$, the equation $F(f\cp g)=F(f)\cp F(g)$ holds in $\cat{D}$.
\end{enumerate}
\end{definition}

\begin{example}
For example, here we draw three functors $F\colon\Cat{2}\to\Cat{3}$:
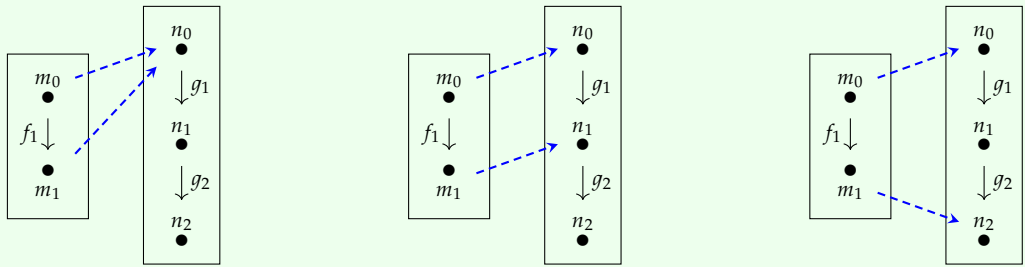
\[
\begin{tikzpicture}[x=.7in, y=.25in, inner sep=5pt, short=0pt]
	\foreach \i in {0,1,2}{
  	\node (A\i-n0) at (3*\i,-1) {$\LMO{m_0}$};
  	\node (A\i-n1) at (3*\i,-3) {$\LMO[under]{m_1}$};
  	\draw[->] (A\i-n0) to node[left=-2pt, font=\scriptsize] {$f_1$} (A\i-n1);
  	\node[draw, inner ysep=1pt, fit=(A\i-n0) (A\i-n1)] (A\i) {};
  	\node (B\i-n0) at (3*\i+1,0) {$\LMO{n_0}$};
  	\node (B\i-n1) at (3*\i+1,-2) {$\LMO{n_1}$};
  	\node (B\i-n2) at (3*\i+1,-4) {$\LMO{n_2}$};
  	\draw[->] (B\i-n0) to node[right=-2pt, font=\scriptsize] {$g_1$} (B\i-n1);
  	\draw[->] (B\i-n1) to node[right=-2pt, font=\scriptsize] {$g_2$} (B\i-n2);
  	\node[draw, inner ysep=1pt, fit=(B\i-n0) (B\i-n2)] (B\i) {};
	}
	\begin{scope}[mapsto]
  	\draw (A0-n0) -- (B0-n0);
  	\draw (A0-n1) -- (B0-n0);
		\draw (A1-n0) -- (B1-n0);
  	\draw (A1-n1) -- (B1-n1);
		\draw (A2-n0) -- (B2-n0);
  	\draw (A2-n1) -- (B2-n2);
	\end{scope}
\end{tikzpicture}
\]
In each case, the dotted arrows show what the functor $F$ does to the vertices in $\Cat{2}$; once that information is specified, it turns out---in this special case---that what $F$ does to the three paths in $\Cat{2}$ is completely determined. In the left-hand diagram, $F$ sends every path to the trivial path, i.e.\ the identity on $n_0$. In the middle diagram $F(m_0)=n_0$, $F(f_1)=g_1$, and $F(m_1)=n_1$. In the right-hand diagram, $F(m_0)=n_0$, $F(m_1)=n_2$, and $F(f_1)=g_1\cp g_2$.%
\index{trivial path}
\end{example}

\begin{exercise}%
\label{exc.all_functors}
Above we wrote down three functors $\Cat{2}\to\Cat{3}$. Find and write down all
the remaining functors $\Cat{2}\to\Cat{3}$.
\end{exercise}

\begin{example}%
\index{commutative square}%
\label{ex.free_comm_sq}
Recall the categories presented by $\mathrm{Free\_square}$ and
$\mathrm{Comm\_square}$ in \cref{subsec.presenting_cats}. Here they are again,
with $'$ added to the labels in $\mathrm{Free\_square}$ to help distinguish
them:
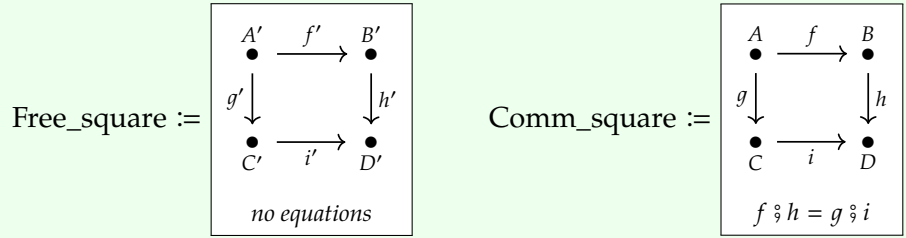
\[
\mathrm{Free\_square}\coloneqq\boxCD{
\begin{tikzcd}[ampersand replacement=\&]
	\LMO{A'}\ar[r, "f'"]\ar[d, "g'"']\&\LMO{B'}\ar[d, "h'"]\\
	\LMO[under]{C'}\ar[r, "i'"']\&\LMO[under]{D'}
\end{tikzcd}
  \\~\\\footnotesize
  \textit{no equations}
}
\hspace{.4in}
\mathrm{Comm\_square}\coloneqq\boxCD{
\begin{tikzcd}[ampersand replacement=\&]
	\LMO{A}\ar[r, "f"]\ar[d, "g"']\&\LMO{B}\ar[d, "h"]\\
	\LMO[under]{C}\ar[r, "i"']\&\LMO[under]{D}
\end{tikzcd}
  \\~\\\footnotesize
  $f\cp h=g\cp i$
}
\]
There are lots of functors from the free square category (let's call it
$\cat{F}$) to the commutative square category (let's call it $\cat{C}$). 

However, there is exactly one functor $F\colon\cat{F}\to\cat{C}$ that
sends $A'$ to $A$, $B'$ to $B$, $C'$ to $C$, and $D'$ to $D$. That is, once we
have made this decision about how $F$ acts on objects, each of the ten paths in $\cat{F}$ is forced to go to a certain path in $\cat{C}$: the one with the right source and target.
\end{example}

\begin{exercise}%
\label{exc.functor_on_morphisms}
Say where each of the ten morphisms in $\cat{F}$ is sent under the functor $F$ from \cref{ex.free_comm_sq}.
\end{exercise}

All of our example functors so far have been completely determined by what they do on objects, but this is usually not the case.

\begin{exercise}%
\label{exc.functors_morphisms_practice}
Consider the free categories $\cat{C}=\fbox{$\bullet\to\bullet$}$ and $\cat{D}=\fbox{$\bullet\tto\bullet$}$. Give two functors $F,G\colon\cat{C}\to\cat{D}$ that act the same on objects but differently on morphisms.
\end{exercise}

\begin{example}
There are also lots of functors from the commutative square category $\cat{C}$ to the free square category $\cat{F}$, but \emph{none} that sends $A$ to $A'$, $B$ to $B'$, $C$ to $C'$, and $D$ to $D'$. The reason is that if $F$ were such a functor, then since $f\cp h=g\cp i$ in $\cat{C}$, we would have $F(f\cp h)=F(g\cp i)$, but then the rules of functors would let us reason as follows:
\[f'\cp h'=F(f)\cp F(h)=F(f\cp h)=F(g\cp i)=F(g)\cp F(i)=g'\cp i'\]
The resulting equation, $f'\cp h'=g'\cp i'$ does not hold in $\cat{F}$ because it is a free category (there are ``no equations''): every two paths are considered different morphisms. Thus our proposed $F$ is not a functor.
\end{example}

\begin{example}[Functors between preorders are monotone maps]
\label{ex.preorder_functor}%
\index{monotone map!as functor}
Recall from \cref{subsubsec.pos_free_spectrum} that preorders are categories with
at most one morphism between any two objects. A functor between preorders is
exactly a monotone map.

For example, consider the preorder $(\NN,\leq)$ considered as a category $\cat{N}$ with objects $\Ob(\cat{N})=\NN$ and a unique morphism $m\to n$ iff $m\leq n$. A functor $F\colon\cat{N}\to\cat{N}$ sends each object $n\in\NN$ to an object $F(n)\in\NN$. It must send morphisms in $\cat{N}$ to morphisms in $\NN$. This means if there is a morphism $m\to n$ then there had better be a morphism $F(m)\to F(n)$. In other words, if $m\leq n$, then we had better have $F(m)\leq F(n)$. But as long as $m\leq n$ implies $F(m)\leq F(n)$, we have a functor.

Thus a functor $F\colon\cat{N}\to\cat{N}$ and a monotone map $\NN\to\NN$ are the same thing.
\end{example}

\begin{exercise}[The category of categories]%
\label{exc.cat_of_cats}%
\index{primordial ooze}%
\index{identity!functor}%
\index{category of categories|see {category, of categories}}%
\index{category!of categories}
Back in the primordial ooze, there is a category $\Cat{Cat}$ in which \emph{the
objects are themselves categories}. Your task here is to construct this
category.
\begin{enumerate}
  \item Given any category $\cat{C}$, show that there exists a functor $\id_{\cat{C}}\colon
  \cat{C} \to \cat{C}$, known as the \emph{identity functor on $\cat{C}$}, that
  maps each object to itself and each morphism to itself.
\end{enumerate}
Note that a functor $\cat{C} \to \cat{D}$ consists of a function from $\Ob(\cat{C})$ to $\Ob(\cat{D})$ and for each pair of objects $c_1,c_2 \in \cat{C}$ a function from $\cat{C}(c_1,c_2)$ to $\cat{D}(F(c_1),F(c_2))$. 
\begin{enumerate}[resume]
  \item Show that given
  $F\colon \cat{C} \to \cat{D}$ and $G\colon \cat{D} \to \cat{E}$, we can define a
  new functor $(F\cp G)\colon \cat{C} \to \cat{E}$ just by composing functions. 
  \item Show that there is a category, call it $\Cat{Cat}$, where the objects are categories, morphisms
  are functors, and identities and composition are given as above.
  \qedhere
\end{enumerate}
\end{exercise}

\index{functor|)}

\subsection{Database instances as $\smset$-valued functors}%
\index{database!instance|(}

Let $\cat{C}$ be a category, and recall the category $\smset$ from \cref{def.category_of_sets}. A functor $F\colon\cat{C}\to\smset$ is known as a \emph{set-valued functor} on $\cat{C}$. Much of database theory (not how to make them fast, but what they are and what you do with them) can be cast in this light.

Indeed, we already saw in \cref{rem.db_schemas_are_cats} that any database schema can be regarded as (presenting) a 
category $\cat{C}$. The next thing to notice is that the data itself---any
instance of the database---is given by a set-valued functor
$I\colon\cat{C}\to\smset$. The only additional detail is that for any white
node, such as $c=\LTO[\circ]{string}$, we want to force $I$ to map to the set of
strings. We suppress this detail in the following definition.

\begin{definition}%
\label{rdef.instance}%
\index{category!finitely presented|see {presentation of}}%
\index{database schema}%
\index{category!presentation of}
Let $\cat{C}$ be a schema, i.e.\ a finitely-presented category. A
\emph{$\cat{C}$-instance} is a functor
$I\colon\cat{C}\to\smset$.%
\tablefootnote{Warning: a $\cat{C}$-instance is a state of the database ``at an instant in time.'' The term ``instance'' should not be confused with its usage in object oriented programming, which would correspond more to what we call a row $r\in I(c)$.}
\end{definition}

\begin{exercise}%
\label{ex.set_1}
Let $\Cat{1}$ denote the category with one object, called 1, one identity morphism $\id_1$, and no other morphisms. For any functor $F\colon\Cat{1}\to\smset$ one can extract a set $F(1)$. Show that for any set $S$, there is a functor $F_S\colon\Cat{1}\to\smset$ such that $F_S(1)=S$.
\end{exercise}

The above exercise reaffirms that the set of planets, the set of prime numbers, and the set of flying pigs are all set-valued functors---instances---on the schema $\Cat{1}$. Similarly, set-valued functors on the category $\Cat{2}$ are functions. All our examples so far are for the situation where the schema is a free category (no equations). Let's try an example of a category that is not free.

\begin{example}
Consider the following category:
\begin{equation}%
\label{eqn.idempotent}
\cat{C}\coloneqq\boxCD{\begin{tikzcd}[ampersand replacement=\&]
	\LMO[under]{z}\ar[loop above, "s"]
\end{tikzcd}
\\~\\\footnotesize
$s\cp s=s$
}
\end{equation}
What is a set-valued functor $F\colon\cat{C}\to\smset$? It will consist of a set $Z\coloneqq F(z)$ and a function $S\coloneqq F(s)\colon Z\to Z$, subject to the requirement that $S\cp S=S$. Here are some examples
\begin{itemize}
	\item $Z$ is the set of US citizens, and $S$ sends each citizen to her
	or his president. The president's president is her- or him-self.
	\item $Z=\NN$ is the set of natural numbers and $S$ sends each number to $0$. In particular, 0 goes to itself.
	\item $Z$ is the set of all well-formed arithmetic expressions, such as $13+(2*4)$ or $-5$, that one can write using integers and the symbols $+,-,*,(,)$. The function $S$ evaluates the expression to return an integer, which is itself a well-formed expression. The evaluation of an integer is itself.
	\item $Z=\NN_{\geq 2}$, and $S$ sends $n$ to its smallest prime factor. The smallest prime factor of a prime is itself.
\end{itemize}
\[
\begin{array}{ c | c}
  \NN_{\geq2}&\textbf{\small smallest prime factor}\\\hline
	2&2\\
	3&3\\
	4&2\\
  \vdots&\vdots\\
  49&7\\
  50&2\\
  51&3\\
  \vdots&\vdots
\end{array}
\qedhere
\]
\end{example}

\begin{exercise}%
\label{exc.schema_sense}
Above, we thought of the sort of data that would make sense for the schema \eqref{eqn.idempotent}. Give an example of the sort of data that would make sense for the following schemas:\qquad
\begin{enumerate*}[itemjoin=\hspace{1in}]
\item \boxCD{\begin{tikzcd}[ampersand replacement=\&]
	\LMO[under]{z}\ar[loop above, "s"]
\end{tikzcd}
\\~\\\footnotesize
$s\cp s=z$
}
\item
\boxCD{\begin{tikzcd}[ampersand replacement=\&]
	\LMO{a}\ar[r, "f"]\&\LMO{b}\ar[r, shift left, "g"]\ar[r, shift right, "h"']\&\LMO{c}
\end{tikzcd}
\\~\\\footnotesize
$f\cp g=f\cp h$
}
\qedhere
\end{enumerate*}
\end{exercise}

The main idea is this: a database schema is a category, and an instance on that schema---the data itself---is a set-valued functor. All the constraints, or business rules, are ensured by the rules of functors, namely that functors preserve composition.%
\footnote{One can put more complex constraints, called \emph{embedded dependencies}, on a database; these correspond category theoretically to what are called ``lifting problems'' in category theory.%
\index{database!constraints}%
\index{lifting problems} See \cite{Spivak:2014c} for more on this.}

\index{database!instance|)}

\subsection{Natural transformations}%
\index{natural transformation|(}

If $\cat{C}$ is a schema---i.e.\ a finitely-presented category---then there are many database instances on it, which we can organize into a category. But this is part of a larger story, namely that of natural transformations. An abstract picture to have in mind is this:
\[
\begin{tikzcd}[column sep=large]
	\cat{C} 
	\ar[r, bend left, "F"{name=up}]
	\ar[r, bend right, "G"'{name=down}]
	&\cat{D}.
        \ar[from=up, to=down, Rightarrow, shorten <=5pt,
	shorten >=5pt,"\scriptstyle\alpha"]
\end{tikzcd}
\]

\begin{definition}%
\label{def.natural_transformation}
Let $\cat{C}$ and $\cat{D}$ be categories, and let $F,G\colon\cat{C}\to\cat{D}$
be functors. To specify a \emph{natural transformation} $\alpha\colon F\Rightarrow G$,
\begin{enumerate}[label=(\roman*)]
	\item for each object $c\in\cat{C}$, one specifies a morphism $\alpha_c\colon F(c)\to G(c)$ in $\cat{D}$, called the \emph{$c$-component of $\alpha$}.%
\index{natural transformation!component of}
\end{enumerate}
These components must satisfy the following, called the \emph{naturality condition}:
\begin{enumerate}[label=(\alph*)]
	\item for every morphism $f\colon c\to d$ in $\cat{C}$, the following equation must hold:
\[F(f)\cp\alpha_d=\alpha_c\cp G(f).\]	
\end{enumerate}

A natural transformation $\alpha\colon F\to G$ is called a \emph{natural isomorphism} if each component $\alpha_c$ is an isomorphism in $\cat{D}$.
\end{definition}
\index{natural transformation!naturality condition}%
\index{commutative diagram}%
\index{commutative
square}%
\index{diagram!commutative}

The naturality condition can also be written as a so-called \emph{commutative diagram}. A
diagram in a category is drawn as a graph whose vertices and arrows are labeled by objects and morphisms in the category. For example, here is a diagram that's relevant to the naturality condition in \cref{def.natural_transformation}:
\begin{equation}%
\label{eqn.naturality_condition}
\begin{tikzcd}
	F(c)\ar[r, "\alpha_c"]\ar[d, "F(f)"']&G(c)\ar[d, "G(f)"]\\
	F(d)\ar[r, "\alpha_d"']&G(d)
\end{tikzcd}
\end{equation}

\begin{definition}%
\label{def.diagram_commutes}%
\index{diagram!as functor}%
\index{functor!diagram as}%
\index{category!indexing}
A \emph{diagram} $D$ in $\cat{C}$ is a functor
$D\colon\cat{J} \to \cat{C}$ from any category $\cat{J}$, called the \emph{indexing category} of the diagram $D$. We say that $D$ \emph{commutes} if $D(f)=D(f')$ holds for every parallel pair of morphisms $f,f'\colon a \to b$ in $\cat{J}$.%
\tablefootnote{We could package this formally by saying that $D$ commutes iff it factors through the preorder reflection of $\cat{J}$.}
\end{definition}

In terms of \cref{eqn.naturality_condition}, the only case of two parallel morphisms is that of $F(c)\tto G(d)$, so to say that the diagram commutes is to say that $F(f)\cp\alpha_d=\alpha_c\cp G(f)$. This is exactly the naturality
condition from \cref{def.natural_transformation}.

\begin{example}
A representative picture is as follows:
\[
\begin{tikzpicture}[x=.7in, y=.3in, inner sep=5pt,short=0pt]
	\node (1) at (0,0) {$\LMO{1}$};
	\node (2) at (1,0) {$\LMO{2}$};
	\draw[->] (1) to node[above=-2pt, font=\scriptsize] {$f$} (2);
  	\node[draw, fit=(1) (2)] (box) {};
		\node[left=0 of box] {$\cat{C}\coloneqq$};
  	\node (a) at (2.8,0) {$\LMO{u}$};
  	\node[blue] (b) at (3.5,1) {$\LMO{v}$};
  	\node[blue] (c) at (4.5,1) {$\LMO{w}$};
  	\node[red] (d) at (3.5,-1) {$\LMO{x}$};
  	\node[red] (e) at (4.5,-1) {$\LMO{y}$};
  	\node (f) at (5.2,0) {$\LMO{z}$};
  	\draw[->] (a) to node[above, font=\scriptsize] {$a$} (b);
  	\draw[->] (a) to node[below, font=\scriptsize] {$b$} (d);
  	\draw[->,blue] (b) to node[below=-2pt, font=\scriptsize] {$d$} (c);
  	\draw[->,green!50!black] (b) to node[right=-2pt, font=\scriptsize] {$c$} (d);
  	\draw[->,red] (d) to node[above=-2pt, font=\scriptsize] {$e$} (e);
  	\draw[->,green!50!black] (c) to node[left=-2pt, font=\scriptsize] {$g$} (e);
  	\draw[->] (c) to node[above, font=\scriptsize] {$h$} (f);
  	\draw[->] (e) to node[below, font=\scriptsize] {$k$} (f);	
  	\node[draw, fit=(a) (b) (c) (d) (e) (f)] (box) {};
		\node[right=0 of box] {$=:\cat{D}$};
	\begin{scope}[mapsto, thick, bend left=25]
  	\draw (1) to node[below=-1pt] {$F$} (b);
  	\draw (2) to (c);
	\end{scope}
	\begin{scope}[densely dotted, thick, ->, red, bend right=25]
  	\draw (1) to node[above=-2pt] {$G$} (d);
  	\draw (2) to (e);
	\end{scope}
\end{tikzpicture}
\]
We have depicted, in blue and red respectively, two functors $F,G
\colon \cat{C} \to \cat{D}$. A natural transformation $\alpha\colon F
\Rightarrow G$ is given by choosing components $\alpha_1\colon v\to x$ and
$\alpha_2\colon w\to y$. We have highlighted the only choice for each in green;
namely, $\alpha_1=c$ and $\alpha_2=g$.

The key point is that the functors $F$ and $G$ are ways of viewing the category
$\cat{C}$ as lying inside the category $\cat{D}$. The natural transformation
$\alpha$, then, is a way of relating these two views using the morphisms in
$\cat{D}$. Does this help you to see and appreciate the notation
$
\begin{tikzcd}[column sep=large]
	\cat{C}\ar[r, bend left=18pt, "F"]\ar[r, bend right=18pt, "G"']\ar[r, phantom, "\scriptstyle\alpha\!\Downarrow\;"]&
	\cat{D}?
\end{tikzcd}
$
\end{example}%
\index{natural transformation!component of}

\begin{example}%
\label{ex.1_inst}
We said in \cref{ex.set_1} that a functor $\Cat{1}\to\smset$ can be identified with a set. So suppose $A$ and $B$ are sets considered as functors $A,B\colon\Cat{1}\to\smset$. A natural transformation between these functors is just a function between the sets.
\end{example}

\begin{definition}%
\label{def.functor_cat}%
\index{category!of functors}
Let $\cat{C}$ and $\cat{D}$ be categories. We denote by $\cat{D}^{\cat{C}}$ the category whose objects are functors $F\colon\cat{C}\to\cat{D}$ and whose morphisms $\cat{D}^{\cat{C}}(F,G)$ are the natural transformations $\alpha\colon F\to G$. This category $\cat{D}^\cat{C}$ is called the \emph{functor category}, or the \emph{category of functors from $\cat{C}$ to $\cat{D}$}.
\end{definition}

\begin{exercise}%
\label{exc.exponential_cat}
Let's look more deeply at how $\cat{D}^{\cat{C}}$ is a category.
\begin{enumerate}
	\item Figure out how to compose natural transformations. (Hint: an expert tells you ``for each object $c\in\cat{C}$, compose the $c$-components.'')
	\item Propose an identity natural transformation on any object
	$F\in\cat{D}^\cat{C}$, and check that it is unital (i.e. that it obeys
	condition (a) of \cref{def.category}).%
\index{identity!natural transformation}%
\index{natural transformation!identity}
	\qedhere
\end{enumerate}
\end{exercise}

\begin{example}
In our new language, \cref{ex.1_inst} says that $\smset^{\Cat{1}}$ is equivalent to $\smset$.
\end{example}

\begin{example}
Let $\cat{N}$ denote the category associated to the preorder $(\NN,\leq)$, and recall from \cref{ex.preorder_functor} that we can identify a functor $F\colon\cat{N}\to\cat{N}$ with a non-decreasing sequence $(F_0,F_1,F_2,\ldots)$ of natural numbers, i.e.\ $F_0\leq F_1\leq F_2\leq\cdots$. If $G$ is another functor, considered as a non-decreasing sequence, then what is a natural transformation $\alpha\colon F\to G$?

Since there is at most one morphism between two objects in a preorder, each
component $\alpha_n\colon F_n\to G_n$ has no data, it just tells us a fact: that
$F_n\leq G_n$. And the naturality condition is vacuous: every square in a
preorder commutes. So a natural transformation between $F$ and $G$ exists iff
$F_n\leq G_n$ for each $n$, and any two natural transformations $F\Rightarrow G$
are the same. In other words, the category $\cat{N}^\cat{N}$ is itself a preorder; namely the preorder of monotone maps $\nn\to\nn$.
\end{example}%
\index{natural transformation!between monotone maps}

\begin{exercise}%
\label{exc.true_false_preorder_nt}
Let $\cat{C}$ be an arbitrary category and let $\cat{P}$ be a preorder, thought of as a category. Consider the following statements:
\begin{enumerate}
	\item For any two functors $F,G\colon\cat{C}\to\cat{P}$, there is at most one natural transformation $F\to G$.
	\item For any two functors $F,G\colon\cat{P}\to\cat{C}$, there is at most one natural transformation $F\to G$.
\end{enumerate}
For each, if it is true, say why; if it is false, give a counterexample.
\end{exercise}

\begin{remark} %
\label{rem.preorder_boolcats2}%
\index{equivalence of categories}
Recall that in \cref{rem.preorder_boolcats} we said the category of preorders is
equivalent to the category of $\Bool$-categories. We can now state the precise meaning
of this sentence. First, there exists a category $\Cat{PrO}$ in which the
objects are preorders and the morphisms are monotone maps. Second, there exists a
category $\Bool\textrm{-}\Cat{Cat}$ in which the objects are $\Bool$-categories and the
morphisms are $\Bool$-functors. We call these two categories equivalent because
there exist functors $F\colon \Cat{PrO} \to \Bool\textrm{-}\Cat{Cat}$ and $G\colon
\Bool\textrm{-}\Cat{Cat} \to \Cat{PrO}$ such that there exist natural isomorphisms $F\cp G
\cong \id_{\Cat{PrO}}$ and $G\cp F \cong \id_{\Bool\textrm{-}\Cat{Cat}}$ in the sense of \cref{def.natural_transformation}.
\end{remark}

\index{natural transformation|)}

\subsection{The category of instances on a schema} %
\label{subsec.instances_cat}%
\index{database!instance}%
\index{category!of instances on a database schema}

\begin{definition}%
\label{def.instance}%
\index{database!instance homomorphism}
Suppose that $\cat{C}$ is a database schema and $I,J\colon\cat{C}\to\smset$ are database instances. An \emph{instance homomorphism} between them is a natural transformation $\alpha\colon I\to J$. Write $\cat{C}\inst\coloneqq\smset^\cat{C}$ to denote the functor category as defined in \cref{def.functor_cat}.
\end{definition}

We saw in \cref{ex.1_inst} that $\Cat{1}\inst$ is equivalent to the category
$\smset$. In this subsection, we will show that there is a schema whose instances are graphs and whose instance homomorphisms are graph homomorphisms.%
\index{graphs!homomorphism of}

\paragraph{Extended example: the category of graphs as a functor category.}
\index{graph!as $\smset$-valued functor} %
\index{primordial ooze}%
\index{graphs!database schema for|(}

You may find yourself back in the primordial ooze (first discussed in
\cref{subsec.preorders_Bool_enriched}), because while previously we have been
using graphs to present categories, now we obtain graphs themselves as database
instances on a specific schema (which is itself a graph):%
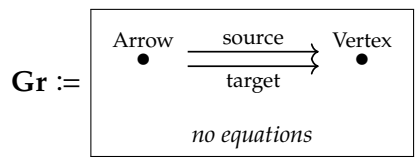
\index{category!of graphs}
\[
\Cat{Gr}\coloneqq\boxCD{
\begin{tikzcd}[ampersand replacement=\&, column sep=50pt]
	\LTO{Arrow}\ar[r, shift left, "\text{source}"]\ar[r, shift right, "\text{target}"']\&\LTO{Vertex}
\end{tikzcd}
  \\~\\\footnotesize
  \textit{no equations}
}
\]
Here's an example $\Cat{Gr}$-instance, i.e.\ set-valued functor $I\colon\Cat{Gr}\to\smset$, in table form:%
\index{functor!set@$\smset$-valued}
\begin{equation}%
\label{eqn.sample_Gr_instance}
\begin{tabular}{ c | c c}
  \textbf{Arrow}&\textbf{source}&\textbf{target}\\\hline
	$a$&1&2\\
	$b$&1&3\\
	$c$&1&3\\
	$d$&2&2\\
	$e$&2&3
\end{tabular}
\hspace{1in}
\begin{tabular}{ c |}
	\textbf{Vertex}\\\hline
	1\\
	2\\
	3\\
	4\\
	~
\end{tabular}
\end{equation}
Here $I(\mathrm{Arrow})=\{a,b,c,d,e\}$, and $I(\mathrm{Vertex})=\{1,2,3,4\}$. One can draw the instance $I$ as a graph:
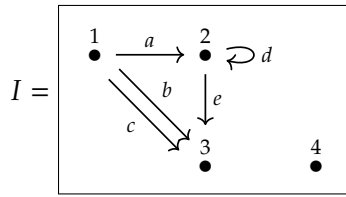
\[
I=\fbox{
\begin{tikzcd}[ampersand replacement=\&]
	\LMO{1}\ar[r, "a"]\ar[dr, shift left, "b"]\ar[dr, shift right, "c"']\&
	\LMO{2}\ar[d, "e"]\ar[loop right, "d"]\\
	\&\LMO{3}\&\LMO{4}
\end{tikzcd}
}
\]
Every row in the Vertex table is drawn as a vertex, and every row in the Arrow table is drawn as an arrow, connecting its specified source and target. Every possible graph can be written as a database instance on the schema $\Cat{Gr}$, and every possible $\Cat{Gr}$-instance can be represented as a graph.

\begin{exercise}%
\label{exc.graph_instance}
In \cref{eqn.free_schema}, a graph is shown (forget the distinction between white and black nodes). Write down the corresponding $\Cat{Gr}$-instance, as in \cref{eqn.sample_Gr_instance}. (Do not be concerned that you are in the primordial ooze.)
\end{exercise}

Thus the objects in the category $\Cat{Gr}\inst$ are graphs. The morphisms
in $\Cat{Gr}\inst$ are called \emph{graph homomorphisms}.\index{graph!homomorphism} Let's unwind this. Suppose that $G,H\colon\Cat{Gr}\to\smset$ are functors (i.e.\ $\Cat{Gr}$-instances); that is, they are objects $G,H\in\Cat{Gr}\inst$. A morphism $G\to H$ is a natural transformation $\alpha\colon G\to H$ between them; what does that entail?%
\index{natural transformation!graph homomorphism as}

By \cref{def.natural_transformation}, since $\Cat{Gr}$ has two objects, $\alpha$ consists of two components,
\[
  \alpha_{\Set{Vertex}}\colon G(\Set{Vertex})\to H(\Set{Vertex})
  \qquad\text{ and }\qquad
  \alpha_{\Set{Arrow}}\colon G(\Set{Arrow})\to H(\Set{Arrow}),
\]
both of which are morphisms in $\smset$. In other words, $\alpha$ consists of a function from vertices of $G$ to vertices of $H$ and a function from arrows of $G$ to arrows of $H$. For these functions to constitute a graph homomorphism, they must ``respect source and target'' in the precise sense that the naturality condition, \cref{eqn.naturality_condition} holds. That is, for every morphism in $\Cat{Gr}$, namely $\text{source}$ and $\text{target}$, the following diagrams must commute:
\[
\begin{tikzcd}[column sep=large]
	G(\text{Arrow})\ar[r, "\alpha_{\Set{Arrow}}"]\ar[d, "G(\text{source})"']&
	H(\text{Arrow})\ar[d, "H(\text{source})"]\\
	G(\text{Vertex})\ar[r, "\alpha_{\Set{Vertex}}"']&
	H(\text{Vertex})
\end{tikzcd}
\hspace{.8in}
\begin{tikzcd}[column sep=large]
	G(\text{Arrow})\ar[r, "\alpha_{\Set{Arrow}}"]\ar[d, "G(\text{target})"']&
	H(\text{Arrow})\ar[d, "H(\text{target})"]\\
	G(\text{Vertex})\ar[r, "\alpha_{\Set{Vertex}}"']&
	H(\text{Vertex})
\end{tikzcd}
\]
These may look complicated, but they say exactly what we want. We want the functions $\alpha_{\Set{Vertex}}$ and $\alpha_{\Set{Arrow}}$ to respect source and targets in $G$ and $H$. The left diagram says ``start with an arrow in $G$. You can either apply $\alpha$ to the arrow and then take its source in $H$, or you can take its source in $G$ and then apply $\alpha$ to that vertex; either way you get the same answer.'' The right-hand diagram says the same thing about targets.

\begin{example}%
\label{ex.two_graphs_as_instances}
Consider the graphs $G$ and $H$ shown below
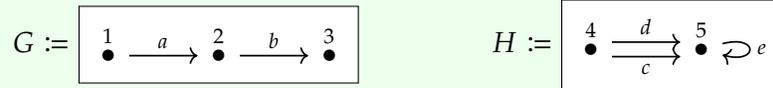
\[
G\coloneqq\boxCD{
\begin{tikzcd}[ampersand replacement=\&]
	\LMO{1}\ar[r, "a"]\& \LMO{2}\ar[r, "b"]\& \LMO{3}
\end{tikzcd}
}
\hspace{.7in}
H\coloneqq\boxCD{
\begin{tikzcd}[ampersand replacement=\&]
	\LMO{4}\ar[r, shift right, "c"']\ar[r, shift left, "d"]\& \LMO{5}\ar[loop right, "e"]
\end{tikzcd}
}
\]

Here they are, written as database instances---i.e.\ set-valued functors---on $\Cat{Gr}$:%
\index{functor!set@$\smset$-valued}
\[
\begin{array}{c c c c c}
	G\coloneqq&&
	\begin{array}{c | c c}
		\textbf{Arrow}&\textbf{source}&\textbf{target}\\
		a&1&2\\
		b&2&3\\
		~
	\end{array}
	&&
	\begin{array}{c |}
		\textbf{Vertex}\\
		1\\
		2\\
		3\\
	\end{array}
	\\
	H\coloneqq&
	\begin{array}{c | c c}
		\color{gray}{\textbf{Arrow}}&\color{gray}{\textbf{source}}&\color{gray}{\textbf{target}}\\
		c&4&5\\
		d&4&5\\
		e&5&5
	\end{array}
	&&
	\begin{array}{c |}
		\color{gray}{\textbf{Vertex}}\\
		4\\
		5\\
		~
	\end{array}
\end{array}
\]
The top row is $G$ and the bottom row is $H$. They are offset so you can more easily complete the following exercise.
\end{example}

\begin{exercise}%
\index{graphs !homomorphism of}%
\label{exc.unique_alpha}
We claim that---with $G,H$ as in \cref{ex.two_graphs_as_instances}---there is exactly one graph homomorphism $\alpha\colon G\to H$ such that $\alpha_{\Set{Arrow}}(a)=d$.
\begin{enumerate}
	\item What is the other value of $\alpha_{\Set{Arrow}}$, and what are the three values of $\alpha_{\Set{Vertex}}$?
	\item In your own copy of the tables of \cref{ex.two_graphs_as_instances}, draw $\alpha_{\Set{Arrow}}$ as two lines connecting the cells in the ID column of $G(\Set{Arrow})$ to those in the ID column of $H(\Set{Arrow})$. Similarly, draw $\alpha_{\Set{Vertex}}$ as connecting lines.
	\item Check the source column and target column and make sure that the matches are natural, i.e.\ that ``alpha-then-source equals source-then-alpha'' and similarly for ``target.''
\qedhere
\end{enumerate}
\end{exercise}

\section{Adjunctions and data migration}%
\label{sec.adjunctions_mig}%
\index{adjunction|(}%
\index{data migration}%
\index{functor!set@$\smset$-valued}

We have talked about how set-valued functors on a schema can be understood as filling that schema with data. But there are also functors between schemas. When the two sorts of functors are composed, data is migrated. This is the simplest form of data migration; more complex ways to migrate data come from using adjoints. All of the above is the subject of this section.

\subsection{Pulling back data along a functor}%
\label{subsec.pullback_data}
\index{data migration!pullback}

To begin, we will migrate data between the graph-indexing schema $\Cat{Gr}$ and the loop schema, which we call $\Cat{DDS}$, shown below
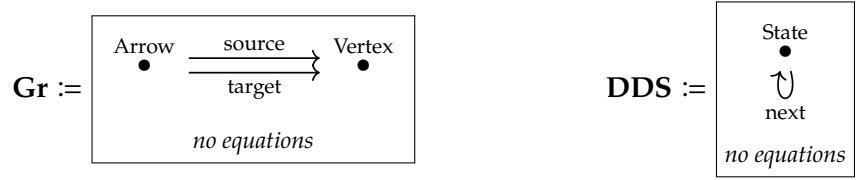
\[
\Cat{Gr}\coloneqq\boxCD{
\begin{tikzcd}[ampersand replacement=\&, column sep=50pt]
	\LTO{Arrow}\ar[r, shift left, "\text{source}"]\ar[r, shift right, "\text{target}"']\&\LTO{Vertex}
\end{tikzcd}
  \\~\\\footnotesize
  \textit{no equations}
}
\hspace{1in}
\Cat{DDS}\coloneqq\boxCD{
\begin{tikzcd}[ampersand replacement=\&]
	\LTO{State}\ar[loop below, "\text{next}"]
\end{tikzcd}
\\~\\\footnotesize
\textit{no equations}
}
\]
We begin by writing down a sample instance $I\colon\Cat{DDS}\to\smset$ on this schema:
\begin{equation}%
\label{eqn.state_table}
\begin{array}{c | c}
	\textbf{State}&\textbf{next}\\\hline
	1 & 4\\
	2 & 4\\
	3 & 5\\
	4 & 5\\
	5 & 5\\
	6 & 7\\
	7 & 6
\end{array}
\end{equation}
We call the schema $\Cat{DDS}$ to stand for discrete dynamical system.%
\index{discrete dynamical system} Indeed, we may think of the data in the
$\Cat{DDS}$-instance of \cref{eqn.state_table} as listing the states and
movements of a deterministic machine: at every point in time the machine is in
one of the listed states, and given the machine in one of the states, in the
next instant it moves to a uniquely determined next state.

Our goal is to migrate the data in \cref{eqn.state_table} to data on $\Cat{Gr}$;
this will give us the data of a graph and so allow us to visualise our machine.

We will use a functor connecting these schemas in order to move data between them. The reader can create any functor she likes, but we will use a specific functor $F\colon\Cat{Gr}\to\Cat{DDS}$ to migrate data in a way that makes sense to us, the authors. Here we draw $F$, using colors to hopefully aid understanding:
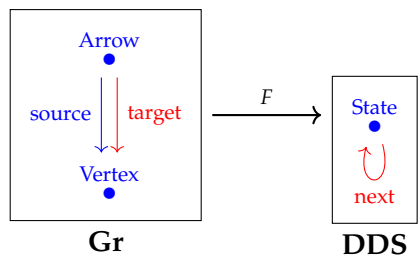
\[
\begin{tikzpicture}[color=blue]
	\node (Arrow) {$\LTO{Arrow}$};
	\node[below=of Arrow] (Vertex) {$\LTO{Vertex}$};
	\draw[->] 
		($(Arrow.south)+(-3pt,0)$) to node[left, font=\scriptsize] (source) {source} 
		($(Vertex.north)+(-3pt,0)$);
	\draw[->, color=red]
		($(Arrow.south)+(3pt,0)$) to node[right, font=\scriptsize] (target) {target} 
		($(Vertex.north)+(3pt,0)$);
	\node[draw, color=black, fit=(Arrow) (Vertex) (source) (target)] (Gr) {};
	\node[below=0 of Gr, black] {$\Cat{Gr}$};
	\node[right=2 of target] (State) {$\LTO{State}$} edge [loop below, red] node[font=\scriptsize] (next) {next} ();
	\node[draw, color=black, fit=(State) (next)] (DDS) {};
	\node[below=0 of DDS, black] {$\Cat{DDS}$};
	\draw[functor, black] (Gr.east|-target) to node[above, font=\scriptsize] {$F$} (DDS.west|-State);
\end{tikzpicture}
\]
The functor $F$ sends both objects of $\Cat{Gr}$ to the `State' object of
$\Cat{DDS}$ (as it must). On morphisms, it sends the `source' morphism to the identity morphism on `State', and the `target' morphism to the morphism `next'.

A sample database instance on $\Cat{DDS}$ was given in \cref{eqn.state_table};
recall this is a functor $I\colon\Cat{DDS}\to\smset$. So now we have two functors as follows:
\begin{equation}%
\label{eqn.Gr_dds_pb}
\begin{tikzcd}
	\Cat{Gr}\ar[r, "F"]&\Cat{DDS}\ar[r, "I"]&\smset.
\end{tikzcd}
\end{equation}
Objects in $\Cat{Gr}$ are sent by $F$ to objects in $\Cat{DDS}$, which are sent by $I$ to objects in $\smset$, which are sets. Morphisms in $\Cat{Gr}$ are sent by $F$ to morphisms in $\Cat{DDS}$, which are sent by $I$ to morphisms in $\smset$, which are functions. This defines a composite functor $F\cp I\colon\Cat{Gr}\to\smset$. Both $F$ and $I$ respect identities and composition, so $F\cp I$ does too. Thus we have obtained an instance on $\Cat{Gr}$, i.e.\ we have converted our discrete dynamical system from \cref{eqn.state_table} into a graph! What graph is it? %
\index{dynamical system!discrete}

For an instance on $\Cat{Gr}$, we need to fill an Arrow table and a Vertex table. Both of these are sent by $F$ to State, so let's fill both with the rows of State in \cref{eqn.state_table}. Similarly, since $F$ sends `source' to the identity and sends `target' to `next', we obtain the following tables:
\[
\begin{tabular}{ c | c c}
  \textbf{Arrow}&\textbf{source}&\textbf{target}\\\hline
	1&1&4\\
	2&2&4\\
	3&3&5\\
	4&4&5\\
	5&5&5\\
	6&6&7\\
	7&7&6
\end{tabular}
\hspace{1in}
\begin{tabular}{ c |}
	\textbf{Vertex}\\\hline
	1\\
	2\\
	3\\
	4\\
	5\\
	6\\
	7
\end{tabular}
\]
Now that we have a graph, we can draw it.
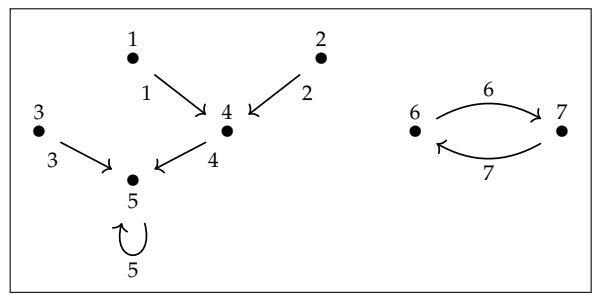
\[
\boxCD{
\begin{tikzcd}[column sep=20pt, row sep=5, pos=.1, ampersand replacement=\&]
	\&\LMO{1}\ar[dr, "1"']\&\&\LMO{2}\ar[dl,"2"]\\
	\LMO{3}\ar[dr, "3"']\&\&\LMO{4}\ar[dl, "4"]\&\&\LMO{6}\ar[r, bend left, pos=.5, "6"]\&[20pt]\LMO{7}\ar[l, bend left, pos=.5, "7"]\\
	\&\LMO[under]{5}\ar[loop below, looseness=5, pos=.5, "5"]
\end{tikzcd}
}
\]
Each arrow is labeled by its source vertex, as if to say, ``What I do next is determined by what I am now.''

\begin{exercise}%
\label{exc.dds_graph}
Consider the functor $G\colon\Cat{Gr}\to\Cat{DDS}$ given by sending `source' to `next' and sending `target' to the identity on `State'. Migrate the same data, called $I$ in \cref{eqn.state_table}, using the functor $G$. Write down the tables and draw the corresponding graph.
\end{exercise}

We refer to the above procedure---basically just composing functors as in \cref{eqn.Gr_dds_pb}---as ``pulling back data along a functor.'' We just now pulled
back data $I$ along functor $F$.%
\index{pullback!along a map}

\begin{definition}%
\label{def.pullback_along}
Let $\cat{C}$ and $\cat{D}$ be categories and let $F\colon\cat{C}\to\cat{D}$ be
a functor. For any set-valued functor $I\colon\cat{D}\to\smset$, we refer to the
composite functor $F\cp I\colon\cat{C}\to\smset$ as the \emph{pullback of $I$ along $F$}.%
\index{functor!set@$\smset$-valued}%
\index{data migration!pullback}%
\index{functor!data migration|see {data migration}}

Given a natural transformation $\alpha\colon I\Rightarrow J$, there is a natural
transformation $\alpha_F\colon F\cp I\Rightarrow F\cp J$, whose component $(F\cp I)(c)\to (F\cp J)(c)$ for any $c\in\Ob(\cat{C})$ is given by $(\alpha_F)_c\coloneqq\alpha_{Fc}$.
\[
\begin{tikzcd}[column sep=large]
	\cat{C}
	\ar[r, "F"]
	&\cat{D} 
	\ar[r, bend left, "I"{name=up}]
	\ar[r, bend right, "J"'{name=down}]
	&\smset
        \ar[from=up, to=down, Rightarrow, shorten <=5pt,
	shorten >=5pt,"\scriptstyle\alpha"]
\end{tikzcd}
\hspace{1cm}
\leadsto
\hspace{1cm}
\begin{tikzcd}[column sep=large]
	\cat{C}
	\ar[r, bend left, "F\cp I"{name=up, pos=.55}]
	\ar[r, bend right, "F\cp J"'{name=down, pos=.55}]
	&\smset
        \ar[from=up, to=down, Rightarrow, shorten <=5pt,
	shorten >=5pt,"\scriptstyle\alpha_F"]
\end{tikzcd}
\]

This uses the data of $F$ to define a functor
$\Delta_F\colon\cat{D}\inst\to\cat{C}\inst.$
\end{definition}

Note that the term pullback is also used for a certain sort of limit, for more
details see \cref{rem.pullbackpullback}.

\index{graphs!database schema for|)}

\subsection{Adjunctions}%
\label{subsec.adjoints_lims_colims}

In \cref{sec.galois_connections} we discussed Galois connections, which are
adjunctions between preorders.  Now that we've defined categories and functors, we
can discuss adjunctions in general. The relevance to databases is that the
data migration functor $\Delta$ from \cref{def.pullback_along} always has two
adjoints of its own: a left adjoint which we denote $\Sigma$ and a right adjoint
which we denote $\Pi$. %
\index{data migration!adjoints}%
\index{adjunction!examples of}

Recall that an adjunction between preorders $P$ and $Q$ is a pair of monotone maps
$f\colon P \to Q$ and $g\colon Q \to P$ that are \emph{almost} inverses: we have 
\begin{equation}%
\label{eqn.preorder_adjunction_revisit}
f(p) \le q \mbox{ if and only if } p \le g(q).
\end{equation}
Recall from \cref{subsubsec.pos_free_spectrum} that in a preorder $P$, a hom-set
$P(a,b)$ has one element when $a \le b$, and no elements otherwise. We can thus rephrase \cref{eqn.preorder_adjunction_revisit} as an isomorphism of sets $Q(f(p),q) \cong P(p,g(q))$: either both are one-element sets or both are 0-element sets. 
This suggests how to define adjunctions in the general case.

\begin{definition}%
\label{def.adjoints}%
\index{adjunction}%
\index{adjoint|see {adjunction}}
Let $\cat C$ and $\cat D$ be categories, and $L\colon \cat C \to \cat D$ and $R
\colon \cat D \to \cat C$ be
functors. We say that \emph{$L$ is left adjoint to $R$} (and that \emph{$R$ is
right adjoint to $L$}) if, for any $c\in\cat{C}$ and $d\in\cat{D}$, there is an isomorphism of hom-sets
\[\alpha_{c,d}\colon\cat{C}(c,R(d))\To{\cong}\cat{D}(L(c),d)\]
that is natural in $c$ and $d$.%
\tablefootnote{This naturality is between functors
  $\cat{C}\op\times\cat{D}\to\smset$. It says that for any morphisms $f\colon
  c'\to c$ in $\cat{C}$ and $g\colon d\to d'$ in $\cat{D}$, the following
  diagram commutes:
\[
\begin{tikzcd}[sep=large, ampersand replacement=\&]
	\cat{C}(c,Rd) \ar[d, "{\cat C(f,Rg)}"'] \ar[r, "\alpha_{c,d}"] \&
	\cat{D}(Lc,d) \ar[d, "{\cat D(Lf,g)}"] \\
	\cat{C}(c',Rd') \ar[r, "\alpha_{c',d'}"'] \&
	\cat{D}(Lc',d')
\end{tikzcd}
\]
}

Given a morphism $f\colon c\to R(d)$ in $\cat{C}$, its image $g\coloneqq\alpha_{c,d}(f)$ is called its \emph{mate}. Similarly, the mate of $g\colon L(c)\to d$ is $f$.

To denote an adjunction we write $L \dashv R$, or in diagrams,
\[
\adj{\cat{C}}{L}{R}{\cat{D}}
\] 
with the $\Rightarrow$ in the direction of the left adjoint. 
\end{definition}

\begin{example}
Recall that every preorder $\cat{P}$ can be regarded as a category. Galois connections between preorders and adjunctions between the corresponding categories are exactly the same thing.
\end{example}

\begin{example}%
\label{ex.currying}%
\index{currying}
Let $B\in\Ob(\smset)$ be any set. There is an adjunction called `currying
$B$,' after the logician Haskell Curry:
\[
\adj{\smset}{-\times B}{(-)^B}{\smset}
\hspace{1in}
\smset(A\times B,C)\cong\smset(A,C^B)
\]
Abstractly we write it as on the left, but what this means is that for any sets
$A,C$, there is a natural isomorphism as on the right.

To explain this, we need to talk about exponential objects in $\smset$. Suppose
that $B$ and $C$ are sets. Then the set of functions $B\to C$ is also a set;
let's denote it $C^B$. It's written this way because if $C$ has $10$ elements and $B$ has $3$ elements then $C^B$ has $10^3$ elements, and more generally for any two finite sets $|C^B|=|C|^{|B|}$.

The idea of currying is that given sets $A$, $B$, and $C$, there is a one-to-one
correspondence between functions $(A\times B)\to C$ and functions $A\to C^B$.
Intuitively, if I have a function $f$ of two variables $a,b$, I can ``put off''
entering the second variable: if you give me just $a$, I'll return a function
$B\to C$ that's waiting for the $B$ input. This is the curried version of $f$. As one might guess, there is a formal connection between exponential objects and what we called hom-elements $b\multimap c$ in \cref{def.monoidal_closed}.
\end{example}

\begin{exercise}%
\label{exc.currying_practice}
In \cref{ex.currying}, we discussed an adjunction between functors $-\times B$
and $(-)^B$. But we only said how these functors worked on objects: for an arbitrary set
$X$, they return sets $X\times B$ and $X^B$ respectively.
\begin{enumerate}
	\item Given a morphism $f\colon X\to Y$, what morphism should $-\times B\colon X\times B\to Y\times B$ return?
	\item Given a morphism $f\colon X\to Y$, what morphism should $(-)^B\colon X^B\to Y^B$ return?
	\item Consider the function $+\colon\NN\times\NN\to\NN$, which sends
	$(a,b)\mapsto a+b$. Currying $+$, we get a certain function
	$p\colon\NN\to\NN^\NN$. What is $p(3)$?
	\qedhere
\end{enumerate}
\end{exercise}

\begin{example}%
\label{ex.adjunctions}%
\index{adjunction!examples of}
If you know some abstract algebra or topology, here are some other examples of adjunctions.
\begin{enumerate}
	\item Free constructions: given any set you get a free group, free monoid, free ring, free vector space, etc.; each of these is a left adjoint. The corresponding right adjoint takes a group, a monoid, a ring, a vector space etc.\ and forgets the algebraic structure to return the underlying set. %
\index{group!free}
\index{monoid!free}
\index{ring!free}
\index{vector space!free}
	
	\item Similarly, given a graph you get a free preorder or a free
	category, as we discussed in \cref{subsubsec.pos_free_spectrum}; each is
	a left adjoint. The corresponding right adjoint is the underlying graph
	of a preorder or of a category.
	\index{preorder!free}
	\index{category!free}

	\item Discrete things: given any set you get a discrete preorder,
	discrete graph, discrete metric space, discrete category, discrete
	topological space; each of
	these is a left adjoint. The corresponding right adjoint is again
	underlying set. 
	\index{preorder!discrete}
	\index{graph!discrete}
	\index{metric space!discrete}
	\index{category!discrete}
	\index{topology!discrete}

	\item Codiscrete things: given any set you get a codiscrete preorder,
	complete graph, codiscrete category, codiscrete topological space; each
	of these is a right adjoint. The corresponding left adjoint is the
	underlying set.
	\index{preorder!codiscrete}
	\index{graph!complete}
	\index{category!codiscrete}
	\index{topology!codiscrete}

	\item Given a group, you can quotient by its commutator subgroup to get an abelian group; this is a left adjoint. The right adjoint is the inclusion of abelian groups into groups.
	\index{group!commutator subgroup}
	\qedhere%
\index{quotient}
\end{enumerate}
\end{example}

\subsection{Left and right pushforward functors, $\Sigma$ and $\Pi$}%
\index{data migration|(}
Given $F\colon\cat{C}\to\cat{D}$, the data migration functor $\Delta_F$ turns
$\cat{D}$-instances into $\cat{C}$-instances. This functor has both a left and a
right adjoint:%
\index{data migration!left pushforward}%
\index{data migration!right pushforward}
\[
\begin{tikzcd}[column sep=70pt]
	\cat{C}\inst\ar[r, shift left=8pt, bend left=8pt, "\Sigma_F"]\ar[r, shift right=8pt, bend right=8pt, "\Pi_F"']
	\ar[r, shift left=7.5pt, Rightarrow, shorten <=22pt,
	shorten >=22pt]
	&
	\cat{D}\inst\ar[l, "\Delta_F" description]
	\ar[l, shift left=7.5pt, Rightarrow, shorten <=22pt,
	shorten >=22pt]
\end{tikzcd}
\]
Using the names $\Sigma$ and $\Pi$ in this context is fairly standard in category theory. In the case of
databases, they have the following helpful mnemonic:
\[
\begin{tabular}{c | p{.8in} p{1.1in} p{1.5in}}
	\textbf{Migration Functor}&\textbf{Pronounced}&\textbf{Reminiscent of}&\textbf{Database idea}\\\hline
	$\Delta$ & Delta & Duplicate\qquad~ or destroy&Duplicate or destroy tables or columns\\
	$\Sigma$ & Sigma & Sum & Union (sum up) data\\\\
	$\Pi$ & Pi & Product & Pair%
\footnote{This is more commonly called ``join''
	by database programmers.} and query data 
\end{tabular}%
\index{union!and data migration}
\]

Just like we used $\Delta_F$ to pull back any discrete dynamical system along $F\colon\Cat{Gr}\to\Cat{DDS}$ and get a graph,
the migration functors $\Sigma_F$ and $\Pi_F$ can be used to turn any graph into a
discrete dynamical system. That is, given an instance
$J\colon\Cat{Gr}\to\smset$, we can get instances $\Sigma_F(J)$ and $\Pi_F(J)$ on $\Cat{DDS}$. This, however, is quite
technical, and we leave it to the adventurous reader to compute an example, with
help perhaps from \cite{Spivak:2014a}, which explores the definitions of
$\Sigma$ and $\Pi$ in detail. A less technical shortcut is simply to code up the
computation in the open-source FQL software.%
\index{discrete dynamical system!induced graph of}%
\index{dynamical system!discrete|see {discrete dynamical system}}

To get the basic idea across without getting mired in technical details, here we
shall instead discuss a very simple example. Recall the schemas from
\cref{eqn.airline_schemas}. We can set up a functor between them, the one sending black dots to black dots and white dots to white dots:
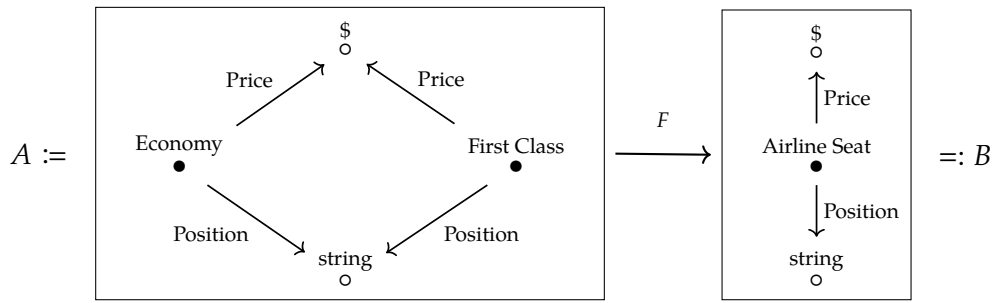
\[
\begin{tikzpicture}[commutative diagrams/every diagram, inner sep=10pt]
	\matrix[matrix of math nodes, name=A, row sep=20pt, commutative diagrams/every cell] {
		&
			|(AD)|\LTO[\circ]{\$}
		\\
			|(AE)|\LTO{Economy}
		&&
			|(AF)|\LTO{First Class}
		\\
		&
			|(AS)|\LTO[\circ]{string}
		\\
	};
	\path[commutative diagrams/.cd, every arrow, every label, font=\scriptsize]
		(AE) edge["Price"]     (AD)
				 edge["Position"'] (AS)
		(AF) edge["Price"']    (AD)
				 edge["Position"]  (AS);
	\node[draw, inner ysep=0pt, fit=(AE) (AD) (AF) (AS)] (A box) {};
	\node[left=0 of A box] {$A\coloneqq$};
	\matrix[matrix of math nodes, name=B, row sep=20pt, commutative diagrams/every cell, matrix anchor=south, right=3 of A.south east] {
			|(BD)|\LTO[\circ]{\$}
		\\
			|(BAS)|\LTO{Airline Seat}
		\\
			|(BS)|\LTO[\circ]{string}
		\\
	};
	\path[commutative diagrams/.cd, every arrow, every label, font=\scriptsize]
		(BAS) edge["Price"']    (BD)
				 edge["Position"] (BS);
	\node[draw, inner ysep=0pt, fit=(BAS) (BD) (BS)] (B box) {};				
	\node[right=0 of B box] {$=:B$};
	\draw[functor] (A box.east) to node[above, font=\scriptsize] {$F$} (B box.west);
\end{tikzpicture}
\]
With this functor $F$ in hand, we can transform any $B$-instance into an
$A$-instance using $\Delta_F$. Whereas $\Delta$ was interesting in the case of
turning discrete dynamical systems into graphs in \cref{subsec.pullback_data},
it is not very interesting in this case. Indeed, it will just copy---$\Delta$
for duplicate---the rows in Airline seat into both Economy and First Class. 

$\Delta_F$ has two adjoints, $\Sigma_F$ and $\Pi_F$, both of which transform
any $A$-instance $I$ into a $B$-instance. The functor $\Sigma_F$ does what one
would most expect from reading the names on each object: it will put into Airline Seat the union of Economy and First Class:
\[\Sigma_F(I)(\mathrm{Airline\ Seat})=I(\mathrm{Economy})\sqcup I(\mathrm{First\ Class}).\]

The functor $\Pi_F$ puts into Airline Seat the set of those pairs $(e,f)$ where
$e$ is an Economy seat, $f$ is a First Class seat, and $e$ and $f$ have the same
price and position. In this particular example, one imagines that there should
be no such seats in a valid instance $I$, in which case
$\Pi_F(I)(\mathrm{Airline\ Seat})$ would be empty. But in other uses of these
same schemas, $\Pi_F$ can be a useful operation. For example, in the schema $A$
replace the label `Economy' by `Rewards Program', and in $B$ replace `Airline
Seat' by `First Class Seats'. Then the operation $\Pi_F$ finds those first class
seats that are also rewards program seats. This operation is a kind of database
query; querying is the operation that databases are built
for.%
\index{database!query}

The moral is that complex data migrations can be specified by constructing
functors $F$ between schemas and using the ``induced'' functors $\Delta_F$, $\Sigma_F$, and $\Pi_F$.
Indeed, in practice essentially all useful migrations can be built up from these. Hence the language of categories provides a framework for specifying
and reasoning about data migrations.

\subsection{Single set summaries of databases}%
\index{summaries!limits and colimits as}

To give a stronger idea of the flavor of $\Sigma$ and $\Pi$, we consider another
special case, namely where the target category $\cat{D}$ is equal to $\Cat{1}$; see \cref{exc.Cat_n}.
In this case, there is exactly one functor $\cat{C}\to\Cat{1}$ for any
$\cat{C}$; let's denote it
\begin{equation}%
\label{eqn.terminal_functor}
	!\colon\cat{C}\to\Cat{1}.
\end{equation}

\begin{exercise} %
\label{ex.terminal_cat}
Describe the functor $!\colon \cat{C} \to \Cat{1}$ from \cref{eqn.terminal_functor}. Where does it send each
object? What about each morphism?
\end{exercise}

We want to consider the data migration functors
$\Sigma_!\colon\cat{C}\inst\to\Cat{1}\inst$ and
$\Pi_!\colon\cat{C}\inst\to\Cat{1}\inst$. In \cref{ex.1_inst}, we saw that an
instance on $\Cat{1}$ is the same thing as a set. So let's identify
$\Cat{1}\inst$ with $\smset$, and hence discuss
\[
	\Sigma_!\colon\cat{C}\inst\to\smset
	\qquad\text{and}\qquad
	\Pi_!\colon\cat{C}\inst\to\smset.
\]
Given any schema $\cat{C}$ and instance $I\colon\cat{C}\to\smset$, we will get
sets $\Sigma_!(I)$ and $\Pi_!(I)$. Thinking of these sets as database instances,
each corresponds to a single one-column table---a controlled
vocabulary---summarizing an entire database instance on the schema $\cat{C}$.

Consider the following schema
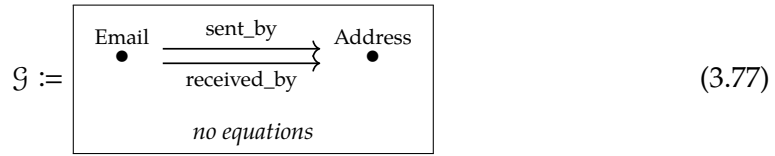
\begin{equation}%
\label{eqn.graph_email}
\cat{G}\coloneqq\boxCD{
\begin{tikzcd}[ampersand replacement=\&, column sep=60pt]
	\LTO{Email}\ar[r, shift left, "\text{sent\_by}"]\ar[r, shift right, "\text{received\_by}"']\&\LTO{Address}
\end{tikzcd}
  \\~\\\footnotesize
  \textit{no equations}
}
\end{equation}
Here's a sample instance $I\colon\cat{G}\to\smset$:
\[
\begin{tabular}{ c | c c}
  \textbf{Email}&\textbf{sent\_by}&\textbf{received\_by}\\\hline
	Em\_1 & Bob & Grace\\
	Em\_2 & Grace & Pat\\
	Em\_3 & Bob & Emmy\\
	Em\_4 & Sue & Doug\\
	Em\_5 & Doug & Sue\\
	Em\_6 & Bob & Bob
\end{tabular}
\hspace{1in}
\begin{tabular}{ c |}
	\textbf{Address}\\\hline
	Bob\\
	Doug\\
	Emmy\\
	Grace\\
	Pat\\
	Sue
\end{tabular}
\] 

\begin{exercise}%
\label{exc.draw_graph}
Note that $\cat{G}$ from \cref{eqn.graph_email} is isomorphic to the schema $\Cat{Gr}$. In
\cref{subsec.instances_cat} we saw that instances on $\Cat{Gr}$ are graphs. Draw
the above instance $I$ as a graph.
\end{exercise}


Now we have a unique functor $!\colon\cat{G}\to\Cat{1}$, and we want to say what $\Sigma_!(I)$ and $\Pi_!(I)$ give us as single-set summaries. First, $\Sigma_!(I)$ tells us all the emailing groups---the ``connected components''---in $I$:
\[
\begin{tabular}{ c |}
	\textbf{1}\\\hline
	Bob-Grace-Pat-Emmy\\
	Sue-Doug
\end{tabular}
\]
This form of summary, involving identifying entries into common groups, or
quotients, is typical of $\Sigma$-operations. \index{quotient!as data migration}

The functor $\Pi_!(I)$ lists the emails from $I$ which were sent from a person to
her- or him-self.
\[
\begin{tabular}{ c |}
	\textbf{1}\\\hline
	Em\_6
\end{tabular}
\]
This is again a sort of query, selecting the entries that fit the criterion of
self-to-self emails. Again, this is typical of $\Pi$-operations.

Where do these facts---that $\Pi_!$ and $\Sigma_!$ act the way we said---come
from? Everything follows from the definition of adjoint functors
(\ref{def.adjoints}): indeed we hope this, together with the examples given in \cref{ex.adjunctions}, give the reader some idea of how general and useful adjunctions are, both in mathematics and in database theory.

One more point: while we will not spell out the details, we note that these
operations are also examples of constructions known as colimits and limits in
$\smset$. We end this chapter with bonus material, exploring these key category theoretic
constructions. The reader should keep in mind that, in general and not just for
functors to $\Cat{1}$, $\Sigma$-operations are built from colimits in $\smset$,
and $\Pi$-operations are built from limits in $\smset$.

\index{adjunction|)}

\section{Bonus: An introduction to limits and colimits}%
\label{sec.bonus_lims_colims}%
\index{limit|(}

What do products of sets, the results of $\Pi_!$-operations on database instances, and meets in a preorder all have in
common? The answer, as we shall see, is that they are all examples of limits. Similarly, disjoint unions of sets, the results of $\Sigma_!$-operations on database instances, and joins in a preorder are all colimits. Let's begin with limits.%
\index{union!disjoint}

Recall that $\Pi_!$ takes a database instance $I\colon \cat{C} \to \smset$ and turns it
into a set $\Pi_!(I)$. More generally, a limit turns a functor $F\colon \cat{C} \to
\cat{D}$ into an object of $\cat{D}$. 

\subsection{Terminal objects and products} %
\label{subsec.terminals_products}

Terminal objects and products are each a sort of limit. Let's discuss them in turn.

\paragraph{Terminal objects.}%
\index{limit!terminal object as}%
\index{terminal object!as limit}
The most basic limit is a terminal object.
\begin{definition}%
\index{terminal object}
Let $\cat{C}$ be a category. Then an object $Z$ in $\cat{C}$ is a \emph{terminal
object} if, for each object $C$ of $\cat{C}$, there exists a unique morphism
$!\colon C \to Z$.
\end{definition}

Since this \emph{unique} morphism exists \emph{for all} objects in $\cat{C}$, we
say that terminal objects have a \emph{universal property}.%
\index{universal property}%
\index{terminal object!universal property of}

\begin{example}
In $\smset$, any set with exactly one element is a terminal object. Why?
Consider some such set $\{\bullet\}$. Then for any other set $C$ we need to
check that there is exactly one function $!\colon C \to \{\bullet\}$. This
unique function is the one that does the only thing that can be done: it maps
each element $c \in C$ to the element $\bullet \in \{\bullet\}$.
\end{example}

\begin{exercise}%
\label{exc.terminal_in_preorder}
Let $(P,\le)$ be a preorder, let $z\in P$ be an element, and let $\cat{P}$ be the corresponding category (see \cref{subsubsec.pos_free_spectrum}). Show that $z$ is a terminal object in $\cat{P}$ if and
only if it is a \emph{top element} in $P$: that is, if and only if for all $c \in P$ we have $c \le
z$.
\end{exercise}%
\index{top element|see {terminal object}}

\begin{exercise}%
\label{exc.terminal_cat}
Name a terminal object in the category $\Cat{Cat}$. (Hint: recall
\cref{ex.terminal_cat}.)
\end{exercise}

\begin{exercise}%
\label{exc.cat_wo_terminal}
Not every category has a terminal object. Find one that doesn't.
\end{exercise}

\begin{proposition}
All terminal objects in a category $\cat{C}$ are isomorphic.
\end{proposition}
\begin{proof}
This is a simple, but powerful standard argument. Suppose $Z$ and $Z'$ are both
terminal objects in some category $\cat{C}$. Then there exist (unique) maps
$a\colon Z \to Z'$ and $b\colon Z' \to Z$. Composing these, we get a map $a\cp
b\colon Z \to Z$. Now since $Z$ is terminal, this map $Z \to Z$ must be unique.
But $\id_Z$ is also such a map. So we must have $a\cp b = \id_Z$.  Similarly, we
find that $b \cp a= \id_{Z'}$. Thus $a$ is an isomorphism, with inverse $b$.
\end{proof}

\begin{remark}[``The limit'' vs.\ ``a
  limit'']%
\label{rem.the_vs_a}%
\index{universal property}%
\index{unique up to unique isomorphism}
Not only are all terminal objects isomorphic, there is a unique isomorphism between any two. We hence say ``terminal objects are unique up to unique isomorphism.'' To a category theorist, this is
very nearly the same thing as saying ``all terminal objects are equal.'' Thus we often abuse terminology and talk of `the' terminal object, rather than ``a'' terminal
object. We will do the same for any sort of limit or colimit, e.g.\ speak of ``the product'' of two sets, rather than ``a product.'' We saw a similar phenomenon in \cref{def.meets_joins}.
\end{remark}

\paragraph{Products.} Products are slightly more complicated to formalize than terminal objects, but they are familiar in practice.

\begin{definition}%
\index{product!in a category}%
\index{limit!product as}
Let $\cat{C}$ be a category, and let $X,Y$ be objects in $\cat{C}$. A
\emph{product} of $X$ and $Y$ is an object, denoted $X\times Y$, together with
morphisms $p_X\colon X\times Y \to X$ and $p_Y\colon X \times Y \to Y$ such that
for all objects $C$ together with morphisms $f\colon C \to X$ and $g\colon C \to
Y$, there exists a unique morphism $C \to X \times Y$, denoted $\pair{f,g}$, for which the following diagram commutes:
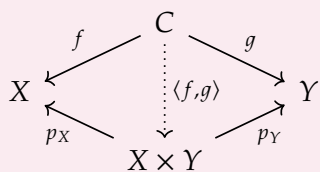
\[
\begin{tikzcd}[row sep=small]
& C \ar[dl,"f",swap] \ar[dr, "g"] \ar[dd,"\pair{f,g}",dotted]\\
X && Y \\
& X\times Y \ar[ul, "p_X"] \ar[ur,"p_Y",swap]
\end{tikzcd}
\]
\end{definition}

We will try to bring this down to earth in \cref{ex.product_in_set}. Before we do, note that $X\times Y$ is an object equipped with morphisms to $X$ and $Y$. Roughly speaking, it is like ``the best object-equipped-with-morphisms-to-$X$-and-$Y$'' in all of $\cat{C}$, in the sense that any other object-equipped-with-morphisms-to-$X$-and-$Y$ maps to it uniquely. This is called a \emph{universal property}. It's customary to use a dotted line
to indicate the unique morphism that exists because of some universal property.

\begin{example}%
\label{ex.product_in_set}%
\index{product!of sets}
In $\smset$, a product of two sets $X$ and $Y$ is their usual cartesian product
\[
X \times Y\coloneqq \{(x,y) \mid x \in X, y \in Y\},
\]
which comes with two projections $p_X\colon X\times Y\to X$ and $p_Y\colon X\times Y\to Y$, given by $p_X(x,y)\coloneqq x$ and $p_Y(x,y)\coloneqq y$.

Given any set $C$ with functions $f\colon C \to X$ and $g\colon C \to Y$, the
unique map from $C$ to $X\times Y$ such that the required diagram commutes is
given by $\langle f,g\rangle(c)\coloneqq (f(c),g(c))$.

Here is a picture of the product $\ord{6} \times \ord{4}$ of sets $\ord{6}$ and
$\ord{4}$.
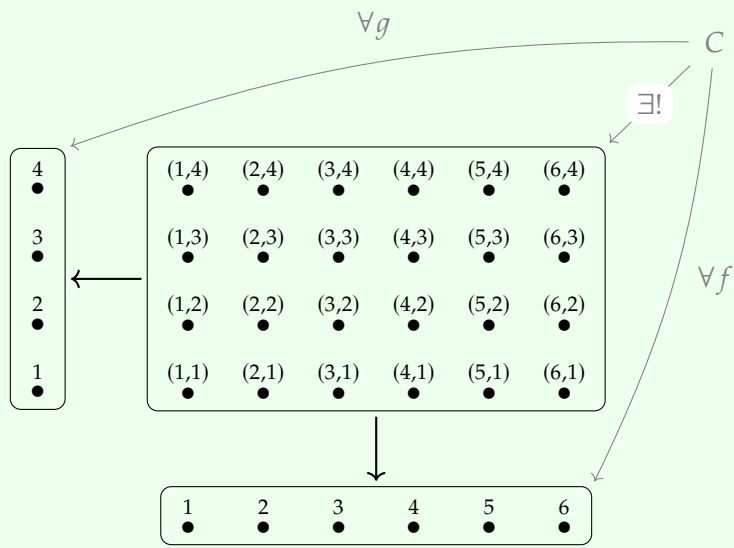
\[
\begin{tikzpicture}[y=.9cm, rounded corners, shorten <=2pt, shorten >=2pt, baseline=(allf)]
	\foreach \i in {1,...,6}{
		\foreach \j in {1,...,4} {
			\node (ij\i\j) at (\i,\j) {$\LMO{(\i,\j)}$};
			\ifthenelse{\i=1}{\node (j\j) at (-1,\j) {$\LMO{\j}$};}{};
		}
		\node (i\i) at (\i,-1) {$\LMO{\i}$};
	}
	\node[draw, inner ysep=0pt, fit=(i1) (i6)] (i) {};
	\node[draw, inner ysep=0pt, fit=(j1) (j4)] (j) {};
	\node[draw, inner ysep=0pt, fit=(ij11) (ij64)] (ij) {};
	\draw[->, thick] (ij) -- (i);
	\draw[->, thick] (ij) -- (j);
	\begin{scope}[gray]
  	\node (X) at (8,6) {$C$};
  	\draw[->, bend left=10pt] (X) to node[right=5pt] (allf) {$\forall f$} (i.north east);
  	\draw[->, bend right=10pt] (X) to node[above=5pt] {$\forall g$} (j.north east);
  	\draw[->] (X) to node[fill=white] {$\exists!$}(ij.north east);
	\end{scope}
\end{tikzpicture}
\qedhere
\]
\end{example}

\begin{exercise}%
\index{meet}%
\label{exc.meet_product}%
\index{product!meet as}
Let $(P,\le)$ be a preorder, let $x,y \in P$ be elements, and let $\cat{P}$ be the corresponding category. Show that the product $x\times y$ in $\cat{P}$ agrees with their meet $x\wedge y$ in $P$.
\end{exercise}

\begin{example}%
\index{product!of categories}
Given two categories $\cat{C}$ and $\cat{D}$, their product $\cat{C} \times
\cat{D}$ may be given as follows. The objects of this category are pairs
$(c,d)$, where $c$ is an object of $\cat{C}$ and $d$ is an object of $\cat{D}$.
Similarly, morphisms $(c,d) \to (c',d')$ are pairs $(f,g)$ where $f\colon c \to
c'$ is a morphism in $\cat{C}$ and $g\colon d \to d'$ is a morphism in
$\cat{D}$. Composition of morphisms is simply given by composing each entry in
the pair separately, so $(f,g)\cp (f',g')=(f\cp f',g\cp g')$.
\end{example}

\begin{exercise}%
\label{exc.product_cats}
\begin{enumerate}
  \item What are the identity morphisms in a product category $\cat{C}\times\cat{D}$?
  \item Why is composition in a product category associative?%
\index{associativity!in product category}
  \item What is the product category $\Cat{1} \times \Cat{2}$?
  \item What is the product category $\cat{P} \times \cat{Q}$ when $P$ and $Q$ are preorders and $\cat{P}$ and $\cat{Q}$ are the corresponding categories?
\qedhere
\end{enumerate}
\end{exercise}

These two constructions, terminal objects and products, are subsumed by the
notion of limit.

\subsection{Limits}

We'll get a little abstract. Consider the definition of product. This says that
given any pair of maps $X \xleftarrow{f} C \xrightarrow{g} Y$, there exists a
unique map $C \to X\times Y$ such that certain diagrams commute. This has the
flavor of being terminal---there is a unique map to $X\times Y$---but it seems a bit more
complicated. How are the two ideas related?

It turns out that products \emph{are} terminal objects, but of a different
category, which we'll call $\Cat{Cone}(X,Y)$, \emph{the category of cones
over $X$ and $Y$ in $\cat{C}$}. We will see in \cref{exc.prod_as_term_cone} that $X \xleftarrow{p_X} X\times Y\xrightarrow{p_Y} Y$ is a
terminal object in $\Cat{Cone}(X,Y)$.%
\index{category!of cones}

An object of $\Cat{Cone}(X,Y)$ is simply a pair of maps $X \xleftarrow{f} C
\xrightarrow{g} Y$. A morphism from $X \xleftarrow{f} C
\xrightarrow{g} Y$ to $X \xleftarrow{f'} C' \xrightarrow{g'} Y$ in $\Cat{Cone}(X,Y)$
is a morphism $a\colon C \to C'$ in $\cat{C}$ such that the following diagram commutes:
\[
\begin{tikzcd}[row sep=small]
& C \ar[dl,"f",swap] \ar[dr, "g"] \ar[dd,"a"]\\
X && Y \\
& C' \ar[ul, "f'"] \ar[ur,"g'",swap]
\end{tikzcd}
\]

\begin{exercise}%
\label{exc.prod_as_term_cone}
Check that a product $X \xleftarrow{p_X} X\times Y\xrightarrow{p_Y} Y$ is
exactly the same as a terminal object in $\Cat{Cone}(X,Y)$.
\end{exercise}

We're now ready for the abstract definition. Don't worry if the details are
unclear; the main point is that it is possible to unify terminal objects, maximal
elements, and meets, products of sets, preorders, and categories, and many other familiar
friends under the scope of a single definition. In fact, they're all just terminal objects in different categories.

Recall from \cref{def.diagram_commutes} that formally speaking, a diagram in $\cat{C}$ is just a functor $D\colon\cat{J}\to\cat{C}$. Here $\cat{J}$ is called the \emph{indexing category} of the diagram $D$.%
\index{category!indexing}

\begin{definition}%
\label{def.cones_limits}%
\index{limit}%
\index{terminal object!limit as}
Let $D\colon \cat{J}\to \cat{C}$ be a diagram. A \emph{cone} $(C,c_\ast)$ over $D$ consists of 
\begin{enumerate}[label=(\roman*)]
\item an object $C \in \cat{C}$;
\item for each object $j \in \cat{J}$, a morphism $c_j\colon C \to D(j)$.
\end{enumerate}
To be a cone, these must satisfy the following property:
\begin{enumerate}[label=(\alph*)]
\item for each $f\colon j \to k$ in $\cat{J}$, we have $c_k=c_j\cp D(f)$.
\end{enumerate}
\index{cone}

A \emph{morphism of cones} $(C,c_\ast) \to (C',c'_\ast)$ is a morphism $a\colon C \to
C'$ in $\cat{C}$ such that for all $j \in \cat{J}$ we have $c_j=a\cp c'_j$.
Cones over $D$, and their morphisms, form a category $\Cat{Cone}(D)$.

The \emph{limit} of $D$, denoted $\lim(D)$, is the terminal object in the category
$\Cat{Cone}(D)$. Say it is the cone $\lim(D)=(C,c_*)$; we refer to $C$ as the \emph{limit object} and the map $c_j$ for any $j\in\cat{J}$ as the \emph{$j^{\tn{th}}$ projection map}.
\end{definition}

For visualization purposes, if $\cat{J}$ is the free category on the graph
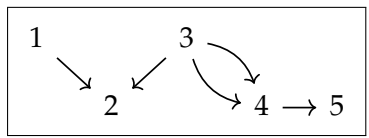
\[
\boxCD{
  \begin{tikzcd}[sep=small, ampersand replacement=\&]
    1\ar[dr]\&\&3\ar[ld]\ar[dr, bend left]\ar[dr, bend right]\\
    \&2\&\&4\ar[r]\&5
  \end{tikzcd}
}
\]
with five objects and five non-identity morphisms, then we may draw a diagram
$D\colon\cat{J}\to\cat{C}$ inside $\cat{C}$ as on the left below, and a cone on
it as on the right:
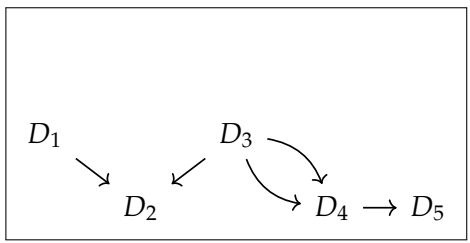
\[
\boxCD{
  \begin{tikzcd}[sep=small, ampersand replacement=\&]
  	{\color{white}C}\\[10pt]
    D_1\ar[dr]\&\&D_3\ar[ld]\ar[dr, bend left]\ar[dr, bend right]\\
    \&D_2\&\&D_4\ar[r]\&D_5
  \end{tikzcd}
}
\hspace{.8in}
\boxCD
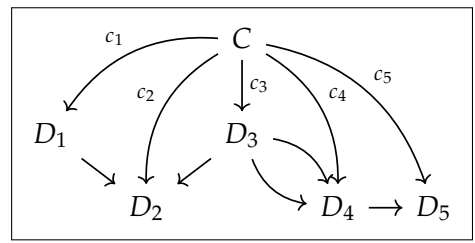
{
  \begin{tikzcd}[sep=small, ampersand replacement=\&]
		\&\&C\ar[dll, bend right, "c_1"']\ar[ddl, bend right, "c_2"']\ar[d, "c_3"]\ar[ddr, bend left, "c_4"]\ar[ddrr, bend left, "c_5"]\\[10pt]
    D_1\ar[dr]\&\&D_3\ar[ld]\ar[dr, bend left]\ar[dr, bend right]\\
    \&D_2\&\&D_4\ar[r]\&D_5
  \end{tikzcd}
}
\]
Here, any two parallel paths that start at $C$ are considered the same. Note
that both these diagrams depict a collection of objects and morphisms inside the category
$\cat{C}$.

\begin{example}
Terminal objects are limits where the indexing category is empty, $\cat{J}=\varnothing$.
\end{example}

\begin{example}%
\label{ex.prods_as_lims}%
\index{product!as limit}
  Products are limits where the indexing category consists of two objects $v, w$
  and no arrows, $\cat{J}=\fbox{$\LMO{v}\quad\LMO{w}$}$.
\end{example}

\subsection{Finite limits in $\smset$}

Recall that this discussion was inspired by wanting to understand
$\Pi$-operations, and in particular $\Pi_!$. We can now see that a database
instance $I \colon \cat{C} \to \smset$ is a diagram in $\smset$. The functor
$\Pi_!$ takes the limit of this diagram. In this subsection we give a formula
describing the result. This captures \emph{all finite limits in $\smset$}.

In database theory, we work with categories $\cat{C}$ that are presented by a
finite graph plus equations. We won't explain the details, but it's in fact
enough just to work with the graph part: as far as limits are concerned, the
equations in $\cat{C}$ don't matter. For consistency with the rest of this
section, let's denote the database schema by $\cat{J}$ instead of $\cat{C}$.

\begin{theorem} %
\label{thm.set_limits}%
\index{limit!formula for finite limits in
  $\smset$}%
\index{functor!set@$\smset$-valued}
Let $\cat{J}$ be a category presented by the finite graph $(V,A,s,t)$ together with
some equations, and let $D\colon \cat{J} \to \smset$ be a set-valued functor. Write $V=\{v_1,\ldots,v_n\}$. The set
\begin{multline*}
  \lim_\cat{J} D \coloneqq \big\{(d_1,\ldots,d_n)\mid d_i\in D(v_i)\text{ for all }1\leq i\leq n\text{ and }\\
  \text{for all }a\colon v_i\to v_j\in A, \text{ we have } D(a)(d_i)=d_j\big\}.
\end{multline*}
together with the projection maps $p_i\colon(\lim_\cat{J} D)\to D(v_i)$ given by $p_i(d_1,\ldots,d_n)\coloneqq d_i$, is a limit of $D$.
\end{theorem}

\begin{example}
   If $J$ is the empty graph \fbox{\color{white}$\LMO{}$}, then $n=0$: there are no vertices. There is
   exactly one empty tuple $(\ )$, which vacuously satisfies the properties, so
   we've constructed the limit as the singleton set $\{ (\ ) \}$ consisting of
   just the empty tuple. Thus the limit of the empty diagram, i.e.\ the terminal object in $\smset$ is the singleton set. See \cref{rem.the_vs_a}.
\end{example}
  
\begin{exercise}%
\label{exc.limit_formula_products}
Show that the limit formula in \cref{thm.set_limits} works for products. See \cref{ex.prods_as_lims}.
\end{exercise}

\begin{exercise}  
  If $D\colon\Cat{1}\to\smset$ is a functor, what is the limit of $D$? Compute it using \cref{thm.set_limits}, and check your answer against \cref{def.cones_limits}.
\end{exercise}

\paragraph{Pullbacks.}%
\index{limit!pullback as}
\index{pullback!as limit}

In particular, the condition that the limit of $D\colon\cat{J}\to\smset$ selects tuples $(d_1,\dots,d_n)$
such that $D(a)(d_i)=d_j$ for each morphism $a\colon i\to j$ in $\cat{J}$ allows us to use limits to select data that satisfies
certain equations or constraints. This is what allows us to express queries in
terms of limits. Here is an example.%
\index{database!query}

\begin{example}%
\label{ex.pullback}
  If $J$ is presented by the \emph{cospan} graph
  \fbox{$\LMO{x}\Too{f}\LMO{a}\Fromm{g}\LMO{y}$}, then its limit is
  known as a \emph{pullback}. Given the diagram $X \xrightarrow{f} A
  \xleftarrow{g} Y$, the pullback is the cone shown on the left below:
  \[
  \begin{tikzcd}
  	C\ar[d, "c_x"']\ar[r, "c_y"]\ar[dr, description, "c_a"]&Y\ar[d, "g"]\\
		X\ar[r, "f"']&A
  \end{tikzcd}
\hspace{1in}
  \begin{tikzcd}
  	X\times_AY\ar[d, "c_x"']\ar[r, "c_y"]&Y\ar[d, "g"]\\
		X\ar[r, "f"']&A\ar[ul, phantom, very near end, "\lrcorner"]
  \end{tikzcd}
  \]
  The fact that the diagram commutes means that the diagonal arrow $c_a$ is in some sense superfluous, so one generally denotes pullbacks by dropping the diagonal arrow, naming the cone point $X\times_AY$, and adding the $\lrcorner$ symbol, as shown to the right above.
  
  Here is a picture to help us unpack the definition in $\smset$. We take $X
  =\ord{6}$, $Y=\ord{4}$, and $A$ to be the set of colors
  $\{\textrm{red},\textrm{blue},\textrm{black}\}$.
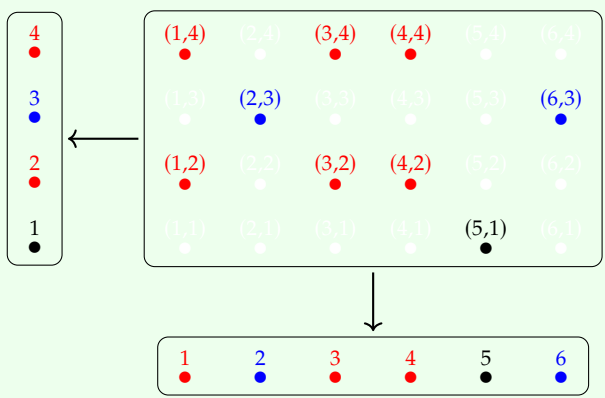
\[
\begin{tikzpicture}[rounded corners, y=5ex, shorten <=2pt, shorten >=2pt]
	\foreach \i in {1,...,6}{
		\tikzmath{
			if or(or(\i==1, \i==3), \i==4) then {let \coli=red; let \ci=1;} else {
				if or(\i==2, \i==6) then {let \coli=blue; let \ci=2;} else {let \coli=black; let \ci=3;};
			};
		}
		\foreach \j in {1,...,4} {
			\tikzmath{
				if or(\j==2, \j==4) then {let \colj=red; let \cj=1;} else {
					if \j==3 then {let \colj=blue; let \cj=2;} else {let \colj=black; let \cj=3;};
				};
			}
			\ifthenelse{\ci=\cj}
				{\node[\coli] (ij\i\j) at (\i,\j) {$\LMO{(\i,\j)}$};}
				{\node[white] (ij\i\j) at (\i,\j) {$\LMO{(\i,\j)}$};};
			\ifthenelse{\i=1}{\node[\colj] (j\j) at (-1,\j) {$\LMO{\j}$};}{};
		}
		\node[\coli] (i\i) at (\i,-1) {$\LMO{\i}$};
	}
	\node[draw, inner ysep=0, fit=(i1) (i6)] (i) {};
	\node[draw, inner ysep=0, fit=(j1) (j4)] (j) {};
	\node[draw, inner ysep=0, fit=(ij11) (ij64)] (ij) {};
	\draw[->, thick] (ij) -- (i);
	\draw[->, thick] (ij) -- (j);
\end{tikzpicture}
\]
The functions $f\colon \ord{6} \to A$ and $g\colon \ord{4} \to A$ are expressed in
the coloring of the dots: for example, $g(2)=g(4)=\textrm{red}$, while
$f(5)=\textrm{black}$. The pullback selects pairs $(i,j) \in \ord{6} \times
\ord{4}$ such that $f(i)$ and $g(j)$ have the same color.
\end{example}

\begin{remark} \label{rem.pullbackpullback}
  As mentioned following \cref{def.pullback_along}, this definition of pullback
  is not to be confused with the pullback of a set-valued functor along a
  functor; they are for now best thought of as different concepts which
  accidentally have the same name. Due to the power of the primordial ooze,
  however, the pullback along a functor is a special case of pullback as the
  limit of a cospan: it can be understood as the pullback of a certain cospan in
  $\smcat$.  To unpack this, however, requires the notions of category of
  elements and discrete opfibration; ask your friendly neighborhood category
  theorist.
\end{remark}
\index{limit|)}

\subsection{A brief note on colimits}
\label{subsec.brief_colimits}%
\index{colimit}
Just like upper bounds have a dual concept---namely that of lower bounds---so
limits have a dual concept: colimits. To expose the reader to this concept, we
provide a succinct definition of these using opposite categories and opposite
functors. The point, however, is just exposure; we will return to explore
colimits in detail in \cref{chap.hypergraph_cats}.%
\index{dual notions!colimits and limits as}

\begin{exercise}%
\label{exc.opposite_functor}
Recall from \cref{def.opposite_cat} that every category $\cat{C}$ has an
opposite $\cat{C}\op$. Let $F\colon \cat{C} \to \cat{D}$ be a functor. How
should we define its opposite, $F\op\colon \cat{C}\op \to \cat{D}\op$? That is,
how should $F\op$ act on objects, and how should it act on morphisms?
\end{exercise}

\begin{definition}%
\label{def.colimit}%
\index{cocone}
Given a category $\cat{C}$ we say that a \emph{cocone} in $\cat{C}$ is a cone
in $\cat{C}\op$.

Given a diagram $D \colon \cat{J} \to \cat{C}$, we may take the limit of the
functor $D\op \colon \cat{J}\op \to \cat{C}\op$. This is a cone in $\cat{C}\op$,
and so by definition a cocone in $\cat{C}$. The \emph{colimit} of $D$ is this
cocone.
\end{definition}

\cref{def.colimit} is like a compressed file: useful for transmitting quickly, but completely useless for working with, unless you can successfully unpack it. We will unpack it later in \cref{chap.hypergraph_cats} when we discuss electric circuits. 

\section{Summary and further reading}%
\label{sec.ch2_further_reading}

Congratulations on making it through one of the longest chapters in the book! We
apologize for the length, but this chapter had a lot of work to do. Namely it
introduced the ``big three'' of category theory---categories, functors, and
natural transformations---as well as discussed adjunctions, limits, and very briefly colimits.

That's really quite a bit of material. For more on all these subjects, one can
consult any standard book on category theory, of which there are many. The bible
(old, important, seminal, and requires a priest to explain it) is
\cite{MacLane:1998a}; another thorough introduction is \cite{Borceux:1994a}; a
logical perspective is given in \cite{Awodey:2010a}; a computer science
perspective is given in \cite{Barr.Wells:1990a} and \cite{Pierce:1991} and
\cite{Walters:1992}; math students should probably read \cite{Leinster:2014a} or
\cite{Riehl:2017a} or \cite{Grandis:2018}; a general audience might start with \cite{Spivak:2014a}.

We presented categories from a database perspective, because data is pretty ubiquitous in our world. A database schema---i.e.\ a system of interlocking tables---can be captured by a category $\cat{C}$, and filling it with data corresponds to a functor $\cat{C}\to\smset$. Here $\smset$ is the category of sets, perhaps the most important category to mathematicians.

The perspective of using category theory to model databases has been rediscovered several times. It seems to have first been discussed by various authors around the mid-90's \cite{Islam.Phoa:1994a,Cadish.Diskin:1995a,Piessens.Steegmans:1995a,Tuijn.Gyssens:1996a}. Bob Rosebrugh and collaborators took it much further in a series of papers including \cite{Fleming.Gunther.Rosebrugh:2003a,Johnson.Rosebrugh:2002a,Rosebrugh.Wood:1992a}. Most of these authors tend to focus on sketches, which are more expressive categories. Spivak rediscovered the idea again quite a bit later, but focused on categories rather than sketches, so as to have all three data migration functors $\Delta,\Sigma,\Pi$; see \cite{spivakfunctorial2012,Spivak.Wisnesky:2015a}. The version of this story presented in the chapter, including the white and black nodes in schemas, is part of a larger theory of algebraic databases, where a programming language such as Java or Haskell is attached to a database. The technical details are worked out in \cite{Schultz.Spivak.Vasilakopoulou.Wisnesky:2017a}, and its use in database integration projects can be found in \cite{Schultz.Wisnesky:2015a,Wisnesky.Spivak.Schultz.Subrahmanian:2015a,}.

Before we leave this chapter, we want to emphasize two things: coherence conditions and universal constructions.

\paragraph{Coherence conditions.}%
\index{coherence}
In the definitions of category, functor, and natural transformations, we have
data (indexed by (i)) that is required to satisfy certain properties (indexed by
(a)). Indeed, for categories it was about associativity and unitality of
composition, for functors it was about respecting composition and identities,
and for natural transformations it was the naturality condition. These
conditions are often called \emph{coherence conditions}: we want the various
structures to cohere, to work well together, rather than to flop around
unattached.%
\index{associativity!as coherence condition}%
\index{unitality!as coherence condition}

Understanding why these particular structures and coherence conditions are ``the right ones'' is more science than mathematics: we empirically observe that certain combinations result in ideas that are both widely applicable and also strongly compositional. That is, we become satisfied with coherence conditions when they result in beautiful mathematics down the road.

\paragraph{Universal constructions.}%
\index{universal property}

Universal constructions are one of the most important themes of category theory. Roughly speaking, one gives some specified shape in a category and says ``find me the best solution!'' And category theory comes back and says ``do you want me to approximate from the left or the right (colimit or limit)?'' You respond, and either there is a best solution or there is not. If there is, it's called the (co)limit; if there's not we say ``the (co)limit does not exist.'' 

Even data migration fits this form. We say ``find me the closest thing in $\cat{D}$ that matches my $\cat{C}$-instance using my functor $F\colon\cat{C}\to\cat{D}$.'' In fact this approach---known as Kan extensions---subsumes the others. One of the two founders of category theory, Saunders Mac Lane, has a section in his book \cite{MacLane:1998a} called ``All concepts are Kan extensions,'' a big statement, no?
\index{Kan extension}

\setcounter{chapter}{3}
\chapter[Co-design: profunctors and monoidal categories]{Collaborative design:\\Profunctors, categorification, and
monoidal categories}%
\label{chap.codesign}%
\index{co-design|(}%
\index{co-design|seealso {feasibility}}

\section{Can we build it?}

When designing a large-scale system, many different fields of expertise are
joined to work on a single project. Thus the whole project team is divided into
multiple sub-teams, each of which is working on a sub-project. And we recurse
downward: the sub-project is again factored into sub-sub-projects, each with
their own team. One could refer to this sort of hierarchical design process as \emph{collaborative design}, or co-design. In this
chapter, we discuss a mathematical theory of co-design, due to Andrea Censi \cite{Censi:2015a}.%
\index{design}

Consider just one level of this hierarchy: a project and a set of teams working
on it. Each team is supposed to \emph{provide} resources---sometimes called
``functionalities''---to the project, but the team also \emph{requires} resources in
order to do so. Different design teams must be allowed to plan and work
independently from one another in order for progress to be made. Yet the design
decisions made by one group affect the design decisions others can make: if $A$
wants more space in order to provide a better radio speaker, then $B$ must use less space. So
these teams---though ostensibly working independently---are dependent on each other
after all.%
\index{resource}

The combination of dependence and independence is crucial for progress to be
made, and yet it can cause major problems. When a team requires more resources
than it originally expected to require, or if it cannot provide the resources
that it originally claimed it could provide, the usual response is for the team
to issue a design-change notice. But these affect neighboring teams: if team $A$
now requires more than originally claimed, team $B$ may have to change their design, which can in turn affect team $C$. Thus these design-change notices can ripple through the system through feedback loops and can cause whole projects to fail \cite{Subrahmanian.Lee.Granger:2015a}.

As an example, consider the design problem of creating a robot to carry some load at some velocity. The top-level planner breaks the problem into three design teams: chassis team, motor team, and battery team. Each of these teams could break up into multiple parts and the process repeated, but let's remain at the top level and consider the resources produced and the resources required by each of our three teams.

The chassis in some sense provides all the functionality---it carries the load at the velocity---but it requires some things in order to do so. It requires money to build, of course, but more to the point it requires a source of torque and speed. These are supplied by the motor, which in turn needs voltage and current from the battery. Both the motor and the battery cost money, but more importantly they need to be carried by the chassis: they become part of the load. A feedback loop is created: the chassis must carry all the weight, even that of the parts that power the chassis. A heavier battery might provide more energy to power the chassis, but is the extra power worth the heavier load?

In the following picture, each part---chassis, motor, battery, and robot---is shown as a box with ports on the left and right. The functionalities, or resources produced by the part are shown as ports on the left of the box, and the resources required by the part are shown as ports on its right.
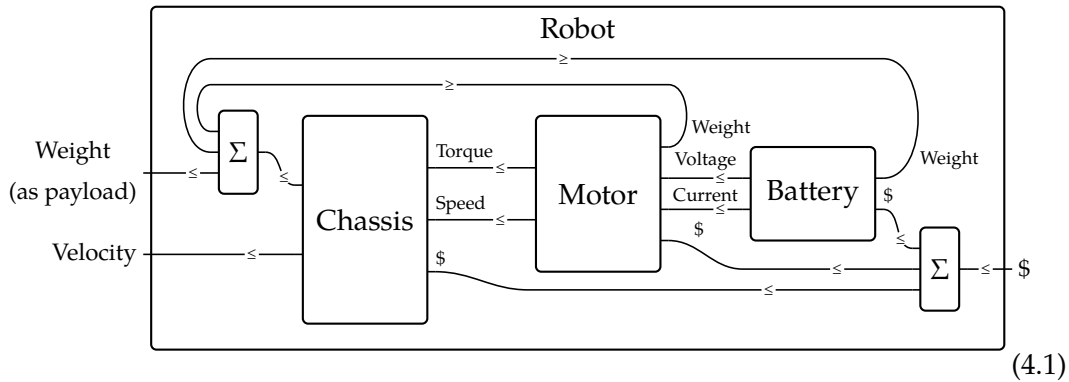
\begin{equation}%
\label{eqn.chassis}

\end{equation}
The boxes marked $\Sigma$ correspond to summing inputs. These boxes are not to
be designed, but we will see later that they fit easily into the same conceptual
framework. Note also the $\leq$'s on each wire; they indicate that if box $A$
requires a resource that box $B$ produces, then $A$'s requirement must be
less-than-or-equal-to $B$'s production. The chassis requires torque, and the motor must produce at least that much torque.

To formalize this a bit more, let's call diagrams like the one above
\emph{co-design diagrams}. Each of the wires in a co-design diagram represents a
preorder of resources. For example, in \cref{eqn.chassis} every wire corresponds to
a resource type---weights, velocities, torques, speeds, costs, voltages, and
currents---where resources of each type can be ordered from less useful to more
useful. In general, these preorders do not have to be linear orders, though in the
above cases each will likely correspond to a linear order: $\$10\leq \$20$,
$5\textrm{W}\leq6\textrm{W}$, and so on.%
\index{co-design!diagram}

Each of the boxes in a co-design diagram corresponds to what we call a
\emph{feasibility relation}. A feasibility relation matches resource production
with requirements. For every pair $(p,r)\in P\times R$, where $P$ is the preorder
of resources to be produced and $R$ is the preorder of resources to be required,
the box says ``true'' or ``false''---feasible or infeasible---for that pair. In
other words, ``yes I can provide $p$ given $r$'' or ``no, I cannot provide $p$
given $r$.''%
\index{feasibility relation}

Feasibility relations hence define a function $\Phi\colon P \times R \to \Bool$.
For a function $\Phi\colon P\times R\to\Bool$ to make sense as a feasibility
relation, however, there are two conditions:
\begin{enumerate}[label=(\alph*)]
	\item If $\Phi(p,r)=\true$ and $p' \le p$, then $\Phi(p',r)=\true$.
	\item If $\Phi(p,r)=\true$ and $r \le r'$ then $\Phi(p,r')=\true$.
\end{enumerate}
These conditions, which we will see again in
\cref{def.feasibility_relationship}, say that if you can produce $p$ given
resources $r$, you can (a) also produce less $p'\leq p$ with the same resources
$r$, and (b) also produce $p$ given more resources $r'\geq r$. We will see that
these two conditions are formalized by requiring $\Phi$ to be a monotone map
$P\op\times R\to\Bool$.%
\index{booleans!and feasibility}

A \emph{co-design problem}%
\index{co-design!problem}, represented by a co-design
diagram, asks us to find the composite of some feasibility relations. It asks,
for example, given these capabilities of the chassis, motor, and battery teams,
can we build a robot together? Indeed, a co-design diagram factors a
problem---for example, that of designing a robot---into interconnected
subproblems, as in \cref{eqn.chassis}.  Once the feasibility relation is worked
out for each of the subproblems, i.e.\ the inner boxes in the diagram, the
mathematics provides an algorithm producing the feasibility relation of the
whole outer box. This process can be recursed downward, from the largest problem
to tiny subproblems. 

In this chapter, we will understand co-design problems in terms of enriched
profunctors, in particular $\Bool$-profunctors. A $\Bool$-profunctor is like a
bridge connecting one preorder to another. We will show how the co-design framework
gives rise to a structure known as a compact closed category, and that any
compact closed category can interpret the sorts of wiring diagrams we see in
\cref{eqn.chassis}. 

\section{Enriched profunctors}%
\index{profunctor|(}

In this section we will understand how co-design problems form a category. Along
the way we will develop some abstract machinery that will allow us to replace
preorder design spaces with other enriched categories.

\subsection{Feasibility relationships as $\Bool$-profunctors}%
\index{feasibility relation|(}%
\index{profunctor!Bool@$\Bool$}

The theory of co-design is based on preorders: each resource---e.g.\ velocity,
torque, or \$---is structured as a preorder. The order $x\leq y$ represents the
\emph{availability of $x$ given $y$}, i.e.\ that whenever you have $y$, you also
have $x$. For example, in our preorder of wattage, if 5W $\leq$ 10W, it means that
whenever we are provided 10W, we implicitly also have 5W. Above we referred to this as an order from less useful to more useful: if $x$ is always available given $y$, then $x$ is less useful than $y$.%

We know from \cref{subsec.preorders_Bool_enriched} that a preorder $\cat{X}$ can be
conceived of as a $\Bool$-category. Given $x,y\in X$, we have
$\cat{X}(x,y)\in\BB$; this value responds to the assertion ``$x$ is available
given $y$,'' marking it either $\true$ or $\false$.

Our goal is to see feasibility relations as $\Bool$-profunctors, which are a
special case of something called enriched profunctors. Indeed, we hope that this
chapter will give you some intuition for profunctors, arising from the table
\[
\begin{tabular}{r|l}
 $\Bool$-category & preorder \\
 $\Bool$-functor & monotone map \\
 $\Bool$-profunctor & feasibility relation
\end{tabular}
\]
Because enriched profunctors are a bit abstract, we first concretely discuss
$\Bool$-profunctors as feasibility relations. Recall that if $\cat{X}=(X,\leq)$
is a preorder, then its opposite $\cat{X}\op=(X,\geq)$ has $x\geq y$ iff $y\leq x$.

\begin{definition}%
\label{def.feasibility_relationship}
Let $\cat{X}=(X,\leq_X)$ and $\cat{Y}=(Y,\leq_Y)$ be preorders. A \emph{feasibility relation} for $\cat{X}$ given $\cat{Y}$ is a monotone map
\begin{equation}%
\label{eqn.feasibility}
	\Phi\colon \cat{X}\op\times \cat{Y}\to\Bool.
\end{equation}
We denote this by $\Phi\colon\cat{X}\tickar\cat{Y}$.

Given $x\in X$ and $y\in Y$, if $\Phi(x,y)=\true$ we say \emph{$x$ can be obtained given $y$}.
\end{definition}

As mentioned in the introduction, the requirement that $\Phi$ is monotone says
that if $x'\leq_X x$ and $y\leq_Y y'$ then $\Phi(x,y)\leq_\Bool \Phi(x',y')$. In
other words, if $x$ can be obtained given $y$, and if $x'$ is available given
$x$, then $x'$ can be obtained given $y$. And if furthermore $y$ is available
given $y'$, then $x'$ can also be obtained given $y'$.

\begin{exercise} %
\label{exc.a_profunctor}
Suppose we have the preorders
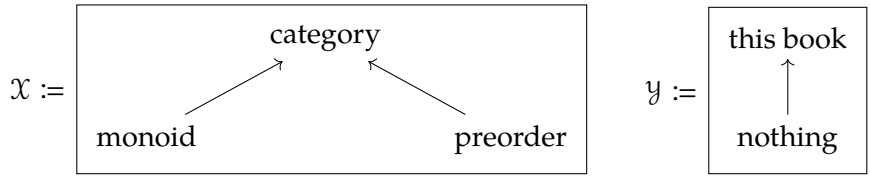
\[
\begin{tikzpicture}
	\node (cat) {category};
	\node[below right=1 of cat] (pos) {preorder};
	\node[below left=1 of cat] (mon) {monoid};
	\draw[->] (pos) -- (cat);
	\draw[->] (mon) -- (cat);
	\node[draw, fit=(cat) (pos) (mon)] (X) {};
	\node[left=0 of X] {$\cat{X}\coloneqq$};
	\node[right=2 of pos] (0) {nothing};
	\node at (0|-cat) (1) {this book};
	\draw[->] (0) -- (1);
	\node[draw, fit=(0) (1)] (Y) {};
	\node[left=0 of Y] {$\cat{Y}\coloneqq$};
\end{tikzpicture}
\]
\begin{enumerate}
  \item Draw the Hasse diagram for the preorder $\cat{X}\op\times\cat{Y}$. 
  \item Write down a profunctor $\Lambda\colon\cat{X} \tickar \cat{Y}$ and, reading
  $\Lambda(x,y)=\true$ as ``my aunt can explain an $x$ given $y$,'' give an
  interpretation of the fact that the preimage of $\true$ forms an upper set in
  $\cat{X}\op\times\cat{Y}$. %
\index{preimage}
  \qedhere
\end{enumerate}
\end{exercise}

To generalize the notion of feasibility relation, we must notice that the
symmetric monoidal preorder $\Bool$ has more structure than just that of a
symmetric monoidal preorder: as mentioned in \cref{exc.Bool_quantale},
$\Bool$ is a quantale. That means it has all joins $\vee$, and a closure
operation, which we'll write $\imp\colon\BB\times\BB\to\BB$. By definition, this
operation satisfies the property that for all $b,c,d\in\BB$ one has
\begin{equation}%
\label{eqn.implies_internal_hom}
	b\wedge c\leq d\quad\text{iff}\quad b\leq (c\imp d).
\end{equation}
The operation $\imp$ is given by the following table:
\begin{equation}%
\label{eqn.implies}
\begin{array}{cc|c}
c&d&c\imp d\\\hline
\true&\true&\true\\
\true&\false&\false\\
\false&\true&\true\\
\false&\false&\true
\end{array}
\end{equation}

\begin{exercise} %
\label{exc.implies_is_hom}
Show that $\imp$ as defined in \cref{eqn.implies} indeed satisfies \cref{eqn.implies_internal_hom}.
\end{exercise}

\index{quantale}
On an abstract level, it is the fact that $\Bool$ is a quantale which makes
everything in this chapter work; any other (unital commutative) quantale also
defines a way to interpret co-design diagrams. For example, we could use the
quantale $\Cost$, which would describe not \emph{whether} $x$ is available given
$y$ but the \emph{cost} of obtaining $x$ given $y$; see
\cref{ex.Lawveres_base,def.cat_enriched_mpos}.

\subsection{$\cat{V}$-profunctors}%
\index{profunctor!enriched}
We are now ready to recast \cref{eqn.feasibility} in abstract terms. Recall the
notions of enriched product (\cref{def.enriched_prod}), enriched functor
(\cref{def.enriched_functor}), and quantale (\cref{def.monoidal_closed}).

\begin{definition}%
\label{def.enriched_profunctor}%
\index{profunctor}
Let $\cat{V}=(V,\leq,I,\otimes)$ be a (unital commutative) quantale,%
\tablefootnote{From here on, as in \cref{chap.resource_theory}, whenever we speak of
quantales we mean unital commutative quantales.
}
 and let $\cat{X}$ and $\cat{Y}$ be $\cat{V}$-categories. A
 \emph{$\cat{V}$-profunctor} from $\cat{X}$ to $\cat{Y}$, denoted
 $\Phi\colon\cat{X}\tickar\cat{Y}$, is a $\cat{V}$-functor
\[
\Phi\colon\cat{X}\op\times\cat{Y}\to\cat{V}.
\]
\end{definition}

Note that a $\cat{V}$-functor must have $\cat{V}$-categories for domain and
codomain, so here we are considering $\cat{V}$ as enriched in itself; see
\cref{rem.quantale_enriches_itself}.%
\index{V-profunctor@$\cat{V}$-profunctor|see {profunctor, enriched}}

\begin{exercise} %
\label{exc.profunctor_def_alt}
Show that a $\cat{V}$-profunctor (\cref{def.enriched_profunctor}) is
the same as a function $\Phi\colon\Ob(\cat{X})\times\Ob(\cat{Y})\to V$ such that
for any $x,x'\in\cat{X}$ and $y,y'\in\cat{Y}$ the following inequality holds in
$\cat{V}$: 
\[
  \cat{X}(x',x)\otimes\Phi(x,y)\otimes\cat{Y}(y,y')\leq\Phi(x',y').
  \qedhere
\]
\end{exercise}

\begin{exercise}%
\label{exc.Bool_enriched_is_feas}%
\index{feasibility relation!as $\Bool$-profunctor}
Is it true that a $\Bool$-profunctor, as in \cref{def.enriched_profunctor}, is exactly the same as a feasibility relation, as in \cref{def.feasibility_relationship}, once you peel back all the jargon? Or is there some subtle difference?
\end{exercise}

We know that \cref{def.enriched_profunctor} is quite abstract. But have no fear, we will take you through it in pictures.

\begin{example}[$\Bool$-profunctors and their interpretation as
  bridges]%
\label{ex.bool_profunctor_bridges}%
\index{profunctor!as bridges}
Let's discuss \cref{def.enriched_profunctor} in the case $\cat{V}=\Bool$. One
way to imagine a $\Bool$-profunctor $\Phi\colon X\tickar Y$ is in terms
of building bridges between two cities.  Recall that a preorder (a $\Bool$-category) 
can be drawn using a Hasse
diagram. We'll think of the preorder as a city, and each vertex in it as some point of interest. An arrow $A\to B$ in the Hasse diagram means that there exists a way to get from point $A$ to
point $B$ in the city. So what's a profunctor?

A profunctor is just a bunch of bridges connecting points in one city to points
in another. Let's see a specific example. Here is a picture of a
$\Bool$-profunctor $\Phi\colon X \tickar Y$:
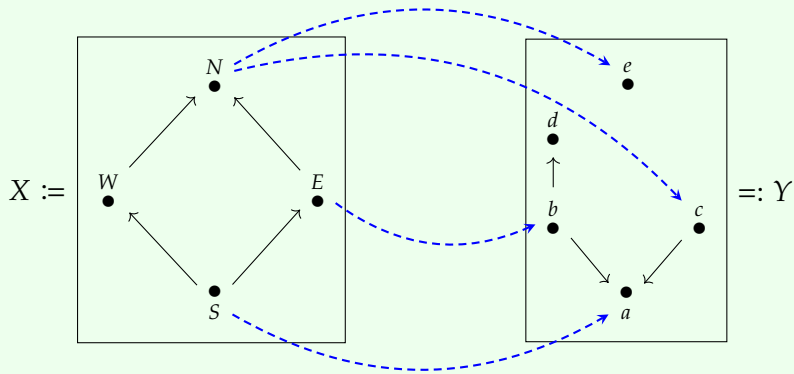
\[
\begin{tikzpicture}
	\node (mid) {};
	\node[above=of mid] (XN) {$\LMO{N}$};
	\node[left=of mid] (XW) {$\LMO{W}$};
	\node[right=of mid] (XE) {$\LMO{E}$};
	\node[below=of mid] (XS) {$\LMO[under]{S}$};
	\node[draw, fit=(XN) (XW) (XS) (XE)] (X) {};
	\draw[->]
		(XS) edge (XW) edge (XE) 
		(XW) edge (XN) 
		(XE) edge (XN)
	;
	\node[left=0 of X, font=\normalsize] {$X\coloneqq$};
	\node[right=5 of XS] (Ya) {$\LMO[under]{a}$};
	\node[above left=.4 and .5 of Ya] (Yb) {$\LMO{b}$};
	\node[above right=.4 and .5 of Ya] (Yc) {$\LMO{c}$};
	\node[above=.4 of Yb] (Yd) {$\LMO{d}$};
	\node[right=5 of XN] (Ye) {$\LMO{e}$};
	\node[draw, fit=(Ye) (Yd) (Yc) (Ya)] (Y) {};
	\draw[->]
		(Yb) edge (Ya) edge (Yd)
		(Yc) edge (Ya)
	;
	\node[right=0 of Y, font=\normalsize] {$=:Y$};
	\draw[mapsto, bend right] (XS) to (Ya);
	\draw[mapsto, bend right] (XE) to (Yb);
	\draw[mapsto, bend left] (XN) to (Yc);
	\draw[mapsto, bend left] (XN) to (Ye);
\end{tikzpicture}
\]
Both $X$ and $Y$ are preorders, e.g.\ with $W\leq N$ and $b\leq a$. With bridges
coming from the profunctor in blue, one can now use both paths within the cities
and the bridges to get from points in city $X$ to points in city $Y$. For example,
since there is a path from $N$ to $e$ and $E$ to $a$, we have $\Phi(N,e)=\true$
and $\Phi(E,a)=\true$. On the other hand, since there is no path from $W$ to
$d$, we have $\Phi(W,d)=\false.$

In fact, one could put a box around this entire picture and see a new preorder with
$W\leq N\leq c\leq a$, etc. This is called the \emph{collage} of $\Phi$; we'll
explore this in more detail later; see \cref{def.collage_prof}.%
\index{collage}
\end{example}

\begin{exercise} %
\label{exc.feas_matrix}
We can express $\Phi$ as a matrix where the $(m,n)$th entry is the value of
$\Phi(m,n)\in \BB$. Fill out the $\Bool$-matrix:
\[
\begin{array}{c|ccccc}
	\Phi&a&b&c&d&e\\\hline
  N&\?&\?&\?&\?&\true\\
  E&\true&\?&\?&\?&\?\\
  W&\?&\?&\?&\false&\?\\
  S&\?&\?&\?&\?&\?
\end{array}
\qedhere
\]
We'll call this the \emph{feasibility matrix} of
$\Phi$.%
\index{matrix!feasibility}
\end{exercise}

\begin{example}[$\Cost$-profunctors and their interpretation as
  bridges]%
\index{profunctor!$\Cost$}%
\index{profunctor!$\Cost$|seealso {Cost}}
Let's now consider $\Cost$-profunctors. Again we can view these as bridges, but
this time our bridges are labelled by their length. Recall from \cref{def.Lawvere_metric_space,eqn.cities_distances} that
$\Cost$-categories are Lawvere metric spaces, and can be depicted using weighted
graphs. We'll think of such a weighted graph as a chart of distances between points in a city, and generate a $\Cost$-profunctor by building a few bridges between
the cities.

Here is a depiction of a $\Cost$-profunctor $\Phi\colon X \tickar Y$:
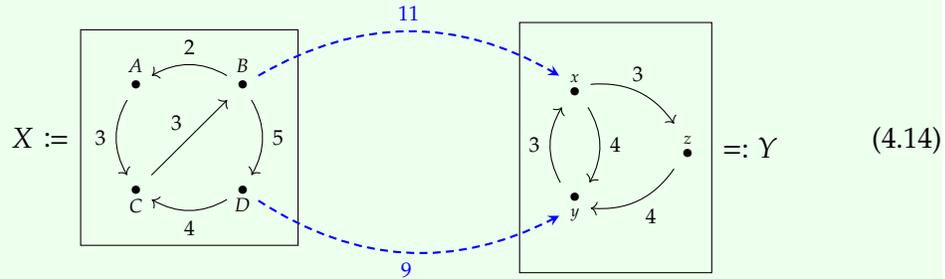
\begin{equation}%
\label{eqn.bridges_Lawv}

\end{equation}
The distance from a point $x$ in city $X$ to a point $y$ in city $Y$ is given by the shortest
path that runs from $x$ through $X$, then across one of the bridges, and then
through $Y$ to the destination $y$. So for example
\[
	\Phi(B,x)=11,\quad \Phi(A,z)=20,\quad \Phi(C,y)=17.
	\qedhere
\]
\end{example}

\begin{exercise}%
\label{exc.distance_matrix_Phi}
Fill out the $\Cost$-matrix:
\[
\begin{array}{c|ccc}
  \Phi&x&y&z\\\hline
  A&\?&\?&20\\
  B&11&\?&\?\\
  C&\?&17&\?\\
  D&\?&\?&\?
\end{array}
\qedhere
\]
\end{exercise}

\begin{remark}[Computing profunctors via matrix multiplication]
\index{matrix}
We can give an algorithm for computing the above distance matrix using matrix
multiplication. First, just like in \cref{eqn.matrix_Y}, we can begin with
the labelled graphs in \cref{eqn.bridges_Lawv} and read off the matrices of
arrow labels for $X$, $Y$, and $\Phi$:
\[
\begin{array}{c|cccc}
  M_X&A&B&C&D\\\hline
	A & 0 &\infty&3&\infty\\
	B & 2 & 0 & \infty & 5\\
	C & \infty & 3 & 0 &\infty\\
	D & \infty & \infty & 4 & 0
\end{array}
\hspace{.7in}
\begin{array}{c|ccc}
  M_\Phi&x&y&z\\\hline
  A&\infty&\infty&\infty\\
  B&11&\infty&\infty\\
  C&\infty&\infty&\infty\\
  D&\infty&9&\infty
\end{array}
\hspace{.7in}
\begin{array}{c|ccc}
  M_Y&x&y&z\\\hline
  x&0&4&3\\
  y&3&0&\infty\\
  z&\infty&4&0\\
\end{array}
\]
Recall from \cref{subsubsec.nav_matrix_mult} that the matrix of distances $d_Y$ for
$\Cost$-category $X$ can be obtained by taking the matrix power of $M_X$ with
smallest entries, and similarly for $Y$. The matrix of distances for the profunctor
$\Phi$ will be equal to $d_X*M_\Phi*d_Y$. In fact, since $X$ has four elements
and $Y$ has three, we also know that $\Phi= M_X^3*M_\Phi*M_Y^2$.
\end{remark}

\begin{exercise} %
\label{exc.cost_matrix_mult}%
\index{quantale}
Calculate $M_X^3*M_\Phi*M_Y^2$, remembering to do matrix multiplication according to the $(\min,+)$-formula for matrix multiplication in the quantale $\Cost$; see \cref{eqn.quantale_matrix_mult}.

Your answer should agree with what you got in \cref{exc.distance_matrix_Phi}; does it?
\end{exercise}

\subsection{Back to co-design diagrams}%
\index{co-design!diagram}

Each box in a co-design diagram has a left-hand and a right-hand side, which in turn consist of a collection of ports, which in turn are labeled by preorders. For example, consider the chassis box below:
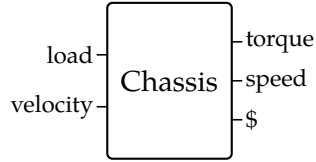
\[
\begin{tikzpicture}[oriented WD]
  \node[bb port sep=2,bb min width=4.3em, bb={2}{3}, below right=-3 and 1 of Sigma1.south east] (C) {Chassis};
  \draw[label, font=\footnotesize]
  	node[left=1pt of C_in1] {load}
  	node[left=1pt of C_in2] {velocity}
  	node[right=1pt of C_out1] {torque}
  	node[right=1pt of C_out2] {speed}
  	node[right=1pt of C_out3] {\$}
  ;
\end{tikzpicture}
\]
Its left side consists of two ports---one for load and one for velocity---and these are the functionality that the chassis produces. Its right side consists of three ports---one for torque, one for speed, and one for \$---and these are the resources that the chassis
requires. Each of these resources is to be taken as a preorder. For example, load might be
the preorder $([0,\infty],\leq)$, where an element $x\in[0,\infty]$ represents the
idea ``I can handle any load up to $x$.,'' while \$ might be the two-element preorder
$\{\texttt{up\_to\_\$100}, \;\texttt{more\_than\_\$100}\}$, where the first
element of this set is less than the second.

We then multiply---i.e.\ we take the product preorder---of all preorders on the left,
and similarly for those on the right. The box then represents a feasibility
relation between the results. For example, the chassis box above represents a
feasibility relation
\[
\textrm{Chassis}\colon \big(\textrm{load}\times\textrm{velocity}\big) \tickar
\big(\textrm{torque} \times \textrm{speed} \times \$\big)\\
\]


Let's walk through this a bit more concretely. Consider the design problem of
filming a movie, where you must pit the tone and entertainment value against the
cost. A feasibility relation describing this situation details what tone and
entertainment value can be obtained at each cost; as such, it is described by a
feasibility relation $\Phi\colon (T\times E)\tickar\$$. We represent this by
the box
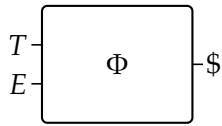
\[

\]
where $T$, $E$, and \$ are the preorders drawn below:
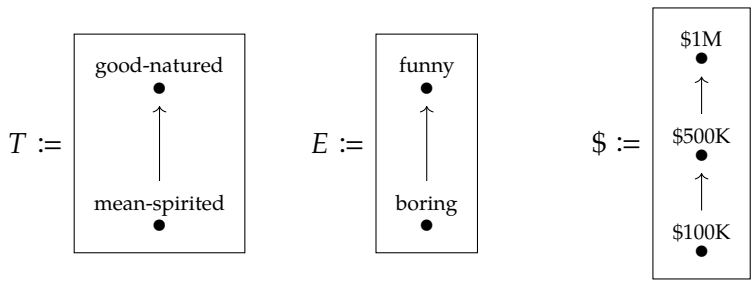
\[
%
\]
A possible feasibility relation is then described by the profunctor
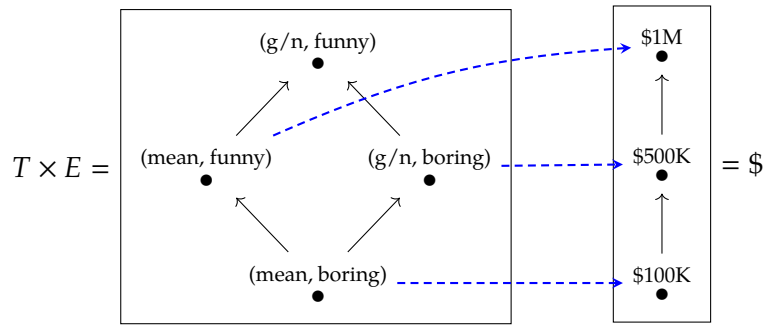
\[
%
\]
This says, for example, that a good-natured but boring movie costs \$500K to
produce (of course, the producers would also be happy to get \$1M).

To elaborate, each arrow in the above diagram is to be interpreted as saying,
``I can provide the source given the target''. For example, there are arrows
witnessing each of ``I can provide \$500K given \$1M'', ``I can provide a
good-natured but boring movie given \$500K'', and ``I can provide a mean and
boring movie given a good-natured but boring movie''.  Moreover, this
relationship is transitive, so the path from (mean, boring) to \$1M indicates
also that ``I can provide a mean and boring movie given \$1M''.  

Note the similarity and difference with the bridge interpretation of profunctors
in \cref{ex.bool_profunctor_bridges}: the arrows still indicate the possibility
of moving between source and target, but in this co-design driven interpretation
we understand them as indicating that it is possible to get \emph{to} the source
\emph{from} the target.

\begin{exercise} %
\label{exc.bad_humour}
In the above diagram, the node (g/n, funny) has no dashed blue arrow emerging
from it. Is this valid? If so, what does it mean?
\end{exercise}

\index{feasibility relation|)}

\section{Categories of profunctors}%
\index{profunctors!category of|(}

There is a category $\Feas$ whose objects are preorders and whose morphisms are
feasibility relations. In order to describe it, we must give the composition
formula and the identities, and prove that they satisfy the properties of being
a category: unitality and associativity.%
\index{associativity}%
\index{unitality}

\subsection{Composing profunctors}%
\index{profunctors!composition of}
If feasibility relations are to be morphisms, we need to give a formula for
composing two of them in series. Imagine you have cities $\cat{P}$, $\cat{Q}$,
and $\cat{R}$ and you have bridges---and hence feasibility matrices---connecting
these cities, say $\Phi\colon\cat{P}\tickar\cat{Q}$ and
$\Psi\colon\cat{Q}\tickar\cat{R}$. 
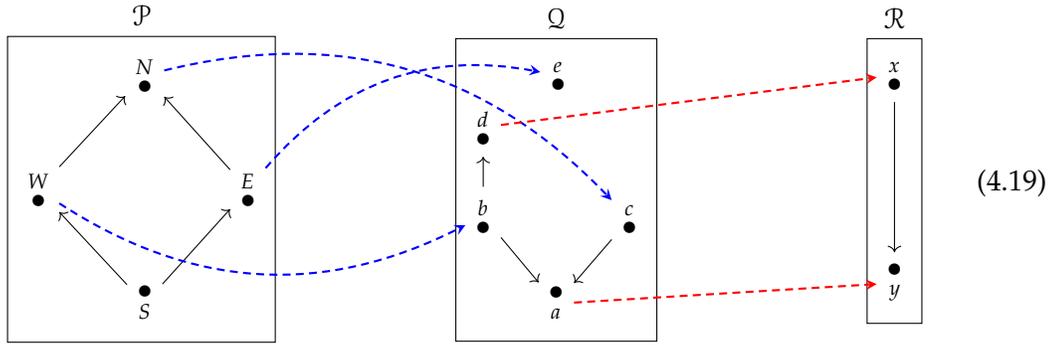
\begin{equation}%
\label{eqn.comp_profunctors}
\begin{tikzpicture}[baseline=(Y)]
	\node (mid) {};
	\node[above=of mid] (XN) {$\LMO{N}$};
	\node[left=of mid] (XW) {$\LMO{W}$};
	\node[right=of mid] (XE) {$\LMO{E}$};
	\node[below=of mid] (XS) {$\LMO[under]{S}$};
	\node[draw, fit=(XN) (XW) (XS) (XE)] (X) {};
	\draw[->]
		(XS) edge (XW) edge (XE) 
		(XW) edge (XN) 
		(XE) edge (XN)
	;
	\node[above=0 of X, font=\normalsize] {$\cat{P}$};
	\node[right=5 of XS] (Ya) {$\LMO[under]{a}$};
	\node[above left=.4 and .5 of Ya] (Yb) {$\LMO{b}$};
	\node[above right=.4 and .5 of Ya] (Yc) {$\LMO{c}$};
	\node[above=.4 of Yb] (Yd) {$\LMO{d}$};
	\node[right=5 of XN] (Ye) {$\LMO{e}$};
	\node[draw, fit=(Ye) (Yd) (Yc) (Ya)] (Y) {};
	\draw[->]
		(Yb) edge (Ya) edge (Yd)
		(Yc) edge (Ya)
	;
	\node[above=0 of Y, font=\normalsize] {$\cat{Q}$};
	\node[right=4 of Ye] (Zx) {$\LMO{x}$};
	\node[below=2 of Zx] (Zy) {$\LMO[under]{y}$};
	\node[draw, fit=(Zx) (Zy)] (Z) {};
	\draw[->] (Zx) -- (Zy);
	\node[above=0 of Z, font=\normalsize] {$\cat{R}$};
	\draw[mapsto, bend right] (XW) to (Yb);
	\draw[mapsto, bend left] (XN) to (Yc);
	\draw[mapsto, bend left] (XE) to (Ye);
	\draw[mapsto, red] (Ya) to (Zy);
	\draw[mapsto, red] (Yd) to (Zx);
\end{tikzpicture}
\end{equation}
The feasibility matrices for $\Phi$ (in blue) and $\Psi$ (in red) are:
\[
\begin{array}{c|ccccc}
	\Phi&a&b&c&d&e\\\hline
	N & \true& \false& \true& \false& \false\\ 
	E & \true& \false& \true& \false& \true\\
	W & \true& \true & \true& \true & \false\\ 
	S & \true& \true & \true& \true & \true
\end{array}
\hspace{1in}
\begin{array}{c|cc}
	\Psi&x&y\\\hline
	a & \false& \true\\
	b & \true & \true\\
	c & \false& \true\\
	d & \true & \true\\
	e & \false&\false
\end{array}
\]
As in \cref{rem.personify_navigator}, we personify a quantale as a
navigator.%
\index{navigator} So
imagine a navigator is trying to give a feasibility matrix $\Phi\cp\Psi$ for
getting from $\cat{P}$ to $\cat{R}$. How should this be done? Basically, for
every pair $p\in\cat{P}$ and $r\in\cat{R}$, the navigator searches through
$\cat{Q}$ for a way-point $q$, somewhere both to which we can get from $p$ AND
from which we can get to $r$. It is true that we can navigate from $p$ to $r$ iff there is a way-point $q$ through which to travel; this is a big OR over all possible $q$. The
composition formula is thus:
\begin{equation}%
\label{eqn.comp_V_prof}
  (\Phi\cp\Psi)(p,r)\coloneqq\bigvee_{q\in\cat{Q}}\Phi(p,q)\wedge\Psi(q,r).
\end{equation}
But as we said in \cref{eqn.quantale_matrix_mult}, this can be thought of as matrix multiplication. In our
example, the result is
\[
\begin{array}{c|cc}
	\Phi\cp\Psi&x&y\\\hline
	N & \false& \true\\
	E & \false& \true\\
	W & \true & \true\\
	S & \true & \true
\end{array}
\]
and one can check that this answers the question, ``can you get from here to there'' in \cref{eqn.comp_profunctors}: you can't get from $N$ to $x$ but you can get from $N$ to $y$.

The formula \eqref{eqn.comp_V_prof} is written in terms of the quantale $\Bool$,
but it works for arbitrary (unital commutative) quantales. We give the following
definition.

\begin{definition}%
\label{def.composite_profunctor}%
\index{profunctors!composition of}
Let $\cat{V}$ be a quantale, let $\cat{X}$, $\cat{Y}$, and
$\cat{Z}$ be $\cat{V}$-categories, and let $\Phi\colon\cat{X}\tickar\cat{Y}$ and
$\Psi\colon\cat{Y}\tickar\cat{Z}$ be $\cat{V}$-profunctors. We define their
\emph{composite}, denoted $\Phi\cp\Psi\colon\cat{X}\tickar\cat{Z}$ by the formula
\[(\Phi\cp\Psi)(p,r)=\bigvee_{q\in Q}\big(\Phi(p,q)\otimes\Psi(q,r)\big).\]
\end{definition}

\begin{exercise}%
\label{exc.compose_Lawv_profs}
Consider the $\Cost$-profunctors $\Phi\colon\cat{X}\tickar\cat{Y}$ and $\Psi\colon\cat{Y}\tickar\cat{Z}$ shown below:
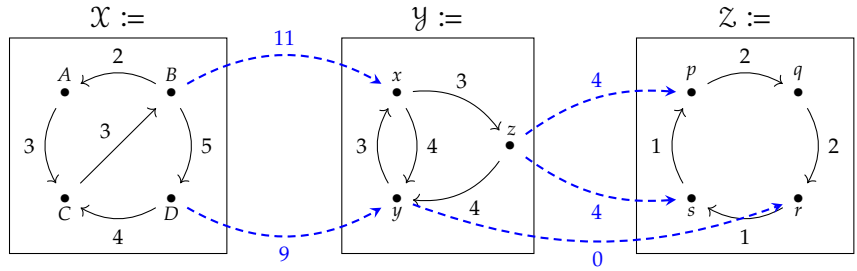
\[

\]
Fill in the matrix for the composite profunctor:
\[
\begin{array}{c|cccc}
  \Phi\cp\Psi&p&q&r&s\\\hline
  A&\?&24&\?&\?\\
  B&\?&\?&\?&\?\\
  C&\?&\?&\?&\?\\
  D&\?&\?&9&\?
\end{array}
\qedhere
\]
\end{exercise}

\subsection{The categories $\cat{V}\textrm{-}\Cat{Prof}$ and $\Feas$}

A composition rule suggests a category, and there is indeed a category where
the objects are $\Bool$-categories and the morphisms are $\Bool$-profunctors. To
make this work more generally, however, we need to add one technical condition.

Recall from \cref{rem.partial_orders} that a preorder is a skeletal preorder if
whenever $x \le y$ and $y \le x$, we have $x=y$. Skeletal preorders are also
known as posets. We say a quantale is skeletal if its underlying preorder is
skeletal; $\Bool$ and $\Cost$ are skeletal quantales.%
\index{partial order}

\begin{theorem}%
\label{thm.quantale_prof}%
\index{profunctors!category of}
For any skeletal quantale $\cat{V}$, there is a category $\Prof_{\cat{V}}$ whose
objects are $\cat{V}$-categories $\cat{X}$, whose morphisms are
$\cat{V}$-profunctors $\cat{X}\tickar\cat{Y}$, and with composition defined as in \cref{def.composite_profunctor}.
\end{theorem}

\begin{definition}%
\index{feasibility relation!as $\Bool$-profunctor}
We define $\Feas\coloneqq\Prof_{\Bool}$.
\end{definition}

At this point perhaps you have two questions in mind. What are the identity
morphisms? And why did we need to specialize to skeletal quantales? It turns out
these two questions are closely related. 

Define the \emph{unit profunctor}%
\index{profunctor!unit}%
\index{unit!profunctor} $\Unit{\cat{X}}\colon
\cat{X} \tickar \cat{X}$ on a $\cat{V}$-category $\cat{X}$ by the formula 
\begin{equation}%
\label{eqn.unit_profunctor}
	\Unit{\cat{X}}(x,y)\coloneqq\cat{X}(x,y).
\end{equation}%
\index{identity!profunctor|seealso {unit}}
How do we interpret this? Recall that, by \cref{def.cat_enriched_mpos},
$\cat{X}$ already assigns to each pair of elements $x,y\in\cat{X}$ an hom-object
$\cat{X}(x,y)\in\cat{V}$. The unit profunctor $\Unit{\cat{X}}$ just assigns each
pair $(x,y)$ that same object.%
\index{hom object}

In the $\Bool$ case the unit profunctor on some preorder $\cat{X}$ can be drawn like this:
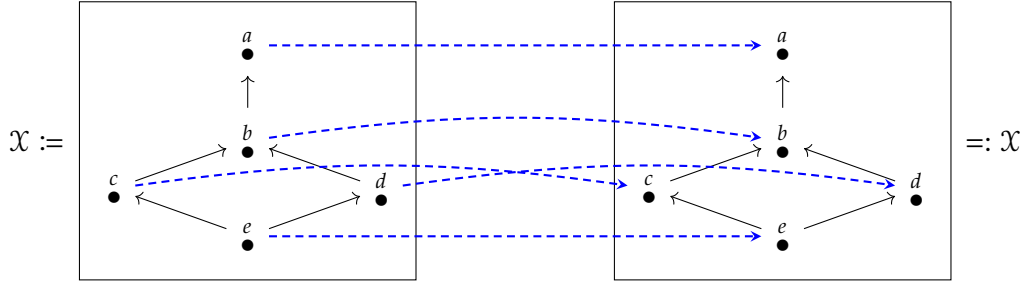
\[
\begin{tikzpicture}[x=.7in, y=.25in, inner sep=5pt]
	\node (a) at (0,5) {$\LMO{a}$};
	\node (b) at (0,3) {$\LMO{b}$};
	\node (c) at (-1,2) {$\LMO{c}$};
	\node (d) at (1,2) {$\LMO{d}$};
	\node (e) at (0,1) {$\LMO{e}$};
	\node[draw, fit=(a) (b) (c) (d) (e)] (A) {};
	\node[left=0 of A, font=\normalsize] {$\cat{X}\coloneqq$};
	\draw[->] (b) to (a);
	\draw[->] (c) to (b);
	\draw[->] (d) to (b);
	\draw[->] (e) to (c);
	\draw[->] (e) to (d);
	\node (a') at (0+4,5) {$\LMO{a}$};
	\node (b') at (0+4,3) {$\LMO{b}$};
	\node (c') at (-1+4,2) {$\LMO{c}$};
	\node (d') at (1+4,2) {$\LMO{d}$};
	\node (e') at (0+4,1) {$\LMO{e}$};
	\node[draw, fit=(a') (b') (c') (d') (e')] (A') {};
	\draw[->] (b') to (a');
	\draw[->] (c') to (b');
	\draw[->] (d') to (b');
	\draw[->] (e') to (c');
	\draw[->] (e') to (d');
	\node[right=0 of A', font=\normalsize] {$=:\cat{X}$};

\begin{scope}[mapsto]
	\draw (a) to (a');
	\draw[bend left=8pt] (b) to (b');
	\draw[bend left=8pt] (c) to (c');
	\draw[bend left=8pt] (d) to (d');
	\draw (e) to (e');
\end{scope}
\end{tikzpicture}
\]
Obviously, composing a feasibility relation with with the unit leaves it
unchanged; this is the content of \cref{lemma:unital_serial}.

\begin{exercise} %
\label{exc.draw_a_bridge}
Choose a not-too-simple $\Cost$-category $\cat{X}$. Give a bridge-style diagram for the unit profunctor $U_{\cat{X}}\colon\cat{X}\tickar\cat{X}$.
\end{exercise}

\begin{lemma}%
\label{lemma:unital_serial}
Composing any profunctor $\Phi\colon\cat{P}\to\cat{Q}$ with either unit profunctor, $\Unit{\cat{P}}$ or $\Unit{\cat{Q}}$, returns $\Phi$:
\[\Unit{\cat{P}}\cp\Phi=\Phi=\Phi\cp\Unit{\cat{Q}}\]
\end{lemma}
\begin{proof}
We show that $\Unit{\cat{P}}\cp\Phi=\Phi$ holds; proving
$\Phi=\Phi\cp\Unit{\cat{Q}}$ is similar. Fix $p\in P$ and $q\in Q$. Since
$\cat{V}$ is skeletal, to prove the equality it's enough to show
$\Phi\leq\Unit{\cat{P}}\cp\Phi$ and $\Unit{\cat{P}}\cp\Phi\leq\Phi$. We have one
direction:
\begin{equation}%
\label{eqn.direction_rand597}
	\Phi(p,q)= I\otimes\Phi(p,q)\leq\cat{P}(p,p)\otimes\Phi(p,q)\leq\bigvee_{p_1\in P}\big(\cat{P}(p,p_1)\otimes\Phi(p_1,q)\big)=(\Unit{\cat{P}}\cp\Phi)(p,q).
\end{equation}
For the other direction, we must show $\bigvee_{p_1\in P}\big(\cat{P}(p,p_1)\otimes\Phi(p_1,q)\big)\leq\Phi(p,q)$. But by definition of join, this holds iff $\cat{P}(p,p_1)\otimes\Phi(p_1,q)\leq\Phi(p,q)$ is true for each $p_1\in\cat{P}$. This follows from \cref{def.cat_enriched_mpos,def.enriched_profunctor}:
\begin{equation}%
\label{eqn.rand_16749}
  \cat{P}(p,p_1)\otimes\Phi(p_1,q)=\cat{P}(p,p_1)\otimes\Phi(p_1,q)\otimes I\leq\cat{P}(p,p_1)\otimes\Phi(p_1,q)\otimes\cat{Q}(q,q)\leq\Phi(p,q).
\end{equation}
\end{proof}

\begin{exercise} %
\label{exc.prof_unitality}
\begin{enumerate}
  \item Justify each of the four steps $(=, \leq, \leq, =)$ in
  \cref{eqn.direction_rand597}.
  \item In the case $\cat{V}=\Bool$, we can directly show each of the four steps
  in \cref{eqn.direction_rand597} is actually an equality. How? 
  \item Justify each of the three steps $(=,\leq,\leq)$ in \cref{eqn.rand_16749}.
\qedhere
\end{enumerate}
\end{exercise}

Composition of profunctors is also associative; we leave the proof to you. 

\begin{lemma}%
\label{lemma:assoc_serial}%
\index{associativity!of profunctor composition}
Serial composition of profunctors is associative. That is, given profunctors $\Phi\colon\cat{P}\to\cat{Q}$, $\Psi\colon\cat{Q}\to\cat{R}$, and $\Upsilon\colon\cat{R}\to\cat{S}$, we have
\[(\Phi\cp\Psi)\cp\Upsilon=\Phi\cp(\Psi\cp\Upsilon).\]
\end{lemma}

\begin{exercise} %
\label{exc.prof_associativity}
Prove \cref{lemma:assoc_serial}. (Hint: remember to use the fact that $\cat{V}$
is skeletal.)
\end{exercise} 

So, feasibility relations form a category. Since this is the case, we
can describe feasibility relations using wiring diagrams for categories (see also \cref{subsec.reflection_wds}), which are very simple. Indeed, each box can only have
one input and one output, and they're connected in a line:
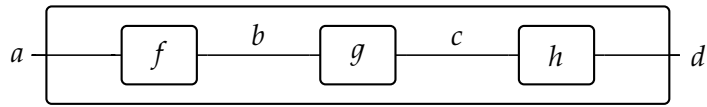
\[

\]
On the other hand, we have seen that feasibility relations are the building
blocks of co-design problems, and we know that co-design problems can be
depicted with a much richer wiring diagram, for example:
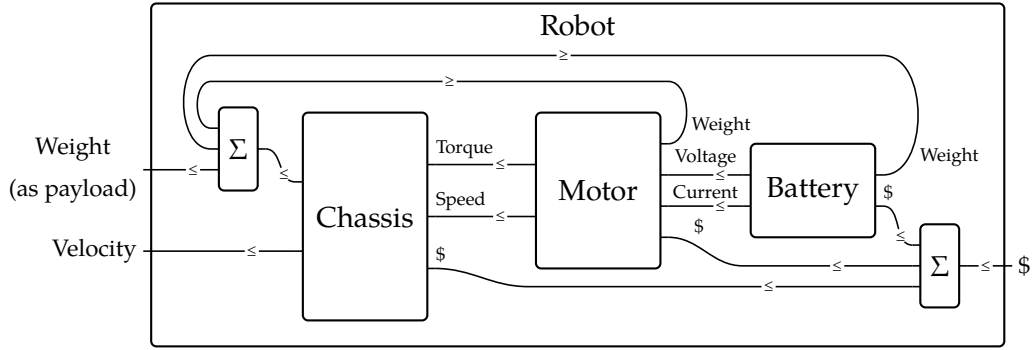
\[
%
\]
This hints that the category $\Feas$ has more structure. We've seen wiring
diagrams where boxes can have multiple inputs and outputs before, in
\cref{chap.resource_theory}; there they depicted morphisms in a monoidal
preorder. On other hand the boxes in the wiring diagrams of
\cref{chap.resource_theory} could not have distinct labels, like the boxes in a
co-design problem: all boxes in a wiring diagram for monoidal preorders indicate
the order $\le$, while above we see boxes labelled by ``Chassis'', ``Motor'',
and so on. Similarly, we know that $\Feas$ is a proper category, not just a
preorder. To understand these diagrams then, we must introduce a new structure,
called a \emph{monoidal category}. A monoidal category is a \emph{categorified}
monoidal preorder.

\begin{remark}%
\index{categorification}
While we have chosen to define $\Prof_\cat{V}$ only for skeletal quantales in
\cref{thm.quantale_prof}, it is not too hard to work with non-skeletal ones.
There are two straightforward ways to do this. First, we might let the morphisms
of $\Prof_\cat{V}$ be isomorphism classes of $\cat{V}$-profunctors. This is
analogous to the trick we will use when defining the category $\cospan{\cat{C}}$
in \cref{def.cospan_sym_mon_cat}.  Second, we might relax what we mean by
category, only requiring composition to be unital and associative `up to
isomorphism'. This is also a type of categorification, known as bicategory theory.%
\index{bicategory}%
\index{associativity!weak}%
\index{unitality!weak}
\end{remark}

In the next section we'll discuss categorification and introduce monoidal
categories. First though, we finish this section by discussing why profunctors
are called profunctors, and by formally introducing something called the \emph{collage} of a profunctor.

\index{profunctors!category of|)}

\subsection{Fun profunctor facts: companions, conjoints, collages}

\paragraph{Companions and conjoints.}
Recall that a preorder is a $\Bool$-category and a monotone map is a $\Bool$-functor. We said above that a profunctor is a generalization of a functor; how so?

In fact, every $\cat{V}$-functor gives rise to two $\cat{V}$-profunctors, called the companion and the conjoint.

\begin{definition}%
\label{def.companion_conjoint}%
\index{companion profunctor}%
\index{conjoint profunctor}
Let $F\colon\cat{P}\to\cat{Q}$ be a $\cat{V}$-functor. The \emph{companion of
$F$}, denoted $\comp{F}\colon\cat{P}\tickar\cat{Q}$ and the \emph{conjoint of
$F$}, denoted $\conj{F}\colon\cat{Q}\tickar\cat{P}$ are defined to be the
following $\cat{V}$-profunctors:
\[
  \comp{F}(p,q)\coloneqq\cat{Q}(F(p),q)
  \quad\text{and}\quad
  \conj{F}(q,p)\coloneqq\cat{Q}(q,F(p))
\]
\end{definition}

Let's consider the $\Bool$ case again. One can think of a monotone map
$F\colon\cat{P}\to\cat{Q}$ as a bunch of arrows, one coming out of each vertex
$p\in P$ and landing at some vertex $F(p)\in Q$.
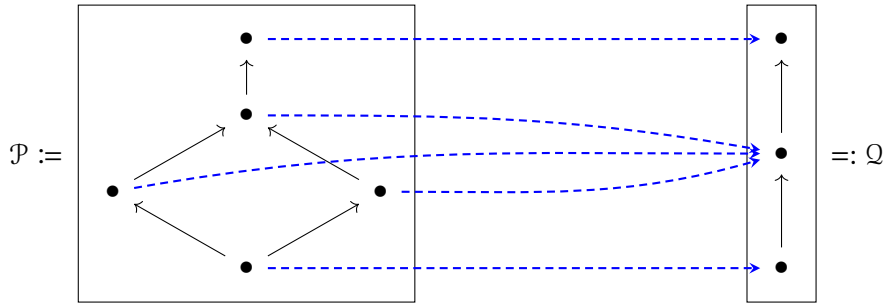
\[
\begin{tikzpicture}[x=.7in, y=.4in, inner sep=5pt]
  	\node (a) at (0,4) {$\bullet$};
  	\node (b) at (0,3) {$\bullet$};
  	\node (c) at (-1,2) {$\bullet$};
  	\node (d) at (1,2) {$\bullet$};
  	\node (e) at (0,1) {$\bullet$};
  	\node[draw, fit=(a) (b) (c) (d) (e)] (A) {};
  	\node[left=0 of A, font=\normalsize] {$\cat{P}\coloneqq$};
  	\draw[->] (b) to (a);
  	\draw[->] (c) to (b);
  	\draw[->] (d) to (b);
  	\draw[->] (e) to (c);
  	\draw[->] (e) to (d);
  	\node (B0) at (4,4) {$\bullet$};
  	\node (B1) at (4,2.5) {$\bullet$};
  	\node (B2) at (4,1) {$\bullet$};
  	\draw[->] (B2) to (B1);
  	\draw[->] (B1) to (B0);
  	\node[draw, fit=(B0) (B1) (B2)] (B) {};
  	\node[right=0 of B, font=\normalsize] {$=:\cat{Q}$};
	\begin{scope}[mapsto]
    \draw (a) -- (B0);
  	\draw (b) to[out=0,in=170] (B1);
  	\draw (c) to[out=10,in=180] (B1);
  	\draw (d) to[out=0,in=196] (B1);
  	\draw (e) -- (B2);
	\end{scope}
\end{tikzpicture}
\]
This looks like the pictures of bridges connecting cities, and if one regards
the above picture in that light, they are seeing the companion $\comp{F}$. But
now mentally reverse every dotted arrow, and the result would be bridges
$\cat{Q}$ to $\cat{P}$. This is a profunctor $\cat{Q} \tickar \cat{P}$! We call
it $\conj{F}$.

\begin{example}%
\label{ex.unit_profunctor}%
\index{profunctor!unit}
For any preorder $\cat{P}$, there is an identity functor
$\id\colon\cat{P}\to\cat{P}$. Its companion and conjoint agree
$\comp{\id}=\conj{\id}\colon\cat{P}\tickar\cat{P}$. The resulting profunctor is
in fact the unit profunctor, $\Unit{\cat{P}}$, as defined in \cref{eqn.unit_profunctor}.
\end{example}

\begin{exercise} %
\label{exc.unit_companion}
Check that the companion $\comp{\id}$ of $\id\colon\cat{P}\to\cat{P}$ really has
the unit profunctor formula given in \cref{eqn.unit_profunctor}.
\end{exercise}

\label{exc.comp_conj_unit}
%

\begin{example}%
\label{ex.plus_3}
Consider the function $+\colon\RR\times\RR\times\RR\to\RR$, sending a triple $(a,b,c)$ of real numbers to $a+b+c\in\RR$. This function is monotonic, because if $(a,b,c)\leq(a',b',c')$---i.e.\ if $a\leq a'$ and $b\leq b'$, and $c\leq c'$---then obviously $a+b+c\leq a'+b'+c'$. Thus it has a companion and a conjoint.

Its companion $\comp{+}\colon(\RR\times\RR\times\RR)\tickar\RR$ is the function that sends $(a,b,c,d)$ to $\true$ if $a+b+c\leq d$ and to $\false$ otherwise.
\end{example}

\begin{exercise} %
\label{exc.plus_conjoint}
	Let $+\colon\RR\times\RR\times\RR\to\RR$ be as in \cref{ex.plus_3}. What is its conjoint $\conj{+}$?
\end{exercise}

\begin{remark}[$\cat{V}$-Adjoints]%
\index{adjunction!relationship to companions
  and conjoints}
Recall from \cref{def.galois} the definition of Galois connection between preorders $\cat{P}$ and $\cat{Q}$. The definition of adjoint can be extended from the $\Bool$-enriched setting (of preorders and monotone maps) to the $\cat{V}$-enriched setting for arbitrary monoidal preorders $\cat{V}$. In that case, the definition of a $\cat{V}$-adjunction is a pair of $\cat{V}$-functors $F\colon\cat{P}\to\cat{Q}$ and $G\colon\cat{Q}\to\cat{P}$ such that the following holds for all $p\in P$ and $q\in Q$.
\begin{equation}%
\label{eqn.adjoint_V_funs}
	\cat{P}(p,G(q))\cong\cat{Q}(F(p),q)
\end{equation}
\end{remark}

\begin{exercise} %
\label{exc.adjoint_comp_conj}
Let $\cat{V}$ be a skeletal quantale, let $\cat{P}$ and $\cat{Q}$ be $\cat{V}$-categories, and let $F\colon\cat{P}\to\cat{Q}$ and $G\colon\cat{Q}\to\cat{P}$ be $\cat{V}$-functors.
\begin{enumerate}
	\item Show that $F$ and $G$ are $\cat{V}$-adjoints (as in
	\cref{eqn.adjoint_V_funs}) if and only if the companion of the former
	equals the conjoint of the latter: $\comp{F}=\conj{G}$.
	\item Use this to prove that $\comp{\id}=\conj{\id}$, as was stated in \cref{ex.unit_profunctor}.
	\qedhere
\qedhere
\end{enumerate}
\end{exercise}

\paragraph{Collage of a profunctor.}%
\index{collage|(}%
\index{profunctor!collage of|see {collage}}

We have been drawing profunctors as bridges connecting cities. One may get an inkling that given a $\cat{V}$-profunctor $\Phi\colon X\tickar Y$ between $\cat{V}$-categories $\cat{X}$ and $\cat{Y}$, we have turned $\Phi$ into a some sort of new $\cat{V}$-category that has $\cat{X}$ on the left and $\cat{Y}$ on the right. This works for any $\cat{V}$ and profunctor $\Phi$, and is called the collage construction.
\index{profunctor!as bridges}

\begin{definition}%
\label{def.collage_prof}%
\index{collage}
	Let $\cat{V}$ be a quantale, let $\cat{X}$ and $\cat{Y}$ be $\cat{V}$-categories, and let $\Phi\colon\cat{X}\tickar\cat{Y}$ be a $\cat{V}$-profunctor. The \emph{collage of $\Phi$}, denoted $\Cat{Col}(\Phi)$ is the $\cat{V}$-category defined as follows:
	\begin{enumerate}[label=(\roman*)]
		\item $\Ob(\Cat{Col}(\Phi))\coloneqq\Ob(\cat{X})\sqcup\Ob(\cat{Y})$;
		\item For any $a,b\in\Ob(\Cat{Col}(\Phi))$, define $\Cat{Col}(\Phi)(a,b)\in\cat{V}$ to be
		\[
			\Cat{Col}(\Phi)(a,b)\coloneqq
			\begin{cases}
				\cat{X}(a,b)&\tn{ if }a,b\in\cat{X}\\
				\Phi(a,b)&\tn{ if }a\in\cat{X},b\in\cat{Y}\\
				\varnothing&\tn{ if }a\in\cat{Y},b\in\cat{X}\\
				\cat{Y}(a,b)&\tn{ if }a,b\in\cat{Y}
			\end{cases}
		\]
	\end{enumerate}
There are obvious functors $i_\cat{X}\colon\cat{X}\to\Cat{Col}(\Phi)$ and $i_\cat{Y}\colon\cat{Y}\to\Cat{Col}(\Phi)$, sending each object and morphism to ``itself,'' called  \emph{collage inclusions}.
\end{definition}

Some pictures will help clarify this.

\begin{example}%
\index{Hasse diagram!for profunctors}
Consider the following picture of a $\Cost$-profunctor $\Phi\colon\cat{X}\tickar\cat{Y}$:
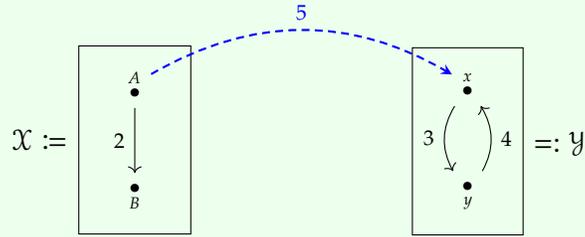
\[

\]
It corresponds to the following matrices
\[
\begin{array}{c|cc}
\cat{X}&A&B\\\hline
A&0&2\\
B&\infty&0
\end{array}
\hspace{.5in}
\begin{array}{c|cc}
\Phi&x&y\\\hline
A&5&8\\
B&\infty&\infty
\end{array}
\hspace{.5in}
\begin{array}{c|cc}
\cat{Y}&x&y\\\hline
x&0&3\\
y&4&0
\end{array}
\]
A generalized Hasse diagram of the collage can be obtained by simply taking the
union of the Hasse diagrams for $\cat{X}$ and $\cat{Y}$, and adding in the
bridges as arrows. Given the above profunctor $\Phi$, we draw the Hasse diagram
for $\Cat{Col}(\Phi)$ below left, and the $\Cost$-matrix representation of the
resulting $\Cost$-category on the right:
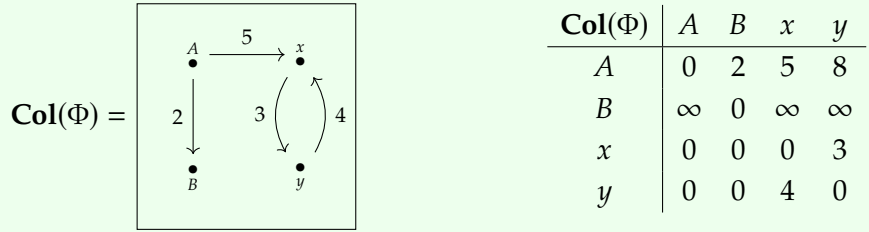
\[

\hspace{1in}
\begin{array}{c|cccc}
  \Cat{Col}(\Phi)&A&B&x&y\\\hline
  A&0&2&5&8\\
  B&\infty&0&\infty&\infty\\
  x&0&0&0&3\\
  y&0&0&4&0
\end{array}
\qedhere
\]
\end{example}

\begin{exercise}%
\label{exc.collage_practice}
Draw a Hasse diagram for the collage of the profunctor shown here:
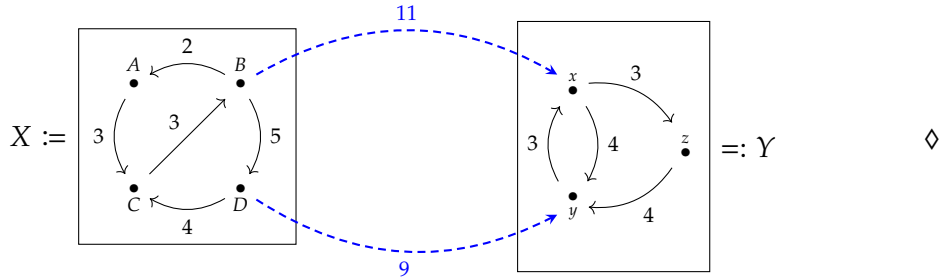
\[
%
\qedhere
\]
\end{exercise}

\index{profunctor|)}%
\index{collage|)}

\section{Categorification}%
\label{sec.categorification}%
\index{categorification|(}

Here we switch gears, to discuss a general concept called \emph{categorification}. We will begin again with the basics, categorifying several of the notions we've encountered already. The goal is to define compact closed categories and their feedback-style wiring diagrams. At that point we will return to the story of co-design, and $\cat{V}$-profunctors in general, and show that they do in fact form a compact closed category, and thus interpret the diagrams we've been drawing since \cref{eqn.chassis}. 

\subsection{The basic idea of categorification}

The general idea of categorification is that we take a thing we know and
add structure to it, so that what were formerly \emph{properties} become
\emph{structures}. We do this in such a way that we can recover the thing we
categorified by forgetting this new structure. This is rather vague; let's give
an example.%
\index{categorification}

Basic arithmetic concerns properties of the natural numbers $\nn$, such as the
fact that $5+3=8$. One way to categorify $\nn$ is to use the category
$\finset$ of finite sets and functions. To obtain a categorification, we replace
the brute $5$, $3$, and $8$ with sets of that many elements, say
$\ol{5}=\{\textrm{apple}, \textrm{banana}, \textrm{cherry},
\textrm{dragonfruit}, \textrm{elephant}\}$, $\ol{3}=\{\textrm{apple},
\textrm{tomato}, \textrm{cantaloupe}\}$, and $\ol{8}=\{\textrm{Ali},
\textrm{Bob}, \textrm{Carl}, \textrm{Deb}, \textrm{Eli}, \textrm{Fritz},
\textrm{Gem}, \textrm{Helen}\}$ respectively. We also replace $+$ with disjoint
union of sets $\sqcup$, and the brute property of equality with the structure of
an isomorphism. What makes this a good categorification is that, having made
these replacements, the analogue of $5+3=8$ is still true: $\ol{5}\sqcup \ol{3}
\cong \ol{8}$.%
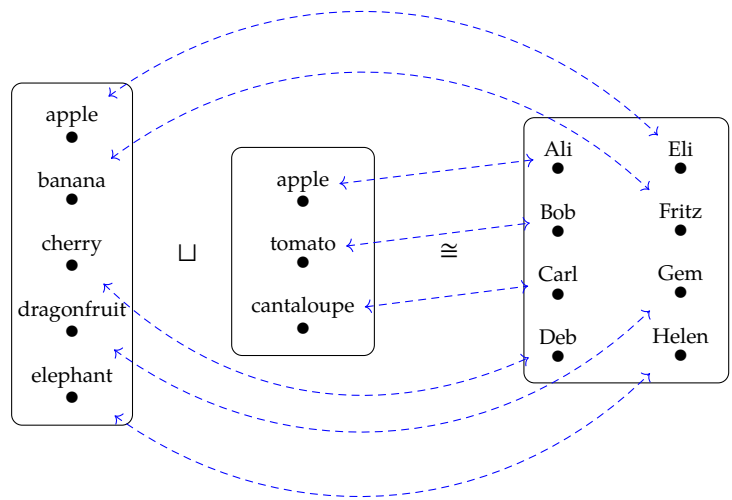
\index{union!disjoint}
\[
\begin{tikzpicture}[y=.3ex, rounded corners]
	\node (a1) {$\LTO{apple}$};
	\node [below=1 of a1] (a2) {$\LTO{banana}$};
	\node [below=1 of a2] (a3) {$\LTO{cherry}$};
	\node [below=1 of a3] (a4) {$\LTO{dragonfruit}$};
	\node [below=1 of a4] (a5) {$\LTO{elephant}$};
	\node [draw, fit=(a1) (a5)] (a) {};
	\node [right=2 of a2] (b1) {$\LTO{apple}$};
	\node [below=1 of b1] (b2) {$\LTO{tomato}$};
	\node [below=1 of b2] (b3) {$\LTO{cantaloupe}$};
	\node [draw, fit=(b1) (b3)] (b) {};	
	\node at ($(a)!.5!(b)$) {$\sqcup$};
%
%
	\node at ($(a1.south)!.5!(a2.north)$) (helper) {};
	\node [right=6 of helper] (y1) {$\LTO{Ali}$};
	\node [below=1 of y1] (y2) {$\LTO{Bob}$};
	\node [below=1 of y2] (y3) {$\LTO{Carl}$};
	\node [below=1 of y3] (y4) {$\LTO{Deb}$};
	\node [right=1 of y1] (y5) {$\LTO{Eli}$};
	\node [below=1 of y5] (y6) {$\LTO{Fritz}$};
	\node [below=1 of y6] (y7) {$\LTO{Gem}$};
	\node [below=1 of y7] (y8) {$\LTO{Helen}$};
	\node [draw, fit= (y1) (y8)] (y) {};
	\node at ($(b.east)!.5!(y.west)$) {$\cong$};
	\begin{scope}[densely dashed, blue, <->, bend left=40]
		\draw (a1) to (y5);
		\draw (a2) to (y6);
	\end{scope}
	\begin{scope}[densely dashed, blue, <->, bend right=34]
		\draw (a3) to (y4);
	\end{scope}
	\begin{scope}[densely dashed, blue, <->, bend right=40]
		\draw (a4) to (y7);
		\draw (a5) to (y8);
	\end{scope}
	\begin{scope}[densely dashed, blue, <->]
		\draw (b1) to (y1);
		\draw (b2) to (y2);
		\draw (b3) to (y3);
	\end{scope}
\end{tikzpicture}
\]
In this categorified world, we have more structure available to talk about the
relationships between objects, so we can be more precise about how they relate
to each other. Thus it's not the case that $\ol{5} \sqcup \ol{3}$ is
\emph{equal} to our chosen eight-element set $\ol{8}$, but more precisely that
there exists an invertible function comparing the two, showing that they are
isomorphic in the \emph{category} $\finset$.

Note that in the above construction we made a number of choices; here we must
beware. Choosing a good categorification---like designing a good algebraic
structure such as that of preorders or quantales---is part of the \emph{art} of
mathematics.  There is no prescribed way to categorify, and the success of a
chosen categorification is often empirical: its richer structure should allow us more
insights into the subject we want to model.

As another example, an empirically pleasing way to categorify preorders is to
categorify them as, well, categories.  In this case, rather than the brute
property ``there exists a morphism $a\to b$,'' denoted $a\leq b$ or
$\cat{P}(a,b)=\true$, we instead say ``here is a set of morphisms $a\to b$.'' We
get a hom-set rather than a hom-Boolean. In fact---to state this in a way
straight out of the primordial ooze---just as preorders are $\Bool$-categories,
ordinary categories are actually $\Cat{Set}$-categories.
\index{primordial ooze}

%
\index{categorification|)}

\subsection{A reflection on wiring diagrams}%
\label{subsec.reflection_wds}%
\index{wiring diagram|(}

Suppose we have a preorder. We introduced a very simple sort of wiring diagram in
\cref{ssec.wirdia2}. These allowed us to draw a box 
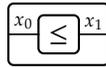
\[

\]
whenever $x_0 \le x_1$. Chaining these together, we could prove facts in our
preorder. For example
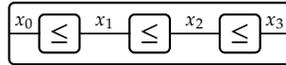
\[
%
\]
provides a proof that $x_0\leq x_3$ (the exterior box) using three facts (the
interior boxes), $x_0\leq x_1$, $x_1\leq x_2$, and $x_2\leq x_3$.

\index{wiring diagram!for categories}
As categorified preorders, categories have basically the same sort of wiring
diagram as preorders---namely sequences of boxes inside a box. But since we have
replaced the fact that $x_0 \le x_1$ with the structure of a \emph{set} of
morphisms, we need to be able to label our boxes with morphism names:
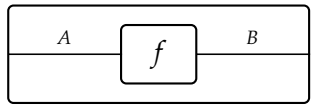
\[
%
\]
Suppose given additional morphisms $g\colon B\to C$, and $h\colon C\to D$.
Representing these each as boxes like we did for $f$, we might be tempted to
stick them together to form a new box:
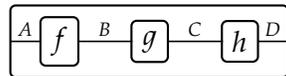
\[
%
\]
Ideally this would also be a morphism in our category: after all, we have said
that we can represent morphisms with boxes with one input and one output. But
wait, you say! We don't know which morphism it is. Is it $f\cp (g \cp h)$? Or
$(f \cp g) \cp h$? It's good that you are so careful. Luckily, we are saved by
the properties that a category must have. Associativity says $f\cp (g \cp h)=(f
\cp g) \cp h$, so it doesn't matter which way we chose to try to decode the box.

Similarly, the identity morphism on an object $x$ is drawn as on the left below,
but we will see that it is not harmful to draw $\id_x$ in any of the following
three ways:%
\index{identity!in wiring diagrams}
\[

\]
By \cref{def.category} the morphisms in a category satisfy two properties, called the unitality property and the associativity property. The unitality says that $\id_x\cp f=f=f\cp\id_y$ for any $f\colon x\to y$.%
\index{associativity}%
\index{unitality} In terms of diagrams this would say 
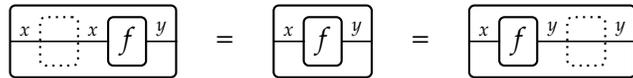
\[
%
\]
This means you can insert or discard any identity morphism you see in a wiring
diagram.  From this perspective, the coherence laws of a category---that is, the
associativity law and the unitality law---are precisely what are needed to
ensure we can lengthen and shorten wires without ambiguity.%
\index{associativity}%
\index{unitality}

In \cref{ssec.wirdia2}, we also saw wiring diagrams for monoidal preorders. Here we
were allowed to draw boxes which can have multiple typed inputs and outputs, but
with no choice of label (always $\leq$):
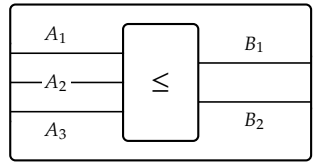
\[
%
\]
If we combine these ideas, we will obtain a categorification of symmetric monoidal preorders: symmetric monoidal categories. A
symmetric monoidal category is an algebraic structure in which we have
labelled boxes with multiple typed inputs and outputs:
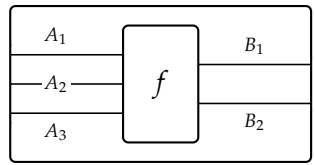
\[
%
\]
Furthermore, a symmetric monoidal category has a composition rule and a monoidal
product, which permit us to combine these boxes to interpret diagrams like this:
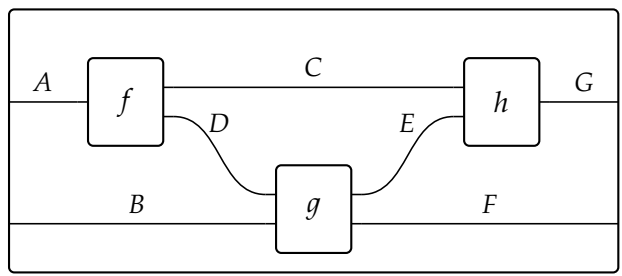
\[
%
\]
Finally, this structure must obey coherence laws, analogous to associativity and
unitality in categories, that allow such diagrams to be unambiguously
interpreted. In the next section we will be a bit more formal, but it is useful
to keep in mind that, when we say our data must be ``well behaved,'' this is all we
mean.
\index{wiring diagram!for monoidal categories}
\index{wiring diagram|)}%
\index{coherence!conditions}

\subsection{Monoidal categories}%
\index{monoidal category|(}

We defined $\cat{V}$-categories, for a symmetric monoidal preorder $\cat{V}$ in \cref{def.cat_enriched_mpos}. Just like preorders turned out to be special kinds of categories (see \cref{subsubsec.pos_free_spectrum}), monoidal preorders are special kinds of monoidal categories. And just like we can consider $\cat{V}$-categories for a monoidal preorder, we can also consider $\cat{V}$-categories when $\cat{V}$ is a monoidal category. This is another sort of categorification.

We will soon meet the monoidal category $(\smset,\{1\},\times)$. The monoidal
product will take two sets, $S$ and $T$, and return the set $S\times
T=\{(s,t)\mid s\in S, t\in T\}$. But whereas for monoidal preorders we had the
brute associative property $(p\otimes q)\otimes r=p\otimes (q\otimes r)$, the
corresponding idea in $\smset$ is not quite true:%
\index{associativity!property vs.\ structure}
\begin{align*}
& S \times (T \times U) :=\big\{\big(s,(t,u)\big)\,\big|\, s\in S, t\in T, u\in U\big\}
	\\
	\quad =^?\quad &
(S \times T) \times U:=\big\{\big((s,t),u\big)\,\big|\, s\in S, t\in T, u\in U\big\}.
\end{align*}
They are slightly different sets: the first contains pairs consisting of an
elements in $S$ and an element in $T\times U$, while the second contains pairs
consisting of an element in $S \times T$ and an element in $U$. The sets are not
equal, but they are clearly isomorphic, i.e.\ the difference between them is
``just a matter of bookkeeping.'' We thus need a structure---a bookkeeping
isomorphism---to keep track of the associativity:
\[
	\alpha_{s,t,u}\colon\{(s,(t,u))\mid s\in S, t\in T, u\in U\}
	\To{\cong}
	\{((s,t),u)\mid s\in S, t\in T, u\in U\}.
\]
\index{coherence!as bookkeeping}

There are a couple things to mention before we dive into these ideas. First,
just because you replace brute things and properties with structures, it does
not mean that you no longer have brute things and properties: new ones emerge!
Not only that, but second, the new brute stuff tends to be more complex than
what you started with. For example, above we replaced the associativity equation
with an isomorphism $\alpha_{s,t,u}$, but now we need a more complex property to
ensure that all these $\alpha$'s behave reasonably! The only way out of this morass is to
add infinitely much structure, which leads one to ``$\infty$-categories,'' but
we will not discuss that here.

Instead, we will continue with our categorification of monoidal preorders, starting
with a rough definition of symmetric monoidal categories. It's rough in the
sense that we suppress the technical bookkeeping, hiding it under the name
``well behaved.''

\begin{roughDef}%
\label{rdef.sym_mon_cat}%
\index{category!monoidal structure on}%
\index{product!monoidal}
Let $\cat{C}$ be a category. A \emph{symmetric monoidal structure} on $\cat{C}$ consists of the following constituents:
\begin{enumerate}[label=(\roman*)]
	\item an object $I\in\Ob(\cat{C})$ called the \emph{monoidal unit}, and
	\item a functor $\otimes\colon\cat{C}\times\cat{C}\to\cat{C}$, called the \emph{monoidal product}
\end{enumerate}
subject to well-behaved, natural isomorphisms
\begin{enumerate}[label=(\alph*)]
	\item $\lambda_c\colon I\otimes c\cong c$ for every $c\in\Ob(\cat{C})$,
	\item $\rho_c\colon c\otimes I\cong c$ for every $c\in\Ob(\cat{C})$,
	\item $\alpha_{c,d,e}\colon (c\otimes d)\otimes e\cong c\otimes(d\otimes e)$ for every $c,d,e\in\Ob(\cat{C})$, and
	\item $\sigma_{c,d}\colon c\otimes	d\cong d \otimes c$ for every $c,d\in\Ob(\cat{C})$, called the \emph{swap map}, such that $\sigma\circ\sigma=\id$.
\end{enumerate}
A category equipped with a symmetric monoidal structure is called a
\emph{symmetric monoidal category}.
\end{roughDef}

\begin{remark} %
\label{rem.strict_mon_cat}%
\index{coherence!Mac Lane's theorem}
If the isomorphisms in (a), (b), and (c)---but 
\emph{not} (d)---are replaced by equalities, then we say that the monoidal
structure is \emph{strict}, and this is a complete (non-rough) definition of
\emph{symmetric strict monoidal category}. In fact, symmetric strict monoidal
categories are almost the same thing as symmetric monoidal categories, via a
result known as Mac Lane's coherence theorem. An upshot of this theorem is that
we can, when useful to us, pretend that our monoidal categories are strict: for
example, we implicitly do this when we draw wiring diagrams. Ask your friendly
neighborhood category theorist to explain how! 
\end{remark}

\begin{remark}
For those yet to find a friendly expert category theorist, we make the following
remark. A complete (non-rough) definition of symmetric monoidal category is that
a symmetric monoidal category is a category equipped with an equivalence to
(the underlying category of) a symmetric strict monoidal category. This can be
unpacked, using \cref{rem.strict_mon_cat} and our comment about equivalence of
categories in \cref{rem.preorder_boolcats2}, but we don't expect you to do
so. Instead, we hope this gives you more incentive to ask a friendly expert
category theorist!
\end{remark}

\begin{exercise} %
\label{exc.mon_preorder_is_cat}%
\index{preorder!symmetric monoidal}
Check that monoidal categories indeed generalize monoidal preorders: a monoidal preorder is a
monoidal category $(\cat{P},I,\otimes)$ where, for every $p,q\in\cat{P}$, the
set $\cat{P}(p,q)$ has at most one element. 
\end{exercise}

\begin{example}%
\label{ex.set_as_mon_cat}
As we said above, there is a monoidal structure on $\Cat{Set}$ where the monoidal unit is some choice of singleton set, say $I\coloneqq\{1\}$, and the monoidal product is $\otimes\coloneqq\times$. What it means that $\times$ is a functor is that:
\begin{itemize}
	\item For any pair of objects, i.e.\ sets, $(S,T)\in\Ob(\Cat{Set}\times\Cat{Set})$, one obtains a set $(S\times T)\in\Ob(\Cat{Set})$. We know what it is: the set of pairs $\{(s,t)\mid s\in S, t\in T\}$.
	\item For any pair of morphisms, i.e.\ functions, $f\colon S\to S'$ and $g\colon T\to T'$, one obtains a function $(f\times g)\colon(S\times T)\to (S'\times T')$. It works pointwise: $(f\times g)(s,t)\coloneqq(f(s),g(t))$.
	\item These should preserve identities: $\id_S\times\id_T=\id_{S\times T}$ for any sets $S,T$.
	\item These should preserve composition: for any functions $S\To{f}S'\To{f'}S''$ and $T\To{g}T'\To{g'}T''$, one has
	\[
	(f\times g)\cp(f'\times g')=(f\cp g)\times(f'\cp g').
      \]
\end{itemize}

The four conditions, (a), (b), (c), and (d) give isomorphisms $\{1\}\times S\cong S$, etc. These maps are obvious in the case of $\Cat{Set}$, e.g. the function $\{(1,s)\mid s\in S\}\to S$ sending $(1,s)$ to $s$. We have been calling such things bookkeeping.
\end{example}


\begin{exercise} %
\label{exc.read_string_diag}
Consider the monoidal category $(\Cat{Set},1,\times)$, together with the diagram
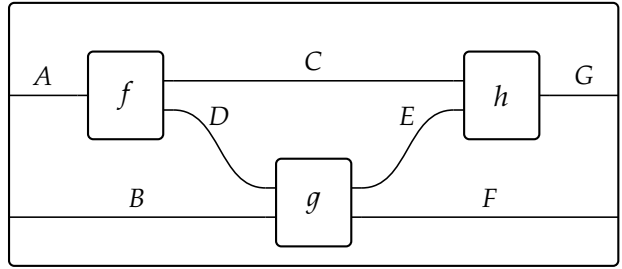
\[
\begin{tikzpicture}[oriented WD, string decoration={}]
	\node[bb={1}{2}] (X11) {$f$};
	\node[bb={2}{2}, below right=of X11] (X12) {$g$};
	\node[bb={2}{1}, above right=of X12] (X13) {$h$};
	\node[bb={0}{0}, fit={($(X11.north west)+(.3,1.5)$) (X12)  ($(X13.east)+(-.3,0)$)}] (Y1) {};
	\begin{scope}[font=\small]
  	\draw[ar] (Y1.west|-X11_in1) to node[above] {$A$} (X11_in1);	
  	\draw[ar] (Y1.west|-X12_in2) to node[above] {$B$} (X12_in2);
  	\draw[ar] (X11_out1) to node[above] {$C$} (X13_in1);
  	\draw[ar] (X11_out2) to node[above=5pt] {$D$} (X12_in1);
  	\draw[ar] (X12_out1) to node[above=5pt] {$E$} (X13_in2);
  	\draw[ar] (X12_out2) to node[above] {$F$} (X12_out2-|Y1.east);
  	\draw[ar] (X13_out1) to node[above] {$G$} (X13_out1-|Y1.east);
	\end{scope}
\end{tikzpicture}
\]
Suppose that $A=B=C=D=F=G=\ZZ$ and $E=\BB=\{\true,\false\}$, 
and suppose that $f_C(a)=|a|$, $f_D(a)=a*5$, $g_E(d,b)= ``d\leq b$,'' $g_F(d,b)=d-b$, and $h(c,e)=\tn{if }e\tn{ then }c\tn{ else }1-c$.
\begin{enumerate}
	\item What are $g_E(5,3)$ and $g_F(5,3)$?
	\item What are $g_E(3,5)$ and $g_F(3,5)$?
	\item What is $h(5,\true)$?
	\item What is $h(-5,\true)$?
	\item What is $h(-5,\false)$?
\end{enumerate}
The whole diagram now defines a function $A\times B\to G\times F$; call it $q$.
\begin{enumerate}[resume]
	\item What are $q_G(-2,3)$ and $q_F(-2,3)$?
	\item What are $q_G(2,3)$ and $q_F(2,3)$?
	\qedhere
\qedhere
\end{enumerate}
\end{exercise}

We will see more monoidal categories throughout the remainder of this book. %
\index{monoidal category|)}

\subsection{Categories enriched in a symmetric monoidal category}
\label{subsec.SMC_enrichment}
We will not need this again, but we once promised to explain why
$\cat{V}$-categories, where $\cat{V}$ is a symmetric monoidal preorder, deserve to
be seen as types of categories. The reason, as we have hinted, is that
categories should really be called $\smset$-categories. But wait, $\smset$ is not a
preorder! We'll have to generalize---categorify---$\cat{V}$-categories.

We now give a rough definition of categories enriched in a symmetric monoidal category $\cat{V}$. As in \cref{rdef.sym_mon_cat}, we suppress some technical parts in this sketch, hiding them under the name ``usual associative and unital laws.''

\begin{roughDef}%
\label{def.enriched_in_mon_cat}%
\index{enriched category!general
  definition}
\showhide
{
Let $\cat{V}$ be a symmetric monoidal category, as in \cref{rdef.sym_mon_cat}. To
specify a \emph{category enriched in $\cat{V}$}, or a \emph{$\cat{V}$-category},
denoted $\cat{X}$,\begin{enumerate}[label=(\roman*)]
	\item one specifies a collection $\Ob(\cat{X})$, elements of which are called \emph{objects};
	\item for every pair $x,y\in\Ob(\cat{X})$, one specifies an object $\cat{X}(x,y)\in\cat{V}$, called the \emph{hom-object} for $x,y$;%
\index{hom object}
	\item for every $x\in\Ob(\cat{X})$, one specifies a morphism $\id_x\colon I\to\cat{X}(x,x)$ in $\cat{V}$, called the \emph{identity element};%
\index{identity!in enriched categories|see {enriched category, identity in}}%
\index{enriched category!identity in}
	\item for each $x,y,z\in\Ob(\cat{X})$, one specifies a morphism $\cp\colon\cat{X}(x,y)\otimes\cat{X}(y,z)\to\cat{X}(x,z)$, called the \emph{composition morphism}.%
\index{composition!in enriched categories|see {enriched categories, composition in}}%
\index{enriched category!composition in}
\end{enumerate}
These constituents are required to satisfy the usual associative and unital laws.%
\index{associativity!in enriched categories}%
\index{unitality!in enriched categories}
}
{
Let $\cat{V}$ be a symmetric monoidal category, as in \cref{rdef.sym_mon_cat}. A \emph{category enriched in $\cat{V}$}, or a \emph{$\cat{V}$-category}, denoted $\cat{X}$, consists of the following constituents:
\begin{enumerate}[%
\label=(\roman*)]
	\item a set $\Ob(\cat{X})$, elements of which are called \emph{objects};
	\item for every pair $x,y\in\Ob(\cat{X})$, an object $\cat{X}(x,y)\in\cat{V}$, called the \emph{hom-object} for $x,y$;
	\item for every $x\in\Ob(\cat{X})$, a morphism $\id_x\colon I\to\cat{X}(x,x)$ in $\cat{V}$, called the \emph{identity element};
	\item for each $x,y,z\in\Ob(\cat{X})$, a morphism $\cp\colon\cat{X}(x,y)\otimes\cat{X}(y,z)\to\cat{X}(x,z)$, called the \emph{composition morphism}.
\end{enumerate}
These constituents are required to make the following diagrams commute:
\[
\begin{tikzcd}
	\cat{X}(x,y)\otimes I\ar[d,"{\cat{X}(x,y)\otimes\id_y}"']\ar[r,"\cong"]&\cat{X}(x,y)\\
	\cat{X}(x,y)\otimes\cat{X}(y,y)\ar[r, "\cp"']&\cat{X}(x,y)\ar[u, equal]
\end{tikzcd}
\hspace{1in}
\begin{tikzcd}
	I\otimes\cat{X}(x,y)\ar[d,"{\id_x\otimes\cat{X}(x,y)}"']\ar[r,"\cong"]&\cat{X}(x,y)\\
	\cat{X}(x,x)\otimes\cat{X}(x,y)\ar[r, "\cp"']&\cat{X}(x,y)\ar[u, equal]
\end{tikzcd}
\]
\[
\begin{tikzcd}[column sep=-17pt]
	\cat{X}(w,x)\otimes(\cat{X}(x,y)\otimes\cat{X}(y,z))\ar[rr,"\cong"]\ar[d,"{\cat{X}(w,x)\otimes\cp}"']
	&&
	(\cat{X}(w,x)\otimes\cat{X}(x,y))\otimes\cat{X}(y,z)\ar[d,"{\cp\otimes\cat{X}(y,z)}"]
	\\
	\cat{X}(w,x)\otimes\cat{X}(x,z)\ar[dr,"\cp"']
	&&
	\cat{X}(w,y)\otimes\cat{X}(y,z)\ar[dl,"\cp"]
	\\
	&\cat{X}(w,z)
\end{tikzcd}
\]
}
\end{roughDef}

The precise, non-rough, definition can be found in other sources, e.g.\ \cite{Nlab:symmeric-monoidal-category}, \cite{wiki:Symmetric-monoidal-category}, \cite{Kelly:1982a}.

\begin{exercise} %
\label{exc.cat_is_set_enriched}
Recall from \cref{ex.set_as_mon_cat} that $\cat{V}=(\smset,\{1\},\times)$ is a
symmetric monoidal category. This means we can apply
\cref{def.enriched_in_mon_cat}. Does the (rough) definition roughly agree with
the definition of category given in \cref{def.category}? Or is there a subtle
difference?
\end{exercise}

\begin{remark} %
\label{rem.cats_and_vcats2}%
\index{enriched categories}
We first defined $\cat{V}$-categories in \cref{def.cat_enriched_mpos}, where
$\cat{V}$ was required to be a monoidal preorder. To check we're not abusing our
terms, it's a good idea to make sure that $\cat{V}$-categories as per
\cref{def.cat_enriched_mpos} are still $\cat{V}$-categories as per
\cref{def.enriched_in_mon_cat}. 

The first thing to observe is that every symmetric monoidal preorder is a symmetric
monoidal category (\cref{exc.mon_preorder_is_cat}). So given a symmetric monoidal
preorder $\cat{V}$, we can apply \cref{def.enriched_in_mon_cat}.
The required data (i) and (ii) then get us off to a good start: both definitions of
$\cat{V}$-category require objects and hom-objects, and they are specified in
the same way. On the other hand, \cref{def.enriched_in_mon_cat} requires two
additional pieces of data: (iii) identity elements and (iv) composition
morphisms. Where do these come from?%
\index{hom object}

In the case of preorders, there is at most one morphism between any two objects, so
we do not need to choose an identity element and a composition morphism.
Instead, we just need to make sure that an identity element and a composition
morphism exist. This is exactly what properties (a) and (b) of
\cref{def.cat_enriched_mpos} say. 

For example, the requirement (iii) that a $\cat{V}$-category $\cat{X}$ has a
chosen identity element $\id_x\colon I \to \cat{X}(x,x)$ for the object $x$
simply becomes the requirement (a) that $I \le \cat{X}(x,x)$ is true in
$\cat{V}$. This is typical of the story of categorification: what were mere
properties in \cref{def.cat_enriched_mpos} have become structures in
\cref{def.enriched_in_mon_cat}.
\end{remark}

\begin{exercise} %
\label{exc.metric_space_identities}
What are identity elements in Lawvere metric spaces (that is,
$\Cost$-categories)? How do we interpret this in terms of distances?
\end{exercise}

\section{Profunctors form a compact closed category}%
\index{profunctors!compact closed category of}%
\index{compact closed category|(}

In this section we will define compact closed categories and show that $\Feas$, and more generally $\cat{V}$-profunctors, form such a thing. Compact-closed categories are monoidal categories whose wiring diagrams allow feedback. The wiring diagrams look like this:
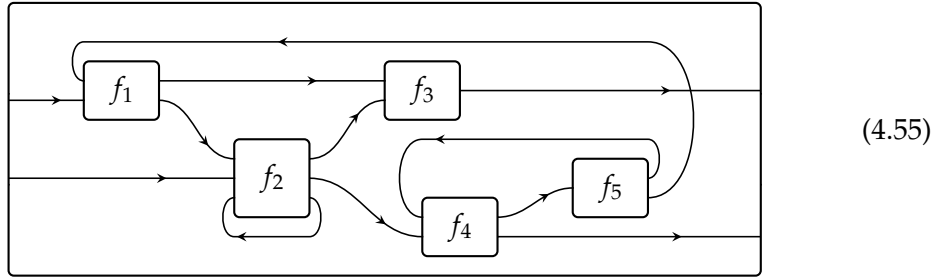
\begin{equation}%
\label{eqn.generic_compact_closed_diag}
\begin{tikzpicture}[oriented WD, bbx=1cm, bb port sep=1, bb port length=0, baseline=(Z)]
	\node[bb={2}{2}] (X11) {$f_1$};
	\node[bb={3}{3}, below right=of X11] (X12) {$f_2$};
	\node[bb={2}{1}, above right=of X12] (X13) {$f_3$};
	\node[bb={2}{2}, below right = -1 and 1.5 of X12] (X21) {$f_4$};
	\node[bb={1}{2}, above right=-1 and 1 of X21] (X22) {$f_5$};
  \node[bb={2}{2}, fit = {($(X11.north east)+(-1,2)$) (X12) (X13) ($(X21.south)$) ($(X22.east)+(.5,0)$)}] (Z) {};
	\draw[ar] (X21_out1) to (X22_in1);
	\draw[ar] let \p1=(X22.north east), \p2=(X21.north west), \n1={\y1+\bby}, \n2=\bbportlen in
          (X22_out1) to[in=0] (\x1+\n2,\n1) -- (\x2-\n2,\n1) to[out=180] (X21_in1);
	\draw[ar] (X11_out1) to (X13_in1);
	\draw[ar] (X11_out2) to (X12_in1);
	\draw[ar] (X12_out1) to (X13_in2);
	\draw[ar] (Z_in1'|-X11_in2) to (X11_in2);	
	\draw[ar] (Z_in2'|-X12_in2) to (X12_in2);
	\draw[ar] (X12_out2) to (X21_in2);
	\draw[ar] (X21_out2) to (Z_out2'|-X21_out2);
	\draw[ar] let \p1=(X12.south east), \p2=(X12.south west), \n1={\y1-\bby}, \n2=\bbportlen in
	  (X12_out3) to[in=0] (\x1+\n2,\n1) -- (\x2-\n2,\n1) to[out=180] (X12_in3);
	\draw[ar] let \p1=(X22.north east), \p2=(X11.north west), \n1={\y2+\bby}, \n2=\bbportlen in
          (X22_out2) to[in=0] (\x1+\n2,\n1) -- (\x2-\n2,\n1) to[out=180] (X11_in1);
	\draw[ar] (X13_out1) to (Z_out1'|-X13_out1);
\end{tikzpicture}
\end{equation}
It's been a while since we thought about co-design, but these were the kinds of wiring diagrams we drew, e.g.\ connecting the chassis, the motor, and the battery in \cref{eqn.chassis}. Compact closed categories are symmetric monoidal categories, with a bit more structure that allow us to formally interpret the sorts of feedback that occur in co-design problems. This same structure shows up in many other fields, including quantum mechanics and dynamical systems.%
\index{wiring diagram}

In \cref{eqn.styles_of_WD,subsec.SMPs_science} we discussed various flavors of
wiring diagrams, including those with icons for splitting and terminating wires.%
\index{icon}
For compact-closed categories, our additional icons allow us to bend outputs
into inputs, and vice versa. To keep track of this, however, we draw arrows on
our wire, which can either point forwards or backwards. For example, we can draw this
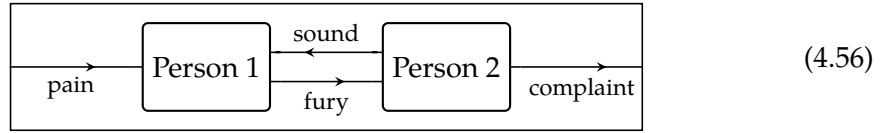
\begin{equation}%
\label{eqn.sound_fury}%
\index{icon}
\begin{tikzpicture}[oriented WD, baseline=(pair)]
	\node[bb={1}{2}] (P1) {Person 1};
	\node[bb={2}{1}, right=1 of P1] (P2) {Person 2};
	\node[draw, fit=(P1) (P2), inner xsep=50pt] (pair) {};
	\begin{scope}[font=\footnotesize]
  	\draw[ar] (pair.west|-P1_in1) to node[below] {pain} (P1_in1);
  	\draw[ar] (P2_in1) to node[above] {sound} (P1_out1);
  	\draw[ar] (P1_out2) to node[below] {fury} (P2_in2);
		\draw[ar] (P2_out1) to node[below] {complaint} (P2_out1-|pair.east);
	\end{scope}
\end{tikzpicture}
\end{equation}
We then add icons---called a cap and a cup---allowing any wire to reverse
direction from forwards to backwards and from backwards to forwards.%
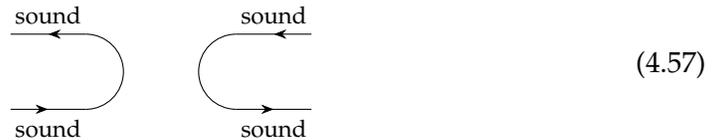
\index{icon!cup and cap}
\begin{equation}%
\label{eqn.cap_cup}
\begin{tikzpicture}[decoration={markings, mark=at position 0.5 with {\arrow{Stealth};}}, x=1cm, font=\footnotesize, baseline=(A)]
	\draw[postaction={decorate}] (0,0) to node[below] {sound} (1,0);
	\draw (1,0) arc (-90:90:.5cm);
	\draw[postaction={decorate}] (1,1) to node[above] {sound} (0,1);
	\draw[postaction={decorate}] (3,0) to node[below] {sound} (4,0);
	\draw (3,1) arc (90:270:.5cm);
	\draw[postaction={decorate}] (4,1) to node[above] {sound} (3,1);
	\node (A) at (0,.5) {};
\end{tikzpicture}
\end{equation}
Thus we can draw the following
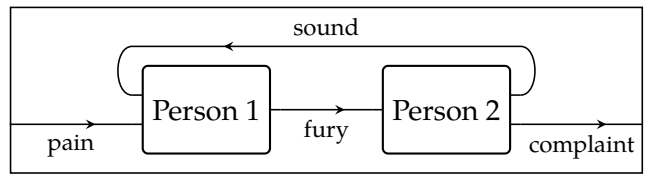
\[
\begin{tikzpicture}[oriented WD]
	\node[bb={2}{1}] (P1) {Person 1};
	\node[bb={1}{2}, right=1 of P1] (P2) {Person 2};
	\node[draw, fit={($(P1.north west)+(0,2)$) (P2)}, inner xsep=50pt] (pair) {};
	\begin{scope}[font=\footnotesize]
  	\draw[ar] (pair.west|-P1_in2) to node[below] {pain} (P1_in2);
  	\draw[ar] (P1_out1) to node[below] {fury} (P2_in1);
  	\draw[ar] let \p1=(P2.north east), \p2=(P1.north west), \n1=\bbportlen, \n2=\bby in
			(P2_out1) to[in=0] (\x1+\n1, \y1+\n2) to[in=0,out=180] node[above] {sound} (\x2-\n1,\y2+\n2) to[out=180] (P1_in1);
		\draw[ar] (P2_out2) to node[below] {complaint} (P2_out2-|pair.east);
	\end{scope}
\end{tikzpicture}
\]
and its meaning is equivalent to that of \cref{eqn.sound_fury}.

We will begin by giving the axioms for a compact closed category. Then we will look again at feasibility relations in co-design---and more generally at enriched profunctors---and show that they indeed form a compact closed category.%
\index{profunctor!enriched|see {profunctor}}

\subsection{Compact closed categories}

As we said, compact closed categories are symmetric monoidal
categories (see \cref{rdef.sym_mon_cat}) with extra structure.

\begin{definition}%
\label{def.compact_closed}%
\index{category!compact closed|see {compact closed category}}%
\index{closed category!compact|see {compact closed category}}%
\index{compact closed category}
	Let $(\cat{C},I,\otimes)$ be a symmetric monoidal category, and $c\in\Ob(\cat{C})$ an object. A \emph{dual for $c$} consists of three constituents
	\begin{enumerate}[label=(\roman*)]
		\item an object $c^*\in\Ob(\cat{C})$, called the \emph{dual of $c$},%
\index{dual!object}
		\item a morphism $\eta_c\colon I\to c^*\otimes c$, called the \emph{unit for $c$},%
\index{unit}
		\item a morphism $\epsilon_c\colon c\otimes c^*\to I$, called the \emph{counit for $c$}.%
\index{counit}%
\index{compact closed category!duals in}
	\end{enumerate}
These are required to satisfy two equations for every $c\in\Ob(\cat{C})$, which we draw as commutative diagrams:
\begin{equation}%
\label{eqn.yanking}
\begin{tikzcd}
	c\ar[d, "\cong"']\ar[r, equals]&c\\
	c\otimes I\ar[d,"c\otimes\eta_c"']&	I\otimes c\ar[u, "\cong"']\\
	c\otimes(c^*\otimes c)\ar[r,"\cong"']&(c\otimes c^*)\otimes c\ar[u,"\epsilon_c\otimes c"']\\
\end{tikzcd}
\hspace{.8in}
\begin{tikzcd}
	c^*\ar[d, "\cong"']\ar[r, equals]&c^*\\
	I\otimes c^*\ar[d,"\eta_c\otimes c^*"']&	c^*\otimes I\ar[u, "\cong"']\\
	(c^*\otimes c)\otimes c^*\ar[r,"\cong"']&c^*\otimes (c\otimes c^*)\ar[u,"c^*\otimes \epsilon_c"']\\
\end{tikzcd}
\end{equation}
These equations are sometimes called the \emph{snake equations}.%
\index{snake
equations}

If for every object $c\in\Ob(\cat{C})$ there exists a dual $c^*$ for $c$, then we say that $(\cat{C},I,\otimes)$ is \emph{compact closed}.
\end{definition}

In a compact closed category, each wire is equipped with a direction. For any object $c$, a forward-pointing wire labeled $c$ is considered equivalent to a backward-pointing wire labeled $c^*$, i.e. $\To{c}$ is the same as $\From{c^*}$. The cup and cap discussed above are in fact the unit and counit morphisms; they are drawn as follows.
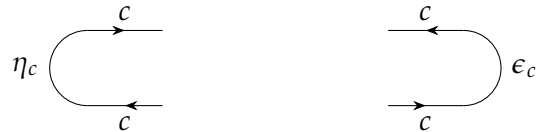
\[
\begin{tikzpicture}[decoration={markings, mark=at position 0.5 with {\arrow{Stealth};}}, x=1cm]
	\draw[postaction={decorate}] (1,0) to node[below] {$c$} (0,0);
	\draw (0,0) arc (270:90:.5cm);
	\node[left] at (-.5,.5) {$\eta_c$};
	\draw[postaction={decorate}] (0,1) to node[above] {$c$} (1,1);
	\draw[postaction={decorate}] (4,0) to node[below] {$c$} (5,0);
	\draw (5,0) arc (-90:90:.5cm);
	\node[right] at (5.5,.5) {$\epsilon_c$};
	\draw[postaction={decorate}] (5,1) to node[above] {$c$} (4,1);
\end{tikzpicture}
\]
In wiring diagrams, the snake equations \eqref{eqn.yanking} are then drawn as follows:
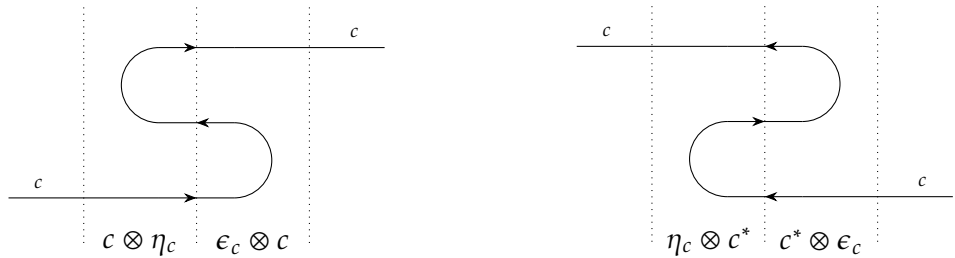
\[
\begin{tikzpicture}[decoration={markings, mark=at position 0.5 with {\arrow{Stealth};}}, x=1cm]
	\draw (-2,-1) to node[above, pos=.2, font=\tiny] {$c$} (0,-1);
	\draw[postaction={decorate}] (0,-1) to (1,-1);
	\draw (1,-1) arc (-90:90:.5cm);
	\draw[postaction={decorate}] (1,0) to (0,0);
	\draw (0,0) arc (270:90:.5cm);
	\draw[postaction={decorate}] (0,1) to (1,1);
	\draw (1,1) to node[above, pos=.8, font=\tiny] {$c$} (3,1);
	\draw[dotted] (-1,-1.6) -- (-1,1.6);
	\draw[dotted] (.5,-1.6) -- (.5,1.6);
	\draw[dotted] (2,-1.6) -- (2,1.6);
	\node at (-.25,-1.6) {$c\otimes\eta_c$};
	\node at (1.25,-1.6) {$\epsilon_c\otimes c$};
\end{tikzpicture}
\hspace{1in}
\begin{tikzpicture}[decoration={markings, mark=at position 0.5 with {\arrow{Stealth};}}, x=1cm]
	\draw (3,-1) to node[above, pos=.2, font=\tiny] {$c$} (1,-1);
	\draw[postaction={decorate}] (1,-1) to (0,-1);
	\draw (0,-1) arc (270:90:.5cm);
	\draw[postaction={decorate}] (0,0) to (1,0);
	\draw (1,0) arc (-90:90:.5cm);
	\draw[postaction={decorate}] (1,1) to (0,1);
	\draw (0,1) to node[above, pos=.8, font=\tiny] {$c$} (-2,1);
	\draw[dotted] (-1,-1.6) -- (-1,1.6);
	\draw[dotted] (.5,-1.6) -- (.5,1.6);
	\draw[dotted] (2,-1.6) -- (2,1.6);
	\node at (-.25,-1.6) {$\eta_c\otimes c^*$};
	\node at (1.25,-1.6) {$c^*\otimes\epsilon_c$};	
\end{tikzpicture}
\]
Note that the pictures in \cref{eqn.cap_cup} correspond to $\epsilon_{\text{sound}}$ and $\eta_{\text{sound}^*}$\space. 

Recall the notion of monoidal closed preorder; a monoidal category can also be monoidal
closed. This means that for every pair of objects $c,d\in\Ob(\cat{C})$ there is
an object $c\multimap d$ and an isomorphism $\cat{C}(b\otimes
c,d)\cong\cat{C}(b,c\multimap d)$, natural in $b$. While we will not provide a
full proof here, compact closed categories are so-named because they are a
special type of monoidal closed category.

\begin{proposition}%
\label{prop.double_dual}%
\index{closed category!compact implies monoidal}
If $\cat{C}$ is a compact closed category, then
\begin{enumerate}
	\item $\cat{C}$ is monoidal closed;
\end{enumerate}
and for any object $c\in\Ob(\cat{C})$,
\begin{enumerate}[resume]
	\item if $c^*$ and $c'$ are both duals to $c$ then there is an
	isomorphism $c^*\cong c'$; and
	\item there is an isomorphism between $c$ and its double-dual, $c\cong c^{**}$.%
\index{dual!double}
\end{enumerate}
\end{proposition}
To prove 1., the key idea is that for any $c$ and $d$, the object $c \multimap
d$ is given by $c^{*} \otimes d$, and the natural isomorphism $\cat{C}(b\otimes
c,d)\cong\cat{C}(b,c\multimap d)$ is given by precomposing with $\id_b\otimes
\eta_c$.

Before returning to co-design, we give another example of a compact closed
category, called $\Cat{Corel}$, which we'll see again in the chapters to come.

\begin{example} %
\label{ex.corel}%
\index{corelation}%
\index{equivalence relation!as
  corelation}
  Recall, from \cref{def.equivalence_relation}, that an equivalence relation on
  a set $A$ is a reflexive, symmetric, and transitive binary relation on $A$.
  Given two finite sets, $A$ and $B$, a \emph{corelation} $A \to B$ is an
  equivalence relation on $A\dju B$.

  So, for example, here is a corelation from a set $A$ having five elements to a set
  $B$ having six elements; two elements are equivalent if they are encircled by the
  same dashed line.
  
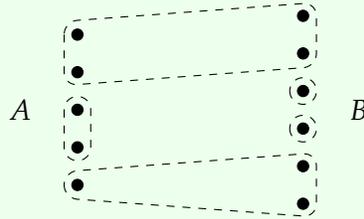
\[
  \begin{tikzpicture}
	\begin{pgfonlayer}{nodelayer}
		\node [contact, outer sep=5pt] (0) at (-2, 1) {};
		\node [contact, outer sep=5pt] (1) at (-2, 0.5) {};
		\node [contact, outer sep=5pt] (2) at (-2, -0) {};
		\node [contact, outer sep=5pt] (3) at (-2, -0.5) {};
		\node [contact, outer sep=5pt] (4) at (-2, -1) {};
		\node [contact, outer sep=5pt] (5) at (1, 1.25) {};
		\node [contact, outer sep=5pt] (6) at (1, 0.75) {};
		\node [contact, outer sep=5pt] (7) at (1, 0.25) {};
		\node [contact, outer sep=5pt] (8) at (1, -0.25) {};
		\node [contact, outer sep=5pt] (9) at (1, -0.75) {};
		\node [contact, outer sep=5pt] (10) at (1, -1.25) {};
		\node [style=none] (11) at (-2.75, -0) {$A$};
		\node [style=none] (12) at (1.75, -0) {$B$};
	\end{pgfonlayer}
	\begin{pgfonlayer}{edgelayer}
		\draw [rounded corners=5pt, dashed] 
   (node cs:name=0, anchor=north west) --
   (node cs:name=1, anchor=south west) --
   (node cs:name=6, anchor=south east) --
   (node cs:name=5, anchor=north east) --
   cycle;
		\draw [rounded corners=5pt, dashed] 
   (node cs:name=2, anchor=north west) --
   (node cs:name=3, anchor=south west) --
   (node cs:name=3, anchor=south east) --
   (node cs:name=2, anchor=north east) --
   cycle;
		\draw [rounded corners=5pt, dashed] 
   (node cs:name=4, anchor=north west) --
   (node cs:name=4, anchor=south west) --
   (node cs:name=10, anchor=south east) --
   (node cs:name=9, anchor=north east) --
   cycle;
   		\draw [rounded corners=5pt, dashed] 
   (node cs:name=7, anchor=north west) --
   (node cs:name=7, anchor=south west) --
   (node cs:name=7, anchor=south east) --
   (node cs:name=7, anchor=north east) --
   cycle;
   		\draw [rounded corners=5pt, dashed] 
   (node cs:name=8, anchor=north west) --
   (node cs:name=8, anchor=south west) --
   (node cs:name=8, anchor=south east) --
   (node cs:name=8, anchor=north east) --
   cycle;
	\end{pgfonlayer}
\end{tikzpicture}
\]

  There exists a category, denoted $\corel{}$, where the objects are finite
  sets, and where a morphism from $A \to B$ is a corelation $A \to B$. 
  The composition rule is simpler to look at than to write down formally.%
  \tablefootnote{
  To compose corelations $\alpha\colon A \to B$
  and $\beta\colon B \to C$, we need to construct an equivalence relation
  $\alpha\cp\beta$ on $A\dju C$. To do so requires three steps: (i) consider $\alpha$ and $\beta$ as
  relations on $A\dju B \dju C$, (ii) take the transitive closure of their
  union, and then (iii) restrict to an equivalence relation on $A\dju C$. Here is the formal description. Note that as binary relations, we have $\alpha \subseteq (A\dju B)
  \times (A\dju B)$, and $\beta\subseteq (B\dju C) \times (B\dju C)$. We also
  have three inclusions: $\iota_{A\dju B}\colon A \dju B \to A \dju B \dju C$,
  $\iota_{B\dju C}\colon B \dju C \to A \dju B \dju C$, and $\iota_{A\dju
  C}\colon A \dju C \to A \dju B \dju C$. Recalling our notation from \cref{sec.galois_connections}, we define
  \[
    \alpha\cp\beta \coloneq \iota^\ast_{A\dju C}((\iota_{A\dju B})_!(\alpha)\vee
    (\iota_{B\dju C})_!(\beta)).
  \]
  }
  If in addition to the corelation $\alpha\colon A \to B$ above we have
another corelation $\beta\colon B \to C$
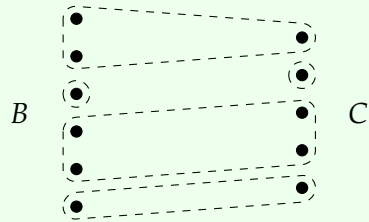
\[

\]
Then the composite $\beta\circ\alpha$ of our two corelations is given by
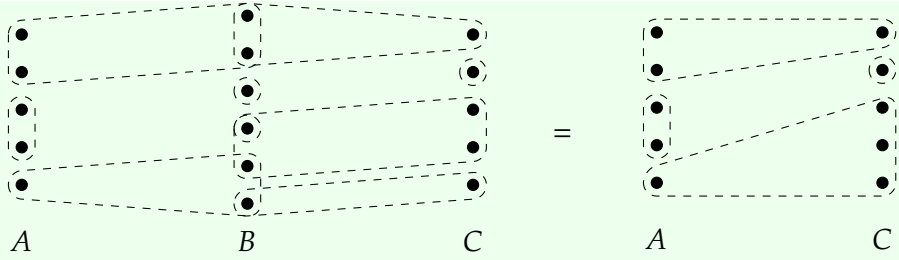
\vspace{-1ex}
\[
  \begin{aligned}
%
\end{aligned}
\]
That is, two elements are equivalent in the composite corelation if we may
travel from one to the other staying within equivalence classes of either
$\alpha$ or $\beta$.

The category $\corel{}$ may be equipped with the symmetric monoidal structure
$(\varnothing, \dju)$.  This monoidal category is compact closed, with every
finite set its own dual.%
\index{dual!self}  Indeed, note that for any finite set $A$ there is an
equivalence relation on $A \dju A\coloneq\{(a,1), (a,2) \,|\, a \in A\}$ where
each part simply consists of the two elements $(a,1)$ and $(a,2)$ for each $a
\in A$. The unit on a finite set $A$ is the corelation $\eta_A\colon \varnothing
\to A \dju A$ specified by this equivalence relation; similarly the counit on
$A$ is the corelation $\epsilon_A\colon A \dju A \to \varnothing$ specifed by
this same equivalence relation.
\end{example}

\begin{exercise} %
\label{exc.corelations}
Consider the set $\ul{3}=\{1,2,3\}$.
\begin{enumerate}
	\item Draw a picture of the unit corelation $\varnothing\to\ul{3}\sqcup\ul{3}$.
	\item Draw a picture of the counit corelation $\ul{3}\sqcup\ul{3}\to\varnothing$.
	\item Check that the snake equations \eqref{eqn.yanking} hold. (Since every object is its own dual, you only need to check one of them.)
\qedhere
\end{enumerate}
\end{exercise}

\subsection{$\Feas$ as a compact closed category}
\index{co-design}%
\index{feasibility relation!compact closed category of|(}
We close the chapter by returning to co-design and showing that $\Feas$ has a compact closed structure. This is what allows us to draw the kinds of wiring diagrams we saw in \cref{eqn.chassis,eqn.generic_compact_closed_diag,eqn.sound_fury}: it is what puts actual mathematics behind these pictures.

Instead of just detailing this compact closed structure for $\Feas =
\Prof_\Bool$, it's no extra work to prove that for any skeletal (unital,
commutative) quantale $(\cat{V},I,\otimes)$ the profunctor category
$\Prof_\cat{V}$ of \cref{thm.quantale_prof} is compact closed, so we'll discuss
this general fact.

\begin{theorem}
Let $\cat{V}$ be a skeletal quantale. The category $\Prof_\cat{V}$ can be given
the structure of a compact closed category, with monoidal product given by the
product of $\cat{V}$-categories.
\end{theorem}

Indeed, all we need to do is construct the monoidal structure and duals for
objects. Let's sketch how this goes.

\paragraph{Monoidal products in $\Prof_\cat{V}$ are just product categories.}
\index{product!of categories}

In terms of wiring diagrams, the monoidal structure looks like stacking wires or boxes on top of one another, with no new interaction.%
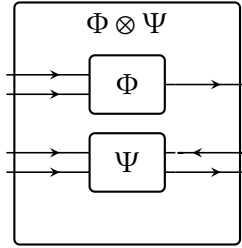
\index{stacking|see {monoidal product}}
\[
\begin{tikzpicture}[oriented WD, bbx=1cm, bb port sep=1]
	\node[bb={2}{1}] (A) {$\Phi$};
	\node[bb={2}{2}, below=1 of A] (B) {$\Psi$};
	\node[bb={0}{0}, fit=(A) (B), inner ysep=20pt, bb name=$\Phi\otimes\Psi$] (outer) {};
\begin{scope}[shorten <=-3pt, ar]
	\draw[ar] (outer.west|-A_in1) -- (A_in1);
	\draw[ar] (outer.west|-A_in2) -- (A_in2);
	\draw[ar] (outer.west|-B_in1) -- (B_in1);
	\draw[ar] (outer.west|-B_in2) -- (B_in2);
\end{scope}
\begin{scope}[shorten >=-3pt]
	\draw[ar] (A_out1) to (A_out1-|outer.east);
	\draw[ar] (B_out1-|outer.east) to (B_out1);
	\draw[ar] (B_out2) to (B_out2-|outer.east);
\end{scope}
\end{tikzpicture}
\]
We take our monoidal product on $\Prof_\cat{V}$ to be that given by the product of $\cat{V}$-categories; the definition was given in \cref{def.enriched_prod}, and we worked out several examples there. To recall, the formula for the hom-sets in $\cat{X}\times\cat{Y}$ is given by
\[(\cat{X}\times\cat{Y})((x,y),(x',y'))\coloneqq\cat{X}(x,x')\otimes\cat{Y}(y,y').\]
But monoidal products need to be given on morphisms also, and the morphisms in $\Prof_\cat{V}$ are $\cat{V}$-profunctors. So given $\cat{V}$-profunctors $\Phi\colon\cat{X}_1\tickar\cat{X}_2$ and $\Psi\colon\cat{Y}_1\tickar\cat{Y}_2$, one defines a $\cat{V}$-profunctor $(\Phi\times\Psi)\colon\cat{X}_1\times\cat{Y}_1\tickar\cat{X}_2\times\cat{Y}_2$ by
\[(\Phi\times\Psi)((x_1,y_1),(x_2,y_2))\coloneqq\Phi(x_1,x_2)\otimes\Psi(y_1,y_2).\]

\begin{exercise}%
\label{exc.explain_monoidal_prod_feas}
Interpret the monoidal products in $\Prof_{\Bool}$ in terms of feasibility. That
is, preorders represent resources ordered by availability ($x\leq x'$ means that
$x$ is available given $x'$) and a profunctor is a feasibility relation. Explain
why $\cat{X}\times\cat{Y}$ makes sense as the monoidal product of resource
preorders $\cat{X}$ and $\cat{Y}$ and why $\Phi\times\Psi$ makes sense as the monoidal product of feasibility relations $\Phi$ and $\Psi$.
\end{exercise}

\paragraph{The monoidal unit in $\Prof_\cat{V}$ is $\Cat{1}$.}

To define a monoidal structure on $\Prof_\cat{V}$, we need not only a monoidal product---as defined above---but also a monoidal unit. Recall the $\cat{V}$-category $\Cat{1}$; it has one object, say $1$, and $\cat{1}(1,1)=I$ is the monoidal unit of $\cat{V}$. We take $\Cat{1}$ to be the monoidal unit of $\Prof_\cat{V}$.

\begin{exercise} %
\label{exc.prof_monoidal_unit}
In order for $\Cat{1}$ to be a monoidal unit, there are supposed to be
isomorphisms $\cat{X}\times\Cat{1}\tickar\cat{X}$ and
$\Cat{1}\times\cat{X}\tickar\cat{X}$ in $\Prof_\cat{V}$, for any $\cat{V}$-category $\cat{X}$. What
are they?
\end{exercise}

\paragraph{Duals in $\Prof_\cat{V}$ are just opposite categories.}%
\index{opposite category!as dual}

In order to regard $\Prof_\cat{V}$ as a compact closed category
(\cref{def.compact_closed}), it remains to specify duals and the corresponding
cup and cap.%
\index{dual!category as opposite, in $\Prof$}

Duals are easy: for every $\cat{V}$-category $\cat{X}$, its dual is its opposite
category $\cat{X}\op$ (see \cref{def.enriched_op}).  The unit and counit then
look like identities. To elaborate, the unit is a $\cat{V}$-profunctor
$\eta_{\cat{X}}\colon\Cat{1}\tickar\cat{X}\op\times\cat{X}$. By definition, this
is a $\cat{V}$-functor
\[
\eta_{\cat{X}}\colon\Cat{1}\times\cat{X}\op\times\cat{X}\to\cat{V};
\]
we define it by $\eta_{\cat{X}}(1,x,x')\coloneqq\cat{X}(x,x')$. Similarly, the
counit is the profunctor
$\epsilon_{\cat{X}}\colon(\cat{X}\times\cat{X}\op)\tickar\Cat{1}$, defined by
$\epsilon_{\cat{X}}(x,x',1)\coloneqq\cat{X}(x,x')$.

\begin{exercise} %
\label{exc.prof_duals}
Check these proposed units and counits do indeed obey the snake equations
\cref{eqn.yanking}.
\end{exercise}

\index{feasibility relation!compact closed category of|)}
\index{compact closed category|)}


\section{Summary and further reading}

This chapter introduced three important ideas in category theory: profunctors,
categorification, and monoidal categories. Let's talk about them in turn.

Profunctors generalize binary relations. In particular, we saw that the idea of
profunctor over a monoidal preorder gave us the additional power necessary to
formalize the idea of a feasibility relation between resource preorders. The idea
of a feasibility relation is due to Andrea Censi; he called them \emph{monotone
codesign problems}. The basic idea is explained in \cite{Censi:2015a}, where he
also gives a programming language to specify and solve codesign problems. In
\cite{censi:2017a}, Censi further discusses how to use estimation to make
solving codesign problems computationally efficient.

We also saw profunctors over the preorder $\Cost$, and how to think of these as
bridges between Lawvere metric space. We referred earlier to Lawvere's paper
\cite{Lawvere:1973a}; plenty more on $\Cost$-profunctors can be found there.

Profunctors, however are vastly more general than the two examples we have
discussed; $\cat{V}$-profunctors can be defined not only when $\cat{V}$ is a
preorder, but for any symmetric monoidal category. A delightful, detailed
exposition of profunctors and related concepts such as equipments, companions
and conjoints, symmetric monoidal bicategories can be found in
\cite{Shulman:2008a,Shulman:2010a}.

We have not defined symmetric monoidal bicategories, but you would be correct if
you guessed this is a sort of categorification of symmetric monoidal
categories. Baez and Dolan tell the subtle story of categorifying categories to
get ever \emph{higher} categories in \cite{Baez.Dolan:1998}. Crane and Yetter
give a number of examples of categorification in \cite{Crane.Yetter:1996a}.

Finally, we talked about monoidal categories and compact closed categories.
Monoidal categories are a classic, central topic in category theory, and a quick
introduction can be found in \cite{MacLane:1998a}. Wiring diagrams play a huge
role in this book and in applied category theory in general; while informally
used for years, these were first formalized in the case of monoidal categories.
You can find the details here
\cite{Joyal.Street:1993a,Joyal.Street.Verity:1996a}.

Compact closed categories are a special type of structured monoidal category;
there are many others. For a broad introduction to the different flavors of
monoidal category, detailed through their various styles of wiring diagram, see
\cite{selinger2010survey}.%
\index{wiring diagram!styles of}

\index{co-design|)}


\setcounter{chapter}{4}

\chapter[Signal flow graphs: Props, presentations, \& proofs]{Signal flow graphs:\\Props, presentations, and proofs}%
\label{chap.SFGs}

\section{Comparing systems as interacting signal processors}
\index{cyber-physical system}
Cyber-physical systems are systems that involve tightly interacting physical and
computational parts. An example is an autonomous car: sensors inform a decision
system that controls a steering unit that drives a car, whose movement changes the sensory input. While such systems
involve complex interactions of many different subsystems---both physical ones, such
as the driving of a wheel by a motor, or a voltage placed across a wire, and
computational ones, such as a program that takes a measured velocity and returns
a desired acceleration---it is often useful to model the system behavior as
simply the passing around and processing of signals. For this illustrative sketch, we
will just think of signals as things which we can add and multiply, such as real
numbers. 

Interaction in cyber-physical systems can often be understood as variable
sharing; i.e.\ when two systems are linked, certain variables become shared. For
example, when we connect two train carriages by a physical coupling, the train
carriages must have the same velocity, and their positions differ by a constant.
Similarly, when we connect two electrical ports, the electric potentials at
these two ports now must be the same, and the current flowing into one must
equal the current flowing out of the other. %
\index{interconnection!as variable
sharing} Of course, the way the shared variable is actually used may be very different for the different subsystems using it, but sharing the variable serves to couple those systems nonetheless.

Note that both the above examples involve the physical joining of two systems; more
figuratively, we might express the interconnection by drawing a line connecting the
boxes that represent the systems. In its simplest form, this is captured by the
formalism of signal flow graphs, due to Claude Shannon in the 1940s. Here is an example
of a signal flow graph:%
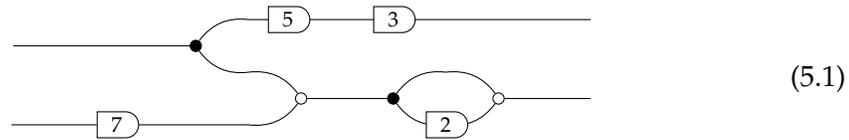
\index{signal flow graph|(}
\begin{equation}%
\label{eq.examplesfg}
  \begin{aligned}
\begin{tikzpicture}[scale=.7]
	\begin{pgfonlayer}{nodelayer}
		\node [style=none] (0) at (-6, -0) {};
		\node [style=none] (5) at (-6, 1.5) {};
		\node [style=wamp] (2) at (-4, 0) {$\scriptstyle 7$};
		\node [style=bdot] (6) at (-2.5, 1.5) {};
		\node [style=none] (7) at (-1.5, 1) {};
		\node [style=none] (18) at (-1.5, -0) {};
		\node [style=none] (9) at (-1.5, 2) {};
		\node [style=wamp] (10) at (-0.75, 2) {$\scriptstyle 5$};
		\node [style=wdot] (8) at (-0.5, 0.5) {};
		\node [style=bdot] (13) at (1.25, 0.5) {};
		\node [style=wamp] (16) at (1.25, 2) {$\scriptstyle 3$};
		\node [style=none] (12) at (2.25, 1) {};
		\node [style=wamp] (14) at (2.25, -0) {$\scriptstyle 2$};
		\node [style=wdot] (21) at (3.25, 0.5) {};
		\node [style=none] (17) at (5, 2) {};
		\node [style=none] (20) at (5, 0.5) {};
	\end{pgfonlayer}
	\begin{pgfonlayer}{edgelayer}
	\begin{scope}[font=\footnotesize]
		\draw (0.center) to (2);
		\draw (2) to (18.center);
		\draw [bend right, looseness=1] (18.center) to (8);
		\draw [bend right, looseness=1] (8) to (7.center);
		\draw [bend left, looseness=1.00] (7.center) to (6);
		\draw (6) to (5);
		\draw [bend left, looseness=1.00] (6) to (9.center);
		\draw (9.center) to (10.center);
		\draw (10.center) to (16.center);
		\draw (16.center) to (17.center);
		\draw [bend right, looseness=1.00] (12.center) to  (13);
		\draw (8) to (13);
		\draw [bend right, looseness=1.00] (13) to (14.center);
		\draw [bend right, looseness=1.00] (14.center) to (21);
		\draw [bend left, looseness=1.00] (12.center) to (21);
		\draw (21) to (20.center);
	\end{scope}
	\end{pgfonlayer}
\end{tikzpicture}
\end{aligned}
\end{equation}
We consider the dangling wires on the left as inputs, and those on the right as
outputs. In \cref{eq.examplesfg} we see three types of signal processing units, which we interpret as follows:
\begin{itemize}
	\item Each unit labelled by a number $a$ takes an input and multiplies it by $a$.
  \item Each black dot takes an input and produces two copies of it.
  \item Each white dot takes two inputs and produces their sum.
\end{itemize}
Thus the above signal flow graph takes in two input signals, say $x$ (on the upper left wire) and $y$ (on the
lower left wire), and---going from left to right as described above---produces two output signals: $u=15x$ (upper right) and $v=3x+21y$ (lower
right). Let's show some steps from this computation (leaving others off to avoid clutter):
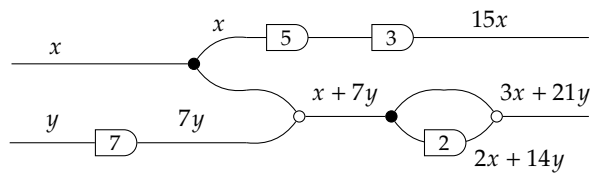
\[
\begin{tikzpicture}[scale=.7]
	\begin{pgfonlayer}{nodelayer}
		\node [style=none] (0) at (-6, -0) {};
		\node [style=none] (5) at (-6, 1.5) {};
		\node [style=wamp] (2) at (-4, 0) {$\scriptstyle 7$};
		\node [style=bdot] (6) at (-2.5, 1.5) {};
		\node [style=none] (7) at (-1.5, 1) {};
		\node [style=none] (18) at (-1.5, -0) {};
		\node [style=none] (9) at (-1.5, 2) {};
		\node [style=wamp] (10) at (-0.75, 2) {$\scriptstyle 5$};
		\node [style=wdot] (8) at (-0.5, 0.5) {};
		\node [style=bdot] (13) at (1.25, 0.5) {};
		\node [style=wamp] (16) at (1.25, 2) {$\scriptstyle 3$};
		\node [style=none] (12) at (2.25, 1) {};
		\node [style=wamp] (14) at (2.25, -0) {$\scriptstyle 2$};
		\node [style=wdot] (21) at (3.25, 0.5) {};
		\node [style=none] (17) at (5, 2) {};
		\node [style=none] (20) at (5, 0.5) {};
	\end{pgfonlayer}
	\begin{pgfonlayer}{edgelayer}
	\begin{scope}[font=\footnotesize]
		\draw (0.center) to node[above] {$y$} (2);
		\draw (2) to node[above] {$7y$} (18.center);
		\draw [bend right, looseness=1] (18.center) to (8);
		\draw [bend right, looseness=1] (8) to (7.center);
		\draw [bend left, looseness=1.00] (7.center) to (6);
		\draw (6) to node[above, near end] {$x$} (5);
		\draw [bend left, looseness=1.00] (6) to node[above] {$x$} (9.center);
		\draw (9.center) to (10.center);
		\draw (10.center) to (16.center);
		\draw (16.center) to node[above] {$15x$} (17.center);
		\draw [bend right, looseness=1.00] (12.center) to (13);
		\draw (8) to node[above] {$x+7y$} (13);
		\draw [bend right, looseness=1.00] (13) to (14.center);
		\draw [bend right, looseness=1.00] (14.center) to node[below right=0
		and -.1] {$2x+14y$}(21);
		\draw [bend left, looseness=1.00] (12.center) to (21);
		\draw (21) to node[above] {$3x+21y$}(20.center);
	\end{scope}
	\end{pgfonlayer}
\end{tikzpicture}
\]
In words, the signal flow graph first multiplies $y$ by $7$, then splits $x$
into two copies, adds the second copy of $x$ to the lower signal to get $x+7y$, and so
on.

A signal flow graph might describe an existing system, or it might specify a
system to be built. In either case, it is important to be able to analyze these
diagrams to understand how the composite system converts inputs to outputs.
This is reminiscent of a co-design problem from \cref{chap.codesign}, which asks
how to evaluate the composite feasibility relation from a diagram of simpler feasibility
relations. We can use this process of evaluation to determine whether two
different signal flow graphs in fact specify the same composite system, and
hence to validate that a system meets a given specification. 

In this chapter, however, we introduce categorical tools---props and
their presentations---for reasoning more directly with the diagrams. Recall from
\cref{chap.resource_theory} that symmetric monoidal preorders are a type of
symmetric monoidal category where the \emph{morphisms} are constrained to be
very simple: there can be at most one morphism between any two objects. Here
shall see that signal flow graphs represent morphisms in a different,
complementary simplification of the symmetric monoidal category concept, known as a \emph{prop}.%
\footnote{
Historically, the word `prop' was written in all caps, `PROP,' standing for `products and permutations category.' However, we find `PROP' a bit loud, so like many modern authors we opt for writing it as `prop.'
}
A prop is a symmetric monoidal category where the \emph{objects} are constrained
to be very simple: they are generated, using the monoidal product, by just a
single object. Just as the wiring diagrams for symmetric monoidal preorders did
not require labels on the boxes, this means that wiring diagrams for props do
not require labels on the wires. This makes props particularly suited for
describing diagrammatic formalisms such as signal flow graphs, which only have
wires of a single type.

Finally, many systems behave in what is called a \emph{linear} way, and linear systems form a foundational part of control theory, a branch of engineering that works on cyber-physical systems. Similarly, linear
algebra is a foundational part of modern mathematics, both pure and applied,
which includes not only control theory, but also the practice of computing,
physics, statistics, and many others. As we analyze signal flow graphs, we shall see that they are in fact a
way of recasting linear algebra---more specifically, matrix operations---in graphical terms. More formally, we shall say
that signal flow graphs have \emph{functorial semantics} as matrices. %
\index{control theory}

\section{Props and presentations}%
\label{sec.props_and_presentations}
\index{prop|(}

Signal flow graphs as in \cref{eq.examplesfg} are easily seen to be wiring
diagrams of some sort. However they have the property that, unlike for monoidal
preorders and monoidal categories, there is no need to label the wires. This
corresponds to a form of symmetric monoidal category, known as a prop, which has
a very particular set of objects.

\subsection{Props: definition and first examples}

Recall the definition of symmetric strict monoidal category from
\cref{rdef.sym_mon_cat,rem.strict_mon_cat}.

\begin{definition}%
\label{def.prop}%
\index{prop}
A \emph{prop} is a symmetric strict monoidal category $(\cat{C},0,+)$ for which $\Ob(\cat{C})=\NN$, the monoidal unit is $0\in\NN$, and the monoidal product on objects is given by addition.
\end{definition}

Note that each object $n$ is the $n$-fold monoidal product of the object $1$; we
call $1$ the \emph{generating object}. Since the objects of a prop are always
the natural numbers, to specify a prop $P$ it is enough to specify five things:
\begin{enumerate}[label=(\roman*)]
	\item a set $\cat{C}(m,n)$ of morphisms $m\to n$, for $m,n\in\NN$.
	\item for all $n \in \NN$, an identity map $\id_n\colon n \to n$.
	\item for all $m, n \in \NN$, a symmetry map $\sigma_{m,n}\colon m+n \to
	n+m$.%
\index{symmetry}
	\item a composition rule: given $f\colon m \to n$ and $g\colon n \to p$,
	a map $(f\cp g)\colon m \to p$.
	\item a monoidal product on morphisms: given $f\colon m\to m'$ and $g\colon n\to n'$, a map $(f+g)\colon m+n\to m'+n'$.
\end{enumerate}
Once one specifies the above data, he should check that his specifications
satisfy the rules of symmetric monoidal categories (see
\cref{rdef.sym_mon_cat}).%
\footnote{We use `his' terminology because this definition is for boys only. The rest of the book is for girls only.}

\begin{example}%
\label{ex.FinSet_Prop}%
\index{prop!$\finset$}
  There is a prop $\Cat{FinSet}$ where the morphisms $f\colon m \to n$ are
  functions from $\ord{m}=\{1,\dots m\}$ to $\ord{n}=\{1,\dots,n \}$. (The
  identities, symmetries, and composition rule are obvious.) The monoidal
  product on functions is given by the disjoint union of functions: that is,
  given $f\colon m \to m'$ and $g\colon n \to n'$, we
  define $f+ g\colon m+n \longrightarrow m'+n'$ by
  \begin{equation}%
\label{eqn.mon_prod_finset}
    i \longmapsto 
    \begin{cases}
      f(i) &\mbox{ if }1\leq i\leq m ;\\
      m'+g(i) &\mbox{ if }m+1\leq i \leq m+n.
    \end{cases}
    \qedhere
  \end{equation}
\end{example}

\begin{exercise}%
\label{exc.finset_as_prop}%
\index{obvious!conventional mathematical meaning of}
In \cref{ex.FinSet_Prop} we said that the identities, symmetries, and
composition rule in $\Cat{FinSet}$ ``are obvious.'' In math lingo, this just
means ``we trust that the reader can figure them out, if she spends the time tracking
down the definitions and fitting them together.''
\begin{enumerate}
	\item Draw a morphism $f\colon 3\to 2$ and a morphism $g\colon 2\to 4$ in $\finset$.
	\item Draw $f+g$.
	\item What is the composition rule for morphisms $f\colon m\to n$ and $g\colon n\to p$ in $\Cat{FinSet}$?
	\item What are the identities in $\Cat{FinSet}$? Draw some.
	\item Choose $m,n \in \NN$, and draw the symmetry map $\sigma_{m,n}$
	in $\Cat{FinSet}$?
\qedhere
\end{enumerate}
\end{exercise}

\begin{example}%
\label{ex.function_bijection}%
\index{function!bijective}
  Recall from \cref{def.function} that a bijection is a function that is both
  surjective and injective.  There is a prop $\Cat{Bij}$ where the morphisms
  $f\colon m \to n$ are bijections $\ord{m} \to \ord{n}$. Note that in this case
  morphisms $m \to n$ only exist when $m = n$; when $m\neq n$ the homset
  $\Cat{Bij}(m,n)$ is empty. Since $\Cat{Bij}$ is a subcategory of
  $\Cat{FinSet}$, we can define the monoidal product to be as in
  \cref{eqn.mon_prod_finset}.
\end{example}%
\index{category!of finite sets}%
\index{category!of bijections}

\begin{example}%
\label{ex.corelation}%
\index{corelation}%
\index{compact closed category}
  The compact closed category $\Cat{Corel}$, in which the morphisms $f\colon
  m \to n$ are partitions on $\ord{m}\sqcup \ord{n}$ (see \cref{ex.corel}), is a prop.
\end{example}

\begin{example}%
\label{ex.relation}%
\index{binary relation}
There is a prop $\Cat{Rel}$ for which morphisms $m\to n$ are relations, $R\ss\ord{m}\times\ord{n}$. The composition of $R$ with $S\ss\ord{n}\times\ord{p}$ is
\[R\cp S\coloneqq\{(i,k)\in\ord{m}\times\ord{p}\mid\exists(j\in\ord{n})\ldotp (i,j)\in R\text{ and }(j,k)\in S\}.\]
The monoidal product is relatively easy to formalize using universal properties,%
\tablefootnote{The monoidal product $R_1+R_2$ of relations $R_1\ss \ord{m_1}\times \ord{n_1}$ and $R_2\ss \ord{m_2}\times \ord{n_2}$ is given by $R_1\sqcup R_2\ss(\ord{m_1}\times \ord{n_1})\sqcup(\ord{m_2}\times \ord{n_2})\ss(\ord{m_1}\sqcup \ord{m_2})\times(\ord{n_1}\sqcup \ord{n_2})$.}
but one might get better intuition from pictures:
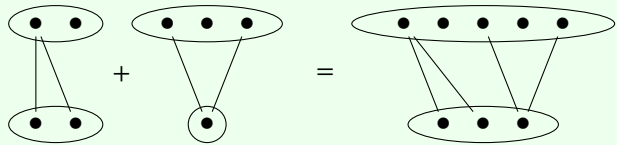
\[
\begin{tikzpicture}[x=1ex, inner sep=2pt]
	\node (a1) {$\bullet$};
	\node[right=1 of a1] (a2) {$\bullet$};
	\node[draw, ellipse, inner sep=0, fit=(a1) (a2)] (a) {};
	\node[below=1 of a1] (b1) {$\bullet$};
	\node[below=1 of a2] (b2) {$\bullet$};
	\node[draw, ellipse, inner sep=0, fit=(b1) (b2)] (b) {};
	\draw (a1) -- (b1);
	\draw (a1) -- (b2);
	\node [right=5 of a2] (c1) {$\bullet$};
	\node [right=1 of c1] (c2) {$\bullet$};
	\node [right=1 of c2] (c3) {$\bullet$};
	\node[draw, ellipse, inner sep=0, fit=(c1) (c3)] (c) {};
	\node[below=1 of c2] (d1) {$\bullet$};
	\node[draw, ellipse, inner sep=0, fit=(d1)] (d) {};
	\draw (c1) -- (d1);
	\draw (c3) -- (d1);
	\node[right=10 of c3] (x1) {$\bullet$};
	\node[right=1 of x1] (x2) {$\bullet$};
	\node[right=1 of x2] (x3) {$\bullet$};
	\node[right=1 of x3] (x4) {$\bullet$};
	\node[right=1 of x4] (x5) {$\bullet$};
	\node[draw, ellipse, inner sep=0, fit=(x1) (x5)] (x) {};
	\node[below=1 of x2] (y1) {$\bullet$};
	\node[right=1 of y1] (y2) {$\bullet$};
	\node[right=1 of y2] (y3) {$\bullet$};
	\node[draw, ellipse, inner sep=0, fit=(y1) (y3)] (y) {};
	\draw (x1) -- (y1);
	\draw (x1) -- (y2);
	\draw (x3) -- (y3);
	\draw (x5) -- (y3);
	\node at ($(a2)!.5!(c1)$)  (midx) {};
	\node at ($(a2)!.5!(b2)$) (midy) {};
	\node at (midx |- midy) (plus) {+};
	\node at ($(c3)!.5!(x1)$) (midxx) {};
	\node at (plus-|midxx) {=};	

\end{tikzpicture}
\qedhere
\]
\end{example}

\begin{exercise}%
\index{prop!posetal}%
\label{exc.posetal_prop}
  A posetal prop is a prop that is also a poset. That is, a posetal prop is a
  symmetric monoidal preorder of the form $(\NN,\preceq)$, for some poset relation
  $\preceq$ on $\NN$, where the monoidal product on objects is addition. We've
  spent a lot of time discussing order structures on the natural numbers. Give
  three examples of a posetal prop.
\end{exercise}

%
%

\begin{exercise}%
\label{exc.prop_practice}
Choose one of \cref{ex.function_bijection,ex.corelation,ex.relation} and explicitly provide the five aspects of props discussed below \cref{def.prop}.
\end{exercise}

\begin{definition}%
\label{def.prop_functor}%
\index{functor!prop}
  Let $\cat{C}$ and $\cat{D}$ be props. A \emph{functor} $F\colon\cat{C}\to\cat{D}$ is called a \emph{prop functor} if
  \begin{enumerate}[label=(\alph*)]
  	\item $F$ is identity-on-objects, i.e.\ $F(n)=n$ for all $n\in\Ob(\cat{C}) = \Ob(\cat{D}) = \nn$, and
		\item for all $f\colon m_1\to m_2$ and $g\colon n_1\to n_2$ in $\cat{C}$, we have $F(f)+F(g)=F(f+g)$ in $\cat{D}$.
	\end{enumerate}
\end{definition}

\begin{example}
The inclusion $i\colon\Cat{Bij}\to\Cat{FinSet}$ is a prop functor. Perhaps more interestingly, there is a prop functor $F\colon\Cat{FinSet}\to\Cat{Rel}_\Cat{Fin}$. It sends a function $f\colon \ord{m}\to\ord{n}$ to the relation $F(f)\coloneqq\{(i,j)\mid f(i)=j\}\ss\ord{m}\times\ord{n}$.%
\index{relation}
\end{example}

\subsection{The prop of port graphs}%
\label{subsec.prop_PGs}%
\index{port graph|(}
An important example of a prop is the one in which morphisms are open, directed, acyclic port graphs, as we next define. We will just call them port graphs.

\begin{definition}%
\label{def.port_graph}%
\index{port graph}
For $m,n\in\nn$, an \emph{$(m,n)$-port graph} $(V,\pgin,\pgout, \iota)$ is
specified by
\begin{enumerate}[label=(\roman*)]
\item a set $V$, elements of which are called \emph{vertices},
\item functions $\pgin,\pgout \colon V \to \nn$, where $\pgin(v)$ and
$\pgout(v)$ are called the \emph{in degree} and \emph{out degree} of each $v \in V$, and
\item a bijection $\iota\colon \ord{m} \sqcup O \stackrel{\cong}{\to} I \sqcup
\ord{n}$, where $I=\{(v,i)\mid v \in V, \, 1 \le i \le \pgin(v)\}$ is the set of
\emph{vertex inputs}, and $O=\{(v,i)\mid v \in V,\, 1 \le i \le \pgout(v)\}$ is
the set of \emph{vertex outputs}.
\end{enumerate}
This data must obey the following acyclicity condition. First, use the bijection
$\iota$ to construct the graph with vertices $V$ and with
an arrow $e^{u,i}_{v,j}\colon u \to v$ for every $i,j \in \nn$ such that
$\iota(u,i)=(v,j)$; call it the \emph{internal flow graph}. If the internal flow graph is acyclic---that is, if the only
path from any vertex $v$ to itself is the trivial path---then we say that
$(V,\pgin,\pgout, \iota)$ is a port graph.%
\index{port graph!acyclicity condition}%
\index{trivial path}
\end{definition}

This seems quite a technical construction, but it's quite intuitive once you
unpack it a bit. Let's do this.

\begin{example}%
\label{ex.a_port_graph}
Here is an example of a $(2,3)$-port graph, i.e.\ with $m=2$ and $n=3$:
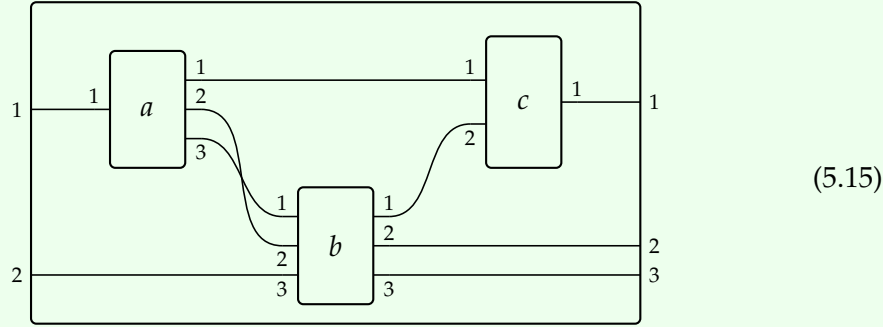
\begin{equation}%
\label{eqn.port_graph}

\end{equation}
Since the port graph has type $(2,3)$, we draw two ports on the
left hand side of the outer box, and three on the right. 
The vertex set is $V=\{a,b,c\}$ and, for example $\pgin(a)=1$ and
$\pgout(a)=3$, so we draw
one port on the left-hand side and three ports on the right-hand side
of the box labelled $a$. The bijection $\iota$ is what tells us how the ports are connected by wires:
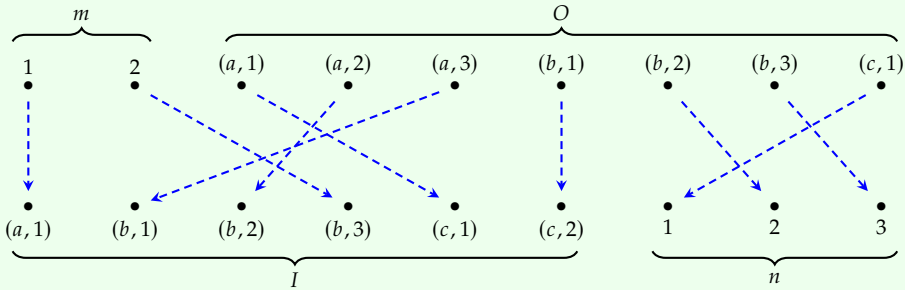
\[
%
\]
The internal flow graph---which one can see is acyclic---is shown below:
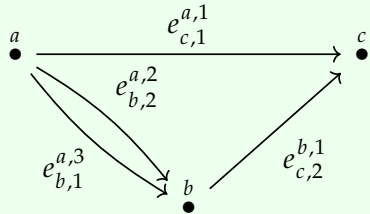
\[
\begin{tikzcd}[column sep = 50pt, row sep=35pt]
	\LMO{a}\ar[rr, "\displaystyle e^{a,1}_{c,1}"]\ar[rd, bend left=10,
	"\displaystyle e^{a,2}_{b,2}"]\ar[rd, bend right=10, "\displaystyle e^{a,3}_{b,1}"']&&\LMO{c}\\
	&\LMO{b}\ar[ur,"\displaystyle e^{b,1}_{c,2}"']
\end{tikzcd}
\]
\end{example}

As you might guess from \eqref{eqn.port_graph}, port graphs are closely related to
wiring diagrams for monoidal categories, and even more closely related to wiring
diagrams for props.

\paragraph{A category $\Cat{PG}$ whose morphisms are port graphs.}
\label{page.port_graphs_prop}%
\index{port graph!as morphism}

Given an $(m,n)$-port graph $(V,\pgin,\pgout,\iota)$ and an $(n,p)$-port graph
$(V',\pgin',\pgout',\iota')$, we may compose them to produce an $(m,p)$-port
graph $(V\sqcup V',\copair{\pgin,\pgin'},\copair{\pgout,\pgout'},\iota'')$. Here
$\copair{\pgin,\pgin'}$ denotes the function $V\sqcup V' \to \nn$ which maps elements
of $V$ according to $\pgin$, and elements of $V'$ according to $\pgin'$, and
similarly for $\copair{\pgout,\pgout'}$. The
bijection $\iota''\colon\ord{m}\sqcup O\sqcup O'\to I\sqcup I'\sqcup\ord{p}$ is defined as follows:
\[
\iota''(x) = \begin{cases}
\iota(x) & \mbox{ if }\iota(x) \in I \\
\iota'(\iota(x)) & \mbox{ if }\iota(x) \in \ord{n} \\
\iota'(x) & \mbox{ if }x \in O.'
\end{cases}
\]

\begin{exercise}%
\label{exc.port_graph_comp}
Describe how port graph composition looks, with respect to the visual
representation of \cref{ex.a_port_graph}, and give a nontrivial example.
\end{exercise}

We thus have a category $\Cat{PG}$, whose objects are natural numbers $\Ob(\Cat{PG})=\NN$, whose morphisms are port graphs $\Cat{PG}(m,n)=\{(V,\pgin,\pgout,\iota)\mid\text{as in \cref{def.port_graph}}\}$.
Composition of port graphs is as above, and the identity port graph on $n$ is the $(n,n)$-port graph
$(\varnothing, !,!,\id_{\ord{n}})$, where
$!\colon \varnothing \to \nn$ is the unique function. The identity on an object,
say $3$, is depicted as follows:%
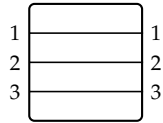
\index{identity!port graph}
\[
\begin{tikzpicture}[oriented WD, bb port length=0]
	\node[bb={3}{3}, minimum width = 1.5cm] (Y) {};
	\draw (Y.west|-Y_in1) to (Y.east|-Y_out1);
	\draw (Y.west|-Y_in2) to  (Y.east|-Y_out2);
	\draw (Y.west|-Y_in3) to  (Y.east|-Y_out3);
	\draw[label]
		node[left=3pt of Y_in1] {$1$}
		node[left=3pt of Y_in2] {$2$}
		node[left=3pt of Y_in3] {$3$}
		node[right=3pt of Y_out1] {$1$}
		node[right=3pt of Y_out2] {$2$}
		node[right=3pt of Y_out3] {$3$}
	;	
\end{tikzpicture}
\]

\paragraph{The monoidal structure structure on $\Cat{PG}$.}

This category $\Cat{PG}$ is in fact a prop. The monoidal
product of two port graphs $G\coloneqq (V,\pgin,\pgout,\iota)$ and
$G'\coloneqq (V',\pgin',\pgout',\iota')$ is given by taking the disjoint union
of $\iota$ and $\iota'$:
\begin{equation}%
\label{eqn.PG_prop}
G+G'\coloneqq\big((V\sqcup V'),[\pgin,\pgin'],[\pgout,\pgout'],(\iota\sqcup\iota')\big).
\end{equation}
The monoidal unit is $(\varnothing,!,!,!)$.

\begin{exercise}%
\label{exc.mon_prod_of_morphisms}
Draw the monoidal product of the morphism shown in \cref{eqn.port_graph} with itself. It will be a $(4,6)$-port graph, i.e.\ a morphism $4\to 6$ in $\Cat{PG}$.
\end{exercise}

\index{port graph|)}

\subsection{Free constructions and universal properties}%
\label{sec.free_constructions}
\index{universal property|(}

Given some sort of categorical structure, such as a preorder, a category, or a prop,
it is useful to be able to construct one according to your own specification.
(This should not be surprising.) The minimally-constrained structure that
contains all the data you specify is called the \emph{free structure} on your
specification: it's free from unneccessary constraints! We have already seen
some examples of free structures; let's recall and explore them.

\begin{example}[The free preorder on a relation]%
\index{preorder!free on a relation}%
\index{free!preorder}%
\index{relation!free preorder on}
For preorders, we saw the construction of taking the reflexive, transitive closure
of a relation. That is, given a relation $R\ss P\times P$, the reflexive,
transitive closure of $R$ is the called the free preorder on $R$. Rather than
specify all the inequalities in the preorder $(P,\leq)$, we can specify just a few
inequalities $p \le q$, and let our ``closure machine'' add in the minimum
number of other inequalities necessary to make $P$ a preorder. To obtain a preorder
out of a graph, or Hasse diagram, we consider a
graph $(V,A,s,t)$ as defining a relation $\{(s(a),t(a)) \mid a \in A\} \ss V
\times V$, and apply this closure machine.

But in what sense is the reflexive, transitive closure of a relation $R\ss P\times P$ really the \emph{minimally-constrained} preorder containing $R$? One way of understanding this is that the extra equalities impose
no further constraints when defining a monotone map \emph{out} of $P$. We are
claiming that freeness has something to do with maps \emph{out}! As strange as
an asymmetry might seem here (one might ask, ``why not maps in?''), the reader
will have an opportunity to explore it for herself in Exercises~\ref{exc.free_preorder_check_1}~and~\ref{exc.free_preorder_check_2}.%

A higher-level justification understands freeness as a left adjoint (see
\cref{ex.adjunctions}), but we will not discuss that here.
\end{example}

\begin{exercise}%
\label{exc.free_preorder_check_1}
Let $P$ be a set, let $R\ss P \times P$ a relation, let $(P,\leq_P)$ be the preorder
obtained by taking the reflexive, transitive closure of $R$, and let $(Q,\leq_Q)$ be an arbitrary preorder. Finally, let $f\colon P\to Q$ be a function, not assumed monotone.
\begin{enumerate}
	\item Suppose that for every $x,y\in P$, if $R(x,y)$ then $f(x)\leq f(y)$. Show that $f$ defines a monotone map $f\colon (P,\leq_P)\to (Q,\leq_Q)$.
	\item Suppose that $f$ defines a monotone map $f\colon (P,\leq_P)\to (Q,\leq_Q)$. Show that for every $x,y\in P$, if $R(x,y)$ then $f(x)\leq_Q f(y)$.
\end{enumerate}
We call this the \emph{universal property} of the free preorder $(P,\leq_P)$.%
\index{universal
property}
\end{exercise}

\begin{exercise}%
\label{exc.free_preorder_check_2}
Let $P$, $Q$, $R$, etc.\ be as in \cref{exc.free_preorder_check_1}. We want to see that the universal property is really about maps out of---and not maps in to---the reflexive, transitive closure $(P,\leq)$. So let $g\colon Q\to P$ be a function.
\begin{enumerate}
	\item Suppose that for every $a,b\in Q$, if $a\leq b$ then $(g(a),g(b))\in R$. Is it automatically true that $g$ defines a monotone map $g\colon(Q,\leq_Q)\to(P,\leq_P)$?
	\item Suppose that $g$ defines a monotone map $g\colon(Q,\leq_Q)\to(P,\leq_P)$. Is it automatically true that for every $a,b\in Q$, if $a\leq b$ then $(g(a),g(b))\in R$?
\end{enumerate}
The lesson is that maps between structured objects are defined to preserve
constraints. This means the domain of a map must be somehow more constrained than the
codomain. Thus having the fewest additional constraints coincides with having
the most maps out---every function that respects our generating constraints
should define a map.%
\end{exercise}

\begin{example}[The free category on a graph]%
\index{category!free}%
\index{free!category}
There is a similar story for categories. Indeed, we saw in
\cref{def.free_category} the construction of the free category $\free(G)$ on a
graph $G$. The objects of $\free(G)$ and the vertices of $G$ are the
same---nothing new here---but the morphisms of $\free(G)$ are not just the
arrows of $G$ because morphisms in a category have stricter requirements: they
must compose and there must be an identity. Thus morphisms in $\free(G)$ are the
\emph{closure} of the set of arrows in $G$ under these operations. Luckily
(although this happens often in category theory), the result turns out to
already be a relevant graph concept: the morphisms in $\free(G)$ are exactly the
paths in $G$. So $\free(G)$ is a category that in a sense contains $G$ and obeys
no equations other than those that categories are forced to obey.
\end{example}

\begin{exercise}%
\label{exc.free_cat_is_free}
Let $G=(V,A,s,t)$ be a graph, and let $\cat{G}$ be the free category on $G$. Let $\cat{C}$ be another category whose set of morphisms is denoted $\Mor(\cat{C})$. 
\begin{enumerate}
	\item Someone tells you that there are ``domain and codomain'' functions $\dom,\cod\colon\Mor(\cat{C})\to\Ob(\cat{C})$; interpret this statement.
	\item Show that the set of functors $\cat{G} \to
\cat{C}$ are in one-to-one correspondence with the set of pairs of functions
$(f,g)$, where $f\colon V \to \Ob(\cat{C})$ and $g\colon A\to\Mor(\cat{C})$ for which $\dom(g(a))=f(s(a))$ and $\cod(g(a))=f(t(a))$ for all $a$.
	\item Is $(\Mor(\cat{C}),\Ob(\cat{C}),\dom,\cod)$ a graph? If so, see if you can use the word ``adjunction'' in a sentence that describes the statement in part 2. If not, explain why not.
\qedhere
\end{enumerate}
\end{exercise}

\begin{exercise}[The free monoid on a set]%
\index{monoid!free}%
\index{free!monoid}%
\label{exc.free_monoid}
Recall from \cref{ex.monoid_nats} that monoids are one-object categories. For any set $A$, there is a graph $\Cat{Loop}(A)$ with one vertex and with one arrow from the vertex to itself for each $a\in A$. So if $A=\{a,b\}$ then $\Cat{Loop}(A)$ looks like this:
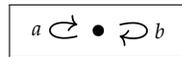
\[
\fbox{
\begin{tikzcd}[ampersand replacement=\&]
	\bullet\ar[loop left, "a"]\ar[loop right, "b"]
\end{tikzcd}
}
\]
The free category on this graph is a one-object category, and hence a monoid; it's called the free monoid on $A$.
\begin{enumerate}
	\item What are the elements of the free monoid on the set $A=\{a\}$?
	\item Can you find a well-known monoid that is isomorphic to the free monoid on $\{a\}$?%
\index{natural numbers!as free monoid}
	\item What are the elements of the free monoid on the set $A=\{a,b\}$?
\qedhere
\end{enumerate}
\end{exercise}

\subsection{The free prop on a signature}

We have been discussing free constructions, in particular for preorders and
categories. A similar construction exists for props. Since we already know what
the objects of the prop will be---the natural numbers---all we need to specify
is a set $G$ of \emph{generating morphisms}, together with the arities,%
\footnote{The arity of a prop morphism is a pair $(m,n)\in\NN\times\NN$, where
$m$ is the number of inputs and $n$ is the number of outputs.} that we want to
be in our prop. This information will be called a \emph{signature}. Just as we
can generate the free category from a graph, so too can we generate the free
prop from a signature. 

We now give an explicit construction of the free prop in terms of port graphs (see
\cref{def.port_graph}). 

\begin{definition}%
\label{def.free_prop}%
\index{free!prop}%
\index{prop!free}%
\index{port graph}
  A \emph{prop signature} is a tuple $(G,s,t)$, where $G$ is a set and $s,t\colon G \to \nn$ are functions; each element $g\in G$ is called a \emph{generator} and $s(g),t(g)\in\nn$ are called its \emph{in-arity and out-arity}. We often denote $(G,s,t)$ simply by $G$, taking $s,t$ to be implicit.%
\index{prop!signature of}
  
  A \emph{$G$-labeling} of a port graph $\Gamma=(V,\pgin,\pgout,\iota)$ is a
  function $\ell\colon V\to G$ such that the arities agree:
  $s(\ell(v))=\pgin(v)$ and $t(\ell(v))=\pgout(v)$ for each $v\in V$.
  
  Define the \emph{free prop on $G$}, denoted $\free(G)$, to have as morphisms
  $m\to n$ all $G$-labeled $(m,n)$-port graphs. The composition and monoidal
  structure are just those for port graphs $\Cat{PG}$ (see \cref{eqn.PG_prop});
  the labelings (the $\ell$'s) are just carried along. 
\end{definition}

The morphisms in $\free(G)$ are port graphs $(V,\fun{in},\fun{out},\iota)$ as in
\cref{def.port_graph}, that are equipped with a $G$-labeling. To draw a port
graph, just as in \cref{ex.a_port_graph}, we draw each vertex $v\in V$ as a box
with $\fun{in}(v)$-many ports on the left and $\fun{out}(v)$-many ports on the
right. In wiring diagrams, we depict the labeling function $\ell\colon V\to G$
by using $\ell$ to add labels (in the usual sense) to our boxes. Note that
multiple boxes can be labelled with the same generator. For example, if
$G=\{f\colon 1 \to 1, g\colon 2 \to 2, h\colon 2 \to 1\}$, then the following is
a morphism $3\to 2$ in $\free(G)$:
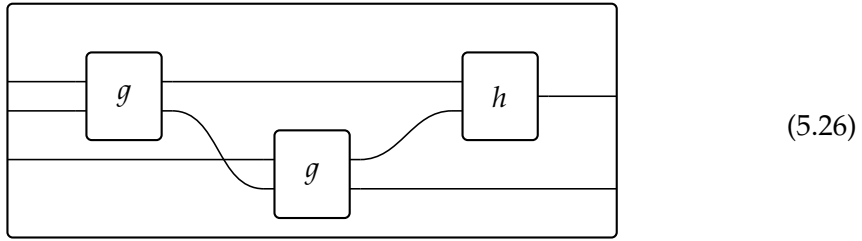
\begin{equation}%
\label{eqn.labeled_pg}
\begin{aligned}
\begin{tikzpicture}[oriented WD]
	\node[bb={2}{2}] (X11) {$g$};
	\node[bb={2}{2}, below right=-.5 and 1 of X11] (X12) {$g$};
	\node[bb={2}{1}, above right=-.5 and 1 of X12] (X13) {$h$};
	\node[bb={0}{0}, fit={($(X11.north west)+(.3,1.5)$) (X12)  ($(X13.east)+(-.3,0)$)}] (Y1) {};
	\draw (Y1.west|-X11_in1) to (X11_in1);	
	\draw (Y1.west|-X11_in2) to (X11_in2);	
	\draw (Y1.west|-X12_in1) to (X12_in1);
	\draw (X11_out1) to (X13_in1);
	\draw (X11_out2) to (X12_in2);
	\draw (X12_out1) to (X13_in2);
	\draw (X12_out2) to (X12_out2-|Y1.east);
	\draw (X13_out1) to (X13_out1-|Y1.east);
\end{tikzpicture}
\end{aligned}
\end{equation}
Note that the generator $g$ is used twice, while the generator
$f$ is not used at all in \cref{eqn.labeled_pg}. This is perfectly fine.

\begin{example}%
\index{function!bijective}
The free prop on the empty set $\varnothing$ is $\Cat{Bij}$. This is because
each morphism must have a labelling function of the form $V \to \varnothing$,
and hence we must have $V=\varnothing$; see \cref{exc.map_to_empty}. Thus the only morphisms $(n,m)$ are
those given by port graphs $(\varnothing, !,!,\sigma)$, where $\sigma\colon n
\to m$ is a bijection.
\end{example}

\begin{exercise}%
\index{port graph!as morphism}%
\label{exc.free_prop_port_graph}
Consider the following prop signature:
\[
  G\coloneqq\{\rho_{m,n} \mid m, n \in \nn\},\qquad s(\rho_{m,n})\coloneqq m,\quad t(\rho_{m,n})\coloneqq n,
\]
i.e.\ having one generating morphism for each $(m,n)\in\nn^2$. Show that
$\free(G)$ is the prop $\Cat{PG}$ of port graphs from \cref{subsec.prop_PGs}.
\end{exercise}

Just like free preorders and free categories, the free prop is characterized by
a universal property in terms of maps out. The following can be proved in a
manner similar to \cref{exc.free_cat_is_free}.

\begin{proposition}
The free prop $\free(G)$ on a signature $(G,s,t)$ has the property that, for any
prop $\cat{C}$, the prop functors $\free(G) \to \cat{C}$ are in one-to-one
correspondence with functions $G \to \cat{C}$ that send each $g \in G$ to a
morphism $s(g) \to t(g)$ in $\cat{C}$. 
\end{proposition}

\paragraph{An alternate way to describe morphisms in $\free(G)$.}

Port graphs provide a convenient formalism of thinking about morphisms in the free prop on a signature $G$, but there is another approach which is also useful. It is syntactic, in the sense that we start with a small stock of basic morphisms, including elements of $G$, and then we inductively build new morphisms from them using the basic operations of props: namely composition and monoidal product. Sometimes the conditions of monoidal categories---e.g.\ associativity, unitality, functoriality, see \cref{rdef.sym_mon_cat}---force two such morphisms to be equal, and so we dutifully equate them. When we are done, the result is again the free prop $\free(G)$. Let's make this more formal.

First, we need the notion of a prop expression. Just as prop signatures are the
analogue of the graphs used to present categories, prop expressions are the
analogue of paths in these graphs.

\begin{definition}%
\label{def.prop_expressions}%
\index{prop!expression in}
  Suppose we have a set $G$ and functions $s,t\colon G \to \nn$. We define a
  \emph{$G$-generated prop expression}, or simply \emph{expression} $e\colon m\to n$, where $m,n\in\nn$, 
  inductively as follows:
\begin{itemize}
\item The empty morphism $\id_0 \colon 0 \to 0$, the identity morphism $\id_1\colon 1 \to
1$, and the symmetry $\sigma\colon 2 \to 2$ are expressions.%
\tablefootnote{One can think of $\sigma$ as the ``swap'' icon $\swap{1em}\colon 2 \to
  2$}
\item the generators $g \in G$ are expressions $g\colon s(g) \to t(g)$.
\item if $\alpha\colon m \to n$ and $\beta\colon p \to q$ are expressions, then
$\alpha+\beta\colon m+p \to n+q$ is an expression.
\item if $\alpha\colon m \to n$ and $\beta\colon n \to p$ are expressions, then
$\alpha\cp\beta\colon m \to p$ is an expression.
\end{itemize}
We write $\expr(G)$ for the set of expressions in $G$. If $e\colon m\to n$ is an expression, we refer to $(m,n)$ as its \emph{arity}.
\end{definition}%
\index{icon}

\begin{example}
Let $G=\{f\colon 1 \to 1, g\colon 2 \to 2, h\colon 2 \to 1\}$. Then 
\begin{itemize}
\item $\id_1\colon 1 \to 1$, 
\item $f\colon 1 \to 1$,
\item $f\cp\id_1\colon 1\to 1$,
\item $h+\id_1\colon 3 \to 2$, and 
\item $(h+\id_1) \cp \sigma \cp g\cp \sigma \colon 3 \to 2$
\end{itemize}
are all $G$-generated prop expressions.
\end{example}

Both $G$-labeled port graphs and $G$-generated prop expressions are ways to
describe morphisms in the free prop $\free(G)$. Note, however, that unlike for
$G$-labeled port graphs, there may be two $G$-generated prop expressions that
represent the same morphism. For example, we want to consider $f\cp \id_1$ and
$f$ to be the same morphism, since the unitality axiom for categories says $f\cp
\id_1 =f$. Nonetheless, we only consider two $G$-generated prop expressions
equal when some axiom from the definition of prop requires that they be so;
again, the free prop is the \emph{minimally-constrained} way to take $G$ and
obtain a prop.
\index{quotient}

Since both port graphs and prop expressions describe morphisms in $\free(G)$,
you might be wondering how to translate between them. Here's how to turn a port
graph into a prop expression: imagine a vertical line moving through the port
graph from left to right. Whenever you see ``action''---either a box or wires
crossing---write down the sum (using $+$) of all the boxes $g$, all the
symmetries $\sigma$, and all the wires $\id_1$ in that column. Finally, compose
all of those action columns. For example, in the picture below we see four
action columns:
\[
\begin{tikzpicture}[oriented WD]
	\node[bb={2}{2}] (X11) {$g$};
	\node[bb={2}{2}, below right=-.5 and 1 of X11] (X12) {$g$};
	\node[bb={2}{1}, above right=-.5 and 1 of X12] (X13) {$h$};
	\node[bb={0}{0}, fit={($(X11.north west)+(.3,1.5)$) (X12)  ($(X13.east)+(-.3,0)$)}] (Y1) {};
	\draw (Y1.west|-X11_in1) to (X11_in1);	
	\draw (Y1.west|-X11_in2) to (X11_in2);	
	\path[name path=Q, draw] (Y1.west|-X12_in1) to (X12_in1);
	\draw (X11_out1) to (X13_in1);
	\path[name path=R, draw] (X11_out2) to (X12_in2);
	\draw (X12_out1) to (X13_in2);
	\draw (X12_out2) to (X12_out2-|Y1.east);
	\draw (X13_out1) to (X13_out1-|Y1.east);
	\coordinate (col1) at ($(Y1.west)!.29!(Y1.east)$);
	\coordinate (col2) at ($(Y1.west)!.41!(Y1.east)$);
	\coordinate (col3) at ($(Y1.west)!.65!(Y1.east)$);
	\begin{scope}[blue, dashed, very thick]
	\draw (col1|-Y1.north) -- (col1|-Y1.south);
	\draw (col2|-Y1.north) -- (col2|-Y1.south);
	\draw (col3|-Y1.north) -- (col3|-Y1.south);
	\end{scope}
\end{tikzpicture}
\]
Here the result is $(g+\id_1)\cp(\id_1+\sigma)\cp(\id_1+g)\cp(h+\id_1)$.%

\begin{exercise}%
\label{exc.free_prop_pic}
Consider again the free prop on generators $G=\{f\colon 1 \to 1, g\colon 2 \to
2, h\colon 2 \to 1\}$. Draw a picture of
$(f+\id_1+\id_1)\cp(\sigma+\id_1)\cp(\id_1+h)\cp \sigma\cp g$, where
$\sigma\colon 2\to 2$ is the symmetry map.
\end{exercise}

Another way of describing when we should consider two prop expressions equal is
to say that they are equal if and only if they represent the same port graph. In
either case, these notions induce an equivalence relation on the set of prop
expressions. To say that we consider these certain prop expressions equal is to
say that the morphisms of the free prop on $G$ are the $G$-generated prop
expressions \emph{quotiented} by this equivalence relation (see
\cref{def.quotient}).

\subsection{Props via presentations}
\label{sec.prop_presentations}

In \cref{subsec.presenting_cats} we saw that a presentation for a category, or
database schema, consists of a graph together with imposed equations between
paths. Similarly here, sometimes we want to construct a prop whose morphisms
obey specific equations. But rather than mere paths, the things we want to
equate are prop expressions as in \cref{def.prop_expressions}.%
\index{database!schema}
\index{category!presentation of}

\begin{roughDef}%
\label{rdef.presentation_prop}%
\index{prop!presentation of}%
\index{presentation!of prop}
A \emph{presentation} $(G,s,t,E)$ for a prop is a set $G$, functions $s,t\colon G\to\nn$, and a set $E\subseteq \expr(G) \times \expr(G)$ of pairs of $G$-generated prop expressions, such that $e_1$ and $e_2$ have the same arity for each $(e_1,e_2)\in E$. We refer to $G$ as the set of generators and to $E$ as the set of \emph{equations} in the presentation.%
\tablefootnote{Elements of $E$, which we call equations, are traditionally called ``relations.'' We think of $(e_1,e_2)\in E$ as standing for the equation $e_1=e_2$, as this will be forced soon.}%
\index{generators and relations|see {presentation}}

The prop $\cat{G}$ \emph{presented} by the presentation $(G,s,t,E)$ is the prop whose
morphisms are elements in $\expr(G)$, quotiented by both the equations
$e_1=e_2$ where $(e_1,e_2) \in E$, and by the axioms of symmetric strict monoidal categories.
\end{roughDef}

\begin{remark}%
\label{rem.free_prop_universal_property}%
\index{universal property}
Given a presentation $(G,s,t,E)$, it can be shown that the prop $\cat{G}$ has a
universal property in terms of ``maps out.'' Namely prop functors from $\cat{G}$
to any other prop $\cat{C}$ are in one-to-one correspondence with functions $f$
from $G$ to the set of morphisms in $\cat{C}$ such that 
\begin{itemize}
\item for all $g \in G$, $f(g)$ is a morphism $s(g) \to t(g)$, and 
\item for all $(e_1,e_2) \in E$, we have that $f(e_1)=f(e_2)$ in
$\cat{C}$, where $f(e)$ denotes the morphism in $\cat{C}$ obtained by
applying $f$ to each generators in the expression $e$, and then composing
the result in $\cat{C}$. 
\end{itemize}
\end{remark}

\begin{exercise}%
\label{exc.same_free_prop}
Is it the case that the free prop on generators $(G,s,t)$, defined in \cref{def.free_prop}, is the same thing as the prop presented by $(G,s,t,\varnothing)$, having no relations, as defined in \cref{rdef.presentation_prop}? Or is there a subtle difference somehow?
\end{exercise}

\index{universal property|)}

\section{Simplified signal flow graphs}

We now return to signal flow graphs, expressing them in terms of props. We will
discuss a simplified form without feedback (the only sort we have discussed so
far), and then extend to the usual form of signal flow graphs in
\cref{subsec.full_SFGs}. But before we can do that, we must say what we mean by
signals; this gets us into the algebraic structure of ``rigs.'' We will get to
signal flow graphs in \cref{subsec.icons_sfgs}.

\subsection{Rigs}%
\index{rig|(}
Signals can be amplified, and they can be added. Adding and amplification
interact via a distributive law, as follows: if we add two signals, and then amplify them by some amount $a$, it should be the same as amplifying the two signals
separately by $a$, then adding the results. 

We can think of all the possible amplifications as forming a structure called a rig,%
\footnote{Rigs are also known as \emph{semi-rings}.}
defined as follows.

\begin{definition}%
\label{def.rig}%
\index{rig}
A \emph{rig} is a tuple $(R,0,+,1,*)$, where $R$ is a set, $0,1\in R$ are elements, and $+,*\colon R
\times R \to R$ are functions, such that 
\begin{enumerate}[label=(\alph*)]
	\item $(R,+,0)$ is a commutative monoid,
	\item $(R,*,1)$ is a monoid,%
	\tablefootnote{
	Note that we did not demand that $(R,*,1)$ be commutative; we will see a naturally-arising example where it is not commutative in \cref{ex.mat_rig}.
	}
	and	
	\item $a*(b+c)=a* b +a* c$ and $(a+b)*c=a*c+b*c$ for all $a,b,c \in R$.
	\item $a*0=0=0*a$ for all $a\in R$.
\end{enumerate}
\end{definition}

We have already encountered many examples of rigs.
\begin{example}%
\label{ex.rig_nat}%
\index{natural numbers!as rig}
The natural numbers form a rig $(\nn,0,+,1,*)$.
\end{example}

\begin{example}%
\index{booleans!as rig}
The Booleans form a rig $(\bb,\false,\vee,\true,\wedge)$.
\end{example}

\begin{example}%
\label{ex.quantale_as_rig}%
\index{quantale}
Any quantale $\cat{V}=(V,\leq,I,\otimes)$ determines a rig $(V,0,\vee,I,\otimes)$, where $0=\bigvee\varnothing$ is the empty join. See \cref{def.monoidal_closed}.
\end{example}

\begin{example}%
\label{ex.mat_rig}%
\index{matrices!rig of}%
\index{rig!matrices as}
If $R$ is a rig and $n\in\nn$ is any natural number, then the set
$\Set{Mat}_n(R)$ of $(n\times n)$-matrices in $R$ forms a rig. A matrix
$M\in\Set{Mat}_n(R)$ is a function $M\colon \ord{n}\times\ord{n} \to R$.
Addition $M+N$ of matrices is given by $(M+N)(i,j)\coloneqq M(i,j)+N(i,j)$ and multiplication $M*N$ is given by $(M*N)(i,j)\coloneqq\sum_{k\in\ord{n}}M(i,k)*N(k,j)$. The $0$-matrix is $0(i,j)\coloneqq 0$ for all $i,j\in\ord{n}$. Note that $\Set{Mat}_n(R)$ is generally not commutative.
\end{example}

\begin{exercise}%
\label{exc.rigs_mats}
\begin{enumerate}
	\item We said in \cref{ex.mat_rig} that for any rig $R$, the set $\Set{Mat}_n(R)$ forms a rig. What is its multiplicative identity $1\in\Set{Mat}_n(R)$?
	\item We also said that $\Set{Mat}_n(R)$ is generally not commutative. Pick an $n$ and show that that $\Set{Mat}_n(\NN)$ is not commutative, where $\NN$ is as in \cref{ex.rig_nat}.
	\qedhere
\end{enumerate}
\end{exercise}

The following is an example for readers who are familiar with the algebraic structure known as ``rings.''

\begin{example}%
\index{rig!vs.\ ring}
Any ring forms a rig. In particular, the real numbers $(\rr,0,+,1,*)$ are a
rig. The difference between a ring and rig is that a ring, in addition to
all the properties of a rig, must also have additive inverses, or
\emph{negatives}. A common mnemonic is that a rig is a ri\textbf{n}g without
\textbf{n}egatives.
\end{example}

\index{rig|)}

\subsection{The iconography of signal flow graphs}%
\label{subsec.icons_sfgs}
\index{signal flow graph!simplified}%
\index{icon|(}

A signal flow graph is supposed to keep track of the amplification, by elements
of a rig $R$, to which signals are subjected. While not strictly necessary,%
\footnote{The necessary requirement for the material below to make sense is that
  the signals take values in an \emph{$R$-module} $M$. We will not discuss this
here, keeping to the simpler requirement that $M=R$.} we will assume the signals
themselves are elements of the same rig $R$. We refer to elements of $R$ as
\emph{signals} for the time being.

Amplification of a signal by some value $a \in R$ is simply depicted like
so:
\[\tag{scalar mult.}
\begin{aligned}

  \end{aligned}
\]
We interpret the above icon as a depicting a system where a signal enters on the left-hand wire, is multiplied by $a$, and is output on the right-hand wire.

What is more interesting than just a single signal amplification, however, is
the interaction of signals. There are four other important icons in signal flow
graphs.
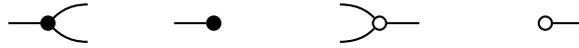
\[
%
\]
Let's go through them one by one. The first two are old friends from
\cref{chap.resource_theory}: copy and discard.
\[\tag{copy}
\begin{aligned}
%
\end{aligned}
\]
We interpret this diagram as taking in an input signal on the left, and
outputting that same value to both wires on the right. It is basically the
``copy'' operation from \cref{subsubsec.informatics}.

Next, we have the ability to discard signals.
\[\tag{discard}
\begin{aligned}
%
\end{aligned}
\]
This takes in any signal, and outputs nothing. It is basically the ``waste'' operation from \cref{subsubsec.manufacturing}.

Next, we have the ability to add signals.
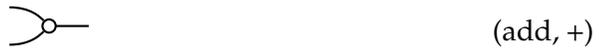
\[\tag{add, +}
\begin{aligned}
%
\end{aligned}
\]
This takes the two input signals and adds them, to produce a single output
signal.  

Finally, we have the zero signal.
\[\tag{zero, 0}
\begin{aligned}
%
\end{aligned}
\]
This has no inputs, but always outputs the $0$ element of the rig.

Using these icons, we can build more complex signal flow graphs.  To compute the
operation performed by a signal flow graph we simply trace the paths with the
above interpretations, plugging outputs of one icon into the inputs of the next
icon.

For example, consider the rig $R =\nn$ from \cref{ex.rig_nat}, where the scalars
are the natural numbers. Recall the signal flow graph from \cref{eq.examplesfg}
in the introduction:
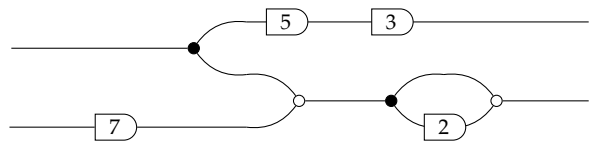
\[
  \begin{aligned}
%
\end{aligned}
\]
As we explained, this takes in two input signals $x$ and $y$, and returns two
output signals $a=15x$ and $b=3x+21y$.

In addition to tracing the processing of the values as they move forward through
the graph, we can also calculate these values by summing over paths. More
explicitly, to get the contribution of a given input wire to a given output
wire, we take the sum, over all paths $p$ joining the wires, of the total amplification along that
path.

So, for example, there is one path from the top input to the top output.
On this path, the signal is first copied, which does not affect its value, then
amplified by 5, and finally amplified by 3. Thus, if $x$ is the first input
signal, then this contributes $15x$ to the first output. Since there is no path
from the bottom input to the top output (one is not allowed to traverse paths
backwards), the signal at the first output is exactly $15x$. Both inputs contribute to the bottom output. In fact, each input contributes in
two ways, as there are two paths to it from each input. The top input thus
contributes $3x=x+2x$, whereas the bottom input, passing through an additional
$\ast 7$ amplification, contributes $21y$.

\begin{exercise}%
\label{exc.a_signal_flow_graph}
  The following flow graph   takes in two natural numbers $x$ and $y$
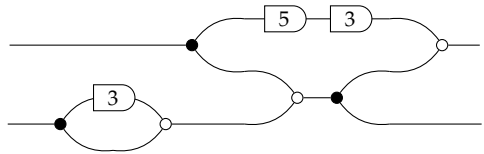
\[
\begin{tikzpicture}[scale=.7]
	\begin{pgfonlayer}{nodelayer}
		\node [style=none] (0) at (-6, -0) {};
		\node [style=bdot] (1) at (-5, -0) {};
		\node [style=wamp] (2) at (-4, 0.5) {$\scriptstyle 3$};
		\node [style=none] (3) at (-4, -0.5) {};
		\node [style=wdot] (4) at (-3, -0) {};
		\node [style=none] (5) at (-6, 1.5) {};
		\node [style=bdot] (6) at (-2.5, 1.5) {};
		\node [style=none] (7) at (-1.5, 1) {};
		\node [style=wdot] (8) at (-0.5, 0.5) {};
		\node [style=none] (9) at (-1.5, 2) {};
		\node [style=wamp] (10) at (-0.75, 2) {$\scriptstyle 5$};
		\node [style=wdot] (11) at (2.25, 1.5) {};
		\node [style=none] (12) at (1.25, 1) {};
		\node [style=bdot] (13) at (0.25, 0.5) {};
		\node [style=none] (14) at (1.25, -0) {};
		\node [style=none] (15) at (3, -0) {};
		\node [style=wamp] (16) at (0.5, 2) {$\scriptstyle {3}$};
		\node [style=none] (17) at (1.25, 2) {};
		\node [style=none] (18) at (-1.5, -0) {};
		\node [style=none] (19) at (3, 1.5) {};
	\end{pgfonlayer}
	\begin{pgfonlayer}{edgelayer}
		\draw (0.center) to (1);
		\draw [bend left, looseness=1.00] (1) to (2.center);
		\draw [bend right, looseness=1.00] (1) to (3.center);
		\draw [bend left, looseness=1.00] (2.center) to (4);
		\draw [bend right, looseness=1.00] (3.center) to (4);
		\draw (4) to (18.center);
		\draw [bend right, looseness=1.00] (18.center) to (8);
		\draw [bend right, looseness=1.00] (8) to (7.center);
		\draw [bend left, looseness=1.00] (7.center) to (6);
		\draw (6) to (5);
		\draw [bend left, looseness=1.00] (6) to (9.center);
		\draw (9.center) to (10.center);
		\draw (10.center) to (16.center);
		\draw (16.center) to (17.center);
		\draw [bend left, looseness=1.00] (17.center) to (11);
		\draw [bend left, looseness=1.00] (11) to (12.center);
		\draw (11) to (19);
		\draw [bend right, looseness=1.00] (12.center) to (13);
		\draw (8) to (13);
		\draw [bend right, looseness=1.00] (13) to (14.center);
		\draw (14.center) to (15.center);
	\end{pgfonlayer}
\end{tikzpicture}
\]
  and produces two output signals. What are they?
\end{exercise}

\begin{example}%
\index{differential equation}
This example is for those who have some familiarity with differential equations. A linear system of differential equations provides a simple way to specify the movement of a particle. For example, consider a particle whose position $(x,y,z)$ in $3$-dimensional space is determined by the following equations:
\begin{align*}
	\dot{x}+3\ddot{y}-2z=0\\
	\ddot{y}+5\dot{z}=0
\end{align*}
Using what is known as the Laplace transform, one can convert this into a linear system involving a formal variable $D$, which stands for ``differentiate.'' Then the system becomes
\begin{align*}
	Dx+3D^2y-2z=0\\
	D^2y+5Dz=0
\end{align*}
which can be represented by the signal flow graph
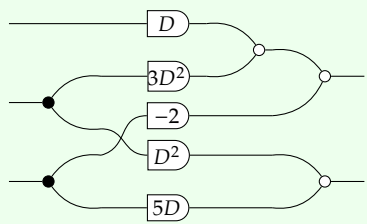
\[

\qedhere
\]
\end{example}

\paragraph{Signal flow graphs as morphisms in a free prop.}%
\index{signal flow graph!as morphism}

We can formally define simplified signal flow graphs using props.

\begin{definition}%
\label{def.sig_flow_graph_gens}
Let $R$ be a rig (see \cref{def.rig}). Consider the set 
\[
G_R := \left\{\add{.05\textwidth}, \zero{.05\textwidth},\comult{.05\textwidth},
\counit{.05\textwidth}\right\} \cup \left\{\scalar{.07\textwidth} \mid a \in R\right\},
\]
and let $s,t\colon G_R \to \nn$ be given by the number of dangling wires on the
left and right of the generator icon respectively.  A \emph{simplified signal flow graph}
is a morphism in the free prop $\free(G_R)$ on this set $G_R$ of generators. We
define $\sfg_R\coloneq \free(G_R)$.
\end{definition}

For now we'll drop the term `simplified', since these are the only sort of
signal flow graph we know. We'll return to signal flow graphs in their full
glory---i.e.\ including feedback---in \cref{subsec.full_SFGs}.

\begin{example}
To be more in line with our representations of both wiring diagrams and port
graphs, morphisms in $\free(G_R)$ should be drawn slightly differently. For example, technically the signal flow graph from
\cref{exc.a_signal_flow_graph} should be drawn as follows:
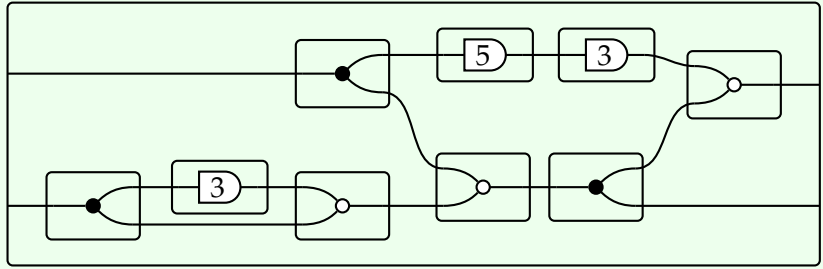
\[
\begin{tikzpicture}[oriented WD, spider diagram, bbx=3ex, bby=2ex]
	\node[spider={1}{2}, fill=black] (black1) {};
	\node[bb={0}{0}, fit=(black1)] (black1outer) {};
	\node[spider={2}{1}, fill=white, right=6 of black1] (white1) {};
	\node[bb={0}{0}, fit=(white1)] (white1outer) {};
	\node at ($(black1_out1)!.5!(white1_in1)$) (helper1) {};
	\node[spider={1}{1}, wamp] at (helper1) (scalar1) {$3$};
	\node[bby=.75ex, bbx=2ex, bb={0}{0}, fit=(scalar1)] (scalar1outer) {};
	\node[spider={1}{2}, fill=black, above=4.5 of white1] (black2) {}; 
	\node[bb={0}{0}, fit=(black2)] (black2outer) {};
	\node[spider={2}{1}, fill=white, right=6 of scalar1] (white2) {};
	\node[bb={0}{0}, fit=(white2)] (white2outer) {};
	\node at (black2_out1-|white2) (helper2) {};
	\node[spider={1}{1}, wamp] at (helper2) (scalar2) {$5$};
	\node[bby=.75ex, bbx=2ex, bb={0}{0}, fit=(scalar2)] (scalar2outer) {};
	\node[spider={1}{1}, wamp, right=2 of scalar2] (scalar3) {$3$};
	\node[bby=.75ex, bbx=2ex, bb={0}{0}, fit=(scalar3)] (scalar3outer) {};
	\node[spider={1}{2}, fill=black, right=2.5 of white2] (black3) {};
	\node[bb={0}{0}, fit=(black3)] (black3outer) {};
	\node[spider={2}{1}, fill=white, below right=.3 and 3 of scalar3] (white3) {};
	\node[bb={0}{0}, fit=(white3)] (white3outer) {};
	\node[bb={0}{0}, fit=(black1outer) (scalar3outer) (white3outer)] (outer) {};
	\draw (outer.west|-black1_in1) to (black1_in1);
	\draw (outer.west|-black2_in1) to (black2_in1);
	\draw (black1_out1) to (scalar1_in1);
	\draw (black1_out2) to (white1_in2);
	\draw (scalar1_out1) to (white1_in1);
	\draw (black2_out1) to (scalar2_in1);
	\draw (black2_out2) to (white2_in1);
	\draw (white1_out1) to (white2_in2);
	\draw (scalar2_out1) to (scalar3_in1);
	\draw (scalar3_out1) to (white3_in1);
	\draw (white2_out1) to (black3_in1);
	\draw (black3_out1) to (white3_in2);
	\draw (black3_out2) to (black3_out2-|outer.east);
	\draw (white3_out1) to (white3_out1-|outer.east);
\end{tikzpicture}
\]
because we said we would label boxes with the elements of $G$. But it is easier on the eye to draw remove the boxes and just look at the icons inside as in \cref{exc.a_signal_flow_graph},
and so we'll draw our diagrams in that fashion.
\end{example}
\index{icon|)}

More importantly, props provide language to understand the semantics of
signal flow graphs. Although the signal flow graphs themselves are free props,
their semantics---their meaning in our model of signals flowing---will arise
when we add equations to our props, as in \cref{rdef.presentation_prop}. These
equations will tell us when two signal flow graphs act the same way on signals.%
\index{signal flow graph!semantics of}
For example,
  \begin{equation}%
\label{eqn.equivalence.rand89}
\begin{aligned}
\begin{tikzpicture}[spider diagram]
	\node[spider={1}{2}, fill=black] (a) {};
	\node[special spider={1}{2}{0}{\leglen}, fill=black, right=.1 of a_out1] (b) {};
	\draw (a_out1) to (b_in1);
	\draw (a_out2) to (b_out1|-a_out2);
\end{tikzpicture}
\end{aligned}
    \qquad
    \mbox{and}\qquad
\begin{aligned}
\begin{tikzpicture}[spider diagram]
	\node[spider={1}{2}, fill=black] (a) {};
	\node[special spider={1}{2}{0}{\leglen}, fill=black, right=.1 of a_out2] (b) {};
	\draw (a_out2) to (b_in1);
	\draw (a_out1) to (b_out1|-a_out1);
\end{tikzpicture}
\end{aligned}
  \end{equation}
both express the same behavior: a single input signal is copied twice so that
three identical copies of the input signal are output.

If two signal flow graphs $S,T$ are almost the same, with the one exception
being that somewhere we replace the left-hand side of
\cref{eqn.equivalence.rand89} with the right-hand side, then $S$ and $T$ have
the same behavior. But there are other replacements we could make to a signal
flow graph that do not change its behavior. Our next goal is to find a complete
description of these replacements.

\subsection{The prop of matrices over a rig}%
\index{rig!matrices over}
Signal flow graphs are closely related to matrices. In previous chapters we showed how a
matrix with values in a quantale $\cat{V}$---a closed monoidal preorder with all
joins---represents a system of interrelated points and connections between
them, such as a profunctor.  The quantale gave us the structure and axioms we
needed in order for matrix multiplication to work properly. But we know from
\cref{ex.quantale_as_rig} that quantales are examples of rigs, and in fact
matrix multiplication makes sense in any rig $R$. In \cref{ex.mat_rig}, we
explained that the set $\Set{Mat}_n(R)$ of $(n\times n)$-matrices in $R$ can
naturally be assembled into a rig, for any fixed choice of $n\in\nn$. But what
if we want to do better, and assemble \emph{all} matrices into a single
algebraic structure? The result is a prop!

An \emph{$(m\times n)$-matrix $M$ with values in $R$} is a function
$M\colon(\ord{m}\times\ord{n}) \to R$.  Given an $(m\times n)$-matrix $M$ and an $(n
\times p)$-matrix $N$, their \emph{composite} is the $(m\times p)$-matrix $M\cp N$
defined as follows for any $a \in \ord{m}$ and $c \in \ord{p}$:
\begin{equation}%
\label{eqn.matrix_mult_again}
  M\cp N(a,c) \coloneqq \sum_{b \in \ord{n}} M(a,b)\times N(b,c),
\end{equation}
Here the $\sum_{b\in \ord{n}}$ just
means repeated addition (using the rig $R$'s $+$ operation), as usual.

\begin{remark}
Conventionally, one generally considers a matrix $A$ acting on a vector $v$ by multiplication in the order $Av$, where $v$ is a column vector. In keeping with our composition convention, we use the opposite order, $v\cp A$, where $v$ is a row vector. See for example \cref{eqn.gens_as_matrices} for when this is implicitly used.
\end{remark}%
\index{vector}

\begin{definition}%
\label{def.prop_matrices}%
\index{prop!of matrices}
  Let $R$ be a rig. We define the \emph{prop of $R$-matrices}, denoted $\mat(R)$,
  to be the prop whose morphisms $m\to n$ are the $(m\times n)$-matrices
  with values in $R$. Composition of morphisms is given
  by matrix multiplication as in \cref{eqn.matrix_mult_again}. The monoidal product is given by the direct sum of
  matrices: given matrices $A\colon m \to n$ and $b\colon p \to q$, we define
  $A+B\colon m+p \to n+q$ to be the block matrix
  \[

  \]
  where each $0$ represents a matrix of zeros of the appropriate dimension ($m\times q$ and $n\times p$). We refer to any combination of multiplication and direct sum as a \emph{interconnection} of matrices.
\end{definition}

\begin{exercise}%
\label{exc.monoidal_mat_prod}
Let $A$ and $B$ be the following matrices with values in $\NN$:
\[
A=\left(
\begin{array}{ccc}
	3&3&1\\
	2&0&4
\end{array}
\right)
\hspace{1in}
B=\left(
\begin{array}{cccc}
	2&5&6&1\\
\end{array}
\right).
\]
What is the direct sum matrix $A+B$?
\end{exercise}

\subsection{Turning signal flow graphs into matrices}%
\index{matrix!associated to a signal flow graph|(}
Let's now consider more carefully what we mean when we talk about the meaning, or \emph{semantics},
of each signal flow graph. We'll use matrices.
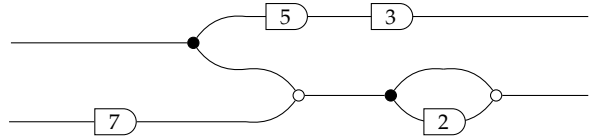
\[
%
\]
In the examples like the above (copied from \cref{eq.examplesfg}), the signals emanating from output wires, say $a$ and $b$, are given by certain sums of
amplified input values, say $x$ and $y$. If we can only measure the input and output signals, and
care nothing for what happens in between, then each signal flow graph may as
well be reduced to a matrix of amplifications. We can represent the signal flow graph of \cref{eq.examplesfg} by either the matrix on the left (for more detail) or the matrix on the right if the labels are clear from context:
\[
  \begin{array}{c|cc}
    &a&b\\\hline
    x & 15 & 3\\
    y & 0 & 21\\
  \end{array}%
\]

\paragraph{Every signal flow graph can be interpreted as a matrix.}

The generators $G_R$ from \cref{def.sig_flow_graph_gens} are shown again in the table below, where each is interpreted as a matrix. For example, we interpret amplification by $a \in R$ as the
$1\times 1$ matrix $(a)\colon 1 \to 1$: it is an operation that takes an input
$x\in R$ and returns $a*x$.  Similarly, we can interpret $\add{1.2em}$ as the $2\times 1$ matrix $\left(%
\right)$: it is an operation that takes a row vector consisting of two
inputs, $x$ and $y$, and returns $x+y$.  Here is a table showing the
interpretation of each generator.%
\index{vector}

\begin{equation}%
\label{eqn.gens_as_matrices}
\renewcommand{\arraystretch}{1.4}
\begin{tabular}{c|c|c|c}
  generator & 
  icon & 
  matrix & 
  arity 
\\\hline
  amplify by $a \in R$ & 
  $\scalar{.07\textwidth}$  & \renewcommand{\arraystretch}{.9}
  $\begin{pmatrix} a\end{pmatrix}$ &
  $1 \to 1$
\\
  add & 
  $\add{2em}$ & 
  \renewcommand{\arraystretch}{.9} 
  $\begin{pmatrix} 1 \\ 1 \end{pmatrix}$ & 
  $2 \to 1$
\\
  zero & 
  $\zero{2em}$ & 
  $()$ &
  $0 \to 1$
\\
  copy & 
  $\comult{2em}$ &
  \renewcommand{\arraystretch}{.9} 
  $\begin{pmatrix} 1 & 1 \end{pmatrix}$ &
  $1 \to 2$
\\
  discard &  
  $\counit{2em}$ &
  $()$ &
  $1 \to 0$
\end{tabular}
\end{equation}
Note that both zero and discard are represented by empty matrices, but of
differing dimensions. In linear algebra it is unusual to consider matrices of
the form $0\times n$ or $n\times 0$ for various $n$ to be different, but they can be kept distinct for bookkeeping purposes: you can multiply a $0\times 3$ matrix by a $3\times n$ matrix for any
$n$, but you can not multiply it by a $2\times n$ matrix.

Since signal flow graphs are morphisms in a free prop, the table in
\eqref{eqn.gens_as_matrices} is enough to show that we can
interpret any signal flow diagram as a matrix.

\begin{theorem}%
\label{thm.sfg_to_mat}
There is a prop functor $S\colon \sfg_R \to \mat(R)$ that sends the generators $g\in G$
icons to the matrices as described in Table \ref{eqn.gens_as_matrices}.
\end{theorem}
\begin{proof}
This follows immediately from the universal property of free props, \cref{rem.free_prop_universal_property}.
\end{proof}

We have now constructed a matrix $S(g)$ from any signal flow graph $g$. But how can we produce this matrix explicitly? Both for the example signal flow graph in \cref{eq.examplesfg}
and for the generators in \cref{def.sig_flow_graph_gens}, the associated matrix
has dimension $m\times n$, where $m$ is the number of inputs and $n$ the number of outputs,
with $(i,j)$th entry describing the amplification of the $i$th input that
contributes to the $j$th output. This is how one would hope or expect the
functor $S$ to work in general; but does it? We have used a big hammer---the
universal property of free constructions---to obtain our functor $S$. Our next
goal is to check that it works in the expected way. Doing so is a matter of
using induction over the set of prop expressions, as we now see.%
\footnote{Mathematical induction is a formal proof technique that can be thought of like a domino rally: if you knock over all the starting dominoes, and you're sure that each domino will be knocked down if its predecessors are, then you're sure every domino will eventually fall. If you want more rigor, or you want to understand the proof of \cref{prop.Simages} as a genuine case of induction, ask a friendly neighborhood mathematician!}%
\index{induction}

\begin{proposition} %
\label{prop.Simages}
  Let $g$ be a signal flow graph with $m$ inputs and $n$ outputs. The matrix
  $S(g)$ is the $(m\times n)$-matrix whose $(i,j)$-entry describes the
  amplification of the $i$th input that contributes to the $j$th output.
\end{proposition}
\begin{proof}
  Recall from \cref{def.prop_expressions} that an arbitrary $G_R$-generated prop expression is built from the morphisms
  $\id_0\colon 0 \to 0$, $\id_1\colon 1 \to 1$, $\sigma\colon 2 \to
  2$, and the generators in $G_R$, using the following two rules:
  \begin{itemize}
    \item if $\alpha\colon m \to n$ and $\beta\colon p \to q$ are expressions, then
      $(\alpha+\beta)\colon (m+p) \to (n+q)$ is an expression.
    \item if $\alpha\colon m \to n$ and $\beta\colon n \to p$ are expressions, then
      $\alpha\cp\beta\colon m \to p$ is an expression.
  \end{itemize}
  $S$ is a prop functor by \cref{thm.sfg_to_mat}, which by \cref{def.prop_functor}
  must preserve identities, compositions, monoidal products, and symmetries.
  We first show that the proposition is true when $g$ is equal
  to $\id_0$, $\id_1$, and $\sigma$.
  
  The empty
  signal flow graph $\id_0\colon 0 \to 0$ must be sent to the unique (empty) matrix
  $()\colon 0 \to 0$. The morphisms $\id_1$, $\sigma$, and $a\in R$ map to the identity matrix, the swap matrix, and the scalar matrix $(a)$ respectively:%
\index{identity!matrix}
  \[
    \idone{.08\textwidth}\;\mapsto\; \begin{pmatrix}1\end{pmatrix}
    \qquad\text{and}\qquad
    \swap{.06\textwidth}\;\mapsto\; \begin{pmatrix}0 & 1 \\ 1 & 0\end{pmatrix}
    \qquad\text{and}\qquad
    \scalar{3em}  \;\mapsto\; \begin{pmatrix}a\end{pmatrix}   
  \]
 In each case, the $(i,j)$-entry gives the amplification of the $i$th input to the $j$th output.
 
 It remains to show that if the proposition holds for
  $\alpha\colon m \to n$ and $\beta\colon p \to q$, then it holds for
  (i) $\alpha\cp\beta$ (when $n=p$) and for (ii) $\alpha+\beta$ (in general).

  To prove (i), consider the following picture of $\alpha\cp\beta$:
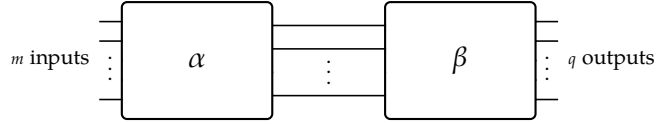
\[
  \begin{tikzpicture}[oriented WD, bb port length=0pt, bb port sep=1]
	\node[bb={5}{4}, minimum width = 2cm] (X) {$\alpha$};
	\node[bb={4}{5}, right= 1 of X, minimum width = 2cm] (Y) {$\beta$};
	\draw ($(X_in1)-(.2,0)$) to (X_in1);
	\draw ($(X_in2)-(.2,0)$) to (X_in2);
	\draw ($(X_in5)-(.2,0)$) to (X_in5);
	\draw (X_out1) to (Y_in1);
	\draw (X_out2) to (Y_in2);
	\draw (X_out4) to (Y_in4);
	\draw (Y_out1) to ($(Y_out1)+(.2,0)$);
	\draw (Y_out2) to ($(Y_out2)+(.2,0)$);
	\draw (Y_out5) to ($(Y_out5)+(.2,0)$);
	\draw[label]
		node at ($.5*(X_out3)+.5*(Y_in3)+(0,3pt)$) {$\vdots$}
		node[left=3pt of X_in3] {$\vdots$}
		node[left=12pt of X_in3] {\scriptsize $m$ inputs}
		node[right=3pt of Y_out3] {$\vdots$}
		node[right=12pt of Y_out3] {\scriptsize $q$ outputs}
	;	
\end{tikzpicture}
\]
Here $\alpha\colon m \to n$ and $\beta\colon n \to q$ are signal flow graphs,
assumed to obey the proposition. Consider the $i$th input and $k$th output of
$\alpha\cp\beta$; we'll just call these $i$ and $k$.  We want to show that the amplification that $i$
contributes to $k$ is the sum---over all paths from $i$ to
$k$---of the amplification along that path. So let's also fix some $j \in \ord{n}$, and consider paths from $i$ to $k$
that run through $j$. By distributivity of the rig $R$, the total amplification from $i$ to $k$ through $j$
is the total amplification over all paths from $i$ to $j$ times the total
amplication over all paths from $j$ to $k$. Since all paths from $i$ to $k$ must
run through some $j$th output of $\alpha$/input of $\beta$, the amplification
that $i$ contributes to $k$ is
\[
  \sum_{j \in \ord{n}} \alpha(i,j)\ast\beta(j,k).
\]
This is exactly the formula for matrix multiplication, which is composition $S(\alpha)\cp S(\beta)$ in the prop $\mat(R)$; see \cref{def.prop_matrices}. So $\alpha\cp\beta$ obeys
the proposition when $\alpha$ and $\beta$ do.

  Proving (ii) is more straightforward. The monoidal product $\alpha+\beta$ of signal flow graphs looks
  like this:
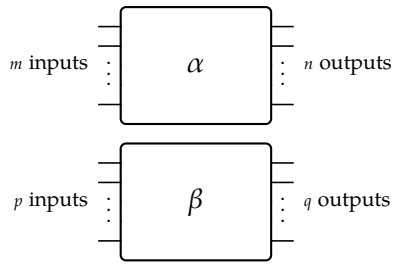
\[
  \begin{tikzpicture}[oriented WD, bb port length=0pt, bb port sep=1]
	\node[bb={5}{5}, minimum width = 2cm] (X) {$\alpha$};
	\node[bb={5}{5}, below= 1 of X, minimum width = 2cm] (Y) {$\beta$};
	\draw ($(X_in1)-(.2,0)$) to (X_in1);
	\draw ($(X_in2)-(.2,0)$) to (X_in2);
	\draw ($(X_in5)-(.2,0)$) to (X_in5);
	\draw (X_out1) to ($(X_out1)+(.2,0)$);
	\draw (X_out2) to ($(X_out2)+(.2,0)$);
	\draw (X_out5) to ($(X_out5)+(.2,0)$);
	\draw ($(Y_in1)-(.2,0)$) to (Y_in1);
	\draw ($(Y_in2)-(.2,0)$) to (Y_in2);
	\draw ($(Y_in5)-(.2,0)$) to (Y_in5);
	\draw (Y_out1) to ($(Y_out1)+(.2,0)$);
	\draw (Y_out2) to ($(Y_out2)+(.2,0)$);
	\draw (Y_out5) to ($(Y_out5)+(.2,0)$);
	\draw[label]
		node[left=3pt of X_in3] {$\vdots$}
		node[left=12pt of X_in3] {\scriptsize $m$ inputs}
		node[right=3pt of X_out3] {$\vdots$}
		node[right=12pt of X_out3] {\scriptsize $n$ outputs}
		node[left=3pt of Y_in3] {$\vdots$}
		node[left=12pt of Y_in3] {\scriptsize $p$ inputs}
		node[right=3pt of Y_out3] {$\vdots$}
		node[right=12pt of Y_out3] {\scriptsize $q$ outputs}
	;	
\end{tikzpicture}
\]
No new paths are created; the only change is to reindex the inputs and outputs.
In particular, the $i$th input of $\alpha$ is the $i$th input of $\alpha+\beta$,
the $j$th output of $\alpha$ is the $j$th output of $\alpha+\beta$, the $i$th
input of $\beta$ is the $(m+i)$th output of $\alpha+\beta$, and the $j$th output
of $\beta$ is the $(n+j)$th output of $\alpha+\beta$. This means that the matrix
with $(i,j)$th entry describing the amplification of the $i$th input that
contributes to the $j$th output is $S(\alpha)+S(\beta)=S(\alpha+\beta)$, as in \cref{def.prop_matrices}. This proves the proposition.
\end{proof}

\begin{exercise}%
\label{exc.sig_flow_mats}
\begin{enumerate}
	\item  What matrix does the signal flow graph
  
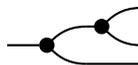
\[
\begin{aligned}
\begin{tikzpicture}[spider diagram]
	\node[spider={1}{2}, fill=black] (a) {};
	\node[special spider={1}{2}{0}{\leglen}, fill=black, right=.1 of a_out1] (b) {};
	\draw (a_out1) to (b_in1);
	\draw (a_out2) to (b_out1|-a_out2);
\end{tikzpicture}
\end{aligned}
  \]
  represent?
  \item What about the signal flow graph
  
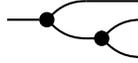
\[
\begin{aligned}
\begin{tikzpicture}[spider diagram]
	\node[spider={1}{2}, fill=black] (a) {};
	\node[special spider={1}{2}{0}{\leglen}, fill=black, right=.1 of a_out2] (b) {};
	\draw (a_out2) to (b_in1);
	\draw (a_out1) to (b_out1|-a_out1);
\end{tikzpicture}
\end{aligned}
  \]
  \item Are they equal?
\qedhere
\end{enumerate}
\end{exercise}
\index{matrix!associated to a signal flow graph|)}

\subsection{The idea of functorial semantics} %
\index{functorial semantics}

Let's pause for a moment to reflect on what we have just learned. First, signal
flow diagrams are the morphisms in a prop. This means we have two special
operations we can do to form new signal flow diagrams from old, namely
composition (combining in series) and monoidal product (combining in parallel). We might think of this as
specifying a `grammar' or `syntax' for signal flow diagrams. 

As a language, signal flow graphs have not only syntax but also semantics: each signal flow diagram can be interpreted as a matrix. Moreover, matrices have the same grammatical structure:
they form a prop, and we can construct new matrices from old using composition
and monoidal product. In \cref{thm.sfg_to_mat} we completed this
picture by showing that semantic interpretation is a prop functor between the
prop of signal flow graphs and the prop of matrices. Thus we say that matrices
give \emph{functorial semantics} for signal flow diagrams.

Functorial semantics is a key manifestation of compositionality. It says that
the matrix meaning $S(g)$ for a big signal flow graph $g$ can be computed by:
\begin{enumerate}
	\item splitting $g$ up into little pieces,
	\item computing the very simple matrices for each piece, and
	\item using matrix multiplication and direct sum to put the pieces back together to obtain the desired meaning, $S(g)$.
\end{enumerate}
This functoriality is useful in practice, for example in
speeding up computation of the semantics of signal flow graphs: for large signal
flow graphs, composing matrices is much faster than tracing paths.

\section{Graphical linear algebra}

In this section we will begin to develop something called graphical linear
algebra, which extends the ideas above. This formalism is actually quite
powerful. For example, with it we can easily and \emph{graphically} prove
certain conjectures from control theory that, although they were eventually
solved, required fairly elaborate matrix algebra arguments
\cite{Fong.Sobocinski.Rapisardo:2016a}.

\subsection{A presentation of $\mat(R)$}%
\index{presentation!of linear algebra}

Let $R$ be a rig, as defined in \cref{def.rig}. The main theorem of the previous section, \cref{thm.sfg_to_mat}, provided a functor $S\colon \sfg_R \to \mat(R)$ that converts any signal flow graph into a matrix. Next we show that $S$ is ``full'': that any matrix can be represented by a signal flow graph. 

\begin{proposition} %
\label{prop.Sfull}
Given any matrix $M \in \mat(R)$, there exists a signal flow graph $g \in \sfg_R$
such that such that $S(g)=M$.
\end{proposition}
\begin{proof}[Proof sketch]
Let $M \in \mat(R)$ be an $(m \times n)$-matrix. We want a signal flow
graph $g$ such that $S(g)=M$.  In particular, to compute
$S(g)(i,j)$, we know that we can simply compute the amplification that the $i$th
input contributes to the $j$th output. The key idea then is to construct $g$ so that there is exactly one path from
$i$th input to the $j$th output, and that this path has
exactly one scalar multiplication icon, namely $M(i,j)$.

The general construction is a little technical (see
\cref{exc.general_case_S_full}), but the idea is clear from just considering the
case of $2 \times 2$-matrices. Suppose $M$ is the $2\times2$-matrix
$(
)$. Then we define $g$ to
be the signal flow graph
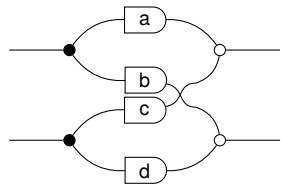
\begin{equation}%
\label{eqn.two_by_two}
%
\end{equation}
Tracing paths, it is easy to see that $S(g)=M$. Note that $g$ is the composite
of four layers, each layer respectively a monoidal product of (i) copy and discard maps,
(ii) scalar multiplications, (iii) swaps and identities, (iv) addition and zero maps.

For the general case, see \cref{exc.general_case_S_full}.
\end{proof}

\begin{exercise} %
\label{ex.drawsfgs}
  Draw signal flow graphs that represent the following matrices:\\
  \begin{enumerate*}[itemjoin=\hspace{1in}]
  \item[] \hspace{-.2in}
    \item $
      %
$
  \end{enumerate*}
\end{exercise}

\begin{exercise}%
\label{exc.general_case_S_full}
Write down a detailed proof of \cref{prop.Sfull}. Suppose $M$ is an $m \times
n$-matrix. Follow the idea of the $(2\times 2)$-case in \cref{eqn.two_by_two},
and construct the signal flow graph $g$---having $m$ inputs and $n$ outputs---as
the composite of four layers, respectively comprising (i) copy and discard maps, (ii)
scalars, (iii) swaps and identities, (iv) addition and zero maps.
\erase{
\begin{itemize}
  \item For the first layer $g_1$, take
    \[
      g_1\coloneq c_n +\dots+ c_n\colon m \to (m \times n),
    \]
    where we take the monoidal product of $m$ copies of $c_n$, and 
    \[
      c_n\coloneq
      \comult{1em}.(1+\comult{1em}).(1+1+\comult{1em}).\dots.(1+\dots+1+\comult{1em})\colon
      1 \to n
    \]
    is the signal flow diagram that makes $n$ copies of a single input. 
  \item Next, define
    \begin{align*}
      g_2\coloneq &\quad s_{M(1,1)}+\dots+s_{M(1,n)} \\
      & +s_{M(2,1)}+\dots+s_{M(2,n)} \\
      &+\dots \\
      &+s_{M(m,1)}+\dots+s_{M(m,n)}\colon (m\times n) \to (m\times n),
    \end{align*}
    where $s_a\colon 1 \to 1$ is the scalar multiplication by $a$ signal flow
    graph generator. This layer amplifies each copy of the input signal by the
    relevant rig element.
  \item The third layer rearranges wires. We will not write this down
    explicitly, but simply say it is the signal flow graph $g_3\colon m \times
    n \to m \times n$, that is the
    composite and monoidal product of swap and identity maps, such that the
    $(i-1)m+j$th input maps to the $(j-1)n+i$th output, where $1 \le i \le n$
    and $1 \le j \le m$.
  \item Finally, the fourth layer is similar to the first, but instead adds the
    amplified input signals. We define 
    \[
      g_4\coloneq a_m+\dots+a_m \colon (m \times n)\to n,
    \]
    where 
    \[
      a_m\coloneq
      (1+\dots+1+\add{1em}).\dots.(1+1+\add{1em}).(1+\add{1em}).\add{1em}\colon
      m \to 1
    \]
    is the signal flom graph that adds $m$ inputs to produce a single output
\end{itemize}
Using \cref{prop.Simages}, it is easily checked that $g=g_1.g_2.g_3.g_4\colon m
\to n$ has the property that $S(g)=M$.
}
\end{exercise}

We can also use \cref{prop.Sfull} and its proof to give a presentation of
$\mat(R)$, which was defined in \cref{def.prop_matrices}.

\begin{theorem}%
\label{thm.presentation_mat}%
\index{prop!of matrices}
  The prop $\mat(R)$ is isomorphic to the prop with the following presentation.
  The set of generators is the set
\[
G_R := \left\{\add{.05\textwidth}, \zero{.05\textwidth},\comult{.05\textwidth},
\counit{.05\textwidth}\right\} \cup \left\{
\begin{aligned}
  \resizebox{.05\textwidth}{!}{
}
\end{aligned}
\mid a \in R\right\},
\]
the same as the set of generators for $\sfg_R$; see \cref{def.sig_flow_graph_gens}.

  We have the following equations for any $a,b\in R$:
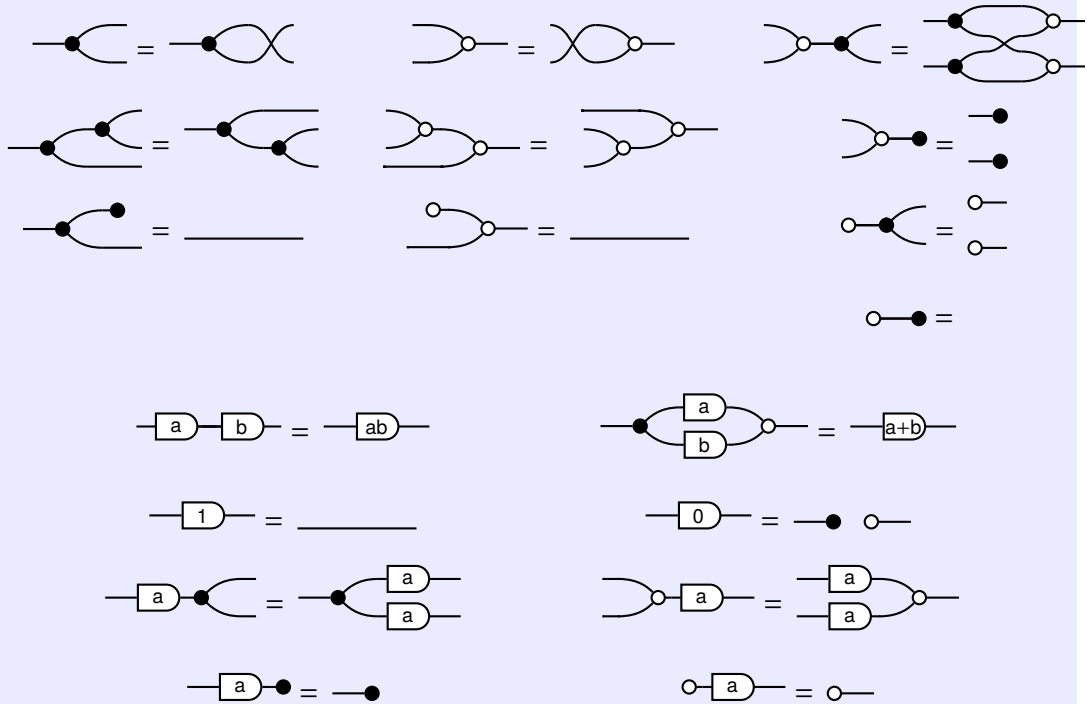
  \[
  \arraycolsep=1.5ex
  \renewcommand{\arraystretch}{2}
  \begin{array}{ccc}
\begin{aligned}
\begin{tikzpicture}[spider diagram]
	\node[spider={1}{2}, fill=black] (a) {};
	\coordinate (b1) at ($(a_out1)+(.2,0)$);
	\coordinate (b2) at ($(a_out2)+(.2,0)$);
	\draw (a_out1) to (b1);
	\draw (a_out2) to (b2);
\end{tikzpicture}
\end{aligned}
=
\begin{aligned}
\begin{tikzpicture}[spider diagram]
	\node[spider={1}{2}, fill=black] (a) {};
	\coordinate (b1) at ($(a_out1)+(.6,0)$);
	\coordinate (b2) at ($(a_out2)+(.6,0)$);
	\draw (a_out2) to (b1);
	\draw (a_out1) to (b2);
\end{tikzpicture}
\end{aligned}
& 
\begin{aligned}
\begin{tikzpicture}[spider diagram]
	\node[spider={2}{1}, fill=white] (a) {};
	\coordinate (b1) at ($(a_in1)-(.2,0)$);
	\coordinate (b2) at ($(a_in2)-(.2,0)$);
	\draw (a_in1) to (b1);
	\draw (a_in2) to (b2);
\end{tikzpicture}
\end{aligned}
=
\begin{aligned}
\begin{tikzpicture}[spider diagram]
	\node[spider={2}{1}, fill=white] (a) {};
	\coordinate (b1) at ($(a_in1)-(.6,0)$);
	\coordinate (b2) at ($(a_in2)-(.6,0)$);
	\draw (b1) to (a_in2);
	\draw (b2) to (a_in1);
\end{tikzpicture}
\end{aligned} 
&
\begin{aligned}
\begin{tikzpicture}[spider diagram]
	\node[spider={2}{1}, fill=white] (a) {};
	\node[special spider={1}{2}{0}{\leglen}, fill=black,right =.3 of a] (b) {};
	\draw (a.east) to (b_in1);
\end{tikzpicture}
\end{aligned}
=
\begin{aligned}
\begin{tikzpicture}[spider diagram, leg length=12pt]
	\node[spider={1}{2}, fill=black] (a1) {};
	\node[spider={1}{2}, fill=black, below=.4 of a1] (a2) {};
	\node[spider={2}{1}, fill=white, right=1.1 of a1] (b1) {};
	\node[spider={2}{1}, fill=white, right=1.1 of a2] (b2) {};
	\draw (a1_out1) to (b1_in1);
	\draw (a1_out2) to (b2_in1);
	\draw (a2_out1) to (b1_in2);
	\draw (a2_out2) to (b2_in2);
\end{tikzpicture}
\end{aligned}
\\
\end{array}
  \]
  \[
  \arraycolsep=5ex
  \renewcommand{\arraystretch}{2}
  \begin{array}{cc}
\begin{aligned}
\begin{tikzpicture}[spider diagram]
	\node[spider={1}{1}, wamp] (a) {$\scriptstyle \mathsf a$};
	\node[spider={1}{1}, wamp,right=.3 of a] (b) {$\scriptstyle \mathsf b$};
	\draw (a.east) to (b.west);
\end{tikzpicture}
\end{aligned}
=
\begin{aligned}
\begin{tikzpicture}[spider diagram, leg length=20pt]
	\node[spider={1}{1}, wamp] (a) {$\scriptstyle \mathsf{ab}$};
\end{tikzpicture}
\end{aligned}
&
\begin{aligned}
\begin{tikzpicture}[spider diagram]
	\node[spider={1}{2}, fill=black] (a) {};
	\node[spider={2}{1}, fill=white, right=1.5 of a] (b) {};
	\node at ($(a_out1)!.5!(b_in1)$) (helper1) {};
	\node[special spider={1}{1}{0.01}{0.01}, wamp] at (helper1) (scalar1) {$\scriptstyle \mathsf{a}$};
	\node at ($(a_out2)!.5!(b_in2)$) (helper2) {};
	\node[special spider={1}{1}{0.01}{0.01}, wamp] at (helper2) (scalar2) {$\scriptstyle \mathsf{b}$};
	\draw (a_out1) to (scalar1.west);
	\draw (scalar1.east) to (b_in1);
	\draw (a_out2) to (scalar2.west);
	\draw (scalar2.east) to (b_in2);
\end{tikzpicture}
\end{aligned}
=
\begin{aligned}
\begin{tikzpicture}[spider diagram, leg length=20pt]
	\node[spider={1}{1}, wamp] (a) {$\scriptstyle \mathsf{a+b}$};
\end{tikzpicture}
\end{aligned}
\\
\begin{aligned}
\begin{tikzpicture}[spider diagram, leg length=20pt]
	\node[spider={1}{1}, wamp] (a) {$\scriptstyle \mathsf 1$};
\end{tikzpicture}
\end{aligned}
=
\begin{aligned}
\begin{tikzpicture}[spider diagram]
	\coordinate (a) {};
	\coordinate (b) at ($(a)+(3*\leglen,0)$) {};
	\draw (a) to (b);
\end{tikzpicture}
\end{aligned}
&
\begin{aligned}
\begin{tikzpicture}[spider diagram, leg length=20pt]
	\node[spider={1}{1}, wamp] (a) {$\scriptstyle \mathsf 0$};
\end{tikzpicture}
\end{aligned}
=
\begin{aligned}
\begin{tikzpicture}[spider diagram]
	\node[spider={1}{0}, fill=black] (a) {};
	\node[spider={0}{1}, fill=white, right=.3 of a] (b) {};
\end{tikzpicture}
\end{aligned}
\\
\end{array}
\]
\end{theorem}
\begin{proof}
  The key idea is that these equations are sufficient to rewrite any
  $G_R$-generated prop expression into a normal form---the one used in the proof of
  \cref{prop.Sfull}---with all the black nodes to the left, all the white nodes to the right, and all the scalars in the middle. This is enough to show the equality of any two
  expressions that represent the same matrix.
  Details can be found in
  \cite{Baez.Erbele:2015a} or \cite{Bonchi.Sobocinski.Zanasi:2017a}.
\end{proof}

\paragraph{Sound and complete presentation of matrices.}%
\index{soundness!of proof system}%
\index{completeness!of proof system}

Once you get used to it, \cref{thm.presentation_mat} provides an intuitive,
visual way to reason about matrices. Indeed, the theorem implies two signal flow
graphs represent the same matrix if and only if one can be turned into the other
by local application of the above equations and the prop axioms.

The fact that you can prove two SFGs to be the same by using only graphical rules can be stated in the jargon
of logic: we say that the graphical rules provide a \emph{sound and complete reasoning system}. To be more
specific, \emph{sound} refers to the forward direction of the above statement:
two signal flow graphs represent the same matrix if one can be turned into the
other using the given rules.  \emph{Complete} refers to the reverse direction:
if two signal flow graphs represent the same matrix, then we can convert one into the other using the equations of \cref{thm.presentation_mat}.

\begin{example}
Both of the signal flow graphs below represent the same matrix, $0\choose6$\ :
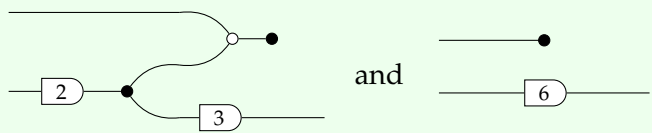
\[
  \begin{aligned}
.
This means that one can be transformed into the other by using only the
equations from \cref{thm.presentation_mat}. Indeed, here 
\begingroup
\allowdisplaybreaks
\begin{align*}
  \begin{aligned}
%
\end{aligned}
\qedhere
\end{align*}
\endgroup
\end{example}

\begin{exercise}%
\label{exc.rep_mats}
  \begin{enumerate}
  	\item For each matrix in \cref{ex.drawsfgs}, draw another signal flow graph that
  represents that matrix.
  	\item Using the above equations and the prop axioms, prove
  that the two signal flow graphs represent the same matrix.
  \qedhere
  \end{enumerate}
\end{exercise}

\begin{exercise}%
\label{exc.proof_using_diag_lang}
  Consider the signal flow graphs
  
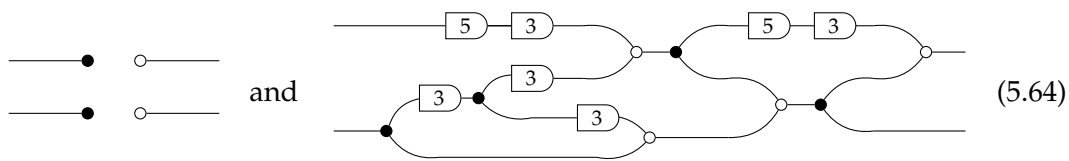
\begin{equation}%
\label{eqn.exc_two_sfgs}
  \begin{aligned}
    %
\end{aligned}
  \end{equation}
  \begin{enumerate}
  	\item Let $R=(\NN,0,+,1,*)$. By examining the presentation of $\mat(R)$ in \cref{thm.presentation_mat}, and without computing the
  matrices that the two signal flow graphs in \cref{eqn.exc_two_sfgs} represent, prove that they do \emph{not} represent the same
  matrix.
  	\item Now suppose the rig is $R=\NN/3\nn$; if you do not know what this means, just replace all 3's with 0's in the right-hand diagram of \cref{eqn.exc_two_sfgs}. Find what you would call a minimal representation of this diagram, using the presentation in \cref{thm.presentation_mat}.
	\qedhere
\end{enumerate}
\end{exercise}

\subsection{Aside: monoid objects in a monoidal category}
\label{ssec.alg_theories}%
\index{algebraic theory}%
\index{monoidal category!monoid object in|(}%
\index{monoid|seealso {monoidal category, monoid object in}}

Various subsets of the equations in \cref{thm.presentation_mat} encode
structures that are familiar from many other parts of mathematics, e.g.\
representation theory. For example one can find the axioms for (co)monoids,
(co)monoid homomorphisms, Frobenius algebras, and (with a little rearranging)
Hopf algebras, sitting inside this collection. The first example, the notion of
monoids, is particularly familiar to us by now, so we briefly discuss it below,
both in algebraic terms (\cref{def.monoid_object}) and in diagrammatic terms
(\cref{ex.diagrammatic_monoid_obj}).%
\index{Frobenius!algebra}

\begin{definition}%
\label{def.monoid_object}%
\index{monoidal category!monoid object in}
  A \emph{monoid object} $(M,\mu,\eta)$ in a symmetric monoidal category
  $(\cat{C},I,\otimes)$ is an object $M$ of $\cat{C}$ together with morphisms
  $\mu\colon M \otimes M \to M$ and $\eta\colon I \to M$ such that 
\begin{enumerate}[label=(\alph*)]
  \item $(\mu \otimes \id)\cp\mu = (\id \otimes \mu)\cp\mu$ and
  \item $(\eta \otimes \id)\cp\mu = \id = (\id \otimes \eta)\cp\mu$.
\end{enumerate}
A \emph{commutative monoid object} is a monoid object that further obeys
\begin{enumerate}[resume, label=(\alph*)]
  \item $\sigma_{M,M}\cp\mu = \mu$.
\end{enumerate}
where $\sigma_{M,M}$ is the swap map on $M$ in $\cat{C}$. We often denote it simply by $\sigma$.
\end{definition}

Monoid objects are so-named because they are an abstraction of the usual concept
of monoid.

\begin{example}
A monoid object in $(\smset,1,\times)$ is just a regular old monoid, as defined in \cref{ex.monoid}; see also \cref{ex.monoid_nats}. That is, it is a set $M$, a function $\mu\colon M\times M\to M$, which we denote by infix notation $*$, and an element $\eta(1)\in M$, which we denote by $e$, satisfying $(a*b)*c=a*(b*c)$ and $a*e=a=e*a$.
\end{example}

\begin{exercise}%
\index{real numbers}%
\label{exc.two_monoids_on_reals}
Consider the set $\RR$ of real numbers.
	\begin{enumerate}
		\item Show that if $\mu\colon\RR\times\RR\to\RR$ is defined by $\mu(a,b)=a*b$ and if $\eta\in\RR$ is defined to be $\eta=1$, then $(\RR,*,1)$ satisfies all three conditions of \cref{def.monoid_object}.
		\item Show that if $\mu\colon\RR\times\RR\to\RR$ is defined by $\mu(a,b)=a+b$ and if $\eta\in\RR$ is defined to be $\eta=0$, then $(\RR,+,0)$ satisfies all three conditions of \cref{def.monoid_object}.
		\qedhere
\end{enumerate}
\end{exercise}

\begin{example}%
\label{ex.diagrammatic_monoid_obj}
  Graphically, we can depict $\mu = \add{1em}$ and $\eta = \zero{1em}$. Then
  axioms (a), (b), and (c) from \cref{def.monoid_object} become:
  \begin{enumerate}[label=(\alph*)]
    \item $
      \begin{aligned}
	\resizebox{4em}{!}{

	  }
	\end{aligned}$
  \end{enumerate}
  All three of these are found in \cref{thm.presentation_mat}. Thus we can immediately conclude the following: the triple
  $(1,\add{1em},\zero{1em})$ is a commutative monoid object in the prop $\mat(R)$.
\end{example}

\begin{exercise}%
\label{exc.check_monoid_obj}
For any rig $R$, there is a functor $U\colon\mat(R)\to\smset$, sending the object $n\in\NN$ to the set $R^n$, and sending a morphism (matrix) $M\colon m\to n$ to the function $R^m\to R^n$ given by vector-matrix multiplication. %
\index{vector}

Recall that in $\mat(R)$, the monoidal unit is $0$ and the monoidal product is $+$, because it is a prop. Recall also that in (the usual monoidal structure on) $\smset$, the monoidal unit is $\{1\}$, a set with one element, and the monoidal product is $\times$ (see \cref{ex.set_as_mon_cat}).

\begin{enumerate}
	\item Check that the functor $U\colon\mat(R)\to\smset$, defined above, preserves the monoidal unit and the monoidal product.
	\item Show that if $(M,\mu,\eta)$ is a monoid object in $\mat(R)$ then $(U(M),U(\mu),U(\eta))$ is a monoid object in $\smset$. (This works for any monoidal functor---which we will define in \cref{roughdef.monoidal_functor}---not just for $U$ in particular.)
	\item In \cref{ex.diagrammatic_monoid_obj}, we said that the triple
  $(1,\add{1em},\zero{1em})$ is a commutative monoid object in the prop $\mat(R)$. If $R=\RR$ is the rig of real numbers, this means that we have a monoid structure on the set $\RR$. But in \cref{exc.two_monoids_on_reals} we gave two such monoid structures. Which one is it?
  \qedhere
\end{enumerate}%
\index{monoidal functor}
\end{exercise}

\begin{example}
  The triple $(1,\comult{1em},\counit{1em})$ in $\mat(R)$ forms a commutative
  monoid object in $\mat(R)\op$. We hence also say that
  $(1,\comult{1em},\counit{1em})$ forms a \emph{co-commutative comonoid object} in $\mat(R)$.
\end{example}

\begin{example}
A \emph{symmetric strict monoidal category}, is just a commutative monoid object in
$(\Cat{Cat},\times,\Cat{1})$. We will unpack this in \cref{sec.monoidal_cats_full}. %
\index{primordial ooze}
\end{example}

\begin{example}
A symmetric monoidal preorder, which we defined in
\cref{def.symm_mon_structure}, is just a commutative monoid object in the
symmetric monoidal category $(\Cat{Preord},\times,\Cat{1})$ of preorders and
monotone maps.
\end{example}

\begin{example}
For those who know what tensor products of commutative monoids are (or can
guess): A rig is a monoid object in the symmetric monoidal category
$(\Cat{CMon},\otimes,\nn)$ of commutative monoids with tensor product.
\end{example}

\begin{remark}%
\label{rem.theory_of_monoids}%
\index{theory!of monoids}
If we present a prop $\cat{M}$ using two generators $\mu\colon 2\to 1$ and
$\eta\colon 0 \to 1$, and the three equations from \cref{def.monoid_object}, we could call it `the theory of monoids in monoidal categories.' This means that in any monoidal category $\cat C$, the monoid objects in $\cat{C}$ correspond
to strict monoidal functors $\cat{M}\to\cat C$. This sort of idea leads to the study of
algebraic theories, due to Bill Lawvere and extended by many others; see
\cref{sec.ch5_further_reading}.%
\index{monoidal functor}
\end{remark}

\index{monoidal category!monoid object in|)}

\subsection{Signal flow graphs: feedback and more}%
\label{subsec.full_SFGs}%
\index{feedback}

At this point in the story, we have seen that every signal flow graph represents
a matrix, and this gives us a new way of reasoning about matrices. This is just
the beginning of a beautiful tale, one not only of graphical matrices, but of
graphical \emph{linear algebra}. We close this chapter with some brief hints at
how the story continues.

The pictoral nature of signal flow graphs invites us to play with them.
While we normally draw the copy icon like so, $\comult{1em}$, we
could just as easily reverse it and draw an icon $\mult{1em}$. What might it mean? Let's think
again about the semantics of flow graphs.

\paragraph{The behavioral approach.}%
\index{behavioral approach}

A signal flow graph $g\colon m \to n$ takes an input $x \in R^m$ and gives an output $y
\in R^n$. In fact, since this is all we care about, we might just think about
representing a signal flow graph $g$ as describing a set of input and
output pairs $(x,y)$. We'll call this set the \emph{behavior} of $g$ and denote it $\beh(g)\ss 
R^m\times R^n$. For example, the `copy' flow graph
\[
\begin{aligned}
\begin{tikzpicture}[spider diagram]
	\node[spider={1}{2}, fill=black] (a) {};
\end{tikzpicture}
\end{aligned}
\]
sends the input $1$ to the output $(1,1)$, so we consider $(1,(1,1))$ to be an element of copy-behavior. Similarly, $(x,(x,x))$ is copy behavior for every $x\in R$, thus we have
\[
\beh(\comult{1.3em}) = \{ (x,(x,x))\mid x \in R\}.
\]
In the abstract, the signal flow graph $g\colon m \to n$ has the behavior
\begin{equation}%
\label{eqn.behavior_g}
\beh(g) = \big\{\big(x,S(g)(x)\big) \mid x \in R^m\big\} \subseteq R^m \times R^n.
\end{equation}

\paragraph{Mirror image of an icon.}%
\index{mirror image|see {transpose}}%
\index{reverse icon|see {transpose}}

The above behavioral perspective provides a clue about how to interpret the mirror images of the
diagrams discussed above. Reversing an icon $g\colon m\to n$ exchanges the inputs with the outputs, so if we denote this reversed icon by $g\op$, we must have $g\op\colon n\to m$. Thus if $\beh(g)\ss R^m\times R^n$ then we need $\beh(g\op)\ss R^n\times R^m$. One simple way to do this is to replace each $(a,b)$ with $(b,a)$, so we would have
\begin{equation}%
\label{eqn.behavior_gop}
\beh(g\op)\coloneqq\big\{\big(S(g)(x),x\big) \mid x \in R^m\big\} \subseteq R^n \times R^m.
\end{equation}
This is called the \emph{transposed relation}.%
\index{transpose}

\begin{exercise}%
\label{exc.understand_reversed_icons}
\begin{enumerate}
	\item What is the behavior $\beh(\coadd{1.3em})$ of the reversed addition icon $\coadd{1.3em}\colon 1 \to 2$?
	\item What is the behavior $\beh(\mult{1.3em})$ of the reversed copy icon, $\mult{1.3em}\colon 2\to 1$?
	\qedhere
\end{enumerate}
\end{exercise}

\cref{eqn.behavior_g,eqn.behavior_gop} give us formulas for interpreting signal flow graphs and their mirror
images. But this would easily lead to disappointment, if we couldn't combine the two directions behaviorally; luckily we can. 

\paragraph{Combining directions.}
What
should the behavior be for a diagram such as the following:
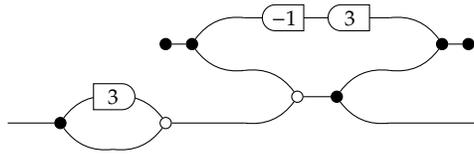
\[
  \begin{tikzpicture}[scale=.7]
	\begin{pgfonlayer}{nodelayer}
		\node [style=none] (0) at (-6, -0) {};
		\node [style=bdot] (1) at (-5, -0) {};
		\node [style=wamp] (2) at (-4, 0.5) {$\scriptstyle 3$};
		\node [style=none] (3) at (-4, -0.5) {};
		\node [style=wdot] (4) at (-3, -0) {};
		\node [style=bdot] (5) at (-3, 1.5) {};
		\node [style=bdot] (6) at (-2.5, 1.5) {};
		\node [style=none] (7) at (-1.5, 1) {};
		\node [style=wdot] (8) at (-0.5, 0.5) {};
		\node [style=none] (9) at (-1.5, 2) {};
		\node [style=wampop] (10) at (-0.75, 2) {$\scriptstyle{-1}$};
		\node [style=bdot] (11) at (2.25, 1.5) {};
		\node [style=none] (12) at (1.25, 1) {};
		\node [style=bdot] (13) at (0.25, 0.5) {};
		\node [style=none] (14) at (1.25, -0) {};
		\node [style=none] (15) at (3, -0) {};
		\node [style=wampop] (16) at (0.5, 2) {$\scriptstyle 3$};
		\node [style=none] (17) at (1.25, 2) {};
		\node [style=none] (18) at (-1.5, -0) {};
		\node [style=bdot] (19) at (2.75, 1.5) {};
	\end{pgfonlayer}
	\begin{pgfonlayer}{edgelayer}
		\draw (0.center) to (1);
		\draw [bend left, looseness=1.00] (1) to (2.center);
		\draw [bend right, looseness=1.00] (1) to (3.center);
		\draw [bend left, looseness=1.00] (2.center) to (4);
		\draw [bend right, looseness=1.00] (3.center) to (4);
		\draw (4) to (18.center);
		\draw [bend right, looseness=1.00] (18.center) to (8);
		\draw [bend right, looseness=1.00] (8) to (7.center);
		\draw [bend left, looseness=1.00] (7.center) to (6);
		\draw (6) to (5);
		\draw [bend left, looseness=1.00] (6) to (9.center);
		\draw (9.center) to (10.center);
		\draw (10.center) to (16.center);
		\draw (16.center) to (17.center);
		\draw [bend left, looseness=1.00] (17.center) to (11);
		\draw [bend left, looseness=1.00] (11) to (12.center);
		\draw (11) to (19);
		\draw [bend right, looseness=1.00] (12.center) to (13);
		\draw (8) to (13);
		\draw [bend right, looseness=1.00] (13) to (14.center);
		\draw (14.center) to (15.center);
	\end{pgfonlayer}
\end{tikzpicture}
\]
Let's formalize our thoughts a bit and begin by thinking about behaviors. The
behavior of a signal flow graph $m \to n$ is a subset $B\ss R^m \times R^n$, i.e.\ a relation. Why
not try to construct a prop where the morphisms $m \to n$ are relations?

We'll need to know how to compose and take monoidal products of relations. And if we want this prop of relations to contain the old prop $\mat(R)$, we need the new compositions and monoidal products to generalize the old ones in $\mat(R)$. Given signal flow graphs with matrices $M\colon m \to n$ and $N\colon n
\to p$, we see that their behaviors are the relations $B_1\coloneq\{(x,Mx) \mid x \in R^m\}$ and
$B_2\coloneq\{(y,Ny) \mid y \in R^n\}$, while the behavior of $M\cp N$ is the relation
$\{(x,x\cp M\cp N) \mid x \in R^m\}$. This is a case of relation composition. Given relations $B_1\ss R^m\times R^n$ and $B_2\ss R^n\times R^p$, their composite $B_1\cp B_2\ss R^m\times R^p$ is given by
\begin{equation}%
\label{eqn.comp_rule_rels}%
\index{relations!composition of}
  B_1\cp B_2\coloneqq
  \{(x,z)\mid \textrm{there exists } y \in R^n\textrm{ such that } (x,y) \in
  B_1\text{ and } (y,z) \in B_2\}.
\end{equation}
We shall use this as the general definition for composing two behaviors. 

\begin{definition}%
\label{def.rel_prop}%
\index{prop!of $R$-relations}
Let $R$ be a rig. We define the prop $\rel_R$ of $R$-relations to have subsets
$B\ss R^m \times R^n$ as morphisms. These are composed by the composition
rule from \cref{eqn.comp_rule_rels}, and we take the product of two sets to form their monoidal product.
\end{definition}

\begin{exercise}%
\label{exc.monoidal_prod_+}
In \cref{def.rel_prop} we went quickly through monoidal products $+$ in the prop $\rel_R$. If $B\ss R^m\times R^n$ and $C\ss R^p\times R^q$ are morphisms in $\rel_R$, write down $B+C$ in set-notation.
\end{exercise}

\paragraph{(No-longer simplified) signal flow graphs.}
Recall that above, e.g.\ in \cref{def.sig_flow_graph_gens}, we wrote $G_R$ for the set of generators of signal flow graphs. In \cref{subsec.full_SFGs}, we wrote $g\op$ for the mirror image of $g$, for each $g\in G_R$. So let's write $G_R\op\coloneqq\{g\op\mid g\in G_R\}$ for the set of all the mirror images of generators.  We define a prop
\begin{equation}%
\label{eqn.SFG_plus}
	\sfg_R^+\coloneqq \free\left(G_R\dju G_R\op\right).
\end{equation}
We call a morphism in the prop $\sfg_R^+$ a \emph{(non-simplified) signal flow
graph}: these extend our simplified signal flow graphs from
\cref{def.sig_flow_graph_gens} because now we can also use the mirrored icons.
By the universal property of free props, since we have said what the behavior of
the generators is (the behavior of a reversed icon is the transposed relation; see \cref{eqn.behavior_gop}), we have specified the behavior of any signal flow graph.
\index{signal flow graph!general}

The following two exercises help us understand what this behavior is.

\begin{exercise}%
\label{exc.SFG_composite_behavior}
Let $g\colon m \to n$, $h\colon \ell \to n$ be signal flow graphs. Note that
$h\op\colon n \to \ell$ is a signal flow graph, and we can form the
composite $g\cp (h\op)$:
\[

\]
Show that the behavior of $g\cp(h\op)\ss R^m\times R^\ell$ is equal to 
\[
  \beh(g\cp(h\op))=\{(x,y)\,\mid\, S(g)(x)=S(h)(y)\}.
  \qedhere
\]
\end{exercise}

\begin{exercise}%
\label{exc.sfg_behavior_again}
Let $g\colon m \to n$, $h\colon m \to p$ be signal flow graphs. Note that
$(g\op)\colon n \to m$ is a signal flow graph, and we can form the
composite $g\op\cp h$
\[
  %
\]
Show that the behavior of $g\op\cp h$ is equal to 
\[
  \beh((g\op)\cp h)=\{(S(g)(x),S(h)(x))\,\mid\, x \in R^m\}.
  \qedhere
\]
\end{exercise}

\paragraph{Linear algebra via signal flow graphs.}%
\index{signal flow graph!and
linear algebra}

In \cref{eqn.behavior_g} we see that every matrix, or linear map, can be
represented as the behavior of a signal flow graph, and in
\cref{exc.SFG_composite_behavior} we see that solution sets of linear equations
can also be represented. This includes central concepts in linear algebra, like
kernels and images.

\begin{exercise} %
\label{exc.linear_relations}
Here is an exercise for those that know linear algebra, in particular kernels and cokernels. Let $R$ be a field, let $g\colon
m \to n$ be a signal flow graph, and let $S(g)\in\mat(R)$ be the associated $(m\times n)$-matrix (see \cref{thm.sfg_to_mat}).
\begin{enumerate}
  \item Show that the composite of $g$ with $0$-reverses, shown here
  
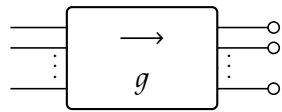
\[
    \begin{tikzpicture}[oriented WD, bb port length=0pt, bb port sep=1pt]
      \node[bb={4}{4}, minimum width = 2cm] (X) 
      {$\begin{array}{c} \longrightarrow \\ g \end{array}$};
  	\draw ($(X_in1)-(.5,0)$) to (X_in1);
  	\draw ($(X_in2)-(.5,0)$) to (X_in2);
  	\draw ($(X_in4)-(.5,0)$) to (X_in4);
  	\draw (X_out1) to ($(X_out1)+(.5,0)$);
  	\draw (X_out2) to ($(X_out2)+(.5,0)$);
  	\draw (X_out4) to ($(X_out4)+(.5,0)$);
  	\draw[label]
  		node[above left=-4pt and 3pt of X_in3] {$\vdots$}
  		node[above right=-4pt and 3pt of X_out3] {$\vdots$}
  		node[wdot] at ($(X_out1)+(.5,0)$) {}
  		node[wdot] at ($(X_out2)+(.5,0)$) {}
  		node[wdot] at ($(X_out4)+(.5,0)$) {}
  	;	
  \end{tikzpicture}
  \]
  is equal to the kernel of the matrix $S(g)$.
  \item Show that the composite of discard-reverses with $g$, shown here 
  
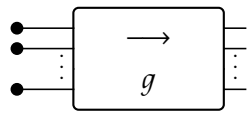
\[
    \begin{tikzpicture}[oriented WD, bb port length=0pt, bb port sep=1pt]
      \node[bb={4}{4}, minimum width = 2cm] (X) 
      {$\begin{array}{c} \longrightarrow \\ g \end{array}$};
  	\draw ($(X_in1)-(.5,0)$) to (X_in1);
  	\draw ($(X_in2)-(.5,0)$) to (X_in2);
  	\draw ($(X_in4)-(.5,0)$) to (X_in4);
  	\draw (X_out1) to ($(X_out1)+(.2,0)$);
  	\draw (X_out2) to ($(X_out2)+(.2,0)$);
  	\draw (X_out4) to ($(X_out4)+(.2,0)$);
  	\draw[label]
  		node[above left=-4pt and 3pt of X_in3] {$\vdots$}
  		node[above right=-4pt and 3pt of X_out3] {$\vdots$}
  		node[bdot] at ($(X_in1)-(.5,0)$) {}
  		node[bdot] at ($(X_in2)-(.5,0)$) {}
  		node[bdot] at ($(X_in4)-(.5,0)$) {}
  	;	
  \end{tikzpicture}
  \]
  is equal to the image of the matrix $S(g)$.
  \item Show that for any signal flow graph $g$, the subset
  $\beh(g)\subseteq R^m \times R^n$ is a linear subspace. That is, if $b_1,b_2\in \beh(g)$ then so are $b_1+b_2$ and $r*b_1$, for any $r\in R$.
  \qedhere
\end{enumerate}
\end{exercise}

We have thus seen that signal flow graphs provide a uniform, compositional
language to talk about many concepts in linear algebra. Moreover, in \cref{exc.linear_relations} we showed that the behavior of a signal flow graph is a linear relation, i.e.\ a relation whose elements can be added and multiplied by scalars $r\in R$. In fact the converse is true too: any linear relation $B\ss R^m\times R^n$ can be represented by a signal flow graph.%
\index{linear relation}%
\index{relation!linear|see {linear relation}}

\begin{exercise}%
\label{exc.linrel_prop}
One might want to show that linear relations on $R$ form a prop, $\Cat{LinRel}_R$. That is, one might want to show that there is a sub-prop of the prop $\rel_R$ from \cref{def.rel_prop}, where the morphisms $m\to n$ are the subsets $B\ss R^m\times R^n$ such that $B$ is linear. In other words, where for any $(x,y)\in B$ and $r\in R$, the element $(r*x,r*y)\in R^m\times R^n$ is in $B$, and for any $(x',y')\in B$, the element $(x+x',y+y')$ is in $B$.

This is certainly doable, but for this exercise, we only ask that you prove that the composite of two linear relations is linear.
\end{exercise}

Just like we gave a sound and complete presentation for the prop of matrices in
\cref{thm.presentation_mat}, it is possible to give a sound and complete
presentation for linear relations on $R$. Moreover, it is possible to give such
a presentation whose generating set is $G_R\sqcup G_R\op$ as in
\cref{eqn.SFG_plus} and whose equations include those from
\cref{thm.presentation_mat}, plus a few more. This presentation gives a
graphical method for doing linear algebra: an equation between linear subspaces
is true \emph{if and only if} it can be proved using the equations from the
presentation.

Although not difficult, we leave the full presentation to further reading
(\cref{sec.ch5_further_reading}). Instead, we'll conclude our exploration of the
prop of linear relations by noting that some of these `few more' equations state
that relations---just like co-design problems in \cref{chap.codesign}---form a
compact closed category.

\paragraph{Compact closed structure.}%
Using the icons available to us for signal flow graphs, we can build morphisms
that look like the `cup' and `cap' from \cref{def.compact_closed}:
\begin{equation}%
\label{eqn.cup_cap_sfgs}
  \begin{aligned}

  \end{aligned}
\end{equation}
The behaviors of these graphs are respectively 
\[
 \{(0,(x,x)) \mid x\in R\} \subseteq R^0\times R^2  \quad\mbox{ and }\quad 
 \{((x,x),0) \mid x\in R\} \subseteq R^2 \times R^0.
\]
In fact, these show the object $1$ in the prop $\rel_R$ is dual to itself: the
morphisms from \cref{eqn.cup_cap_sfgs} serve as the $\eta_1$ and $\epsilon_1$
from \cref{def.compact_closed}. Using monoidal products of these morphisms, one
can show that any object in $\rel_R$ is dual to itself.

Graphically, this means that the three signal flow graphs 
\[
  \begin{aligned}
    %
  \end{aligned}
\]
all represent the same relation.

Using these relations, it is straightforward to check the following result.
\begin{theorem} %
\label{thm.rel_is_ccc}
\index{compact closed category}
The prop $\rel_R$ is a compact closed category in which every object $n\in\nn$ is dual to itself, $n=n^*$.%
\index{dual!self}
\end{theorem}

To make our signal flow graphs simpler, we define new icons cup and cap
by the equations
\[
  \begin{aligned}
    %
  \end{aligned}
\]

\paragraph{Back to control theory.}%
\index{control theory}

Let's close by thinking about how to represent a simple control theory problem
in this setting. Suppose we want to design a system to maintain the speed of a
car at a desired speed $u$. We'll work in signal flow diagrams over the rig
$\rr[s,s^{-1}]$ of polynomials in $s$ and $s\inv$ with coefficients in $\rr$ and
where $ss\inv=s\inv s=1$. This is standard in control theory: we think of $s$ as
integration, and $s\inv$ as differentiation.

There are three factors that contribute to the actual speed
$v$. First, there is the actual speed $v$. Second, there are external forces
$F$. Third, we have our control system: this will take some linear combination
$a*u+b*v$ of the desired speed and actual speed, amplify it by some factor $p$ to give a (possibly
negative) acceleration. We can represent this system as follows, where $m$ is the
mass of the car.
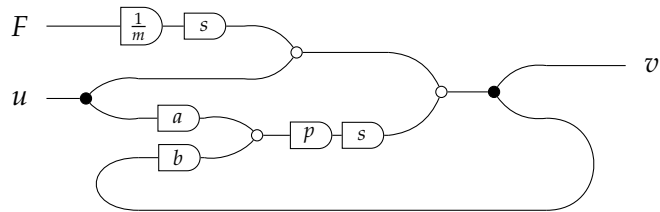
\[

\]
This can be read as the following equation, where one notes that $v$ occurs twice:
\[v=\int \frac{1}{m}F(t) dt + u(t) + p\int au(t)+bv(t) dt.\]

Our control problem then asks: how do we choose $a$ and $b$ to make the
behavior of this signal flow graph close to the relation $\{(F,u,v)\mid u=v\}$? By
phrasing problems in this way, we can use extensions of the logic we have
discussed above to reason about such complex, real-world problems.

\section{Summary and further reading}%
\label{sec.ch5_further_reading}

The goal of this chapter was to explain how props formalize signal flow graphs,
and provide a new perspective on linear algebra. To do this, we examined the
idea of free and presented structures in terms of universal properties. This
allowed us to build props that exactly suited our needs.

Pawe\l\ Soboci\'nski's \emph{Graphical Linear Algebra} blog is an accessible and
fun exploration of the key themes of this chapter, which goes on to describe how
concepts such as determinants, eigenvectors, and division by zero
can be expressed using signal flow graphs \cite{Sobocinski:blog}. For
the technical details, one could start with Baez and Erbele
\cite{Baez.Erbele:2015a}, or Zanasi's thesis \cite{zanasi:thesis} and its
related series of papers
\cite{bonchi2014categorical,bonchi2015full,Bonchi.Sobocinski.Zanasi:2017a}.  For
details about applications to control theory, see
\cite{Fong.Sobocinski.Rapisardo:2016a}. From the control theoretic perspective,
the ideas and philosophy of this chapter are heavily influenced by Willems'
behavioral approach \cite{Willems:2007a}.

For the reader that has not studied abstract algebra, we mention that rings,
monoids, and matrices are standard fare in abstract algebra, and can be found in
any standard introduction, such as \cite{Fraleign:1967a}. Rigs, also known as
semirings, are a bit less well known, but no less interesting; a comprehensive
survey of the literature can be found in \cite{Glazek:2013a}.%
\index{semiring|see {rig}}

Perhaps the most significant idea in this chapter is the separation of structure
into syntax and semantics, related by a functor. This is not only present in the
running theme of studying signal flow graphs, but in our aside
\cref{ssec.alg_theories}, where we talk, for example, about monoid objects in
monoidal categories. The idea of functorial semantics is yet another due to
Lawvere, first appearing in his thesis \cite{Lawvere:2004}.

\index{prop|)}
\index{signal flow graph|)}

\setcounter{chapter}{5}
\chapter[Circuits: hypergraph categories and operads]{Electric circuits:\\Hypergraph categories and operads}%
\label{chap.hypergraph_cats}


\section{The ubiquity of network languages}%
\index{network!diagram}

Electric circuits, chemical reaction networks, finite state automata, Markov
processes: these are all models of physical or computational systems that are
commonly described using network diagrams. Here, for example, we draw a diagram
that models a flip-flop, an electric circuit---important in computer
memory---that can store a bit of information:%
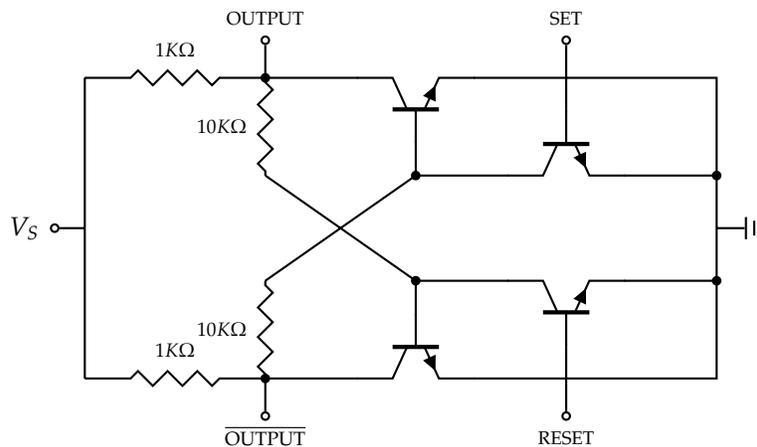
\index{flip-flop}
\index{electric circuit}
\[
\tikzset{int/.style={draw, fill, circle, inner sep=1pt}}
\begin{circuitikz}[circuit ee IEC, set resistor graphic=var resistor IEC
graphic,xscale=2,yscale=1,thick]
\coordinate (1) at (.8,0) {};
\coordinate (2) at (2,0) {};
\coordinate (3) at (3,0) {};
\coordinate (4) at (4,0) {};
\coordinate (5) at (5,0) {};
\coordinate (z) at (0,2.5) {};
\coordinate (a) at (0,2) {};
\coordinate (b) at (0,.7) {};
\coordinate (c) at (0,0) {};
\coordinate (d) at (0,-.7) {};
\coordinate (e) at (0,-2) {};
\coordinate (f) at (0,-2.5) {};
\node[npn,rotate=90] at (a-|3) (t1){};
\node[npn,xscale=-1,rotate=-90] at (e-|3) (t2){};
\node[npn,xscale=-1,rotate=-90] at (b-|4) (t3){};
\node[npn,rotate=90] at (d-|4) (t4){};
\node[draw, circle, inner sep=1pt] (source) at (.6,0) {};
\node at (.4,0) {\small $V_S$};
\node[draw, circle, inner sep=1pt, "\tiny \textrm{OUTPUT}"] (out) at (z-|2) {};
\node[draw, circle, inner sep=1pt, label={[label distance=-1.5pt, below]\tiny
$\overline{\mathrm{OUTPUT}}$}] (out2) at (f-|2) {};
\node[draw, circle, inner sep=1pt, "\tiny \textrm{SET}"] (set) at (z-|4) {};
\node[draw, circle, inner sep=1pt, "\tiny \textrm{RESET}" below] (reset) at (f-|4) {};
\node[ground, scale=.6] (ground) at ($(5)+(.25,0)$) {};
\node[int] at (a-|2) {};
\node[int] at (e-|2) {};
\node[int] at (b-|3) {};
\node[int] at (d-|3) {};
\node[int] at (b-|5) {};
\node[int] at (d-|5) {};
\draw
(source) to (c-|1)
(a-|1) to (e-|1)
(a-|1) to[resistor] node[label={[label distance=0pt]90:{\tiny$1K\Omega$}}] {} (a-|2)
(e-|1) to[resistor] node[label={[label distance=0pt]90:{\tiny$1K\Omega$}}] {} (e-|2)
(a-|2) to[resistor] node[label={[label distance=-1pt]180:{\tiny$10K\Omega$}}]{}(b-|2)
(e-|2) to[resistor] node[label={[label distance=-1pt]180:{\tiny$10K\Omega$}}]{}(d-|2)
(b-|2) to (d-|3)
(d-|2) to (b-|3)
(b-|3) to (t1.B)
(d-|3) to (t2.B)
(a-|2) to (t1.C)
(e-|2) to (t2.C)
(t1.E) to (a-|5) to (c-|5)
(t2.E) to (e-|5) to (c-|5)
(b-|3) to (t3.C)
(t3.E) to (b-|5)
(t3.B) to (set)
(d-|3) to (t4.C)
(t4.E) to (d-|5)
(t4.B) to (reset)
(a-|2) to (out)
(e-|2) to (out2)
(c-|5) to (ground)
	;
\end{circuitikz}
\]

Network diagrams have time-tested utility. In this chapter, we are interested in
understanding the common mathematical structure that they share, for the
purposes of translating between and unifying them; for example certain types of
Markov processes can be simulated and hence solved using circuits of resisters.
When we understand the underlying structures that are shared by network diagram languages, we can make comparisons between the corresponding mathematical models easily.

At first glance network diagrams appear quite different from the wiring diagrams
we have seen so far.  For example, the wires are undirected in the case above,
whereas in a category---including monoidal categories seen in resource theories or co-design---every
morphism has a domain and codomain, giving it a sense of direction. Nonetheless, we shall see how to use
categorical constructions such as universal properties to create categorical
models that precisely capture the above type of ``network'' compositionality, i.e.\ allowing us to effectively drop directedness when convenient.

In particular we'll return to the idea of a colimit, which we sketched for you
at the end of \cref{chap.databases}, and show how to use colimits in the
category of sets to formalize ideas of connection. Here's the key idea.

\paragraph{Connections via colimits.}%
\index{colimit|(}%
\index{electrical circuit}
Let's say we want to install some lights: we want to create a circuit so that
when we flick a switch, a light turns on or off. To start, we have a bunch of
circuit components: a power source, a switch, and a lamp connected to a resistor:
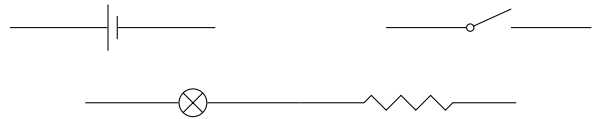
\[
\begin{tikzpicture}[circuit ee IEC, set resistor graphic=var resistor IEC
graphic, set make contact graphic=var make contact IEC graphic]
\node (1) at (1,0) {};
\coordinate (2) at (4,0) {};
\node (3) at (7,0) {};
\node (5) at (0,1) {};
\node (6) at (3,1) {};
\node (7) at (5,1) {};
\node (8) at (8,1) {};
\draw (1) to [bulb] (2);
\draw (2) to [resistor] (3);
\draw (5) to [battery] (6);
\draw (7) to [make contact] (8);
\end{tikzpicture}
\]
We want to connect them together, but there are many ways to do so. How should we describe the particular way that will form a light
switch?%
\index{interconnection}

First, we claim that circuits should really be thought of as open circuits: each
carries the additional structure of an `interface' exposing it to the rest of the electrical world.%
\index{interface}%
\index{open system} Here by \emph{interface} we mean a certain set of locations, or \emph{ports}, at which we are able to connect them with other components.%
\footnote{If your circuit has no such ports, it still falls within our purview, by taking its interface to be the empty set.}
As is so common in category
theory, we begin by making this more-or-less obvious fact explicit.  Let's
depict the available ports using a bold $\bullet$. If we say that in the each of the three drawings above, the 
ports are simply the dangling end points of the wires, they would be redrawn as follows:
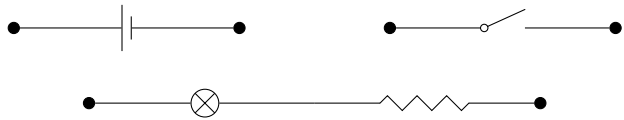
\[

\]
Next, we have to describe which ports should be connected. We'll do this by
drawing empty circles $\circ$ connected by arrows to two ports $\bullet$. Each will be a witness-to-connection, saying
`connect these two!' 
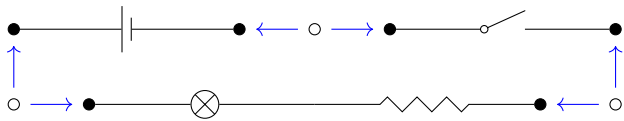
\[
%
\]
Looking at this picture, it is clear what we need to do: just identify---i.e.\ \emph{merge} or \emph{make equal}---the ports as indicated, to get the following circuit:
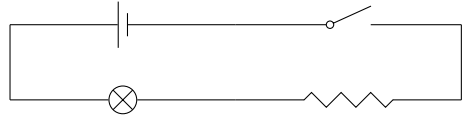
\[
%
\]

But mathematics doesn't have a visual cortex with which to generate the intuitions we can count on with a human reader such as yourself.%
\footnote{Unless the future has arrived since the writing of this book.}
Thus we need to specify formally what `identifying ports as indicated' means mathematically. As it turns out, we can do this using finite colimits in a given category $\cat{C}$.%
\index{future!as not yet arrived}

Colimits are diagrams with certain universal properties, which is kind of an
epiphenomenon of the category $\cat{C}$. Our goal is to obtain $\cat{C}$'s
colimits more directly, as a kind of operation in some context, so that we can think of them as
telling us how to connect circuit parts together. To that end, we produce a
certain monoidal category---namely that of \emph{cospans in $\cat{C}$}, denoted
$\cospan{\cat{C}}$---that can conveniently package $\cat{C}$'s colimits in terms
of its own basic operations: composition and monoidal product.

In summary, the first part of this chapter is devoted to the slogan `colimits
model interconnection'. In addition to universal constructions such as colimits,
however, another way to describe interconnection is to use wiring diagrams. We go full circle when we find that these wiring diagrams are strongly connected to cospans, and hence colimits.%
\index{connection!as colimit}

\paragraph{Composition operations and wiring diagrams.}%
\index{wiring diagram!styles of}
In this book we have seen the utility of defining syntactic or algebraic
structures that describe the sort of composition operations that make sense and
can be performed in a given application area. Examples include monoidal preorders with
discarding, props, and compact closed categories. Each of these has an
associated sort of wiring diagram style, so that any wiring diagram of that
style represents a composition operation that makes sense in the given area: the
first makes sense in manufacturing, the second in signal flow, and the third in
collaborative design. So our second goal is to answer the question, ``how do we
describe the compositional structure of network-style wiring diagrams?''
\index{interconnection!network-type}

Network-type interconnection can be described using something called a hypergraph category. Roughly speaking, these are categories whose
wiring diagrams are those of symmetric monoidal categories together with, for
each pair of natural numbers $(m,n)$, an icon $s_{m,n}\colon m \to n$. These
icons, known as \emph{spiders},%
\footnote{Our spiders have any number of legs.}
are drawn as follows:%
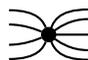
\index{spider}%
\index{icon}
\[
\begin{tikzpicture}[spider diagram]
	\node[spider={4}{5}, fill=black] (a) {};
\end{tikzpicture}
\]
Two spiders can share a leg, and when they do, we can fuse them into one spider.
The intuition is that spiders are connection points for a number of wires, and
when two connection points are connected, they fuse to form an even more
`connect-y' connection point. Here is an example:%
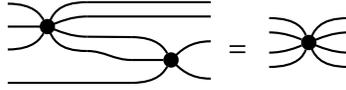
\index{icon!spider}
\[
\begin{aligned}
\begin{tikzpicture}[spider diagram]
	\node[spider={3}{4}, fill=black] (a) {};
	\node[spider={3}{2}, fill=black, below right=.3 and 1.5 of a] (b) {};
	\draw (a_out1) to (b_out1|-a_out1);
	\draw (a_out2) to (b_out1|-a_out2);
	\draw (a_out3) to (b_in1);
	\draw (a_out4) to (b_in2);
	\draw (a_in1|-b_in3) to (b_in3);
\end{tikzpicture}
\end{aligned}
\;
=
\;
\begin{aligned}
\begin{tikzpicture}[spider diagram]
	\node[spider={4}{4}, fill=black] (a) {};
\end{tikzpicture}
\end{aligned}
\]
A hypergraph category may have many species of spiders with the rule that spiders of different species cannot share a leg---and hence not fuse---but two spiders of the same species can share legs and fuse. We add spider diagrams to the iconography of hypergraph categories.%
\index{spider!as iconography}%
\index{icon!spider}

As we shall see, the ideas of describing network interconnection using colimits
and hypergraph categories come together in the notion of a theory. We
first introduced the idea of a theory in \cref{ssec.alg_theories}, but here we
explore it more thoroughly, starting with the idea that, approximately speaking,
cospans in the category $\finset$ form the theory of hypergraph categories.

We can assemble all cospans in $\finset$ into something called an `operad'.%
\index{operad} Throughout this book we have talked about using free structures
and presentations to create instances of algebraic structures such as preorders,
categories, and props, tailored to the needs of a particular situation. Operads
can be used to tailor the algebraic structures \emph{themselves} to the needs of
a particular situation. We will discuss how this works, in particular how
operads encode various sorts of wiring diagram languages and corresponding
algebraic structures, at the end of the chapter.%
\index{operad!as custom compositionality}



\section{Colimits and connection}%
\label{sec.colims_connection}

Universal constructions are central to category theory. They allow us to define
objects, at least up to isomorphism, by describing their relationship with other
objects. So far we have seen this theme in a number of different forms: meets and joins (\cref{sec.meets_joins}), Galois
connections and adjunctions (\cref{sec.galois_connections,sec.adjunctions_mig}),
limits (\cref{sec.bonus_lims_colims}), and free and presented structures
(Section~\ref{sec.free_constructions}-\ref{sec.prop_presentations}).  Here we turn our attention to colimits.
\index{universal property}

In this section, our main task is to have a concrete understanding of colimits
in the category $\finset$ of finite sets and functions. The idea will be to take
a bunch of sets---say two or fifteen or zero---use functions between them to
designate that elements in one set `should be considered the same' as elements
in another set, and then merge the sets together accordingly.

\subsection{Initial objects} %
\index{initial object|(}%
\index{colimit!initial object as}

Just as the simplest limit is a terminal object (see
\cref{subsec.terminals_products}), the simplest colimit is an initial object.
This is the case where you start with no objects and you merge them together.

\begin{definition}%
\label{def.initial_obj}%
\index{initial object}
  Let $\cat C$ be a category. An \emph{initial object} in $\cat C$ is an object
  $\varnothing\in\cat{C}$ such that for each object $T$ in $\cat C$ there exists a unique
  morphism $!_T\colon \varnothing \to T$.
\end{definition}

The symbol $\varnothing$ is just a default name, a notation, intended to evoke
the right idea; see \cref{ex.empty_set} for the reason why we use the notation
$\varnothing$, and \cref{exc.rig_initial_obj} for a case when the default name
$\varnothing$ would probably not be used.

Again, the hallmark of universality is the existence of a unique map to any
other comparable object.

\begin{example}
  An initial object of a preorder is a bottom element---that is, an element that is
  less than every other element. For example $0$ is the initial object in $(\nn,\leq)$, whereas $(\rr,\leq)$ has no initial object.
\end{example}

\begin{exercise}%
\label{exc.initial_ob_practice}
Consider the set $A=\{a,b\}$. Find a preorder relation $\leq$ on $A$ such that
\begin{enumerate}
	\item $(A,\leq)$ has no initial object.
	\item $(A,\leq)$ has exactly one initial object.
	\item $(A,\leq)$ has two initial objects.
\qedhere
\end{enumerate}
\end{exercise}

\begin{example} %
\label{ex.empty_set}
  The initial object in $\finset$ is the empty set. Given any finite set $T$,
  there is a unique function $\varnothing \to T$, since $\varnothing$ has no
  elements.%
\index{initial object!empty set as}
\end{example}

\begin{example}
As seen in \cref{exc.initial_ob_practice}, a category $\cat{C}$ need not have an initial object. As a different sort of example, consider the category
shown here:
\[
\cat{C}\coloneqq\boxCD{
\begin{tikzcd}[ampersand replacement=\&]
	\LMO{A}\ar[r, shift left, "f"]\ar[r, shift right, "g"']\&\LMO{B}
\end{tikzcd}
}
\]
If there were to be an initial object $\varnothing$, it would either be $A$ or $B$. Either way, we need to show that for each object $T\in\Ob(\cat{C})$ (i.e.\ for both $T=A$ and $T=B$) there is a unique morphism $\varnothing\to T$. Trying the case $\varnothing=^?A$ this condition fails when $T=B$: there are two morphisms $A\to B$, not one. And trying the case $\varnothing=^?B$ this condition fails when $T=A$: there are zero morphisms $B\to A$, not one.
\end{example}

\begin{exercise} %
\label{exc.initial_object}
For each of the graphs below, consider the free category on that graph, and say whether it has an initial object.\\
\begin{enumerate*}[itemjoin=\hspace{.8in}]
	\item \fbox{$\LMO{a}$}
	\item \fbox{$\LMO{a}\to\LMO{b}\to\LMO{c}$}
	\item \fbox{$\LMO{a}\qquad\LMO{b}$}
	\item \fbox{$\begin{tikzcd}\LMO{a}\ar[loop right]\end{tikzcd}$}
	\qedhere
\end{enumerate*}
\end{exercise}

\begin{exercise}%
\label{exc.rig_initial_obj}
Recall the notion of rig from \cref{chap.SFGs}. A \emph{rig homomorphism} from $(R,0_R,+_R,1_R,*_R)$ to $(S,0_S,+_S,1_S, *_S)$ is a function $f\colon R\to S$ such that $f(0_R)=0_S$, $f(r_1+_Rr_2)=f(r_1)+_Sf(r_2)$, etc.
\begin{enumerate}
	\item We said ``etc.'' Guess the remaining conditions for $f$ to be a rig homomorphism.
	\item Let $\Cat{Rig}$ denote the category whose objects are rigs and whose morphisms are rig homomorphisms. We claim $\Cat{Rig}$ has an initial object. What is it?
	\qedhere
\end{enumerate}	
\end{exercise}

\begin{exercise} %
\label{exc.universality}
Explain the statement ``the hallmark of universality is the existence of a
unique map to any other comparable object,'' in the context of
\cref{def.initial_obj}. In particular, what is being universal in \cref{def.initial_obj}, and which is the ``comparable object''?
\end{exercise}

\begin{remark}
As mentioned in \cref{rem.the_vs_a}, we often speak of `the' object that
satisfies a universal property, such as `the initial object', even though many
different objects could satisfy the initial object condition. Again, the reason
is that initial objects are unique up to unique isomorphism: any two initial
objects will have a canonical isomorphism between them, which one finds using various applications of the universal
property.

%
\end{remark}

\begin{exercise}%
\label{exc.initials_are_isomorphic}
Let $\cat{C}$ be a category, and suppose that $c_1$ and $c_2$ are initial objects. Find an isomorphism between them, using the universal property from \cref{def.initial_obj}.
\end{exercise}

\index{initial object|)}

\subsection{Coproducts} %
\index{coproduct|(}%
\index{colimit!coproduct as}

Coproducts generalize both joins in a preorder and disjoint unions of sets.%
\index{union!disjoint}%
\index{disjoint union|see {union, disjoint}}

\begin{definition}%
\label{def.coproduct}%
\index{coproduct}
  Let $A$ and $B$ be objects in a category $\cat C$. A \emph{coproduct} of $A$
  and $B$ is an object, which we denote $A+B$, together with a pair of morphisms $(\iota_A \colon A \to A+B,
  \iota_B \colon B \to A+B)$ such that for all objects $T$ and pairs of morphisms $(f\colon A \to T,
  g \colon B \to T)$, there exists a unique morphism $\copair{f,g}\colon A+B \to T$ such that the following diagram commutes:
  \begin{equation}%
\label{eqn.universal_prop_coprod}
    \begin{tikzcd}
      A \ar[r,"\iota_A"] \ar[dr,"f"'] & A+B \ar[d, pos=.4, dashed,"\copair{f,g}"] & B
      \ar[l,"\iota_B"'] \ar[dl, "g"] \\[10pt]
      & T
    \end{tikzcd}
  \end{equation}
  We call $\copair{f,g}$ the \emph{copairing} of $f$ and $g$.
\end{definition}

\begin{exercise}%
\label{exc.join_as_coproduct}%
\index{join!as coproduct}
  Explain why, in a preorder, coproducts are the same as joins.
\end{exercise}

\begin{example} %
\label{ex.setcoproduct}
  Coproducts in the categories $\finset$ and $\smset$ are disjoint unions. More
  precisely, suppose $A$ and $B$ are sets. Then the coproduct of $A$ and $B$ is
  given by the disjoint union $A\dju B$ together with the inclusion functions
  $\iota_A\colon A \too A \dju B$ and $\iota_B \colon B \to A\dju B$.
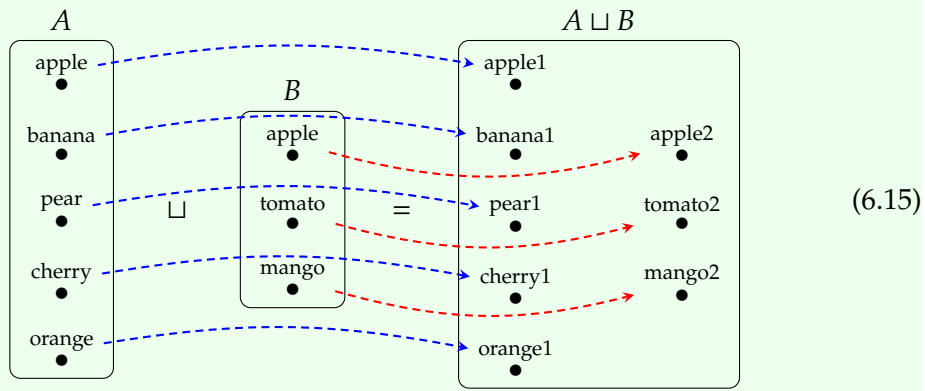
\begin{equation}%
\label{eqn.apples_oranges_again}
\begin{tikzpicture}[y=.8ex, rounded corners, baseline=(sq)]
	\node (a1) {$\LTO{apple}$};
	\node [below=1 of a1] (a2) {$\LTO{banana}$};
	\node [below=1 of a2] (a3) {$\LTO{pear}$};
	\node [below=1 of a3] (a4) {$\LTO{cherry}$};
	\node [below=1 of a4] (a5) {$\LTO{orange}$};
	\node [draw, inner ysep=0pt, fit=(a1) (a5)] (a) {};
	\node [above=0 of a] (alab) {$A$};
	\node [right=2 of a2] (b1) {$\LTO{apple}$};
	\node [below=1 of b1] (b2) {$\LTO{tomato}$};
	\node [below=1 of b2] (b3) {$\LTO{mango}$};
	\node [draw, inner ysep=0pt, fit=(b1) (b3)] (b) {};	
	\node [above=0 of b] {$B$};
	\node at ($(a)!.5!(b)$) (sq) {$\sqcup$};
	\node [right=5 of a1] (x11) {$\LTO{apple1}$};
	\node [below=1 of x11] (x12) {$\LTO{banana1}$};
	\node [below=1 of x12] (x13) {$\LTO{pear1}$};
	\node [below=1 of x13] (x14) {$\LTO{cherry1}$};
	\node [below=1 of x14] (x15) {$\LTO{orange1}$};
	\node [right=1 of x12] (x21) {$\LTO{apple2}$};
	\node [below=1 of x21] (x22) {$\LTO{tomato2}$};
	\node [below=1 of x22] (x23) {$\LTO{mango2}$};
	\node [draw, inner ysep=0pt, fit=(x11) (x21) (x15) (x23)] (x) {};	
	\node [above=0 of x] {$A\sqcup B$};
	\node at ($(b.east)!.5!(x.west)$) {$=$};
	\begin{scope}[mapsto, bend left=10pt]
  	\draw (a1) to (x11);
  	\draw (a2) to (x12);
  	\draw (a3) to (x13);
  	\draw (a4) to (x14);
  	\draw (a5) to (x15);
	\end{scope}
	\begin{scope}[mapsto, red, bend right=15pt]
  	\draw (b1) to (x21);
  	\draw (b2) to (x22);
  	\draw (b3) to (x23);
	\end{scope}
	
\end{tikzpicture}
\end{equation}
  
  Suppose we have functions $f\colon A \to T$ and $g\colon B \to T$ for some other set
  $T$, unpictured. The universal property of coproducts says there is a unique function $\copair{f,g}\colon A\dju B
  \to T$ such that $\iota_A\cp \copair{f,g} =f$ and $\iota_B \cp \copair{f,g} =g$. What is
  it? Any element $x\in A\dju B$ is either `from $A$' or `from $B$', i.e.\ either there is some $a\in A$ with $x=\iota_A(a)$ or there is some $b\in B$ with $x=\iota_B(b)$. By \cref{eqn.universal_prop_coprod}, we must have:
  \[
    \copair{f,g}(x) = 
    \begin{cases}
      f(x) & \mbox{if } x=\iota_A(a) \mbox{ for some }a \in A; \\
      g(x) & \mbox{if } x=\iota_B(b) \mbox{ for some }b \in B.
    \end{cases}
    \qedhere
  \]
\end{example}

\begin{exercise} %
\label{exc.copairing}
	Suppose $T=\{a,b,c,\ldots,z\}$ is the set of letters in the alphabet,
	and let $A$ and $B$ be the sets from \cref{eqn.apples_oranges_again}.
	Consider the function $f\colon A\to T$ sending each element of $A$ to
	the first letter of its label, e.g.\ $f(\mathrm{apple})=a$. Let $g\colon
	B\to T$ be the function sending each element of $B$ to the last letter
	of its label, e.g.\ $g(\mathrm{apple})=e$. Write down the function
	$\copair{f,g}(x)$ for all eight elements of $A\sqcup B$.
\end{exercise}
  
\begin{exercise} %
\label{exc.coprod_properties}
  Let $f \colon A \to C$, $g \colon B \to C$, and $h\colon C \to D$ be morphisms
  in a category $\cat{C}$ with coproducts. Show that
  \begin{enumerate}
  \item $\iota_A \cp \copair{f,g} = f$.
  \item $\iota_B \cp \copair{f,g} = g$.
  \item $\copair{f,g}\cp h = \copair{f \cp h,g\cp h}$.
  \item $\copair{\iota_A,\iota_B} = \id_{A+B}$.
  \qedhere
  \end{enumerate}
\end{exercise}

\begin{exercise} %
\label{exc.coproducts_give_monoidal_structure}
Suppose a category $\cat C$ has coproducts, denoted $+$, and an initial object,
denoted $\varnothing$. Then $(\cat C,+,\varnothing)$ is a symmetric
monoidal category (recall \cref{rdef.sym_mon_cat}). In this exercise we develop
the data relevant to this fact: 
\begin{enumerate}
\item Show that $+$ extends to a functor $\cat C \times \cat C \to \cat C$. In particular, how does it act on morphisms in $\cat{C}\times\cat{C}$?
\item Using the universal properties of the initial object and coproduct, show
that there are isomorphisms $A+\varnothing \to A$ and $\varnothing +A \to A$.
\item Using the universal property of the coproduct, write down morphisms
\begin{enumerate}
\item $(A+B)+C \to A+(B+C)$.
\item $A+B \to B+A$.
\end{enumerate}
If you like, check that these are isomorphisms.
\end{enumerate}
It can then be checked that this data obeys the axioms of a symmetric monoidal
category, but we'll end the exercise here.
\end{exercise}

\index{coproduct|)}

\subsection{Pushouts} %
\index{pushout|(}

Pushouts are a way of combining sets. Like a union of
subsets, a pushout can combine two sets in a non-disjoint way: elements of one
set may be identified with elements of the other. The pushout construction,
however, is much more general: it allows (and requires) the user to specify
exactly which elements will be identified. We'll see a demonstration of this
additional generality in \cref{ex.pushouts_quotient}.

\begin{definition}%
\label{def.pushout}
\index{pushout}%
\index{colimit!pushout as}%
\index{pushout!as colimit}
  Let $\cat{C}$ be a category and let $f\colon A \to X$ and $g \colon A \to Y$ be morphisms in $\cat{C}$ that have a common domain. The \emph{pushout}
  $X+_AY$ is the colimit of the diagram 
  \[
    \begin{tikzcd}[column sep=20pt]
      A \ar[r,"f"] \ar[d,"g"'] & X \\
      Y
    \end{tikzcd}
  \]
\end{definition}
  In more detail, a pushout consists of (i) an object $X+_AY$ and (ii) morphisms
  $\iota_X\colon X \to X+_AY$ and $\iota_Y\colon Y \to X+_AY$ satisfying (a) and
  (b) below.
  \begin{enumerate}[label=(\alph*)]
    \item The diagram
  \begin{equation}%
\label{eq.pushout_diagram}
    \begin{tikzcd}
      A \ar[r,"f"] \ar[d,"g"'] & X \ar[d,"\iota_X"] \\
      Y \ar[r,"\iota_Y"'] & X+_AY\ar[ul, phantom, very near start, "\ulcorner"]
    \end{tikzcd}
  \end{equation}
  commutes. (We will explain the `$\ulcorner$' symbol below.)
  \item For all objects $T$ and morphisms $x\colon X \to T$, $y\colon Y \to T$, if the diagram
  \[
    \begin{tikzcd}
      A \ar[r,"f"] \ar[d,"g"'] & X \ar[d,"x"] \\
      Y \ar[r,"y"'] & T
    \end{tikzcd}
  \]
  commutes, then there exists a unique morphism $t\colon X+_AY \to T$ such that
  \begin{equation}%
\label{eqn.univ_prop_pushout}
    \begin{tikzcd}
      A \ar[r,"f"] \ar[d,"g"'] & X \ar[d,"\iota_X"] \ar[ddr,"x", bend left] \\
      Y \ar[r,"\iota_Y"] \ar[drr,"y", bend right=20pt] & X+_AY \ar[dr,dashed, "t"] \\[-7pt]
      &&[-15pt]T
    \end{tikzcd}
  \end{equation}
  commutes.
  \end{enumerate}
  
  If $X+_AY$ is a pushout, we denote that fact by drawing the commutative square \cref{eq.pushout_diagram}, together with the $\ulcorner$ symbol as shown; we call it a \emph{pushout square}.
  
  We further call $\iota_X$ the \emph{pushout of $g$ along $f$}, and similarly
  $\iota_Y$ the \emph{pushout of $f$ along $g$}.

  \begin{example}
    In a preorder, pushouts and coproducts have a lot in common. The pushout of a diagram $B\from A\to C$ is equal to the coproduct $B\sqcup C$: namely, both are equal to the join $B\vee C$.
  \end{example}

\begin{example} %
\label{ex.pushout_along_identity}%
\index{pushout!along isomorphism}%
\index{isomorphism!as stable under pushout}
  Let $f \colon A \to X$ be a morphism in a category $\cat{C}$. For any isomorphisms $i\colon A\to A'$ and $j\colon X\to X'$, we can take $X'$ to be the pushout
  $X+_AA'$, i.e.\ the following is a pushout square:
  \[
    \begin{tikzcd}
      A \ar[r,"f"] \ar[d,"i"'] & X \ar[d,"j"] \\
      A' \ar[r,"f'"'] & X'\ar[ul, phantom, very near start, "\ulcorner"]
    \end{tikzcd}
  \]
  where $f'\coloneqq i\inv\cp f\cp j$.   To see this, observe that if there is any object $T$ such that the following square commutes:
  \[
    \begin{tikzcd}
      A \ar[r,"f"] \ar[d,"i"'] & X \ar[d,"x"] \\
      A' \ar[r,"a"'] & T
    \end{tikzcd}
  \]
then $f\cp x = i\cp a$, and so we are forced to take $x'\colon X\to T$ to be $x'\coloneqq j\inv\cp x$. This makes the following diagram commute:
  \[
    \begin{tikzcd}
      A \ar[r,"f"] \ar[d,"i"'] & X \ar[d,"j"] \ar[ddr,"x", bend left] \\
      A' \ar[r,"f'"'] \ar[drr, bend right=20pt, "a"'] & X' \ar[dr,dashed, "x'"] \\[-7pt]
      &&[-15pt]T
    \end{tikzcd}
  \]
 because $f'\cp x'=i\inv\cp f\cp j\cp j\inv\cp x=i\inv\cp i\cp a=a$.
\end{example}

\begin{exercise}%
\label{exc.disc_cats_have_pushouts}
For any set $S$, we have the discrete category $\Cat{Disc}_S$, with $S$ as objects and only identity morphisms.
\begin{enumerate}
	\item Show that all pushouts exist in $\Cat{Disc}_S$, for any set $S$.
	\item For what sets $S$ does $\Cat{Disc}_S$ have an initial object?
\qedhere
\end{enumerate}
\end{exercise}
  
\begin{example} %
\label{ex.pushouts}
  In the category $\finset$, pushouts always exist. The pushout of functions
  $f\colon A \to X$ and $g \colon A \to Y$ is the set of equivalence classes of
  $X \dju Y$ under the equivalence relation generated by---that is, the
  reflexive, transitive, symmetric closure of---the relation $\{f(a)\sim g(a)\mid a\in A\}$.

  We can think of this in terms of interconnection too. Each element $a\in A$ provides
  a connection between $f(a)$ in $X$ and $g(a)$ in $Y$. The pushout is the set
  of connected components of $X \dju Y$. 
\end{example}

\begin{exercise} %
\label{exc.pushout}
  What is the pushout of the functions $f\colon \ord{4} \to \ord{5}$ and $g\colon \ord{4} \to
  \ord{3}$ pictured below?
  
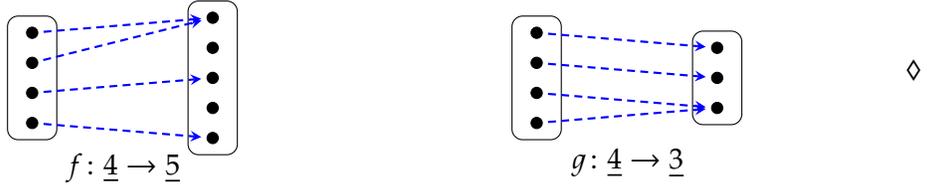
\[ 
    \begin{tikzpicture}[x=.6cm,y=.4cm, baseline=(3a), short=2pt]
      \node at (2,-1.5) {$f\colon \ord{4} \to \ord{5}$};
      \node[contact] at (0,3) (1) {};
      \node[contact] at (0,2) (2) {};
      \node[contact] at (0,1) (3) {};
      \node[contact] at (0,0) (4) {};
      \node[contact] at (4,3.5) (1a) {};
      \node[contact] at (4,2.5) (2a) {};
      \node[contact] at (4,1.5) (3a) {};
      \node[contact] at (4,0.5) (4a) {};
      \node[contact] at (4,-0.5) (5a) {};
      \node[draw, rounded corners, inner xsep=7pt, inner ysep=4pt, fit=(1) (4)] (B) {};
      \node[draw, rounded corners, inner xsep=7pt, inner ysep=4pt, fit=(1a) (5a)] (B) {};
      \begin{scope}[mapsto]
      	\draw (1) to (1a);
      	\draw (2) to (1a);
      	\draw (3) to (3a);
      	\draw (4) to (5a);
      \end{scope}
    \end{tikzpicture}
    \hspace{.25\textwidth}
    \begin{tikzpicture}[x=.6cm,y=.4cm, baseline=(2a),short=2pt]
      \node at (2,-1.35) {$g\colon \ord{4} \to \ord{3}$};
      \node[contact] at (0,3) (1) {};
      \node[contact] at (0,2) (2) {};
      \node[contact] at (0,1) (3) {};
      \node[contact] at (0,0) (4) {};
      \node[contact] at (4,2.5) (1a) {};
      \node[contact] at (4,1.5) (2a) {};
      \node[contact] at (4,0.5) (3a) {};
      \node[draw, rounded corners, inner xsep=7pt, inner ysep=4pt, fit=(1) (4)] (B) {};
      \node[draw, rounded corners, inner xsep=7pt, inner ysep=4pt, fit=(1a) (3a)] (B) {};
      \begin{scope}[mapsto]
      	\draw (1) to (1a);
      	\draw (2) to (2a);
      	\draw (3) to (3a);
      	\draw (4) to (3a);
      \end{scope}
    \end{tikzpicture}
    \qedhere
  \]
  Check your answer using the abstract description from \cref{ex.pushouts}.
\end{exercise}

\begin{example}%
\label{ex.pushout_initial}
Suppose a category $\cat{C}$ has an initial object $\varnothing$. For any two objects $X,Y\in\Ob\cat{C}$, there is a unique morphism $f\colon\varnothing\to X$ and a unique morphism $g\colon\varnothing\to Y$; this is what it means for $\varnothing$ to be initial.

The diagram $X\From{f}\varnothing\To{g}Y$ has a pushout in $\cat{C}$ iff $X$ and $Y$ have a coproduct in $\cat{C}$, and the pushout and the coproduct will be the same. Indeed, suppose $X$ and $Y$ have a coproduct $X+ Y$; then the diagram to the left
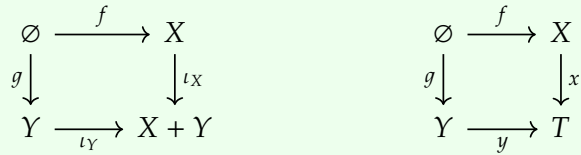
\[
\begin{tikzcd}
	\varnothing\ar[r,"f"]\ar[d,"g"']&X\ar[d,"\iota_X"]\\
	Y\ar[r,"\iota_Y"']&X+ Y
\end{tikzcd}
\hspace{1in}
\begin{tikzcd}
	\varnothing\ar[r,"f"]\ar[d,"g"']&X\ar[d,"x"]\\
	Y\ar[r,"y"']&T
\end{tikzcd}
\]
commutes (why?$^1$), and for any object $T$ and commutative diagram as to the
right, there is a unique map $X+ Y\to T$ making the diagram as in
\cref{eqn.univ_prop_pushout} commute (why?$^2$). This shows that $X+ Y$ is
a pushout, $X+_\varnothing Y\cong X+ Y$.

Similarly, if a pushout $X+_\varnothing Y$ exists, then it satisfies the universal property of the coproduct (why?$^3$).
\end{example}

\begin{exercise} %
\label{exc.pushout_initial}
In \cref{ex.pushout_initial} we asked ``why?'' three times.
\begin{enumerate}
	\item Give a justification for ``why?$^1$''.
	\item Give a justification for ``why?$^2$''.
	\item Give a justification for ``why?$^3$''.	
\qedhere
\end{enumerate}
\end{exercise}

\begin{example}%
\label{ex.pushouts_quotient}
Let $A=X=Y=\NN$. Consider the functions $f\colon A\to X$ and $g\colon A\to Y$ given by the `floor' functions, 
$f(a)\coloneqq\floor{a/2}$ and $g(a)\coloneqq\floor{(a+1)/2}$.
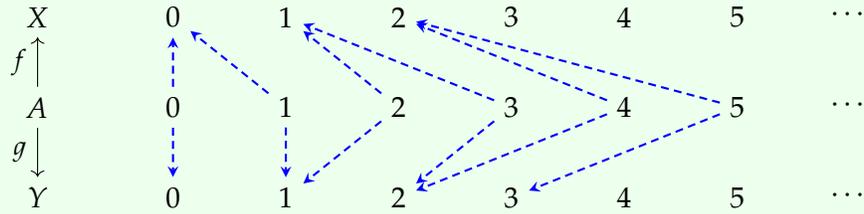
\[
\begin{tikzpicture}[xscale=1.5,yscale=1.2]
 \node at (-1.2,1) (x) {$X$};
 \node at (-1.2,0) (a) {$A$};
 \node at (-1.2,-1) (y) {$Y$};
  \foreach \i in {0,...,5}{
  \node at (\i,1) (x\i) {$\i$};
  \node at (\i,0) (a\i) {$\i$};
  \node at (\i,-1) (y\i) {$\i$};}
  \node at (6,1) {$\cdots$};
  \node at (6,0) {$\cdots$};
  \node at (6,-1) {$\cdots$};
  \draw[->] (a) to node[left] {\small $f$} (x);
  \draw[->] (a) to node[left] {\small $g$} (y);
  \begin{scope}[mapsto]
  \draw (a0) to (x0);
  \draw (a0) to (y0);
  \draw (a1) to (x0);
  \draw (a1) to (y1);
  \draw (a2) to (x1);
  \draw (a2) to (y1);
  \draw (a3) to (x1);
  \draw (a3) to (y2);
  \draw (a4) to (x2);
  \draw (a4) to (y2);
  \draw (a5) to (x2);
  \draw (a5) to (y3);
  \end{scope}
\end{tikzpicture}
\]
What is their pushout? Let's figure it out using the definition.

If $T$ is any other set and we have maps $x\colon X\to T$ and $y\colon Y\to T$ that commute with $f$ and $g$, i.e.\ $f\cong x=g\cong y$, then this commutativity implies that
\[y(0)=y(g(0))=x(f(0))=x(0).\]
In other words, $Y$'s 0 and $X$'s 0 go to the same place in $T$, say $t$. But since $f(1)=0$ and $g(1)=1$, we also have that $t=x(0)=x(f(1))=y(g(1))=y(1)$. This means $Y$'s 1 goes to $t$ also. But since $g(2)=1$ and $f(2)=1$, we also have that $t=g(1)=y(g(2))=x(f(2))=x(1)$, which means that $X$'s 1 also goes to $t$. One can keep repeating this and find that every element of $Y$ and every element of $X$ go to $t$! Using mathematical induction, one can prove that the pushout is in fact a 1-element set, $X\sqcup_AY\cong\{1\}$.
\index{induction}
\end{example}

\index{pushout|)}

\subsection{Finite colimits}%
\index{colimit!finite|(}

Initial objects, coproducts, and pushouts are all types of colimits. We gave the
general definition of colimit in \cref{subsec.brief_colimits}. Just as a limit
in $\cat{C}$ is a terminal object in a category of cones over a diagram $D\colon\cat{J}\to\cat{C}$, a colimit is an initial
object in a category of cocones over some diagram $D\colon\cat{J}\to\cat{C}$. For our purposes it is enough to discuss finite colimits---i.e.\ when $\cat{J}$ is a finite category---which subsume initial objects, coproducts, and pushouts.%
\footnote{If a category $\cat{J}$ has finitely many morphisms, we say that $\cat{J}$ is a \emph{finite category}. Note that in this case it must have finitely many objects too, because each object $j\in\Ob\cat{J}$ has its own identity morphism $\id_j$.}%
\index{category!of cocones}

In \cref{def.colimit}, cocones in $\cat{C}$ are defined to be cones in $\cat{C}\op$. For visualization purposes, if $D\colon\cat{J}\to\cat{C}$ looks like the diagram to the left, then a cocone on it shown in the diagram to the right:
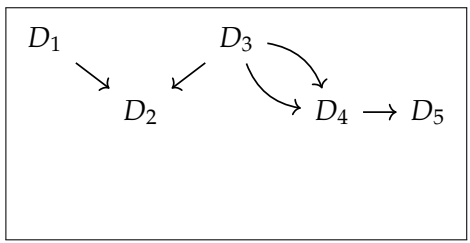
\[
\boxCD{
  \begin{tikzcd}[sep=small, ampersand replacement=\&]
    D_1\ar[dr]\&\&D_3\ar[ld]\ar[dr, bend left]\ar[dr, bend right]\\
    \&D_2\&\&D_4\ar[r]\&D_5\\[10pt]
    {\color{white}C}
  \end{tikzcd}
}
\hspace{.8in}
\boxCD
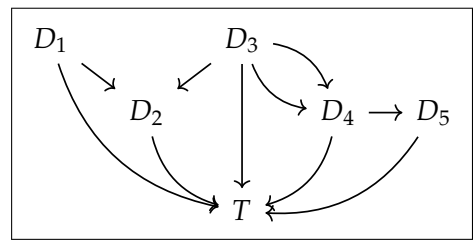
{
  \begin{tikzcd}[sep=small, ampersand replacement=\&]
    D_1\ar[dr]\ar[ddrr, bend right]\&\&D_3\ar[ld]\ar[dd]\ar[dr, bend left]\ar[dr, bend right]\\
    \&D_2\ar[dr, bend right]\&\&D_4\ar[dl, bend left]\ar[r]\&D_5\ar[dll, bend left]\\[10pt]
    \&\&T
  \end{tikzcd}
}
\]
Here, any two parallel paths that end at $T$ are equal in $\cat{C}$.

\begin{definition}%
\index{category!having finite colimits}
We say that a category $\cat{C}$ \emph{has finite colimits} if a colimit, $\colim_\cat{J} D$, exists whenever $\cat{J}$ is a finite category and $D\colon\cat{J}\to\cat{C}$ is a diagram.
\end{definition}

\begin{example}%
\index{initial object!as colimit}
The initial object in a category $\cat{C}$, if it exists, is the colimit of the
functor $!\colon \Cat{0} \to \cat{C}$, where $\Cat{0}$ is the category with no
objects and no morphisms, and $!$ is the unique such functor. Indeed, a cocone
over $!$ is just an object of $\cat{C}$, and so the initial cocone over $!$ is
just the initial object of $\cat{C}$.

Note that $\Cat{0}$ has finitely many objects (none); thus initial objects are finite
colimits.
\end{example}

We often want to know that a category $\cat{C}$ has \emph{all} finite colimits (in which case, we often drop the `all' and just say `$\cat{C}$ has finite colimits'). To check that $\cat{C}$ has (all) finite colimits, it's enough to check it has a few simpler forms of colimit, which generate all the rest. 
\begin{proposition}%
\label{prop.finite_colimits}
Let $\cat{C}$ be a category. The following are equivalent:
\begin{enumerate}
	\item $\cat{C}$ has all finite colimits.
	\item $\cat{C}$ has an initial object and all pushouts.
	\item $\cat{C}$ has all coequalizers and all finite coproducts.
\end{enumerate}
\end{proposition}
\begin{proof}
We will not give precise details here, but the key idea is an inductive one: one can build arbitrary finite diagrams using some basic building blocks. Full details can be found in \cite[Prop 2.8.2]{Borceux:1994a}.
\end{proof}

\begin{example}%
\label{ex.W_colim}
Let $\cat{C}$ be a category with all pushouts, and suppose we want to take the colimit of the following diagram in $\cat{C}$:
\begin{equation}%
\label{eqn.W_colim}
\begin{tikzcd}
	&
	B\ar[r]\ar[d]&
	Z\\
	A\ar[r]\ar[d]&
	C\\
	D
\end{tikzcd}
\end{equation}
In it we see two diagrams ready to be pushed out, and we know how to take pushouts. So suppose we do that; then we see another pushout diagram so we take the pushout again:
\[
\begin{tikzcd}
	&
	B\ar[r]\ar[d]&
	Z\ar[d]\\
	A\ar[r]\ar[d]&
	Y\ar[r]\ar[d]&
	R\ar[ul, phantom, very near start, "\ulcorner"]\\
	X\ar[r]&
	Q\ar[ul, phantom, very near start, "\ulcorner"]
\end{tikzcd}
\hspace{1in}
\begin{tikzcd}
	&
	B\ar[r]\ar[d]&
	Z\ar[d]\\
	A\ar[r]\ar[d]&
	Y\ar[r]\ar[d]&
	R\ar[ul, phantom, very near start, "\ulcorner"]\ar[d]\\
	X\ar[r]&
	Q\ar[ul, phantom, very near start, "\ulcorner"]\ar[r]&
	S\ar[ul, phantom, very near start, "\ulcorner"]
\end{tikzcd}
\]
is the result---consisting of the object $S$, together with all the morphisms from the original diagram to $S$---the colimit of the original diagram? One can check that it indeed has the correct universal property and thus is a colimit.
\end{example}

\begin{exercise}%
\label{exc.W_colim}
Check that the pushout of pushouts from \cref{ex.W_colim} satisfies the universal property of the colimit for the original diagram, \cref{eqn.W_colim}.
\end{exercise}

We have already seen that the categories $\finset$ and $\smset$ both have an
initial object and pushouts. We thus have the following corollary.
\begin{corollary} %
\label{cor.set_cocomplete}
The categories $\finset$ and $\smset$ have (all) finite colimits.
\end{corollary}

In \cref{thm.set_limits} we gave a general formula for computing finite limits
in $\smset$. It is also possible to give a formula for computing finite colimits. There is a duality between products and coproducts and between subobjects and quotient objects, so whereas a finite limit is given by a subset of a product, a finite colimit is given by a quotient of a coproduct.%
\index{dual notions!subobjects and quotients}%
\index{quotient}

\begin{theorem}%
\label{thm.colims_in_set}%
\index{colimit!formula for finite colimits in $\smset$}
Let $\cat{J}$ be presented by the finite graph $(V,A,s,t)$ and some equations, and let $D\colon \cat{J} \to \smset$ be a diagram. Consider the set
\[\colim_{\cat{J}} D \coloneqq \big\{(v,d)\mid v\in V \text{ and }d\in D(v)\big\}/\sim\]
where this denotes the set of equivalence classes under the equivalence relation
$\sim$ generated by putting $(v,d)\sim (w,e)$ if there is an arrow $a\colon v\to w$ in $J$ such that $D(a)(d)=e$. Then this set, together with the functions $\iota_v\colon D(v)\to\colim_{\cat{J}} D$ given by sending $d\in D(v)$ to its equivalence class, constitutes a colimit of $D$.%
\index{equivalence relation}
\index{quotient}
\end{theorem}

\begin{example} %
\label{ex.initial_obj}
   Recall that an initial object is the colimit on the empty graph. The formula thus says
   the initial object in $\smset$ is the empty set $\varnothing$: there are no $v
   \in V$.
\end{example}

\begin{example} %
\label{ex.coproduct}
A coproduct is a colimit on the graph $\cat{J}=\fbox{$\LMO{v_1}\quad\LMO{v_2}$}$. A functor $D\colon\cat{J}\to\smset$ can be identified with a choice of two sets, $X\coloneqq D(v_1)$ and $Y\coloneqq D(v_2)$. Since there are no arrows in $\cat{J}$, the equivalence relation $\sim$ is vacuous, so the
formula in \cref{{thm.colims_in_set}} says that a coproduct is given by
\[\{(v,d) \mid d \in D(v), \text{ where } v=v_1 \text{ or }v=v_2\}.\]
In other words, the coproduct of sets $X$ and $Y$ is their disjoint
union $X \sqcup Y$, as expected.
\end{example}

\begin{example}
If $\cat{J}$ is the category $\Cat{1} = \fbox{$\LMO{v}$}\,$, the formula in \cref{thm.colims_in_set}
yields the set
  \[\{(v,d) \mid d \in D(v)\}\]
This is isomorphic to the set $D(v)$. In other words, if $X$ is a set considered
as a diagram $X \colon \Cat{1} \to \smset$, then its colimit (like its limit)
is just $X$ again.
\end{example}

\begin{exercise} %
\label{exc.pushout_formula}
Use the formula in \cref{thm.colims_in_set} to show that pushouts---colimits on
a diagram $X \xleftarrow{f} N \xrightarrow{g} Y$---agree with the description we
gave in \cref{ex.pushouts}.
\end{exercise}

\begin{example}%
\index{coequalizer} %
\label{ex.coequalizer}%
\index{colimit!coequalizer as}
Another important type of finite colimit is the \emph{coequalizer}. These are
colimits over the graph \fbox{$\LMO{}\tto \LMO{}$} consisting of two
parallel arrows.

Consider some diagram $\begin{tikzcd} X \ar[r, "f", shift left=3pt]\ar[r, "g"',
shift right=3pt] & Y\end{tikzcd}$ on this graph in $\smset$. The coequalizer of
this diagram is the set of equivalence classes of $Y$ under equivalence relation
generated by declaring $y \sim y'$ whenever there exists $x$ in $X$ such that $f(x)=y$ and $g(x)=y'$.%
\index{equivalence relation}

Let's return to the example circuit in the introduction to hint at why colimits
are useful for interconnection.%
\index{colimit!and interconnection} Consider the following picture:%
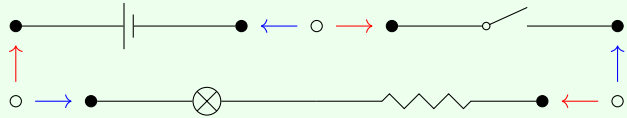
\index{electrical circuit}
\[
\begin{tikzpicture}[circuit ee IEC, set resistor graphic=var resistor IEC
graphic, set make contact graphic=var make contact IEC graphic]
\node [contact] (1) at (1,0) {};
\coordinate (2) at (4,0) {};
\node [contact] (3) at (7,0) {};
\node [contact] (5) at (0,1) {};
\node [contact] (6) at (3,1) {};
\node [contact] (7) at (5,1) {};
\node [contact] (8) at (8,1) {};
\node [draw, inner sep=1.5pt,circle] (a) at (0,0) {};
\node [draw, inner sep=1.5pt,circle] (b) at (4,1) {};
\node [draw, inner sep=1.5pt,circle] (c) at (8,0) {};
\draw (1) to [bulb] (2);
\draw (2) to [resistor] (3);
\draw (5) to [battery] (6);
\draw (7) to [make contact] (8);
\begin{scope}[shorten <=5pt, shorten >=5pt, ->,red]
\draw (a) to (5);
\draw (b) to (7);
\draw (c) to (3);
\end{scope}
\begin{scope}[shorten <=5pt, shorten >=5pt, ->,blue]
\draw (a) to (1);
\draw (b) to (6);
\draw (c) to (8);
\end{scope}
\end{tikzpicture}
\]
We've redrawn this picture with one change: some of the arrows are now red, and
others are now blue. If we let $X$ be the set of white circles $\circ$, and $Y$
be the set of black circles $\bullet$, the blue and red arrows respectively
define functions $f,g\colon X \to Y$. Let's leave the actual circuit components out of the picture for now; we're just interested in the dots. What is the coequalizer?

It is a three element set, consisting of one element for each newly-connected
pair of $\bullet$'s . Thus the colimit describes the set of terminals after
performing the interconnection operation. In \cref{sec.decorated_cospans} we'll
see how to keep track of the circuit components too.%
\index{electrical circuit}
\end{example}

\index{colimit|)}%
\index{colimit!finite|)}

\subsection{Cospans}%
\label{subsec.cospans}
When a category $\cat{C}$ has finite colimits, an extremely useful way to package them is by considering the category of cospans in $\cat{C}$.%
\index{cospan|(}

\begin{definition}%
\index{cospan}
Let $\cat{C}$ be a category. A \emph{cospan} in $\cat{C}$ is just a pair of morphisms to a common object $A\to N\from B$. The common object $N$ is called the \emph{apex} of the cospan and the other two objects $A$ and $B$ are called its \emph{feet}.%
\index{cospan!foot of}%
\index{cospan!apex of}
\end{definition}

\index{cospans!composition of}
If we want to say that cospans form a category, we should begin by saying how composition would work. So suppose we have two cospans in $\cat{C}$%
\index{category!of cospans}
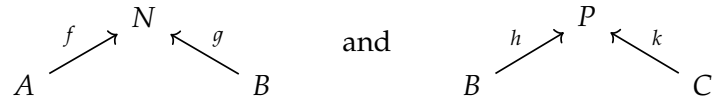
\[
\begin{tikzcd}[row sep=small]
& N \\
A \ar[ur,"f"] && B \ar[ul,"g"'] 
\end{tikzcd}
\qquad\mbox{and}\qquad
\begin{tikzcd}[row sep=small]
& P \\
B \ar[ur,"h"] && C \ar[ul,"k"'] 
\end{tikzcd}
\]
Since the right foot of the first is equal to the left foot of the second, we
might stick them together into a diagram like this:
  
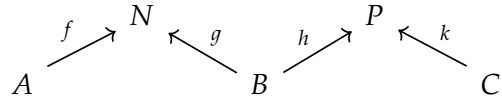
\[
  \begin{tikzcd}[row sep=small]
      & N && P  \\
      \quad A \ar[ur,"f"] && B \ar[ul, "g"'] \ar[ur, "h"] && C
      \quad \ar[ul, "k"']
  \end{tikzcd}
  \]
Then, if a pushout of $N \xleftarrow{g} B \xrightarrow{h} P$ exists in $\cat{C}$, as shown on the left, we can extract a new cospan in $\cat{C}$, as shown on the right:
  
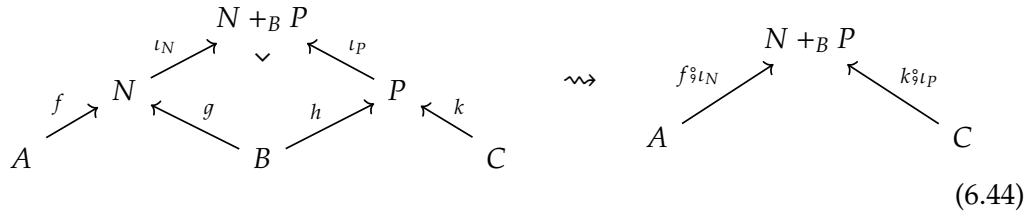
\begin{equation}%
\label{eqn.composecospan}
  \begin{tikzcd}[row sep=small, column sep=20pt]
     && N+_BP \ar[dd, phantom, very near start, "\rotatebox{-45}{$\lrcorner$}"]\\
      & N \ar[ur,"\iota_N"] && P \ar[ul, "\iota_P"'] \\
       \quad A  \ar[ur,"f"] && B \ar[ul, "g"'] \ar[ur, "h"] &&  C \quad
       \ar[ul, "k"']
  \end{tikzcd}
  \:
  \rightsquigarrow
  \quad
\begin{tikzcd}
& N+_BP \\
A \ar[ur,"f\cp \iota_N"] && C \ar[ul,"k\cp\iota_P"'] 
\end{tikzcd}  
  \end{equation}%
\index{pushout!in cospan composition}

It might look like we have achieved our goal, but we're missing some things. First, we need an identity on every object $C\in\Ob\cat{C}$; but that's not hard: use $C\to C\from C$ where both maps are identities in $\cat{C}$. More importantly, we don't know that $\cat{C}$ has all pushouts, so we don't know that every two sequential morphisms $A\to B\to C$ can be composed. And beyond that, there is a technical condition that when we form pushouts, we only get an answer `up to isomorphism': anything isomorphic to a pushout counts as a pushout (check the definition to see why). We want all these different choices to count as the same thing, so we define two cospans to be equivalent iff there is an isomorphism between their respective apexes. That is, the cospan $A\to P\from B$ and $A\to P'\from B$ in the diagram shown left below are equivalent iff there is an isomorphism $P\cong P'$ making the diagram to the right commute:
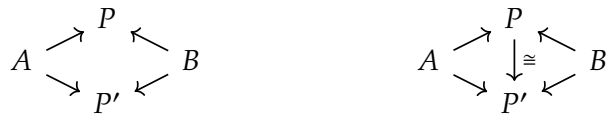
\[
\begin{tikzcd}[sep=small, row sep=0]
	&P\\A\ar[ru]\ar[rd]&&B\ar[lu]\ar[ld]\\
	&P'
\end{tikzcd}
\hspace{1in}
\begin{tikzcd}[sep=small, row sep=0]
	&P\ar[dd, "\cong"]\\A\ar[ru]\ar[rd]&&B\ar[lu]\ar[ld]\\
	&P'
\end{tikzcd}
\]

Now we are getting somewhere. As long as our category $\cat{C}$ has pushouts, we are in business: $\cospan{\cat{C}}$ will form a category. But in fact, we are very close to getting more. If we also demand that $\cat{C}$ has an initial object $\varnothing$ as well, then we can upgrade $\cospan{\cat{C}}$ to a symmetric monoidal category. 

Recall from \cref{prop.finite_colimits} that a category $\cat{C}$ has all finite colimits iff it has an initial object and all pushouts.

\begin{definition}%
\label{def.cospan_sym_mon_cat}
  Let $\cat{C}$ be a category with finite colimits. Then there exists a
  category $\cospan{\cat{C}}$ with the same objects as $\cat{C}$, i.e.\ $\Ob(\cospan{\cat{C}})=\Ob(\cat{C})$, where the morphisms $A \to B$ are the (equivalence classes of)
  cospans from $A$ to $B$, and composition is given by the above pushout construction.
  
  There is a symmetric monoidal structure on this category, denoted
  $(\cospan{\cat{C}},\varnothing,+)$. The monoidal unit is the initial object
  $\varnothing\in\cat{C}$ and the monoidal product is given by coproduct. The coherence
  isomorphisms, e.g.\ $A+\varnothing\cong A$, can be defined in a similar way to
  those in \cref{exc.coproducts_give_monoidal_structure}. 
\end{definition}

It is a straightforward but time-consuming exercise to verify that $(\cospan{\cat{C}},\varnothing,+)$ from \cref{def.cospan_sym_mon_cat} really does satisfy all the axioms of a symmetric
monoidal category, but it does.

\begin{example} %
\label{ex.cospan_finset}%
\index{cospans!category of}
  The category $\finset$ has finite colimits (see \ref{cor.set_cocomplete}). So,
  we can define a symmetric monoidal category $\cospan{\finset}$. What does it
  look like? It looks a lot like wires connecting ports. 

  The objects of $\cospan\finset$ are finite sets; here let's draw them as
  collections of $\bullet$'s. The morphisms are cospans of functions.  Let $A$ and
  $N$ be five element sets, and $B$ be a six element set.  Below are two
  depictions of a cospan $A \To{f} N \From{g} B$. 
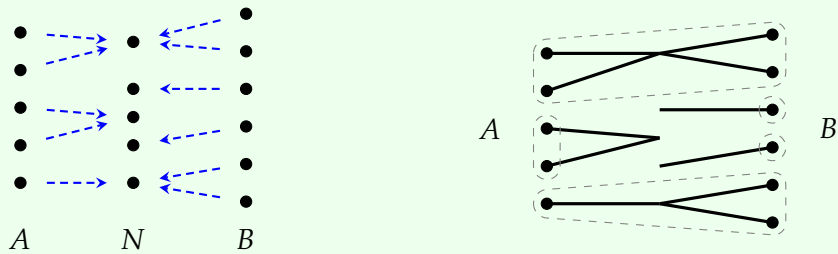
  \[
    \begin{aligned}

    \end{aligned}
  \]
  In the depiction on the left, we simply represent the functions $f$ and $g$
  by drawing arrows from each $a\in A$ to $f(a)$ and each $b\in B$ to $g(b)$. In the depiction on the
  right, we make this picture resemble wires a bit more, simply drawing a wire
  where before we had an arrow, and removing the unnecessary center dots. We also draw a dotted line around points that are
  connected, to emphasize an important perspective, that cospans establish that
  certain ports are connected, i.e.\ part of the same
  equivalence class.

  The monoidal category $\cospan\finset$ then provides two operations for
  combining cospans: composition and monoidal product. Composition is given by
  taking the pushout of the maps coming from the common foot, as described in
  \cref{def.cospan_sym_mon_cat}.  Here is an example of cospan composition,
  where all the functions are depicted with arrow notation:
  
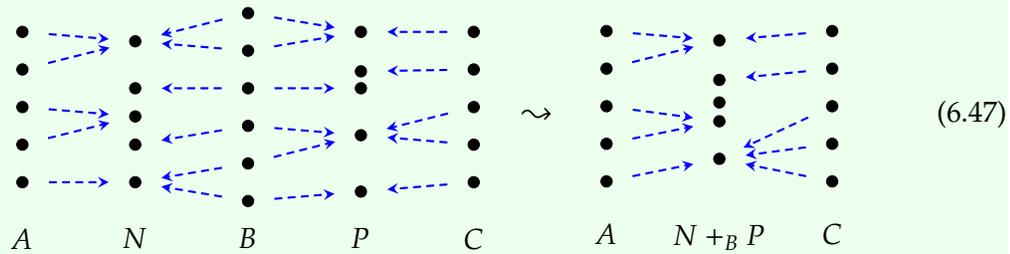
\begin{equation} %
\label{eq.cospan_comp}
    \begin{aligned}

    \end{aligned}
  \end{equation}
  The monoidal product is given simply by the disjoint union of two cospans; in
  pictures it is simply combining two cospans by stacking one above another.
\end{example}%
\index{monoidal product!as stacking}

\begin{exercise} %
\label{exc.cospan_tensor}
In \cref{eq.cospan_comp} we showed morphisms $A\to B$ and $B\to C$ in
$\cospan{\finset}$. Draw their monoidal product as a morphism $A+B\to
B+C$ in $\cospan{\finset}$.
\end{exercise}

\begin{exercise} %
\label{exc.cospan_comp}
  Depicting the composite of cospans in \cref{eq.cospan_comp} with the wire
  notation gives
  \begin{equation} %
\label{eq.cospan_comp_wires}
    \begin{aligned}

    \end{aligned}
  \end{equation}
  Comparing \cref{eq.cospan_comp} and \cref{eq.cospan_comp_wires}, describe the
  composition rule in $\cospan\finset$ in terms of wires and connected
  components.
\end{exercise}

\index{cospan|)}

\section{Hypergraph categories}%
\index{hypergraph category|(}
A hypergraph category is a type of symmetric monoidal category whose wiring
diagrams are networks. We will soon see that electric circuits can be organized into a hypergraph category; this is what we've been building up to. But to define hypergraph categories, it is useful to first introduce Frobenius monoids.

\subsection{Frobenius monoids} %
\label{sec.escfm}%
\index{Frobenius!structure|(}
The pictures of cospans we saw above, e.g.\ in \cref{eq.cospan_comp_wires} look
something like icons in signal flow graphs (see \cref{subsec.icons_sfgs}):
various wires merge and split, initialize and terminate. And these follow the
same rules they did for linear relations, which we briefly discussed in
\cref{exc.linear_relations}. There's a lot of potential for confusion, so let's start
from scratch and build back up.%
\index{interconnection!via Frobenius structures}%
\index{icon}

In any symmetric monoidal category $(\mathcal C,I,\otimes)$, recall from \cref{subsec.reflection_wds} that objects can be drawn as wires and morphisms can be drawn as boxes. Particularly noteworthy morphisms might be iconified as dots rather than boxes, to indicate that the morphisms there are not arbitrary but notation-worthy. One case of this is when there is an object $X$ with special ``abilities'', e.g. the ability to duplicate into two, or disappear into nothing. 

To make this precise, recall from \cref{def.monoid_object} that a commutative monoid
$(X,\mu,\eta)$ in symmetric monoidal category $(\cat{C},I,\otimes)$ is an
object $X$ of $\mathcal C$ together with (noteworthy) morphisms
\[
  \xymatrixrowsep{1pt}
  \xymatrix{
    \mult{.075\textwidth} & & \unit{.075\textwidth} \\
    \mu\colon X\otimes X \to X & & \eta\colon I \to X
  }
\]
obeying
\begin{equation}%
\label{eqn.ass_un_comm}
  \xymatrixrowsep{1pt}
  \xymatrixcolsep{25pt}
  \xymatrix{
    \assocl{.1\textwidth} = \assocr{.1\textwidth} & \unitl{.1\textwidth} =
    \idone{.1\textwidth} & \commute{.1\textwidth} = \mult{.07\textwidth} \\
    \textrm{(associativity)} & \textrm{(unitality)} & \textrm{(commutativity)}
  }
\end{equation}
where $\swap{1em}$ is the symmetry on $X \otimes X$. A cocommutative cocomonoid
$(X,\delta,\epsilon)$ is an object $X$ with maps $\delta\colon X \to X \otimes
X$, $\epsilon\colon X \to I$, obeying the mirror images of the laws in \cref{eqn.ass_un_comm}.%
\index{associativity!of monoid operation, wiring diagram for}%
\index{unitality!of monoid operation, wiring diagram for}%
\index{symmetry!of monoid operation, wiring diagram for}

Suppose $X$ has both the structure of a commutative monoid and cocommutative
comonoid, and consider a wiring diagram built only from the icons $\mu,\eta,\delta,$ and $\epsilon$, where every wire is labeled $X$. These diagrams have a left and right, and are pictures of how ports on the left are connected to ports on the right. The commutative monoid and cocommutative comonoid axioms thus both
express when to consider two such connection pictures should be considered the
same. For example, associativity says the order of connecting ports on the left
doesn't matter; coassociativity (not drawn) says the same for the right.%
\index{comonoid}

If you want to go all the way and say ``all I care about is which port is connected to which; I don't even care about left and right'', then you need a few more axioms to say how the morphisms $\mu$ and $\delta$, the merger and the splitter, interact.

\begin{definition}%
\label{def.spec_comm_frob_mon}%
\index{Frobenius!structure}
  Let $X$ be an object in a symmetric monoidal category $(\cat{C},\otimes,I)$. A \emph{Frobenius structure} on $X$ consists of a 4-tuple $(\mu,\eta,\delta,\epsilon)$ such that $(X,\mu,\eta)$ is a commutative monoid and 
  $(X,\delta,\epsilon)$ is a cocommutative comonoid, which satisfies the six equations above ((co-)associativity, (co-)unitality, (co-)commutativity), as well as the following three equations:
  \begin{equation}%
\label{eqn.frob_laws}
  \xymatrixrowsep{1pt}
  \xymatrixcolsep{25pt}
  \xymatrix{
    \frobs{.1\textwidth} = \frobx{.1\textwidth} = \frobz{.1\textwidth} & \spec{.1\textwidth} =
    \idone{.1\textwidth}  \\
    \textrm{(the Frobenius law)} & \textrm{(the special law)}%
\index{Frobenius!law}
    }
  \end{equation}
  We refer to an object $X$ equipped with a Frobenius structure as a \emph{special commutative Frobenius monoid}, or just \emph{Frobenius monoid} for short.%
\index{Frobenius!monoid}
\end{definition}

With these two equations, it turns out that two morphisms $X^{\otimes m}\to
X^{\otimes n}$---defined by composing and tensoring identities on $X$ and the
noteworthy morphisms $\mu,\delta$, etc.---are equal if and only if their string
diagrams connect the same ports. This link between connectivity, and Frobenius
monoids can be made precise as follows. 

\begin{definition}%
\index{spider}
  Let $(X,\mu,\eta,\delta,\epsilon)$ be a Frobenius monoid
  in a monoidal category $(\cat{C},I,\otimes)$. Let $m,n \in\NN$. Define 
  $s_{m,n}\colon X^{\otimes m} \to X^{\otimes n}$ to be the following morphism
 
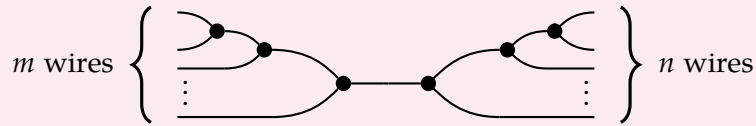
\[
  \begin{tikzpicture}[spider diagram, decoration={brace, amplitude=7pt}]
  	\node[special spider={2}{1}{1.8*\leglen}{\leglen}, fill=black] (a1) {};
  	\node[special spider={2}{1}{\leglen}{0pt}, fill=black, left=0 of a1_in1] (a2) {};
		\node[special spider={2}{1}{\leglen}{0pt}, fill=black, left=0 of a2_in1] (a3) {};
		\node[special spider={1}{2}{\leglen}{1.8*\leglen}, fill=black, right=.5 of a1_out1] (c1) {};
		\node[special spider={1}{2}{0pt}{\leglen}, fill=black, right=0 of c1_out1] (c2) {};
		\node[special spider={1}{2}{0pt}{\leglen}, fill=black, right=0 of c2_out1] (c3) {};
		\draw (a1_out1) -- (c1_in1);
		\draw (a1_in2) -- (a1_in2-|a3_in1) node[coordinate] (b1) {};
		\draw (a2_in2) -- (a2_in2-|a3_in1) node[coordinate] (b2) {};
		\draw (c1_out2) -- (c1_out2-|c3_out1) node[coordinate] (d1) {};
		\draw (c2_out2) -- (c2_out2-|c3_out1) node[coordinate] (d2) {};
		\node at ($(b1)!.5!(b2)+(3pt,3pt)$) {$\vdots$};
		\node at ($(d1)!.5!(d2)+(-3pt,3pt)$) {$\vdots$};
		\draw[decorate, very thick] ($(b1)+(-10pt,-2pt)$) -- ($(a3_in1)+(-10pt,2pt)$);
		\draw[decorate, very thick] ($(c3_out1)+(10pt,2pt)$) -- ($(d1)+(10pt,-2pt)$);
		\node[anchor=east] at ($(b2)+(-20pt,2pt)$) {$m$ wires};
		\node[anchor=west] at ($(d2)+(20pt,2pt)$) {$n$ wires};
  \end{tikzpicture}
  \]
	It can be written formally as $(m-1)$ $\mu$'s followed by $(n-1)$ $\delta$'s, with special cases when $m=0$ or $n=0$.

  We call $s_{m,n}$ the \emph{spider of type $(m,n)$}, and can draw it more simply as the icon
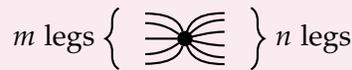
\[
\begin{tikzpicture}[spider diagram, decoration={brace, amplitude=5pt}]
	\node[spider={5}{6}, fill=black] (a) {};
	\draw[decorate, thick] ($(a_in5)+(-10pt,-2pt)$) -- ($(a_in1)+(-10pt,2pt)$);	
	\draw[decorate, thick] ($(a_out1)+(10pt,2pt)$) -- ($(a_out6)+(10pt,-2pt)$);	
	\node[anchor=east] at ($(a.center)+(-30pt,0pt)$) {$m$ legs};
	\node[anchor=west] at ($(a.center)+(30pt,0pt)$) {$n$ legs};
\end{tikzpicture}
\]
\end{definition}%
\index{icon!spider}

So a special commutative Frobenius monoid, aside from being a mouthful, is a `spiderable' wire. You agree that in any monoidal category wiring diagram language, wires represent objects and boxes represent morphisms? Well in our weird way of talking, if a wire is spiderable, it means that we have a bunch of morphisms $\mu,\eta,\delta,\epsilon,\sigma$ that we can combine without worrying about the order of doing so: the result is just ``how many in's, and how many out's'': a spider. Here's a formal statement.%
\index{spiderable wire!Frobenius structure as}

\begin{theorem} %
\label{thm.spider}
  Let $(X,\mu,\eta,\delta,\epsilon)$ be a Frobenius monoid
  in a monoidal category $(\cat{C},I,\otimes)$.  Suppose that we have a map
  $f\colon X^{\otimes m} \to X^{\otimes n}$ each constructed from spiders and
  the symmetry map $\sigma\colon X^{\otimes 2} \to X^{\otimes 2}$ using
  composition and the monoidal product, and such that the string diagram of $f$
  has only one connected component. Then it is a spider: $f=s_{m,n}$.
\end{theorem}

\begin{example}
  As the following two morphisms both (i) have the same number of inputs and
  outputs, (ii) are constructed only from spiders, and (iii) are connected,
  \cref{thm.spider} immediately implies they are equal:
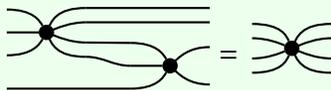
\[
\begin{aligned}

\end{aligned}
\qedhere
\]
\end{example}

\begin{exercise} %
\label{exc.spider}
  Let $X$ be an object equipped with a Frobenius structure. Which of the morphisms $X\otimes X\to X\otimes X\otimes X$ in the following list are necessarily equal?
  \begin{enumerate}
    \item 
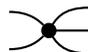
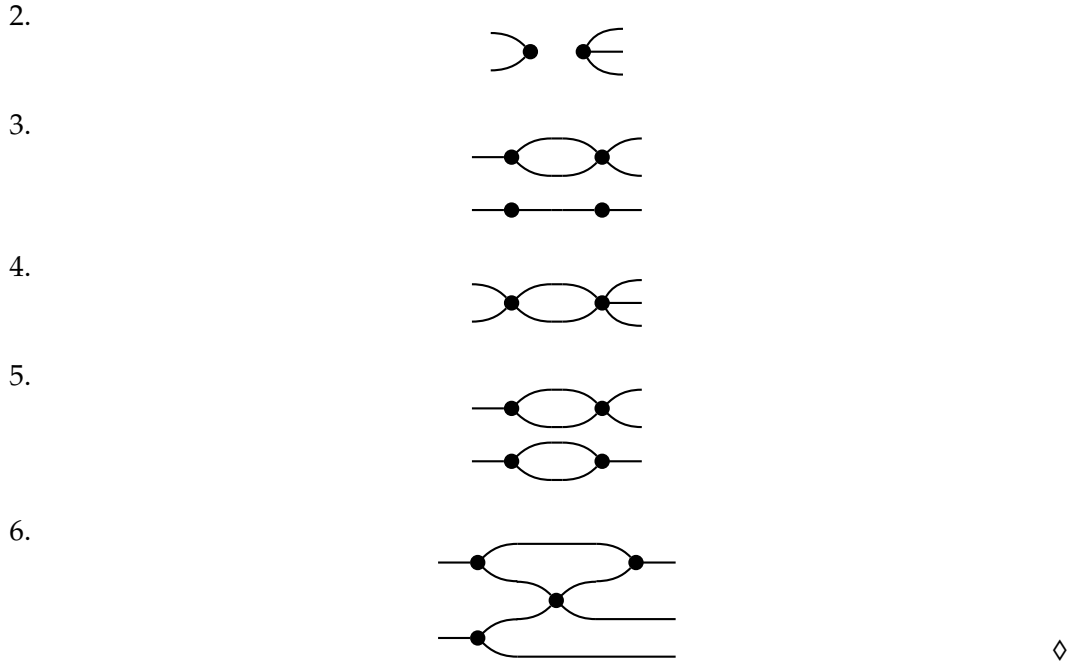
      \[
%
\qedhere
\]
  \qedhere
\end{enumerate}
\end{exercise}

\paragraph{Back to cospans.}
Another way of understanding Frobenius monoids is to relate them to cospans.
Recall the notion of prop presentation from \cref{rdef.presentation_prop}.

\begin{theorem}%
\index{cospans!as theory of Frobenius monoids}
Consider the four-element set $G\coloneqq\{\mu,\eta,\delta,\epsilon\}$ and define $\pgin,\pgout \colon G \to
\nn$ as follows:
\begin{align*}
  \pgin(\mu)&\coloneqq 2,&
  \pgin(\eta)&\coloneqq 0,&
  \pgin(\delta)&\coloneqq 1,&
  \pgin(\epsilon)&\coloneqq 1,
  \\
  \pgout(\mu)&\coloneqq 1,&
  \pgout(\eta)&\coloneqq 1,&
  \pgout(\delta)&\coloneqq 2,&
  \pgout(\epsilon)&\coloneqq 0.
\end{align*}
Let $E$ be the set of Frobenius axioms, i.e.\ the nine equations from \cref{def.spec_comm_frob_mon}. Then the
free prop on $(G,E)$ is equivalent, as a symmetric monoidal category,%
\footnote{
We will not explain precisely what it means to be equivalent as a symmetric
monoidal category, but you probably have some idea: ``they are the same for all category-theoretic intents and purposes.'' The idea is similar to that of
equivalence of categories, as explained in \cref{rem.preorder_boolcats2}.
}
to
$\cospan\finset$.
\end{theorem}

Thus we see that ideal wires, connectivity, cospans, and objects with Frobenius structures are
all intimately related. We use Frobenius structures (all that splitting, merging, initializing, and terminating stuff) as a way to capture the grammar of circuit diagrams.

\subsection{Wiring diagrams for hypergraph categories}%
\index{wiring diagram!for
hypergraph categories}

We introduce hypergraph categories through their wiring diagrams. Just like for
monoidal categories, the formal definition is just the structure required to
unambiguously interpret these diagrams.

Indeed, our interest in hypergraph categories is best seen in their wiring
diagrams. The key idea is that wiring diagrams for hypergraph categories are
network diagrams. This means, in addition to drawing labeled boxes with inputs
and outputs, as we can for monoidal categories, and in addition to bending these
wires around as we can for compact closed categories, we are allowed to split,
join, terminate, and initialize wires.

Here is an example of a wiring diagram that represents a composite of morphisms
in a hypergraph category
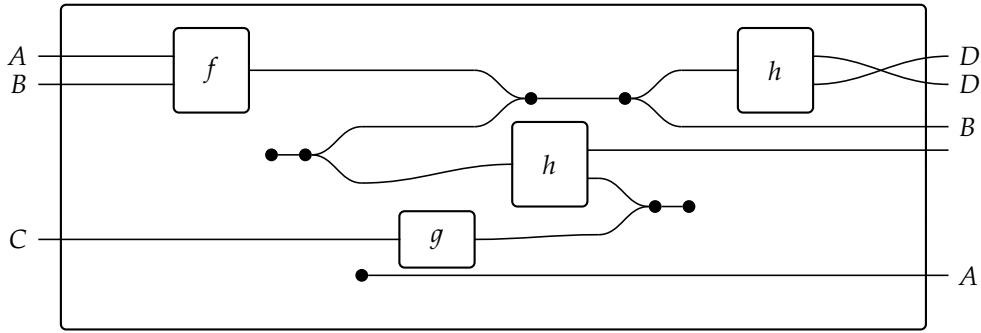
\[

\]
We have suppressed some of the object/wire labels for readability, since all
types can be inferred from the labeled ones.

\begin{exercise}%
\label{exc.suppressed_labels}~ 
\begin{enumerate}
	\item What label should be on the input to $h$?
	\item What label should be on the output of $g$?
	\item What label should be on the fourth output wire of the composite?
	\qedhere
\qedhere
\end{enumerate}
\end{exercise}

Thus hypergraph categories are general enough to talk about all network-style
diagrammatic languages, like circuit diagrams. %
\index{network!language}

\index{Frobenius!structure|)}
\subsection{Definition of hypergraph category}
We are now ready to define hypergraph categories formally. Since the wiring diagrams for hypergraph categories are just those for symmetric
monoidal categories with a few additional icons, the definition is relatively
straightforward: we just want a Frobenius structure on every object. The only
coherence condition is that these interact nicely with the monoidal product.

\begin{definition}%
\index{category!hypergraph structure on}%
\index{hypergraph category}
  A \emph{hypergraph category} is a symmetric monoidal category $(\cat{C},I,\otimes)$ in which each
  object $X$ is equipped with a Frobenius structure
  $(X,\mu_X,\delta_X,\eta_X,\epsilon_X)$ such that 
  \[
    \tikzset{every path/.style={line width=1.1pt}}
    \xymatrixrowsep{1ex}
    
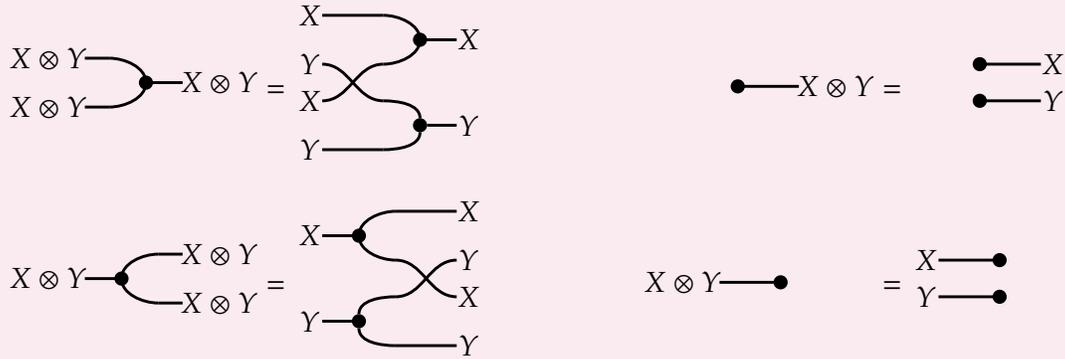
\xymatrix{
    \begin{aligned}

\qquad \quad
  \end{aligned}
}
\]
for all objects $X$, $Y$, and such that $\eta_I=\id_I=\epsilon_I$.

A \emph{hypergraph prop} is a hypergraph category that is also a prop, e.g.\ $\Ob(\cat{C})=\NN$, etc.%
\index{prop!hypergraph}
\end{definition}

\begin{example} %
\label{ex.cospan_hypergraph}%
\index{cospans!hypergraph
  category of|see {hypergraph category, of cospans}}%
\index{hypergraph category!of cospans}
For any $\cat C$ with finite colimits, $\cospan{\cat C}$ is a hypergraph category.
The Frobenius morphisms $\mu_X,\delta_X,\eta_X,\epsilon_X$ for each object $X$ are
constructed using the universal properties of colimits:
\begin{align*}
  \mu_X \coloneq&\quad\big(
  X+X 
  \xrightarrow{\copair{\id_X,\id_X}} X 
  \xleftarrow{\makebox[4.5em]{$\id_X$}} X \big)\\
  \eta_X \coloneq& \quad\big(
  \varnothing 
  \xrightarrow{\makebox[4.5em]{$!_X$}} X
  \xleftarrow{\makebox[4.5em]{$\id_X$}} X \big)\\
  \delta_X \coloneq&\quad\big(
  X 
  \xrightarrow{\makebox[4.5em]{$\id_X$}} X 
  \xleftarrow{\copair{\id_X,\id_X}} X+X \big)\\
  \epsilon_X \coloneq &\quad\big(
  X 
  \xrightarrow{\makebox[4.5em]{$\id_X$}} X 
  \xleftarrow{\makebox[4.5em]{$!_X$}} \varnothing\big)
  \qedhere
\end{align*}
\end{example}

\begin{exercise} %
\label{exc.frob_cospan}
  By \cref{ex.cospan_hypergraph}, the category $\cospan\finset$ is a hypergraph
  category. (In fact, it is equivalent to a hypergraph prop.) Draw the Frobenius
  morphisms for the object $\ul{1}$ in $\cospan\finset$ using both the function
  and wiring depictions as in \cref{ex.cospan_finset}.
\end{exercise}

\begin{exercise} %
\label{exc.frob_cospan2}
  Using your knowledge of colimits, show that the maps defined in
  \cref{ex.cospan_hypergraph} do indeed obey the special law (see
  \cref{def.spec_comm_frob_mon}).
%
\end{exercise}


\begin{example} %
\label{ex.frob_corel}%
\index{corelations!hypergraph category of|see {hypergraph category, of corelations}}%
\index{hypergraph category!of corelations}
  Recall the monoidal category $(\Cat{Corel},\varnothing, \sqcup)$ from \cref{ex.corel}; its objects are finite sets and its morphisms are
  corelations. Given a finite set $X$, define the corelation $\mu_X\colon
  X\sqcup X \to X$ such that two elements of $X \sqcup X \sqcup X$ are
  equivalent if and only if they come from the same underlying element of $X$.
  Define $\delta_X\colon X \to X\sqcup X$ in the same way, and define
  $\eta_X\colon \varnothing \to X$ and $\epsilon_X\colon X \to \varnothing$ such
  that no two elements of $X = \varnothing \sqcup X = X \sqcup \varnothing$ are
  equivalent.

  These maps define a special commutative Frobenius monoid
  $(X,\mu_X,\eta_X,\delta_X,\epsilon_X)$, and in fact give $\Cat{Corel}$
  the structure of a hypergraph category.
\end{example}

\begin{example} %
\index{linear relations!hypergraph category of|see {hypergraph category, of
linear relations}}\index{hypergraph category!of linear relations}
The prop of linear relations, which we briefly mentioned in
\cref{exc.linear_relations}, is a hypergraph category. In fact, it is a hypergraph
category in two ways, by choosing either the black `copy' and `discard'
generators or the white `add' and `zero' generators as the Frobenius maps.
\end{example}

We can generalize the construction we gave in \cref{thm.rel_is_ccc}.
\begin{proposition} %
\label{prop.hyp_cat_comp_closed}%
\index{compact closed category!hypergraph
  category as}
  Hypergraph categories are self-dual compact closed categories, if we define the cup and cap
  to be%
\index{dual!self}
  \[
  \begin{aligned}

  \end{aligned}
\]
\end{proposition}
\begin{proof}
The proof is a straightforward application of the Frobenius and unitality
axioms:
\begin{align*}
  \begin{aligned}
    %
  \end{aligned}\tag{unitality}
\end{align*}
\end{proof}

\begin{exercise} %
\label{exc.fill_in_diagram}
Fill in the missing diagram in the proof of \cref{prop.hyp_cat_comp_closed} using the equations from \cref{eqn.ass_un_comm}, their opposites, and \cref{eqn.frob_laws}.
\end{exercise}

\index{hypergraph category|)}

\section{Decorated cospans}%
\label{sec.decorated_cospans}%
\index{cospans!decorated}
The goal of this section is to show how we can construct a hypergraph category whose morphisms are electric circuits. To do this, we first must
introduce the notion of structure-preserving map for symmetric monoidal
categories, a generalization of monoidal monotones known as symmetric
monoidal functors. Then we introduce a general method---that of decorated cospans---for producing hypergraph categories. Doing all this will tie up lots of loose ends: colimits, cospans, circuits, and hypergraph categories. 

\subsection{Symmetric monoidal functors}%
\label{sec.monoidal_cats_full}

\begin{roughDef}%
\index{monoidal functor}%
\label{roughdef.monoidal_functor}
Let $(\cat{C},I_{\cat{C}},\otimes_\cat{C})$ and
$(\cat{D},I_{\cat{D}},\otimes_\cat{D})$ be symmetric monoidal categories. To
specify a \emph{symmetric monoidal functor} $(F,\varphi)$ between them,
\begin{enumerate}[label=(\roman*)]
	\item one specifies a functor $F\colon\cat{C}\to\cat{D}$;
	\item one specifies a morphism $\varphi_I\colon I_\cat{D}\to F(I_\cat{C})$.
	\item for each $c_1,c_2\in\Ob(\cat{C})$, one specifies a morphism 
	\[
	    \varphi_{c_1,c_2}\colon F(c_1)\otimes_\cat{D} F(c_2)\to F(c_1\otimes_\cat{C} c_2),
	  \]
	  natural in $c_1$ and $c_2$.
\end{enumerate}
We call the various maps $\varphi$ \emph{coherence maps}. We require the
coherence maps to obey bookkeeping axioms that ensure they are well behaved with
respect to the symmetric monoidal structures on $\cat{C}$ and
$\cat{D}$.%
\index{coherence} If $\varphi_I$ and $\varphi_{c_1,c_2}$ are isomorphisms for all $c_1,c_2$, we say that $(F,\varphi)$ is \emph{strong}.

\end{roughDef}

\begin{example} %
\label{ex.powset}%
\index{power set}
Consider the power set functor $\powset\colon\smset\to\smset$. It acts on objects by sending a set
$S\in\smset$ to its set of subsets $\powset(S)\coloneqq\{R\ss S\}$. It acts on morphisms by sending a function $f\colon
S\to T$ to the image map $\im_f\colon\powset(S)\to\powset(T)$, which maps $R\ss
S$ to $\{f(r)\mid r\in R\}\ss T$.

Now consider the symmetric monoidal structure $(\{1\},\times)$ on $\smset$ from
\cref{ex.set_as_mon_cat}. To make $\powset$ a symmetric monoidal functor, we
need to specify a function $\varphi_I\colon\{1\}\to\powset(\{1\})$ and for all
sets $S$ and $T$, a functor $\varphi_{S,T}\colon \powset(S)\times\powset(T) \to
\powset(S \times T)$. One possibility is to define $\varphi_I(1)$ to be the maximal subset $\{1\}\ss\{1\}$, and
given subsets $A \ss S$ and $B \ss T$, to define $\varphi_{S,T}(A,B)$ to be the product subset $A
\times B \ss S \times T$. With these definitions, $(\powset,\varphi)$ is a symmetric monoidal functor.
\end{example}

\begin{exercise} %
\label{exc.powset_mon_coherence}
  Check that the maps $\varphi_{S,T}$ defined in \cref{ex.powset} are natural in
  $S$ and $T$. In other words, given $f\colon S\to S'$ and $g\colon T\to T'$, show that the diagram below commutes:
  \[
  \begin{tikzcd}[column sep=large]
  	\powset(S)\times\powset(T)\ar[r,
	"\varphi_{S,T}"]\ar[d,"\im_f\times \im_g"']&
		\powset(S\times T)\ar[d, "\im_{f\times g}"]\\
		\powset(S')\times\powset(T')\ar[r, "\varphi_{S',T'}"']&
		\powset(S'\times T')
  \end{tikzcd}
  \qedhere
  \]
\end{exercise}

\subsection{Decorated cospans} %
\label{sec.deccospans}

Now that we have briefly introduced symmetric monoidal functors, we return to the task at hand: constructing a hypergraph category of circuits. To do so, we introduce the method of decorated cospans.%
\index{electrical circuit|(}

Circuits have lots of internal structure, but they also have some external ports---also called `terminals'---by which to interconnect them with others. Decorated cospans are ways of discussing exactly that: things with external ports and internal structure. 

To see how this works, let us start with the following example circuit:
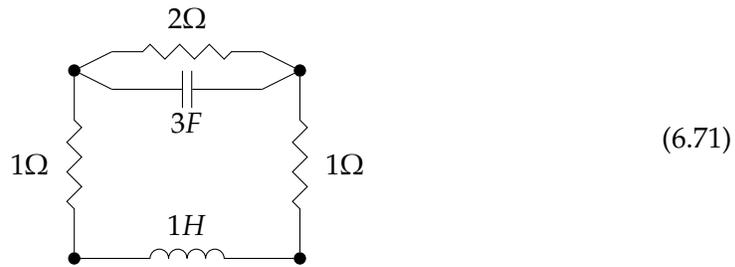
\begin{equation} %
\label{eq.circuit}
\begin{aligned}

    \end{aligned}
\end{equation}
We might formally consider this as a graph on the set of four ports, where each edge
is labeled by a type of circuit component (for example, the top edge would be
labeled as a resistor of resistance $2\Omega$). For this circuit to be a
morphism in some category, i.e.\ in order to allow for interconnection, we must equip the circuit with
some notion of interface. We do this by marking the ports in the interface using
functions from finite sets:
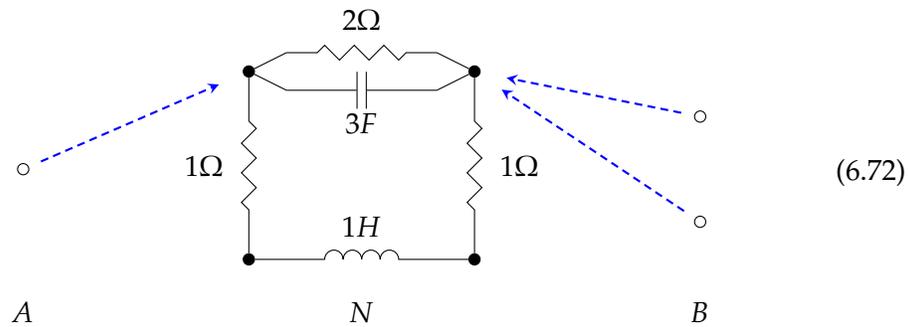
\begin{equation} %
\label{eq.circuitcospan}
    \begin{aligned}
    %
    \end{aligned}
\end{equation}
Let $N$ be the set of nodes of the circuit. Here the finite sets $A$, $B$, and
$N$ are sets consisting of one, two, and four elements respectively, drawn as points, and the
values of the functions $A \to N$ and $B \to N$ are indicated by the grey
arrows. This forms a
cospan in the category of finite sets, for which the apex set $N$
has been \emph{decorated} by our given circuit.%
\index{category!of finite sets}

Suppose given another such decorated cospan with input $B$
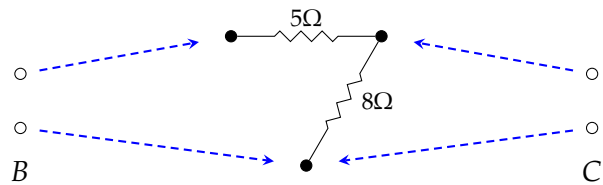
\begin{center}
    %
\end{center}
Since the output of the first equals the input of the second (both are $B$), we can stick them together into a single diagram:
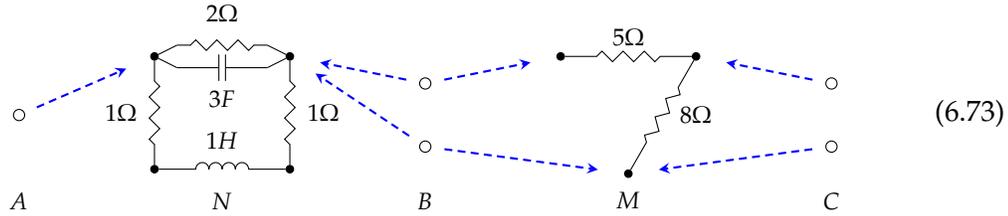
\begin{equation} %
\label{eq.circuitcomposition}
  \begin{aligned}
    %
  
  \end{aligned}
\end{equation}
The composition is given by gluing the circuits along the identifications specified by $B$, resulting in the decorated
cospan
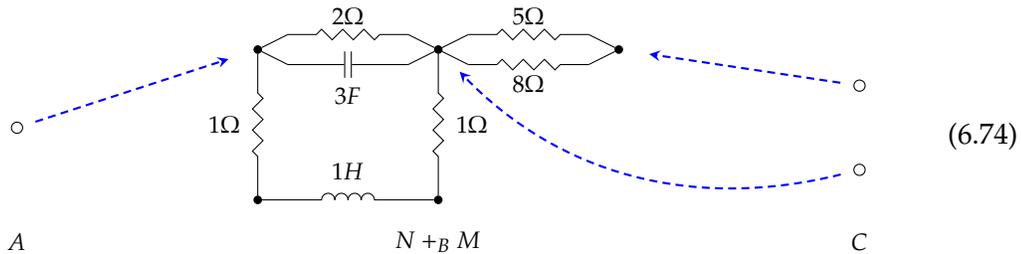
\begin{equation}%
\label{eqn.circuitcomposed}
\begin{aligned}
    %
  \end{aligned}
\end{equation}
We've seen this sort of gluing before when we defined composition of cospans in \cref{def.cospan_sym_mon_cat}. But now there's this whole `decoration' thing; our goal is to formalize it.

\begin{definition}%
\label{def.decorated_cospan}%
\index{cospan!decorated}%
\index{category!having finite colimits}%
\index{monoidal functor}
  Let $\mathcal C$ be a category with finite colimits, and $(F,\varphi) \colon  (\mathcal
  C,+) \longrightarrow (\smset, \times)$
  be a symmetric monoidal functor. An \emph{$F$-decorated cospan} is a pair 
  consisting of a cospan $A \stackrel{i}\rightarrow N \stackrel{o}\leftarrow B$ in
  $\mathcal C$ together with an element $s \in F(N)$.%
  \tablefootnote{Just like in \cref{def.cospan_sym_mon_cat}, we should technically use equivalence classes of cospans. We will elide this point to get the bigger idea across. The interested reader should consult \cref{sec.c6_further_reading}.}
  We call $(F,\varphi)$ the \emph{decoration functor} and $s$ the
  \emph{decoration}.
\end{definition}

The intuition here is to use $\cat{C}=\finset$, and, for each object $N\in\finset$, the functor $F$ assigns the set of all legal decorations on a set $N$ of nodes. When you choose an $F$-decorated cospan, you choose a set $A$ of left-hand external ports, a set $B$ of right-hand external ports, each of which maps to a set $N$ of nodes, and you choose one of the available decorations on $N$ nodes, taken from the set $F(N)$.

So, in our electrical circuit case, the decoration functor $F$ sends a finite set
$N$ to the set of circuit diagrams---graphs whose edges are labeled by resistors, capacitors, etc.---that have $N$ vertices.

Our goal is still to be able to compose such diagrams; so how does that work exactly? Basically one combines the way cospans are composed with the structures defining our decoration functor: namely $F$ and $\varphi$.

Let $(A \xrightarrow{f} N \xleftarrow{g} B,s)$ and $(B \xrightarrow{h} P
\xleftarrow{k} C,t)$ represent decorated cospans. Their
composite is represented by the composite of the cospans $A \xrightarrow{f} N
\xleftarrow{g} B$ and $B \xrightarrow{h} P \xleftarrow{k} C$, paired with the
following element of $F(N+_BP)$:
\begin{equation} %
\label{eqn.compositedecoration}
F([\iota_N,\iota_P])(\varphi_{N,P}(s,t))%
\end{equation}%
\label{page.horrendous}

That's rather compact! We'll unpack it, in a concrete case, in just a second.
But let's record a theorem first.

\begin{theorem}%
\label{thm.hypergraph_from_deccospan}
  Given a category $\mathcal C$ with finite colimits and a symmetric
  monoidal functor $(F,\varphi)\colon  (\mathcal C,+) \longrightarrow (\smset, \times)$,
  there is a hypergraph category $\cospan{F}$ whose objects are the objects of
  $\mathcal C$, and whose morphisms are equivalence classes of $F$-decorated cospans. 
  

  The symmetric monoidal and hypergraph structures are derived from those on
  $\cospan{\cat C}$.
\end{theorem}

\begin{exercise} %
\label{exc.cospan_as_fcospan}
	Suppose you're worried that the notation $\cospan{\cat{C}}$ looks like the notation $\cospan{F}$, even though they're very different. An expert tells you ``they're not so different; one is a special case of the other. Just use the constant functor $F(c)\coloneqq\{*\}$.'' What does the expert mean? 
\end{exercise}

\subsection{Electric circuits} %
\label{subsec.circuits}

In order to work with the above abstractions, we will get a bit more precise
about the circuits example and then have a detailed look at how composition
works in decorated cospan categories.

\paragraph{Let's build some circuits.}

To begin, we'll need to choose which components we want in our circuit. This is
simply a matter of what's in our electrical toolbox. Let's say we're carrying some
lightbulbs, switches, batteries, and resistors of every possible resistance. That is, define a set
\[
  C \coloneq \{\light,\switch,\battery\} \sqcup \{x\Omega \,|\, x \in \rr^+\}.
\]
To be clear, the $\Omega$ are just labels; the above set is isomorphic to
$\{\light,\switch,\battery\}\sqcup \rr^+$. But we write $C$ this way to remind
us that it consists of circuit components. If we wanted, we could also add
inductors, capacitors, and even elements connecting more than two ports, like
transistors, but let's keep things simple for now.

Given our set $C$, a $C$-circuit is just a graph $(V,A,s,t)$, where $s,t \colon
A \to V$ are the source and target functions, together with a function
$\ell\colon A \to C$ labeling each edge with a certain circuit component from
$C$.

For example, we might have the simple case of $V=\{1,2\}$, $A=\{e\}$, $s(e)=1$,
$t(e)=2$---so $e$ is an edge from $1$ to $2$---and $\ell(e)=3\Omega$. This represents
a resistor with resistance $3\Omega$:
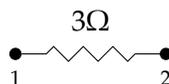
\[
  \begin{tikzpicture}[circuit ee IEC, set resistor graphic=var resistor IEC graphic]
    \node[contact]         (A) at (0,0) {};
    \node[contact]         (B) at (2,0) {};
    \path (A) edge  [resistor] node[label={[label distance=1pt]90:{$3\Omega$}}]
    {} (B);
    \node[below=0 of A, font=\scriptsize] {$1$};
    \node[below=0 of B, font=\scriptsize] {$2$};
  \end{tikzpicture}
\]
Note that in the formalism we have chosen, we have multiple ways to represent any
circuit, as our representations explicitly choose directions for the edges. The
above resistor could also be represented by the `reversed graph', with data $V=\{1,2\}$, $A=\{e\}$,
$s(e)=2$, $t(e)=1$, and $\ell(e)=3F$.

\begin{exercise} %
\label{exc.circuit_tuple}
  Write a tuple $(V,A,s,t,\ell)$ that represents the circuit in
  \cref{eq.circuit}.
\end{exercise}

\paragraph{A decoration functor for circuits.}

We want $C$-circuits to be our decorations, so let's use them to define a
decoration functor as in \cref{def.decorated_cospan}. We'll call the functor
$(\elec,\psi)$. We start by defining the functor part
\[
  \elec \colon (\finset,+) \longrightarrow (\smset, \times)
\]
as follows. On objects, simply send a finite set $V$ to the set of $C$-circuits:
\[
  \elec(V)\coloneqq\{(V,A,s,t,\ell)\,|\,\mbox{ where }s,t \colon A \to V, \ell\colon E \to
C\}.
\]
On morphisms, $\elec$ sends a function $f\colon V \to V'$ to the function
\begin{align*}
  \elec(f)\colon \elec(V) & \longrightarrow \elec(V');\\
  (V,A,s,t,\ell)&\longmapsto \big(V',A, (s\cp f), (t\cp f),\ell\big).
\end{align*}
This defines a functor; let's explore it a bit in an exercise.

\begin{exercise} %
\label{exc.pushforwardcircuit}
  To understand this functor better, let $c\in \elec(\ord{4})$ be the circuit
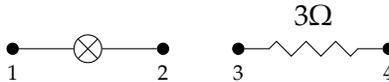
\[
  \begin{tikzpicture}[circuit ee IEC, set resistor graphic=var resistor IEC graphic]
    \node[contact]         (A) at (0,0) {};
    \node[contact]         (B) at (2,0) {};
    \node[contact]         (C) at (3,0) {};
    \node[contact]         (D) at (5,0) {};
    \path (A) edge  [bulb] (B);
    \path (C) edge  [resistor] node[label={[label distance=1pt]90:{$3\Omega$}}]
    {} (D);
    \node[below=0 of A, font=\scriptsize] {$1$};
    \node[below=0 of B, font=\scriptsize] {$2$};
    \node[below=0 of C, font=\scriptsize] {$3$};
    \node[below=0 of D, font=\scriptsize] {$4$};
  \end{tikzpicture}
\]
and let $f\colon \ord{4} \to \ord{3}$ be the function
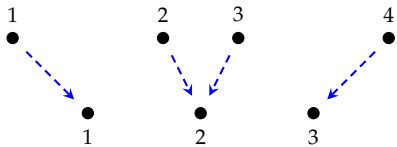
\[
  \begin{tikzpicture}[y=.5cm,short=5pt]
    \node[contact]         (A) at (0,0) {};
    \node[contact]         (B) at (2,0) {};
    \node[contact]         (C) at (3,0) {};
    \node[contact]         (D) at (5,0) {};
    \node[contact]         (1) at (1,-2) {};
    \node[contact]         (2) at (2.5,-2) {};
    \node[contact]         (3) at (4,-2) {};
    \begin{scope}[font=\scriptsize]
      \node[above=0 of A] {$1$};
      \node[above=0 of B] {$2$};
      \node[above=0 of C] {$3$};
      \node[above=0 of D] {$4$};
      \node[below=0 of 1] {$1$};
      \node[below=0 of 2] {$2$};
      \node[below=0 of 3] {$3$};
		\end{scope}
    \begin{scope}[mapsto]
      \draw (A) to (1);
      \draw (B) to (2);
      \draw (C) to (2);
      \draw (D) to (3);
    \end{scope}
  \end{tikzpicture}
\]
Draw a picture of the circuit $\elec(f)(c)$.
\end{exercise}

We're trying to get a decoration functor $(\elec,\psi)$ and so far we have $\elec$. For the coherence maps $\psi_{V,V'}$ for finite sets $V,V'$, we define
\begin{align}
  \psi_{V,V'}\colon \elec(V) \times \elec(V') & \longrightarrow \elec(V+V'); \nonumber \\
  \big((V,A,s,t,\ell),(V',A',s',t',\ell')\big)&\longmapsto(V+V',A+A',s+s',t+t',\copair{\ell,\ell'}).
\label{eqn.Flaxator}
\end{align}
This is simpler than it may look: it takes a circuit on $V$ and a circuit on $V'$, and just
considers them together as a circuit on the disjoint union of vertices $V+V'$.

\begin{exercise} %
\label{exc.parallelcirc}
  Suppose we have circuits  
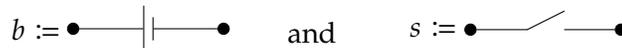
\[
\begin{aligned}
  \begin{tikzpicture}[circuit ee IEC, set make contact graphic=var make contact IEC graphic]
    \node at (-.5,0) {$b\coloneq$};
    \node[contact]         (A) at (0,0) {};
    \node[contact]         (B) at (2,0) {};
    \path (A) edge  [battery] (B);
  \end{tikzpicture}
\end{aligned}
  \qquad
  \mbox{and}
  \qquad
\begin{aligned}
  \begin{tikzpicture}[circuit ee IEC, set resistor graphic=var resistor IEC graphic]
    \node at (-.5,0) {$s\coloneq$};
    \node[contact]         (A) at (0,0) {};
    \node[contact]         (B) at (2,0) {};
    \path (A) edge  [make contact] (B);
  \end{tikzpicture}
\end{aligned}
\]
in $\elec(\ord{2})$. Use the definition of $\psi_{V,V'}$ from
\eqref{eqn.Flaxator} to figure out what $4$-vertex circuit $\psi_{\ord{2},\ord{2}}(b,s) \in
\elec(\ord{2} + \ord{2}) = \elec(\ord{4})$ should be, and draw a picture.
\end{exercise}

\paragraph{Open circuits using decorated cospans.}%
\index{electrical circuit!via cospans}

From the above data, just a monoidal functor $(\elec,\psi)\colon(\finset,+) \to
(\smset, \times)$, we can construct our promised hypergraph category of
circuits!%
\index{monoidal functor}

Our notation for this category is $\cospan\elec$. Following
\cref{thm.hypergraph_from_deccospan}, the objects of this category are the same
as the objects of $\finset$, just finite sets. We'll reprise our notation from
the introduction and \cref{ex.coequalizer}, and draw these finite sets as
collections of white circles $\circ$. For example, we'll represent the object
$\ord{2}$ of $\cospan\elec$ as two white circles:
\[
\begin{tikzpicture}[circuit ee IEC, set resistor graphic=var resistor IEC
graphic, set make contact graphic=var make contact IEC graphic]
\node [draw, inner sep=1.5pt,circle] (a) at (0,0) {};
\node [draw, inner sep=1.5pt,circle] (b) at (.6, 0) {};
\end{tikzpicture}
\]
These white circles mark interface points of an open circuit.

More interesting than the objects, however, are the morphisms in $\cospan\elec$. These are open
circuits. By \cref{thm.hypergraph_from_deccospan}, a morphism $\ord{m} \to
\ord{n}$ is a $\elec$-decorated cospan: that is, cospan $\ord{m} \to \ord{p}
\from \ord{n}$ together with an element $c$ of $\elec(\ord{p})$. As an example,
consider the cospan $\ord{1} \xrightarrow{i_1} \ord{2} \xleftarrow{i_2} \ord{1}$ where
$i_1(1) =1$ and $i_2(1) =2$, equipped with the battery element of $\elec(\ord{2})$ connecting node 1 and node 2. We'll depict this as follows:
\begin{equation} %
\label{eqn.decorated_cospan}
\begin{tikzpicture}[circuit ee IEC, set resistor graphic=var resistor IEC
graphic, set make contact graphic=var make contact IEC graphic, short=4pt]
  \node [contact] (l) at (1,0) {};
  \node [contact] (r) at (3,0) {};
	\node [draw, rounded corners, gray, inner ysep=10pt, inner xsep=10pt, fit=(l) (r)] {};
  \node [draw, inner sep=1.5pt,circle] (a) at (0,0) {};
  \node [draw, inner sep=1.5pt,circle] (b) at (4,0) {};
  \draw (l) to [battery] (r);
  \begin{scope}[mapsto]
    \draw (a) to (l);
    \draw (b) to (r);
  \end{scope}
\end{tikzpicture}
\end{equation}
\begin{exercise} %
\label{exc.namethedecoration}
Morphisms of $\cospan\elec$ are $\elec$-decorated cospans, as defined in
\cref{def.decorated_cospan}. This means \eqref{eqn.decorated_cospan} depicts a
cospan together with a \emph{decoration}, which is some $C$-circuit
$(V,A,s,t,\ell) \in \elec(\ord{2})$. What is it?
\end{exercise}

Let's now see how the hypergraph operations in $\cospan\elec$ can be used to
construct electric circuits. 

\paragraph{Composition in $\cospan\elec$.}
First we'll consider composition. Consider the following decorated cospan from
$\ord{1}$ to $\ord{1}$:
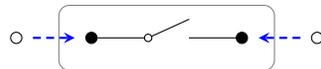
\[

\]
Since this and the circuit in \eqref{eqn.decorated_cospan} are both morphisms
$\ord{1} \to \ord{1}$, we may compose them to get another morphism $\ord{1} \to
\ord{1}$. How do we do this? There are two parts: to get the new cospan, we
simply compose the cospans of our two circuits, and to get the new decoration,
we use the formula $\elec([\iota_N,\iota_P])(\psi_{N,P}(s,t))$ from
\eqref{eqn.compositedecoration}.  Again, this is rather compact! Let's unpack it
together. 

We'll start with the cospans. The cospans we wish to compose are
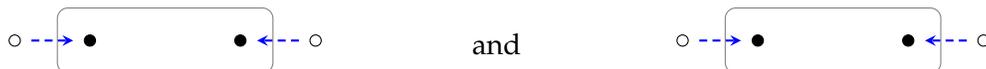
\[
\begin{aligned}
%
\end{aligned}
\]
(We simply ignore the decorations for now.) If we pushout over the common set
$\ord{1} = \{\circ\}$, we obtain the pushout square
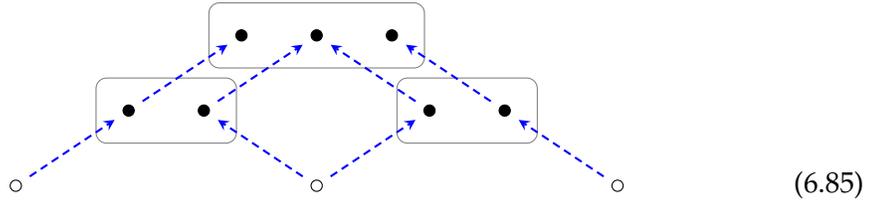
\begin{equation} %
\label{eqn.pushoutexample}
  %
\end{equation}
This means the composite cospan is 
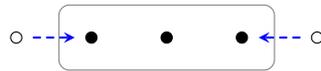
\[
%
\]

In the meantime, we already had you start us off unpacking the formula for the
new decoration. You told us what the map
$\psi_{\ord{2},\ord{2}}$ does in \cref{exc.parallelcirc}. It takes the
two decorations, both circuits in $\elec(\ord{2})$, and turns them into the single, disjoint
circuit
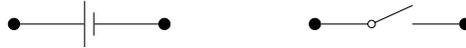
\[
%
\]
in $\elec(\ord{4})$. So this is what the $\psi_{N,P}(s,t)$ part means. What does
the $[\iota_N,\iota_P]$ mean? Recall this is the copairing of the pushout maps,
as described in \cref{ex.setcoproduct,ex.pushouts}. In our case, the relevant
pushout square is given by \eqref{eqn.pushoutexample}, and $[\iota_N,\iota_P]$
is in fact the function $f$ from \cref{exc.pushforwardcircuit}! This means the
decoration on the composite cospan is
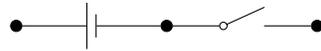
\[
%
\]
Putting this all together, the composite circuit is
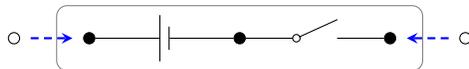
\[
%
\]

\begin{exercise} %
\label{exc.composefcospans}
Refer back to the example at the beginning of \cref{sec.deccospans}. In
particular, consider the composition of circuits in
\cref{eq.circuitcomposition}. Express the two circuits in this diagram as
morphisms in $\cospan\elec$, and compute their composite. Does it match the
picture in \cref{eqn.circuitcomposed}?
\end{exercise}

%

\paragraph{Monoidal products in $\cospan\elec$.}
Monoidal products in $\cospan\elec$ are much simpler than composition. On objects, we again just
work as in $\finset$: we take the disjoint union of finite sets. Morphisms again
have a cospan, and a decoration. For cospans, we again just work in
$\cospan{\finset}$: given two cospans $A \to M \from B$ and $C \to N \from D$,
we take their coproduct cospan $A+C \to M+N \from B+D$. And for decorations,
we use the map $\psi_{M,N}: \elec(M) \times \elec(N) \to \elec(M+N)$. So, for
example, suppose we want to take the monoidal product of the open circuits
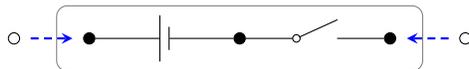
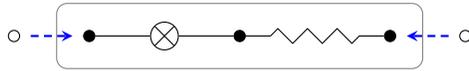
\[

\]
The result is given by stacking them. In other words, their monoidal product is:
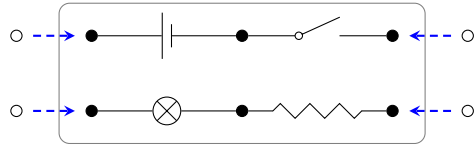
\begin{equation} %
\label{eqn.circuit}
\begin{aligned}
%
\end{aligned}
\end{equation}
Easy, right?%
\index{monoidal product!as stacking}

We leave you to do two compositions of your own.

\begin{exercise} %
\label{exc.buildcircuit}
Write $x$ for the open circuit in \eqref{eqn.circuit}. Also define cospans $\eta\colon 0\to 2$ and $\eta\colon 2\to 0$ as follows: 
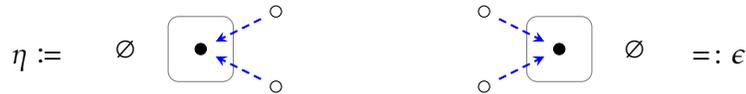
\[
\eta \coloneqq
\quad
\begin{aligned}
%
\end{aligned}
\quad
=\colon\epsilon
\]
where each of these are decorated by the empty circuit
$(\ord{1},\varnothing,!,!,!) \in \elec(\ord{1})$.%
\tablefootnote{As usual $!$ denotes the unique function, in this case from the empty set to the relevant codomain.}

Compute the composite $\eta \cp x \cp \epsilon$ in $\cospan\elec$. This is a
morphism $\ord{0} \to \ord{0}$; we call such things \emph{closed circuits}.%
\index{electrical circuit!closed}
\end{exercise}

%
%

%
\index{electrical circuit|)}
\section{Operads and their algebras}%
\index{operad|(}
In \cref{thm.hypergraph_from_deccospan} we described how decorating cospans
builds a hypergraph category from a symmetric monoidal functor. We then explored
how that works in the case that the decoration functor is somehow ``all circuit
graphs on a set of nodes''.

In this book, we have devoted a great deal of attention to different sorts of
compositional theories, from monoidal preorders to compact closed categories to
hypergraph categories. Yet for an application you someday have in mind, it may
be the case that none of these theories suffice. You need a different structure,
customized to a particular situation. For example in \cite{Vagner.Spivak.Lerman:2015a} the
authors wanted to compose continuous dynamical systems with control-theoretic properties
and realized that in order for feedback to make sense, the wiring diagrams could
not involve what they called `passing wires'.%
\index{dynamical system!continuous}

So to close our discussion of compositional structures, we want to quickly
sketch something we can use as a sort of meta-compositional structure, known as
an operad. We saw in \cref{subsec.circuits} that we can build electric circuits
from a symmetric monoidal functor $\finset\to\smset$. Similarly we'll see that
we can build examples of new algebraic structures from operad functors
$\cat{O}\to\smset$.

\subsection{Operads design wiring diagrams}%
\index{operad!of wiring diagrams}
Understanding that circuits are morphisms in a hypergraph category is useful: it
means we can bring the machinery of category theory to bear on understanding
electrical circuits. For example, we can build functors that express the compositionality
of circuit semantics, i.e.\ how to derive the functionality of the whole from
the functionality and interaction pattern of the parts. Or we can use the
category-theoretic foundation to relate circuits to other sorts of network
systems, such as signal flow graphs.  Finally, the basic coherence theorems for
monoidal categories and compact closed categories tell us that wiring diagrams
give sound and complete reasoning in these settings.

However, one perhaps unsatisfying result is that the hypergraph category
introduces artifacts like the domain and codomain of a circuit, which are not
inherent to the structure of circuits or their composition. Circuits just have a
single boundary interface, not `domains' and `codomains'. This is not to say the above
model is not useful: in many applications, a vector space does not have a
preferred basis, but it is often useful to pick one so that we may use matrices
(or signal flow graphs!). But it would be worthwhile to have a
category-theoretic model that more directly represents the compositional structure
of circuits. In general, we want the category-theoretic model to fit our desired
application like a glove. Let us quickly sketch how this can be done.%
\index{vector space}

Let's return to wiring diagrams for a second. We saw that wiring diagrams for
hypergraph categories basically look like this:
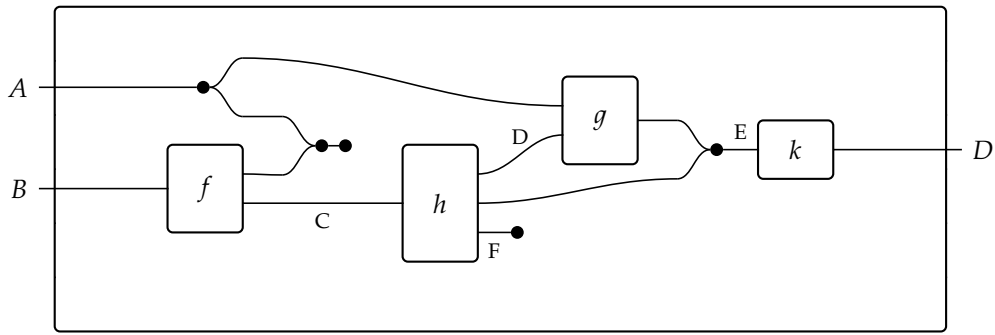
\begin{equation}%
\label{eqn.SMC_WD_rand147}

\end{equation}
Note that if you had a box with $A$ and $B$ on the left and $D$ on the right,
you could plug the above diagram right inside it, and get a new open circuit.
This is the basic move of operads. 

But before we explain this, let's get where we said we wanted to go: to a model where there aren't ports on the left and ports on the right, there are just ports. We want a more succinct model of composition for
circuit diagrams; something that looks more like this:
\begin{equation}%
\label{eqn.operad_WD_rand248}
%
\end{equation}
Do you see how diagrams \cref{eqn.SMC_WD_rand147} and \cref{eqn.operad_WD_rand248} are actually exactly the same in terms of interconnection pattern? The only difference is that the latter does not have left/right distinction: we have lost exactly what we wanted to lose.

The cost is that the `boxes' $f,g,h,k$ in \cref{eqn.operad_WD_rand248} no longer have a left/right distinction; they're just circles now. That wouldn't be bad except that it means they can no longer represent morphisms in a category---like they used to above, in \cref{eqn.SMC_WD_rand147}---because morphisms in a category by definition have a domain and codomain. Our new circles have no such distinction. So now we need a whole new way to think about `boxes' categorically: if they're no longer morphisms in a category, what are they? The answer is found in the theory of operads.

In understanding operads, we will find we need to navigate one of the level
shifts that we first discussed in \cref{ssec.level_shift}. Notice that for
decorated cospans, we define a hypergraph \emph{category} using a symmetric
monoidal \emph{functor}.%
\index{monoidal functor} This is reminiscent of our brief discussion of
algebraic theories in \cref{ssec.alg_theories}, where we defined something called the theory of monoids as a prop $\cat{M}$, and define monoids using functors $\cat{M}\to\smset$; see \cref{rem.theory_of_monoids}. In the same way, we can view the category $\cospan\finset$ as some
sort of `theory of hypergraph categories', and so define hypergraph categories
as functors $\cospan\finset\to\smset$.

So that's the idea. An operad $\cat{O}$ will define a theory or grammar of
composition, and operad functors $\cat{O}\to\smset$, known as
\emph{$\cat{O}$-algebras}, will describe particular applications that obey that
grammar.

\begin{roughDef}%
\index{operad}
To specify an \emph{operad} $\cat O$,
\begin{enumerate}[label=(\roman*)]
\item one specifies a collection $T$, whose elements are called \emph{types};
\item for each tuple $(t_1,\dots,t_n,t)$ of types, one specifies a set $\cat
O(t_1,\dots,t_n;t)$, whose elements are called \emph{operations of arity
$(t_1,\dots,t_n;t)$};%
\index{operation|see {operad}}%
\index{operad!operation in}
\item for each pair of tuples $(s_1,\dots,s_m,t_i)$ and $(t_1,\dots,t_n,t)$, one specifies a
function 
\[
\circ_i\colon \cat O(s_1,\dots,s_m;t_i) \times \cat O(t_1,\dots,t_n;t)
\to \cat O(t_1,\dots,t_{i-1},s_1,\dots,s_m,t_{i+1},\dots, t_n;t);
\]
called \emph{substitution}; and
\item for each type $t$, one specifies an operation $\id_t \in O(t;t)$ called the
\emph{identity operation}.
\end{enumerate}
These must obey generalized identity and associativity laws.%
\tablefootnote{Often what we call types are called objects or colors, what we call operations are called morphisms, what we call substitution is called composition, and what we call operads are called multicategories. A formal definition can be found in \cite{Leinster:2004a}.}
\end{roughDef}%
\index{associativity!of composition in an operad}

Let's ignore types for a moment and think about what this structure models. The intuition is that an operad consists of, for each $n$, a set of operations
of arity $n$---that is, all the operations that accept $n$ arguments. If we take an
operation $f$ of arity $m$, and plug the output into the $i$th argument of an
operation $g$ of arity $n$, we should get an operation of arity $m+n-1$: we have
$m$ arguments to fill in $m$, and the remaining $n-1$ arguments to fill in $g$.
Which operation of arity $m+n-1$ do we get? This is described by the substitution function
$\circ_i$, which says we obtain the operation $f\circ_i g \in \cat O(m+n-1)$.
The coherence conditions say that these functions $\circ_i$ capture the following
intuitive picture:
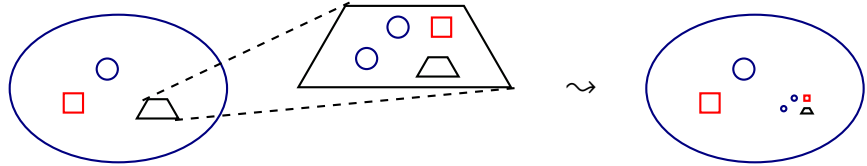
\[
\begin{tikzpicture}[unoriented WD, pack inside color=white]
	\node[pack] (A) {};
	\node[draw, red, below left=of A] (B) {};
	\node[trapezium, green!.9!black, minimum width=.5cm, draw, below right=1 and 1.5 of A] (C) {};
	\node[outer pack, fit=(A) (B) (C)] (outer) {};
	\node[pack, above right=1.5 and 9 of C] (C1) {};
	\node[pack, above right= of C1] (C2) {};
	\node[trapezium, green!.9!black, draw, below right=1 and 1 of C2] (C3) {};
	\node[red, draw, right=1 and 1 of C2] (C4) {};
	\node[trapezium, green!.9!black, draw, fit=(C1) (C2) (C3) (C4)] (outerC) {};
	\draw[dashed] (C.top left corner) -- (outerC.top left corner);
	\draw[dashed] (C.bottom right corner) -- (outerC.bottom right corner);
	\node[right=2 of outerC.bottom right corner] {$\leadsto$};
	\node[pack, right=29 of A] (XA) {};
	\node[draw, red, below left=of XA] (XB) {};
	\node[trapezium, minimum width=.5cm, below right=1 and 1.5 of XA] (XC) {};
	\node[pack, pack size=2pt] at (XC.north) (XD) {};
	\node[pack, pack size=2pt, below left=2pt and 2pt of XD] (XE) {};
	\node[red, draw, minimum width=0, inner sep=1pt, right=2pt of XD] (XF) {};
	\node[trapezium, green!.9!black, draw, inner sep=1pt, below=2pt of XF] (XG) {};
	\node[outer pack, fit=(XA) (XB) (XC)] (Xouter) {};
\end{tikzpicture}
\]

The types then allow us to specify the, well, types of the
arguments---inputs---that each function takes. So making tea is a 2-ary
operation, an operation with arity 2, because it takes in two things. To make tea
you need some warm water, and you need some tea leaves.


\begin{example}%
\index{context free grammar}
  Context-free grammars are to operads as graphs are to categories.  Let's
  sketch what this means. First, a context-free grammar is a way of describing a
  particular set of `syntactic categories' that can be formed from a set of
  symbols. For example, in English we have syntactic categories like nouns,
  determiners, adjectives, verbs, noun phrases, prepositional phrases,
  sentences, etc. The symbols are words, e.g.\ cat, dog, the, chases.
  
  To define a context-free grammar on some alphabet, one specifies some
  \emph{production rules}, which say how to form an entity in some syntactic
  category from a bunch of entities in other syntactic categories. For example,
  we can form a noun phrase from a determiner (the), an adjective (happy), and a
  noun (boy). Context free grammars are important in both linguistics and
  computer science. In the former, they're a basic way to talk about the
  structure of sentences in natural languages. In the latter, they're crucial
  when designing parsers for programming languages.

  So just like graphs present free categories, context-free grammars present free operads. This idea was first noticed in \cite{Hermida.Makkai.Power:1998a}.
\end{example}

\subsection{Operads from symmetric monoidal categories}
\index{operad!from monoidal category}
\label{subsec.mon_cat_operads}
We will see in \cref{def.underlying_operad} that a large class of operads come from symmetric monoidal categories. Before we explain this, we give a couple of examples. Perhaps the most important operad is that of $\smset$. 

\begin{example}%
\index{operad!of sets}
The operad $\oprdset$ of sets has 
\begin{enumerate}[label=(\roman*)]
\item Sets $X$ as types.
\item Functions $X_1\times \dots \times X_n \to Y$ as operations of arity
$(X_1,\dots, X_n;Y)$.
\item Substitution defined by 
\begin{multline*}
(g\circ_if)(x_1,\dots,x_{i-1},w_1,\dots,w_m,x_{i+1},\dots,x_n)\\
 =
g\big(x_1,\dots,x_{i-1},f(w_1,\dots,w_m),x_{i+1},\dots,x_n\big)
\end{multline*}
where $f\in \oprdset(W_1,\dots,W_m;X_i)$, $g \in \oprdset(X_1,\dots,X_n;Y)$, and
hence $g\circ_if$ is a function 
\[
(g\circ_if)\colon
X_1 \times \dots \times X_{i-1}\times W_1\times \dots \times
W_m\times X_{i+1}\times \dots \times X_n \longrightarrow Y 
\]
\item Identities $\id_X\in \oprdset(X;X)$ are given by the identity function
$\id_X\colon X \to X$.
\qedhere
\end{enumerate}
\end{example}

Next we give an example that reminds us what all this operad stuff was for: wiring diagrams.

\begin{example}%
\index{operad!of cospans}
The operad $\oprdcospan$ of finite-set cospans has
\begin{enumerate}[label=(\roman*)]
\item Natural numbers $a \in \nn$ as types.
\item Cospans $\ord{a_1} + \dots + \ord{a_n} \to \ord{p} \leftarrow
\ord{b}$ of finite sets as operations of arity $(a_1,\dots, a_n;b)$.
\item Substitution defined by pushout.
\item Identities $\id_a\in \oprdset(a;a)$ just given by the identity cospan
$\ord{a} \xrightarrow{\id_{\ord{a}}}\ord{a} \xleftarrow{\id_{\ord{a}}} \ord{a}$.
\end{enumerate}
This is the operadic analogue of the monoidal category $(\cospan\finset, 0, +)$.

We can depict operations in this operad using diagrams like we drew above. For example, here's a picture of an operation:
\begin{equation}%
\label{eqn.unoriented_WD_rand327}
\begin{tikzpicture}[unoriented WD, spacing=16pt, every label/.style={font=\small}]
	\node[pack] (f1) {$f$};
	\node[pack, below right=1.7 and 1 of f1] (f3) {$h$};
	\node[pack, above right=1.7 and 1 of f3] (f2) {$g$};
	\node[pack, below right=1.7 and 1 of f2] (f4) {$k$};
	\node[outer pack, fit=(f1) (f4)] (outer) {};
	\node[link] at ($(f1)!-.4!(outer)$) (t) {};
	\node[link] at ($(f1)!.5!(f3)$) (u) {};
	\node[link] at (f3 |-f1) (v){};
	\node[link] at ($(f3)!.5!(f2)$) (w) {};
	\node[link] at ($(w)!.3!(f4)$) (x) {};
	\node[link, below=.5 of f3] (y) {};
	\node[link] at ($(f4)!-.4!(outer)$) (z) {};
	\draw[shorten >=-2pt, shorten >=-2pt] (f1) -- (outer);
	\draw (v) -- (outer.north-|v);
	\draw[shorten >=-2pt, shorten >=-2pt] (f4) -- (outer);
	\draw (f1) -- (v) -- (f2) -- (f3) -- (y);
	\draw (f1) -- (f3) to[bend right=10] (x);
	\draw (x) to [bend right=10] (f2);
	\draw (x) -- (f4);
\end{tikzpicture}
\end{equation}
This is an operation of arity $(\ord{3},\ord{3},\ord{4},\ord{2};\;\ord{3})$. Why? The circles
marked $f$ and $g$ have 3 ports, $h$ has 4 ports, $k$ has 2 ports, and the
outer circle has 3 ports: 3, 3, 4, 2; 3. 

So how exactly is \cref{eqn.unoriented_WD_rand327} a morphism in this operad? Well a morphism of this arity is, by (ii), a cospan $\ord{3}+\ord{3}+\ord{4}+\ord{2}\To{a}\ord{p}\From{b}\ord{3}$. In the diagram above, the apex $\ord{p}$ is the set $\ord{7}$, because there are 7
nodes $\bullet$ in the diagram. The function $a$ sends each port on one of the small circles
to the node it connects to, and the function $b$ sends each port of the outer circle to
the node it connects to.

We are able to depict each operation in the operad $\oprdcospan$ as a wiring
diagram. It is often helpful to think of operads as describing a wiring diagram grammar. The substitution operation of the operad signifies inserting one wiring diagram
into a circle or box in another wiring diagram.
\end{example}

\begin{exercise}%
\label{exc.wd_drawing_practice}
\begin{enumerate}
	\item Consider the following cospan $f\in\oprdcospan(2,2; 2)$:
	\[

	\]
	Draw it as a wiring diagram with two inner circles, each with two ports, and one outer circle with two ports.
	\item Draw the wiring diagram corresponding to the following cospan $g\in\oprdcospan(2,2,2;0)$:
	\[
	%
	\]	
	\item Compute the cospan $g\circ_1 f$. What is its arity?
	\item Draw the cospan $g\circ_1 f$. Do you see it as substitution?
\qedhere
\end{enumerate}
\end{exercise}

We can turn any symmetric monoidal category into an operad in a way that
generalizes the above two examples.

\begin{definition}%
\label{def.underlying_operad}
For any symmetric monoidal category $(\cat{C},I,\otimes)$, there is an operad $\cat{O}_\cat{C}$, called the \emph{operad underlying $\cat{C}$}, defined as having:
\begin{enumerate}[label=(\roman*)]
	\item $\Ob(\cat{C})$ as types.
	\item morphisms $C_1\otimes\cdots\otimes C_n\to D$ in $\cat{C}$ as the operations of arity $(C_1,\ldots,C_n;D)$.
	\item substitution is defined by
	\[(f\circ_ig)\coloneqq f\circ(\id,\ldots,\id,g,\id,\ldots,\id)\]
	\item identities $\id_a\in\cat{O}_\cat{C}(a;a)$ defined by $\id_a$.
\end{enumerate}
\end{definition}
We can also turn any monoidal functor into what's called an operad functor.

\subsection{The operad for hypergraph props}%
\index{hypergraph prop!operad for}

An operad functor takes the types of one operad to the types of another, and
then the operations of the first to the operations of the second in a way that
respects this. 
\begin{roughDef}%
\index{functor!operad}
Suppose given two operads $\cat O$ and $\cat P$ with type collections $T$ and $U$ respectively. To specify an
operad functor $F\colon \cat O \to \cat P$,
\begin{enumerate}[label=(\roman*)]
\item one specifies a function $f\colon T \to U$.
\item For all arities $(t_1,\dots,t_n;t)$ in $\cat O$, one specifies a function 
\[
F\colon \cat O(t_1,\dots,t_n;t) \to \cat P(f(t_1),\dots,f(t_n);f(t))
\]
\end{enumerate}
such that composition and identities are preserved.
\end{roughDef}

Just as set-valued functors $\cat{C}\to\smset$ from any category $\cat{C}$ are of particular
interest---we saw them as database instances in \cref{chap.databases}---so to are $\smset$-valued functors $\cat{O}\to\smset$ from any operad $\cat{O}$.%
\index{functor!set@$\smset$-valued}

\begin{definition}%
\index{operad!algebra of }
An \emph{algebra} for an operad $\cat O$ is an operad functor
$F\colon \cat O \to \oprdset$.
\end{definition}

We can think of functors $\cat{O}\to\oprdset$ as defining a set of possible ways to fill the
boxes in a wiring diagram. Indeed, each box in a wiring diagram represents a
type $t$ of the given operad $\cat O$ and an algebra $F\colon \cat O \to
\oprdset$ will take a type $t$ and return a set $F(t)$ of fillers for box $t$.  Moreover, given an
operation (i.e., a wiring diagram) $f\in \cat O(t_1,\dots,t_n;t)$, we get a
function $F(f)$ that takes an element of each set $F(t_i)$, and returns an element
of $F(t)$. For example, it takes $n$ circuits with interface
$t_1,\dots,t_n$ respectively, and returns a circuit with boundary $t$.

\begin{example}
For electric circuits, the types are again finite sets, $T=\Ob(\finset)$, where each finite set $t\in T$ corresponds to a cell with $t$ ports. Just as before, we have a set $\elec(t)$ of fillers, namely the set of electric circuits with that $t$-marked terminals. As an operad algebra, $\elec\colon\oprdcospan\to\smset$ transforms wiring diagrams like this one
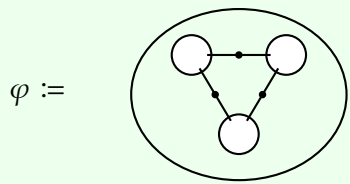
\[

\]	
into formulas that build a new circuit from a bunch of existing ones. In the above-drawn case, we would get a morphism $\elec(\varphi)\in\smset(\elec(2),\elec(2),\elec(2);\elec(0))$, i.e.\ a function
\[\elec(\varphi)\colon\elec(2)\times\elec(2)\times\elec(2)\to\elec(0).\]
We could apply this function to the three elements of $\elec(2)$ shown here
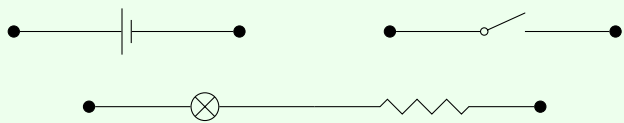
\[
%
\]
and the result would be the closed circuit from the beginning of the chapter:
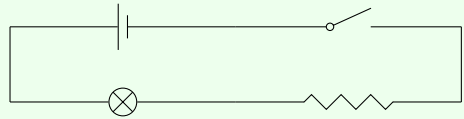
\[
%
\]
\end{example}

This is reminiscent of the story for decorated cospans: gluing fillers together to form hypergraph categories. An advantage of the decorated cospan construction is that one obtains an explicit category (where morphisms have domains and codomains and can hence be composed associatively), equipped with Frobenius structures that allow us to get around the strictures of domains and codomains. The operad perspective has other advantages. First, whereas decorated cospans can produce only some hypergraph categories, $\oprdcospan$-algebras can produce any hypergraph category.

\begin{proposition}%
\index{theory!of hypergraph props}%
\index{hypergraph prop!theory of}
There is an equivalence between $\oprdcospan$-algebras and hypergraph props.
\end{proposition}


Another advantage of using operads is that one can vary the operad itself, from $\oprdcospan$ to something similar (like the operad of `cobordisms'), and get slightly different compositionality rules. 

In fact, operads---with the additional complexity in their definition---can be customized
even more than all compositional structures defined so far. For example, we can
define operads of wiring diagrams where the wiring diagrams must obey precise
conditions far more specific than the constraints of a category, such as
requiring that the diagram itself has no wires that pass straight through it.
In fact, operads are strong enough to define themselves: roughly speaking, there is an operad for operads: the category of operads is equivalent to the category of algebras for a certain operad \cite[Example 2.2.23]{Leinster:2004a}.
While operads can, of course, be generalized again, they conclude our
march through an informal hierarchy of compositional structures, from preorders to
categories to monoidal categories to operads.%
\index{category!of operads}%
\index{category!of algebras for an operad}

\index{operad|)}

\section{Summary and further reading}%
\label{sec.c6_further_reading}

This chapter began with a detailed exposition of colimits in the category of
sets; as we saw, these colimits describe ways of joining or interconnecting sets. Our
second way of talking about interconnection was the use of Frobenius monoids and
hypergraph categories; we saw these two themes come together in the idea of a
decorated cospans. The decorated cospan construction uses a certain type of
structured functor to construct a certain type of structured category. More
generally, we might be interested in other types of structured category, or
other compositional structure. To address this, we briefly saw how these ideas
fit into the theory of operads.

Colimits are a fundamental concept in category theory. For more on colimits,
one might refer to any of the introductory category theory
textbooks we mentioned in \cref{sec.ch2_further_reading}.

Special commutative Frobenius monoids and hypergraph categories were first
defined, under the names `separable commutative Frobenius algebra' and `well-supported compact closed category', by Carboni and Walters
\cite{Carboni:1987a,Carboni:1991a}. The use of
decorated cospans to construct them is detailed in
\cite{fong2015decorated,fong2017decorated,Fong:2016a}. The application to
networks of passive linear systems, such as certain electrical circuits, is
discussed in \cite{baez2015compositional}, while further applications, such as to Markov
processes and chemistry can be found in
\cite{baez2016compositional,baez2017compositional}. For another interesting application of hypergraph
categories, we recommend the pixel array method for approximating solutions to
nonlinear equations \cite{Spivak.Dobson.Kumari.Wu:2016a}.%
\index{chemistry} The story of this chapter is fleshed out in a couple of
recent, more technical papers \cite{fong2018hypergraph,Fong.Sarazola:2018}.

Operads were introduced by May to describe compositional structures arising in
algebraic topology \cite{May:1972a}; Leinster has written a great book on the subject
\cite{Leinster:2004a}. More recently, with collaborators author-David has discussed
using operads in applied mathematics, to model composition of structures in
logic, databases, and dynamical systems
\cite{Rupel.Spivak:2013a,Spivak:2013b,Vagner.Spivak.Lerman:2015a}.

\setcounter{chapter}{6}
\chapter[Logic of behavior: Sheaves, toposes, languages]{Logic of behavior:\\Sheaves, toposes, and internal languages}%
\label{chap.temporal_topos}



\section{How can we prove our machine is safe?}%
\label{sec.safe_machines}

Imagine you are trying to design a system of interacting components. You
wouldn't be doing this if you didn't have a goal in mind: you want the system to
do something, to behave in a certain way. In other words, you want to restrict
its possibilities to a smaller set: you want the car to remain on the road, you
want the temperature to remain in a particular range, you want the bridge to be safe for
trucks to pass. Out of all the possibilities, your system should only permit some.

Since your system is made of components that interact in specified ways, the possible behavior of the whole---in any environment---is determined by the possible behaviors of each of its components in their local environments, together with the precise way in which they interact.%
\footnote{
The well-known concept of emergence is not about possibilities, it is about
prediction. Predicting the behavior of a system given predictions of its
components is notoriously hard. The behavior of a double pendulum is
chaotic---meaning extremely sensitive to initial conditions---whereas those of
the two component pendulums are not. However, the set of possibilities for the
double pendulum is completely understood: it is the set of possible angular
positions and velocities of both arms. When we speak of a machine's properties
in this chapter, we always mean the guarantees on its behaviors, not the
probabilities involved, though the latter would certainly be an interesting
thing to contemplate.%
\index{prediction vs.\ possibility}
}
In this chapter, we will discuss a logic wherein one can describe general types
of behavior that occur over time, and prove properties of a larger-scale system
from the properties and interaction patterns of its components.%
\index{behavior|(}%
\index{interaction}

\label{subsec.systems_and_components}

For example, suppose we want an autonomous vehicle to maintain a distance of some $\const{safe}\in\rr$ from other objects. To do so, several
components must interact: a sensor that approximates the real distance by an
internal variable $S'$, a controller that uses $S'$ to decide what action $A$ to
take, and a motor that moves the vehicle with an acceleration based on $A$.
This in turn affects the real distance $S$, so there is a feedback loop.

Consider the following model diagram:
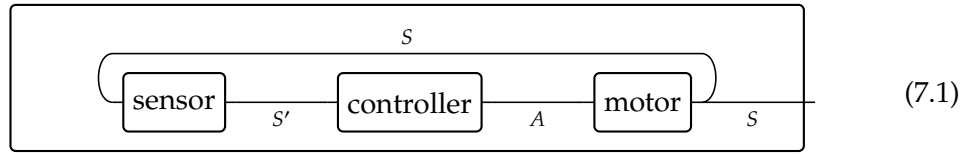
\begin{equation}%
\label{eqn.sens_cont_mot}
\begin{tikzpicture}[oriented WD, baseline=(controller)]
	\node[bb={1}{1}] (sensor) {sensor};
	\node[bb={1}{1}, right=of sensor] (controller) {controller};
	\node[bb={1}{1}, right=of controller] (motor) {motor};
	\node[bb={0}{0}, fit={(sensor) ($(controller.north)+(0,2.5)$) (motor)}] (system) {};
	\draw (sensor_out1) to node[below=3pt, label] {$S'$} (controller_in1);
	\draw (controller_out1) to node[below=3pt, label] {$A$} (motor_in1);
	\draw (motor_out1) to node[below=3pt, label] {$S$} (system.east|-motor_out1) -- +(.15cm,0);
	\draw let \p1=(motor.north east), \p2=(sensor.north west) in
		(motor_out1) to[in=0] (\x1+\bbportlen,\y1+\bby) to[in=0, out=180] node[above=3pt, label] {$S$} (\x2-\bbportlen,\y1+\bby) to[out=180] (sensor_in1);
\end{tikzpicture}
\end{equation}
In the diagram shown, the distance $S$ is exposed by the exterior interface. This just means we imagine $S$ as being a variable that other components of a larger system may want to interact with. We could have exposed no variables (making it a closed system) or we could have exposed $A$ and/or $S'$ as well.%
\index{interface}

In order for the system to ensure $S\geq\const{safe}$, we need each of the components to ensure a property of its own. But what are these components, `sensor, controller, motor', and what do they do?

One way to think about any of the components is to open it up and see how it is put together; with a detailed study we may be able to say what it will do. For example, just as $S$ was exposed in the diagram above, one could imagine opening up the `sensor' component box in \cref{eqn.sens_cont_mot} and seeing an interaction between subcomponents%
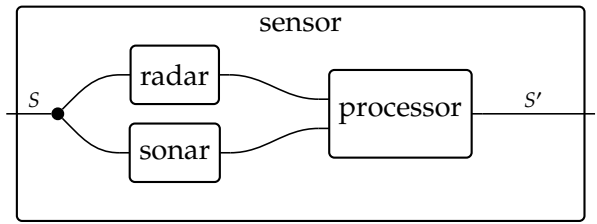
\index{system!component}
\[
\begin{tikzpicture}[oriented WD]
	\node[bb={1}{1}] (radar) {radar};
	\node[bb={1}{1}, below=of radar] (sonar) {sonar};
	\node (halfway) at ($(radar)!.5!(sonar)$) {};
	\node[bb={2}{1}] at ($(halfway)+(2,0)$) (processor) {processor};
	\node[bb={1}{1}, fit={($(radar.north west)+(0,1)$) ($(sonar.south west)+(0,-1)$) (processor)}, bb name=sensor] (sensor) {};
	\node [bdot, right=.4 of sensor_in1] (dot) {};
	\draw (sensor_in1') to node[label, above=2pt] {$S$} (dot);
	\draw (dot) to[out=45] (radar_in1);
	\draw (dot) to[out=-45] (sonar_in1);
	\draw (radar_out1) to (processor_in1);
	\draw (sonar_out1) to (processor_in2);
	\draw (processor_out1) to node[label, above=2pt] {$S'$} (sensor_out1');
\end{tikzpicture}
\]
This ability to zoom in and see a single unit as being composed of others is
important for design. But at the end of the day, you eventually need to stop
diving down and simply use the properties of the components in front of you to
prove properties of the composed system. Have no fear: everything we do in this chapter will be fully compositional, i.e.\ compatible with opening up lower-level subsystems and using the fractal-like nature of composition. However at a given time, your job is to design the system at a
given level, taking the component properties of lower-level systems as given.%
\index{compositionality}

We will think of each component in terms of the relationship it maintains (through time) between the changing values on its ports. ``Whenever I see a flash, I will increase pressure on the button'': this is a relationship I maintain through time between the changing values on my eye port and my finger port. We will make this more precise soon, but fleshing out the situation in \cref{eqn.sens_cont_mot} should help. The sensor maintains a relationship between $S$ and $S'$, e.g.\ that the real distance $S$ and its internal representation $S'$ differ by no more than $5\mathrm{cm}$. The controller maintains a relationship between $S'$ and the action signal $A$, e.g.\ that if at any time $S<\const{safe}$, then within one second it will emit the signal $A=\const{go}$. The motor maintains a relationship between $A$ and $S$, e.g.\ that $A$ dictates the second derivative of $S$ by the formula 
\begin{equation}%
\label{eqn.go_stop}
  \left((A=\const{go})\imp\ddot{S}>1\right)\wedge\left((A=\const{stop})\imp\ddot{S}=0\right).
\end{equation}
If we want to prove properties of the whole interacting system, then the relationships maintained by each component need to be written in a formal logical language, something like what we saw in \cref{eqn.go_stop}. From that basis, we can use standard proof techniques to combine properties of subsystems into properties of the whole. This is our objective in the present chapter.%
\index{system!property}

We have said how component systems, wired together in some arrangement, create larger-scale systems. We have also said that, given the wiring arrangement, the behavioral properties of the component systems dictate the behavioral properties of the whole. But what exactly are behavioral properties?

In this chapter, we want to give a formal language and semantics for a very
general notion of behavior. Mathematics is itself a formal language; the usual
style of mathematical modeling is to use any piece of this vast language at any
time and for any reason. One uses ``human understanding'' to ensure that the
different models are fitting together in an appropriate way when
different systems are combined.%
\index{language}
The present work differs in that we want to find a
domain-specific language for modeling behavior, any sort of behavior, and
nothing but behavior. Unlike in the wide world of math, we want a setting where
the only things that can be discussed are behaviors.

For this, we will construct what is called a \emph{topos}, which is a special kind of category. Our topos, let's call it $\Cat{BT}$, will have behavior types---roughly speaking, sets whose elements can change through time---as its objects. An amazing fact about toposes%
\footnote{The plural of topos is often written \emph{topoi}, rather than toposes. This seems a bit fancy for our taste. As Johnstone suggests in \cite{Johnstone:1977a}, we might ask those who ``persist in talking about topoi whether, when they go out for a ramble on a cold day, they carry supplies of hot tea with them in thermoi.'' It's all in good fun; either term is perfectly reasonable and well-accepted.%
\index{topos}
}
is that they come with an \emph{internal language} that looks very much like the
usual formal language of mathematics itself. Thus one can define graphs, groups,
topological spaces, etc.\ in any topos. But in $\Cat{BT}$, what we call graphs
will actually be graphs that change through time, and similarly what we call
groups and spaces will actually be groups and spaces that change through time.
\index{language!internal}

The topos $\Cat{BT}$ not only has an internal language, but also a
mathematical semantics using the notion of sheaves. Technically, a sheaf is a
certain sort of functor, but one can imagine it as a space of possibilities, varying in a controlled way; in our case it will be a space of possible behaviors varying in a certain notion of time. Every property we prove in our logic of
behavior types will have meaning in this category of sheaves.%
\index{sheaf}%
\index{semantics}%
\index{category!of sheaves}

When discussing systems and components---such as sensors, controllers, motors, etc.---we mentioned behavior types; these will be the objects in the topos $\Cat{BT}$. Every wire in the picture below will stand for a behavior type, and every box $X$ will stand for a behavioral property, a relation that $X$ maintains between the changing values on its ports.
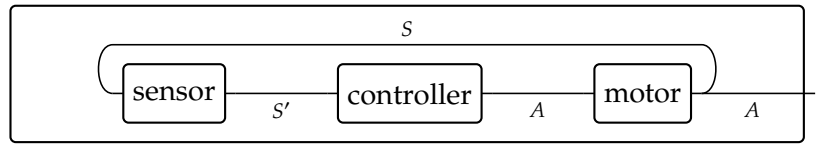
\[
\begin{tikzpicture}[oriented WD]
	\node[bb={1}{1}] (sensor) {sensor};
	\node[bb={1}{1}, right=of sensor] (controller) {controller};
	\node[bb={1}{1}, right=of controller] (motor) {motor};
	\node[bb={0}{0}, fit={(sensor) ($(controller.north)+(0,2)$) (motor)}] (system) {};
	\draw (sensor_out1) to node[below=3pt, label] {$S'$} (controller_in1);
	\draw (controller_out1) to node[below=3pt, label] {$A$} (motor_in1);
	\draw (motor_out1) to node[below=3pt, label] {$A$} (system.east|-motor_out1) -- +(.15cm,0);
	\draw let \p1=(motor.north east), \p2=(sensor.north west) in
		(motor_out1) to[in=0] (\x1+\bbportlen,\y1+\bby) to[in=0, out=180] node[above=3pt, label] {$S$} (\x2-\bbportlen,\y1+\bby) to[out=180] (sensor_in1);
\end{tikzpicture}
\]
For example we could imagine that
\begin{itemize}
	\item $S$ (wire): The behavior of $S$ over a time-interval $[a,b]$ is that of all continuous real-valued functions $[a,b]\to\RR$.
	\item $A$ (wire): The behavior of $A$ over a time-interval $[a,b]$ is all piecewise constant functions, taking values in the finite set such as $\{\const{go}, \const{stop}\}$.
	\item controller (box): the relation $\{(S',A)\mid\cref{eqn.go_stop}\}$,
	i.e.\ all behavioral pairs $(S',A)$ that conform to what we said our controller is supposed to do in \cref{eqn.go_stop}.
\end{itemize}%
\index{behavior!properties of}

\section{The category $\smset$ as an exemplar topos}%
\label{sec.logic_and_set}
\index{topos!of sets}%
\index{category!of sets}

We want to think about a very abstract sort of thing, called a topos, because we will see that behavior types form a topos. To get started, we begin with one of the easiest toposes to think about, namely the topos $\smset$ of sets. In this section we will discuss commonalities between sets and every other topos. We will go into some details about the category of sets, so as to give intuition for other toposes. In particular, we'll pay careful attention to the logic of sets, because we eventually want to understand the logic of behaviors.%
\index{logic}

Indeed, logic and sets are closely related. For example, the logical statement---more
formally known as a predicate---\texttt{likes\_cats} defines a function from the
set $P$ of people to the set $\BB=\{\false,\true\}$ of truth values, where
$\texttt{Brendan} \in P$ maps to $\true$ because he likes cats whereas
$\texttt{Ursula} \in P$ maps to $\false$ because she does not. Alternatively,
\texttt{likes\_cats} also defines a subset of $P$, consisting of exactly the
people that do like cats%
\index{predicate}
\[
\{p\in P\mid \texttt{likes\_cats}(p)\}.
\]
In terms of these subsets, logical operations correspond to set operations, e.g. AND corresponds to intersection: indeed, the set of people for
(mapped to $\true$ by) the predicate \texttt{likes\_cats\_AND\_likes\_dogs} is equal to the
intersection of the set for \texttt{likes\_cats} and the set for \texttt{likes\_dogs}.%
\index{operation!logical}%
\index{intersection}

We saw in \cref{chap.databases} that such operations, which are examples of
database queries, can be described in terms of limits and colimits in $\smset$. Indeed, the
category $\smset$ has many such structures and properties, which together make logic possible in that setting. In this
section we want to identify these properties, and show how logical operations
can be defined using them. %
\index{limit}%
\index{colimit}%
\index{database}

Why would we want to abstractly find such structures and properties? In the next section, we'll
start our search for other categories that also have them. Such
categories, called toposes, will be $\smset$-like enough to do logic, but
have much more complex and interesting semantics. Indeed, we will discuss one whose logic allows us to reason not about properties of sets, but about behavioral properties of very general machines.%
\index{behavior|)}

\subsection{$\smset$-like properties enjoyed by any topos}%
\label{subsec.set_like}%
\index{topos!properties of|(}

Although we will not prove it in this book, toposes are categories that are similar to $\smset$ in many ways. Here are some facts that are true of any topos $\cat{E}$:
\begin{enumerate}%
\label{page.topos_properties}
	\item $\cat{E}$ has all limits,
	\item $\cat{E}$ has all colimits,
	\item $\cat{E}$ is cartesian closed,%
\index{closed category!cartesian}
	\item $\cat{E}$ has epi-mono factorizations,%
\index{epimorphism}%
\index{monomorphism}%
\index{epi-mono factorization}
	\item $\cat{E}$ has a subobject classifier $1\To{\true}\Omega$.%
\index{subobject classifier}
\end{enumerate}
In particular, since $\smset$ is a topos, all of the above facts are true for $\cat{E}=\smset$. Our first goal is to briefly review these concepts, focusing most on the subobject classifier.

\paragraph{Limits and colimits.}%
\index{limit}%
\index{colimit}
We discussed limits and colimits briefly in \cref{subsec.adjoints_lims_colims},
but the basic idea is that one can make new objects from old by taking products,
using equations to define subobjects, forming disjoint unions, and taking
quotients.
object $0$.  One of the most important types of limit (resp.\ colimit) is that
of pullbacks (resp.\ pushouts); see \cref{ex.pullback,def.pushout}.
\index{pullback} For our work below, we'll need to know a touch more about
pullbacks than we have discussed so far, so let's begin there.

Suppose that $\cat{C}$ is a category and consider the diagrams below:
\[
\begin{tikzcd}
	A\ar[r]\ar[d]&B\ar[r]\ar[d]&C\ar[d]\\
	D\ar[r]&E\ar[r]&F\ar[ul, phantom, very near end, "\lrcorner"]
\end{tikzcd}
\hspace{.8in}
\begin{tikzcd}
	A\ar[r]\ar[d]&B\ar[r]\ar[d]&C\ar[d]\\
	D\ar[r]&E\ar[r]\ar[ul, phantom, very near end, "\lrcorner"]&F
\end{tikzcd}
\]
In the left-hand square, the corner symbol $\lrcorner$ unambiguously means that the square $(B,C,E,F)$ is a pullback. But in the right-hand square, does the corner symbol mean that $(A,B,D,E)$ is a pullback or that $(A,C,D,F)$ is a pullback? It's ambiguous, but as we next show, it becomes unambiguous if the right-hand square is a pullback.

\begin{proposition}%
\label{prop.pullback_pasting}%
\index{pullbacks!pasting of}
In the commutative diagram below, suppose that the $(B,C,B',C')$ square is a pullback:
\[
\begin{tikzcd}
	A\ar[r]\ar[d]&B\ar[r]\ar[d]&C\ar[d]\\
	A'\ar[r]&B'\ar[r]\ar[ul, phantom, very near end, "\lrcorner"]&C'\ar[ul, phantom, very near end, "\lrcorner"]
\end{tikzcd}
\]
Then the $(A,B,A',B')$ square is a pullback iff the $(A,C,A',C')$ rectangle is a pullback.
\end{proposition}

\begin{exercise}%
\label{exc.pullback_pasting}
Prove \cref{prop.pullback_pasting} using the definition of limit from \cref{subsec.adjoints_lims_colims}.
\end{exercise}

\paragraph{Epi-mono factorizations.}
The abbreviation `epi' stands for \emph{epimorphism}, and the abbreviation `mono' stands for monomorphism. Epimorphisms are maps that act like surjections, and monomorphisms are maps that act like injections.%
\footnote{
Surjections are sometimes called `onto' and injections are sometimes called `one-to-one', hence the Greek prefixes \emph{epi} and \emph{mono}.}
We can define them formally in terms of pushouts and pullbacks.
\begin{definition}%
\index{epimorphism}%
\index{monomorphism}%
\label{def.mono_epi}
Let $\cat{C}$ be a category, and let $f\colon A\to B$ be a morphism. It is called a \emph{monomorphism} (resp.\ \emph{epimorphism}) if the square to the left is a pullback (resp.\ the square to the right is a pushout):
\[
\begin{tikzcd}
	A\ar[r, "\id_A"]\ar[d, "\id_A"']&A\ar[d, "f"]&[10pt]
		A\ar[r, "f"]\ar[d, "f"']&B\ar[d, "\id_B"]\\
	A\ar[r, "f"']&B\ar[ul, phantom, very near end, "\lrcorner"]&
		B\ar[r, "\id_B"']&B\ar[ul, phantom, very near start, "\ulcorner"]
\end{tikzcd}
\]
\end{definition} %
\index{pullback!monomorphism in terms of}%
\index{pushout!epimorphism in terms of}

\begin{exercise}%
\index{function!injective}%
\label{exc.mono_inj}%
\index{monomorphism!injection as}
Show that in $\smset$, monomorphisms are just injections:
\begin{enumerate}
	\item Show that if $f$ is a monomorphism then it is injective.
	\item Show that if $f\colon A\to B$ is injective then it is a monomorphism.
\qedhere
\end{enumerate}
\end{exercise}

\begin{exercise}%
\label{exc.pullback_iso_iso}%
\index{pullback!of isomorphism}
\begin{enumerate}
	\item	Show that the pullback of an isomorphism along any morphism is an isomorphism. That is, suppose that $i\colon B'\to B$ is an isomorphism and $f\colon A\to B$ is any morphism. Show that $i'$ is an isomorphism, in the following diagram:%
\index{isomorphism!as stable under pullback}
	\[
	\begin{tikzcd}[ampersand replacement=\&]
		A'\ar[r, "f'"]\ar[d, pos=.6, "i'"', "\cong"]\&B'\ar[d, pos=.6, "i", "\cong"']\\
		A\ar[r, "f"']\&B\ar[ul, phantom, very near end, "\lrcorner"]
	\end{tikzcd}
	\]	
	\item Show that for any map $f\colon A\to B$, the square shown is a pullback:
	\[
	\begin{tikzcd}
		A\ar[r, "f"]\ar[d, equal]&
		B\ar[d, equal]\\
		A\ar[r, "f"']&
		B\ar[ul, phantom, very near end, "\lrcorner"]
	\end{tikzcd}
	\qedhere
	\]
\end{enumerate}
\end{exercise}

\begin{exercise}%
\label{exc.monos_pb_pasting}
Suppose the following diagram is a pullback in a category $\cat{C}$:
\[
\begin{tikzcd}
	A'\ar[r, "g"]\ar[d, "f'"']&A\ar[d, tail, "f"]\\
	B'\ar[r, "h"']&B\ar[ul, phantom, very near end, "\lrcorner"]
\end{tikzcd}
\]
Use \cref{prop.pullback_pasting,exc.pullback_iso_iso} to show that if $f$ is a monomorphism, then so is $f'$.%
\index{monomorphism!as stable under pullback}
\end{exercise}

Now that we have defined epimorphisms and monomorphisms, we can say what epi-mono factorizations are. We say that a morphism $f\colon C\to D$ in $\cat{E}$ has an epi-mono factorization if it has an `image'; that is, there is an object $\im(f)$, an epimorphism $C\surj\im(f)$, and a monomorphism $\im(f)\inj D$, whose composite is $f$.%
\index{epi-mono factorization}

In $\smset$, epimorphisms are surjections and monomorphisms are injections.
Every function $f\colon C \to D$ may be factored as a surjective function onto
its image $\im(f)=\{f(c) \mid c \in C\}$, followed by the inclusion of this
image into the codomain $D$. Moreover, this factorization is unique up to
isomorphism.%
\index{epimorphism!surjection as}

\begin{exercise}%
\label{exc.epi_mono_practice}
Factor the following function $f\colon \ord{3}\to \ord{3}$ as an epimorphism followed by a monomorphism.
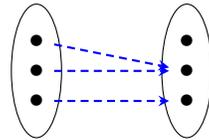
\[
  \begin{tikzpicture}[short=0pt]
		\foreach \x in {0,...,2} 
			{\draw (0,.4-.4*\x) node (X0\x) {$\bullet$};}
		\node[draw, ellipse, inner sep=0pt, fit=(X00) (X02)] (X0) {};
		\foreach \x in {0,...,2} 
			{\draw (2,.4-.4*\x) node (Y0\x) {$\bullet$};}
		\node[draw, ellipse, inner sep=0pt, fit=(Y00) (Y02)] (Y0) {};
		\draw[mapsto] (X00) -- (Y01);
		\draw[mapsto] (X01) -- (Y01);
		\draw[mapsto] (X02) -- (Y02);
  \end{tikzpicture}
  \qedhere
\]
\end{exercise}

This is the case in any topos $\cat{E}$: for any morphism $f\colon c\to d$,
there exists an epimorphism $e$ and a monomorphism $m$ such that $f=(e\cp m)$ is their composite.

\paragraph{Cartesian closed.}%
\index{closed category!cartesian}
A category $\cat{C}$ being cartesian closed means that $\cat{C}$ has a symmetric
monoidal structure given by products, and it is monoidal closed with respect to
this. (We previously saw monoidal closure in \cref{def.monoidal_closed} (for
preorders) and \cref{prop.double_dual}, as a corollary of compact closure.)
Slightly more down-to-earth, cartesian closure means that for any two objects
$C,D\in\cat{C}$, there is a `hom-object' $D^C\in\cat{C}$ and a natural
isomorphism for any $A\in\cat{C}$:
\begin{equation}%
\label{eqn.currying}
	\cat{C}(A\times C,D)\cong\cat{C}(A,D^C)
\end{equation}

Think of it this way. Suppose you're $A$ and I'm $C$, and we're interacting through some game $f(-,-)\colon A\times C\to D$: for whatever action $a\in A$ that you take and action $c\in C$ that I take, $f(a,c)$ is some value in $D$. Since you're self-centered but loving, you think of this situation as though you're creating a game experience for me. When you do $a$, you make a game $f(a,-)\colon C\to D$ for me alone. In the formalism, $D^C$ represents the set of games for me. So now you've transformed a two-player game, valued in $D$, into a one-player game, you're the player, valued in... one player games valued in $D$. This transformation is invertible---you can switch your point of view at will---and it's called \emph{currying}. This is the content of \cref{ex.currying}.%
\index{currying}

\begin{exercise}%
\label{ex.ccposet_quantale}
Let $\cat{V}=(V,\leq,I,\otimes)$ be a (unital, commutative) quantale---see
\cref{def.quantale}---and suppose it satisfies the following for all $v,w,x\in V$:
\begin{itemize}
	\item $v\leq I$,
	\item $v\otimes w\leq v$ and $v\otimes w\leq w$, and
	\item if $x\leq v$ and $x\leq w$ then $x\leq v\otimes w$.\bigskip
\end{itemize}
\begin{enumerate}
	\item Show that $\cat{V}$ is a cartesian closed category, in fact a cartesian closed preorder.
	\item Can every cartesian closed preorder be obtained in this way?
\qedhere
\end{enumerate}
\end{exercise}

\paragraph{Subobject classifier.}%
\index{subobject classifier|(}
The concept of a subobject classifier requires more attention, because its existence has huge consequences for a category $\cat{C}$. In particular, it creates the setting for a rich system of \emph{higher order logic} to exist inside $\cat{C}$; it does so by providing some things called `truth values'. The higher order logic manifests in its fully glory when $\cat{C}$ has finite limits and is cartesian closed, because these facts give rise to the logical operations on truth values.%
\footnote{A category that has finite limits, is cartesian closed, and has a subobject classifier is called an \emph{elementary topos}. We will not discuss these further, but they are the most general notion of topos in ordinary category theory. When someone says topos, you might ask ``Grothendieck topos or elementary topos?,'' because there does not seem to be widespread agreement on which is the default.
}
In particular, the higher order logic exists in any topos.

We will explain subobject classifiers in as much detail as we can; in fact, it will be our subject for the rest of \cref{sec.logic_and_set}.

\subsection{The subobject classifier}%
\label{subsec.subobj_classifier}
Before giving the definition of subobject classifiers, recall that monomorphisms in $\smset$ are injections, and any injection $X\inj Y$ is isomorphic to a subset of $Y$. This gives a simple and useful way to conceptualize monomorphisms into $Y$ when reading the following definition: it will do no harm to think of them as subobjects of $Y$.

\begin{definition}%
\label{def.subobject_classifier}%
\index{subobject classifier}
Let $\cat{E}$ be a category with finite limits, i.e.\ with pullbacks and a terminal object $1$. A \emph{subobject classifier} in $\cat{E}$ consists of an object $\Omega\in\cat{E}$, together with a monomorphism $\true\colon 1\to\Omega$, satisfying the following property: for any objects $X$ and $Y$ and monomorphism $m\colon X\inj Y$ in $\cat{E}$, there is a unique morphism $\corners{m}\colon Y\to\Omega$ such that the diagram on the left of \cref{eqn.subobj_class_pullbacks} is a pullback in $\cat{E}$:
\begin{equation}%
\label{eqn.subobj_class_pullbacks}
\begin{tikzcd}[sep=large]
	X\ar[r, "!"]\ar[d, tail, "m"']&1\ar[d,"\true"]\\
	Y\ar[r, "\corners{m}"']&\Omega\ar[ul, phantom, very near end, "\lrcorner"]
\end{tikzcd}
\hspace{1in}
\begin{tikzcd}[sep=large]
	\{Y\mid p\}\ar[r, "!"]\ar[d, tail]&1\ar[d,"\true"]\\
	Y\ar[r, "p"']&\Omega\ar[ul, phantom, very near end, "\lrcorner"]
\end{tikzcd}
\end{equation}
We refer to $\corners{m}$ as the \emph{characteristic map} of $m$, or we say that $\corners{m}$ \emph{classifies} $m$. Conversely, given any map $p\colon Y\to\Omega$, we denote the pullback of $\true$ as on the right of \cref{eqn.subobj_class_pullbacks}.

A \emph{predicate} on $Y$ is a morphism $Y\to\Omega$.%
\index{predicate}
\end{definition}

\cref{def.subobject_classifier} is a bit difficult to get one's mind around,
partly because it is hard to imagine its consequences. It is like a superdense
nugget from outer space, and through scientific explorations in the latter half
of the 20th century, we have found that it brings super powers to whichever
categories possess it. We will explain some of the consequences below, but very
quickly, the idea is the following.%
\index{subobject classifier!as superdense nugget from outer space}

When a category has a subobject classifier, it provides a translator, turning subobjects of any object $Y$ into maps from that $Y$ to the particular object $\Omega$. Pullback of the monomorphism $\true\colon 1\to\Omega$ provides a translator going back, turning maps $Y\to\Omega$ into subobjects of $Y$. We can replace our fantasy of the superdense nugget with a slightly more refined story: ``any object $Y$ understands itself---its parts and the logic of how they fit together---by asking questions of the oracle $\Omega$, looking for what's true.'' Or to fully be precise but dry, ``subobjects of $Y$ are classified by predicates on $Y$.''

Let's move from stories and slogans to concrete facts.

\paragraph{The subobject classifier in $\smset$.}
\index{subobject classifier!in $\smset$}%
\index{booleans!as subobject classifier}
Since $\smset$ is a topos, it has a subobject classifier. It will be a set with supposedly wonderful properties; what set is it?

The subobject classifier in $\smset$ is the set of booleans,
\begin{equation}%
\label{eqn.Omega_Set}
	\Omega_\smset\coloneqq\BB=\{\true,\false\}.
\end{equation}
So in $\smset$, the truth values are true and false.

By definition (Def.~\ref{def.subobject_classifier}), the subobject classifier comes equipped with a morphism, generically called $\true\colon1\to\Omega$; in the case of $\smset$ it is played by the function $1\to\{\true,\false\}$ that sends 1 to $\true$. In other words, the morphism $\true$ is aptly named in this case.

For sets, monomorphism just means injection, as we mentioned above. So
\cref{def.subobject_classifier} says that for any injective function $m\colon
X\inj Y$ between sets, we are supposed to be able to find a characteristic
function $\corners{m}\colon Y\to\{\true,\false\}$ with some sort of pullback
property. We propose the following definition of $\corners{m}$:
\[
\corners{m}(y)\coloneqq
\begin{cases}
	\true&\tn{ if }m(x)=y\text{ for some $x\in X$}\\
	\false&\tn{ otherwise}
\end{cases}
\]
In other words, if we think of $X$ as a subobject of $Y$, then we make $\corners{m}(y)$ equal to $\true$ iff $y\in X$.

In particular, the subobject classifier property turns subsets $X\ss Y$ into functions $p\colon Y \to \BB$, and vice versa. How it works is encoded in \cref{def.subobject_classifier}, but the basic idea is that $X$ will be the set of all things in $Y$ that $p$ sends to $\true$:
\begin{equation}%
\label{eqn.pullback_concrete_sets}
	X=\{y\in Y\mid p(y)=\true\}.
\end{equation}
This might help explain our abstract notation $\{Y\mid p\}$ in \cref{eqn.subobj_class_pullbacks}.%
\index{notation!for classified subobjects}

\begin{exercise}%
\label{exc.characteristic_practice}
  Let $X=\NN=\{0,1,2,\ldots\}$ and $Y=\ZZ=\{\ldots,-1,0,1,2,\ldots\}$; we have
  $X\ss Y$, so consider it as a monomorphism $m\colon X\inj Y$. It has a
  characteristic function $\corners{m}\colon Y\to\BB$, as in
  \cref{def.subobject_classifier}.
  \begin{enumerate}
    \item What is $\corners{m}(-5)\in\BB$?
    \item What is $\corners{m}(0)\in\BB$?
  \qedhere
\end{enumerate}
\end{exercise}

\begin{exercise}%
\label{exc.simple_char_funs}
  \begin{enumerate}
    \item Consider the identity function $\id_\NN\colon \NN\to \NN$. It is an
      injection, so it has a characteristic function $\corners{\id_\NN}\colon
      \NN\to\BB$.  Give a concrete description of $\corners{\id_\NN}$, i.e.\ its
      exact value for each natural number $n\in\NN$.
    \item Consider the unique function $!_\NN\colon\varnothing\to\NN$ from the
      empty set. Give a concrete description of $\corners{!_\NN}\colon\nn\to\bb$.
  \qedhere
\end{enumerate}
\end{exercise}

\index{subobject classifier|)}%
\index{topos!properties of|)}

\subsection{Logic in the topos $\smset$}%
\label{subsubsec.logic_set}%
\index{logic|(}

As we said above, the subobject classifier of any topos $\cat{E}$ gives the setting in which to do logic. Before we explain a bit about how topos logic works in general, we continue to work concretely by focusing on logic in the topos $\smset$.%
\index{subobject classifier!and logic}

\paragraph{Obtaining the AND operation.}%
\index{AND operation}

Consider the function $1\to\BB\times\BB$ picking out the element $(\true,\true)$. This is a monomorphism, so it defines a characteristic function $\corners{(\true,\true)}\colon\BB\times\BB\to\BB$. What function is it? By \cref{eqn.pullback_concrete_sets} the only element of $\BB\times\BB$ that can be sent to $\true$ is $(\true,\true)$. Thus $\corners{(\true,\true)}(P,Q)\in\BB$ must be given by the following truth table
\[
\begin{array}{cc||c}
	P&Q&\corners{(\true,\true)}(P,Q)\\\hline
	\true&\true&\true\\
	\true&\false&\false\\
	\false&\true&\false\\
	\false&\false&\false
\end{array}
\]
This is exactly the truth table for the AND of $P$ and $Q$, i.e.\ for $P\wedge Q$. In
other words, $\corners{(\true,\true)}=\wedge$. Note that this defines $\wedge$
as a function $\wedge\colon \BB \times \BB \to \BB$, and we use the usual
infix notation $x \wedge y\coloneqq \wedge(x,y)$.%
\index{infix notation}

\paragraph{Obtaining the OR operation.}%
\index{OR operation}
Let's go backwards this time. The truth table for the OR of $P$ and $Q$, i.e.\ that of the function $\vee\colon\BB\times\BB\to\BB$ defining OR, is:
\begin{equation}%
\label{eqn.function_vee}
\begin{array}{cc||c}
	P&Q&P\vee Q\\\hline
	\true&\true&\true\\
	\true&\false&\true\\
	\false&\true&\true\\
	\false&\false&\false
\end{array}
\end{equation}
If we wanted to obtain this function as the characteristic function
$\corners{m}$ of some subset $m\colon X\ss\BB\times\BB$, what subset would $X$
be? By \cref{eqn.pullback_concrete_sets}, $X$ should be the set of $y\in Y$ that
are sent to $\true$. Thus $m$ is the characteristic map for the three element
subset \[X=\{(\true,\true),(\true,\false),(\false,\true)\}\ss\BB\times\BB.\] To
prepare for later generalization of this idea in any topos, we want a way of thinking of $X$ only in terms of properties
listed at the beginning of \cref{subsec.set_like}. In fact, one can think of $X$
as the union of $\{\true\}\times\BB$ and $\BB\times\{\true\}$---a colimit of
limits involving the subobject classifier and terminal object. This description will construct
an analogous subobject of $\Omega\times\Omega$, and hence classify a map $\Omega\times\Omega\to\Omega$, in any topos $\cat{E}$.

\begin{exercise}%
\index{NOT operation}%
\label{exc.neg_char}
Every boolean has a negation, $\neg\false=\true$ and $\neg\true=\false$. The function $\neg\colon\BB\to\BB$ is the characteristic function of some thing, (*?*).
\begin{enumerate}
	\item What sort of thing should (*?*) be? For example, should $\neg$ be the characteristic function of an object? A topos? A morphism? A subobject? A pullback diagram?
	\item Now that you know the sort of thing (*?*) is, which thing of that sort is it?
\qedhere
\end{enumerate}
\end{exercise}

\begin{exercise}%
\index{IMPLIES operation}%
\label{exc.implies_char}%
\index{booleans!as monoidal closed}
Given two booleans $P,Q$, define $P\imp Q$ to mean $P=(P\wedge Q)$.
\begin{enumerate}
	\item Write down the truth table for the statement $P=(P\wedge Q)$:
	\[
	\begin{array}{cc||c|c}
		P&Q&P\wedge Q&P=(P\wedge Q)\\
		\true&\true&\?&\?\\
		\true&\false&\?&\?\\
		\false&\true&\?&\?\\
		\false&\false&\?&\?\\
	\end{array}
	\]
	\item If you already have an idea what $P\imp Q$ should mean, does it agree with the last column of table above?
	\item What is the characteristic function $m\colon \BB\times\BB\to\BB$ for $P\imp Q$?
	\item What subobject does $m$ classify?
\qedhere
\end{enumerate}
\end{exercise}

\begin{exercise}%
\label{exc.even_prime_10}
Consider the sets $E\coloneqq\{n\in\NN\mid n\text{ is even}\}$, $P\coloneqq\{n\in\NN\mid n\text{ is prime}\}$, and $T\coloneqq\{n\in\NN\mid n\geq 10\}$. Each is a subset of $\NN$, so defines a function $\NN\to\BB$.
\begin{enumerate}
	\item What is $\corners{E}(17)$?
	\item What is $\corners{P}(17)$?
	\item What is $\corners{T}(17)$?
	\item Name the smallest three elements in the set classified by $(\corners{E}\wedge\corners{P})\vee\corners{T}$.
\qedhere
\end{enumerate}
\end{exercise}

\paragraph{Review.}
Let's take stock of where we are and where we're going. In \cref{sec.safe_machines}, we set out our goal of proving properties about behavior, and we said that topos theory is a good mathematical setting for doing that. We are now at the end of \cref{sec.logic_and_set}, which was about $\smset$ as an examplar topos. What happened?

In \cref{subsec.set_like}, we talked about properties of $\smset$ that are enjoyed by any topos: limits and colimits, cartesian closure, epi-mono factorizations, and subobject classifiers. Then in \cref{subsec.subobj_classifier} we launched into thinking about the subobject classifier in general and in the specific topos $\smset$, where it is the set $\bb$ of booleans because any subset of $Y$ is classified by a specific predicate $p\colon Y\to\BB$. Finally, in \cref{subsubsec.logic_set} we discussed how to understand logic in terms of $\Omega$: there are various maps $\wedge,\vee,\imp\colon\Omega\times\Omega\to\Omega$ and $\neg\colon\Omega\to\Omega$ etc., which serve as logical connectives. These are operations on truth values.%
\index{limit}%
\index{colimit}

We have talked a lot about toposes, but we've only seen one so far: the category of sets. But we've actually seen more without knowing it: the category $\cat{C}\inst$ of instances on any database schema from \cref{def.instance} is a topos. Such toposes are called \emph{presheaf toposes} and are fundamental, but we will focus on \emph{sheaf toposes}, because our topos of behavior types will be a sheaf topos.%
\index{topos!database instances as}

Sheaves are fascinating, but highly abstract mathematical objects. They are not for the faint of mathematical heart (those who are faint of physical heart are welcome to proceed). 

\index{logic|)}

\section{Sheaves}%
\label{sec.sheaves}%
\index{sheaf|(}

Sheaf theory began before category theory, e.g.\ in the form of something called
``local coefficient systems for homology groups.'' However its modern
formulation in terms of functors and sites is due to Grothendieck, who also
invented toposes.

The basic idea is that rather than study spaces, we should study what happens
\emph{on} spaces. A space is merely the `site' at which things happen. For
example, if we think of the plane $\RR^2$ as a space, we might examine only
points and regions in it.  But if we think of $\RR^2$ as a site where things happen,
then we might think of things like weather systems throughout the plane, or sand
dunes, or trajectories and flows of material. There are many sorts of things
that can happen on a space, and these are the sheaves: a sheaf on a space is roughly ``a
sort of thing that can happen on the space.'' If we want to think about points
or regions from the sheaf perspective, we would consider them as different points of view on what's happening. That is, it's all about
what happens on a space: the parts of the space are just perspectives from which
to watch the show.%
\index{site}

This is reminiscent of databases. The schema of a database is not the interesting part; the data is what's interesting. To be clear, the schema of a database is a site---it's acting like the space---and the category of all instances on it is a topos. In general, we can think of any small category $\cat{C}$ as a site; the corresponding topos is the category of functors $\cat{C}\op\to\smset$.%
\footnote{The category of functors $\cat{C}\to\smset$ is also a topos: use
$\cat{C}\op$ as the defining site.}%
\index{database!instances form a topos}
Such functors are called \emph{presheaves on $\cat{C}$}.%
\index{category!of database instances}%
\index{site!database schema as}

Did you notice that we just introduced a huge class of toposes? For any category $\cat{C}$, we said there is a topos of presheaves on it. So before we go on to sheaves, let's discuss this preliminary topic of presheaves. We will begin to develop some terminology and ways of thinking that will later generalize to sheaves.%
\index{presheaves!topos of}

\subsection{Presheaves}%
\label{subsec.what_are_toposes}%
\index{presheaf|(}

Recall the definition of functor and natural transformation from \cref{sec.cat_fun_nt_db}. Presheaves are just functors, but they have special terminology that leads us to think about them in a certain geometric way.

\begin{definition}%
\label{def.presheaf}%
\index{presheaf}%
\index{opposite category!and presheaves}
Let $\cat{C}$ be a small category. A \emph{presheaf} $P$ on $\cat{C}$ is a functor $P\colon\cat{C}\op\to\smset$. To each object $c\in\cat{C}$, we refer to the set $P(c)$ as \emph{the set of sections of $P$ over $c$}. To each morphism $f\colon c'\to c$, we refer to the function $P(f)\colon P(c)\to P(c')$ as the \emph{restriction map along $f$}. For any section $s\in P(c)$, we may denote $P(f)(s)\in P(c')$, i.e.\ its restriction along $f$, by $\restrict{s}{f}$.%
\index{restriction map|see {presheaf, restriction map}}%
\index{presheaf!restriction maps}%
\index{presheaf!sections}%
\index{functor!presheaf as}

If $P$ and $Q$ are presheaves, a \emph{morphism} $\alpha\colon P\to Q$ between them
is a natural transformation of functors%
\index{presheaves!morphism of}
\[
\begin{tikzcd}
  \cat{C}\op\ar[r, bend left, "P", ""'{name=P, below}]\ar[r, bend right, "Q"', ""{name=Q, above}]\ar[from=P, to=Q, Rightarrow, "\alpha"]&\smset.
\end{tikzcd}
\]
\end{definition}%
\index{natural transformation!as presheaf morphism}

\begin{example}%
\label{ex.graph_presheaf_topos}%
\index{graphs!topos of}
Let $\Cat{ArShp}$ be the category shown below:
\[
\Cat{ArShp}\coloneqq\boxCD{
\begin{tikzcd}[ampersand replacement=\&, column sep=large]
	\LTO{Vertex}\ar[r, shift left, "\mathrm{src}"]\ar[r, shift right, "\mathrm{tgt}"']\&\LTO{Pure Arrow}
\end{tikzcd}
}
\]
The reason we call our category $\Cat{ArShp}$ is that we can imagine of it as an `arrow shape.'
\begin{equation}%
\label{eqn.arshape}
\begin{tikzpicture}
	\node[coordinate] (Left) {};
	\node[coordinate, right=1 of Left] (Right) {};
	\node at ($(Left)!.5!(Right)$) (Center) {};
	\fill[fill=black] (Left)  circle (1.1pt);
	\fill[fill=black] (Right) circle (1.1pt);
	\draw[->, very thick] (Left) -- (Right);
	\node[draw, inner sep=10pt, fit=(Left) (Right)] (Arrow) {};
	\node[left=0 of Arrow] {Pure Arrow$\coloneqq$};
	\node[above=2 of Center] (V) {};
	\fill[fill=black] (V) circle (2.2pt);
	\node[draw, fit=(V)] (Vertex) {};
	\node[left=0 of Vertex] {Vertex$\coloneqq$};
	\begin{scope}[|->, blue, shorten >=3pt]
  	\draw[bend right=15pt] (V) to node[left] {$\mathrm{src}$} (Left);
  	\draw[bend left=15pt] (V) to node[right] {$\mathrm{tgt}$} (Right);
	\end{scope}
\end{tikzpicture}
\end{equation}
A presheaf on $\Cat{ArShp}$ is a functor $I\colon\Cat{ArShp}\op\to\smset$,
which is a database instance on $\Cat{ArShp}\op$. Note that $\Cat{ArShp}\op$ is
what we called $\Cat{Gr}$ in \cref{subsec.instances_cat}; there we showed that
database instances on $\Cat{Gr}$---i.e.\ presheaves on $\Cat{ArShp}$--- are just
directed graphs, e.g.\
\[
P\coloneqq\boxCD{
\begin{tikzcd}[ampersand replacement=\&]
	\bullet\&\bullet\ar[l]\ar[r]\&\bullet\ar[d, shift left]\ar[d, shift right]\&\bullet\ar[loop below]\\
	\&\bullet\ar[r, shift left]\&\bullet\ar[l, shift left]\&\bullet\ar[r]\&\bullet\ar[ul, bend right]
\end{tikzcd}
}
:\Cat{ArShp}\op\to\smset
\]%
\index{presheaf!as database instance}

Thinking of presheaves on any category $\cat{C}$, it often makes sense to imagine the objects of $\cat{C}$ as shapes of some sort, and the morphisms of $\cat{C}$ as continuous maps between shapes, just like we did for the arrow shape in \cref{eqn.arshape}. In that context, one can think of a presheaf $P$ as a kind of lego construction: $P$ is built out of the shapes in $\cat{C}$, connected together using the morphisms in $\cat{C}$. In the case where $\cat{C}$ is the arrow shape, a presheaf is a graph. So this would say that a graph is a sort of lego construction, built out of vertices and arrows connected together using the inclusion of a vertex as the source or target of an arrow. Can you see it?

This statement can be made pretty precise; though we cannot go through it here,
the above lego idea is summarized by the formal statement that ``the category of
presheaves on $\cat{C}$ is the free colimit completion of
$\cat{C}$.''%
\index{colimit!presheaves form colimit completion} Ask a friendly neighborhood category theorist for details.%
\index{category!of presheaves}
\end{example}

However one thinks of presheaves---in terms of lego assemblies or database instances---they're relatively straightforward. The difference between presheaves and sheaves is that sheaves take into account some sort of `covering information.' The trivial notion of covering is to say that every object covers itself and nothing more; if one uses this trivial covering, presheaves and sheaves are the same thing. In our behavioral context we will need a non-trivial notion of covering, so sheaves and presheaves will be slightly different. Our next goal is to understand sheaves on a topological space.%
\index{cover}

\index{presheaf|)}

\subsection{Topological spaces} %
\label{subsec.topology}%
\index{topological space|(}
We said in \cref{sec.sheaves} that, rather than study spaces, we consider spaces
as mere `sites' on which things happen. We also said the things that can
happen on a space are called sheaves, and always form a type of category called
a topos. To define a topos of sheaves, we must start with the site on which they
exist.

Sites are very abstract mathematical objects, and we will not make them precise
in this book. However, one of the easiest sorts of sites to think about are
those coming from topological spaces: every topological space naturally has the
structure of a site. We've talked about spaces for a while without making them
precise; let's do so now.%
\index{site!topological space as}

\begin{definition}%
\label{def.topological_space}%
\index{topological space}%
\index{open set}
  Let $X$ be a set, and let $\powset(X)=\{U\ss X\}$ denote its set of subsets. A
  \emph{topology} on $X$ is a subset $\Op\ss\powset(X)$, elements of which we call \emph{open sets},%
  \tablefootnote{In other words, we refer to a subset $U\ss X$ as \emph{open} if $U\in\Op$.}
  satisfying the following conditions:
  \begin{enumerate}[label=(\alph*)]
  	\item Whole set: the subset $X\ss X$ is open, i.e.\ $X\in\Op$.
  	\item Binary intersections: if $U,V\in\Op$ then $(U\cap V)\in\Op$.
	\item Arbitrary unions: if $I$ is a set and if we are given an open set
	  $U_i\in\Op$ for each $i$, then their union is also open,
	  $\big(\bigcup_{i\in I}U_i\big)\in\Op$. We interpret the particular
	  case where $I=\varnothing$ to mean that the empty set is open:
	  $\varnothing\in\Op$. 
  \end{enumerate}
  If $U=\bigcup_{i\in I}U_i$, we say that $(U_i)_{i\in I}$ \emph{covers} $U$.

  A pair $(X,\Op)$, where $X$ is a set and $\Op$ is a topology on $X$, is called a \emph{topological space}.%
\index{cover}

  A \emph{continuous function} between topological spaces $(X,\Op_X)$ and
  $(Y,\Op_Y)$ is a function $f\colon X\to Y$ such that for every $U\in\Op_Y$,
  the preimage $f^{-1}(U)$ is in $\Op_X$. %
\index{continuous function}%
\index{preimage}
\end{definition}

At the very end of \cref{subsec.what_are_toposes} we mentioned how sheaves differ from presheaves in that they take into account `covering information.' The notion of covering an open set by a union of other open sets was defined in \cref{def.topological_space}, and it will come into play when we define sheaves in \cref{def.sheaf}.

\begin{example}%
\label{ex.usual_R}%
\index{real numbers!topology on}
The usual topology $\Op$ on $\RR^2$ is based on `$\epsilon$-balls.' For any $\epsilon\in\RR$ with $\epsilon>0$, and any point $p=(x,y)\in\RR^2$, define \emph{the $\epsilon$-ball centered at $p$} to be:
\[B(p;\epsilon)\coloneqq\{p'\in\RR^2\mid d(p,p')<\epsilon\}%
\tablefootnote{Here, $d((x,y),(x',y'))\coloneqq\sqrt{(x-x')^2+(y-y')^2}$ is the usual `Euclidean distance' between two points. One can generalize $d$ to any metric.}
\]
In other words, $B(x,y;\epsilon)$ is the set of all points within $\epsilon$ of $(x,y)$.

For an arbitrary subset $U\ss\RR^2$, we call it open and put it in $\Op$ if, for every $(x,y)\in U$ there exists a (small enough) $\epsilon>0$ such that $B(x,y;\epsilon)\ss U$.
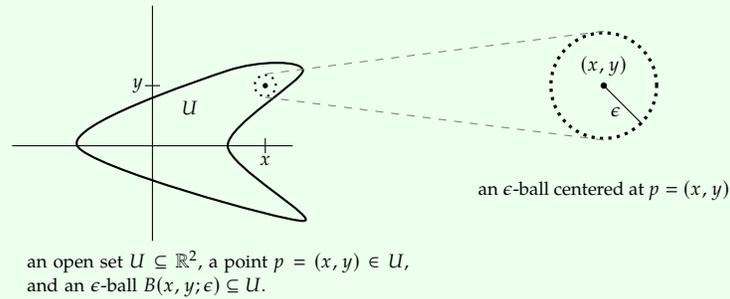
\[
\begin{tikzpicture}[font=\tiny]
	\node (north) at (0,2) {};
	\node (south) at (0,-1.4) {};
	\node (east) at (2,0) {};
	\node (west) at (-2,0) {};
	\draw (north) to (south);
	\draw (east) to (west);
  \draw [thick] plot [smooth cycle] coordinates {(-1,0) (1,1) (2,1) (1,0) (2,-1)};
  \node (xy) at (1.5,.8) {};
	\draw (1.5, -.1) to[pos=-.5] node {$x$} (1.5, .1);
	\draw (-.1, .8) to[pos=-.5] node {$y$} (.1, .8);
  \filldraw (xy) circle (.8pt);
  \node [dotted, thick, circle, inner sep=0, minimum size=8pt, draw] (surround xy) at (xy.center) {};
  \node (bigxy) at (6, .8) {};
  \node[above=-4pt of bigxy] {$(x,y)$};
  \filldraw (bigxy) circle (1pt);
  \node[dotted, very thick, circle, minimum size=40pt, draw] (surround bigxy) at (bigxy) {};
  \draw (bigxy.center) to node[below left=-3pt and -3pt] {$\epsilon$} (surround bigxy.south east);
  \draw[dashed, gray] (surround xy.north) -- (surround bigxy.north);
  \draw[dashed, gray] (surround xy.south) -- (surround bigxy.south);
  \node[below=.4 of surround bigxy] {an $\epsilon$-ball centered at $p=(x,y)$};
  \node at (.5,.5) {$U$};
  \node[text width=2.1in] at (1,-1.7) {an open set $U\ss\RR^2$, a point $p=(x,y)\in U$, and an $\epsilon$-ball $B(x,y;\epsilon)\ss U$.};
 \end{tikzpicture}
\]

The same idea works if we replace $\rr^2$ with any other metric space $X$
(\cref{def.ord_metric_space}): it can be considered as a topological space
where the open sets are subsets $U$ such that for any $p\in U$ there is an
$\epsilon$-ball centered at $p$ and contained in $U$. So every metric space can be considered as a topological space.%
\index{metric space!as topological space}
\end{example}

\begin{exercise}%
\label{ex.usual_top_R}
Consider the set $\rr$. It is a metric space with $d(x_1,x_2)\coloneqq|x_1-x_2|$.
\begin{enumerate}
	\item What is the 1-dimensional analogue of $\epsilon$-balls as found in \cref{ex.usual_R}? That is, for each $x\in\RR$, define $B(x,\epsilon)$.
	\item When is an arbitrary subset $U\ss\RR$ called open, in analogy with \cref{ex.usual_R}?
	\item Find three open sets $U_1$, $U_2$, and $U$ in $\RR$, such that $(U_i)_{i\in\{1,2\}}$ covers $U$.
	\item Find an open set $U$ and a collection $(U_i)_{i\in I}$ of opens sets where $I$ is infinite, such that $(U_i)_{i\in I}$ covers $U$.
\qedhere
\end{enumerate}
\end{exercise}

\begin{example}%
\label{ex.coarse_fine}%
\index{topology!discrete}
For any set $X$, there is a `coarsest' topology, having as few open sets as possible: $\Op_{\mathrm{crse}}=(\varnothing,X)$. There is also a `finest' topology, having as many open sets as possible: $\Op_{\mathrm{fine}}=\powset(X)$. The latter, $(X,\powset(X))$ is called the \emph{discrete space on the set $X$}.
\end{example}

\begin{exercise}%
\label{exc.course_fine}
\begin{enumerate}
	\item Verify that for any set $X$, what we called $\Op_{\mathrm{crse}}$ in \cref{ex.coarse_fine} really is a topology, i.e.\ satisfies the conditions of \cref{def.topological_space}.
	\item Verify also that $\Op_{\mathrm{fine}}$ really is a topology.
	\item Show that if $(X,\powset(X))$ is discrete and $(Y,\Op_Y)$ is any topological space, then every function $X\to Y$ is continuous.
\qedhere
\end{enumerate}
\end{exercise}

\begin{example}%
\label{ex.Sierpinski}%
\index{Sierpinski space}
There are four topologies possible on $X=\{1,2\}$. Two are $\Op_{\mathrm{crse}}$ and $\Op_{\mathrm{fine}}$ from \cref{ex.coarse_fine}. The other two are:
\[
  \Op_1\coloneqq\left\{\varnothing,\{1\},X\right\}
  \qquad\text{and}\qquad
  \Op_2\coloneqq\left\{\varnothing,\{2\},X\right\}
\]
The two topological spaces $(\{1,2\},\Op_1)$ and $(\{1,2\},\Op_2)$ are isomorphic; either one can be called \emph{the Sierpinski space}.
\end{example}

\paragraph{The open sets of a topological space form a preorder.}
\index{preorder!of open sets}

Given a topological space $(X,\Op)$, the set $\Op$ has the structure of a preorder using the subset relation, $(\Op,\ss)$. It is reflexive because $U\ss U$ for any $U\in\Op$, and it is transitive because if $U\ss V$ and $V\ss W$ then $U\ss W$.

Recall from \cref{subsubsec.pos_free_spectrum} that we can regard any preorder, and hence $\Op$, as a category: its objects are the open sets $U$ and for any $U,V$ the set of morphisms $\Op(U,V)$ is empty if $U\not\ss V$ and it has one element if $U\ss V$.

\begin{exercise}%
\label{exc.opens_Sierp}
Recall the Sierpinski space, say $(X,\Op_1)$ from \cref{ex.Sierpinski}.
\begin{enumerate}
	\item Write down the Hasse diagram for its preorder of opens.
	\item Write down all the covers.
\qedhere
\end{enumerate}
\end{exercise}

\begin{exercise}%
\label{exc.subspace_topology}%
\index{topology!subspace}
  Given any topological space $(X,\Op)$, any subset $Y\subseteq X$ can be given the
  \emph{subspace topology}, call it $\Op_{?\cap Y}$. This topology defines any $A \subseteq Y$ to be open, $A\in\Op_{?\cap Y}$,
if there is an open set $B\in\Op$ such that $A = B \cap Y$.
\begin{enumerate}
	\item Find a $B\in\Op$ that shows that the whole set $Y$ is open, i.e.\ $Y\in\Op_{?\cap Y}$.
	\item Show that $\Op_{?\cap Y}$ is a topology in the sense of \cref{def.topological_space}.%
	\footnote{Hint 1: for any set $I$, collection of sets $(U_i)_{i\in I}$ with $U_i\ss X$, and set $V\ss X$, one has $\left(\bigcup_{i\in I}U_i\right)\cap V=\bigcup_{i\in I}(U_i\cap V)$. Hint 2: for any $U, V, W\ss X$, one has $(U\cap W)\cap (V\cap W)=(U\cap V)\cap W$.}
	\item Show that the inclusion function $Y \hookrightarrow X$ is a
	  continuous function.
\qedhere
\end{enumerate}
\end{exercise}

\begin{remark}%
\label{rem.top_sp_quantale}%
\index{quantale!of open sets}
Suppose $(X,\Op)$ is a topological space, and consider the preorder $(\Op,\ss)$ of open sets. It turns out that $(\Op,\ss,X,\cap)$ is always a quantale in the sense of \cref{def.monoidal_closed}. We will not need this fact, but we invite the reader to think about it a bit in \cref{exc.top_sp_quantale}.
\end{remark}

\begin{exercise}%
\label{exc.top_sp_quantale}%
\index{booleans!as base of enrichment for preorders}
In \cref{subsec.Lawv_metric_spaces,subsec.preorders_Bool_enriched} we discussed how $\Bool$-categories are preorders and $\Cost$-categories are Lawvere metric spaces, and in \cref{subsec.variations_quantale} we imagined interpretations of $\cat{V}$-categories for other quantales $\cat{V}$.

If $(X,\Op)$ is a topological space and $\cat{V}$ the corresponding quantale as in \cref{rem.top_sp_quantale}, how might we imagine a $\cat{V}$-category? 
\end{exercise}

\subsection{Sheaves on topological spaces}%
\label{subsec.sheaves_on_spaces}

To summarize where we are, a topological space $(X,\Op)$ is a set $X$ together with a bunch of subsets we call `open'; these open subsets form a preorder---and hence category---denoted $\Op$. Sheaves on $X$ will be presheaves on $\Op$ with a special property, aptly named the `sheaf condition.'

Recall the terminology and notation for presheaves: a presheaf on $\Op$ is a functor $P\colon \Op\op\to\smset$. Thus to every open set $U\in\Op$ we have a set $P(U)$, called the set of \emph{sections over $U$}, and to every inclusion of open sets $V\ss U$ we have a function $P(U)\to P(V)$ called the \emph{restriction}. If $s\in P(U)$ is a section over $U$, we may denote its restriction to $V$ by $\restrict{s}{V}$. Recall that we say a collection of open sets $(U_i)_{i\in I}$ \emph{covers} an open set $U$ if $U=\bigcup_{i\in I}U_i$.

We are now ready to give the following definition, which comes in several waves: we first define matching families, then gluing, then sheaf condition, then sheaf, and finally the category of sheaves.%
\index{category!of sheaves}

\begin{definition}%
\label{def.sheaf}%
\index{sheaf}%
\index{sheaf!condition}%
\index{matching family}%
\index{cover}%
\index{gluing}
Let $(X,\Op)$ be a topological space, and let $P\colon\Op\op\to\smset$ be a presheaf on $\Op$. 

Let $(U_i)_{i\in I}$ be a collection of open sets $U_i\in\Op$ covering $U$. A \emph{matching family $(s_i)_{i\in I}$ of $P$-sections over $(U_i)_{i\in I}$} consists of a section $s_i\in P(U_i)$ for each $i\in I$, such that for every $i,j\in I$ we have 
\[
  \restrict{s_i}{U_i\cap U_j}=\restrict{s_j}{U_i\cap U_j}.
\]

Given a matching family $(s_i)_{i\in I}$ for the cover $U=\bigcup_{i\in I}U_i$,
we say that $s\in P(U)$ is a \emph{gluing}, or \emph{glued section}, of the
matching family if $\restrict{s}{U_i}=s_i$ holds for all $i\in I$.
  
If there exists a unique gluing $s\in P(U)$ for every matching family $(s_i)_{i\in I}$, we say that $P$ \emph{satisfies the sheaf condition for the cover $U=\bigcup_{i\in I}U_i$}. If $P$ satisfies the sheaf condition for every cover, we say that $P$ is a \emph{sheaf} on $(X,\Op)$.

Thus a sheaf is just a presheaf satisfying the sheaf condition for every open cover. If $P$ and $Q$
are sheaves, then a \emph{morphism} $f\colon P\to Q$ between these sheaves is just a morphism---that
is, a natural transformation---between their underlying presheaves. We denote by
$\Shv(X,\Op)$ the category of sheaves on $X$.%
\index{sheaves!morphism of}
\end{definition}

The category of sheaves on $X$ is a topos, but we'll get to that.%
\index{topos!as category of sheaves}

\begin{example}%
\label{ex.empty_cover}%
\index{cover!empty}
Here is a funny---but very important---special case to which the notion of matching family applies. We do not give this example for intuition, but because (to emphasize) it's an important and easy-to-miss case. Just like the sum of no numbers is 0 and the product of no numbers is $1$, the union of no sets is the empty set. Thus if we take $U=\varnothing\ss X$ and $I=\varnothing$, then the empty collection of subsets (one for each $i\in I$, of which there are none) covers $U$. In this case the empty tuple $()$ counts a matching family of sections, and it is the only matching family for the empty cover of the empty set.

In other words, in order for a presheaf $P\colon\Op\op\to\smset$ to be a sheaf, a necessary (but rarely sufficient) condition is that $P(\varnothing)\cong\{()\}$, i.e.\ $P(\varnothing)$ must be a set with one element.
\end{example}

\paragraph{Extended example: sections of a function.}%
\index{sheaf!of sections of a function}
This example is for intuition, and gives a case where the `section' and `restriction' terminology are easy to visualize.

Consider the function $f\colon X\to Y$ shown below, where each element of $X$ is
sent to the element of $Y$ immediately below it. For example,
$f(a_1)=f(a_2)=a$, $f(b_1)=b$, and so on.
\begin{equation}%
\label{eqn.sections_of_function}
\begin{tikzpicture}[y=.35cm, every label/.style={font=\scriptsize}, baseline=(f)]
	\node[label={[above=-5pt]:$a$}] (Ya) {$\bullet$};
	\node[right=1 of Ya,  label={[above=-5pt]:$b$}] (Yb) {$\bullet$};
	\node[right=1 of Yb,  label={[above=-5pt]:$c$}] (Yc) {$\bullet$};
	\node[right=1 of Yc,  label={[above=-5pt]:$d$}] (Yd) {$\bullet$};
	\node[right=1 of Yd,  label={[above=-5pt]:$e$}] (Ye) {$\bullet$};
	\node[draw, inner ysep=4pt, fit={(Ya) ($(Yb.north)+(0,1ex)$) (Ye)}] (Y) {};
	\node[left=0 of Y] (Ylab) {$Y\coloneqq$};
  \node[above=4 of Ya, label={[above=-5pt]:$a_1$}] (Xa1) {$\bullet$};
  \node[above=1 of Xa1, label={[above=-5pt]:$a_2$}] (Xa2) {$\bullet$};
  \node[above=4 of Yb, label={[above=-5pt]:$b_1$}] (Xb1) {$\bullet$};
  \node[above=1 of Xb1, label={[above=-5pt]:$b_2$}] (Xb2) {$\bullet$};
  \node[above=1 of Xb2, label={[above=-5pt]:$b_3$}] (Xb3) {$\bullet$};
  \node[above=4 of Yc, label={[above=-5pt]:$c_1$}] (Xc1) {$\bullet$};
  \node[above=4 of Ye, label={[above=-5pt]:$e_1$}] (Xe1) {$\bullet$};
  \node[above=1 of Xe1, label={[above=-5pt]:$e_2$}] (Xe2) {$\bullet$};
	\node[draw, inner ysep=3pt, fit={(Xa2) ($(Xb3.north)+(0,1ex)$) (Xe1)}] (X) {};
	\node[left=0 of X] {$X\coloneqq$};
	\draw[->, shorten <=3pt, shorten >=3pt] (X) to node[left] (f) {$f$} (Y);
\end{tikzpicture}
\end{equation}
For each point $y\in Y$, the preimage set $f^{-1}(y)\ss X$ above it is often called the \emph{fiber over $y$}. Note that different $f$'s would arrange the eight elements of $X$ differently over $Y$: elements of $Y$ would have different fibers.%
\index{preimage}%
\index{fiber|see {preimage}}

\begin{exercise}%
\label{exc.fiber_practice}
Consider the function $f\colon X\to Y$ shown in \cref{eqn.sections_of_function}.
\begin{enumerate}
	\item What is the fiber of $f$ over $a$?
	\item What is the fiber of $f$ over $c$?
	\item What is the fiber of $f$ over $d$?
	\item Gave an example of a function $f'\colon X\to Y$ for which every fiber has either one or two elements.
\qedhere
\end{enumerate}
\end{exercise}

Let's consider $X$ and $Y$ as discrete topological spaces, so every subset is open, and $f$ is automatically  continuous (see \cref{exc.course_fine}). We will think of $f$ as an arrangement of $X$ over $Y$, in terms of fibers as above, and use it to build a sheaf on $Y$. To do this, we begin by building a presheaf---i.e.\ a functor $\Fun{Sec}_f\colon\Op(Y)\op\to\smset$---and then we'll prove it's a sheaf.

Define the presheaf $\Fun{Sec}_f$ on an arbitrary subset $U\ss Y$ by:
\[\Fun{Sec}_f(U)\coloneqq\{s\colon U\to X\mid(s\cp f)(u)=u\text{ for all }u\in U\}.\]
One might describe $\Fun{Sec}_f(U)$ as the set of all ways to pick a `cross-section' of the $f$ arrangement over $U$. That is, an element $s\in\Fun{Sec}_f(U)$ is a choice of one element per fiber over $U$.%
\index{cross section|see {section}}

As an example, let's say $U=\{a,b\}$. How many such $s$'s are there in
$\Fun{Sec}_f(U)$? To answer this, let's clip the picture \eqref{eqn.sections_of_function} and look only at the relevant part:
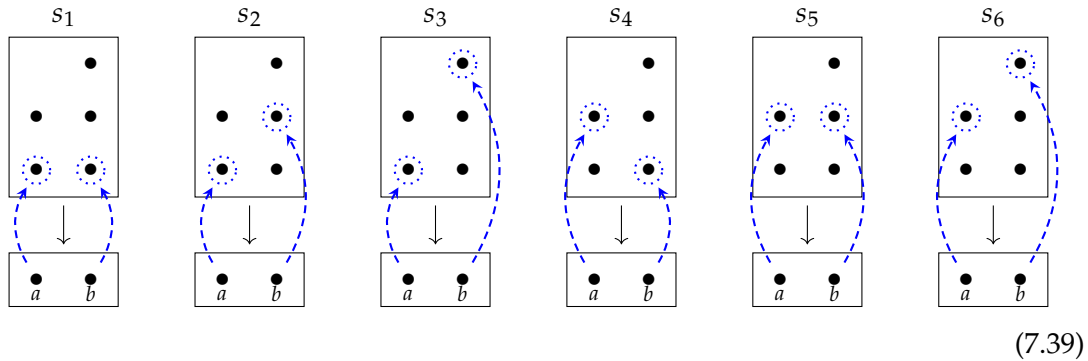
\begin{equation}%
\label{eqn.all_six_sections}
\begin{tikzpicture}[y=.25cm, x=.25cm, every label/.style={font=\scriptsize}, trim left=2cm]
	\node (Ya6) {};
	\foreach \i [remember=\i as \lasti (initially 6)] in {0,...,5} {
  	\node[label={[below=6.5pt]:$a$}, right=8 of Ya\lasti]   (Ya\i)  {$\bullet$};
  	\node[label={[below=5pt]:$b$}, right=1 of Ya\i]  (Yb\i)  {$\bullet$};
    \node[above=4 of Ya\i]  (Xa1\i) {$\bullet$};
    \node[above=1 of Xa1\i] (Xa2\i) {$\bullet$};
    \node[above=4 of Yb\i]  (Xb1\i) {$\bullet$};
    \node[above=1 of Xb1\i] (Xb2\i) {$\bullet$};
    \node[above=1 of Xb2\i] (Xb3\i) {$\bullet$};
    \node[draw, fit=(Ya\i) (Yb\i)] (Y\i) {};
    \node[draw, fit=(Xa1\i) (Xb3\i)] (X\i) {};
		\draw[->, shorten <=3pt, shorten >=3pt] (X\i) -- (Y\i);
		\tikzmath{
  		int \a, \b, \ii;
  		\a = 1+div(\i, 3);
  		\b = 1+mod(\i,3);
			\ii=1+\i;
		}
    \node[above=0 of X\i] {$s_{\ii}$};
		\begin{scope}[mapsto]
  		\draw[bend left] (Ya\i) to (Xa\a\i);
  		\draw[bend right] (Yb\i) to (Xb\b\i);
			\node[circle, dotted, draw] at (Xa\a\i) {};
			\node[circle, dotted, draw] at (Xb\b\i) {};
		\end{scope}
	}
\end{tikzpicture}
\end{equation}
Looking at the picture \eqref{eqn.all_six_sections}, do you see how we get all cross-sections of $f$ over $U$? 

\begin{exercise}%
\label{exc.presheaf_ex_cont}
Refer to \cref{eqn.sections_of_function}.
\begin{enumerate}
	\item Let $V_1=\{a,b,c\}$. Draw all the sections over it, i.e.\ all elements of $\Fun{Sec}_f(V_1)$, as we did in \cref{eqn.all_six_sections}.
	\item Let $V_2=\{a,b,c,d\}$. Again draw all the sections, $\Fun{Sec}_f(V_2)$.
	\item Let $V_3=\{a,b,d,e\}$. How many sections (elements of $\Fun{Sec}_f(V_3)$) are there?
\qedhere
\end{enumerate}
\end{exercise}

By now you should understand the sections of $\Fun{Sec}_f(U)$ for various $U\ss X$. This is $\Fun{Sec}_f$ on objects, so you are half way to understanding $\Fun{Sec}_f$ as a presheaf. That is, as a presheaf, $\Fun{Sec}_f$ also includes a restriction maps for every subset $V\ss U$. Luckily, the restriction maps are easy: if $V\ss U$, say $V=\{a\}$ and $U=\{a,b\}$, then given a section $s$ as in \cref{eqn.all_six_sections}, we get a section over $V$ by `restricting' our attention to what $s$ does on $\{a\}$. 
\begin{equation}%
\label{eqn.restriction_map_rand9384}
\begin{tikzpicture}[y=.25cm, x=.75cm, every label/.style={font=\scriptsize}, trim left=4.5cm]
	\node (Ya6) {};
	\foreach \i [remember=\i as \lasti (initially 6)] in {0,...,1} {
  	\node[label={[below=6.5pt]:$a$}, right=8 of Ya\lasti]   (Ya\i)  {$\bullet$};
    \node[above=4 of Ya\i]  (Xa1\i) {$\bullet$};
    \node[above=1 of Xa1\i] (Xa2\i) {$\bullet$};
    \node[draw, fit=(Ya\i)] (Y\i) {};
    \node[draw, fit=(Xa1\i) (Xa2\i)] (X\i) {};
		\draw[->, shorten <=3pt, shorten >=3pt] (X\i) -- (Y\i);
		\tikzmath{
  		int \ii;
			\ii=1+\i;
		}
		\begin{scope}[mapsto]
  		\draw[ bend left] (Ya\i) to (Xa\ii\i);
			\node[circle, dotted, draw] at (Xa\ii\i) {};
		\end{scope}
	}
	\node[above=0 of X0] {$\restrict{s_1}{V}=\restrict{s_2}{V}=\restrict{s_3}{V}$};
	\node[above=0 of X1] {$\restrict{s_4}{V}=\restrict{s_5}{V}=\restrict{s_6}{V}$};
\end{tikzpicture}
\end{equation}

\begin{exercise}%
\label{exc.section_practice}
\begin{enumerate}
	\item Write out the sets of sections $\Fun{Sec}_f(\{a,b,c\})$ and $\Fun{Sec}_f(\{a,c\})$.
	\item Draw lines from the first to the second to indicate the restriction map.
\qedhere
\end{enumerate}
\end{exercise}

Now we have understood $\Fun{Sec}_f$ as a presheaf; we next explain how to see that it is a sheaf, i.e.\ that it satisfies the sheaf condition for every cover. To understand the sheaf condition, consider the set $U_1=\{a,b\}$ and $U_2=\{b,e\}$. These cover the set $U=\{a,b,e\}=U_1\cup U_2$. By \cref{def.sheaf}, a matching family for this cover consists of a section over $U_1$ and a section over $U_2$ that agree on the overlap set, $U_1\cap U_2=\{b\}$.

So consider $s_1\in \Fun{Sec}_f(U_1)$ and $s_2\in \Fun{Sec}_f(U_2)$ shown below.
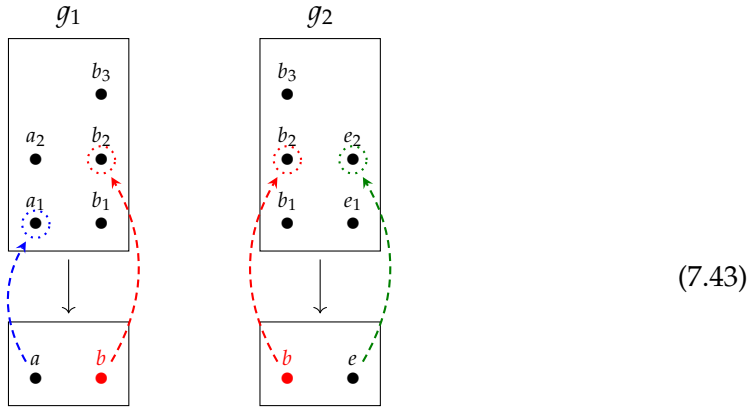
\begin{equation}%
\label{eqn.one_section}

\end{equation}
Since sections $g_1$ and $g_2$ agree on the overlap---they both send $b$ to
$b_2$---the two sections shown in \cref{eqn.one_section} can be glued to form a single section over
$U=\{a,b,e\}$:
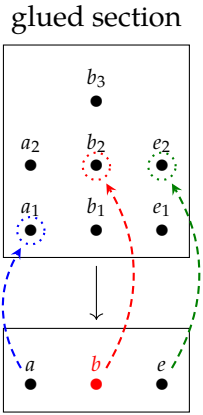
\[
%
\qedhere
\]

\begin{exercise}%
\label{exc.sections_agree_overlap}
Again let $U_1=\{a,b\}$ and $U_2=\{b,e\}$, so the overlap is $U_1\cap U_2=\{b\}$.
\begin{enumerate}
	\item Find a section $s_1\in\Fun{Sec}_f(U_1)$ and a section $s_2\in\Fun{Sec}_f(U_2)$ that \emph{do not} agree on the overlap.
	\item For your answer ($s_1,s_2)$ in part 1, can you find a section $s\in\Fun{Sec}_f(U_1\cup U_2)$ such that $\restrict{s}{U_1}=s_1$ and $\restrict{s}{U_2}=s_2$?
	\item Find a section $h_1\in\Fun{Sec}_f(U_1)$ and a section $h_2\in\Fun{Sec}_f(U_2)$ that \emph{do} agree on the overlap, but which are different than our choice in \cref{eqn.one_section}.
	\item Can you find a section $h\in\Fun{Sec}_f(U_1\cup U_2)$ such that $\restrict{h}{U_1}=h_1$ and $\restrict{h}{U_2}=h_2$?
\qedhere
\end{enumerate}
\end{exercise}

\index{topological space|)}

\paragraph{Other examples of sheaves.}

The extended example above generalizes to any continuous function $f\colon X\to Y$ between
topological spaces.

\begin{example}%
\label{ex.sheaf_from_cont_function}%
\index{sheaf!of sections of a function}
Let $f\colon(X,\Op_X)\to(Y,\Op_Y)$ be a continuous function. Consider the functor $\Fun{Sec}_f\colon\Op_Y\op\to\smset$ given by
\[
\Fun{Sec}_f(U)\coloneqq\{g\colon U\to X\mid g\text{ is continuous and
}(g\cp f)(u)=u\text{ for all }u\in U\},
\]
The morphisms of $\Op_Y$ are inclusions $V\ss U$. Given $g\colon U\to X$ and $V\ss U$, what we call the restriction of $g$ to $V$ is the usual thing we mean by restriction, the same as it was in \cref{eqn.restriction_map_rand9384}. One can again check that $\Fun{Sec}_f$ is a sheaf.
\end{example}

\begin{example}%
\index{sheaf!of vector fields}%
\label{ex.tangent_bundle}%
\index{tangent bundle}%
\index{vector!tangent}%
\index{vector field}
A nice example of a sheaf on a space $M$ is that of vector fields on $M$. If you
calculate the wind velocity at every point on Earth, you will have what's called
a vector field on Earth. If you know the wind velocity at every point in
Afghanistan and I know the wind velocity at every point in Pakistan, and our
calculations agree around the border, then we can glue our information together
to get the wind velocity over the union of the two countries. All possible wind
velocity fields over all possible open sets of the Earth's surface together form
the sheaf of vector fields. 

Let's say this a bit more formally. A manifold $M$---you can just imagine a sphere such as the Earth's surface---always has something called a tangent bundle. It is a space $TM$ whose points are pairs $(m,v)$, where $m\in M$ is a point in the manifold and $v$ is a tangent vector emanating from it. Here's a picture of one tangent plane---all the tangent vectors emanating from some fixed point---on a sphere:
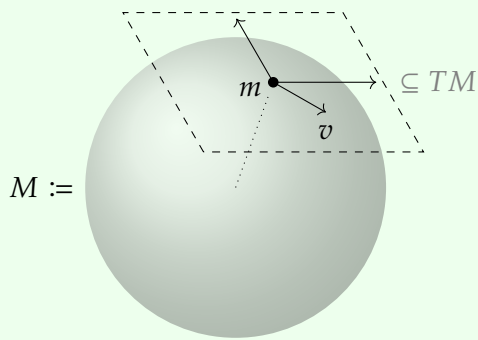
\[
\begin{tikzpicture}
  \shade[ball color = gray!40, opacity = 0.4] (0,0) circle (2cm);
  \node[label={[below left=0 and 0]:$m$}] at (.5, 1.4) (m) {};
  \fill[fill=black] (m) circle (2pt);
  \node[trapezium, dashed, draw, trapezium left angle=120, trapezium right angle=60, minimum width=4cm] at (m) (trap) {};
  \draw[dotted, thin] (0,0) -- (m);
  \draw[->, thin, shorten >=.1cm] (m.center) -- (trap.top side);
  \draw[->, thin, shorten >=.1cm] (m.center) -- (trap.right side);
  \node[label={[below=5pt]:$v$}] at ($(m)+(.7,-.4)$) (v) {};
  \draw[->, thin] (m.center) -- (v.center);
  \node[left] at (-2cm,0) {$M\coloneqq$};
  \node[right=2pt of trap, gray] {$\ss TM$};
\end{tikzpicture}
\]
The tangent bundle $TM$ includes the whole tangent plane shown above---including the three vectors drawn on it---as well as the tangent plane at every other point on the sphere. 

The tangent bundle $TM$ on a manifold $M$ comes with a continuous map $\pi\colon TM\to M$ back down to the manifold, sending $(m,v)\mapsto m$. One might say that $\pi$ ``forgets the tangent vector and just remembers the point it emanated from.'' By \cref{ex.sheaf_from_cont_function}, $\pi$ defines a sheaf $\Fun{Sec}_\pi$. It could be called the sheaf of `tangent vector sections on $M$', but its usual name is the sheaf of \emph{vector fields on $M$}. This is what we were describing when we spoke of the sheaf of wind velocities on Earth, above. Given an open subset $U\ss M$, an element $v\in\Fun{Sec}_\pi(U)$ is called a vector field over $U$ because it continuously assigns a tangent vector $v(u)$ to each point $u\in U$. The tangent vector at $u$ tells us the velocity of the wind at that point.%
\index{tangent bundle!vector fields as sections of}

Here's a fun digression: in the case of a spherical manifold $M$ like the Earth,
it's possible to prove that for every open set $U$, as long as $U\neq M$, there
is a vector field $v\in\Fun{Sec}_\pi(U)$ that is never 0: the wind could be
blowing throughout $U$. However, a theorem of Poincar\'e says that if you look
at the whole sphere, there is guaranteed to be a point $m\in M$ at which the
wind is not blowing at all. It's like the eye of a hurricane or perhaps a
cowlick. A cowlick in someone's hair occurs when the hair has no direction to
go, so it sticks up! Hair sticking up would not count as a tangent vector:
tangent vectors must start out lying flat along the head. Poincar\'{e} proved
that if your head was covered completely with inch-long hair, there would be at
least one cowlick. This difference between local sections (over arbitrary $U\ss
X$) and global sections (over $X$)---namely that hair can be well combed
whenever $U\neq X$ but cannot be well combed when $U=X$---can be thought of as a
generative effect, and can be measured by cohomology (see
\cref{ch1.further_reading}).%
\index{cowlick!inevitability of}
\end{example}%
\index{generative effect}

\begin{exercise}%
\label{exc.whats_the_sheaf}
If $M$ is a sphere as in \cref{ex.tangent_bundle}, we know from \cref{def.sheaf} that we can consider the category $\Shv(M)$ of sheaves on $M$; in fact, such categories are toposes and these are what we're getting to.

But are the sheaves on $M$ the vector fields? That is, is there a one-to-one correspondence between sheaves on $M$ and vector fields on $M$? If so, why? If not, how are sheaves on $M$ and vector fields on $M$ related?
\end{exercise}

\begin{example}%
\label{ex.Sets_as_sheaves}%
\index{set!as sheaf on one-point space}
For every topological space $(X,\Op)$, we have the topos of sheaves on it. The topos of sets, which one can regard as the story of set theory, is the category of sheaves on the one-point space $\{*\}$. In topos theory, we see the category of sets---an huge, amazing, and rich category---as corresponding to a single point. Imagine how much more complex arbitrary toposes are, when they can take place on much more interesting topological spaces (and in fact even more general `sites'). 
\end{example}

\begin{exercise}%
\label{exc.sierpinski}%
\index{Sierpinski space}
Consider the Sierpinski space $(\{1,2\},\Op_1)$ from \cref{ex.Sierpinski}.
\begin{enumerate}
	\item What is the category $\Op$ for this space? (You may have already figured this out in \cref{exc.opens_Sierp}; if not, do so now.)
	\item What does a presheaf on $\Op$ consist of?
	\item What is the sheaf condition for $\Op$?
	\item How do we identify a sheaf on $\Op$ with a function?
	\qedhere
\end{enumerate}
\end{exercise}

\index{sheaf|)}


\section{Toposes} %
\label{subsec.topos_logic} %
\label{subsec.def_and_properties_toposes}
\index{topos|(}

A \emph{topos} is defined to be a category of sheaves.%
\footnote{This is sometimes called a \emph{sheaf topos} or a \emph{Grothendieck topos}. There is a more general sort of topos called an \emph{elementary topos} due to Lawvere.}%
\index{topos!as category of sheaves}%
\index{sheaves!topos of}
So for any topological space $(X,\Op)$, the category $\Shv(X,\Op)$ defined in
\cref{def.sheaf} is a topos. In particular, taking the one-point space $X =\Cat{1}$ with its unique topology, we find that the category
$\smset$ is a topos, as we've been saying all along and saw again explicitly in \cref{ex.Sets_as_sheaves}. And for any database schema---i.e.\ finitely presented category---$\cat{C}$, the category $\cat{C}\inst$ of database instances on $\cat{C}$ is also a topos.%
\footnote{We said that a topos is a category of sheaves, yet database instances are presheaves; so how is $\cat{C}\inst$ a topos? Well, presheaves in fact count as sheaves. We apologize that this couldn't be clearer. All of this could be made formal if we were to introduce \emph{sites}. Unfortunately, that concept is simply too abstract for the scope of this chapter.}
Toposes encompass both of these sources of examples, and many more.

Toposes are incredibly nice structures, for a variety of seemingly disparate
reasons. In this sketch, the reason in focus is that every topos has many of the
same structural properties that the category $\smset$ has.  Indeed, we
discussed in \cref{subsec.set_like} that every topos has limits and colimits, is
cartesian closed, has epi-mono factorizations, and has a subobject classifier
(see \cref{subsec.subobj_classifier}). Using these properties, one can do logic
with semantics in the topos $\cat{E}$. We explained this for sets, but now
imagine it for sheaves on a topological space. There, the same logical symbols
$\wedge, \vee, \neg,\imp,\exists,\forall$ become operations that mean something
about sub-sheaves---e.g.\ vector fields, sections of continuous functions,
etc.---not just subsets.%
\index{logic!in a topos}%
\index{epi-mono factorization}

To understand this more deeply, we should say what the subobject classifier
$\true\colon 1 \to \Omega$ is in more generality.  We said that, in the topos $\smset$, the  subobject classifier is the set of booleans
$\Omega=\BB$. In a sheaf topos $\cat{E}=\Shv(X,\Op)$, the object $\Omega\in\cat{E}$ is
a sheaf, not just a set. What sheaf is it? 

\subsection{The subobject classifier $\Omega$ in a sheaf
topos}%
\label{subsec.subob_class_sheaf_topos}%
\index{subobject classifier|(}

In this subsection we aim to understand the subobject classifier $\Omega$, i.e.\
the object of truth values, in the sheaf topos $\Shv(X,\Op)$. Since $\Omega$ is
a sheaf, let's understand it by going through the definition of sheaf
(\cref{def.sheaf}) slowly in this case. A sheaf $\Omega$ is a presheaf that
satisfies the sheaf condition. As a presheaf it is just a functor
$\Omega\colon\Op\op\to\smset$; it assigns a set $\Omega(U)$ to each open
$U\ss X$ and comes with a restriction map $\Omega(U)\to\Omega(V)$ whenever $V\ss
U$. So in our quest to understand $\Omega$, we first ask the question: what
presheaf is it?

The answer to our question is that $\Omega$ is the presheaf that assigns to
$U\in\Op$ the set of open subsets of $U$:
\begin{equation}%
\label{eqn.omega_as_opens}
  \Omega(U)\coloneqq\{U'\in\Op\mid U'\ss U\}.
\end{equation}
That was easy, right? And given the restriction map for $V\ss U$ is given by
\begin{align}%
\label{eqn.omega_open_rest}
	\Omega(U)&\to\Omega(V)\\\nonumber	
	U'&\mapsto U'\cap V.
\end{align}
One can check that this is
functorial---see \cref{exc.Omega_functorial}---and after doing so we will still
need to see that it satisfies the sheaf condition. But at least we don't have to
struggle to understand $\Omega$: it's a lot like $\Op$ itself.

\begin{exercise}%
\label{exc.booleans_as_subspace_1}%
\index{booleans!as subobject classifier}
Let $X=\{1\}$ be the one point space. We said above that its subobject classifier is the set $\bb$ of booleans, but how does that align with the definition of $\Omega$ given in \cref{eqn.omega_as_opens}?
\end{exercise}

\begin{exercise}%
\label{exc.Omega_functorial}~
\begin{enumerate}
	\item Show that the definition of $\Omega$ given above in \cref{eqn.omega_as_opens,eqn.omega_open_rest} is functorial, i.e., that whenever $W\ss V\ss U$, the restriction map $\Omega(U)\to\Omega(V)$ followed by the restriction map $\Omega(V)\to\Omega(W)$ is the same as the restriction map $\Omega(U)\to\Omega(W)$.
	\item Is that all that's necessary to conclude that $\Omega$ is a presheaf?
\qedhere
\end{enumerate}
\end{exercise}

To see that $\Omega$ as defined in \cref{eqn.omega_as_opens} satisfies the sheaf condition (see \cref{def.sheaf}), suppose that we have a cover
$U=\bigcup_{i\in I}U_i$, and suppose given an element $V_i\in\Omega(U_i)$, i.e.\
an open set $V_i\ss U_i$, for each $i\in I$. Suppose further that for all
$i,j\in I$, it is the case that $V_i\cap U_j=V_j\cap U_i$, i.e.\ that the
elements form a matching family. Define
$V\coloneqq\bigcup_{i\in I}V_i$; it is an open subset of $U$, so we can consider
$V$ as an element of $\Omega(U)$. The following verifies that $V$ is indeed a gluing for the $(V_i)_{i\in I}$:
\[V\cap U_j=\left(\bigcup_{i\in I}V_i\right)\cap U_j=\bigcup_{i\in I}(V_i\cap U_j)=\bigcup_{i\in I}(V_j\cap U_i)=\left(\bigcup_{i\in I}U_i\right)\cap V_j=V_j\]
In other words $V\cap U_j=V_j$ for any $j\in I$. So our $\Omega$ has been upgraded from presheaf to sheaf!

The eagle-eyed reader will have noticed that we haven't yet given all the data
needed to define a subobject classifier. To turn the object $\Omega$ into a subobject classifier in good standing, we also need to give a sheaf morphism
$\true\colon\{1\} \to \Omega$. Here $\{1\}\colon \Op\op \to \smset$ is the terminal
sheaf; it maps every open set to the terminal, one element set $\{1\}$. The correct morphism
$\true\colon \{1\} \to \Omega$ for the subobject classifier is the sheaf morphism that assigns, for
every $U \in \Op$ the function $\{1\}=\{1\}(U)\to\Omega(U)$ sending $1\mapsto U$, the largest open set $U\ss U$. From now on we denote $\{1\}$ simply as $1$.

\paragraph{Upshot: Truth values are open sets.}

The point is that the truth values in the topos of sheaves on a space $(X,\Op)$ are the open sets of that space. When someone says ``is property $P$ true?,'' the answer is not yes or no, but ``it is true on the open subset $U$.'' If this $U$ is everything, $U=X$, then $P$ is really true; if $U$ is nothing, $U=\varnothing$, then $P$ is really false. But in general, it's just true some places and not others.

\begin{example}%
\label{ex.subobject_classifier_graphs}
The category $\Cat{Grph}$ of graphs is a presheaf topos, and one can also think of it as the category of instances for a database schema, as we saw in \cref{ex.graph_presheaf_topos}. The subobject classifier $\Omega$ in the topos $\Cat{Gr}$ is thus a graph, so we can draw it. Here's what it looks like:
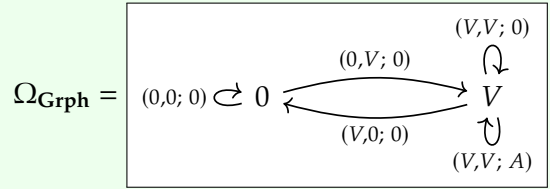
\[
\Omega_{\Cat{Grph}}=\boxCD{
\begin{tikzcd}[column sep=70pt, ampersand replacement=\&]
	0
		\ar[loop left, "{(0,0;\ 0)}"]
		\ar[r, bend left=15pt, "{(0, V;\ 0)}"]
\&
	V
		\ar[loop above, "{(V, V;\ 0)}"]
		\ar[loop below, "{(V, V;\ A)}"]
		\ar[l, bend left=15pt, "{(V, 0;\ 0)}"]
\end{tikzcd}
}
\]
Finding $\Omega$ for oneself is easiest using something called the Yoneda Lemma, but we have not introduced it. For a nice, easy introduction to the topos of graphs, see \cite{vigna2003guided}. The terminal graph is a single vertex with a single loop, and the graph homomorphism $\true\colon 1\to\Omega$ sends that loop to $(V,V;\ A)$.

Given any graph $G$ and subgraph $i\colon H\ss G$, we need to construct a graph homomorphism $\corners{H}\colon G\to \Omega$ classifying $H$. The idea is that for each part of $G$, we decide ``how much of it is in $H$. A vertex in $v$ in $G$ is either in $H$ or not; if so we send it to $V$ and if not we send it to $0$. But arrows $a$ are more complicated. If $a$ is in $H$, we send it $(V,V;A)$. But if it is not in $H$, the mathematics requires us to ask more questions: is its source in $H$? is its target in $G"$? both? neither? Based on the answers to these questions we send $a$ to $(V, 0;\ 0)$, $(0, V;\ 0)$, $(V, V;\ 0)$, or $(0, 0;\ 0)$, respectively.
\end{example}

\begin{exercise}%
\label{exc.classify_subgraph}
Consider the subgraph $H\ss G$ shown here:
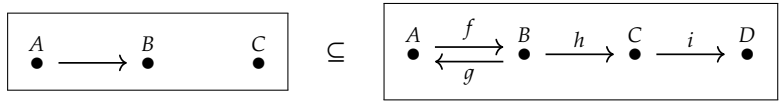
\[
\boxCD{
\begin{tikzcd}[ampersand replacement=\&]
	\LMO{A}\ar[r]\&\LMO{B}\&\LMO{C}
\end{tikzcd}
}
\quad\ss\quad
\boxCD{
\begin{tikzcd}[ampersand replacement=\&]
	\LMO{A}\ar[r, shift left, "f"]\&\LMO{B}\ar[l, shift left, "g"]\ar[r, "h"]\&\LMO{C}\ar[r, "i"]\&\LMO{D}
\end{tikzcd}
}
\]
Find the graph homomorphism $\corners{H}\colon G\to\Omega$ classifying it. See \cref{ex.subobject_classifier_graphs}.
\end{exercise}

\subsection{Logic in a sheaf topos}%
\label{subsec.logic_sheaf_topos}
\index{AND operation}%
\index{OR operation}
Let's consider the logical connectives, AND, OR, IMPLIES, and NOT. Suppose we have a topological space  $X\in\Op$. Given two open sets $U,V$, considered as truth values $U,V\in\Omega(X)$, then their conjunction `$U$ AND $V$' is their intersection, and their disjunction `$U$ OR $V$' is their union;
\begin{equation}%
\label{eqn.AND_OR}
  (U\wedge V)\coloneqq U\cap V
  \qquad\text{and}\qquad
  (U\vee V)\coloneqq U\cup V.
\end{equation}
These formulas are easy to remember, because $\wedge$ looks like $\cap$ and
$\vee$ looks like $\cup$. The implication $U\imp V$ is the largest open set $R$ such that $R\cap U\ss V$, i.e.
\begin{equation}%
\label{eqn.implies_logic}
  (U\imp V)\coloneqq\bigcup_{\{R\in\Op\mid R\cap U\ss V\}}R.
\end{equation}
In general, it is not easy to reduce \cref{eqn.implies_logic} further, so implication is the hardest logical connective to think about topologically.
\index{IMPLIES operation}

\index{NOT operation}
Finally, the negation of $U$ is given by $\neg U\coloneqq(U\imp\false)$, and this turns out to be relatively simple. By the formula in \cref{eqn.implies_logic}, it is the union of all $R$ such that $R\cap U=\varnothing$, i.e.\ the union of all open sets in the complement of $U$. If you know topology, you might recognize that $\neg U$ is the `interior of the complement of $U$.'

\begin{example}
Consider the real line $X=\RR$ as a topological space (see \cref{ex.usual_top_R}). Let $U,V\in\Omega(X)$ be the open sets $U=\{x\in\RR\mid x<3\}$ and $V=\{x\in\RR\mid -4< x < 4\}$. Using interval notation, $U=(-\infty,3)$ and $V=(-4,4)$. Then
\begin{itemize}
	\item $U\wedge V=(-4,3)$.
	\item $U\vee V=(-\infty,4)$.
	\item $\neg U=(3,\infty)$.
	\item $\neg V=(-\infty,-4)\cup(4,\infty)$.
	\item $(U\imp V)=(-4,\infty)$
	\item $(V\imp U)=U$
	\qedhere
\end{itemize}
\end{example}

\begin{exercise}%
\label{exc.real_line_logic}
Consider the real line $\RR$ as a topological space, and consider the open subset $U=\RR-\{0\}$.
\begin{enumerate}
	\item What open subset is $\neg U$?
	\item What open subset is $\neg\neg U$?
	\item Is it true that $U\ss\neg\neg U$?
	\item Is it true that $\neg\neg U\ss U$?
	\qedhere
\end{enumerate}	
\end{exercise}

Above we explained operations on open sets, one corresponding to each logical connective; there are also open sets corresponding to the the symbols $\true$ and $\false$. We explore this in an exercise.

\begin{exercise}%
\label{exc.top_bot_practice}
Let $(X,\Op)$ be a topological space.
\begin{enumerate}
	\item Suppose the symbol $\true$ corresponds to an open set such that for any open set $U\in\Op$, we have $(\true\wedge U)=U$. Which open set is it?
	\item Other things we should expect from $\true$ include $(\true\vee U)=\true$ and $(U\imp \true)=\true$ and $(\true\imp U)=U$. Do these hold for your answer to 1?
	\item The symbol $\false$ corresponds to an open set $U\in\Op$ such that for any open set $U\in\Op$, we have $(\false\vee U)=U$. Which open set is it?
	\item Other things we should expect from $\false$ include $(\false\wedge U)=\false$ and $(\false\imp U)=\true$. Do these hold for your answer to 1?
	\qedhere
\end{enumerate} 
\end{exercise} 

\begin{example}
For a vector bundle $\pi\colon E\to X$ over a space $X$, the corresponding sheaf is $\Fun{Sec}_\pi$ corresponding to its sections: to each open set $i_U\colon U\ss X$, we associate the set of functions $s\colon U\to E$ for which $s\cp\pi=i_U$. For example, in the case of the tangent bundle $\pi\colon TM\to M$ (see \cref{ex.tangent_bundle}), the corresponding sheaf, call it $\const{VF}$, associates to each $U$ the set $\const{VF}(U)$ of vector fields on $U$.

The internal logic of the topos can then be used to consider properties of vector fields. For example, one could have a predicate $\const{Grad}\colon\const{VF}\to\Omega$ that asks for the largest subspace $\const{Grad}(v)$ on which a given vector field $v$ comes from the gradient of some scalar function. One could also have a predicate that asks for the largest open set on which a vector field is non-zero. Logical operations like $\wedge$ and $\vee$ could then be applied to hone in on precise submanifolds throughout which various desired properties hold, and to reason logically about what other properties are forced to hold there.
\end{example}

\subsection{Predicates}%
\label{subsec.predicates}
\index{predicate}
In English, a predicate is the part of the sentence that comes after the subject. For example ``\dots is even'' or ``\dots likes the weather'' are predicates. Not every subject makes sense for a given predicate; e.g.\ the sentence ``7 is even'' may be false, but it makes sense. In contrast, the sentence ``2.7 is even'' does not really make sense, and ``2.7 likes the weather'' certainly doesn't. In computer science, they might say ``The expression `2.7 likes the weather' does not type check.'' 

The point is that each predicate is associated to a type, namely the type of subject that makes sense for that predicate. When we apply a predicate to a subject of the appropriate type, the result has a truth value: ``7 is even'' is either true or false. Perhaps ``Bob likes the weather'' is true some days and false on others. In fact, this truth value might change by the year (bad weather this year), by the season, by the hour, etc. In English, we expect truth values of sentences to change over time, which is exactly the motivation for this chapter. We're working toward a logic where truth values change over time.

In a topos $\cat{E}=\Shv(X,\Op)$, a predicate is a sheaf morphism $p\colon S\to\Omega$ where $S\in\cat{E}$ is a sheaf and $\Omega\in\cat{E}$ is the subobject classifier, the sheaf of truth values. By \cref{def.sheaf} we get a function $p(U)\colon S(U)\to\Omega(U)$ for any open set $U\ss X$. In the above example---which we will discuss more carefully in \cref{subsec.topos_behavior_types}---if $S$ is the sheaf of people (people come and go over time), and $\mathrm{Bob}\in S(U)$ is a person existing over a time $U$, and $p$ is the predicate ``likes the weather,'' then $p(\mathrm{Bob})$ is the set of times during which Bob likes the weather. So the answer to ``Bob likes the weather'' is something like ``in summers yes, and also in April 2018 and May 2019 yes, but in all other times no.'' That's $p(\mathrm{Bob})$, the temporal truth value obtained by applying the predicate $p$ to the subject Bob.

\begin{exercise}%
\label{exc.weather_bob}
Just now we described how a predicate $p\colon S\to\Omega$, such as ``\dots likes the weather,'' acts on sections $s\in S(U)$, say $s=\mathrm{Bob}$. But by \cref{def.subobject_classifier}, any predicate $p\colon S\to\Omega$ also defines a subobject of $\{S\mid p\}\ss S$. Describe the sections of this subsheaf.
\end{exercise}

\paragraph{The poset of subobjects.}
\index{preorder!of subobjects}

For a topos $\cat{E}=\Shv(X,\Op)$ and object (sheaf) $S\in\cat{E}$, the set of $S$-predicates $|\Omega^E|=\cat{E}(S,\Omega)$ is naturally given the structure of a poset, which we denote 
\begin{equation}%
\label{eqn.preorder_of_predicates}
	(|\Omega^S|, \leq^S)
\end{equation}
Given two predicates $p,q\colon S\to\Omega$, we say that $p\leq^S q$ if the first implies the second. More precisely, for any $U\in\Op$ and section $s\in S(U)$ we obtain two open subsets $p(s)\ss U$ and $q(s)\ss U$. We say that $p\leq^S q$ if $p(s)\ss q(s)$ for all $U\in\Op$ and $s\in S(U)$. We often drop the superscript from $\leq^S$ and simply write $\leq$. In formal logic notation, one might write $p\leq^Sq$ using the $\vdash$ symbol, e.g.\ in one of the following ways:
\[s:S\mid p(s)\vdash q(s)\qquad\text{or}\qquad p(s)\vdash_{s:S}q(s).\]
In particular, if $S=1$ is the terminal object, we denote $|\Omega^S|$ by $|\Omega|$, and refer to elements $p\in|\Omega|$ as \emph{propositions}. They are just morphisms $p\colon 1\to\Omega$.

This preorder is partially ordered---a poset---meaning that if $p\leq q$ and $q\leq p$ then $p=q$. The reason is that for any subsets $U,V\ss X$, if $U\ss V$ and $V\ss U$ then $U=V$.

\begin{exercise}%
\label{exc.vdash}
Give an example of a space $X$, a sheaf $S\in\Shv(X)$, and two predicates $p,q\colon S\to\Omega$ for which $p(s)\vdash_{s:S} q(s)$ holds. You do not have to be formal.
\end{exercise}

All of the logical symbols ($\true,\false,\wedge,\vee,\imp,\neg$) from
\cref{subsec.logic_sheaf_topos} make sense in any such poset $|\Omega^S|$. For
any two predicates $p,q\colon S\to\Omega$, we define $(p\wedge q)\colon
S\to\Omega$ by $(p\wedge q)(s)\coloneqq p(s)\wedge q(s)$, and similarly for
$\vee$. Thus one says that these operations are \emph{computed pointwise} on
$S$. With these definitions, the $\wedge$ symbol is the meet and the $\vee$
symbol is the join---in the sense of \cref{def.meets_joins}---for the poset
$|\Omega^S|$.

With all of the logical structure we've defined so far, the poset $|\Omega^S|$ of predicates on $S$ forms what's called a Heyting algebra. We will not define it here, but more information can be found in \cref{sec.C7_further_reading}. We now move on to quantification.

\subsection{Quantification}%
\label{subsec.quantification}
\index{quantification}

Quantification comes in two flavors: universal and existential, or `for all' and `there exists.' Each takes in a predicate of $n+1$ variables and returns a predicate of $n$ variables.

\begin{example}%
\label{ex.worried_pred}
Suppose we have two sheaves $S,T\in\Shv(X,\Op)$ and a predicate $p\colon S\times T\to\Omega$. Let's say $T$ represents what's considered newsworthy and $S$ is again the set of people. So for a subset of time $U$, a section $t\in T(U)$ is something that's considered newsworthy throughout the whole of $U$, and a section $s\in S(U)$ is a person that lasts throughout the whole of $U$. Let's imagine the predicate $p$ as ``$s$ is worried about $t$.'' Now recall from \cref{subsec.predicates} that a predicate $p$ does not simply return true or false; given a person $s$ and a news-item $t$, it returns a truth value corresponding to the subset of times on which $p(s,t)$ is true. 

``For all $t$ in $T$, \dots is worried about $t$'' is itself a predicate on just one variable, $S$, which we denote
\[\forall(t:T)\ldotp p(s,t).\]
Applying this predicate to a person $s$ returns the times when that person is worried about everything in the news. Similarly, ``there exists $t$ in $T$ such that $s$ is worried about $t$'' is also a predicate on $S$, which we denote $\exists(t:T)\ldotp p(s,t)$. If we apply this predicate to a person $s$, we get the times when person $s$ is worried about at least one thing in the news.
\end{example}

\begin{exercise}%
\label{exc.predicate_practice}
In the topos $\smset$, where $\Omega=\BB$, consider the predicate $p\colon\NN\times\ZZ\to\BB$ given by
\[
  p(n,z)=
  \begin{cases}
    \true&\tn{ if }n\leq|z|\\
    \false&\tn{ if }n>|z|.
  \end{cases}
\]
\begin{enumerate}
	\item What is the set of $n\in\NN$ for which the predicate $\forall(z:\ZZ)\ldotp p(n,z)$ holds?
	\item What is the set of $n\in\NN$ for which the predicate $\exists(z:\ZZ)\ldotp p(n,z)$ holds?
	\item What is the set of $z\in\ZZ$ for which the predicate $\forall(n:\NN)\ldotp p(n,z)$ holds?
	\item What is the set of $z\in\ZZ$ for which the predicate $\exists(n:\NN)\ldotp p(n,z)$ holds?	
\qedhere
\end{enumerate}
\end{exercise}

So given $p$, we have a universally- and an existentially-quantified predicate
$\forall(t:T)\ldotp p(s,t)$ and $\exists(t:T)\ldotp p(s,t)$ on $S$. How do we
formally understand them as sheaf morphisms $S\to\Omega$ or, equivalently, as
subsheaves of $S$?

\paragraph{Universal quantification.}

Given a predicate $p\colon S\times T\to\Omega$, the universally-quantified predicate $\forall(t:T)\ldotp p(s,t)$ takes a section $s\in S(U)$, for any open set $U$, and returns a certain open set $V\in\Omega(U)$. Namely, it returns the largest open set $V\ss U$ for which $p(\restrict{s}{V},t)=V$ holds for all $t\in T(V)$. 

\begin{exercise}%
\label{exc.worrying_news_universal}
Suppose $s$ is a person alive throughout the interval $U$. Apply the above definition to the example $p(s,t)=$ ``person $s$ is worried about news $t$'' from \cref{ex.worried_pred}. Here, $T(V)$ is the set of items that are in the news throughout the interval $V$.
\begin{enumerate}
  \item What open subset of $U$ is $\forall(t:T)\ldotp p(s,t)$ for a person $s$?
  \item Does it have the semantic meaning you'd expect, given the less formal description in \cref{subsec.quantification}?
\qedhere
\end{enumerate}
\end{exercise}

Abstractly speaking, the universally-quantified predicate corresponds to the subsheaf given by the following pullback:
\[
\begin{tikzcd}
	\forall_tp\ar[r]\ar[d, tail]&1\ar[d, tail, "\true^T"]\\
	S\ar[r,"p'"']&\Omega^T\ar[ul, phantom, pos=1, "\lrcorner"]
\end{tikzcd}
\]
where $p'\colon S\to\Omega^T$ is the currying of $S\times T\to\Omega$ and
$\true^T$ is the currying of the composite $1\times T\To{!} 1\To{\true}\Omega$.
See \cref{eqn.currying}.%
\index{currying}

\paragraph{Existential quantification.}

Given a predicate $p\colon S\times T\to\Omega$, the existentially quantified predicate $\exists(t:T)\ldotp p(s,t)$ takes a section $s\in S(U)$, for any open set $U$, and returns a certain open set $V\in\Omega(U)$, namely the union $V=\bigcup_iV_i$ of all the open sets $V_i$ for which there exists some $t_i\in T(V_i)$ satisfying $p(\restrict{s}{V_i},t_i)=V_i$. If the result is $U$ itself, you might be tempted to think ``ah, so there exists some $t\in T(U)$ satisfying $p(t)$,'' but that is not necessarily so. There is just a cover of $U=\bigcup U_i$ and local sections $t_i\in T(U_i)$, each satisfying $p$, as explained above. Thus the existential quantifier is doing a lot of work ``under the hood,'' taking coverings into account without displaying that fact in the notation.

\begin{exercise}%
\label{exc.worrying_news_existential}
Apply the above definition to the ``person $s$ is worried about news $t$'' predicate from \cref{ex.worried_pred}.
\begin{enumerate}
  \item What open set is $\exists(t:T)\ldotp p(s,t)$ for a person $s$?
  \item Does it have the semantic meaning you'd expect?
\qedhere
\end{enumerate}
\end{exercise}

Abstractly speaking, the existentially-quantified predicate is given as follows. Start with the subobject classified by $p$, namely $\{(s,t)\in S\times T\mid p(s,t)\}\ss S\times T$, compose with the projection $\pi_S\colon S\times T\to S$ as on the upper right; then take the epi-mono factorization of the composite as on the lower left:%
\index{epi-mono factorization!and existential quantification}
\[
\begin{tikzcd}
  \{S\times T\mid p\}\ar[tail,r]\ar[d, two heads]& S\times T\ar[d,"\pi_S"]\\
  \exists_tp\ar[r, tail]&S
\end{tikzcd}
\]
Then the bottom map is the desired subsheaf of $S$.

\subsection{Modalities}%
\label{subsec.modalities}%
\index{modal operator}

Back in \cref{ex.modal_operator} we discussed modal operators---also known as
modalities---saying they are closure operators on preorders which arise in
logic. The preorders we were referring to are the ones discussed in
\cref{eqn.preorder_of_predicates}: for any object $S\in\cat{E}$ there is the poset
$(|\Omega^S|,\leq^S)$ of predicates on $S$, where $|\Omega^S|=\cat{E}(S,\Omega)$ is just the set of morphisms $S\to\Omega$ in the category $\cat{E}$.%
\index{closure operator}

\begin{definition}%
\label{def.modality}
A \emph{modality} in $\Shv(X)$ is a sheaf morphism $j\colon\Omega\to\Omega$ satisfying three properties for all $U\ss X$ and $p,q\in\Omega(U)$:
\begin{enumerate}[label=(\alph*)]
	\item $p\leq j(p)$; 
	\item $(j\cp j)(p)\leq j(p)$;	and
	\item $j(p\wedge q)=j(p)\wedge j(q)$.
\end{enumerate}
\end{definition}

\begin{exercise}%
\label{exc.two_defs_of_closure}
Suppose $j\colon\Omega\to\Omega$ is a morphism of sheaves on $X$, such that $p\leq j(p)$ holds for all $U\ss X$ and $p\in\Omega(U)$. Show that for all $q\in\Omega(U)$ we have $j(j(q))\leq j(q)$ iff $j(j(q))=j(q)$.
\end{exercise}



In \cref{ex.modal_operator} we informally said that for any proposition $p$, e.g. ``Bob is in San Diego,'' there is a modal operator ``assuming $p$, ....'' Now we are in a position to make that formal.

\begin{proposition}%
\label{prop.open_closed_quasiclosed}
Fix a proposition $p\in|\Omega|$. Then
\begin{enumerate}[label=(\alph*)]
	\item the sheaf morphism $\Omega\to\Omega$ given by sending $q$ to $p\imp q$ is a modality.
	\item the sheaf morphism $\Omega\to\Omega$ given by sending $q$ to $p\vee q$ is a modality.
	\item the sheaf morphism $\Omega\to\Omega$ given by sending $q$ to $(q\imp p)\imp p$ is a modality.
\end{enumerate}
\end{proposition}

We cannot prove \cref{prop.open_closed_quasiclosed} here, but we give references in \cref{sec.C7_further_reading}.

\begin{exercise}%
\label{exc.check_modality}
Let $S$ be the sheaf of people as in \cref{subsec.predicates}, and let $j\colon\Omega\to\Omega$ be ``assuming Bob is in San Diego...'' 
\begin{enumerate}
	\item Name any predicate $p\colon S\to\Omega$, such as ``likes the weather.''
	\item Choose a time interval $U$. For an arbitrary person $s\in S(U)$, what sort of thing is $p(s)$, and what does it mean?
	\item What sort of thing is $j(p(s))$ and what does it mean?
	\item Is it true that $p(s)\leq j(p(s))$? Explain briefly.
	\item Is it true that $j(j(p(s)))=j(p(s))$? Explain briefly.
	\item Choose another predicate $q\colon S\to\Omega$. Is it true that $j(p\wedge q)=j(p)\wedge j(q)$? Explain briefly.
\qedhere
\end{enumerate}
\end{exercise}

\index{subobject classifier|)}

\subsection{Type theories and semantics}
\label{subsec.type_th_and_semantics}%
\index{type theory}%
\index{semantics}

We have been talking about the logic of a topos in terms of open sets, but this is actually a conflation of two ideas that are really better left unconflated. The first is logic, or formal language, and the second is semantics, or meaning. The formal language looks like this:
\begin{equation}%
\label{eqn.logic_surjective}
	\forall(t:T)\ldotp\exists(s:S)\ldotp f(s)=t
\end{equation}
and semantic statements are like ``the sheaf morphism $f\colon S\to T$ is an
epimorphism.'' In the former, logical world, all statements are linguistic expressions formed according to strict
rules and all proofs are deductions that also follow strict rules. In the
latter, semantic world, statements and proofs are about the sheaves themselves, as mathematical objects. We admit these are rough
statements; again, our aim here is only to give a taste, an invitation to further reading.

To \emph{provide semantics} for a logical system means to provide a compiler that converts each logical statement in the formal language into a mathematical statement about particular sheaves and their relationships. A computer can carry out logical deductions without knowing what any of them ``mean'' about sheaves. We say that semantics is \emph{sound} if every formal proof is converted into a true fact about the relevant sheaves.%
\index{semantics!sound}

Every topos can be assigned a formal language, often called its \emph{internal
language}, in which to carry out constructions and formal proofs. This language
has a sound semantics---a sort of logic-to-sheaf compiler---which goes under the name \emph{categorical semantics} or \emph{Kripke-Joyal semantics}. We gave the basic ideas in \cref{subsec.topos_logic}; we give references to the literature in \cref{sec.C7_further_reading}.%
\index{language!internal}

\begin{example}
In every topos $\cat{E}$, and for every $f\colon S\to T$ in $\cat{E}$, the morphism $f$ is an epimorphism if and only if \cref{eqn.logic_surjective} holds. For example, consider the case of database instances on a schema $\cat{C}$, say with 100 tables (one of which might be denoted $c\in\Ob(\cat{C})$) and 500 foreign key columns (one of which might be denoted $f\colon c\to c'$ in $\cat{C}$); see \cref{eqn.free_schema}.

If $S$ and $T$ are two instances and $f$ is a natural transformation between them, then we can ask the question of whether or not \cref{eqn.logic_surjective} holds. This simple formula is compiled by the Kripke-Joyal semantics into asking:
\begin{quote}
	Is it true that for every table $c\in\Ob(\cat{C})$ and every row $s\in S(c)$ there exists a row $t\in T(c)$ such that $f(s)=t$?
\end{quote}
This is exactly what it means for $f$ to be surjective. Maybe this is not too impressive, but whether one is talking about databases or topological spaces, or complex ideas from algebraic geometry, \cref{eqn.logic_surjective} always compiles into the question of surjectivity. For topological spaces it would say something like:
\begin{quote}
	Is it true that for every open set $U\ss X$ and every section $s\in S(U)$ of the bundle $S$, there exists an open covering of $(U_i\ss U)_{i\in I}$ of $U$ and a section $t_i\in T(U_i)$ of the bundle $T$ for each $i\in I$, such that $f(t_i)=\restrict{s}{U_i}$ is the restriction of $s$ to $U_i$?
\end{quote}
\end{example}

\section{A topos of behavior types}%
\label{subsec.topos_behavior_types}%
\index{behavior!topos for}

Now that we have discussed logic in a sheaf topos, we return to our motivating example, a topos of behavior types. We begin by discussing the topological space on which behavior types will be sheaves, a space called the \emph{interval domain}.%

\begin{remark}
Note that above, we were thinking very intuitively about time, e.g.\ when we discussed people being worried about the news. Now we will be thinking about time in a different way, but there is no need to change your answers or reconsider the intuitive thinking done above.
\end{remark}

\subsection{The interval domain}%
\label{subsec.IR}
\index{interval domain, $\IR$|(}

The interval domain $\IR$ is a specific topological space, which we will use to model intervals of time. In other words, we will be interested in the category $\Shv(\IR)$ of sheaves on the interval domain.%
\index{topological space}

To give a topological space, one must give a pair $(X,\Op)$, where $X$ is a set of `points' and $\Op$ is a topology on $X$; see \cref{def.topological_space}. The set of points for $\IR$ is that of all finite closed intervals
\[\IR\coloneqq\{[d,u]\ss\RR\mid d\leq u\}.\]
For $a<b$ in $\RR$, let $o_{[a,b]}$ denote the set $o_{[a,b]}\coloneqq\{[d,u]\in\IR\mid a<d\leq u<b\}$; these are called \emph{basic open sets}. The topology $\Op$ is determined by these basic open sets in that a subset $U$ is open if it is the union of some collection of basic open sets.

Thus for example, $o_{[0,5]}$ is an open set: it contains every $[d,u]$ contained in the open interval $\{x\in\RR\mid 0<x<5\}$. Similarly $o_{[4,8]}$ is an open set, but note that $o_{[0,5]}\cup o_{[4,8]}\neq o_{[0,8]}$. Indeed, the interval $[2,6]$ is in the right-hand side but not the left.

\begin{exercise}%
\label{exc.explain_intervals}
\begin{enumerate}
	\item Explain why $[2,6]\in o_{[0,8]}$.
	\item Explain why $[2,6]\not\in o_{[0,5]}\cup o_{[4,8]}$.
\qedhere
\end{enumerate}
\end{exercise}

Let $\Op$ denote the open sets of $\IR$, as described above, and let $\Cat{BT}\coloneqq\Shv(\IR,\Op)$ denote the topos of sheaves on this space. We call it the topos of \emph{behavior types}.

There is an important subspace of $\IR$, namely the usual space of real numbers $\RR$. We see $\RR$ as a subspace of $\IR$ via the isomorphism
\[\RR\cong\{[d,u]\in\IR\mid d=u\}.\]
We discussed the usual topology on $\RR$ in \cref{ex.usual_R}, but we also get a topology on $\RR$ because it is a subset of $\IR$; i.e.\ we have the subspace topology as described in \cref{exc.subspace_topology}. These agree, as the reader can check.

\begin{exercise}%
\label{exc.R_subsp_IR}
Show that a subset $U\ss\RR$ is open in the subspace topology of $\RR\ss\IR$ iff $U\cap\RR$ is open in the usual topology on $\RR$ defined in \cref{ex.usual_R}.
\end{exercise}

\subsection{Sheaves on $\IR$}

We cannot go into much depth about the sheaf topos $\Cat{BT}=\Shv(\IR,\Op)$, for reasons of space; we refer the interested reader to \cref{sec.C7_further_reading}. In this section we will briefly discuss what it means to be a sheaf on $\IR$, giving a few examples including that of the subobject classifier.

\paragraph{What is a sheaf on $\IR$?}%
\index{sheaf!on $\IR$ as semantics of behavior}

A sheaf $S$ on the interval domain $(\IR,\Op)$ is a functor $S\colon\Op\op\to\smset$:
it assigns to each open set $U$ a set $S(U)$; how should we interpret this? An
element $s\in S(U)$ is something that says is an ``event that takes place throughout the interval $U$.'' Given this $U$-event $s$ together with an open subset of
$V\ss U$, there is a $V$-event $\restrict{s}{V}$ that tells us what $s$ is if we regard it as an event taking place
throughout $V$. If $U=\bigcup_{i\in I}U_i$ and we can find matching $U_i$-events $(s_i)$ for each $i\in I$, then the sheaf condition (\cref{def.sheaf}) says that they have a unique gluing, i.e.\ a $U$-event $s\in S(U)$ that encompasses all of them: $\restrict{s}{U_i}=s_i$ for each $i\in I$.

We said in \cref{subsec.IR} that every open set $U\ss\IR$ can be written as the union of basic open sets $o_{[a,b]}$. This implies that any sheaf $S$ is determined by its values $S(o_{[a,b]})$ on these basic open sets. The sheaf condition furthermore implies that these vary continuously in a certain sense, which we can express formally as
\[S(o_{[a,b]})\cong\lim_{\epsilon>0}S(o_{[a-\epsilon,b+\epsilon]}).\]
However, rather than get into the details, we describe a few sorts of sheaves that may be of interest.

\begin{example}%
\label{ex.const_sheaves}%
\index{sheaf!constant}
For any set $A$ there is a sheaf $\const{A}\in\Shv(\IR)$ that assigns to each open set $U$ the set $\const{A}(U)\coloneqq A$. This allows us to refer to integers, or real numbers, or letters of an alphabet, as though they were behaviors. What sort of behavior is $7\in\NN$? It is the sort of behavior that never changes: it's always seven. Thus $\const{A}$ is called the \emph{constant sheaf on $A$}.
\end{example}

\begin{example}%
\index{sheaf!of local functions}
Fix any topological space $(X,\Op_X)$. Then there is a sheaf $F_X$ of \emph{local functions from $\IR$ to $X$}. That is, for any open set $U\in\Op_{\IR}$, we assign the set $F_X(U)\coloneqq\{f\colon U\to X\mid f\text{ is continuous}\}$. There is also the sheaf $G_X$ of local functions on the subspace $\RR\ss\IR$. That is, for any open set $U\in\Op_{\IR}$, we assign the set $G_X(U)\coloneqq\{f\colon U\cap\RR\to X\mid f\text{ is continuous}\}$. 
\end{example}

\begin{exercise}%
\label{exc.interval_domain_top}
Let's check that \cref{ex.const_sheaves} makes sense. Fix any topological space $(X,\Op_X)$ and any subset $R\ss\IR$ of the interval domain. Define $H_X(U)\coloneqq\{f\colon U\cap R\to X\mid f\text{ is continuous}\}$.
\begin{enumerate}
	\item Is $H_X$ a presheaf? If not, why not; if so, what are the restriction maps?
	\item Is $H_X$ a sheaf? Why or why not?
	\qedhere
\qedhere
\end{enumerate}
\end{exercise}

\begin{example}
Another source of examples comes from the world of open hybrid dynamical systems. These are machines whose behavior is a mixture of continuous movements---generally imagined as trajectories through a vector field---and discrete jumps. These jumps are imagined as being caused by signals that spontaneously arrive. Over any interval of time, a hybrid system has certain things that it can do and certain things that it cannot. Although we will not make this precise here, there is a construction for converting any hybrid system into a sheaf on $\IR$; we will give references in \cref{sec.C7_further_reading}.%
\index{dynamical system!hybrid}
\end{example}

We refer to sheaves on $\IR$ as behavior types because almost any sort of behavior one can imagine is a behavior type. Of course, a complex behavior type---such as the way someone acts when they are in love---would be extremely hard to write down. But the idea is straightforward: for any interval of time, say a three-day interval $(d,d+3)$, let $L(d,d+3)$ denote the set of all possible behaviors a person who is in love could possibly do. Obviously it's a big, unwieldy set, and no one would want to make precise. But to the extent that one can imagine that sort of behavior as occurring through time, they could imagine the corresponding sheaf.

\paragraph{The subobject classifier as a sheaf on $\IR$.}%
\index{subobject
classifier!for behavior types}
In any sheaf topos, the subobject classifier $\Omega$ is itself a sheaf. It is responsible for the truth values in the topos. As we said in \cref{subsec.subob_class_sheaf_topos}, when it comes to sheaves on a topological space $(X,\Op)$, truth values are open subsets $U\in\Op$.

$\Cat{BT}$ is the topos of sheaves on the space $(\IR,\Op)$, as defined in \cref{subsec.IR}. As always, the subobject classifier $\Omega$ assigns to any $U\in\Op$ the set of open subsets of $U$, so these are the truth values. But what do they mean? The idea is that every proposition, such as ``Bob likes the weather'' returns an open set $U$, as if to respond that Bob likes the weather ``...throughout time period $U$.'' Let's explore this just a bit more.

Suppose Bob likes the weather throughout the interval $(0,5)$ and throughout the
interval $(4,8)$. We would probably conclude that Bob likes the weather
throughout the interval $(0,8)$. But what about the more ominous statement ``a
single pair of eyes has remained watching position $p$.'' Then just
because it's true on $(0,5)$ and on $(4,8)$, does not imply that it's been true
on $(0,8)$: there may have been a change of shift, where one watcher was
relieved from their post by another watcher. As another example, consider the
statement ``the stock market did not go down by more than 10 points.'' This
might be true on $(0,5)$ and true on $(4,8)$ but not on $(0,8)$. In order to
capture the semantics of statements like these---statements that take time to
evaluate---we must use the space $\IR$ rather than the space $\RR$.

%
%
%

%
\index{interval domain, $\IR$|)}
\subsection{Safety proofs in temporal logic}%
\index{safety proof}

We now have at least a basic idea of what goes into a proof of safety, say for
autonomous vehicles, or airplanes in the national airspace system. In fact, the underlying ideas of this chapter came out of a project between MIT, Honeywell Inc., and NASA \cite{speranzon2018abstraction}. The background for the project was that the National Airspace System consists of many different systems interacting: interactions between airplanes, each of which is an interaction between physics, humans, sensors, and actuators, each of which is an interaction between still more basic parts. The same sort of story would hold for a fleet of autonomous vehicles, as in the introduction to this chapter.

Suppose that each of the systems---at any level---is guaranteed to satisfy some property. For example, perhaps we can assume that an engine is either out of gas, has a broken fuel line, or is following the orders of a human driver or pilot. If there is a rupture in the fuel line, the sensors will alert
the human within three seconds, etc. Each of the components interact with a number
of different variables. In the case of airplanes, a pilot interacts with the radio, the positions of the
dials, the position of the thruster, and the visual data in front of her. The
component---here the pilot---is guaranteed to keep these variables in some
relation: ``if I see something, I will say something'' or ``if the dials are in
position $\const{bad\_pos}$, I will engage the thruster within 1 second.'' We
call these guarantees \emph{behavior contracts}.%
\index{behavior!contract}

All of the above can be captured in the topos $\Cat{BT}$ of behavior types. The variables are behavior types: the altimeter is a variable whose value $\theta\in \RR_{\geq0}$ is changing continuously with respect to time. The thruster is also a continuously-changing variable whose value is in the range $[0,1]$, etc.

The guaranteed relationships---behavior contracts---are given by predicates on variables. For example, if the pilot will always engage the thruster within one second of the display dials being in position $\const{bad\_pos}$, this can be captured by a predicate $p\colon\const{dials}\times\const{thrusters}\to\Omega$. While we have not written out a formal language for $p$, one could imagine the predicate $p(D,T)$ for $D:\const{dials}$ and $T:\const{thrusters}$ as
\begin{multline}%
\label{eqn.sample_contract}
  \forall(t:\RR)\ldotp @_{t}\big(\const{bad\_pos}(D)\big)\imp\\
  \exists(r:\RR)\ldotp(0<r<1)\wedge\forall(r':\RR)\ldotp 0\leq r'\leq 5\imp @_{t+r+r'}\big(\const{engaged}(T)\big).
\end{multline}
Here $@_t$ is a modality, as we discussed in \cref{def.modality}; in fact it turns out to be one of type 3.\ from \cref{prop.open_closed_quasiclosed}, but we cannot go into that. For a proposition $q$, the statement $@_t(q)$ says that $q$ is true in some small enough neighborhood around $t$. So \eqref{eqn.sample_contract} says ``starting within one second of whenever the dials say that we are in a bad position, I'll engage the thrusters for five seconds.'' %
\index{predicate}

Given an actual playing-out-of-events over a time period $U$, i.e.\ actual section $D\in\const{dials}(U)$ and section $T\in\const{thrusters}(U)$, the
predicate \cref{eqn.sample_contract} will hold on certain parts of $U$ and not others, and this is the truth value of $p$. Hopefully the pilot upholds her behavior contract at all
times she is flying, in which case the truth value will be $\true$ throughout that interval $U$. But if the pilot breaks her contract over certain intervals, then this fact is recorded in $\Omega$.

The logic allows us to record axioms like that shown in \cref{eqn.sample_contract} and then reason from them: e.g.\ if the pilot and the airplane, and at least one of the three radars upholds its contract then safe separation will be maintained. We cannot give further details here, but these matters have been worked out in detail in \cite{Schultz.Spivak:2017a}; see \cref{sec.C7_further_reading}.

\section{Summary and further reading}%
\label{sec.C7_further_reading}

This chapter was about modeling various sorts of behavior using sheaves on a space of time-intervals. Behavior may seem like it's something that occurs now in the present, but in fact our memory of past behavior informs what the current behavior means. In order to commit to anything, to plan or complete any sort of process, one needs to be able to reason over time-intervals. The nice thing about temporal sheaves---indeed sheaves on any site---is that they fit into a categorical structure called a topos, which has many useful formal properties. In particular, it comes equipped with a higher-order logic with which we can formally reason about how temporal sheaves work together when combined in larger systems. A much more detailed version of this story was presented in \cite{Schultz.Spivak:2017a}.\nocite{Spivak.Vasilakopoulou.Schultz:2016a} But it would have been impossible without the extensive edifice of topos theory and domain theory that has been developed over the past six decades.

Sheaf toposes were invented by Grothendieck and his school in the 1960s \cite{Artin.Grothendieck.Verdier:1971a} as an approach to proving conjectures at the intersection of algebraic geometry and number theory, called the Weil conjectures. Soon after, Lawvere and Tierney recognized that toposes had all the structure necessary to do logic, and with a whole host of other category theorists, the subject was developed to an impressive extent in many directions. For a much more complete history, see \cite{Mclarty:1990a}.

There are many sorts of references on topos theory. One that starts by introducing categories and then moves to toposes, focusing on logic, is \cite{Mclarty:1992a}. Our favorite treatment is perhaps \cite{MacLane.Moerdijk:1992a}, where the geometric aspects play a central role. Finally, Johnstone has done the field a huge favor by collecting large amounts of the theory into a single two-volume set \cite{Johnstone:2002a}; it is very dense, but an essential reference for the serious student or researcher. For just categorical (Kripke-Joyal) semantics of logic in a topos, one should see either \cite{MacLane.Moerdijk:1992a}, \cite{Jacobs:1999a}, or \cite{Lambek.Scott:1988a}.

We did not mention domain theory much in this chapter, aside from referring to the interval domain. But domains, in the sense of Dana Scott, play an important role in the deeper aspects of temporal type theory. A good reference is \cite{Gierz.Keimel.Lawson.Mislove.Scott:2003a}, but for an introduction we suggest \cite{Abramsky.Jung:1994a}.

In some sense our application area has been a very general sort of dynamical system. Other categorical approaches to this subject include \cite{Joyal.Nielsen.Winskel:1996a}, \cite{Haghverdi.Tabuada.Pappas:2003a}, \cite{Ames.Sastry:2005a}, and \cite{Lawvere:1986a}, though there are many others.%
\index{dynamical system!}

\bigskip

We hope you have enjoyed the seven sketches in this book. As a next step, consider running a reading course on applied category theory with some friends or colleagues. Simultaneously, we hope you begin to search out categorical ways of thinking about familiar subjects. Perhaps you'll find something you want to contribute to this growing field of applied category theory, or as we sometimes call it, the field of compositionality.%
\index{compositionality}

\index{topos|)}
\index{applied category theory|)}

\appendix
\begingroup
\footnotesize

\chapter{Exercise solutions}\label{chap.solutions}

\section[Solutions for Chapter 1]{Solutions for \cref{chap.preorders}.}

\sol{exc.function_pres}{
Some terminology: a function $f\colon \RR\to\RR$ is said to be
\begin{enumerate}
	\item \emph{order-preserving} if $x\leq y$ implies $f(x)\leq f(y)$, for all $x,y\in\RR$;%
	\footnote{We are often taught to view functions $f\colon\rr\to\rr$ as plots on an $(x,y)$-axis, where $x$ is the domain (independent) variable and $y$ is the codomain (dependent) variable. In this book, we do not adhere to that naming convention; e.g.\ both $x$ and $y$ here are being ``plugged in'' as input to $f$.}
	\item \emph{metric-preserving} if $|x-y|=|f(x)-f(y)|$;
	\item \emph{addition-preserving} if $f(x+y)=f(x)+f(y)$.
\end{enumerate}
For each of the three properties defined above---call it \emph{foo}---find an $f$ that is \emph{foo}-preserving and
an example of an $f$ that is not \emph{foo}-preserving.
}
{
For each of the following properties, we need to find a function $f\colon\rr\to\rr$ that preserves it, and another function---call it $g$---that does not.
\begin{description}
	\item[order-preserving:] Take $f(x)=x+5$; if $x\leq y$ then $x+5\leq y+5$, so $f$ is order-preserving. Take $g(x)\coloneqq -x$; even though $1\leq 2$, the required inequality $-1\leq^{?} -2$ does not hold, so $g$ is not order-preserving.
	\item[metric-preserving:] Take $f(x)\coloneqq x+5$; for any $x,y$ we have $|x-y|=|(x+5)-(y+5)|$ by the rules of arithmetic, so $|x-y|=|f(x)-f(y)|$, meaning $f$ preserves metric. Take $g(x)\coloneqq 2*x$; then with $x=1$ and $y=2$ we have $|x-y|=1$ but $|2x-2y|=2$, so $g$ does not preserve the metric.
	\item[addition-preserving:] Take $f(x)\coloneqq 3*x$; for any $x,y$ we have $3*(x+y)=(3*x)+(3*y)$, so $f$ preserves addition. Take $g(x)\coloneqq x+1$; then with $x=0$ and $y=0$, we have $g(x+y)=1$, but $g(x)+g(y)=2$, so $g$ does not preserve addition.
\end{description}
}

\sol{exc.joining_systems}{
What is the result of joining the following two systems?
\[

\qedhere
\]
}
{Here is the join of the two systems:
\[
%
\]
}

\sol{exc.partitions_practice}{%
\begin{enumerate}
	\item Write down all the partitions of a two element set $\{\bullet,\ast\}$, order them as above, and draw the Hasse diagram.
	\item Now do the same thing for a four element-set, say $\{1,2,3,4\}$. There should be 15 partitions.
\end{enumerate}
Choose any two systems in your 15-element Hasse diagram, call them $A$ and $B$. 
\begin{enumerate}[resume]
	\item What is $A\vee B$, using the definition given in the paragraph above \cref{eqn.generative}?
	\item Is it true that $A\leq (A\vee B)$ and $B\leq (A\vee B)$?
	\item What are all the systems $C$ for which both $A\leq C$ and $B\leq C$.
	\item Is it true that in each case, $(A\vee B)\leq C$?
	\qedhere
\qedhere
\end{enumerate}
}{
\nolisttopbreak
\begin{enumerate}
	\item Here is the Hasse diagram for partitions of the two element set $\{\bullet,\ast\}$:
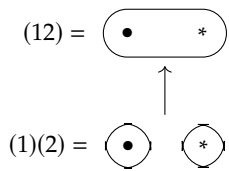
	\[

	\]
	\item Here is a picture (using text, rather than circles) for partitions of the set $1,2,3,4$:
	\[
	%
	\]
For the remaining parts, we choose $A=(12)(3)(4)$ and $B=(13)(2)(4)$.
	\item $A\vee B=(123)(4)$.
	\item Yes, it is true that $A\leq (A\vee B)$ and that $B\leq(A\vee B)$.
	\item The systems $C$ with $A\leq C$ and $B\leq C$ are: $(123)(4)$ and $(1234)$.
	\item Yes, it is true that in each case $(A\vee B)\leq C$.
\end{enumerate}
}

\sol{exc.boolean_vee_practice}{
Using the order $\false\leq\true$ on $\BB=\{\true,\false\}$, what is:
\begin{enumerate}
	\item $\true \vee \false$?
	\item $\false \vee \true$?
	\item $\true \vee \true$?
	\item $\false \vee \false$?
\end{enumerate}
}
{
\begin{enumerate}
	\item $\true \vee \false=\true$.
	\item $\false \vee \true=\true$.
	\item $\true \vee \true=\true$.
	\item $\false \vee \false=\false$.
\end{enumerate}
}

\sol{exc.subsets_products}{
	Let $A\coloneqq\{h,1\}$ and $B\coloneqq\{1,2,3\}$.
	\begin{enumerate}
		\item There are eight subsets of $B$; write them out.
		\item Take any two subsets of $B$ and write out their union.
		\item There are six elements in $A\times B$; write them out.
	\item There are five elements of $A \dju B$; write them out.
	\item If we consider $A$ and $B$ as subsets of the set
	$\{h,1,2,3\}$, there are four elements of $A \cup B$; write them out.
	\qedhere
\end{enumerate}
}{
\begin{enumerate}
	\item The eight subsets of $B\coloneqq\{1,2,3\}$ are
	\[\varnothing,\quad \{1\},\quad \{2\},\quad \{3\},\quad \{1,2\},\quad \{1,3\},\quad \{2,3\},\quad \{1,2,3\}.\]
	\item The union of $\{1,2,3\}$ and $\{1\}$ is $\{1,2,3\}\cup\{1\}=\{1,2,3\}$.
	\item The six elements of $\{h,1\}\times\{1,2,3\}$ are
	\[(h,1),\quad (h,2),\quad (h,3),\quad (1,1),\quad (1,2),\quad (1,3).\]
	\item The five elements of $\{h,1\} \dju \{1,2,3\}$ are
	\[
	(h,1),\quad (1,1),\quad (1,2), \quad (2,2), \quad (3,2).
	\]
	\item The four elements of $\{h,1\} \cup \{1,2,3\}$ are
	\[
	h,\quad 1,\quad 2,\quad 3.
	\]
\end{enumerate}
}

\sol{exc.comprehension_comprehension}{
\begin{enumerate}
	\item Is it true that $\nn=\{n\in\zz\mid n\geq 0\}$?
	\item Is it true that $\nn=\{n\in\zz\mid n\geq 1\}$?
	\item Is it true that $\varnothing=\{n\in\zz\mid 1<n<2\}$?
\end{enumerate}
}{
\begin{enumerate}
	\item This is true: a natural number is exactly an integer that is at least 0.
	\item This is false: $0\in\nn$ but $0\not\in\{n\in\zz\mid n\geq 1\}$.
	\item This is true: no elements of $\zz$ are strictly between 1 and 2.
\end{enumerate}
}

\sol{exc.proof_re_parts}{
Suppose that $A$ is a set and $\{A_p\}_{p\in P}$ and $\{A'_{p'}\}_{p'\in P'}$ are two partitions of $A$ such that for each $p\in P$ there exists a $p'\in P'$ with $A_p=A'_{p'}$.
\begin{enumerate}
	\item Show that for each $p\in P$ there is at most one $p'\in P'$ such that $A_p=A'_{p'}$
	\item Show that for each $p'\in P'$ there is a $p\in P$ such that $A_p=A'_{p'}$.
\end{enumerate}
}{
Suppose that $A$ is a set and $\{A_p\}_{p\in P}$ and $\{A'_{p'}\}_{p'\in P'}$ are two partitions of $A$ such that for each $p\in P$ there exists a $p'\in P'$ with $A_p=A'_{p'}$.
\begin{enumerate}
	\item Given $p\in P$, suppose we had $p_1',p_2'\in P'$ such that $A_p=A'_{p_1'}$ and $A_p=A'_{p_2'}$. Well then $A_{p_1'}=A_{p_2'}$, so in particular $A_{p_1'}\cap A_{p_2'}=A_{p_1'}$. By the definition of partition (\ref{def.partition}), $A_{p_1'}\neq\varnothing$, and yet if $p_1\neq p_2$ then $A_{p_1'}\cap A_{p_2'}=\varnothing$. This can't be, so we must have $p_1'=p_2'$, as desired.
	\item Suppose given $p'\in P'$; we want to show that there is a $p\in P$ such that $A_p=A'_{p'}$. Since $A'_{p'}\neq\varnothing$ is nonempty by definition, we can pick some $a\in A'_{p'}$; since $A'_{p'}\ss A$, we have $a\in A$. Finally, since $A=\bigcup_{p\in P} A_p$, there is some $p$ with $a\in A_p$. This is our candidate $p$; now we show that $A_p=A'_{p'}$. By assumption there is some $p''\in P'$ with $A_p=A'_{p''}$, so now $a\in A'_{p''}$ and $a\in A'_{p'}$, so $a\in A'_{p'}\cap A'_{p''}$. Again by definition, having a nonempty intersection means $p'=p''$. So we conclude that $A_p=A_{p'}$.
\end{enumerate}}

\sol{exc.equiv_rel_practice}{
Consider the partition shown below:
\[
\begin{tikzpicture}[x=1cm]
	\node (a11) {$\LMO{11}$};
	\node[right=.5 of a11] (a12) {$\LMO{12}$};
	\node[right=.5 of a12] (a13) {$\LMO{13}$};
	\node[below=.5 of a11] (a21) {$\LMO{21}$};
	\node[below=.5 of a12] (a22) {$\LMO{22}$};
	\node[below=.5 of a13] (a23) {$\LMO{23}$};
	\draw [rounded corners=9pt] 
     ($(a11)+(135:.45)$) --
     ($(a11)+(-135:.45)$) --
     ($(a12)+(-45:.45)$) --
     ($(a12)+(45:.45)$) --     
     cycle;
	\draw [rounded corners=9pt] 
     ($(a22)+(135:.45)$) --
     ($(a22)+(-135:.45)$) --
     ($(a23)+(-45:.45)$) --
     ($(a23)+(45:.45)$) --     
     cycle;
	\draw [rounded corners=9pt] 
     ($(a13)+(135:.45)$) --
     ($(a13)+(-135:.45)$) --
     ($(a13)+(-45:.45)$) --
     ($(a13)+(45:.45)$) --     
     cycle;
	\draw [rounded corners=9pt] 
     ($(a21)+(135:.45)$) --
     ($(a21)+(-135:.45)$) --
     ($(a21)+(-45:.45)$) --
     ($(a21)+(45:.45)$) --     
     cycle;
	\node[inner sep=10pt, draw, fit=(a11) (a23)] (a) {};
\end{tikzpicture}
\]
Write down every pair $(a,b)$ such that $a\sim b$. There should be 10.
}
{
The pairs $(a,b)$ such that $a\sim b$ are:
\begin{gather*}
(11,11)\quad
(11,12)\quad
(12,11)\quad
(12,12)\quad
(13,13)\\
(21,21)\quad
(22,22)\quad
(12,23)\quad
(23,22)\quad
(23,23)
\end{gather*}
}

\sol{exc.equiv_part_proof}{
Consider the proof of \cref{prop.equivalence_partition}. Suppose that $\sim$ is
an equivalence relation, and let $P$ be the set of $(\sim)$-closed and
$(\sim)$-connected subsets $\{A_p\}_{p\in P}$.
\begin{enumerate}
	\item Show that each $A_p$ is nonempty.
	\item Show that if $p\neq q$, i.e.\ if $A_p$ and $A_q$ are not exactly the same set, then $A_p\cap A_q=\varnothing$.
	\item Show that $A=\bigcup_{p\in P}A_p$.
	\qedhere
\qedhere
\end{enumerate}
}
{
\begin{enumerate}
	\item One aspect in the definition of the parts is that they are connected, and one aspect of that is that they are nonempty. So each part $A_p$ is nonempty.
	\item Suppose $p\neq q$, i.e.\ $A_p$ and $A_q$ are not exactly the same set. To prove $A_p\cap A_q=\varnothing$, we suppose otherwise and derive a contradiction. So suppose there exists $a\in A_p\cap A_q$; we will show that $A_p=A_q$, which contradicts an earlier hypothesis. To show that these two subsets are equal, it suffices to show that $a'\in A_p$ iff $a'\in A_q$ for all $a'\in A$. Suppose $a'\in A_p$; then because $A_p$ is connected, we have $a\sim a'$. And because $A_q$ is closed, $a'\in A_q$. In just the way, if $a'\in A_q$ then because $A_q$ is connected and $A_p$ is closed, $a'\in A_p$, and we are done.
	\item To show that $A=\bigcup_{p\in P}A_p$, it suffices to show that for each $a\in A$, there is some $p\in P$ such that $a\in A_p$. We said that $P$ was the set of closed and connected subsets of $A$, so it suffices to show that there is some closed and connected subset containing $a$. Let $X\coloneqq\{a'\in A\mid a'\sim a\}$; we claim it is closed and connected and contains $a$. To see $X$ is closed, suppose $a'\in X$ and $b\sim a'$; then $b\sim a$ by transitivity and symmetry of $\sim$, so $b\in X$. To see that $X$ is connected, suppose $b,c\in X$; then $b\sim a$ and $c\sim a$ so $b\sim c$ by the transitivity and symmetry of $\sim$. Finally, $a\in X$ by the reflexivity of $\sim$.
\end{enumerate}
}

\sol{exc.inj_surj_fun}{
\begin{enumerate}
	\item Find two sets $A$ and $B$ and a function $f\colon A\to B$ that is injective but not surjective.
	\item Find two sets $A$ and $B$ and a function $f\colon A\to B$ that is surjective but not injective.
\end{enumerate}
Now consider the four relations shown here:
\[

\]
For each relation, answer the following two questions.
\begin{enumerate}[resume]
	\item Is it a function?
	\item If not, why not? If so, is it injective, surjective, both, or neither?
\qedhere
\end{enumerate}
}
{
\begin{enumerate}
	\item The unique function $\varnothing\to\{1\}$ is injective but not surjective.
	\item The unique function $\{a,b\}\to\{1\}$ is surjective but not injective.
	\item The second and third are not functions; the first and fourth are functions.
  \[
  %
  \]
	\item Neither the second nor third is `total'. Moreover, the second one is not deterministic. The first one is a function which is not injective and not surjective. The fourth one is a function which is both injective and surjective.
\end{enumerate}
}

\sol{exc.map_to_empty}{
Suppose that $A$ is a set and $f\colon A\to\varnothing$ is a function to the empty set. Show that $A$ is empty.
}
{
By \cref{def.function}, a function $f\colon A\to\varnothing$ is a subset $F\ss A\times\varnothing$ such that for all $a\in A$, there exists a unique $b\in \varnothing$ with $(a,b)\in F$. But there are no elements $b\in^?\varnothing$, so if $F$ is to have the above property, there can be no $a\in A$ either; i.e.\ $A$ must be empty.
}

\sol{exc.part_surj}{
Write down the surjection corresponding to each of the five partitions in \cref{eqn.parts_of_3}.
}{
Below each partition, we draw a corresponding surjection out of $\{\bullet,\ast,\circ\}$:
\[

    };
  \node[below left=.5 and -.75 of none] (s1) {$\bullet$};
  \node[below=0 of s1] (s2) {$\ast$};
  \node[below=0 of s2] (s3) {$\circ$};
  \node[right=.5 of s1] (t1) {$p_1$};
  \node[right=.5 of s2] (t2) {$p_2$};
  \node[right=.5 of s3] (t3) {$p_3$};
	\draw[|->] (s1) -- (t1);
	\draw[|->] (s2) -- (t2);
	\draw[|->] (s3) -- (t3);
  \node[left=.2] at (s1-|ac) (s1) {$\bullet$};
  \node[below=0 of s1] (s2) {$\ast$};
  \node[below=0 of s2] (s3) {$\circ$};
  \node[right=.5 of s1] (t1) {$p_1$};
  \node[right=.5 of s2] (t2) {$p_2$};
	\draw[|->] (s1) -- (t1);
	\draw[|->] (s2) -- (t2);
	\draw[|->] (s3) -- (t1);
  \node[left=.2] at (s1-|ab) (s1) {$\bullet$};
  \node[below=0 of s1] (s2) {$\ast$};
  \node[below=0 of s2] (s3) {$\circ$};
  \node[right=.5 of s1] (t1) {$p_1$};
  \node[right=.5 of s2] (t2) {$p_2$};
	\draw[|->] (s1) -- (t1);
	\draw[|->] (s2) -- (t1);
	\draw[|->] (s3) -- (t2);
  \node[left=.2] at (s1-|bc) (s1) {$\bullet$};
  \node[below=0 of s1] (s2) {$\ast$};
  \node[below=0 of s2] (s3) {$\circ$};
  \node[right=.5 of s1] (t1) {$p_1$};
  \node[right=.5 of s2] (t2) {$p_2$};
	\draw[|->] (s1) -- (t1);
	\draw[|->] (s2) -- (t2);
	\draw[|->] (s3) -- (t2);
  \node[left=.2] at (s1-|all) (s1) {$\bullet$};
  \node[below=0 of s1] (s2) {$\ast$};
  \node[below=0 of s2] (s3) {$\circ$};
  \node[right=.5 of s1] (t1) {$p_1$};
	\draw[|->] (s1) -- (t1);
	\draw[|->] (s2) -- (t1);
	\draw[|->] (s3) -- (t1);
\end{tikzpicture}
\]
}

\sol{exc.understanding_graphs}{
Fill in the table from \cref{ex.my_first_graph}.
}{
\[
G=\fbox{
\begin{tikzcd}[ampersand replacement=\&]
	\LMO{1}\ar[r, "a"]\ar[dr, shift left, "b"]\ar[dr, shift right, "c"']\&
	\LMO{2}\ar[d, "e"]\ar[loop right, "d"]\\
	\&\LMO{3}\&\LMO{4}
\end{tikzcd}
}
\hspace{.6in}
\begin{array}{l || l | l}
	\textbf{arrow }a&\textbf{source }s(a)\in V&\textbf{target }t(a)\in V\\\hline
	a&1&2\\
	b&1&3\\
	c&1&3\\
	d&2&2\\
	e&2&3
\end{array}
\]
}

\sol{exc.graph_to_hasse}{
	What preorder relation $(P,\leq)$ is depicted by the graph $G$ in \cref{ex.my_first_graph}? That is, what are the elements of $P$ and write down every pair $(p_1,p_2)$ for which $p_1\leq p_2$.
}
{
The graph $G$ from \cref{exc.understanding_graphs} is a strange Hasse diagram because it has two arrows $1\to 3$ and a loop, both of which are ``useless'' from a preorder point-of-view. But that does not prevent our formula from working. The preorder $(P,\leq)$ is given by taking $P\coloneqq V=\{1,2,3,4\}$ and writing $p\leq q$ whenever there exists a path from $p$ to $q$. So:
\[
1\leq 1,\quad 1\leq 2,\quad 1\leq 3,\quad 2\leq 2,\quad 2\leq 3,\quad 3\leq 3,\quad 4\leq 4
\]
}
	\sol{exc.points_as_hasse}{
Does a collection of points, like the one in \cref{ex.disc_preorder}, count as a Hasse diagram?}
{
A collection of points, e.g.\ 
$
\begin{tikzpicture}[baseline=(1.-20)]
  	\node (0) at (0,0) {$\bullet$};
  	\node (1) at (.5,0) {$\bullet$};
  	\node (2) at (1,0) {$\bullet$};
  	\node[draw, rounded corners, inner sep=0, fit=(0) (1) (2)] (A) {};
\end{tikzpicture}
$
is a Hasse diagram, namely for the discrete order, i.e.\ for the order where $x\leq y$ iff $x=y$.
}

\sol{exc.order_parts_practice}{
Let $X$ be the set of partitions of $\{\bullet,\circ,\ast\}$; it has five elements and an order by coarseness, as shown in the Hasse diagram \cref{eqn.parts_of_3}. Write down every pair $(x,y)$ of elements in $X$ such that $x\leq y$. There should be 12.
}
{
Let's write the five elements of $X$ as
\[(\bullet)(\circ)(\ast),\quad (\bullet\circ)(\ast),\quad (\bullet\ast)(\circ),\quad (\bullet)(\circ\ast),\quad (\bullet\circ\ast)\]
Our job is to write down all 12 pairs of $x_1,x_2\in X$ with $x_1\leq x_2$. Here they are:
\begin{align*}
  (\bullet)(\circ)(\ast)&\leq (\bullet)(\circ)(\ast)&
  (\bullet)(\circ)(\ast)&\leq (\bullet\circ)(\ast)&
  (\bullet)(\circ)(\ast)&\leq (\bullet\ast)(\circ)\\
  (\bullet)(\circ)(\ast)&\leq (\bullet)(\circ\ast)&
  (\bullet)(\circ)(\ast)&\leq (\bullet\circ\ast)&
  (\bullet\circ)(\ast)&\leq (\bullet\circ)(\ast)\\
  (\bullet\circ)(\ast)&\leq (\bullet\circ\ast)&
  (\bullet\ast)(\circ)&\leq(\bullet\ast)(\circ)&
  (\bullet\ast)(\circ)&\leq(\bullet\circ\ast)\\
  (\bullet)(\circ\ast)&\leq(\bullet)(\circ\ast)&
  (\bullet)(\circ\ast)&\leq(\bullet\circ\ast)&
  (\bullet\circ\ast)&\leq(\bullet\circ\ast)
\end{align*}
}

\sol{exc.discrete_preorder_comparable}{
Is it correct to say that a discrete preorder is one where \emph{no} two elements are comparable?
}
{
The statement in the text is almost correct. It is correct to say that a discrete preorder is one where $x$ and $y$ are comparable if and only if $x=y$.
}

\sol{exc.total_order1}{
Write down the numbers $1,2,\ldots,10$ and draw an arrow $a\to b$ if $a$ divides perfectly into $b$. 
Is it a total order?
}
{
\[
\boxCD{
\begin{tikzcd}[ampersand replacement=\&,row sep=1.5ex]
  \&\&\LMO{8}\\
  \LMO{9}\&\LMO{6}\&\LMO{4}\ar[u]\&\LMO{10}\\
  \&\LMO{3}\ar[ul]\ar[u]\&\LMO{2}\ar[ul]\ar[u]\ar[ur]\&\LMO{5}\ar[u]\&\LMO{7}\\
  \&\&\LMO{1}\ar[ul]\ar[u]\ar[ur]\ar[urr]
\end{tikzcd}
}
\]
No, it is not a total order; for example $4\not\leq 6$ and $6\not\leq 4$. 
}

\sol{exc.total_order2}{
Is the usual $\leq$ ordering on the set $\rr$ of real numbers a total order?
}
{
Yes, the usual $\leq$ ordering is a total order on $\rr$: for every $a,b\in\rr$ either $a\leq b$ or $b\leq a$.
}

\sol{exc.powerset_Hasse}{
Draw the Hasse diagrams for $\powset(\varnothing)$, $\powset\{1\}$, and $\powset\{1,2\}$.
}
{
The Hasse diagrams for $\powset(\varnothing)$, $\powset\{1\}$, and
$\powset\{1,2\}$ are
\[
\begin{altikz}
  	\node (0) at (0,0) {$\varnothing$};
  	\node[draw, fit=(0)] (A) {};
\end{altikz}
\hspace{.1\textwidth}
\begin{altikz}
  	\node (0) at (0,0) {$\varnothing$};
  	\node (1) at (0,1) {$\{1\}$};
  	\node[draw, fit=(0) (1)] (A) {};
	\draw[->] (0) to (1);
\end{altikz}
\hspace{.1\textwidth}
\begin{altikz}
  	\node (0) at (0,0) {$\varnothing$};
  	\node (1) at (-.5,1) {$\{1\}$};
  	\node (2) at (.5,1) {$\{2\}$};
  	\node (3) at (0,2) {$\{1,2\}$};
  	\node[draw, fit=(0) (1) (2) (3)] (A) {};
	\draw[->] (0) to (1);
	\draw[->] (0) to (2);
	\draw[->] (1) to (3);
	\draw[->] (2) to (3);
\end{altikz}
\]
}

\sol{exc.partitions_and_functions}{
For any set $S$ there is a coarsest partition, having just one part. What surjective function does it correspond to?

There is also a finest partition, where everything is in its own partition. What surjective function does it correspond to?
}
{
The coarsest partition on $S$ corresponds to the unique function $!\colon S \to
\{1\}$. The finest partition on $S$ corresponds to the identity function
$\id_S\colon S \to S$.
}

\sol{exc.uppersets_on_discrete}{
  Prove that the preorder of upper sets on a discrete preorder (see \cref{ex.disc_preorder}) on a set $X$ is simply the power
  set $\powset(X)$.
}
{
If $X$ has the discrete preorder, then every subset $U$ of $X$ is an upper set:
indeed, if $p \in U$, the only $q$ such that $p \le q$ is $p$ itself, so $q$ is
definitely in $U$! This means that $\upset(X)$ contains all subsets of $X$,
so it's exactly the power set, $\upset(X)=\powset(X)$.
}

\sol{exc.product_preorder}
{
Draw the product of the two preorders drawn below:
\[
\boxCD{
\begin{tikzcd}[column sep=0, ampersand replacement=\&]
	\LMO{c}\&\&\LMO{b}\\
	\&\LMO[under]{a}\ar[ul]\ar[ur]
\end{tikzcd}
}
\hspace{.5in}
\boxCD{
\begin{tikzcd}
	\LMO{2}\\
	\LMO[under]{1}\ar[u]
\end{tikzcd}
}
\]
For bonus points, compute the upper set preorder on the result.
}
{
The product preorder and its upper set preorder are:
\[
\begin{altikz}[font=\scriptsize]
  	\node (a1) at (0,0) {$(a,1)$};
  	\node (c1) at (-1,1) {$(c,1)$};
  	\node (a2) at (0,1) {$(a,2)$};
  	\node (b1) at (1,1) {$(b,1)$};
  	\node (c2) at (-.5,2) {$(c,2)$};
  	\node (b2) at (.5,2) {$(b,2)$};
  	\node[draw, fit=(a1) (c1) (b1) (b2)] (A) {};
	\draw[->] (a1) to (c1);
	\draw[->] (a1) to (a2);
	\draw[->] (a1) to (b1);
	\draw[->] (a2) to (c2);
	\draw[->] (a2) to (b2);
	\draw[->] (c1) to (c2);
	\draw[->] (b1) to (b2);
\end{altikz}
\hspace{.05\textwidth}
\begin{altikz}[font=\scriptsize,xscale=2]
  	\node (c2) at (-1,0) {$\{(c,2)\}$};
  	\node (b2) at (1,0) {$\{(b,2)\}$};
  	\node (c1) at (-2,1) {$\{(c,1),(c,2)\}$};
  	\node (bc2) at (0,1) {$\{(b,2),(c,2)\}$};
  	\node (b1) at (2,1) {$\{(b,1),(b,2)\}$};
  	\node (b2c1) at (-2,2) {$\{(b,2),(c,1),(c,2)\}$};
  	\node (a2) at (0,2) {$\{(a,2),(b,2),(c,2)\}$};
  	\node (b1c2) at (2,2) {$\{(b,1),(b,2),(c,2)\}$};
  	\node (ab2c1) at (-1,3) {$\{(a,2),(b,2),(c,1),(c,2)\}$};
  	\node (ac2b1) at (1,3) {$\{(a,2),(b,1),(b,2),(c,2)\}$};
  	\node (a2bc1) at (0,4) {$\{(a,2),(b,1),(b,2),(c,1),(c,2)\}$};
  	\node (a1) at (0,5) {$\{(a,1),(a,2),(b,1),(b,2),(c,1),(c,2)\}$};
	\coordinate (l) at (-2.75,.25);
	\coordinate (r) at (2.75,-.25);
  	\node[draw, fit=(l) (r) (a1)] (A) {};
	\draw[->] (c2) to (c1);
	\draw[->] (c2) to (bc2);
	\draw[->] (b2) to (b1);
	\draw[->] (b2) to (bc2);
	\draw[->] (c1) to (b2c1);
	\draw[->] (bc2) to (b2c1);
	\draw[->] (bc2) to (a2);
	\draw[->] (bc2) to (b1c2);
	\draw[->] (b1) to (b1c2);
	\draw[->] (b2c1) to (ab2c1);
	\draw[->] (a2) to (ab2c1);
	\draw[->] (a2) to (ac2b1);
	\draw[->] (b1c2) to (ac2b1);
	\draw[->] (ab2c1) to (a2bc1);
	\draw[->] (ac2b1) to (a2bc1);
	\draw[->] (a2bc1) to (a1);
\end{altikz}
\]
}

\sol{exc.card_as_monotone}{
Let $X=\{1,2,3\}$.
\begin{enumerate}
	\item Draw the Hasse diagram for $\powset(X)$.
	\item Draw the Hasse diagram for the preorder $0\leq 1\leq\cdots\leq 3$.
	\item Draw the cardinality map $\#$ from \cref{ex.card_as_monotone} as dashed lines between them.
\end{enumerate}
}
{
With $X=\{0,1,2\}$, the Hasse diagram for $\powset(X)$, the preorder $0\leq\cdots\leq 3$, and the cardinality map between them are shown below:
\[

\]
}

\sol{exc.upset_monotone_powset}{
Consider the preorder $\bb$. The Hasse diagram for $\upset(\bb)$ was drawn in \cref{ex.upper_set}, and you drew the Hasse diagram for $\powset(\bb)$ in \cref{exc.powerset_Hasse}. Now draw the monotone map between them, as described in \cref{ex.upset_monotone_powset}.
}
{
\[
%
\]
}

\sol{exc.upper_set}
{
  Let $(P,\leq)$ be a preorder, and recall the notion of opposite preorder from \cref{ex.opposite}.
\begin{enumerate}
	\item Show that the set $\upclose p\coloneqq\{p'\in P\mid p\leq p'\}$ is an upper set, for any $p\in P$.
	\item Show that this construction defines a monotone map $\upclose\colon P\op\to\upset(P)$.
	\item Show that if $p \le p'$ in $P$ if and only if $\upset(p') \subseteq \upset(p)$. 
	\item Draw a picture of the map $\upclose$ in the case where $P$ is the preorder $(b\geq a\leq c)$ from \cref{ex.product_preorder}.
\end{enumerate}
This is known as the \emph{Yoneda lemma} for preorders. The if and only if
condition proved in part 3 implies that, up to equivalence, to know an element
is the same as knowing its upper set---that is, knowing its web of
relationships with the other elements of the preorder. The general Yoneda lemma is
a powerful tool in category theory, and a fascinating philosophical idea besides.
}
{
\begin{enumerate}
\item Let $q \in \upclose p$, and suppose $q \leq q'$. Since $q \in \upclose p$,
we have $p \leq q$. Thus by transitivity $p \leq q'$, so $q' \in \upclose p$. Thus
$\upclose p$ is an upper set.
\item Suppose $p \leq q$ in $P$; this means that $q \leq\op p$ in $P\op$. We
must show that $\upclose q \subseteq \upclose p$. Take any $q' \in \upclose q$.
Then $q \leq q'$, so by transitivitiy $p \leq q'$, and hence $q' \in \upclose
p$. Thus $\upclose q \subseteq \upclose p$.
\item Monotonicity of $\upclose$ says that $p \le p'$ implies $\upclose(p')
  \subseteq \upclose(p)$. We must prove the other direction, that if $p \nleq
  p'$ then $\upclose(p') \nsubseteq \upclose(p)$. This is straightforward, since by
  reflexivity we always have $p' \in \upclose(p')$, but if $p \nleq p'$, then
  $p' \notin \upclose(p)$, so $\upclose(p') \nsubseteq \upclose(p)$.
\item The map $\upclose\colon P\op \to \upset(P)$ can be depicted:
\[
\begin{tikzpicture}
	\node (c) at (-1,0.5) {$c$};
	\node (b) at (1,0.5) {$b$};
	\node (a) at (0,1.5) {$a$};
	\node[draw, fit=(a) (c) (b)] (A) {};
	\draw[->] (c) to (a);
	\draw[->] (b) to (a);
	\node (uc) at (5,0) {$\{c\}$};
	\node (ub) at (7,0) {$\{b\}$};
	\node (bc) at (6,1) {$\{b,c\}$};
	\node (ua) at (6,2) {$\{a,b,c\}$};
	\node[draw, fit=(ua) (uc) (ub)] (B) {};
	\draw[->] (uc) to (bc);
	\draw[->] (ub) to (bc);
	\draw[->] (bc) to (ua);
	\begin{scope}[mapsto]
	\draw (c) to[bend right=10pt] (uc);
	\draw (b) to[bend left=10pt] (ub);
	\draw (a) to (ua);
	\end{scope}
\end{tikzpicture}
\]
\end{enumerate}
}

\sol{exc.monotone_from_discrete}{
As you yourself well know, a monotone map between preorders $f\colon\mbox{$(P,\leq_P)$}\to\mbox{$(Q,\leq_Q)$}$, consists of a function $f\colon P\to Q$ that satisfies a ``monotonicity'' property. Show that when $(P,\leq_P)$ is a discrete preorder, then \emph{every} function $P\to Q$ satisfies the monotonicity property, regardless of the order $\leq_Q$.
}
{
Suppose $(P,\leq_P)$ is a discrete preorder and that $(Q,\leq_Q)$ is any preorder. We want to show that every function $f\colon P\to Q$ is monotone, i.e.\ that if $p_1\leq_Pp_2$ then $f(p_1)\leq_Q f(p_2)$. But in $P$ we have $p_1\leq_Pp_2$ iff $p_1=p_2$; that's what discrete means. If $p_1\leq_P p_2$ then $p_1=p_2$, so $f(p_1)=f(p_2)$, so $f(p_1)\leq f(p_2)$.
}

\sol{exc.f^*_partitions}{
Choose two sets $X$ and $Y$ with at least three elements each and choose a
surjective, non-identity function $f\colon X\to Y$ between them. Write down two
different partitions $P$ and $Q$ of $Y$, and then find $f^*(P)$ and $f^*(Q)$.
}{
Let $X=\zz=\{\ldots,-2,-1,0,1,\ldots\}$ be the set of all integers, and let $Y=\{n,z,p\}$; let $f\colon X\to Y$ send negative numbers to $n$, zero to $z$, and positive integers to $p$. This is surjective because all three elements of $Y$ are hit.

We consider two partitions of $Y$, namely $P\coloneqq (nz)(p)$ and $Q\coloneqq (np)(z)$. Technically, these are notation for $\{\{n,z\},\{p\}\}$ and $\{\{n,p\},\{z\}\}$ as sets of disjoint subsets whose union is $Y$. Their pulled back partitions are $f^*P=(\ldots,-2,-1,0)(1,2,\ldots)$ and $f^*Q=(0)(\ldots,-2,-1,1,2,\ldots)$, or technically
\[
  f^*(P)=\{\{x\in\zz\mid x\leq 0\},\{x\in\zz\mid x\geq 1\}\}
  \quad\text{and}\quad
  f^*(Q)=\{\{0\},\{x\in\zz\mid x\neq 0\}\}.
\]
}

\sol{exc.check_id_comp_monotone}{
Check the following two claims:

For any preorder $(P,\leq_P)$, the identity function is monotone.

If $(Q,\leq_Q)$ and $(R,\leq_R)$ are preorders and $f\colon P\to Q$ and $g\colon Q\to R$ are monotone, then $(f\cp g)\colon P\to R$ is also monotone.
}
{We have preorders $(P,\leq_P)$, $(Q,\leq_Q)$, and $(R,\leq_R)$, and we have monotone maps $f\colon P\to Q$ and $g\colon Q\to R$.
\begin{enumerate}
	\item To see that $\id_P$ is monotone, we need to show that if $p_1\leq_Pp_2$ then $\id_P(p_1)\leq\id_P(p_2)$. But $\id_P(p)=p$ for all $p\in P$, so this is clear.
	\item We have that $p_1\leq_P p_2$ implies $f(p_1)\leq_Q f(p_2)$ and that $q_1\leq_Q q_2$ implies $g(q_1)\leq_Rg(q_2)$. By substitution, $p_1\leq_P p_2$ implies $g(f(p_1))\leq_Rg(f(p_2))$ which is exactly what is required for $(f\cp g)$ to be monotone.
\end{enumerate}
}

\sol{exc.skeletal_dagger_preorder}{
Recall the notion of skeletal preorders (\cref{rem.skeletal_preorder}) and discrete preorders
(\cref{ex.disc_preorder}). Show that a skeletal dagger preorder is just a discrete preorder, and hence just a set.
}{
We need to show that if $(P,\leq_P)$ is both skeletal and dagger, then it is discrete. So suppose it is skeletal, i.e.\ $p_1\leq p_2$ and $p_2\leq p_1$ implies $p_1=p_2$. And suppose it is dagger, i.e.\ $p_1\leq p_2$ implies $p_2\leq p_1$. Well then $p_1\leq p_2$ implies $p_1=p_2$, and this is exactly the definition of $P$ being discrete.
}

\sol{exc.weird_Phi}{
  Show that the map $\Phi$ from \cref{subsec.first_look_gen}, which was roughly given by `Is $\bullet$ connected to $\ast$?'  is
  a monotone map $\prt{\{\ast,\bullet,\circ\}}\to\bb$; see also \cref{eqn.parts_of_3}.
}
{
The map $\Phi$ from \cref{subsec.first_look_gen} took partitions of $\{\bullet,\ast,\circ\}$ and returned true or false based on whether or not $\bullet$ was in the same partition as $\ast$. We need to see that it's actually a monotone map $\Phi\colon\prt{\{\bullet,\ast,\circ\}}\to\bb$. So suppose $P,Q$ are partitions with $P\leq Q$; we need to show that if $\Phi(P)=\true$ then $\Phi(Q)=\true$.

By definition $P\leq Q$ means that $P$ is finer than $Q$: i.e.\ $P$ differentiates more stuff, and $Q$ lumps more stuff together. Technically, $x\sim_Py$ implies $x\sim_Qy$ for all $x,y\in\{\bullet,\ast,\circ\}$. Applying this to $\bullet,\ast$ gives the result.
}

\sol{exc.pullback_upset}{
  Let $P$ and $Q$ be preorders, and $f\colon P \to Q$ be a monotone map. Then we
  can define a monotone map $f^\ast\colon\upset(Q) \to \upset(P)$ sending an
  upper set $U \subseteq Q$ to the upper set $f\inv (U) \subseteq P$. We call this the
  \emph{pullback along $f$}.

  Viewing upper sets as a monotone maps to $\bb$ as in \cref{prop.upperset_presheaf}, the pullback can be understood in terms of composition. Indeed, show
  that the $f^\ast$ is defined by taking $u\colon Q \to \bb$ to $f\cp u
  \colon P \to \bb$.
}{
Given a function $f\colon P\to Q$, we have $f^\ast\colon\upset(Q)\to\upset(P)$ given by $U\mapsto f\inv(U)$. But upper sets in $Q$ are classified by monotone maps $u\colon Q\to\bb$, and similarly for $P$; our job is to show that $f^\ast(U)$ is given by composing the classifier $u$ with $f$.

Given an upper set $U\ss Q$, let $u\colon Q\to\bb$ be the corresponding monotone map, which sends $q\mapsto\true$ iff $q\in U$. Then $(f\cp u)\colon P\to\bb$ sends $p\mapsto\true$ iff $f(p)\in U$; it corresponds to the upper set $\{p\in P\mid f(p)\in U\}$ which is exactly $f\inv(U)$.
}

\sol{exc.0_glb}{
\begin{enumerate}
	\item Why is $0$ a lower bound for $\left\{\frac{1}{n+1}\;\;\middle|\;\; n\in\NN\right\}\ss\rr$?
	\item Why is $0$ a \emph{greatest} lower bound (meet)?
\end{enumerate}
}
{
\begin{enumerate}
	\item $0$ is a lower bound for $S=\{\frac{1}{n+1}\mid n\in\NN\}$ because $0\leq\frac{1}{n+1}$ for any $n\in\nn$.
	\item Suppose that $b$ is a lower bound for $S$; we want to see that $b\leq 0$. If one believes to the contrary that $0<b$, then consider $1/b$; it is a real number so we can find a natural number $n$ that's bigger $1/b<n<n+1$. This implies $1<b(n+1)$ and hence $\frac{1}{n+1}<b$, but that is a contradiction of $b$ being a lower bound for $S$. The false believer is defeated!
\end{enumerate}
}
\sol{exc.wedge_element}{
Let $(P,\leq)$ be a preorder and $p\in P$ an element. Consider the set $A=\{p\}$ with one element.
\begin{enumerate}
	\item Show that $\bigwedge A\cong p$.
	\item Show that if $P$ is in fact a partial order then $\bigwedge A=p$.
	\item Are the analogous facts true when $\bigwedge$ is replaced by $\bigvee$?
\end{enumerate}
}{
We have a preorder $(P,\leq)$, an element $p\in P$, and a subset $A=\{p\}$ with one element.
\begin{enumerate}
	\item To see that $\bigwedge A\cong p$, we need to show that $p\leq a$ for all $a\in A$ and that if $q\leq a$ for all $a\in A$ then $q\leq p$. But the only $a\in A$ is $a=p$, so both are obvious.
	\item We know $p$ is a meet of $A$, so if $q$ is also a meet of $A$ then $q\leq a$ for all $a\in A$ so $q\leq p$; similarly $p\leq a$ for all $a\in A$, so $p\leq q$. Then by definition we have $p\cong q$, and since $(P,\leq)$ is a partial order, $p=q$.	
	\item The analogous facts are true when $\bigwedge$ is replaced by $\bigvee$; the only change in the argument is to replace $\leq$ by $\geq$ and `meet' by `join' everywhere.
\end{enumerate}
}

\sol{exc.division_meet}{
Recall the division ordering on $\nn$ from \cref{ex.orders_on_N}: we write $n\vert m$ if $n$ divides perfectly into $m$. The meet of any two numbers in
this preorder has a common name, you probably learned when you were around 10 years old; what is it? Similarly the join of any two numbers has a common name; what is it?
}{
The meet of $4$ and $6$ is the highest number in the order that divides both of them; the numbers dividing both are $1$ and $2$, and $2$ is higher, so $4\wedge 6=2$. Similar reasoning shows that $4\vee 6=12$. The meet is the `greatest common divisor' and the join is the `least common multiple,' and this holds up for all pairs $m,n\in\nn$ not just $4,6$.
}

\sol{exc.more_stuff}{
In \cref{def.gen_effect}, we defined generativity of $f$ as the inequality $f(a \vee b)\neq f(a)\vee f(b)$, but in the subsequent text we seemed to imply there would be not just a difference, but \emph{more stuff} in $f(a \vee b)$ than in $f(a)\vee f(b)$.

Prove that for any monotone map $f\colon P\to Q$, if $a,b\in P$ have a join and $f(a),f(b)\in Q$ have a join, then indeed $f(a)\vee f(b)\leq f(a \vee b)$.
}
{
Since $f$ is monotone, the facts that $a\leq a\vee b$ and $b\leq a\vee b$ imply that $f(a)\leq f(a\vee b)$ and $f(b)\leq f(a\vee b)$. But by definition of join, $f(a)\vee f(b)$ is the largest element with that property, so $f(a)\vee f(b)\leq f(a \vee b)$, as desired.
}

\sol{exc.right_adj_3times}{
In \cref{ex.adjoint_to_3times} we found a left adjoint for the monotone map $(3\times-)\colon \zz \to \rr$. Now find a right adjoint for the same map, and show it is correct.
}{
By analogy with \cref{ex.adjoint_to_3times}, the right adjoint for $(3\times-)$ should be $\floor{-/3}$. But to prove this is correct, we must show that for any any $r\in\rr$ and $z\in\zz$ we have $z\leq\floor{r/3}$ iff $3*z\leq r$.

Suppose the largest integer below $r/3$ is $z'\coloneqq\floor{r/3}$. Then $z\leq z'$ implies $3*z\leq 3*z'\leq 3*r/3=r$, giving one direction. For the other, suppose $3*z\leq r$. Then dividing both sides by 3, we have $z=3*z/3\leq r/3$. Since $z$ is an integer below $r/3$ it is below $\floor{r/3}$ because $\floor{r/3}$ is the greatest integer below $r/3$, and we are done.
}

\sol{exc.galois_linear_ord}{
Consider the preorder $P=Q=\ord{3}$.
\begin{enumerate}
	\item Let $f,g$ be the monotone maps shown below:
\[
\begin{tikzcd}[ampersand replacement=\&]
	P\ar[d, blue, "f"']\&\LMO{1}\ar[r]\ar[d, |->, bend right, blue]\&\LMO{2}\ar[r]\ar[dl, |->, bend right, blue]\&\LMO{3}\ar[d, |->, bend right, blue]\&P\\
	Q\&\LMO[under]{1}\ar[r]\ar[ur, |->, bend right, red]\&\LMO[under]{2}\ar[r]\ar[u, |->, bend right, red]\&\LMO[under]{3}\ar[u, |->, bend right, red]\&Q\ar[u, red, "g"']
\end{tikzcd}
\]
Is it the case that $f$ is left adjoint to $g$? Check that for each $1\leq p,q\leq 3$, one has $f(p)\leq q$ iff $p\leq g(q)$.
	\item Let $f,g$ be the monotone maps shown below:
\[
\begin{tikzcd}[ampersand replacement=\&]
	P\ar[d, blue, "f"']\&\LMO{1}\ar[r]\ar[d, |->, bend right, blue]\&\LMO{2}\ar[r]\ar[d, |->, bend right, blue]\&\LMO{3}\ar[d, |->, bend right, blue]\&P\\
	Q\&\LMO[under]{1}\ar[r]\ar[ur, |->, bend right, red]\&\LMO[under]{2}\ar[r]\ar[u, |->, bend right, red]\&\LMO[under]{3}\ar[u, |->, bend right, red]\&Q\ar[u, red, "g"']
\end{tikzcd}
\]
Is it the case that $f$ is left adjoint to $g$?
\end{enumerate}
}{
\begin{enumerate}
  \item We need to check that for all nine pairs $\{(p,q)\mid 1\leq p\leq 3\text{ and }1\leq q\leq 3\}$ we have $f(p)\leq q$ iff $p\leq g(q)$, where $f$ and $g$ are the functions shown here:
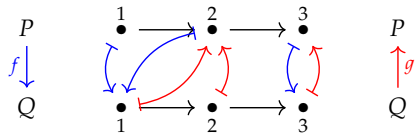
\[
\begin{tikzcd}[ampersand replacement=\&]
	P\ar[d, blue, "f"']\&\LMO{1}\ar[r]\ar[d, |->, bend right, blue]\&\LMO{2}\ar[r]\ar[dl, |->, bend right, blue]\&\LMO{3}\ar[d, |->, bend right, blue]\&P\\
	Q\&\LMO[under]{1}\ar[r]\ar[ur, |->, bend right, red]\&\LMO[under]{2}\ar[r]\ar[u, |->, bend right, red]\&\LMO[under]{3}\ar[u, |->, bend right, red]\&Q\ar[u, red, "g"']
\end{tikzcd}
\]
When $p=q=1$ we have $f(p)=1$ and $g(q)=2$, so both $f(p)=1\leq 1=q$ and $p=1\leq g(q)$; it works! Same sort of story happens when $(p,q)$ is $(1,2)$, $(1,3)$, $(2,1)$, $(2,2)$, $(2,3)$, and $(3,3)$. A different story happens for $p=3, q=1$ and $p=3, q=2$. In those cases $f(p)=3$ and $g(q)=2$, and neither inequality holds: $f(p)\not\leq q$ and $p\not\leq g(q)$. But that's fine, we still have $f(p)\leq q$ iff $p\leq g(q)$ in all nine cases, as desired.
	\item
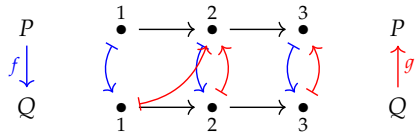
	\[
\begin{tikzcd}[ampersand replacement=\&]
	P\ar[d, blue, "f"']\&\LMO{1}\ar[r]\ar[d, |->, bend right, blue]\&\LMO{2}\ar[r]\ar[d, |->, bend right, blue]\&\LMO{3}\ar[d, |->, bend right, blue]\&P\\
	Q\&\LMO[under]{1}\ar[r]\ar[ur, |->, bend right, red]\&\LMO[under]{2}\ar[r]\ar[u, |->, bend right, red]\&\LMO[under]{3}\ar[u, |->, bend right, red]\&Q\ar[u, red, "g"']
\end{tikzcd}
\]
Here $f$ is \emph{not} left adjoint to $g$ because $f(2)\not\leq 1$ but $2\leq g(1)$.
\end{enumerate}
}

\sol{exc.extra_right_adj_ceil}{
\begin{enumerate}
	\item Does $\ceil{-/3}$ have a left adjoint $L\colon \zz\to\rr$?
	\item If not, why? If so, does its left adjoint have a left adjoint?
	\qedhere
\end{enumerate}
}{
\begin{enumerate}
	\item Let's suppose we have a monotone map $L\colon\zz\to\rr$ that's left adjoint to $\ceil{-/3}$ and see what happens. Writing $C(r)\coloneqq\ceil{r/3}$, then for all $z\in\zz$ and $r\in\rr$ we have $L(z)\leq r$ iff $z\leq C(r)$ by definition of adjunction. So take $z=1$ and $r=.01$; then $\ceil{r/3}=1$ so $z\leq C(r)$, and hence $L(z)\leq r$, i.e.\ $L(1)\leq 0.01$. In the same way $L(1)\leq r$ for all $r>0$, so $L(1)\leq 0$. By definition of adjunction $1\leq C(0)=\ceil{0/3}=0$, a contradiction.
	\item There's no left adjoint, because starting with an arbitrary one, we derived a contradiction.
\end{enumerate}
}

\sol{exc.g_left_part}{
There are 15 different partitions of a set with four elements. Choose 6 different ones and for each one, call it $c\colon S\surj P$, find $g_!(c)$, where $S$, $T$, and $g\colon S\to T$ are the same as they were in \cref{ex.pushforward_part}.
}{
We have $S=\{1,2,3,4\}$, $T=\{12,3,4\}$, and $g\colon S\to T$ the ``obvious'' function between them; see \cref{ex.pushforward_part}. Take $c_1,c_2,c_3,c_4$ to be the following partitions:
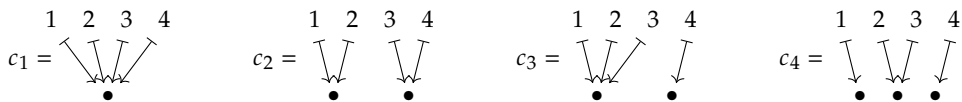
\[

\]
Then the induced partitions $g_!(c_1)$, $g_!(c_2)$, $g_!(c_3)$, and $g_!(c_4)$ on $T$ are:
\[
%
\]
}

\sol{exc.pullback_parts}{
There are five partitions possible on a set with three elements, say $T=\{12,3,4\}$. Using the same $S$ and $g\colon S\to T$ as in \cref{ex.pushforward_part}, determine the partition $g^*(c)$ on $S$ for each of the five partitions $c\colon T\surj P$.
}{
Here are the partitions on $S=\{1,2,3,4\}$ induced via $g^*$ by the five partitions on $T=\{12,3,4\}$:
\[
%
\]
}

\sol{exc.parts_adjunction}{
Let $S$, $T$, and $g\colon S\to T$ be as in \cref{exc.g_left_part}.
\begin{enumerate}
	\item Choose a nontrivial partition $c\colon S\surj P$ and let $g_!(c)$ be its push forward partition on $T$.
	\item Choose any coarser partition $d\colon T\surj P'$, i.e.\ where $g_!(c)\leq d$.
	\item Choose any non-coarser partition $e\colon T\surj Q$, i.e.\ where $g_!(c)\not\leq e$. (If you can't do this, revise your answer for \#1.)
	\item Find $g^*(d)$ and $g^*(e)$.
	\item The adjunction formula \cref{eqn.galois_connection} in this case says that since $g_!(c)\leq d$ and $g_!(c)\not\leq e$, we should have $c\leq g^*(d)$ and $c\not\leq g^*(e)$. Show that this is true.
\end{enumerate}
}{
\begin{enumerate}
  \item We choose the following partition $c$ on $S$ and compute its pushforward $g_!(c)$:
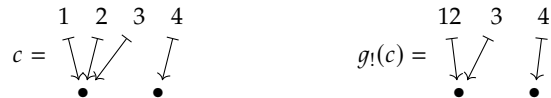
  \[

  \]
	\item Let $d$ be the partition as shown, which was chosen to be coarser than $g_!(c)$.
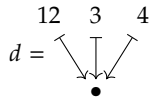
  \[
  %
  \]	
	\item Let $e$ be the partition as shown, which was chosen to \emph{not} be coarser than $g_!(c)$.
  \[
  %
	\]
	\item Comparing $c$, the left-hand partition in part 1., with $g^*(d)$ and $g^*(e)$, we indeed have $c\leq g^*(d)$ but $c\not\leq g^*(e)$, as desired.
\end{enumerate}
}

\sol{exc.proof_galois_monad}{
Complete the proof of \cref{prop.galois_monad_comonad} by showing that
\begin{enumerate}
	\item if $f$ is left adjoint to $g$ then for any $q\in Q$, we have $f(g(q))\leq q$
	\item if \cref{eqn.preorder_monad_comonad} holds, then holds $p\leq g(q)$ iff $f(p)\leq q$ holds, for all $p\in P$ and $q\in Q$.
\end{enumerate}
}{
Suppose $P$ and $Q$ are preorders, and that $f\colon P\leftrightarrows Q\cocolon g$ are monotone maps.
\begin{enumerate}
	\item Suppose $f$ is left adjoint to $g$. By definition this means $f(p)\leq q$ iff $p\leq g(q)$, for all $p\in P$ and $q\in Q$. Then starting with the reflexivity fact $g(q)\leq g(q)$, the definition with $p\coloneqq g(q)$ gives $f(g(q))\leq q$ for all $q$.
	\item Suppose that $p\leq g(f(p))$ and $f(g(q))\leq q$ for all $p\in P$ and $q\in Q$. We first want to show that $p\leq g(q)$ implies $f(p)\leq q$, so assume $p\leq g(q)$. Then applying the monotone map $f$ to both sides, we have $f(p)\leq f(g(q))$, and then by transitivity $f(g(q))\leq q$ implies $f(p)\leq q$, as desired. The other direction is similar.
\end{enumerate}
}

\sol{exc.uniqueness_of_adjoints}{
 	\begin{enumerate}
		\item Show that if $f\colon P\to Q$ has a right adjoint $g$, then it is unique up to isomorphism. That means, for any other right adjoint $g'$, we have $g(q)\cong g'(q)$ for all $q\in Q$.
  	\item Is the same true for left adjoints? That is, if $h\colon P\to Q$ has a left adjoint, is it necessarily unique?
	\end{enumerate}
}
{
\begin{enumerate}
	\item Suppose that $f\colon P\to Q$ has two right adjoints, $g,g'\colon Q\to P$. We want to show that $g(q)\cong g'(q)$ for all $q\in Q$. We will prove $g(q)\leq g'(q)$; the inequality $g'(q)\leq g(q)$ is similar. To do this, we use the fact that $p\leq g'(f(p))$ and $f(g(q))\leq q$ for all $p,q$ by \cref{eqn.preorder_monad_comonad}. Then the trick is to reason as follows:
	\[g(q)\leq g'(f(g(q)))\leq g'(q).\]
	\item It is the same for left adjoints.
\end{enumerate}
}

\sol{exc.mimic_proof_joins}{
  Complete the proof of \cref{prop.right_adj_meets} by showing that left adjoints preserve joins.
}{
  Suppose $f\colon P \to Q$ is left adjoint to $g\colon Q \to P$. Let $A\ss P$ be any subset and let $j\coloneqq\bigvee A$ be
  its join. Then since $f$ is monotone $f(a)\leq f(j)$ for all $a \in A$, so $f(j)$ is an
  upper bound for the set $f(A)$. We want to show it is the least upper bound, so take any other upper bound $b$ for $f(A)$, meaning we have $f(a)\leq b$ for all $a\in A$. Then by definition of adjunction, we also have $a\leq g(b)$ for all $a\in A$. By definition of join, we have $j\leq g(b)$. Again by definition of adjunction $f(j)\leq b$, as desired. 
}  
  
\sol{exc.g_really_is_right_adj}{
To be sure that $g$ really is right adjoint to $f$ in \cref{ex.right_adj_not_joins}, there are twelve tiny things to check; do so. That is, for every $p\in P$ and $q\in Q$, check that $f(p)\leq q$ iff $p\leq g(q)$.
}{
We want to show that in the following picture, $g$ is really right adjoint to $f$:
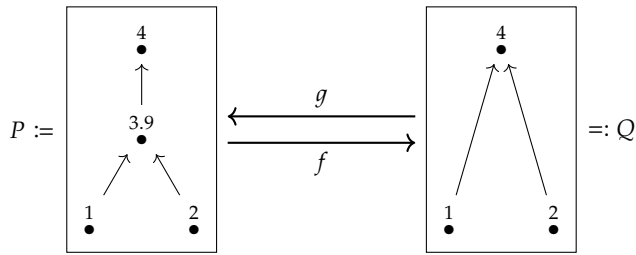
\[

\]
Here $g$ preserves labels and $f$ rounds 3.9 to 4.

There are twelve tiny things to check: for each $p\in P$ and $q\in Q$, we need to see that $f(p)\leq q$ iff $p\leq g(q)$.
\[\footnotesize
\begin{array}{cc|cc|cc|c}
	p&q&f(p)&g(q)&f(p)\leq^?q&p\leq^?g(q)&\textbf{same?}\\\hline
	1&1&1&1&\text{yes}&\text{yes}&\text{yes!}\\
	1&2&1&2&\text{no}&\text{no}&\text{yes!}\\
	1&4&1&4&\text{yes}&\text{yes}&\text{yes!}\\
	2&1&2&1&\text{no}&\text{no}&\text{yes!}\\
	2&2&2&2&\text{yes}&\text{yes}&\text{yes!}\\
	2&4&2&4&\text{yes}&\text{yes}&\text{yes!}\\
	3.9&1&4&1&\text{no}&\text{no}&\text{yes!}\\
	3.9&2&4&2&\text{no}&\text{no}&\text{yes!}\\
	3.9&4&4&4&\text{yes}&\text{yes}&\text{yes!}\\
	4&1&4&1&\text{no}&\text{no}&\text{yes!}\\
	4&2&4&2&\text{no}&\text{no}&\text{yes!}\\
	4&4&4&4&\text{yes}&\text{yes}&\text{yes!}
\end{array}
\]
}  

\sol{exc.subsets_!*}{
Choose sets $X$ and $Y$ with between two and four elements each, and choose a function $f\colon X\to Y$.
\begin{enumerate}
	\item Choose two different subsets $B_1,B_2\ss Y$ and find $f^*(B_1)$ and $f^*(B_2)$.
	\item Choose two different subsets $A_1,A_2\ss X$ and find $f_!(A_1)$ and $f_!(A_2)$.
	\item With the same $A_1,A_2\ss X$, find $f_*(A_1)$ and $f_*(A_2)$.\qedhere
\end{enumerate}
}{
Consider the function shown below, which ``projects straight down'':
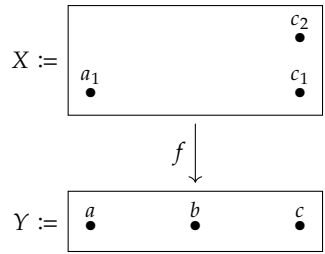
\[
%
\]
\begin{enumerate}
	\item Let $B_1\coloneqq\{a,b\}$ and $B_2\coloneqq\{c\}$. Then $f^*(B_1)=\{a_1\}$ and $f^*(B_2)=\{c_1,c_2\}$.
	\item Let $A_1\coloneqq\varnothing$ and $A_2\coloneqq\{a_1,c_1\}$. Then $f_!(A_1)=\varnothing$ and $f_!(A_2)=\{a,c\}$.
	\item With the same $A_1$ and $A_2$, we compute $f_*(A_1)=\{b\}$ and $f_*(A_2)=\{a,b\}$.
\end{enumerate}
}

\sol{exc.closure}{
Suppose that if $f$ is left adjoint to $g$. Use \cref{prop.galois_monad_comonad} to show the following.
\begin{enumerate}
	\item $p \le (f\cp g)(p)$
	\item $(f\cp g\cp f \cp g)(p) \cong (f\cp g)(p)$. To prove this, show inequalities in both directions, $\leq$ and $\geq$.
\end{enumerate}
}{Assume $f\colon P\to Q$ is left adjoint to $g\colon Q\to P$.
\begin{enumerate}
	\item It is part of the definition of adjunction (\cref{prop.galois_monad_comonad}) that $p\leq g(f(p))$, and of course $g(f(p))$ and $(f\cp g)(p)$ mean the same thing.
	\item We want to show that $g(f(g(f(p))))\leq g(f(p))$ and $g(f(p))\leq g(f(g(f(p))))$ for all $p$. The latter is just the fact that $p'\leq g(f(p'))$ for any $p'$, applied with $g(f(p))$ in place of $p'$. The former uses that $f(g(q))\leq q$, with $f(p)$ substituted for $q$: this gives $f(g(f(p)))\leq f(p)$, and then we apply $g$ to both sides.
\end{enumerate}
}  

\sol{exc.Hasse_binary12}{
Draw the Hasse diagram for the preorder $\Cat{Rel}(\{1,2\})$ of all binary relations on the set $\{1,2\}$.
}
{We denote tuples $(a,b)$ by $ab$ for space reasons. So the relation $\{(1,1),(1,2),(2,1)\}$ will be denoted $\{11,12,21\}$.
\[\footnotesize
\begin{tikzcd}[ampersand replacement=\&, column sep=small]
	\&\&\&[-20pt]\{11,12,21,22\}\&[-20pt]\\
	\&\{11,12,21\}\ar[urr]\&\{11,12,22\}\ar[ur]\&\&\{11,21,22\}\ar[ul]\&\{12,21,22\}\ar[ull]\\
	\{11,12\}\ar[ur]\ar[urr]\&\{11,21\}\ar[u]\ar[urrr]\&\{11,22\}\ar[u]\ar[urr]\&\&\{12,21\}\ar[ulll]\ar[ur]\&\{12,22\}\ar[ulll]\ar[u]\&\{21,22\}\ar[ull]\ar[ul]\\
	\&\{11\}\ar[ul]\ar[u]\ar[ur]\&\{12\}\ar[ull]\ar[urr]\ar[urrr]\&\&\{21\}\ar[ulll]\ar[u]\ar[urr]\&\{22\}\ar[ulll]\ar[u]\ar[ur]\\
	\&\&\&\varnothing\ar[ull]\ar[ul]\ar[ur]\ar[urr]
\end{tikzcd}
\]
}

\sol{exc.understand_adj}{
	Let $S\coloneqq\{1,2,3\}$. Let's try to understand the adjunction discussed above.
	\begin{enumerate}
		\item Come up with any preorder relation $\leq$ on $S$, and define $U(\leq)$ to be the subset $U(\leq)\coloneqq\{(s_1,s_2)\mid s_1\leq s_2\}\ss S\times S$, i.e.\ $U(\leq)$ is the image of $\leq$ under the inclusion $\Cat{Pos}(S)\to\Cat{Rel}(S)$.
		\item Come up with any two binary relations $Q\ss S\times S$ and $Q'\ss S\times S$ such that $Q\ss U(\leq)$ but $Q'\not\ss U(\leq)$. Note that your choice of $Q,Q'$ do not have to come from preorders.
	\end{enumerate}
	We now want to check that in this case, the closure operation $\Fun{Cl}$ is really left adjoint to the ``underlying relation'' map $U$.
	\begin{enumerate}[resume]
		\item Concretely (without using the assertion that there is some sort of adjunction), show that $\Fun{Cl}(Q)\sqsubseteq\;\leq$, where $\sqsubseteq$ is the order on $\Cat{Pos}(S)$, defined on page \pageref{page.order_preorder}.
		\item Concretely show that $\Fun{Cl}(Q')\not\sqsubseteq\;\leq$.
	\qedhere
\end{enumerate}\end{enumerate}
}
{
Let $S\coloneqq\{1,2,3\}$.
\begin{enumerate}
	\item Let $\leq$ be the preorder with $1\leq 2$, and of course $1\leq 1$, $2\leq 2$, and $3\leq 3$. Then $U(\leq)=\{(1,1),(1,2),(2,2),(3,3)\}$.
	\item Let $Q\coloneqq\{(1,1)\}$ and $Q'\coloneqq\{(2,1)\}$.
	\item The closure $\Fun{Cl}(Q)$ of $Q$ is the smallest preorder containing $(1,1)$, which is $\Fun{Cl}(Q)=\{(1,1),(2,2),(3,3)\}$. Similarly, $\Fun{Cl}(Q')=\{(1,1),(2,1),(2,2),(3,3)\}$. It is easy to see that $\Fun{Cl}(Q)\sqsubseteq\;\leq$ because every ordered pair in $\Fun{Cl}(Q)$ is also in $\leq$.
	\item It is easy to see that $\Fun{Cl}(Q')\not\sqsubseteq\;\leq$ because the ordered pair $(2,1)$ is in $\Fun{Cl}(Q')$ but is not in $\leq$.
\end{enumerate}
}

\finishSolutionChapter

\section[Solutions for Chapter 2]{Solutions for \cref{chap.resource_theory}.}

\sol{exc.monoidal_reals}{
Consider again the preorder $(\RR,\leq)$ from \cref{ex.real_nums_preorder}. Someone proposes $1$ as a monoidal unit and $*$ (usual multiplication) as a monoidal product. But an expert walks by and says ``that won't work.'' Figure out why, or prove the expert wrong!
}
{
The expert is right! The proposal violates property (a) when $x_1 = -1$,
$x_2=0$, $y_1=-1$, and $y_2=1$. Indeed $-1 \leq -1$ and $0 \leq 1$, but $-1*0 =
0\not\leq -1 = -1*1$.
}

\sol{exc.disc_mon_preorder}{
  We said it was easy to check that if $(M,*,e)$ is a commutative monoid then
  $(\Cat{Disc}_M,=, *, e)$ is a symmetric monoidal preorder. Are we telling the truth?
}
{
  To check that $(\Cat{Disc}(M),=,*,e)$ is a symmetric monoidal preorder, we need to check our proposed data obeys conditions (a)-(d) of \cref{def.symm_mon_structure}. Condition (a) just states the tautology that $x_1 \otimes x_2 = x_1 \otimes x_2$, conditions (b) and (c) are precisely the equations \cref{eqn.monoid}, and (d) is the commutativity condition. So we're done. We leave it to you to decide whether we were telling the truth when we said it was easy.
}

\sol{exc.proof_by_wiring}{
The string of inequalities in \cref{eqn.almost_proof} is not quite a proof, because technically there is no such thing as $v+w+u$, for example. Instead, there is $(v+w)+u$ and $v+(w+u)$, and so on.
\begin{enumerate}
	\item Formally prove, using only the rules of symmetric monoidal preorders (\cref{def.symm_mon_structure}), that given the assertions in \cref{eqn.some_random334_assertions}, the conclusion in \cref{eqn.random334_conclusion} follows.
	\item Reflexivity and transitivity (\cref{def.preorder}) should show up in your proof. Make sure you are explicit about where they do.
	\item How can you look at the wiring diagram
	\cref{eqn.random334_string_diag} and know that the symmetry axiom
	(\cref{def.symm_mon_structure}(d)) does not need to be invoked?
}
{
\begin{enumerate}
\item Here is a line by line proof, where we write the reason for each step in
parentheses on the right. Recall we call the properties
(a) and (c) in \cref{def.symm_mon_structure} \emph{monotonicity} and
\emph{associativity} respectively.
\begin{align*}
	t+u & \leq (v+w)+u \tag{monotonicity, $t \leq
	v+w$, $u \leq u$} \\
	& = v+(w+u) \tag{associativity} \\
	& \leq v+(x+z) \tag{monotonicity, $v \leq v$,
	$w+u \leq x+v$} \\
	& = (v+x)+z \tag{associativitiy} \\
	& \leq y+z. \tag{monotonicity, $v+x\leq y$, $z
	\leq z$}
\end{align*}
\item We use reflexivity when we assert that $u \leq u$, $v \leq v$ and $z \leq
z$, and use transitivity to assert that the above sequence of inequalities
implies the single inequality $t+u \leq y+z$.
\item We know that the symmetry axiom is not necessary because no pair of wires
cross.
\end{enumerate}
}

\sol{exc.reaction_equations}{
Here is an exercise for people familiar with reaction equations: check that
conditions (a), (b), (c), and (d) of \cref{def.symm_mon_structure} hold.
}
{
Condition (a), monotonicity, says that if $x \rightarrow y$ and $z \rightarrow
w$ are reactions, then $x+z \rightarrow y+w$ is a reaction. Condition (b),
unitality, holds as $0$ represents having no material, and adding no material
to some other material does not change it. Condition (c),
associativity, says that when combining three collections $x$, $y$, and $z$ of
molecules it doesn't matter whether you combine $x$ and $y$ and then $z$, or
combine $x$ with $y$ already combined with $z$. Condition (d), symmetry, says
that combining $x$ with $y$ is the same as combining $y$ with $x$. All these
are true in our model of chemistry, so $(\Set{Mat},\rightarrow,0,+)$ forms a symmetric monoidal
preorder.
}

\sol{exc.boolean_mon_pos_II}{
Let $(\BB,\leq)$ be as above, but now consider the monoidal product to be $\vee$ (OR). 
\[
\begin{array}{c | c c}
	\vee&\false&\true\\\hline
	\false&\false&\true\\
	\true&\true&\true
\end{array}
\hspace{1in}
\begin{array}{c | c c}
	\max&0&1\\\hline
	0&0&1\\
	1&1&1
\end{array}
\]
What must the monoidal unit be in order to satisfy the conditions of \cref{def.symm_mon_structure}? 

Does it satisfy the rest of the conditions?
}
{
The monoidal unit must be $\false$. The symmetric monoidal preorder does satisfy
the rest of the conditions; this can be verified just by checking all cases.
}

\sol{exc.monoidal_naturals}{
Show there is a monoidal structure on $(\NN,\leq)$ where the monoidal product is $*$, i.e.\ $6*4=24$. What should the monoidal unit be?
}
{
The monoidal unit is the natural number $1$. Since we know that $(\NN,\leq)$ is
a preorder, we just need to check that $*$ is monotonic, associative, unital
with $1$, and symmetric. These are all familiar facts from arithmetic.
}

\sol{exc.monoidal_naturals2}{
Again taking the divisibility order $(\NN,\vert\;)$. Someone proposes $0$ as the monoidal unit and $+$ as the monoidal product. Does that proposal satisfy the conditions of \cref{def.symm_mon_structure}? Why or why not?
}
{
This proposal is not monotonic: we have $1\vert1$ and $1 \vert2$, but $(1+1)\nmid(1+2)$.
}

\sol{exc.no_maybe_yes}{
Consider the preorder $(P,\leq)$ with Hasse diagram \fbox{$\const{no}\to\const{maybe}\to\const{yes}$}. We propose a monoidal structure with $\const{yes}$ as the monoidal unit and ``min'' as the monoidal product.
\begin{enumerate}
	\item Make sense of ``$\min$'' by filling in the multiplication table with elements of $P$.
\[
\begin{array}{c|ccc}
	\min&\const{no}&\const{maybe}&\const{yes}\\\hline
	\const{no}&\?&\?&\?\\
	\const{maybe}&\?&\?&\?\\
	\const{yes}&\?&\?&\?
\end{array}
\]
	\item Check the axioms of \cref{def.symm_mon_structure} hold for $\Cat{NMY}\coloneqq(P,\leq,\const{yes},\min)$, given your definition of $\min$.
\end{enumerate}
}
{
\begin{enumerate}
	\item
\[
\begin{array}{c|ccc}
	\min&\const{no}&\const{maybe}&\const{yes}\\\hline
	\const{no}&\const{no}&\const{no}&\const{no}\\
	\const{maybe}&\const{no}&\const{maybe}&\const{maybe}\\
	\const{yes}&\const{no}&\const{maybe}&\const{yes}
\end{array}
\]
	\item We need to show that (a) if $x \leq y$ and $z \leq w$, then
	$\min(x,z) \leq \min(y,w)$, (b) $\min(x,\const{yes}) = x =
	\min(\const{yes},x)$, (c) $\min(\min(x,y),z) = \min(x,\min(y,z))$, and
	(d) $\min(x,y) = \min(y,z)$. The most straightforward way is to just
	check all cases. 
\end{enumerate}
}

\sol{exc.powerset_symm_mon_pos}{
Let $S$ be a set and let $\powset(S)$ be its powerset, the set of all subsets of $S$, including the empty subset, $\varnothing\ss S$, and the ``everything'' subset, $S\ss S$. We can give $\powset(S)$ an order: $A\leq B$ is given by the subset relation $A\ss B$, as discussed in \cref{ex.powerset}. We propose a symmetric monoidal structure on $\powset(S)$ with monoidal unit $S$ and monoidal product given by intersection $A\cap B$.

Does it satisfy the conditions of \cref{def.symm_mon_structure}?
}
{
Yes, $(\powset(S),\leq,S,\cap)$ is a symmetric monoidal preorder.
}

\sol{exc.propositions_preorder}{
Let $\Prop^{\NN}$ denote the set of all mathematical statements one can make
about a natural number, where we consider two statements to be the same if one
is true if and only if the other is true. For example ``$n$ is prime'' is an
element of $\Prop^\NN$, as is ``$n=2$'' and ``$n\geq 11$.'' The statements ``$n
\leq 12$'' and ``$n +1\leq10+3$'' are considered the same. Given
$P,Q\in\Prop^\NN$, we say $P\leq Q$ if for all $n\in\NN$, whenever $P(n)$ is
true, so is $Q(n)$.

Define a monoidal unit and a monoidal product on $\Prop^\NN$ that satisfy the
conditions of \cref{def.symm_mon_structure}.
}
{
Depending on your mood, you might come up with either of the following. First,
we could take the monoidal unit to be some statement $\true$ that is true for
all natural numbers, such as ``$n$ is a natural number.'' We can pair this unit with
the monoidal product $\wedge$, which takes statements $P$ and $Q$
and makes the statement $P\wedge Q$, where $(P\wedge Q)(n)$ is true if $P(n)$
and $Q(n)$ are true, and false otherwise. Then $(\Prop^\NN,\leq,\true,\wedge)$
forms a symmetric monoidal preorder.

Another option is to take define $\false$ to be some statement that is false for
all natural numbers, such as ``$n+10 \leq 1$'' or ``$n$ is made of cheese.'' We
can also define $\vee$ such that $(P\vee Q)(n)$ is true if and only if at least
one of $P(n)$ and $Q(n)$ is true. Then $(\Prop^\NN,\leq,\false,\vee)$ forms a
symmetric monoidal preorder.
}

\sol{exc.complete_op_proof}{
  Complete the proof of \cref{prop.opposite_monoidal_preorder} by proving that the three remaining conditions of \cref{def.symm_mon_structure} are satisfied.
}
{
Unitality and associativity have nothing to do with the order in $\cat{X}\op$:
they simply state that $I \otimes x = x = x \otimes I$, and $(x \otimes y)
\otimes z = x \otimes (y \otimes z)$. Since these are true in $\cat{X}$, they
are true in $\cat{X}\op$. Symmetry is slightly trickier, since in only asks that
$x \otimes y$ is equivalent to $y \otimes x$. Nonetheless, this is still true in
$\cat{X}\op$ because it is true in $\cat{X}$. Indeed the fact that $(x \otimes
y) \cong (y \otimes x)$ in $\cat{X}$ means that $(x \otimes y) \leq (y \otimes
x)$ and $(y \otimes x) \leq (x \otimes y)$ in $\cat{X}$, which respectively
imply that $(y \otimes x) \leq (x \otimes y)$ and $(x \otimes y) \leq (y \otimes
x)$ in $\cat{X}\op$, and hence that $(x \otimes y) \cong (y \otimes x)$ in
$\cat{X}\op$ too.  
}

\sol{exc.costop}{
  Since $\Cost$ is a symmetric monoidal preorder, \cref{prop.opposite_monoidal_preorder} says that $\Cost\op$ is too.
  \begin{enumerate}
    \item What is $\Cost\op$ as a preorder?
    \item What is its monoidal unit?
    \item What is its monoidal product?
\end{enumerate}
}
{
  \begin{enumerate}
    \item The preorder $\Cost\op$ has underlying set $[0,\infty]$, and the usual
    increasing order on real numbers $\le$ as its order.
    \item Its monoidal unit is $0$.
    \item Its monoidal product is $+$.
\end{enumerate}
}

\sol{exc.bool_to_cost}{
\begin{enumerate}
	\item Check that the map $g\colon(\BB,\leq,\true,\wedge)\to([0,\infty],\geq,0,+)$ presented above indeed
  \begin{itemize}
  	\item is monotonic,
  	\item satisfies the condition (a) of \cref{def.monoidal_functor}, and
  	\item satisfies the condition (b) of \cref{def.monoidal_functor}.
  \end{itemize}
  \item Is $g$ strict?
  \qedhere
\qedhere
\end{enumerate}
}
{
\begin{enumerate}
	\item The map $g$ is monotonic as $g(\false) = \infty \geq 0 =
	g(\true)$, satisfies condition (a) since $0 \geq 0 = g(\true)$, and
	satisfies condition (b) since 
	\begin{align*}
	g(\false) + g(\false) = \infty +\infty	&\geq \infty = g(\false \wedge
	\false) \\
	g(\false) + g(\true) = \infty+0 &\geq \infty = g(\false \wedge \true) \\
	g(\true) + g(\false) = 0 +\infty &\geq \infty = g(\true \wedge \false)
	\\
	g(\true) +g(\true) = 0+0 &\geq 0 = g(\true\wedge \true).
	\end{align*}
	\item Since all the inequalities regarding (a) and (b) above are in fact
	equalities, $g$ is a strict monoidal monotone.
\end{enumerate}
}

\sol{exc.bool_to_cost_inverses}{
Let $\Bool$ and $\Cost$ be as above, and consider the following quasi-inverse functions $d,u\colon[0,\infty]\to\BB$ defined as follows:
\[
  d(x)\coloneqq
  \begin{cases}
  	\false&\tn{ if }x>0\\
		\true&\tn{ if } x=0
	\end{cases}
\hspace{1in}
  u(x)\coloneqq
  \begin{cases}
  	\false&\tn{ if }x=\infty\\
		\true&\tn{ if } x<\infty
	\end{cases}
\]
\begin{enumerate}
	\item Is $d$ monotonic?
	\item Does $d$ satisfy conditions a.\ and b.\ of \cref{def.monoidal_functor}?
	\item Is $d$ strict?
	\item Is $u$ monotonic?
	\item Does $u$ satisfy conditions a.\ and b.\ of \cref{def.monoidal_functor}?
	\item Is $u$ strict?
\end{enumerate}
}
{
The answer to all these questions is yes: $d$ and $u$ are both strict monoidal
monotones. Here is one way to interpret this. The function $d$ asks `is $x=0$?'.
This is monotonic, $0$ is $0$, and the sum of two elements of $[0,\infty]$ is
$0$ if and only if they are both $0$. The function $u$ asks `is $x$ finite?'.
Similarly, this is monotonic, $0$ is finite, and the sum of $x$ and $y$ is
finite if and only if $x$ and $y$ are both finite.
}

\sol{exc.natural_monotone}{
\begin{enumerate}
	\item Is $(\NN,\leq,1,*)$ a monoidal preorder, where $*$ is the usual multiplication of natural numbers?
	\item If not, why not? If so, does there exist a monoidal monotone $(\NN,\leq,0,+)\to(\NN,\leq,1,*)$? If not; why not? If so, find it.
	\item Is $(\zz,\leq,*,1)$ a monoidal preorder?
\end{enumerate}
}
{
\begin{enumerate}
	\item Yes, multiplication is monotonic in $\leq$, unital with respect
	  to 1, associative, and symmetric, so $(\NN,\leq,1,*)$ is a monoidal
	  preorder. We also met this preorder in \cref{exc.monoidal_naturals}.
	\item The map $f(n) =1$ for all $n \in \NN$ defines a monoidal monotone
	$f\colon(\NN,\leq,0,+)\to(\NN,\leq,1,*)$. (In fact, it is unique! Why?)
	\item $(\zz,\leq,*,1)$ is not a monoidal preorder because $*$ is not monotone. Indeed $-1\leq 0$ but $(-1*-1)\not\leq (0*0)$.
\end{enumerate}
}

\sol{ex.preorders_as_boolcats}{
  \begin{enumerate}
    \item Start with a preorder $(P,\le)$, and use it to define a $\Bool$-category as we did
      in \cref{ex.preorder_as_boolcat}. In the proof of \cref{thm.preorder_is_bool_cat} we
      showed how to turn that $\Bool$-category back into a preorder. Show that doing so,
      you get the preorder you started with. 

    \item Similarly, show that if you turn a $\Bool$-category into a preorder using the
      above proof, and then turn the preorder back into a $\Bool$-category using your
      method, you get the $\Bool$-category you started with.
      \qedhere
  \end{enumerate}
}
{
  \begin{enumerate}
    \item Let $(P, \le)$ be a preorder. How is this a $\Bool$-category? Following
      \cref{ex.preorder_as_boolcat}, we can construct a $\Bool$-category $\cat{X}_P$
      with $P$ as its set of objects, and with $\cat{X}_P(p,q)\coloneqq \true$ if $p \le q$,
      and $\cat{X}_P(p,q)\coloneqq \false$ otherwise. How do we turn this back into a preorder?
      Following the proof of \cref{thm.preorder_is_bool_cat}, we construct a preorder
      with underlying set $\Ob(\cat{X}_P) = P$, and with $p \le q$ if and only if
      $\cat{X}_P(p,q) = \true$. This is precisely the preorder $(P,\le)$!
    \item Let $\cat{X}$ be a $\Bool$-category. By the proof of
      \cref{thm.preorder_is_bool_cat}, we construct a preorder
      $(\Ob(\cat{X}), \le)$, where $x \le y$ if and only if $\cat{X}(x,y) =
      \true$. Then, following our generalization of
      \cref{ex.preorder_as_boolcat} in 1., we construct a $\Bool$-category
      $\cat{X}'$ whose set of objects is $\Ob(\cat{X})$, and such that
      $\cat{X}'(x,y) = \true$ if and only if $x \le y$ in $(\Ob(\cat{X}), \le)$.
      But by construction, this means $\cat{X}'(x,y) = \cat{X}(x,y)$. So we get
      back the $\Bool$-category we started with.
  \end{enumerate}
}

\sol{ex.regions_of_world}{
Which distance is bigger under the above description, $d(\mathrm{Spain}, \mathrm{US})$ or $d(\mathrm{US}, \mathrm{Spain})$?
}
{
The distance $d(\mathrm{US},\mathrm{Spain})$ is bigger: the distance from, for
example, San Diego to anywhere is Spain is bigger than the distance from
anywhere in Spain to New York City.
}

\sol{exc.finite_Lawvere}{
Consider the symmetric monoidal preorder $(\RR_{\geq0},\geq,0,+)$, which is almost the same as $\Cost$, except it does not include $\infty$. How would you characterize the difference between a Lawvere metric space and a category enriched in $(\RR_{\geq0},\geq,0,+)$?
}
{
The difference between a Lawvere metric space---that is, a category enriched
over $([0,\infty],\geq,0,+)$---and a category enriched over
$(\RR_{\geq0},\geq,0,+)$ is that in the latter, infinite distances are not
allowed between points. You might thus call the latter a finite-distance Lawvere metric
space.
}

\sol{exc.distance_matrix_X}{
Fill out the following table of distances in the weighted graph $X$ from \cref{eqn.cities_distances}
\[
\begin{array}{c|cccc}
  d(\nearrow)&A&B&C&D\\\hline
  A&0&\?&\?&\?\\
  B&2&\?&5&\?\\
  C&\?&\?&\?&\?\\
  D&\?&\?&\?&\?
\end{array}
\]
}
{
The table of distances for $X$ is
\[
\begin{array}{c|cccc}
  d(\nearrow)&A&B&C&D\\\hline
  A&0&6&3&11\\
  B&2&0&5&5\\
  C&5&3&0&8\\
  D&11&9&6&0
\end{array}
\]
}

\sol{exc.adjacency_matrix_X}{
Fill out the matrix $M_X$ associated to the graph $X$ in \cref{eqn.cities_distances}:
\[
M_X=
\begin{array}{c|cccc}
  \nearrow&A&B&C&D\\\hline
  A&0&\?&\?&\?\\
  B&2&0&\infty&\?\\
  C&\?&\?&\?&\?\\
  D&\?&\?&\?&\?
\end{array}
\]
}
{
The matrix of edge weights of $X$ is
\[
M_X=
\begin{array}{c|cccc}
  \nearrow&A&B&C&D\\\hline
  A&0&\infty&3&\infty\\
  B&2&0&\infty&5\\
  C&\infty&3&0&\infty\\
  D&\infty&\infty&6&0
\end{array}
\]
}

\sol{exc.NMY_cats}{
Recall the monoidal preorder $\Cat{NMY}\coloneqq(P,\leq,\const{yes},\min)$ from
\cref{exc.no_maybe_yes}. Interpret what a $\Cat{NMY}$-category is.
}
{
A $\Cat{NMY}$-category $\cat{X}$ is a set $X$ together with, for all pairs of elements
$x,y$ in $X$, a value $\cat{X}(x,y)$ equal to $\const{no}$, $\const{maybe}$, or
$\const{yes}$. Moreover, we must have $\cat{X}(x,x) = \const{yes}$ and
$\min(\cat{X}(x,y), \cat{X}(y,z)) \le \cat{X}(x,z)$ for all $x,y,z$. So a
$\Cat{NMY}$-category can be thought of as set of points together with an
statement---no, maybe, or yes---of whether it is possible to get from one point
to another. In particular, it's always possible to get to a point if you're already
there, and it's as least as possible to get from $x$ to $z$ as it is to get from
$x$ to $y$ and then $y$ to $z$.
}

\sol{ex.modes_of_transport}{
Let $M$ be a set and let $\cat{M}\coloneqq(\powset(M),\ss,M,\cap)$ be the monoidal preorder whose elements are subsets of $M$.

Someone gives the following interpretation, ``for any set $M$, imagine it as the set of modes of transportation (e.g.\ car, boat, foot). Then a category $\cat{X}$ enriched in $\cat{M}$ tells you all the modes that will get you from $a$ all the way to $b$, for any two points $a,b\in\Ob(\cat{X})$.''
\begin{enumerate}
	\item Draw a graph with four vertices and four or five edges, each labeled with a subset of $M=\{\mathrm{car}, \mathrm{boat}, \mathrm{foot}\}$.
	\item This corresponds to an $\cat{M}$-category, where the hom-object
	from $x$ to $y$ is computed as follows: for each path $p$ from $x$ to $y$,
	take the intersection of the settings labelling the edges in $p$. Then, take the
	union of the these sets over all paths $p$ from $x$ to $y$. Write out
	the corresponding four-by-four matrix of hom-objects.
	\item Does the person's interpretation look right, or is it subtly mistaken somehow?
\end{enumerate}
}
{
Here is one way to do this task.
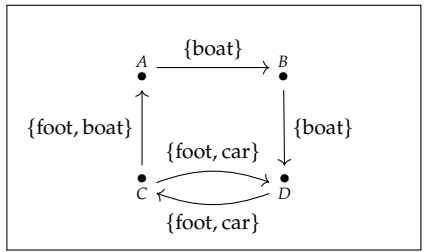
\begin{enumerate}
	\item 	
\[
\begin{tikzpicture}[font=\scriptsize, x=1.5cm]
	\node (a) {$\LMO{A}$};
	\node[right=1 of a] (b) {$\LMO{B}$};
	\node[below=1 of a] (c) {$\LMO[under]{C}$};
	\node[right=1 of c] (d) {$\LMO[under]{D}$};
	\coordinate[left=1 of c] (left);
	\coordinate[right=1 of d] (right);
	\draw[->] (a) to node[above] 
	{$\{\textrm{boat}\}$} (b);
	\draw[->] (b) to node[right] 
	{$\{\textrm{boat}\}$} (d);
	\draw[->] (c) to node[left] 
	{$\{\textrm{foot},\textrm{boat}\}$} (a);
	\draw[bend left=20pt,->] (c) to node[above]
	{$\{\textrm{foot},\textrm{car}\}$} (d);
	\draw[bend left=20pt,->] (d) to node[below]
	{$\{\textrm{foot},\textrm{car}\}$} (c);
	\node[draw, inner ysep=15pt, fit=(left) (right) (a) (b) (c) (d)] (X) {};
\end{tikzpicture}
\]
	\item The corresponding $\cat{M}$-category, call it $\cat{X}$, has hom-objects:
	\[
		\begin{array}{c|cccc}
		\cat{X}(\nearrow)&A&B&C&D\\\hline
  		A&M&\{\textrm{boat}\}&\varnothing&\{\textrm{boat}\}\\
  		B&\varnothing&M&\varnothing&\{\textrm{boat}\}\\
  		C&\{\textrm{foot},\textrm{boat}\}&\{\textrm{boat}\}&M&M\\
  		D&\{\textrm{foot}\}&\varnothing&\{\textrm{foot},\textrm{car}\}&M
		\end{array}
	      \]
	      For example, to compute the hom-object $\cat{X}(C,D)$, we notice that
	      there are two paths: $C \to A \to B \to D$ and $C \to D$. For
	      the first path, the intersection is the set $\{\textrm{boat}\}$.
	      For the second path, the intersection in the set
	      $\{\textrm{foot},\textrm{car}\}$. Their union, and thus the
	      hom-object $\cat{X}(C,D)$, is the entire set $M$.

	      This computation contains the key for why $\cat{X}$ is a
	      $\cat{M}$-category: by taking the union over all paths, we ensure
	      that $\cat{X}(x,y)\cap\cat{X}(y,z) \subseteq \cat{X}(x,z)$ for all
	      $x,y,z$.
	\item The person's interpretation looks right to us.
\end{enumerate}
}

\sol{exc.weight_limit}{
Consider the monoidal preorder $W\coloneqq(\NN\cup\{\infty\}, \leq, \infty,
\min)$.
\begin{enumerate}
	\item Draw a small graph labeled by elements of $\NN\cup\{\infty\}$. 
	\item Write out the matrix whose rows and columns are indexed by the
	nodes in the graph, and whose $xy$th entry is given by the
	\emph{maximum} over all paths $p$ from $x$ to $y$ of the \emph{minimum}
	edge label in $p$.  
	\item Prove that this matrix is the matrix of hom-objects for a
	$\cat{W}$-category. This will give you a feel for how $\cat{W}$ works.
	\item Make up an interpretation, like that in
	\cref{ex.modes_of_transport}, for how to imagine enrichment in
	$\cat{W}$.
\end{enumerate}
\erase{Weight-limit, e.g.\ for trucking.}
}
{
\begin{enumerate}
	\item 
\[
\begin{tikzpicture}[font=\scriptsize, x=2cm,y=1cm]
	\coordinate (x) at (0,0);
	\node[right=.5 of x] (a) {$\LMO{A}$};
	\node[below=1 of x] (b) {$\LMO[under]{B}$};
	\node[right=1 of b] (c) {$\LMO[under]{C}$};
	\draw[->,bend left=12] (b) to node[above left=-5pt] {\tiny $5$} (a);
	\draw[->,bend left=12] (a) to node[above right=-5pt] {\tiny $10$} (c);
	\draw[->,bend left=12] (c) to node[below left=-5pt, pos=.6] {\tiny $10$} (a);
	\draw[->,bend left=12] (b) to node[above=-3pt] {\tiny $6$} (c);
	\draw[->,bend left=12] (c) to node[below=-3pt] {\tiny $10$} (b);
	\node[draw, inner sep=5pt, fit=(a) (b) (c)] (X) {};
\end{tikzpicture}
\]
	\item The matrix $M$ with $(x,y)$\th entry equal to the maximum, taken over all
	paths $p$ from $x$ to $y$, of the minimum edge label in $p$ is
\[
\begin{array}{c|ccc}
  M(\nearrow)&A&B&C\\\hline
  A&\infty&6&10\\
  B&10&\infty&10\\
  C&10&6&\infty
\end{array}
\]
	\item This is a matrix of hom-objects for a $\cat{W}$-category since the
	diagonal values are all equal to the monoidal unit $\infty$, and because
	$\min(M(x,y),M(y,z)) \le M(x,y)$ for all $x,y,z \in \{A,B,C\}$.
	\item One interpretation is as a weight limit (not to be confused with
	`weighted limit,' a more advanced categorical notion), for example for
	trucking cargo between cities. The hom-object indexed by a pair of
	points $(x,y)$ describes the maximum cargo weight allowed on that route.
	There is no weight limit on cargo that remains at some point $x$, so the
	hom-object from $x$ to $x$ is always infinite. The maximum weight that
	can be trucked from $x$ to $z$ is always at least the minimum of that
	that can be trucked from $x$ to $y$ and then $y$ to $z$. (It is `at least' this much because there may be some other, better route that does not pass through $y$.)
\end{enumerate}
}

\sol{exc.regions_preorder}{
Recall the ``regions of the world'' Lawvere metric space from
\cref{ex.regions_of_world} and the text above it. We just learned that we can
convert it to a preorder. Draw the Hasse diagram for the preorder corresponding
to the regions: US, Spain, and Boston. How can you interpret this preorder
relation?
}
{
\[
\begin{tikzpicture}[x=1.5cm, y=1.5cm, font=\footnotesize, text depth=.5ex, text height=1ex]
  	\node[label={[above=-6pt]:{Boston}}] (b) at (0,0) {$\bullet$};
  	\node[label={[above=-6pt]:{US}}] (u) at (1,0) {$\bullet$};
  	\node[label={[above=-6pt]:{Spain}}] (s) at (2,0) {$\bullet$};
  	\node[draw,  fit={($(b.north west)+(-10pt,10pt)$) (u) ($(s.east)+(10pt,0)$)}] (A) {};
	\draw[->] (b) to (u);
\end{tikzpicture}
\]
This preorder describes the `is a part of' relation. That is, $x \le y$ when
$d(x,y) =0$, which happens when $x$ is a part of $y$. So Boston is a part of
the US, and Spain is a part of Spain, but the US is not a part of Boston.
}

\sol{exc.metric_to_poset}{
\begin{enumerate}
	\item Find another monoidal monotone $g\colon\Cost\to\Bool$ different from the one defined in \cref{eqn.Lawv_to_Bool1}.
	\item Using \cref{const.mon_fun_base_change}, both your monoidal monotone $g$
	and the functor $f$ in \cref{eqn.Lawv_to_Bool1} can be used to convert a
	Lawvere metric space into a preorder. Find a Lawvere metric space $\cat{X}$ on which
	they give different answers, $\cat{X}_f\neq\cat{X}_g$.
\end{enumerate}
}
{
\begin{enumerate}
\item Recall the monoidal monotones $d$ and $u$ from
\cref{exc.bool_to_cost_inverses}. The function $f$ is equal to $d$; let $g$ be
equal to $u$.
\item Let $\cat{X}$ be the Lawvere metric space with two objects $A$ and $B$,
such that $d(A,B)=d(B,A)=5$. Then we have 
$\cat{X}_f=
\begin{altikz}[font=\scriptsize]
  	\node (b) at (0,0) {$\LMO{A}$};
  	\node (u) at (1,0) {$\LMO{B}$};
  	\node[draw, fit=(b) (u)] (X) {};
\end{altikz}$
while 
$\cat{X}_g=
\begin{altikz}[font=\scriptsize]
  	\node (b) at (0,0) {$\LMO{A}$};
  	\node (u) at (1,0) {$\LMO{B}$};
  	\node[draw, fit=(b) (u)] (X) {};
	\draw[->] (b) to (u);
\end{altikz}$.
\end{enumerate}
}

\sol{def.enriched_op}{
The concepts of opposite, dagger, and skeleton extend from preorders to
$\cat{V}$-categories. The \emph{opposite} of a $\cat{V}$-category $\cat{X}$ is denoted $\cat{X}\op$
and is defined by
\begin{enumerate}[label=(\roman*)]
	\item $\Ob(\cat{X}\op)\coloneqq\Ob(\cat{X})$, and
	\item for all $x,y\in \cat{X}$, we have $\cat{X}\op(x,y)\coloneqq\cat{X}(y,x)$.
\end{enumerate}
A $\cat{V}$-category $\cat{X}$ is a \emph{dagger}\index{dagger}
$\cat{V}$-category if the identity function is a $\cat{V}$-functor
$\dagger\colon\cat{X}\to\cat{X}\op$.  And a  \emph{skeletal}\index{skeleton}
$\cat{V}$-category is one in which if $I \le \cat{X}(x,y)$ and $I \le
\cat{X}(y,x)$, then $x=y$.

Recall that an extended metric space $(X,d)$ is a Lawvere metric space with two
extra properties; see properties (b) and (c) in \cref{def.ord_metric_space}.
\begin{enumerate}
	\item Show that a skeletal dagger $\Cost$-category is an extended metric space.
	\item Make sense of the following analogy: ``preorders are to sets as
	Lawvere metric spaces are to extended metric spaces.''
\end{enumerate}
}
{
\begin{enumerate}
	\item An extended metric space is a Lawvere metric space that obeys in
	addition the properties (b) if $d(x,y)=0$ then $x=y$, and (c)
	$d(x,y)=d(y,x)$ of \cref{def.ord_metric_space}. Let's consider the
	dagger condition first. It says that the identity function to the
	opposite $\Cost$-category is a functor, and so for all $x,y$ we must
	have $d(x,y) \le d(y,x)$. But this means also that $d(y,x) \le d(x,y)$,
	and so $d(x,y)=d(y,x)$. This is exactly property (c).

	Now let's consider the skeletality condition. This says that if $d(x,y)=0$
	and $d(y,x)=0$, then $x=y$. Thus when we have property (c),
	$d(x,y)=d(y,x)$, this is equivalent to property (b). Thus skeletal
	dagger $\Cost$-categories are the same as extended metric spaces!
	\item Recall from \cref{exc.skeletal_dagger_preorder} that skeletal
	dagger preorders are sets. The analogy ``preorders are to sets as
	Lawvere metric spaces are to extended metric spaces'' is thus the
	observation that just as extended metric spaces are skeletal dagger
	$\Cost$-categories, sets are skeletal dagger $\Bool$-categories. 	
\end{enumerate}
.
}

\sol{exc.check_enriched_prod}{
Let $\cat{X}\times\cat{Y}$ be the $\cat{V}$-product as in \cref{def.enriched_prod}.
\begin{enumerate}
	\item Check that for every object $(x,y)\in\Ob(\cat{X}\times\cat{Y})$ we have $I\leq(\cat{X}\times\cat{Y})((x,y),(x,y))$.
	\item Check that for every three objects $(x_1,y_1)$, $(x_2,y_2),$ and $(x_3,y_3)$, we have
	\[
		(\cat{X}\times\cat{Y})((x_1,y_1),(x_2,y_2))\otimes(\cat{X}\times\cat{Y})((x_2,y_2),(x_3,y_3))\leq (\cat{X}\times\cat{Y})((x_1,y_1),(x_3,y_3)).
		\]
	\item We said at the start of \cref{subsec.V_cats} that the symmetry of $\cat{V}$ (condition (d) of \cref{def.symm_mon_structure}) would be required here. Point out exactly where that condition is used.
\end{enumerate}
}
{
\begin{enumerate}
	\item Let $(x,y) \in \cat{X}\times\cat{Y}$. Since $\cat{X}$ and
	$\cat{Y}$ are $\cat{V}$-categories, we have $I \le \cat{X}(x,x)$ and $I
	\le \cat{Y}(y,y)$. Thus $I = I \otimes I \le \cat{X}(x,x) \otimes
	\cat{Y}(y,y) = (\cat{X} \times \cat{Y})\big((x,y),(x,y)\big)$.
	\item Using the definition of product hom-objects, and the symmetry and
	monotonicity of $\otimes$ we have
	\begin{align*}
		(\cat{X}&\times\cat{Y})\big((x_1,y_1),(x_2,y_2)\big)
		\otimes(\cat{X}\times\cat{Y})\big((x_2,y_2),(x_3,y_3)\big) \\
		&= \cat{X}(x_1,x_2) \otimes \cat{Y}(y_1,y_2) \otimes
		\cat{X}(x_2,x_3) \otimes \cat{Y}(y_2,y_3) \\
		&= \cat{X}(x_1,x_2) \otimes \cat{X}(x_2,x_3) \otimes
		\cat{Y}(y_1,y_2) \otimes \cat{Y}(y_2,y_3) \\
		&\leq \cat{X}(x_1,x_3) \otimes \cat{Y}(y_1,y_3) \\
		&=(\cat{X}\times\cat{Y})\big((x_1,y_1),(x_3,y_3)\big).
	\end{align*}
	\item In particular, we use the symmetry, to
	conclude that $\cat{Y}(y_1,y_2) \otimes	\cat{X}(x_2,x_3) =
		\cat{X}(x_2,x_3) \otimes \cat{Y}(y_1,y_2)$.
\end{enumerate}
}

\sol{exc.not_sqrt40}{
Consider $\RR$ as a Lawvere metric space, i.e.\ as a $\Cost$-category (see \cref{ex.reals_as_metric_space}). Form the $\Cost$-product $\RR\times\RR$. What is the distance from $(5,6)$ to $(-1,4)$? Hint: apply \cref{def.enriched_prod}; the answer is not $\sqrt{40}$.
}
{
We just apply \cref{def.enriched_prod}(ii): $(\RR\times
\RR)\big((5,6),(-1,4)\big) = \RR(5,-1) +\RR(6,4) = 6+2=8$.
}

\sol{exc.closure_is_adj}{
Condition \cref{eqn.monoidal_closed_adj} says precisely that we have a Galois
connection in the sense of \cref{def.galois}. Let's prove this fact. In
particular, we'll prove that a monoidal preorder is monoidal closed if, given
any $v\in V$, the map $-\otimes v\colon V \to V$ given by multiplying with $v$
has a right adjoint. We write this right adjoint $v \multimap -\colon V \to V$.
\begin{enumerate}
\item Using \cref{def.symm_mon_structure}, show that $-\otimes v$ is monotone.
\item Supposing that $\cat{V}$ is closed, show that for all $a,v \in V$ we have $(v \multimap a) \otimes v \le a$.
\item Using 2., show that $v \multimap -$ is monotone.
\item Conclude that a symmetric monoidal preorder is closed if and only if the
monotone map $-\otimes v$ has a right adjoint. \qedhere
\end{enumerate}
}
{
\begin{enumerate}
\item The function $-\otimes v\colon V \to V$ is monotone, because if $u \le u'$ then
$u \otimes v \le u' \otimes v$ by the monotonicity condition (a) in
\cref{def.symm_mon_structure}.
\item Let $a\coloneqq(v \multimap w)$ in
\cref{eqn.monoidal_closed_adj}. The right hand side is thus $(v \multimap w) \leq
(v \multimap w)$, which is true by reflexivity. Thus the left hand side is true
too. This gives $((v \multimap w) \otimes v) \le w$.
\item Let $u \le u'$. Then, using 2., $(v \multimap u) \otimes v \le u \le u'$.
Applying \cref{eqn.monoidal_closed_adj}, we thus have $(v \multimap u) \le (v
\multimap u')$. This shows that the map $(v \multimap -)\colon V \to V$ is
monotone.
\item \cref{eqn.monoidal_closed_adj} is exactly the adjointness condition from
\cref{def.galois}, except for the fact that we do not know $(-\otimes v)$ and $(v
\multimap-)$ are monotone maps. We proved this, however, in items 1 and 3 above.
\end{enumerate}
}

\sol{exc.Bool_monoidal_closed}{
Show that $\Bool=(\BB,\leq,\true,\wedge)$ is monoidal closed.
}
{
We need to find the hom-element. This is given by implication. That is, define
the function $x \imp y$ by the table
\[
\begin{array}{c | c c}
	\imp &\false&\true\\\hline
	\false&\true&\true\\
	\true&\false&\true
\end{array}
\]
Then $(a \wedge v) \le w$ if and only if $a \le (v \imp w)$. Indeed,
if $v=\false$ then $a\wedge \false = \false$, and so the left hand side is
always true. But $(\false \imp w) = \true$, so the right hand side is
always true too. If $v =\true$, then $a \wedge \true = a$ so the left hand side
says $a \le w$. But $(\true \imp w) = w$, so the right hand side is the
same. Thus $\imp$ defines a hom-element as per
\cref{eqn.monoidal_closed_adj}.
}

\sol{exc.Bool_quantale}{
Show that $\Bool=(\BB,\leq,\true,\wedge)$ is a quantale.
}
{
We showed in \cref{exc.Bool_monoidal_closed} that $\Bool$ is symmetric
monoidal closed, and in \cref{exc.boolean_vee_practice} and
\cref{ex.bool_meet_join} that the join is given by the OR operation $\vee$.
Thus $\Bool$ is a quantale.
}

\sol{exc.powerset_quantale}{
Let $S$ be a set and recall the powerset monoidal preorder $(\powset(S),\ss,S,\cap)$ from \cref{exc.powerset_symm_mon_pos}. Is it a quantale?
}
{
Yes, the powerset monoidal preorder $(\powset(S),\ss,S,\cap)$ is a quantale. The
hom-object $B \multimap C$ is given by $\overline{B} \cup C$, where $\overline
B$ is the \emph{complement} of $B$: it contains all elements of $S$ not
contained in $B$. To see that this satisfies \cref{eqn.monoidal_closed_adj},
note that if $(A \cap B) \subseteq C$, then 
\[
A = (A \cap \overline B) \cup (A \cap B) \subseteq \overline B \cup C. 
\]
On the other hand, if $A \subseteq (\overline B \cup C)$, then 
\[
A \cap B \subseteq (\overline B \cup C) \cap B = (\overline B \cap B) \cup (C \cap B) = C \cap B \subseteq C. 
\]
So $(\powset(S),\ss,S,\cap)$ is monoidal closed. Furthermore, joins are given by union of subobjects,
so it is a quantale.
}

\sol{exc.matrix_mult1}{
\begin{enumerate}
	\item What is $\bigvee\varnothing$, which we generally denote $0$, in the case 
  \begin{enumerate}[label=\alph*.]
  	\item $\cat{V}=\Bool=(\BB,\leq,\true,\wedge)$?
  	\item $\cat{V}=\Cost=([0,\infty],\geq,0,+)$?
  \end{enumerate}
  \item What is the join $x\vee y$ in the case
  \begin{enumerate}[label=\alph*.]
  	\item $\cat{V}=\Bool$, and $x,y\in\BB$ are booleans?
  	\item $\cat{V}=\Cost$, and $x,y\in[0,\infty]$ are distances?
	  \qedhere
  \end{enumerate}  
\qedhere
\end{enumerate}
}
{
\begin{enumerate}
	\item[1a.] In $\Bool$, $(\bigvee\varnothing) = \false$, the least element.
	\item[1b.] In $\Cost$, $(\bigvee\varnothing) = \infty$. This is because we
	use the opposite order $\ge$ on $[0,\infty]$, so $\bigvee\varnothing$ is
	the greatest element of $[0,\infty]$. Note that in this case our
	convention from \cref{def.quantale}, where we denote $(\bigvee\varnothing) = 0$, is a bit confusing! Beware!
	\item[2a.] In $\Bool$, $x \vee y$ is the usual join, OR. 
	\item[2b.] In $\Cost$, $x \vee y$ is the minimum $\min(x,y)$. Again
	because we use the opposite order on $[0,\infty]$, the join is the
	greatest number less than or equal to $x$ and $y$.
\end{enumerate}
}

\sol{exc.identity_matrices}{
Write down the $2\times 2$-identity matrix for the quantales $(\nn,\leq,1,*)$, $(\bb,\leq,\true,\wedge)$, and $\Cost=([0,\infty],\geq,0,+)$.
}
{
The $2 \times 2$-identity matrices for $(\nn,\leq,1,*)$, $\Bool$, and $\Cost$ are respectively
  \[
    \begin{pmatrix}
  1&0\\
  0&1
\end{pmatrix},
\quad
\begin{pmatrix}
  \true&\false\\
  \false&\true
\end{pmatrix},
\quad
\mbox{and}
\quad
\begin{pmatrix}
  0&\infty\\
  \infty&0
\end{pmatrix}.
\]
}

\sol{exc.matrix_mult2}{
Let $\cat{V}=(V,\leq,I,\otimes)$ be a quantale. Use \cref{eqn.quantale_matrix_mult} and \cref{prop.properties_closed_mon_preorders} to prove the following.
\begin{enumerate}
	\item Show that for any sets $X$ and $Y$ and $V$-matrix $M\colon X\times Y\to V$, one has $I_X*M=M$.
	\item Prove the associative law: for any matrices $M\colon W\times X\to V$, $N\colon X\times Y\to V$, and $P\colon Y\times Z\to V$, one has $(M*N)*P=M*(N*P)$.
\qedhere
\end{enumerate}
}
{
\begin{enumerate}
\item We first use \cref{prop.properties_closed_mon_preorders} (2) and symmetry to
show that for all $v \in V$, $0 \otimes v = 0$.
\[
0 \otimes v \cong v \otimes 0 \cong \bigg(v \otimes \bigvee_{a \in \varnothing}
a\bigg) = \bigvee_{a \in \varnothing} (v \otimes a) = 0.
\]
Then we may just follow the definition in \cref{eqn.quantale_matrix_mult}:
\begin{align*}
I_X\ast M(x,y) &= \bigvee_{x' \in X} I_X(x,x') \otimes M(x',y) \\
&= \left(I_X(x,x) \otimes M(x,y)\right) \vee \left(\bigvee_{x' \in X, x' \ne x}I_X(x,x')
\otimes M(x',y) \right) \\
&= \left(I \otimes M(x,y)\right) \vee \left(\bigvee_{x'\in X, x'\ne x} 0 \otimes
M(x',y)\right) \\
&= M(x,y) \vee 0 = M(x,y).
\end{align*}
\item This again follows from \cref{prop.properties_closed_mon_preorders} (2) and
symmetry, making use also of the associativity of $\otimes$:
\begin{align*}
((M\ast N) \ast P)(w,z) &= \bigvee_{y \in Y} \bigg(\bigvee_{x \in X} M(w,x)
\otimes N(x,y)\bigg) \otimes P(y,z) \\
&\cong \bigvee_{y \in Y,x \in X} M(w,x) \otimes N(x,y) \otimes P(y,z) \\
&\cong \bigvee_{x \in X} M(w,x) \otimes \bigg(\bigvee_{y \in Y} N(x,y) \otimes
P(y,z) \bigg) \\
&= (M \ast(N \ast P))(w,z).
\end{align*}
\end{enumerate}
}

\sol{exc.computing_distances}{
Recall from \cref{exc.adjacency_matrix_X} the matrix $M_X$, for $X$ as in
\cref{eqn.cities_distances}. Calculate $M_X^2$, $M_X^3$, and $M_X^4$.  Check
that $M_X^4$ is what you got for the distance matrix in
\cref{exc.distance_matrix_X}.
}
{
We have the matrices
\[
M_X=
\begin{pmatrix}
  0&\infty&3&\infty\\
  2&0&\infty&5\\
  \infty&3&0&\infty\\
  \infty&\infty&6&0
\end{pmatrix}
\qquad
M_X^2=
\begin{pmatrix}
  0&6&3&\infty\\
  2&0&5&5\\
  5&3&0&8\\
  \infty&9&6&0
\end{pmatrix}
\qquad
M_X^3=M_X^4=
\begin{pmatrix}
  0&6&3&11\\
  2&0&5&5\\
  5&3&0&8\\
  11&9&6&0
\end{pmatrix}
\]
}

\finishSolutionChapter
\section[Solutions for Chapter 3]{Solutions for \cref{chap.databases}.}

\sol{exc.fks_arrows}{
	Count the number of non-ID columns in \cref{eqn.fav_ex_db}. Count the number of arrows (foreign keys) in \cref{eqn.free_schema}. They should be the same number in this case; is this a coincidence?
}{
There are five non-ID columns in \cref{eqn.fav_ex_db}and five arrows in \cref{eqn.free_schema}. This is not a coincidence: there is always one arrow for every non-ID column.
}

\sol{exc.free_cat}{
For $\free(G)$ to really be a category, we must check that this data we
specified obeys the unitality and associativity properties. Check that these
are obeyed for any graph $G$.
}{
To do this precisely, we should define concatenation technically. If $G=(V,A,s,t)$ is a graph, define a path in $G$ to be a tuple of the form $(v,a_1,\ldots,a_n)$ where $v\in V$ is a vertex, $s(a_1)=v$, and $t(a_i)=s(a_{i+1})$ for all $i\in \{1,\ldots,n-1\}$; the length of this path is $n$, and this definition makes sense for any $n\in\nn$.
 We say that the source of $p$ is $s(p)\coloneqq v$ and the target of $p$ is defined to be
\[
  t(p)\coloneqq
  \begin{cases}
    v\tn{ if } n=0\\
    t(a_n)\tn{ if }n\geq 1
  \end{cases}
\]
Two paths $p=(v,a_1,\ldots,a_m)$ and $q=(w,b_1,\ldots,b_n)$ can be concatenated if $t(p)=s(q)$, in which case the concatenated path $p\cp q$ is defined to be
\[(p\cp q)\coloneqq(v,a_1,\ldots,a_m,b_1,\ldots,b_n).\]

We are now ready to check unitality and associativity. A path $p$ is an identity in $\free(G)$ iff $p=(v)$ for some $v\in V$. It is easy to see from the above that $(v)$ and $(w,b_1,\ldots,b_n)$ can be concatenated iff $v=w$, in which case the result is $(w,b_1,\ldots,b_n)$. Similarly $(v,a_1,\ldots,a_m)$ and $(w)$ can be concatenated iff $w=t(a_m)$, in which case the result is $(v,a_1,\ldots,a_m)$. Finally, for associativity with $p$ and $q$ as above and $r=(x,c_1,\ldots,c_o)$, the formula readily reads that whichever way they are concatenated, $(p\cp q)\cp r$ or $p\cp(q\cp r)$, the result is
\[
	(v,a_1,\ldots,a_m,b_1,\ldots,b_n,c_1,\ldots,c_o).
\]
}

\sol{exc.free_cat2}{
The free category on the graph shown here:%
\footnote{As mentioned above, we elide the difference between the graph and the corresponding free category.}
\begin{equation}\label{eqn.graphs_rand9851}
\Cat{3}\coloneqq{\color{black!20!white}\mathbf{Free}\bigg(}\;\raisebox{-.05in}{\fbox{
\begin{tikzcd}[ampersand replacement=\&]
	\LMO{v_1}\ar[r, "f_1"]\&\LMO{v_2}\ar[r, "f_2"]\&\LMO{v_3}
\end{tikzcd}
}}
{\color{black!20!white}\;\bigg)}
\end{equation}
has three objects and six morphisms: the three vertices and six paths in the graph.

Create six names, one for each of the six morphisms in $\Cat{3}$. Write down a six-by-six table, label the rows and columns by the six names you chose.
\begin{enumerate}
	\item Fill out the table by writing the name of the composite in each cell.
	\item Where are the identities?
\end{enumerate}
}{
We often like to name identity morphisms by the objects they're on, and we do that here: $v_2$ means $\id_{v_2}$. We write $\DNE$ when the composite does not make sense (i.e.\ when the target of the first morphism does not agree with the source of the second).
\[
\begin{array}{c|ccccccc}
	\nearrow		&v_1		&f_1		&f_1\cp f_2		&v_2		&f_2				&v_3				\\\hline
	v_1					&v_1		&f_1		&f_1\cp f_2		&\DNE		&\DNE				&\DNE				\\
	f_1					&\DNE		&\DNE		&\DNE					&f_1		&f_1\cp f_2	&\DNE				\\
	f_1\cp f_2  &\DNE		&\DNE		&\DNE					&\DNE		&\DNE				&f_1\cp f_2	\\
	v_2					&\DNE		&\DNE		&\DNE					&v_2		&f_2				&\DNE				\\
	f_2					&\DNE		&\DNE		&\DNE					&\DNE		&\DNE				&f_2				\\
	v_3					&\DNE		&\DNE		&\DNE					&\DNE		&\DNE				&v_3
\end{array}
\]
}

\sol{exc.Cat_n}{
Let's make some definitions, based on the pattern above:
\begin{enumerate}
	\item What is the category $\Cat{1}$? That is, what are its objects and morphisms?
	\item What is the category $\Cat{0}$?
	\item What is the formula for the number of morphisms in $\Cat{n}$ for arbitrary $n\in\NN$?
\end{enumerate}
}{
\begin{enumerate}
	\item The category $\Cat{1}$ has one object $v_1$ and one morphism, the identity $\id_{v_1}$.
	\item The category $\Cat{0}$ is empty; it has no objects and no morphisms.
	\item The pattern for number of morphisms in $\Cat{0}$, $\Cat{1}$, $\Cat{2}$, $\Cat{3}$ is 0, 1, 3, 6; does this pattern look familiar? These are the first few `triangle numbers,' so we could guess that the number of morphisms in $\Cat{n}$, the free category on the following graph
	\[
	\boxCD
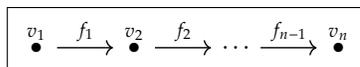
{
		\begin{tikzcd}[ampersand replacement=\&]
		\LMO{v_1}\ar[r, "f_1"]\&\LMO{v_2}\ar[r, "f_2"]\&\cdots\ar[r, "f_{n-1}"]\&\LMO{v_n}
		\end{tikzcd}
	}
	\]
	 is $1+2+\cdots+n$. This makes sense because (and the proof strategy would be to verify that) the above graph has $n$ paths of length 0, it has $n-1$ paths of length $1$, and so on: it has $n-i$ paths of length $i$ for every $0\leq i\leq n$.
\end{enumerate}
}

\sol{exc.nat_comp}{
In \cref{ex.monoid_nats} we identified the paths of the loop graph \eqref{eqn.loop_graph} with numbers $n\in\NN$. Paths can be concatenated. Given numbers $m,n\in\NN$, what number corresponds to the concatenation of their associated paths?
}{
The correspondence was given by sending a path to its length. Concatenating a path of length $m$ with a path of length $n$ results in a path of length $m+n$.
}

\sol{exc.label_free_square}{
\begin{enumerate}
	\item Write down the ten paths in the free square category above.
	\item Name two different paths that are parallel.
	\item Name two different paths that are not parallel.
\end{enumerate}
}{
\[\mathrm{Free\_square}\coloneqq\boxCD{
\begin{tikzcd}[ampersand replacement=\&]
	\LMO{A}\ar[r, "f"]\ar[d, "g"']\&\LMO{B}\ar[d, "h"]\\
	\LMO[under]{C}\ar[r, "i"']\&\LMO[under]{D}
\end{tikzcd}
  \\~\\\footnotesize
  \textit{no equations}
}
\]
\begin{enumerate}
	\item The ten paths are as follows
	\[A,\quad A\cp f,\quad A\cp g,\quad A\cp f\cp h,\quad A\cp g\cp i,\quad B,\quad B\cp h, \quad C,\quad C\cp i,\quad D\]
	\item $A\cp f\cp h$ is parallel to $A\cp g\cp i$, in that they both have the same domain and both have the same codomain.
	\item $A$ is not parallel to any of the other nine paths.
\end{enumerate}
}

\sol{exc.cat_gens_rels}{
Write down all the morphisms in the category presented by the following diagram:
\[
\boxCD{
\begin{tikzcd}[ampersand replacement=\&]
	\LMO{A}\ar[r, "f"]\ar[d, "g"']\ar[dr, "j" description]\&\LMO{B}\ar[d, "h"]\\
	\LMO[under]{C}\ar[r, "i"']\&\LMO[under]{D}
\end{tikzcd}
\\~\\\footnotesize
  $f\cp h=j=g\cp i$
}
\]
}{
The morphisms in the given diagram are as follows:
	\[A,\quad A\cp f,\quad A\cp g,\quad A\cp j,\quad B,\quad B\cp h, \quad C,\quad C\cp i,\quad D\]
Note that $A\cp f\cp h=j=A\cp g\cp i$.
}

\sol{exc.group2}{
Write down all the morphisms in the category $\cat{D}$ from \cref{ex.group_of_order_2}.
}{
There are four morphisms in $\cat{D}$, shown below, namely $z$, $s$, $s\cp s$, and $s\cp s\cp s$:
\[
\cat{D}\coloneqq\boxCD{\begin{tikzcd}[ampersand replacement=\&]
	\LMO[under]{z}\ar[loop above, "s"]
\end{tikzcd}
\\~\\\footnotesize
$s\cp s\cp s\cp s=s\cp s$
}
\]
}

\sol{exc.graph_to_preorder}{
What equations would you need to add to the graphs below in order to present the associated preorders?
\[
G_1=\boxCD{
\begin{tikzcd}[ampersand replacement=\&, column sep=20pt]
	\bullet\ar[r, shift left, "f"]\ar[r, shift right, "g"']\&\bullet
\end{tikzcd}
}
\hspace{.4in}
G_2=\boxCD{
\begin{tikzcd}[ampersand replacement=\&]
	\bullet\ar[loop above, "f"]
\end{tikzcd}
}
\hspace{.4in}
G_3=\boxCD{
\begin{tikzcd}[ampersand replacement=\&, column sep=20pt]
	\bullet\ar[r, "f"]\ar[d, "g"']\&\bullet\ar[d, "h"]\\
	\bullet\ar[r, "i"']\&\bullet
\end{tikzcd}
}
\hspace{.4in}
G_4=\boxCD{
\begin{tikzcd}[ampersand replacement=\&, column sep=20pt]
	\bullet\ar[r, "f"]\ar[d, "g"']\&\bullet\ar[d, "h"]\\
	\bullet\&\bullet
\end{tikzcd}
}
\]
}{
The equations that make the graphs into preorders are shown below
\[
G_1=\boxCD{
\begin{tikzcd}[ampersand replacement=\&, column sep=20pt]
	\bullet\ar[r, shift left, "f"]\ar[r, shift right, "g"']\&\bullet
\end{tikzcd}
\\~\\\footnotesize
  $f=g$
}
\hspace{.4in}
G_2=\boxCD{
\begin{tikzcd}[ampersand replacement=\&]
	\LMO[under]{a}\ar[loop above, "f"]
\end{tikzcd}
\\~\\\footnotesize
  $f=a$
}
\hspace{.4in}
G_3=\boxCD{
\begin{tikzcd}[ampersand replacement=\&, column sep=20pt]
	\bullet\ar[r, "f"]\ar[d, "g"']\&\bullet\ar[d, "h"]\\
	\bullet\ar[r, "i"']\&\bullet
\end{tikzcd}
\\~\\\footnotesize
  $f\cp h=g\cp i$
}
\hspace{.4in}
G_4=\boxCD{
\begin{tikzcd}[ampersand replacement=\&, column sep=20pt]
	\bullet\ar[r, "f"]\ar[d, "g"']\&\bullet\ar[d, "h"]\\
	\bullet\&\bullet
\end{tikzcd}
\\~\\\footnotesize
  \textit{no equations}
}
\]
}

\sol{exc.preorder_refl_N}{
What is the preorder reflection of the category $\NN$ from \cref{ex.monoid_nats}?
}{
The preorder reflection of a category $\cat{C}$ has the same objects and either one morphism or none between two objects, depending on whether or not a morphism between them exists in $\cat{C}$. So the preorder reflection of $\NN$ has one object and one morphism from it to itself, which must be the identity. In other words, the preorder reflection of $\NN$ is $\Cat{1}$.
}

\sol{exc.exponential_practice}{
	Let $\ul{2}=\{1,2\}$ and $\ul{3}=\{1,2,3\}$. These are objects in the
	category $\smset$ discussed in \cref{def.category_of_sets}. Write down all the elements of the set $\smset(\ul{2},\ul{3})$; there should be nine.
}{
A function $f\colon \ul{2}\to\ul{3}$ can be described as an ordered pair $(f(1), f(2))$. The nine such functions are given by the following ordered pairs, which we arrange into a 2-dimensional grid with 3 entries in each dimension, just for ``funzies'':%
\footnote{Of course, this is not \emph{mere funzies}; this is category theory!}
\[
\begin{array}{ccc}
	(1,1)&(1,2)&(1,3)\\
	(2,1)&(2,2)&(2,3)\\
	(3,1)&(3,2)&(3,3)
\end{array}
\]
}

\sol{exc.iso_practice}{
\begin{enumerate}
	\item What is the inverse $f\inv\colon\ul{3}\to A$ of the function $f$ given in \cref{ex.simple_iso}?
	\item How many distinct isomorphisms are there $A\to\ul{3}$?
\end{enumerate}
}{
\begin{enumerate}
	\item The inverse to $f(a)=2$, $f(b)=1$, $f(c)=3$ is given by
	\[f\inv(1)=b,\quad f\inv(2)=a,\quad f\inv(3)=c.\]
	\item There are 6 distinct isomorphisms. In general, if $A$ and $B$ are sets, each with $n$ elements, then the number of isomorphisms between them is $n$-factorial, often denoted $n!$. So for example there are $5*4*3*2*1=120$ isomorphisms between $\{1,2,3,4,5\}$ and $\{a,b,c,d,e\}$.
\end{enumerate}
}

\sol{exc.id_iso}{
Show that in any category $\cat{C}$, for any object $c\in\cat{C}$, the identity $\id_c$ is an isomorphism.
}{
We have to show that for any object $c\in\cat{C}$, the identity $\id_c$ has an inverse, i.e.\ a morphism $f\colon c\to c$ such that $f\cp\id_c=\id_c$ and $\id_c\cp f=\id_c$. Take $f=\id_c$; this works.
}

\sol{exc.monoid_group}{
Recall Examples \ref{ex.monoid_nats} and \ref{ex.group_of_order_2}. A monoid in
which every morphism is an isomorphism is known as a \emph{group}. 
\begin{enumerate}
  \item Is the monoid in \cref{ex.monoid_nats} a group?
  \item What about the monoid $\cat{C}$ in \cref{ex.group_of_order_2}?
\end{enumerate}
}{
\begin{enumerate}
	\item The monoid in \cref{ex.monoid_nats} is not a group, because the morphism $s$ has no inverse. Indeed each morphism is of the form $s^n$ for some $n\in\nn$ and composing it with $s$ gives $s^{n+1}$, which is never $s^0$.$p$
	
	\item $\cat{C}$ from \cref{ex.group_of_order_2} is a group: the identity is always an isomorphism, and the other morphism $s$ has inverse $s$.
\end{enumerate}
}

\sol{exc.iso_free_cat}{
Let $G$ be a graph, and let $\free(G)$ be the corresponding free category. Somebody tells you that the only isomorphisms in $\free(G)$ are the identity morphisms. Is that person correct? Why or why not?
}{
You may have found a person whose mathematical claims you can trust! Whenever you compose two morphisms in $\free(G)$, their lengths add, and the identities are exactly those morphisms whose length is $0$. In order for $p$ to be an isomorphism, there must be some $q$ such that $p\cp q=\id$ and $q\cp p=\id$, in which case the length of $p$ (or $q$) must be 0.
}

\sol{exc.all_functors}{
Above we wrote down three functors $\Cat{2}\to\Cat{3}$. Find and write down all the rest of the functors $\Cat{2}\to\Cat{3}$.
}{
The other three functors $\Cat{2}\to\Cat{3}$ are shown here:
\[
\begin{tikzpicture}[x=.7in, y=.25in, inner sep=5pt]
	\foreach \i in {0,1,2}{
  	\node (A\i-n0) at (3*\i,-1) {$\LMO{m_0}$};
  	\node (A\i-n1) at (3*\i,-3) {$\LMO[under]{m_1}$};
  	\draw[->] (A\i-n0) to node[left=-2pt, font=\scriptsize] {$f_1$} (A\i-n1);
  	\node[draw, inner ysep=1pt, fit=(A\i-n0) (A\i-n1)] (A\i) {};
  	\node (B\i-n0) at (3*\i+1,0) {$\LMO{n_0}$};
  	\node (B\i-n1) at (3*\i+1,-2) {$\LMO{n_1}$};
  	\node (B\i-n2) at (3*\i+1,-4) {$\LMO{n_2}$};
  	\draw[->] (B\i-n0) to node[right=-2pt, font=\scriptsize] {$g_1$} (B\i-n1);
  	\draw[->] (B\i-n1) to node[right=-2pt, font=\scriptsize] {$g_2$} (B\i-n2);
  	\node[draw, inner ysep=1pt, fit=(B\i-n0) (B\i-n2)] (B\i) {};
	}
	\begin{scope}[dotted, thick, ->]
  	\draw (A0-n0) -- (B0-n1);
  	\draw (A0-n1) -- (B0-n1);
		\draw (A1-n0) -- (B1-n2);
  	\draw (A1-n1) -- (B1-n2);
		\draw (A2-n0) -- (B2-n1);
  	\draw (A2-n1) -- (B2-n2);
	\end{scope}
\end{tikzpicture}
\]
}

\sol{exc.functor_on_morphisms}{
Say where each morphism in $\cat{F}$ is sent under the functor $F$ from \cref{ex.free_comm_sq}.
}{
There are nine morphisms in $\cat{F}$; as usual we denote identities by the object they're on. These morphisms are sent to the following morphisms in $\cat{C}$:
\begin{gather*}
	A'\mapsto A,\quad f'\mapsto f,\quad g'\mapsto g,\quad f'\cp h'\mapsto f\cp h,\quad
	g'\cp i'\mapsto f\cp h,\\
	B'\mapsto B,\quad h'\mapsto h,\quad C'\mapsto C,\quad i'\mapsto i,\quad D'\mapsto D.
\end{gather*}
If one of these seems different from the rest, it's probably $g'\cp i'\mapsto f\cp h$. But note that in fact also $g'\cp i'\mapsto g\cp i$ because $g\cp i=f\cp h$, so it's not an outlier after all.
}

\sol{exc.functors_morphisms_practice}{
Consider the free categories $\cat{C}=\fbox{$\bullet\to\bullet$}$ and $\cat{D}=\fbox{$\bullet\tto\bullet$}$. Give two functors $F,G\colon\cat{C}\to\cat{D}$ that act the same on objects but differently on morphisms.
}{
We need to give two functors $F,G$ from\;\;$\LMO{a}\To{f}\LMO{b}$\;\;to\;\;$\LMO{a'}\Tto[15pt]{f_1}{f_2}\LMO{b'}$\;\;whose on-objects parts are the same and whose on-morphisms parts are different. There are only two ways to do this, and we choose one of them:
\[F(a)\coloneqq a',\quad G(a)\coloneqq a', \quad F(b)\coloneqq b',\quad G(b)\coloneqq b',\quad F(f)\coloneqq f_1,\text{ and }G(f)\coloneqq f_2.
\]
The other way reverses $f_1$ and $f_2$.
}

\sol{exc.cat_of_cats}{
Back in the primordial ooze, there is a category $\smcat$ in which \emph{the
objects are themselves categories}. Your task here is to construct this
category.
\begin{enumerate}
  \item Given any category $\cat{C}$, show that there exists a functor $\id_{\cat{C}}\colon
  \cat{C} \to \cat{C}$, known as the \emph{identity functor on $\cat{C}$}, that
  maps each object to itself and each morphism to itself.
\end{enumerate}
Note that a functor $\cat{C} \to \cat{D}$ consists of a function from $\Ob(\cat{C})$ to $\Ob(\cat{D})$ and for each pair of objects $c_1,c_2 \in \cat{C}$ a function from $\cat{C}(c_1,c_2)$ to $\cat{D}(F(c_1),F(c_2))$. 
\begin{enumerate}[resume]
  \item Show that given
  $F\colon \cat{C} \to \cat{D}$ and $G\colon \cat{D} \to \cat{E}$, we can define a
  new functor $F\cp G\colon \cat{C} \to \cat{E}$ just by composing functions. 
  \item Show that there is a category, call it $\smcat$, where the objects are categories, morphisms
  are functors, and identities and composition are given as above.
\end{enumerate}
}
{
\begin{enumerate}
	\item Let $\cat{C}$ be a category. Then defining $\id_{\cat{C}}\colon\cat{C}\to\cat{C}$ by $\id_{\cat{C}}(x)=x$ for every object and morphism in $\cat{C}$ is a functor because it preserves identities $\id_{\cat{C}}(\id_c)=\id_c=\id_{\id{\cat{C}}(c)}$ for each object $c\in\Ob(\cat{C})$, and it preserves composition $\id_{\cat{C}}(f\cp g)=f\cp g=\id_\cat{C}(f)\cp\id_{\cat{C}(g)}$ for each pair of composable morphisms $f,g$ in $\cat{C}$.
	\item Given functors $F\colon\cat{C}\to\cat{D}$ and $G\colon\cat{D}\to\cat{E}$, we need to show that $F\cp G$ is a functor, i.e.\ that it preserves preserves identities and compositions. If $c\in\cat{C}$ is an object then $(F\cp G)(\id_c)=G(F(\id_c))=G(\id_{F(c)})=\id_{G(F(c))}$ because $F$ and $G$ preserve identities. If $f,g$ are composable morphisms in $\cat{C}$ then
	\[(F\cp G)(f\cp g)=G(F(f)\cp F(g))=G(F(f))\cp G(F(g))\]
	because $F$ and $G$ preserve composition.
	\item We have proposed objects, morphisms, identities, and a composition formula for a category $\smcat$: they are categories, functors, and the identities and compositions given above. We need to check that the two properties, unitality and associativity, hold. So suppose $F\colon\cat{C}\to\cat{D}$ is a functor and we pre-compose it as above with $\id_{\cat{C}}$; it is easy to see that the result will again be $F$, and similarly if we post-compose $F$ with $\id_{\cat{D}}$. This gives unitality, and associativity is just as easy, though more wordy. Given $F$ as above and $G\colon\cat{D}\to\cat{E}$ and $H\colon\cat{E}\to\cat{F}$, we need to show that $(F\cp G)\cp H=F\cp (G\cp H)$. It's a simple application of the definition: for any $x\in\cat{C}$, be it an object or morphism, we have
	\[((F\cp G)\cp H)(c)=H((F\cp G)(c))=H(G(F(c)))=(G\cp H)(F(c))=(F\cp(G\cp H))(c).\]
\end{enumerate}
}

\sol{ex.set_1}{
Let $\Cat{1}$ denote the category with one object, called 1, one identity morphism $\id_1$, and no other morphisms. For any functor $F\colon\Cat{1}\to\smset$ one can extract a set $F(1)$. Show that for any set $S$, there is a functor $F_S\colon\Cat{1}\to\smset$ such that $F_S(1)=S$.
}{
Let $S\in\smset$ be a set. Define $F_S\colon\Cat{1}\to\smset$ by $F_S(1)=S$ and $F_S(\id_1)=\id_S$. With this definition, $F_S$ preserves identities and compositions (the only compositions in $\Cat{1}$ is the composite of the identity with itself), so $F_S$ is a functor with $F_S(1)=S$ as desired.
}

\sol{exc.schema_sense}{
Above, we thought of the sort of data that would make sense for the schema \eqref{eqn.idempotent}. Give an example of the sort of data that would make sense for the following schemas:\qquad
\begin{enumerate*}[itemjoin=\hspace{1in}]
\item \boxCD{\begin{tikzcd}[ampersand replacement=\&]
	\LMO[under]{z}\ar[loop above, "s"]
\end{tikzcd}
\\~\\\footnotesize
$s\cp s=z$
}
\item
\boxCD{\begin{tikzcd}[ampersand replacement=\&]
	\LMO{a}\ar[r, "f"]\&\LMO{b}\ar[r, shift left, "g"]\ar[r, shift right, "h"']\&\LMO{c}
\end{tikzcd}
\\~\\\footnotesize
$f\cp g=f\cp h$
}
\end{enumerate*}
}{
We are asked what sort of data ``makes sense'' for the schemas below?\\~\\
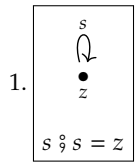
\begin{enumerate*}[itemjoin=\hspace{1in}]
\item \boxCD{\begin{tikzcd}[ampersand replacement=\&]
	\LMO[under]{z}\ar[loop above, "s"]
\end{tikzcd}
\\~\\\footnotesize
$s\cp s=z$
}
\item
\boxCD
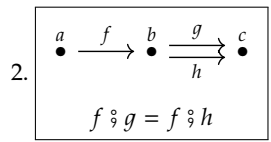
{\begin{tikzcd}[ampersand replacement=\&]
	\LMO{a}\ar[r, "f"]\&\LMO{b}\ar[r, shift left, "g"]\ar[r, shift right, "h"']\&\LMO{c}
\end{tikzcd}
\\~\\\footnotesize
$f\cp g=f\cp h$
}
\end{enumerate*}
\\~\\
This is a subjective question, so we propose an answer for your consideration.
\begin{enumerate}
	\item Data on this schema, i.e.\ a set-valued functor, assigns a set $D(z)$ and a function $D(s)\colon D(z)\to D(z)$, such that applying that function twice is the identity. This sort of function is called an \emph{involution} \index{involution} of the set $D_z$:
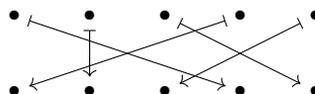
	\[
	\begin{tikzpicture}
	\useasboundingbox (0.9, -1.1) rectangle (5.1, 0.1);
		\foreach \i in {1,2}{
			\foreach \j in {1,...,5}{
				\node at (\j,1-\i) (N\i\j) {$\bullet$};
			}
		}
		\draw[|->] (N11) -- (N24);
		\draw[|->] (N12) -- (N22);
		\draw[|->] (N13) -- (N25);
		\draw[|->] (N14) -- (N21);
		\draw[|->] (N15) -- (N23);
	\end{tikzpicture}
	\]
It's a do-si-do, a ``partner move,'' where everyone picks a partner (possibly themselves) and exchanges with them. One example one could take $D$ to be the set of pixels in a photograph, and take $s$ to be the function sending each pixel to its mirror image across the vertical center line of the photograph.
  \item We could make $D(c)$ the set of people at a ``secret Santa'' Christmas party, where everyone gives a gift to someone, possibly themselves. Take $D(b)$ to be the set of gifts, $g$ the giver function (each gift is given by a person), and $h$ the receiver function (each gift is received by a person), $D(a)$ is the set of people who give a gift to themselves, and $d(f)\colon D(a)\to D(b)$ is the inclusion.
\end{enumerate}
}

\sol{exc.exponential_cat}{
Let's look more deeply at how $\cat{D}^{\cat{C}}$ is a category.
\begin{enumerate}
	\item Figure out how to compose natural transformations. (Hint: an expert tells you ``for each object $c\in\cat{C}$, compose the $c$-components.'')
	\item Propose an identity natural transformation on any object $F\in\cat{D}^\cat{C}$, and check that it is unital.
\end{enumerate}
}{
\begin{enumerate}
	\item The expert packs so much information in so little space! Suppose given three objects $F,G,H\in\cat{D}^{\cat{C}}$; these are functors $F,G,H\colon\cat{C}\to\cat{D}$. Morphisms $\alpha\colon F\to G$ and $\beta\colon G\to H$ are natural transformations. Most beginners seem to think about a natural transformation in terms of its naturality squares, but the main thing to keep in mind is its components; the naturality squares constitute a check that comes later.\\
	
	So for each $c\in\cat{C}$, $\alpha$ has a component $\alpha_c\colon F(c)\to G(c)$ and $\beta$ has a component $\beta_c\colon G(c)\to H(c)$ in $\cat{D}$. The expert has told us to define $(\alpha\cp\beta)_c\coloneqq(\alpha_c\cp\beta_c)$, and indeed that is a morphism $F(c)\to H(c)$.
	
	Now we do the check. For any $f\colon c\to c'$ in $\cat{C}$, the inner squares of the following diagram commute because $\alpha$ and $\beta$ are natural; hence the outer rectangle does too:
	\[
	\begin{tikzcd}[ampersand replacement=\&]
		F(c)\ar[r, "\alpha_c"]\ar[d, "F(f)"']\&
		G(c)\ar[r, "\beta_c"]\ar[d, "G(f)"']\&
		H(c)\ar[d, "H(f)"']\\
		F(c')\ar[r, "\alpha_c"']\&
		G(c')\ar[r, "\beta_c"']\&
		H(c')
	\end{tikzcd}
	\]
	\item We propose that the identity natural transformation $\id_F$ on a functor $F\colon\cat{C}\to\cat{D}$ has as its $c$-component the morphism $(\id_F)_c\coloneqq \id_{F(c)}$ in $\cat{D}$, for any $c$. The naturality square
		\[
	\begin{tikzcd}[ampersand replacement=\&]
		F(c)\ar[r, "\id_{F(c)}"]\ar[d, "F(f)"']\&
		F(c)\ar[d, "F(f)"]\\
		F(c')\ar[r, "\id_{F(c')}"']\&
		F(c')
	\end{tikzcd}
	\]
obviously commutes for any $f\colon c\to c'$. And it is unital: post-composing $\id_F$ with any $\beta\colon F\to G$ (and similarly for precomposing with any $\alpha\colon E\to F$) results in a natural transformation $\id_F\cp\beta$ with components $(\id_F)_c\cp\beta_c=(\id_{F(c)}\cp\beta_c)=\beta_c$, and this is just $\beta$ as desired.
\end{enumerate}
}

\sol{exc.true_false_preorder_nt}{
Let $\cat{C}$ be an arbitrary category and let $\cat{P}$ be a preorder, thought of as a category. Consider the following statements:
\begin{enumerate}
	\item For any two functors $F,G\colon\cat{C}\to\cat{P}$, there is at most one natural transformation $F\to G$.
	\item For any two functors $F,G\colon\cat{P}\to\cat{C}$, there is at most one natural transformation $F\to G$.
\end{enumerate}
For each, if it is true, say why; if it is false, give a counterexample.
}{
We have a category $\cat{C}$ and a preorder $\cat{P}$, considered as a category.
\begin{enumerate}
	\item Suppose that $F,G\colon\cat{C}\to\cat{P}$ are functors and $\alpha,\beta\colon F\to G$ are natural transformations; we need to show that $\alpha=\beta$. It suffices to check that $\alpha_c=\beta_c$ for each object $c\in\Ob(\cat{C})$. But $\alpha_c$ and $\beta_c$ are morphisms $F(c)\to G(c)$ in $\cat{P}$, which is a preorder, and the definition of a preorder---considered as a category---is that it has at most one morphism between any two objects. Thus $\alpha_c=\beta_c$, as desired.
	\item This is false. Let $\cat{P}\coloneqq\Cat{1}$, let $\cat{C}\coloneqq\fbox{$\LMO{a}\Tto[15pt]{f_1}{f_2}\LMO{b}$}$, let $F(1)\coloneqq a$, let $G(1)\coloneqq b$, let $\alpha_1\coloneqq f_1$, and let $\beta_1\coloneqq g_2$. 
\end{enumerate}
}

\sol{exc.graph_instance}{
In \cref{eqn.free_schema}, a graph is shown (forget the distinction between white and black nodes). Write down the corresponding $\Cat{Gr}$-instance, as in \cref{eqn.sample_Gr_instance}. (Do not be concerned that you are in the primordial ooze.)
}{
We need to write down the following
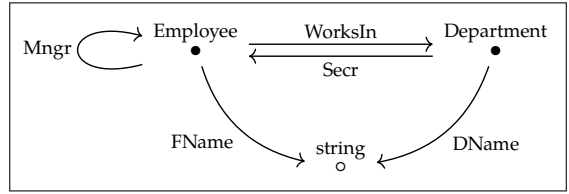
\[\boxCD{
\begin{tikzcd}[row sep=large, ampersand replacement=\&]
  	\LTO{Employee}\ar[rr, shift left, "\text{WorksIn}"]\ar[dr, bend right, "\text{FName}"']\ar[loop left, "\text{Mngr}"]\&\&
  	\LTO{Department}\ar[ll, shift left, "\text{Secr}"]\ar[dl, bend left, "\text{DName}"]\\
  	\&\LTO[\circ]{string}
\end{tikzcd}
}
\]
as a $\Cat{Gr}$-instance, as in \cref{eqn.sample_Gr_instance}. The answer is as follows:
\[
\begin{tabular}{ r | l l}
  \textbf{Arrow}&\textbf{source}&\textbf{target}\\\hline
	Mngr&Employee&Employee\\
	WorksIn&Employee&Department\\
	Secr&Department&Employee\\
	FName&Employee&string\\
	DName&Department&string\\
\end{tabular}
\hspace{1in}
\begin{tabular}{ r |}
	\textbf{Vertex}\\\hline
	Department\\
	Employee\\
	string\\
	~\\
	~
\end{tabular}
\]
}

\sol{exc.unique_alpha}{
We claim that---with $G,H$ as in \cref{ex.two_graphs_as_instances}---there is exactly one graph homomorphism $\alpha\colon G\to H$ such that $\alpha_{\Set{Arrow}}(a)=d$.
\begin{enumerate}
	\item What is the other value of $\alpha_{\Set{Arrow}}$, and what are the three values of $\alpha_{\Set{Vertex}}$?
	\item In your own copy of the tables of \cref{ex.two_graphs_as_instances}, draw $\alpha_{\Set{Arrow}}$ as two lines connecting the cells in the ID column of $G(\Set{Arrow})$ to those in the ID column of $H(\Set{Arrow})$. Similarly, draw $\alpha_{\Set{Vertex}}$ as connecting lines.
	\item Check the source column and target column and make sure that the matches are natural, i.e.\ that ``alpha-then-source equals source-then-alpha'' and similarly for ``target.''
\end{enumerate}
}{
Let $G,H$ be the following graphs:
\[
G\coloneqq\boxCD{
\begin{tikzcd}[ampersand replacement=\&]
	\LMO{1}\ar[r, "a"]\& \LMO{2}\ar[r, "b"]\& \LMO{3}
\end{tikzcd}
}
\hspace{.7in}
H\coloneqq\boxCD{
\begin{tikzcd}[ampersand replacement=\&]
	\LMO{4}\ar[r, shift right, "c"']\ar[r, shift left, "d"]\& \LMO{5}\ar[loop right, "e"]
\end{tikzcd}
}
\]
and let's believe the authors that there is a unique graph homomorphism $\alpha\colon G\to H$ for which $\alpha_{\Set{Arrow}}(a)=d$.
\begin{enumerate}
	\item We have $\alpha_{\Set{Arrow}}(b)=e$ and $\alpha_{\Set{Vertex}}(1)=4$, $\alpha_{\Set{Vertex}}(2)=5$, and $\alpha_{\Set{Vertex}}(3)=5$.
	\item We roughly copy the tables and then draw the lines (shown in black; ignore the dashed lines for now):
	
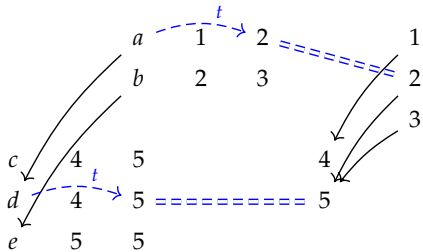
\[
	\begin{tikzcd}[ampersand replacement=\&, column sep=small, row sep=2pt]
			\&	\&a\ar[ddddll, bend right=10pt]\ar[rr, bend left=20pt, blue, dashed, pos=.6, "t"]	\&1	\&2\ar[drrr, dashed, equal, blue]	\&	\&	\&1\ar[dddll, bend right=10pt]\\
			\&	\&b\ar[ddddll, bend right=10pt]	\&2	\&3	\&	\&	\&2\ar[dddll, bend right=10pt]\\
			\&	\&	\&	\&	\&	\&	\&3\ar[ddll, bend right=10pt]\\
		c	\&4	\&5	\&	\&	\&4	\&	\&\\
		d\ar[rr, bend left=20pt, blue, dashed, pos=.6, "t"]	\&4	\&5\ar[rrr, dashed, equal, blue]	\&	\&	\&5	\&	\&\\
		e	\&5	\&5	\&	\&	\&	\&	\&		
	\end{tikzcd}
	\]
	\item It works! One example of the naturality is shown with the help of dashed blue lines above. See how both paths starting at $a$ end at $5$?
\end{enumerate}
}

\sol{exc.dds_graph}{
Consider the functor $G\colon\Cat{Gr}\to\Cat{DDS}$ given by sending `source' to `next' and sending `target' to the identity on `State'. Migrate the same data, \cref{eqn.state_table}, using $G$. Write down the tables and draw the corresponding graph.
}{
We just need to write out the composite of the following functors
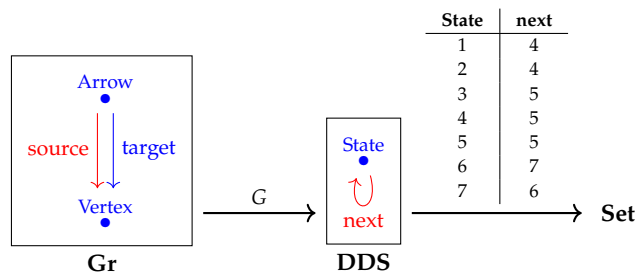
\[
\begin{tikzpicture}[color=blue]
	\node (Arrow) {$\LTO{Arrow}$};
	\node[below=of Arrow] (Vertex) {$\LTO{Vertex}$};
	\draw[->, color=red] 
		($(Arrow.south)+(-3pt,0)$) to node[left, font=\scriptsize] (source) {source} 
		($(Vertex.north)+(-3pt,0)$);
	\draw[->]
		($(Arrow.south)+(3pt,0)$) to node[right, font=\scriptsize] (target) {target} 
		($(Vertex.north)+(3pt,0)$);
	\node[draw, color=black, fit=(Arrow) (Vertex) (source) (target)] (Gr) {};
	\node[below=0 of Gr, black] {$\Cat{Gr}$};
	\node[right=2 of target] (State) {$\LTO{State}$} edge [loop below, red] node[font=\scriptsize] (next) {next} ();
	\begin{scope}[color=black]
  	\node[draw, fit=(State) (next)] (DDS) {};
  	\node[below=0 of DDS] {$\Cat{DDS}$};
  	\draw[functor] (Gr.east|-Vertex) to node[above, font=\scriptsize] {$G$} (DDS.west|-Vertex);
  	\node[right=6 of Vertex] (Set) {$\smset$};
  	\draw[functor] (DDS.east|-Vertex) to 
				node[above, font=\tiny] {
 					\begin{tabular}{c | c}
          	\textbf{State}&\textbf{next}\\\hline
          	1 & 4\\
          	2 & 4\\
          	3 & 5\\
          	4 & 5\\
          	5 & 5\\
          	6 & 7\\
          	7 & 6
        \end{tabular}
        }
				(Set.west|-Vertex);
	\end{scope}
\end{tikzpicture}
\]
in the form of a database, and then draw the graph. The results are given below.
\[\footnotesize
	\begin{array}{c | c c}
		\textbf{Arrow}&\textbf{source}&\textbf{target}\\\hline
		1&4&1\\
		2&4&2\\
		3&5&3\\
		4&5&4\\
		5&5&5\\
		6&7&6\\
		7&6&7
	\end{array}
\hspace{.4in}
	\begin{array}{c |}
		\textbf{Vertex}\\\hline
		1\\
		2\\
		3\\
		4\\
		5\\
		6\\
		7
	\end{array}
\hspace{.6in}
	\boxCD{
  \begin{tikzcd}[column sep=15pt, row sep=5, pos=.9, ampersand replacement=\&]
  	\&\LMO{1}\ar[from=dr, "1"]\&\&\LMO{2}\ar[from=dl,"2"']\\
  	\LMO{3}\ar[from=dr, "3"]\&\&\LMO{4}\ar[from=dl, "4"']\&\&\LMO{6}\ar[from=r, bend left, pos=.5, "6"]\&[20pt]\LMO{7}\ar[from=l, bend left, pos=.5, "7"]\\
  	\&\LMO[under]{5}\ar[loop below, looseness=5, pos=.5, "5"]
  \end{tikzcd}
}
\]
}

\sol{exc.currying_practice}{
In \cref{ex.currying}, we discussed an adjunction between functors $-\times B$
and $(-)^B$. But we only said how these functors worked on objects: for any set
$X$, they return sets $X\times B$ and $X^B$ respectively.
\begin{enumerate}
	\item Given a morphism $f\colon X\to Y$, what morphism should $-\times B\colon X\times B\to Y\times B$ return?
	\item Given a morphism $f\colon X\to Y$, what morphism should $(-)^B\colon X^B\to Y^B$ return?
	\item Consider the function $+\colon\NN\times\NN\to\NN$, which sends
	$(a,b)\mapsto a+b$. Currying $+$, we get a certain function
	$p\colon\NN\to\NN^\NN$. What is $p(3)$?
\end{enumerate}
}{
We are interested in how the functors $-\times B$
and $(-)^B$ should act on morphisms for a given set $B$. We didn't specify this in the text---we only specified $-\times B$ and $(-)^B$ on objects---so in some sense this exercise is open: you can make up anything you want, under the condition that it is functorial. However, the authors cannot think of any such answers except the one we give below.
\begin{enumerate}
	\item Given an arbitrary function $f\colon X\to Y$, we need a function $X\times B\to Y\times B$. We suggest the function which might be denoted $f\times B$; it sends $(x,b)$ to $(f(x),b)$. This assignment is functorial: applied to $\id_X$ it returns $\id_{X\times B}$ and it preserves composition.
	\item Given a function $f\colon X\to Y$, we need a function $X^B\to Y^B$. The canonical function would be denoted $f^B$; it sends a function $g\colon B\to X$ to the composite $(g\cp f)\colon B\to X\to Y$. This is functorial: applied to $\id_X$ it sends $g$ to $g$, i.e.\ $f^B(\id_X)=\id_{X^B}$, and applied to the composite $(f_1\cp f_2)\colon X\to Y\to Z$, we have
	\[(f_1\cp f_2)^B(g)=g\cp (f_1\cp f_2)=(g\cp f_1)\cp f_2=(f_1^B\cp f_2^B)(g)\]
	for any $g\in X^B$.
	\item If $p\colon\nn\to\nn^\nn$ is the result of currying $+\colon\NN\times\NN\to\NN$, then $p(3)$ is an element of $\nn^\nn$, i.e.\ we have $p(3)\colon\nn\to\nn$; what function is it? It is the function that adds three. That is $p(3)(n)\coloneqq n+3$.
\end{enumerate}
}

\sol{ex.terminal_cat}{
Describe the functor $!\colon \cat{C} \to \Cat{1}$ from \cref{eqn.terminal_functor}. Where does it send each
object? What about each morphism?
}{
The functor $!\colon \cat{C} \to \Cat{1}$ from \cref{eqn.terminal_functor} sends each object $c\in\cat{C}$ to the unique object $1\in\Cat{1}$ and sends each morphism $f\colon c\to d$ in $\cat{C}$ to the unique morphism $\id_1\colon 1\to 1$ in $\Cat{1}$.
}

\sol{exc.draw_graph}{
Note that $\cat{G}$ from \cref{eqn.graph_email} is isomorphic to the schema $\Cat{Gr}$. In \cref{subsec.instances_cat} we saw that instances on $\Cat{Gr}$ are graphs. Draw
the above instance $I$ as a graph.
}{
We want to draw the graph corresponding to the instance $I\colon\cat{G}\to\smset$ shown below:
\[
\begin{tabular}{ c | c c}
  \textbf{Email}&\textbf{sent\_by}&\textbf{received\_by}\\\hline
	Em\_1 & Bob & Grace\\
	Em\_2 & Grace & Pat\\
	Em\_3 & Bob & Emory\\
	Em\_4 & Sue & Doug\\
	Em\_5 & Doug & Sue\\
	Em\_6 & Bob & Bob
\end{tabular}
\hspace{1in}
\begin{tabular}{ c |}
	\textbf{Address}\\\hline
	Bob\\
	Doug\\
	Emory\\
	Grace\\
	Pat\\
	Sue
\end{tabular}
\]
Here it is, with names and emails shortened (e.g.\ B=Bob, 3=Em\_3):
\[
\boxCD{
\begin{tikzcd}[ampersand replacement=\&,row sep=1.5ex]
  E\&
  B\ar[r, "1"]\ar[l, "3"']\ar[loop above, "6"]\&
  G\ar[r, "2"]\&
  P\&
  S\ar[r, bend left, "4"]\&
  D\ar[l, bend left, "5"]
\end{tikzcd}
}
\]
}

\sol{exc.terminal_in_preorder}{
Let $(P,\le)$ be a preorder, let $z\in P$ be an element, and let $\cat{P}$ be the corresponding category (see \cref{subsubsec.pos_free_spectrum}). Show that $z$ is a terminal object in $\cat{P}$ if and
only if it is a \emph{top element} in $P$: that is, if and only if for all $c \in P$ we have $c \le
z$.
}{
An object $z$ is terminal in some category $\cat{C}$ if, for every $c\in\cat{C}$ there exists a unique morphism $c\to z$. When $\cat{C}$ is the category underlying a preorder, there is at most one morphism between any two objects, so the condition simplifies: an object $z$ is terminal iff, for every $c\in\cat{C}$ there exists a morphism $c\to z$. The morphisms in a preorder are written with $\leq$ signs, so $z$ is terminal iff, for every $c\in P$ we have $c\leq z$, and this is the definition of top element.
}

\sol{exc.terminal_cat}{
What is the terminal object in the category $\smcat$? (Hint: recall
\cref{ex.terminal_cat}.)
}{
The terminal object in $\smcat$ is $\Cat{1}$ because by \cref{ex.terminal_cat} there is a unique morphism (functor) $\cat{C}\to\Cat{1}$ for any object (category) $\cat{C}\in\smcat$.
}

\sol{exc.cat_wo_terminal}{
Not every category has a terminal object. Find one that doesn't.
}{
Consider the graph $2V\coloneqq\fbox{$\bullet\;\bullet$}$ with two vertices and no arrows, and let $\cat{C}=\free(2V)$; it has two objects and two morphisms (the identities). This category does not have a terminal object because it does not have any morphisms from one object to the other.
}

\sol{exc.meet_product}{
Let $(P,\le)$ be a preorder, let $x,y \in P$ be elements, and let $\cat{P}$ be the corresponding category. Show that their product $x\times y$ in $\cat{P}$ agrees with their meet $x\wedge y$ in $P$.
}{
A product of $x$ and $y$ in $\cat{P}$ is an object $z\in\cat{P}$ equipped with maps $z\to x$ and $z\to y$ such that for any other object $z'$ and maps $z'\to x$ and $z'\to y$, there is a unique morphism $z'\to z$ making the evident triangles commute. But in a preorder, the maps are denoted $\leq$, they are unique if they exist, and all diagrams commute. Thus the above becomes: a product of $x$ and $y$ in $\cat{P}$ is an object $z$ with $z\leq x$ and $z\leq y$ such that for any other $z'$, if $z'\leq x$ and $z'\leq y$ then $z'\leq z$. This is exactly the definition of meet, $z=x\wedge y$.
}

\sol{exc.product_cats}{
\begin{enumerate}
  \item What are the identity morphisms in a product category $\cat{C}\times\cat{D}$?
  \item Why is composition in a product category associative?
  \item What is the product category $\Cat{1} \times \Cat{2}$?
  \item What is the product category $\cat{P} \times \cat{Q}$ when $P$ and $Q$ are preorders and $\cat{P}$ and $\cat{Q}$ are the corresponding categories?
\end{enumerate}
}{
\begin{enumerate}
  \item The identity morphism on the object $(c,d)$ in the product category $\cat{C}\times\cat{D}$ is $(\id_c,\id_d)$.
  \item Suppose given three composable morphisms in $\cat{C}\times\cat{D}$
  \[(c_1,d_1)\To{(f_1,g_1)}(c_2,d_2)\To{(f_2,g_2)}(c_3,d_3)\To{(f_3,g_3)}(c_4,d_4).\]
  We want to check that $((f_1,g_1)\cp(f_2,g_2))\cp(f_3,g_3)=
  (f_1,g_1)\cp((f_2,g_2)\cp(f_3,g_3))$. But composition in a product category is given component-wise. That means the left-hand side is $((f_1\cp f_2)\cp f_3, (g_1\cp g_2)\cp g_3)$, whereas the right-hand side is $(f_1\cp(f_2\cp f_3),g_1\cp(g_2\cp g_3))$, and these are equal because both $\cat{C}$ and $\cat{D}$ individually have associative composition.
  \item The product category $\Cat{1} \times \Cat{2}$ has two objects $(1,1)$ and $(1,2)$ and one non-identity morphism $(1,1)\to(1,2)$. It is not hard to see that it looks the same as $\Cat{2}$. In fact, for any $\cat{C}$ there is an isomorphism of categories $\Cat{1}\times\cat{C}\cong\cat{C}$.
	\item Let $P$ and $Q$ be preorders, let $X=P\times Q$ be their product preorder as defined in \cref{ex.product_preorder}, and let $\cat{P}$, $\cat{Q}$, and $\cat{X}$ be the corresponding categories. Then $\cat{X}=\cat{P}\times\cat{Q}$.
\end{enumerate}
}

\sol{exc.prod_as_term_cone}{
Check that a product $X \xleftarrow{p_X} X\times Y\xrightarrow{p_Y} Y$ is
exactly the same as a terminal object in $\Cat{Cone}(X,Y)$.
}{
A product of $X$ and $Y$ is an object $Z$ equipped with morphisms $X\From{p_X} Z\To{p_Y} Y$ such that for any other object $Z'$ equipped with morphisms $X\From{p'_X} Z'\To{p'_Y}Y$, there is a unique morphism $f\colon Z'\to Z$ making the triangles commute, $f\cp p_X=p'_X$ and $f\cp p_Y=p'_Y$. But ``an object equipped with morphisms to $X$ and $Y$'' is exactly the definition of an object in $\Cat{Cone}(X,Y)$, and a morphism $f$ making the triangles commute is exactly the definition of a morphism in $\Cat{Cone}(X,Y)$. So the definition above becomes: a product of $X$ and $Y$ is an object $Z\in\Cat{Cone}(X,Y)$ such that for any other object $Z'$ there is a unique morphism $Z'\to Z$ in $\Cat{Cone}(X,Y)$. This is exactly the definition of $Z$ being terminal in $\Cat{Cone}(X,Y)$.
}

\sol{exc.limit_formula_products}{
Show that the limit formula in \cref{thm.set_limits} works for products. See \cref{ex.prods_as_lims}.
}{
Suppose $\cat{J}$ is the graph $\fbox{$\LMO{v_1}\;\LMO{v_2}$}$ and $D\colon\cat{J}\to\smset$ is given by two sets, $D(v_1)=A$ and $D(v_2)=B$ for sets $A,B$. The product of these two sets is $A\times B$. Let's check that the limit formula in \cref{thm.set_limits} gives the same answer. It says
\begin{multline*}
  \lim_\cat{J} D \coloneqq \big\{(d_1,\ldots,d_n)\mid d_i\in D(v_i)\text{ for all }1\leq i\leq n\text{ and }\\
  \text{for all }a\colon v_i\to v_j\in A, \text{ we have } D(a)(d_i)=d_j\big\}.
\end{multline*}
But in our case $n=2$, there are no arrows in the graph, and $D(v_1)=A$ and $D(v_2)=B$. So the formula reduces to
\[\lim_\cat{J} D \coloneqq \big\{(d_1,d_2)\mid d_1\in A\text{ and }d_2\in B\big\}.
\]
which is exactly the definition of $A\times B$.
}

\sol{exc.opposite_functor}{
Recall from \cref{def.opposite_cat} that every category $\cat{C}$ has an
opposite $\cat{C}\op$. Let $F\colon \cat{C} \to \cat{D}$ be a functor. How
should we define its opposite, $F\op\colon \cat{C}\op \to \cat{D}\op$? That is,
how should $F\op$ act on objects, and how should it act on morphisms?
}{
Given a functor $F\colon \cat{C} \to \cat{D}$, we define its opposite $F\op\colon \cat{C}\op \to \cat{D}\op$ as follows. For each object $c\in\Ob(\cat{C}\op)=\Ob(\cat{C})$, put $F\op(c)\coloneqq F(c)$. For each morphism $f\colon c_1\to c_2$ in $\cat{C}\op$, we have a corresponding morphism $f'\colon c_2\to c_1$ in $\cat{C}$ and thus a morphism $F(f')\colon F(c_2)\to F(c_1)$ in $\cat{D}$, and thus a morphism $F(f')'\colon F\op(c_1)\to F\op(c_2)$. Hence we can define $F\op(f)\coloneqq F(f')'$. Note that the primes ($-'$) are pretty meaningless, we only put them there to differentiate between things that are very closely related.

It is easy to check that our definition of $F\op$ is functorial: it sends identities to identities and composites to composites.
}
	
\finishSolutionChapter

\section[Solutions for Chapter 4]{Solutions for \cref{chap.codesign}.}

\sol{exc.a_profunctor}{
Suppose we have the preorders
\[
\begin{tikzpicture}
	\node (cat) {category};
	\node[below right=1 of cat] (pos) {preorder};
	\node[below left=1 of cat] (mon) {monoid};
	\draw[->] (pos) -- (cat);
	\draw[->] (mon) -- (cat);
	\node[draw, fit=(cat) (pos) (mon)] (X) {};
	\node[left=0 of X] {$\cat{X}\coloneqq$};
	\node[right=2 of pos] (0) {nothing};
	\node at (0|-cat) (1) {this book};
	\draw[->] (0) -- (1);
	\node[draw, fit=(0) (1)] (Y) {};
	\node[left=0 of Y] {$\cat{Y}\coloneqq$};
\end{tikzpicture}
\]
\begin{enumerate}
  \item Draw the Hasse diagram for the preorder $\cat{X}\op\times\cat{Y}$. 
  \item Write down a profunctor $\Lambda\colon\cat{X} \tickar \cat{Y}$ and, reading
  $\Lambda(x,y)=\true$ as ``my aunt can explain a $x$ given $y$,'' give an
  interpretation of the fact that the preimage of $\true$ forms an upper set in
  $\cat{X}\op\times\cat{Y}$. 
  
\end{enumerate}
}{
\begin{enumerate}
\item The Hasse diagram for $\cat{X}\op \times \cat{Y}$ is shown here (ignore the colors): 
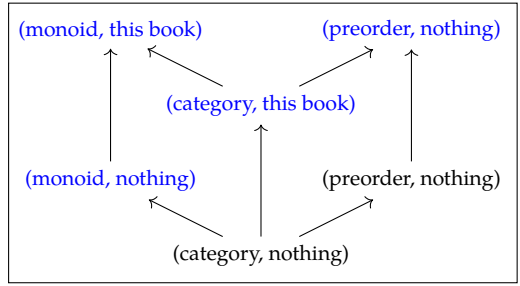
\[
\begin{tikzpicture}[font=\scriptsize,xscale =2]
  	\node (a1) at (0,0) {(category, nothing)};
  	\node[blue] (c1) at (-1,1) {(monoid, nothing)};
  	\node (b1) at (1,1) {(preorder, nothing)};
  	\node[blue] (a2) at (0,2) {(category, this book)};
  	\node[blue] (c2) at (-1,3) {(monoid, this book)};
  	\node[blue] (b2) at (1,3) {(preorder, nothing)};
  	\node[draw, fit=(a1) (c1) (b1) (b2)] (A) {};
	\draw[->] (a1) to (c1);
	\draw[->] (a1) to (a2);
	\draw[->] (a1) to (b1);
	\draw[->] (a2) to (c2);
	\draw[->] (a2) to (b2);
	\draw[->] (c1) to (c2);
	\draw[->] (b1) to (b2);
\end{tikzpicture}
\]
\item There is a profunctor $\Lambda\colon \cat{X} \tickar \cat{Y}$, i.e.\ a functor $\cat{X}\op\times\cat{Y}\to\BB$, 
such that, in the picture above, blue is sent to true and black is sent to false, i.e.\
\begin{align*}
  \Lambda(\mbox{monoid, nothing})& = \Lambda(\mbox{monoid, this book}) \\
  &=\Lambda(\mbox{preorder, this book})
  &=\Lambda(\mbox{category, this book})
  =\true \\
  \Lambda(\mbox{preorder, nothing}) &=\Lambda(\mbox{category, nothing})
  =\false.
\end{align*}
\end{enumerate}
The preorder $\cat{X}\op\times\cat{Y}$ describes tasks in decreasing difficulty. For example, (we hope) it is easier for my aunt to explain a monoid given this
book than for her to explain a monoid without this book. The profunctor
$\Lambda$ describes possible states of knowledge for my aunt: she can describe monoids without help, categories with help from the book, etc. It is an upper set
because we assume that if she can do a task, she can also do any easier task.
}

\sol{exc.implies_is_hom}{
Show that $\imp$ as defined in \cref{eqn.implies} indeed satisfies \cref{eqn.implies_internal_hom}.
}{
We've done this one before! We hope you remembered how to do it. If not, see
\cref{exc.Bool_monoidal_closed}.
}

\sol{exc.profunctor_def_alt}{
Show that a $\cat{V}$-profunctor (\cref{def.enriched_profunctor}) is
the same as a function $\Phi\colon\Ob(\cat{X})\times\Ob(\cat{Y})\to V$ such that
for any $x,x'\in\cat{X}$ and $y,y'\in\cat{Y}$ the following inequality holds in
$\cat{V}$: 
\[
  \cat{X}(x',x)\otimes\Phi(x,y)\otimes\cat{Y}(y,y')\leq\Phi(x',y').
  
\]
}{
Recall from \cref{def.monoidal_functor} that a $\cat{V}$-functor $\Phi\colon
\cat{X}\op \times \cat{Y} \to \cat{V}$ is a function $\Phi\colon\Ob(\cat{X}\op \times
\cat{Y}) \to \Ob(\cat{V})$ such that for all $(x,y)$ and $(x',y')$ in
$\cat{X}\op\times \cat{Y}$ we have 
\[
(\cat{X}\op \times \cat{Y})\big((x,y),(x',y')\big) \le
\cat{V}\big(\Phi(x,y),\Phi(x',y')\big).
\]
Using the definitions of product $\cat{V}$-category (\cref{def.enriched_prod})
and opposite $\cat{V}$-category (\cref{def.enriched_op}) on the left hand side,
and using \cref{rem.quantale_enriches_itself}, which describes how we are
viewing the quantale $\cat{V}$ as enriched over itself, on the right hand side,
this unpacks to
\[
\cat{X}(x',x) \otimes \cat{Y}(y,y') \le \Phi(x,y) \multimap \Phi(x',y')
\]
Using symmetry of $\otimes$ and the definition of hom-element, \cref{eqn.monoidal_closed_adj}, we see that
$\Phi$ is a profunctor if and only if
\[
  \cat{X}(x',x)\otimes\Phi(x,y)\otimes\cat{Y}(y,y')\leq\Phi(x',y').
\]
}

\sol{exc.Bool_enriched_is_feas}{
Is it true that a $\Bool$-profunctor, as in \cref{def.enriched_profunctor} is exactly the same as a feasibility relation, as in \cref{def.feasibility_relationship}, once you peel back all the jargon? Or is there some subtle difference?
}{
Yes, since a $\Bool$-functor is exactly the same as a monotone map, the definition of
$\Bool$-profunctor and that of feasibility relation line up perfectly!
}

\sol{exc.feas_matrix}{
We can express $\Phi$ as a matrix where the $(m,n)$th entry is the value of
$\Phi(m,n)\in \BB$. Fill out the $\Bool$-matrix:
\[
\begin{array}{c|ccccc}
	\Phi&a&b&c&d&e\\\hline
  N&\?&\?&\?&\?&\true\\
  E&\true&\?&\?&\?&\?\\
  W&\?&\?&\?&\false&\?\\
  S&\?&\?&\?&\?&\?
\end{array}

\]
}{
The feasibility matrix for $\Phi$ is
\[
\begin{array}{c|ccccc}
	\Phi&a&b&c&d&e\\\hline
  N&\true&\false&\true&\false&\true\\
  E&\true&\true&\true&\true&\true\\
  W&\true&\false&\true&\false&\true\\
  S&\true&\true&\true&\true&\true
\end{array}
\]
}

\sol{exc.distance_matrix_Phi}{
Fill out the $\Cost$-matrix:
\[
\begin{array}{c|ccc}
  \Phi&x&y&z\\\hline
  A&\?&\?&20\\
  B&11&\?&\?\\
  C&\?&17&\?\\
  D&\?&\?&\?
\end{array}

\]
}{
The $\Cost$-matrix for $\Phi$ is
\[
\begin{array}{c|ccc}
  \Phi&x&y&z\\\hline
  A&17&20&20\\
  B&11&14&14\\
  C&14&17&17\\
  D&12&9&15
\end{array}
\]
}

\sol{exc.cost_matrix_mult}{
Calculate $M_X^4*M_\Phi*M_Y^3$, remembering to do matrix multiplication according to the $(\min,+)$-formula for matrix multiplication in the quantale $\Cost$; see \cref{eqn.quantale_matrix_mult}.

Your answer should agree with what you got in \cref{exc.distance_matrix_Phi}; does it?
}{
\begin{align*}
\Phi = M^3_X\ast M_\Phi \ast M_Y^2 
&= 
\begin{pmatrix}
  0&6&3&11\\
  2&0&5&5\\
  5&3&0&8\\
  11&9&6&0
\end{pmatrix}
\begin{pmatrix}
  \infty&\infty&\infty\\
  11&\infty&\infty\\
  \infty&\infty&\infty\\
  \infty&9&\infty
\end{pmatrix}
\begin{pmatrix}
  0&4&3\\
  3&0&6\\
  7&4&0
\end{pmatrix} \\
&=
\begin{pmatrix}
  17&20&\infty\\
  11&14&\infty\\
  14&17&\infty\\
  20&9&\infty
\end{pmatrix}
\begin{pmatrix}
  0&4&3\\
  3&0&6\\
  7&4&0
\end{pmatrix} \\
&=
\begin{pmatrix}
  17&20&20\\
  11&14&14\\
  14&17&17\\
  12&9&15
\end{pmatrix}
\end{align*}
}

\sol{exc.bad_humour}{
In the above diagram, the node (g/n, funny) has no dashed blue arrow emerging
from it. Is this valid? If so, what does it mean?
}{
Yes, this is valid: it just means that the profunctor $\Phi\colon (T \times E)
\tickar \$$ does not relate (good-natured, funny) to any element of $\$$. More
formally, it means that $\Phi((\mbox{good-natured, funny}),p) = \false$ for all
$p \in \$ = \{\mbox{\$100K, \$500K, \$1M}\}$. We can interpret this to mean that
it is not feasible to produce a good-natured, funny movie for any of the cost
options presented---so at least not for less than a million dollars.
}

\sol{exc.compose_Lawv_profs}{
Consider the $\Cost$-profunctors $\Phi\colon\cat{X}\tickar\cat{Y}$ and $\Psi\colon\cat{Y}\tickar\cat{Z}$ shown below:
\[

\]
Fill in the matrix:
\[
\begin{array}{c|cccc}
  \Phi\cp\Psi&p&q&r&s\\\hline
  A&\?&24&\?&\?\\
  B&\?&\?&\?&\?\\
  C&\?&\?&\?&\?\\
  D&\?&\?&9&\?
\end{array}

\]
}{
There are a number of methods that can be used to get the correct answer. One
way that works well for this example is to search for the shortest paths on the
diagram: it so happens that all the shortest paths go through the bridges from
$D$ to $y$ and $y$ to $r$, so in this case $(\Phi\cp\Psi)(-,-) =
\cat{X}(-,D)+9+\cat{Z}(r,-)$. This gives:
\[
\begin{array}{c|cccc}
  \Phi\cp\Psi&p&q&r&s\\\hline
  A&22&24&20&21\\
  B&16&18&14&15\\
  C&19&21&17&18\\
  D&11&13&9&10
\end{array}
\]
A more methodical way is to use matrix multiplication. Here's one way you might
do the multiplication, using a few tricks.
\begin{align*}
\Phi\cp\Psi &= (M_X^3 \ast M_\Phi \ast M_Y^2) \ast (M_Y^2 \ast M_\Psi \ast M_Z^3) \\
&= M_X^3 \ast M_\Phi \ast M_Y^4 \ast M_\Psi \ast M_Z^3 \\
&= (M_X^3 \ast M_\Phi \ast M_Y^2) \ast M_\Psi \ast M_Z^3 \\
&= \Phi \ast M_\Psi \ast M_Z^3 \\
&=
%
\end{align*}
}

\sol{exc.draw_a_bridge}{
Choose a not-too-simple $\Cost$-category $\cat{X}$. Give a bridge-style diagram for the unit profunctor $U_{\cat{X}}\colon\cat{X}\tickar\cat{X}$.
}{
We choose the $\Cost$-category $\cat{X}$ from
\cref{eqn.cities_distances}. The unit profunctor $U_\cat{X}$ on $\cat{X}$ is
described by the bridge diagram
\[
%
\]
}

\sol{exc.prof_unitality}{
\begin{exercise} \label{exc.prof_unitality}
\begin{enumerate}
  \item Justify each of the four steps $(=, \leq, \leq, =)$ in
  \cref{eqn.direction_rand597}.
  \item In the case $\cat{V}=\Bool$, we can directly show each of the four steps
  in \cref{eqn.direction_rand597} is actually an equality. How? 
  \item Justify each of the three steps $(=,\leq,\leq)$ in \cref{eqn.rand_16749}.
\end{enumerate}
}{
\begin{enumerate}
  \item The first equality is the unitality of $\cat{V}$
  (\cref{def.symm_mon_structure}(b)). The second step uses the monotonicity of
  $\otimes$ (\cref{def.symm_mon_structure}(a)) applied to the inequalities $I \le
  \cat{P}(p,p)$ (the identity law for $\cat{P}$ at $p$,
  \cref{def.cat_enriched_mpos}(a)) and $\Phi(p,q) \le \Phi(p,q)$ (reflexivity of
  preorder $\cat{V}$, \cref{def.preorder}(a)). The third step uses the definition
  of join: for all $x$ and $y$, since any $x \le x$, we have $x \le x \vee y$. The
  final equality is just the definition of profunctor composition,
  \cref{def.composite_profunctor}.
  
  \item Note that in $\Bool$, $I=\true$. Since the identity law at $p$ says $\true
  \le \cat{P}(p,p)$, and $\true$ is the largest element of the preorder $\Bool$, we
  thus have $\cat{P}(p,p)=\true$ for all $p$. This shows that the first inequality
  in \cref{eqn.direction_rand597} is an equality.
  
  The second inequality is more involved. Note that by the above, we can assume
  the left hand side of the inequality is equal to $\Phi(p,q)$. We split into two
  cases. Suppose $\Phi(p,q)= \true$. Then, again since $\true$ is the largest element of
  $\BB$, we must have equality. 
  
  Next, suppose $\Phi(p,q) = \false$. Note that since $\Phi$ is a monotone map
  $\cat{P}\op \times \cat{Q} \to \Bool$, if $p \le p_1$ in $\cat{P}$, then
  $\Phi(p_1,q) \le \Phi(p,q)$ in $\Bool$. Thus if $\cat{P}(p,p_1) = \true$ then
  $\Phi(p_1,q) = \Phi(p,q) =\false$. This implies that for all $p_1 \in \cat{P}$,
  we have either $\cat{P}(p,p_1) =\false$ or $\Phi(p_1,q) = \false$, and hence
  $\bigvee_{p_1\in \cat{P}} \cat{P}(p,p_1) \wedge \Phi(p_1,q) = \bigvee_{p_1 \in
  \cat{P}} \false = \false$.
  
  Thus, in either case, we see that $\Phi(p,q) = \bigvee_{p_1\in \cat{P}}
  \cat{P}(p,p_1) \wedge \Phi(p_1,q)$, as required.
  \item The first equation is unitality in monoidal categories, $v\otimes I=v$ for any $v\in V$. The second is that $I\leq\cat{Q}(q,q)$ by unitality of enriched categories, see \cref{def.cat_enriched_mpos}, together with monotonicity of monoidal product: $v_1\leq v_2$ implies $v\otimes v_1\leq v\otimes v_2$. The third was shown in \cref{exc.profunctor_def_alt}.
\end{enumerate}
}

\sol{exc.prof_associativity}{
Prove \cref{lemma:assoc_serial}. (Hint: remember to use the fact that $\cat{V}$
is skeletal.)
}{
This is very similar to \cref{exc.matrix_mult2}: we exploit the associativity of
$\otimes$. Note, however, we also require $\cat{V}$ to be symmetric monoidal
closed, since this implies the distributivity of $\otimes$ over $\vee$
(\cref{prop.properties_closed_mon_preorders} 2), and $\cat{V}$ to be skeletal,
so we can turn equivalences into equalities.
\begin{align*}
((\Phi\cp \Psi) \cp \Upsilon)(p,s) 
&= \bigvee_{r \in \cat{R}} \bigg(\bigvee_{q \in \cat{Q}} \Phi(p,q) \otimes
\Psi(q,r)\bigg) \otimes \Upsilon(r,s) \\
&= \bigvee_{r \in \cat{R}, q \in \cat{Q}} \Phi(p,q) \otimes
\Psi(q,r) \otimes \Upsilon(r,s) \\
&= \bigvee_{q \in \cat{Q}} \Phi(p,q) \otimes\bigvee_{r \in \cat{R}}
\bigg(\Psi(q,r) \otimes \Upsilon(r,s)\bigg) \\
&= (\Phi\cp(\Psi \cp \Upsilon))(p,s)
\end{align*}
}

\sol{exc.unit_companion}{
Check that the companion $\comp{\id}$ of $\id\colon\cat{P}\to\cat{P}$ really has the formula given in \cref{eqn.unit_profunctor}.
}{
This is very straightforward. We wish to check $\comp{\id}\colon \cat{P} \tickar
\cat{P}$ has the formula $\comp{\id}(p,q) = \cat{P}(p,q)$. By
\cref{def.companion_conjoint}, $\comp\id(p,q) \coloneqq \cat{P}(\id(p),q) =
\cat{P}(p,q)$. So they're the same.
}

\sol{exc.plus_conjoint}{
	Let $+\colon\RR\times\RR\times\RR\to\RR$ be as in \cref{ex.plus_3}. What is its conjoint $\conj{+}$?
}{
The conjoint $\conj{+}\colon \RR \tickar \RR \times \RR \times \RR$ sends
$(a,b,c,d)$ to $\RR(a,b+c+d)$, which is $\true$ if $a \le b+c+d$, and false
otherwise. 
}

\sol{exc.adjoint_comp_conj}{
Let $\cat{V}$ be a skeletal quantale, let $\cat{P}$ and $\cat{Q}$ be $\cat{V}$-categories, and let $F\colon\cat{P}\to\cat{Q}$ and $G\colon\cat{Q}\to\cat{P}$ be $\cat{V}$-functors.
\begin{enumerate}
	\item Show that $F$ and $G$ are $\cat{V}$-adjoints (as in
	\cref{eqn.adjoint_V_funs}) if and only if the companion of the former
	equals the conjoint of the latter: $\comp{F}=\conj{G}$.
	\item Use this to prove that $\comp{\id}=\conj{\id}$, as was stated in \cref{ex.unit_profunctor}.

\end{enumerate}
}{
\begin{enumerate}
	\item By \cref{def.companion_conjoint}, $\comp{F}(p,q) = \cat{Q}(F(p),q)$ and
$\conj{G}(p,q) = \cat{Q}(p,G(q))$. Since $\cat{V}$ is skeletal, $F$ and $G$ are
$\cat{V}$ adjoints if and only if $\cat{Q}(F(p),q) = \cat{Q}(p,G(q))$. Thus $F$
and $G$ are adjoints if and only if $\comp{F} = \conj{G}$.
	\item Note that $\id\colon \cat{P} \to \cat{P}$ is $\cat{V}$-adjoint to
	itself, since both sides of \cref{eqn.adjoint_V_funs} then equal
	$\cat{P}(p,q)$. Thus $\comp{\id} = \conj{\id}$.
\end{enumerate}
}

\sol{exc.collage_practice}{
Draw a Hasse diagram for the collage of the profunctor shown here:
\[

\]
}
{
The Hasse diagram for the collage of the given profunctor is quite simply this:
\[
%
\]
}

\sol{exc.mon_preorder_is_cat}{
Check that monoidal categories generalize monoidal preorders: a monoidal preorder is a
monoidal category $(\cat{P},I,\otimes)$ where, for every $p,q\in\cat{P}$, the
set $\cat{P}(p,q)$ has at most one element.
}{
Since we only have a rough definition, we can only roughly check this: we won't
bother with the notion of well-behaved. Nonetheless, we can still compare
\cref{def.symm_mon_structure} with \cref{rdef.sym_mon_cat}.

First, recall from \cref{subsubsec.pos_free_spectrum} that a preorder is a
category $\cat{P}$ such that for every $p,q \in \cat{P}$, the set $\cat{P}(p,q)$
has at most one element. 

On the surface, all looks promising: both definitions have two constituents and
four properties. In constituent (i), both \cref{def.symm_mon_structure}
and \cref{rdef.sym_mon_cat} call for the same: an element, or object, of the
preorder $\cat{P}$. So far so good. Constituent (ii), however, is where it gets
interesting: \cref{def.symm_mon_structure} calls for merely a function
$\otimes\colon \cat{P} \times \cat{P} \to \cat{P}$, while
\cref{rdef.sym_mon_cat} calls for a \emph{functor}.

Recall from \cref{ex.preorder_functor} that functors between preorders are
exactly monotone maps. So we need for the function $\otimes$ in
\cref{def.symm_mon_structure} to be a monotone map. This is exactly property (a)
of \cref{def.symm_mon_structure}: it says that if $(x_1,x_2) \le (y_1,y_2)$ in
$\cat{P} \otimes \cat{P}$, then we must have $x_1 \otimes x_2 \le y_1 \otimes
y_2$ in $\cat{P}$. So it is also the case that in \cref{def.symm_mon_structure}
that $\otimes$ is a functor.

The remaining properties compare easily, taking the natural isomorphisms to
be equality or equivalence in $\cat{P}$. Indeed, property (b) of
\cref{def.symm_mon_structure} corresponds to \emph{both} properties (a) and (b)
of \cref{rdef.sym_mon_cat}, and then the respective properties (c) and (d)
similarly correspond.
}

\sol{exc.read_string_diag}{
Consider the monoidal category $(\Cat{Set},1,\times)$, together with the diagram
\[
\begin{tikzpicture}[oriented WD, string decoration={}]
	\node[bb={1}{2}] (X11) {$f$};
	\node[bb={2}{2}, below right=of X11] (X12) {$g$};
	\node[bb={2}{1}, above right=of X12] (X13) {$h$};
	\node[bb={0}{0}, fit={($(X11.north west)+(.3,1.5)$) (X12)  ($(X13.east)+(-.3,0)$)}] (Y1) {};
	\begin{scope}[font=\small]
  	\draw[ar] (Y1.west|-X11_in1) to node[above] {$A$} (X11_in1);	
  	\draw[ar] (Y1.west|-X12_in2) to node[above] {$B$} (X12_in2);
  	\draw[ar] (X11_out1) to node[above] {$C$} (X13_in1);
  	\draw[ar] (X11_out2) to node[above=5pt] {$D$} (X12_in1);
  	\draw[ar] (X12_out1) to node[above=5pt] {$E$} (X13_in2);
  	\draw[ar] (X12_out2) to node[above] {$F$} (X12_out2-|Y1.east);
  	\draw[ar] (X13_out1) to node[above] {$G$} (X13_out1-|Y1.east);
	\end{scope}
\end{tikzpicture}
\]
Suppose that $A=B=C=D=F=G=\ZZ$ and $E=\BB=\{\true,\false\}$, 
and suppose that $f_C(a)=|a|$, $f_D(a)=a*5$, $g_E(d,b)=d\leq b$, $g_F(d,b)=d-b$, and $h(c,e)=\tn{if }e\tn{ then }c\tn{ else }1-c$.
\begin{enumerate}
	\item What are $g_E(5,3)$ and $g_F(5,3)$?
	\item What are $g_E(3,5)$ and $g_F(3,5)$?
	\item What is $h(5,\true)$?
	\item What is $h(-5,\true)$?\erase{Error!}
	\item What is $h(-5,\false)$?
\end{enumerate}
The whole diagram now defines a function $A\times B\to G\times F$; call it $q$.
\begin{enumerate}[resume]
	\item What are $q_G(-2,3)$ and $q_F(-2,3)$?
	\item What are $q_G(2,3)$ and $q_F(2,3)$?

\end{enumerate}
}{
\begin{enumerate}
	\item $g_E(5,3) = \false$, $g_F(5,3) = 2$.
	\item $g_E(3,5) = \true$, $g_F(3,5) = -2$.
	\item $h(5,\true) = 5$.
	\item $h(-5,\true) = -5$.
	\item $h(-5,\false) = 6$.
	\item $q_G(-2,3) = 2$, $q_F(-2,3) =-13$.
	\item $q_G(2,3) = -1$, $q_F(2,3) =7$.
\end{enumerate}
}

\sol{exc.cat_is_set_enriched}{
Recall from \cref{ex.set_as_mon_cat} that $\cat{V}=(\smset,\{1\},\times)$ is a
symmetric monoidal category. This means we can apply
\cref{def.enriched_in_mon_cat}. Does the (rough) definition roughly agree with
the definition of category given in \cref{def.category}? Or is there a subtle
difference?
}{
Yes, the rough definition roughly agrees: plain old categories are
$\smset$-categories! In detail, we need to compare
\cref{def.enriched_in_mon_cat} when $\cat{V} = (\smset,\{1\},\times)$ with \cref{def.category}.
In both cases, we see that (i) asks for a collection of objects and (ii) asks for,
for all pairs of objects $x,y$, a \emph{set} $\cat{C}(x,y)$ of morphisms.
Moreover, recall that the monoidal unit $I$ is the one element set $\{1\}$. This
means a morphism $\id_x\colon I \to \cat{C}(x,x)$ is a function $\id_x\colon
\{1\} \to \cat{C}(x,x)$. This is the same data as simply an element $\id_x =
\id_x(1) \in \cat{C}(x,x)$; we call this data the identity morphism on $x$.
Finally, a morphism $\cp\colon \cat{C}(x,y) \otimes \cat{C}(y,z) \to
\cat{C}(x,z)$ is a function $\cp \colon \cat{C}(x,y) \times \cat{C}(y,z) \to
\cat{C}(x,z)$; this is exactly the composite required in \cref{def.category}
(iv).

So in both cases the data agrees. In \cref{def.category} we also require this
data to satify two conditions, unitality and associativity. This is what is
meant by the last sentence of \cref{def.enriched_in_mon_cat}.
}

\sol{exc.metric_space_identities}{
What are identity elements in Lawvere metric spaces (that is,
$\Cost$-categories)? How do we interpret this in terms of distances?
}{
An identity element in a $\Cost$-category $\cat{X}$ is a morphism $I \to
\cat{X}(x,x)$ in $\Cost = ([0,\infty], \ge, 0, +)$, and hence the condition that
$0 \ge \cat{X}(x,x)$. This implies that $\cat{X}(x,x) = 0$. In terms of
distances, we interpret this to mean that the distance from any point to itself
is equal to $0$. We think this is a pretty sensible condition for a notion of
distance to obey.
}

\sol{exc.corelations}{
Consider the set $\ord{3}=\{1,2,3\}$.
\begin{enumerate}
	\item Draw a picture of the unit corelation $\varnothing\to\ord{3}\sqcup\ord{3}$.
	\item Draw a picture of the counit corelation $\ord{3}\sqcup\ord{3}\to\varnothing$.
	\item Check that the snake equations \eqref{eqn.yanking} hold. (Since every object is its own dual, you only need to check one of them.)

\end{enumerate}
}{
\begin{enumerate}
\item Here is a picture of the unit corelation $\varnothing \to \ord{3} \sqcup
\ord{3}$, where we have drawn the empty set with an empty dotted rectangle on the
left:
  \[

\]
\item Here is a picture of the counit corelation $\ord{3} \sqcup
\ord{3} \to \varnothing$:
  \[
  %
\]
\item Here is a picture of the snake equation on the left of \cref{eqn.yanking}.
\[
  \begin{altikz}
	\begin{pgfonlayer}{nodelayer}
		\node [contact, outer sep=5pt] (1cA) at (-5, -1) {};
		\node [contact, outer sep=5pt] (2cA) at (-5, -1.5) {};
		\node [contact, outer sep=5pt] (3cA) at (-5, -2) {};
		\node [contact, outer sep=5pt] (1aB) at (-2, 2) {};
		\node [contact, outer sep=5pt] (2aB) at (-2, 1.5) {};
		\node [contact, outer sep=5pt] (3aB) at (-2, 1) {};
		\node [contact, outer sep=5pt] (1bB) at (-2, 0.5) {};
		\node [contact, outer sep=5pt] (2bB) at (-2, 0) {};
		\node [contact, outer sep=5pt] (3bB) at (-2, -.5) {};
		\node [contact, outer sep=5pt] (1cB) at (-2, -1) {};
		\node [contact, outer sep=5pt] (2cB) at (-2, -1.5) {};
		\node [contact, outer sep=5pt] (3cB) at (-2, -2) {};
		\node [contact, outer sep=5pt] (1aC) at (1, 2) {};
		\node [contact, outer sep=5pt] (2aC) at (1, 1.5) {};
		\node [contact, outer sep=5pt] (3aC) at (1, 1) {};
		\coordinate (1outtB) at ($(1aB)!.5!(1bB)+(-.7,.3)$);
		\coordinate (1outbB) at ($(1aB)!.5!(1bB)+(-.7,-.3)$);
		\coordinate (1inB) at ($(1aB)!.5!(1bB)+(-.4,0)$);
		\coordinate (2outtB) at ($(2aB)!.5!(2bB)+(-.7,.3)$);
		\coordinate (2outbB) at ($(2aB)!.5!(2bB)+(-.7,-.3)$);
		\coordinate (2inB) at ($(2aB)!.5!(2bB)+(-.4,0)$);
		\coordinate (3outtB) at ($(3aB)!.5!(3bB)+(-.7,.3)$);
		\coordinate (3outbB) at ($(3aB)!.5!(3bB)+(-.7,-.3)$);
		\coordinate (3inB) at ($(3aB)!.5!(3bB)+(-.4,0)$);
		\coordinate (1outt) at ($(1bB)!.5!(1cB)+(.7,.3)$);
		\coordinate (1outb) at ($(1bB)!.5!(1cB)+(.7,-.3)$);
		\coordinate (1in) at ($(1bB)!.5!(1cB)+(.4,0)$);
		\coordinate (2outt) at ($(2bB)!.5!(2cB)+(.7,.3)$);
		\coordinate (2outb) at ($(2bB)!.5!(2cB)+(.7,-.3)$);
		\coordinate (2in) at ($(2bB)!.5!(2cB)+(.4,0)$);
		\coordinate (3outt) at ($(3bB)!.5!(3cB)+(.7,.3)$);
		\coordinate (3outb) at ($(3bB)!.5!(3cB)+(.7,-.3)$);
		\coordinate (3in) at ($(3bB)!.5!(3cB)+(.4,0)$);
	\end{pgfonlayer}
	\begin{pgfonlayer}{edgelayer}
	\begin{scope}[rounded corners=5pt, densely dotted, thick]
		\draw  [blue!50!black]
   (node cs:name=1aB, anchor=south west) --
   (node cs:name=1aB, anchor=south east) --
   (node cs:name=1aB, anchor=north east) --
   (node cs:name=1aB, anchor=north west) --
   (node cs:name=1outtB) --
   (node cs:name=1outbB) --
   (node cs:name=1bB, anchor=south west) --
   (node cs:name=1bB, anchor=south east) --
   (node cs:name=1bB, anchor=north east) --
   (node cs:name=1bB, anchor=north west) --
   (node cs:name=1inB) --
   cycle;
		\draw  [red!50!black]  
   (node cs:name=2aB, anchor=south west) --
   (node cs:name=2aB, anchor=south east) --
   (node cs:name=2aB, anchor=north east) --
   (node cs:name=2aB, anchor=north west) --
   (node cs:name=2outtB) --
   (node cs:name=2outbB) --
   (node cs:name=2bB, anchor=south west) --
   (node cs:name=2bB, anchor=south east) --
   (node cs:name=2bB, anchor=north east) --
   (node cs:name=2bB, anchor=north west) --
   (node cs:name=2inB) --
   cycle;
		\draw  [green!50!black]
   (node cs:name=3aB, anchor=south west) --
   (node cs:name=3aB, anchor=south east) --
   (node cs:name=3aB, anchor=north east) --
   (node cs:name=3aB, anchor=north west) --
   (node cs:name=3outtB) --
   (node cs:name=3outbB) --
   (node cs:name=3bB, anchor=south west) --
   (node cs:name=3bB, anchor=south east) --
   (node cs:name=3bB, anchor=north east) --
   (node cs:name=3bB, anchor=north west) --
   (node cs:name=3inB) --
   cycle;
		\draw  [blue!50!black]
   (node cs:name=1bB, anchor=south east) --
   (node cs:name=1bB, anchor=south west) --
   (node cs:name=1bB, anchor=north west) --
   (node cs:name=1bB, anchor=north east) --
   (node cs:name=1outt) --
   (node cs:name=1outb) --
   (node cs:name=1cB, anchor=south east) --
   (node cs:name=1cB, anchor=south west) --
   (node cs:name=1cB, anchor=north west) --
   (node cs:name=1cB, anchor=north east) --
   (node cs:name=1in) --
   cycle;
		\draw  [red!50!black]  
   (node cs:name=2bB, anchor=south east) --
   (node cs:name=2bB, anchor=south west) --
   (node cs:name=2bB, anchor=north west) --
   (node cs:name=2bB, anchor=north east) --
   (node cs:name=2outt) --
   (node cs:name=2outb) --
   (node cs:name=2cB, anchor=south east) --
   (node cs:name=2cB, anchor=south west) --
   (node cs:name=2cB, anchor=north west) --
   (node cs:name=2cB, anchor=north east) --
   (node cs:name=2in) --
   cycle;
		\draw  [green!50!black]
   (node cs:name=3bB, anchor=south east) --
   (node cs:name=3bB, anchor=south west) --
   (node cs:name=3bB, anchor=north west) --
   (node cs:name=3bB, anchor=north east) --
   (node cs:name=3outt) --
   (node cs:name=3outb) --
   (node cs:name=3cB, anchor=south east) --
   (node cs:name=3cB, anchor=south west) --
   (node cs:name=3cB, anchor=north west) --
   (node cs:name=3cB, anchor=north east) --
   (node cs:name=3in) --
   cycle;
		\draw  [blue!50!black]
   (node cs:name=1cA, anchor=south west) --
   (node cs:name=1cB, anchor=south east) --
   (node cs:name=1cB, anchor=north east) --
   (node cs:name=1cA, anchor=north west) --
   cycle;
		\draw  [red!50!black]  
   (node cs:name=2cA, anchor=south west) --
   (node cs:name=2cB, anchor=south east) --
   (node cs:name=2cB, anchor=north east) --
   (node cs:name=2cA, anchor=north west) --
   cycle;
		\draw  [green!50!black]
   (node cs:name=3cA, anchor=south west) --
   (node cs:name=3cB, anchor=south east) --
   (node cs:name=3cB, anchor=north east) --
   (node cs:name=3cA, anchor=north west) --
   cycle;
		\draw  [blue!50!black]
   (node cs:name=1aB, anchor=south west) --
   (node cs:name=1aC, anchor=south east) --
   (node cs:name=1aC, anchor=north east) --
   (node cs:name=1aB, anchor=north west) --
   cycle;
		\draw  [red!50!black]  
   (node cs:name=2aB, anchor=south west) --
   (node cs:name=2aC, anchor=south east) --
   (node cs:name=2aC, anchor=north east) --
   (node cs:name=2aB, anchor=north west) --
   cycle;
		\draw  [green!50!black]
   (node cs:name=3aB, anchor=south west) --
   (node cs:name=3aC, anchor=south east) --
   (node cs:name=3aC, anchor=north east) --
   (node cs:name=3aB, anchor=north west) --
   cycle;
	\end{scope}
	\end{pgfonlayer}
\end{altikz}
\quad
=
\quad
\begin{altikz}
	\begin{pgfonlayer}{nodelayer}
		\node [contact, outer sep=5pt] (1a) at (-2, 1) {};
		\node [contact, outer sep=5pt] (2a) at (-2, 0.5) {};
		\node [contact, outer sep=5pt] (3a) at (-2, 0) {};
		\node [contact, outer sep=5pt] (1b) at (1, 1) {};
		\node [contact, outer sep=5pt] (2b) at (1, 0.5) {};
		\node [contact, outer sep=5pt] (3b) at (1, 0) {};
	\end{pgfonlayer}
	\begin{pgfonlayer}{edgelayer}
	\begin{scope}[rounded corners=5pt, densely dotted, thick]
		\draw  [blue!50!black]
   (node cs:name=1a, anchor=south west) --
   (node cs:name=1b, anchor=south east) --
   (node cs:name=1b, anchor=north east) --
   (node cs:name=1a, anchor=north west) --
   cycle;
		\draw  [red!50!black]  
   (node cs:name=2a, anchor=south west) --
   (node cs:name=2b, anchor=south east) --
   (node cs:name=2b, anchor=north east) --
   (node cs:name=2a, anchor=north west) --
   cycle;
		\draw  [green!50!black]  
   (node cs:name=3a, anchor=south west) --
   (node cs:name=3b, anchor=south east) --
   (node cs:name=3b, anchor=north east) --
   (node cs:name=3a, anchor=north west) --
   cycle;
	\end{scope}
	\end{pgfonlayer}
\end{altikz}
\]
\end{enumerate}
}

\sol{exc.explain_monoidal_prod_feas}{
Interpret the monoidal products in $\Prof_{\Bool}$ in terms of feasibility. That is, preorders represent resources ordered by availability ($x\leq x'$ means that $x$ is available given $x'$) and a profunctor is a feasibility relation. Explain why $\cat{X}\times\cat{Y}$ makes sense as the monoidal product of resource preorders $\cat{X}$ and $\cat{Y}$ and why $\Phi\times\Psi$ makes sense as the monoidal product of feasibility relations $\Phi$ and $\Psi$.
}
{
Given two resource preorders $\cat{X}$ and $\cat{Y}$, the preorder $\cat{X}\times\cat{Y}$ represents the set of all pairs of resources, $x\in \cat{X}$ and $y\in \cat{Y}$, with $(x,y)\leq(x',y')$ iff $x\leq x'$ and $y\leq y'$. That is, if $x$ is available given $x'$ and $y$ is available given $y'$, then $(x,y)$ is available given $(x',y')$.

Given two profunctors $\Phi\colon\cat{X}_1\tickar\cat{X}_2$ and $\Psi\colon\cat{Y}_1\tickar\cat{Y}_2$, the profunctor $\Phi\times\Psi$ represents their conjunction, i.e.\ AND. In other words, if $y_1$ can be obtained given $x_1$ AND $y_2$ can be obtained given $x_2$, then $(y_1,y_2)$ can be obtained given $(x_1,x_2)$.
}

\sol{exc.prof_monoidal_unit}{
In order for $\Cat{1}$ to be a monoidal unit, there are supposed to be
isomorphisms $\cat{X}\times\Cat{1}\tickar\cat{X}$ and
$\Cat{1}\times\cat{X}\tickar\cat{X}$, for any $\cat{V}$-category $\cat{X}$. What
are they?
}{
The profunctor $\cat{X} \times \Cat{1} \tickar \cat{X}$ defined by the functor
$\alpha\colon (\cat{X} \times \Cat{1})\op \times \cat{X} \to \cat{V}$ that maps
$\alpha((x,1),y)\coloneqq\cat{X}(x,y)$ is an isomorphism. It has inverse
$\alpha\inv\colon \cat{X} \tickar \cat{X} \times \Cat{1}$ defined by
$\alpha\inv(x,(y,1))\coloneqq\cat{X}(x,y)$. To see that $\alpha\inv \cp \alpha =
\Unit{\cat{X}}$, note first that the unit law for $\cat{X}$ at $z$ and the definition
of join imply
\[
\cat{X}(x,z) = \cat{X}(x,z) \otimes I \le \cat{X}(x,z) \otimes \cat{X}(z,z) \le
\bigvee_{y \in \cat{X}} \cat{X}(x,y) \otimes \cat{X}(y,z),
\]
while composition says $\cat{X}(x,y) \otimes \cat{X}(y,z) \le \cat{X}(x,z)$ and
hence
\[
\bigvee_{y \in \cat{X}} \cat{X}(x,y) \otimes \cat{X}(y,z) \le \bigvee_{y \in
\cat{X}} \cat{X}(x,z) = \cat{X}(x,z).
\]
Thus, unpacking the definition of composition of profunctor, we have
\[
(\alpha\inv\cp \alpha) (x,z) 
= \bigvee_{(y,1) \in \cat{X} \times \Cat{1}} \alpha(x,(y,1)) \otimes
\alpha\inv((y,1),z) 
= \bigvee_{y \in \cat{X}} \cat{X}(x,y) \otimes \cat{X}(y,z) = \cat{X}(x,z).
\]
Similarly we can show $\alpha\cp\alpha\inv = \Unit{\cat{X} \times \Cat{1}}$, and
hence that $\alpha$ is an isomorphism $\cat{X} \times \Cat{1} \tickar \cat{X}$.

Moreover, we can similarly show that $\beta((1,x),y) \coloneqq \cat{X}(x,y)$ defines an
isomorphism $\beta\colon \Cat{1} \times \cat{X} \tickar \cat{X}$.
}

\sol{exc.prof_duals}{
Check these proposed units and counits do indeed obey the snake equations
\cref{eqn.yanking}.
}{
We check the first snake equation, the one on the left hand side of
\cref{eqn.yanking}. The proof of the one on the right hand side is analogous.

We must show that the composite $\Phi$ of profunctors
\[
\cat{X} \xrightarrow{\alpha\inv} \cat{X} \times \Cat{1} \xrightarrow{\Unit{\cat{X}}
\times \eta_{\cat{X}}} \cat{X} \times \cat{X}\op \times \cat{X}
\xrightarrow{\epsilon_{\cat{X}} \times \Unit{\cat{X}}} \Cat{1} \times \cat{X}
\xrightarrow{\alpha} \cat{X}
\]
is itself the identity (ie. the unit profunctor on $\cat{X}$), where $\alpha$
and $\alpha\inv$ are the isomorphisms defined in the solution to
\cref{exc.prof_monoidal_unit} above. 

Freely using the distributivity of $\otimes$ over $\vee$, the value $\Phi(x,y)$ of this
composite at $(x,y) \in \cat{X}\op \times \cat{X}$ is given by
\begin{align*}
&\bigvee_{a,b,c,d,e \in \cat{X}} 
	\begin{multlined}[t][11cm]
  \alpha\inv(x,(a,1)) \otimes
  (\Unit{\cat{X}}\times \eta_{\cat{X}})((a,1),(b,c,d))\\
  \otimes (\epsilon_{\cat{X}} \times \Unit{\cat{X}})((b,c,d),(1,e)) \otimes
  \alpha((1,e),y)
  \end{multlined}
\\
=&\bigvee_{a,b,c,d,e \in \cat{X}} 
	\begin{multlined}[t][11cm]
  \alpha\inv(x,(a,1)) \otimes
  \Unit{\cat{X}}(a,b) \otimes \eta_{\cat{X}}(1,c,d) \\
  \otimes
  \epsilon_{\cat{X}}(b,c,1) \otimes \Unit{\cat{X}}(d,e) \otimes
  \alpha((1,e),y)
  \end{multlined}
\\
=& \bigvee_{a,b,c,d,e \in \cat{X}} \cat{X}(x,a) \otimes \cat{X}(a,b) \otimes
\cat{X}(c,d) \otimes \cat{X}(b,c) \otimes \cat{X}(d,e) \otimes \cat{X}(e,y) \\
=&\quad \cat{X}(x,y)
\end{align*}
where in the final step we repeatedly use the argument the
\cref{lemma:unital_serial} that shows that composing with the unit profunctor
$\Unit{\cat{X}}(a,b) = \cat{X}(a,b)$ is the identity.

This shows that $\Phi(x,y)$ is the identity profunctor, and hence shows the
first snake equation holds. Again, checking the other snake equation is
analogous.
}

\finishSolutionChapter
\section[Solutions for Chapter 5]{Solutions for \cref{chap.SFGs}.}

\sol{exc.finset_as_prop}{
In \cref{ex.FinSet_Prop} we said that the identities, symmetries, and
composition rule in $\Cat{FinSet}$ ``are obvious.'' In math lingo, this just
means ``we trust that the reader can figure them out, if she spends the time tracking
down the definitions and fitting them together.''
\begin{enumerate}
	\item Draw a morphism $f\colon 3\to 2$ and a morphism $g\colon 2\to 4$ in $\finset$.
	\item Draw $f+g$.
	\item What is the composition rule for morphisms $f\colon m\to n$ and $g\colon n\to p$ in $\Cat{FinSet}$?
	\item What are the identities in $\Cat{FinSet}$? Draw some.
	\item Choose $m,n \in \NN$, and draw the symmetric map $\sigma_{m,n}$
	in $\Cat{FinSet}$?
\end{enumerate}
}{
\begin{enumerate}
	\item Below we draw a morphism $f\colon 3\to 2$ and a morphism $g\colon 2\to 4$ in $\finset$:
	\[

	\]
	\item The composite of morphisms $f\colon m\to n$ and $g\colon n\to p$ in $\Cat{FinSet}$ is the function $(f\cp g)\colon m\to p$ given by $(f\cp g)(i)=g(f(i))$ for all $1\leq i\leq m$.
	\item The identity $\id_m\colon m\to m$ is given by $\id_m(i)=i$ for all $1\leq i\leq m$. Here is a picture of $\id_2$ and $\id_8$:
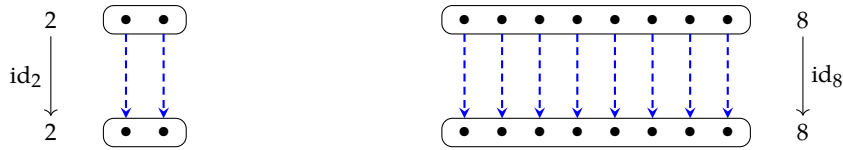
	\[
	%
	\]
	\item Here is a picture of the symmetry $\sigma_{3,5}\colon 8\to 8$:
	\[
	%
	\]
\end{enumerate}
}

\sol{exc.posetal_prop}{
  A posetal prop is a prop that is also a poset. That is, a posetal prop is a
  symmetric monoidal preorder of the form $(\NN,\preceq)$, for some poset relation
  $\preceq$ on $\NN$, where the monoidal product on objects is addition. We've
  spent a lot of time discussing order structures on the natural numbers. Give
  three examples of a posetal prop.
}{
We need to give examples posetal props, i.e.\ each will be a poset whose set of objects is $\nn$, whose order is denoted $m\preceq n$, and with the property that whenever $m_1\preceq n_1$ and $m_2\preceq n_2$ hold then $m_1+m_2\preceq n_1+n_2$ does too.

The question only asks for three, but we will additionally give a quasi-example and a non-example. 
\begin{enumerate}
	\item Take $\preceq$ to be the discrete order: $m\preceq n$ iff $m=n$.
	\item Take $\preceq$ to be the usual order, $m\preceq n$ iff there exists $d\in\nn$ with $d+m=n$.
	\item Take $\preceq$ to be the reverse of the usual order, $m\preceq n$ iff there exists $d\in\nn$ with $m=n+d$.
	\item Take $\preceq$ to be the co-discrete order $m\preceq n$ for all $m,n$. Some may object that this is a preorder, not a poset, so we call it a quasi-example.
	\item (Non-example.) Take $\preceq$ to be the division order, $m\preceq n$ iff there exists $q\in\nn$ with $m*q=d$. This is a perfectly good poset, but it does not satisfy the monotonicity property: we have $2\preceq 4$ and $3\preceq 3$ but not $5\preceq^?7$.
\end{enumerate}
}

\sol{exc.prop_practice}{
Choose one of \cref{ex.function_bijection,ex.corelation,ex.relation} and explicitly provide the five aspects of props discussed below \cref{def.prop}.
}
{
\begin{description}
	\item[\cref{ex.function_bijection}:] The prop $\Cat{Bij}$ has
	\begin{enumerate}
		\item $\Cat{Bij}(m,n)\coloneqq\{f\colon\ord{m}\to\ord{n}\mid f\text{ is a bijection}\}$. Note that $\Cat{Bij}(m,n)=\varnothing$ if $m\neq n$ and it has $n!$ elements if $m=n$.
		\item The identity map $n\to n$ is the bijection $\ord{n}\to\ord{n}$ sending $i\mapsto i$.
		\item The symmetry map $m+n\to n+m$ is the bijection $\sigma_{m,n}\colon\ord{m+n}\to\ord{n+m}$ given by
		\[
  		\sigma_{m,n}(i)\coloneqq
  		\begin{cases}
  			i+n&\tn{ if }i\leq m\\
				i-m&\tn{ if }m+1\leq i
  		\end{cases}
		\]
		\item Composition of bijections $m\to n$ and $n\to p$ is just their composition as functions, which is again a bijection.
		\item Given bijections $f\colon m\to m'$ and $g\colon n\to n'$, their monoidal product $(f+g)\colon (m+n)\to(m'+n')$ is given by
		\[
			(f+g)(i)\coloneqq
			\begin{cases}
				f(i)&\tn{ if }i\leq m\\
				g(i-m)&\tn{ if }m+1\leq i
			\end{cases}
		\]
	\end{enumerate}
	\item[\cref{ex.corelation}:] The prop $\Cat{Corel}$ has
	\begin{enumerate}
		\item $\Cat{Corel}(m,n)$ is the set of equivalence relations on $\ord{m+n}$.
		\item The identity map $n\to n$ is the smallest equivalence relation, which is the smallest reflexive relation, i.e.\ where $i\sim j$ iff $i=j$.
		\item The symmetry map $\sigma_{m,n}$, as an equivalence relation on $\ord{m+n+n+m}$ is ``the obvious thing,'' namely ``equating corresponding $m$'s together and also equating corresponding $n$'s together.'' To be pedantic, $i\sim j$ iff either
		\begin{itemize}
			\item $|i-j|=m+n+n$, or
			\item $m+1\leq i\leq m+n+n$ and $m+1\leq j\leq m+n+n$ and $|i-j|=n$.
		\end{itemize}
		\item Composition of an equivalence relation $\sim$ on $\ord{m+n}$ and an equivalence relation $\dot\sim$ on $\ord{n+p}$ is the equivalence relation $\simeq$ on $\ord{m+p}$ given by $i\simeq k$ iff there exists $j\in\ord{n}$ with $i\sim j$ and $j\dot\sim k$.
		\item Given equivalence relations $\sim$ on $\ord{m+n}$ and $\sim'$ on $\ord{m'+n'}$, we need an equivalence relation $(\sim+\sim')$ on $\ord{m+n+m'+n'}$. We take it to be ``the obvious thing,'' namely ``using $\sim$ on the unprimed stuff and using $\sim'$ on the primed stuff, with no other interaction.'' To be pedantic, $i\sim j$ iff either
		\begin{itemize}
			\item $i\leq m+n$ and $j\leq m+n$ and $i\sim j$, or
			\item $m+n+1\leq i$ and $m+n+1\leq j$ and $i\sim'j$.
		\end{itemize}
	\end{enumerate}
	\item[\cref{ex.relation}:] The prop $\Cat{Rel}$ has
	\begin{enumerate}
		\item $\Cat{Rel}(m,n)$ is the set of relations on the set $\ord{m}\times\ord{n}$, i.e.\ the set of subsets of $\ord{m}\times\ord{n}$, i.e.\ its powerset.
		\item The identity map $n\to n$ is the subset $\{(i,j)\in\ord{n}\times\ord{n}\mid i=j\}$.
		\item The symmetry map $m+n\to n+m$ is the subset of pairs $(i,j)\in(\ord{m+n})\times(\ord{n+m})$ such that either
		\begin{itemize}
			\item $i\leq m$ and $m+1\leq j$ and $i+m=j$, or
			\item $m+1\leq i$ and $j\leq m$ and $j+m=i$.
		\end{itemize}
		\item Composition of relations is as in \cref{ex.relation}.
		\item Given a relation $R\ss\ord{m}\times\ord{n}$ and a relation $R'\ss\ord{m'}\times\ord{n'}$, we need a relation $(R+R')\ss\ord{m+m'}\times\ord{n+n'}$. As stated in the example (footnote), this can be given by a universal property: The monoidal product $R_1+R_2$ of relations $R_1\ss \ord{m_1}\times \ord{n_1}$ and $R_2\ss \ord{m_2}\times \ord{n_2}$ is given by $R_1\sqcup R_2\ss(\ord{m_1}\times \ord{n_1})\sqcup(\ord{m_2}\times \ord{n_2})\ss(\ord{m_1}\sqcup \ord{m_2})\times(\ord{n_1}\sqcup \ord{n_2})$.
	\end{enumerate}
\end{description}
}

\sol{exc.port_graph_comp}{
Describe how port graph composition looks, with respect to the visual
representation of \cref{ex.a_port_graph}, and give a nontrivial example.
}{
Composition of an $(m,n)$-port graph $G$ and an $(n,p)$-port graph $H$ looks visually like sticking them end to end, connecting the wires in order, removing the two outer boxes, and adding a new outer box.

For example, suppose we want to compose the following in the order shown:
\[

\]
The result is:
\[
%
\]
}

\sol{exc.mon_prod_of_morphisms}{
Draw the monoidal product of the morphism shown in \cref{eqn.port_graph} with itself. It will be a $(4,6)$-port graph, i.e.\ a morphism $4\to 6$ in $\Cat{PG}$.
}
{
The monoidal product of two morphisms is drawn by stacking the corresponding port graphs. For this problem, we just stack the left-hand picture on top of itself to obtain the righthand picture:
\[
%
\]
}

\sol{exc.free_preorder_check_1}{
Let $P$ be a set, let $R\ss P \times P$ a relation, let $(P,\leq_P)$ be the preorder
obtained by taking the reflexive, transitive closure of $R$, and let $(Q,\leq_Q)$ be an arbitrary preorder. Finally, let $f\colon P\to Q$ be a function, not assumed monotone.
\begin{enumerate}
	\item Suppose that for every $x,y\in P$, if $R(x,y)$ then $f(x)\leq f(y)$. Show that $f$ defines a monotone map $f\colon (P,\leq_P)\to (Q,\leq_Q)$.
	\item Suppose that $f$ defines a monotone map $f\colon (P,\leq_P)\to (Q,\leq_Q)$. Show that for every $x,y\in P$, if $R(x,y)$ then $f(x)\leq_Q f(y)$.
\end{enumerate}
We call this the \emph{universal property} of the free preorder.\index{universal
property}
}{
We have a relation $R\ss P\times P$ which generates a preorder $\leq_P$ on $P$, we have an arbitrary preorder $(Q,\leq_Q)$ and a function $f\colon P\to Q$, not necessarily monotonic.
\begin{enumerate}
	\item Assume that for every $x,y\in P$, if $R(x,y)$ then $f(x)\leq f(y)$; we want to show that $f$ is monotone, i.e.\ that for every $x\leq_Py$ we have $f(x)\leq_Qf(y)$. By definition of $P$ being the reflexive, transitive closure of $R$, we have $x\leq_Py$ iff there exists $n\in\nn$ and $x_0,\ldots,x_n$ in $P$ with $x_0=x$ and $x_n=y$ and $R(x_i,x_{i+1})$ for each $0\leq i\leq n-1$. (The case $n=0$ handles reflexivity.) But then by assumption, $R(x_i,x_{i+1})$ implies $f(x_i)\leq_Q f(x_{i+1})$ for each $i$. By induction on $i$ we show that $f(x_0)\leq_Q f(x_i)$ for all $0\leq i\leq n$, at which point we are done.
	\item Suppose now that $f$ is monotone, and take $x,y\in P$ for which $R(x,y)$ holds. Then $x\leq_Py$ because $\leq_P$ is the smallest preorder relation containing $R$. (Another way to see this based on the above description is with $n=1$, $x_0=x$, and $x_n=y$, which we said implies $x\leq_Py$.) Since $f$ is monotone, we indeed have $f(x)\leq_Qf(y)$.
\end{enumerate}
}

\sol{exc.free_preorder_check_2}{
Let $P$, $Q$, $R$, etc.\ be as in \cref{exc.free_preorder_check_1}. We want to see that the universal property is really about maps out of---and not maps in to---the reflexive, transitive closure $(P,\leq)$. So let $g\colon Q\to P$ be a function.
\begin{enumerate}
	\item Suppose that for every $a,b\in Q$, if $a\leq b$ then $(g(a),g(b))\in R$. Is it automatically true that $g$ defines a monotone map $g\colon(Q,\leq_Q)\to(P,\leq_P)$?
	\item Suppose that $g$ defines a monotone map $g\colon(Q,\leq_Q)\to(P,\leq_P)$. Is it automatically true that for every $a,b\in Q$, if $a\leq b$ then $(g(a),g(b))\in R$?
\end{enumerate}
The lesson is that maps between structured objects are defined to preserve
structure. This means the domain of a map must be more constrained than the
codomain. Thus having the fewest additional constraints coincides with having
the most maps out---every function that respects our generating constraints
should define a map.%
\footnote{A higher-level justification understands freeness
as a left adjoint---of an adjunction often called the syntax/semantics
adjunction---but we will not discuss that here.}
}{
Suppose that $P$, $Q$, and $R$ are as in \cref{exc.free_preorder_check_1} and we have a function $g\colon Q\to P$.
\begin{enumerate}
	\item If $R(g(a),g(b))$ holds for all $a\leq_Qb$ then $g$ is monotone, because $R(x,y)$ implies $x\leq_Py$.
	\item It is possible for $g\colon(Q,\leq_Q)\to(P,\leq_P)$ to be monotone and yet have some $a,b\in Q$ with $a\leq_Qb$ and $(g(a),g(b))\not\in R$. Indeed, take $Q\coloneqq\{1\}$ to be the free preorder on one element, and take $P\coloneqq\{1\}$ with $R=\varnothing$. Then the unique function $g\colon Q\to P$ is monotone (because $\leq_P$ is reflexive even though $R$ is empty), and yet $(g(1),g(1))\not\in R$.
\end{enumerate}
}

\sol{exc.free_cat_is_free}{
Let $G=(V,A,s,t)$ be a graph, and let $\cat{G}\coloneqq\Cat{Free}(G)$ be the free category on $G$. Let $\cat{C}$ be another category whose set of morphisms is denoted $\Mor(\cat{C})$. 
\begin{enumerate}
	\item Someone tells you that there are ``domain and codomain'' functions $\dom,\cod\colon\Mor(\cat{C})\to\Ob(\cat{C})$; interpret this statement.
	\item Show that the set of functors $\cat{G} \to
\cat{C}$ are in one-to-one correspondence with the set of pairs of functions
$(f,g)$, where $f\colon V \to \Ob(\cat{C})$ and $g\colon A\to\Mor(\cat{C})$ such that $\dom(g(a))=f(s(a))$ and $\cod(g(a))=f(t(a))$ for all $a\in A$.
	\item Is $(\Mor(\cat{C}),\Ob(\cat{C}),\dom,\cod)$ a graph? If so, see if you can use the word ``adjunction'' in a sentence that describes the statement in part 2. If not, explain why not.

\end{enumerate}
}{
Let $G=(V,A,s,t)$ be a graph, let $\cat{G}$ be the free category on $G$, and let $\cat{C}$ be another category, whose set of morphisms is denoted $\Mor(\cat{C})$.
\begin{enumerate}
	\item To give a function $\Mor(\cat{C})\to\Ob(\cat{C})$ means that for every element $\Mor(\cat{C})$ we need to give exactly one element of $\Ob(\cat{C})$. So for $\dom$ we take any $q\in\Mor(\cat{C})$, view it as a morphism $q\colon y\to z$, and send it to its domain $y$. Similarly for $\cod$: we put $\cod(q)\coloneqq z$.
	\item Suppose first that we are given a functor $F\colon\cat{G}\to\cat{C}$. On objects we have a function $\Ob(\cat{G})\to\Ob(\cat{C})$, and this defines $f$ since $\Ob(\cat{G})=V$. On morphisms, first note that the arrows of graph $G$ are exactly the length=1 paths in $G$, whereas $\Mor(\cat{G})$ is the set of all paths in $G$, so we have an inclusion $A\ss\Mor(\cat{G})$. The functor $F$ provides a function $\Mor(\cat{G})\to\Mor(\cat{C})$, which we can restrict to $A$ to obtain $g\colon A\to\Mor(\cat{C})$. All functors satisfy $\dom(F(r))=F(\dom(r))$ and $\cod(F(r))=F(\cod(r))$ for any $r\colon w\to x$. In particular when $r\in A$ is an arrow we have $\dom(r)=s(r)$ and $\cod(r)=t(r)$. Thus we have found $(f,g)$ with the required properties.\\	
	
	Suppose second that we are given a pair of functions $(f,g)$ where $f\colon V \to \Ob(\cat{C})$ and $g\colon A\to\Mor(\cat{C})$ such that $\dom(g(a))=f(s(a))$ and $\cod(g(a))=f(t(a))$ for all $a\in A$. Define $F\colon\cat{G}\to\cat{C}$ on objects by $f$. An arbitrary morphism in $\cat{G}$ is a path $p\coloneqq(v_0,a_1,a_2,\ldots,a_n)$ in $G$, where $v_0\in V$, $a_i\in A$, $v_0=s(a_1)$, and $t(a_i)=s(a_{i+1})$ for all $1\leq i\leq n-1$. Then $g(a_i)$ is a morphism in $\cat{C}$ whose domain is $f(v_0)$ and the morphisms $g(a_i)$ and $g(a_{i+1})$ are composable for every $1\leq i\leq n-1$. We then take $F(p)\coloneqq \id_{f(v_0)}\cp g(a_1)\cp\cdots\cp g(a_n)$ to be the composite. It is easy to check that this is indeed a functor (preserves identities and compositions).\\	

	Third, we want to see that the two operations we just gave are mutually inverse. On objects this is straightforward, and on morphisms it is straightforward to see that, given $(f,g)$, if we turn them into a functor $F\colon\cat{G}\to\cat{C}$ and then extract the new pair of functions $(f',g')$, then $f=f'$ and $g=g'$. Finally, given a functor $F\colon\cat{G}\to\cat{C}$, we extract the pair of functions $(f,g)$ as above and then turn them into a new functor $F'\colon\cat{G}\to\cat{C}$. It is clear that $F$ and $F'$ act the same on objects, so what about on morphisms. The formula says that $F'$ acts the same on morphisms of length $1$ in $\cat{G}$ (i.e.\ on the elements of $A$). But an arbitrary morphism in $\cat{G}$ is just a path, i.e.\ a sequence of composable arrows, and so by functoriality, both $F$ and $F'$ must act the same on arbitrary paths.
	\item $(\Mor(\cat{C}),\Ob(\cat{C}),\dom,\cod)$ is a graph; let's denote it $U(\cat{C})\in\Cat{Grph}$. We have functors $\free\colon\Cat{Grph}\leftrightarrows\smcat\cocolon U$, and $\free$ is left adjoint to $U$.
\end{enumerate}
}

\sol{exc.free_monoid}{
Recall that monoids are one-object categories. For any set $A$, there is a graph with one vertex and an arrow from the vertex to itself for each $a\in A$. Thus we can construct a free monoid from just the data of a set $A$.
\begin{enumerate}
	\item What are the elements of the free monoid on the set $A=\{a\}$?
	\item Can you find a well-known monoid that is isomorphic to the free monoid on $\{a\}$?
	\item What are the elements of the free monoid on the set $A=\{a,b\}$?
\end{enumerate}
}{
\begin{enumerate}
	\item The elements of the free monoid on the set $\{a\}$ are:
	\[a^0, a^1, a^2,a^3,\ldots, a^{2019},\ldots\]
	with monoid multiplication $*$ given by the usual natural number addition on the exponents, $a^i*a^j=a^{i+j}$.
	\item This is isomorphic to $\nn$, by sending $a^i\mapsto i$.
	\item The elements of the free monoid on the set $\{a,b\}$ are `words in $a$ and $b$,' each of which we will represent as a list whose entries are either $a$ or $b$. Here are some:
	\[[\;],\quad [a],\quad [b],\quad [a,a],\quad [a,b],\quad \ldots,\quad [b,a,b,b,a,b,a,a,a,a],\quad\ldots\]
\end{enumerate}
}

\sol{exc.free_prop_port_graph}{
Consider the following prop signature:
\[
  G\coloneqq\{\rho_{m,n}\colon m \to n \mid m, n \in \nn\},\qquad s(\rho_{m,n})\coloneqq m,\quad t(\rho_{m,n})\coloneqq n,
\]
i.e.\ having one generating morphism for each $(m,n)\in\nn^2$. Show that $\free(G)$ is the prop of port graphs.
}{
We have two props: the prop of port graphs and the free prop $\free(G,s,t)$ where
\[
  G\coloneqq\{\rho_{m,n}\colon m \to n \mid m, n \in \nn\},\qquad s(\rho_{m,n})\coloneqq m,\quad t(\rho_{m,n})\coloneqq n;
\]
we want to show they are the same prop. As categories they have the same set of objects (in both cases, $\nn$), so we need to show that for every $m,n\in\nn$, they have the same set of morphisms (and that their composition formulas and monoidal product formulas agree).

By \cref{def.free_prop}, a morphism $m\to n$ in $\free(G)$ is a $G$-labeled port graph, i.e.\ a pair $(\Gamma,\ell)$, where $\Gamma=(V,\pgin,\pgout,\iota)$ is an $(m,n)$-port graph and $\ell\colon V \to G$ is a function, such that the `arities agree.' What does this mean? Recall that every vertex $v\in V$ is drawn as a box with some left-hand ports and some right-hand ports---an arity---and $\ell(v)\in G$ is supposed to have the correct arity; precisely, $s(\ell(v))=\pgin(v)$ and $t(\ell(v)) = \pgout(v)$. But $G$ was chosen so that it has exactly one element with any given arity, so the function $\ell$ has only one choice, and thus contributes nothing: it neither increases nor decreases the freedom. In other words, a morphism in our particular $\free(G)$ can be identified with an $(m,n)$ port graph $\Gamma$, as desired.

Again by definition \cref{def.free_prop}, the `composition and the monoidal structure are just those for port graphs
  $\Cat{PG}$ (see \cref{eqn.PG_prop}); the labelings (the $\ell$'s) are just carried along.' So we are done.
}

\sol{exc.free_prop_pic}{
Consider again the free prop on generators $G=\{f\colon 1 \to 1, g\colon 2 \to 2, h\colon 2 \to
1\}$. Draw a picture of $(f+\id_1+\id_1)\cp(\sigma+\id_1)\cp(\id_1+h)\cp \sigma\cp g$.
}{
Here is a picture of $(f+\id_1+\id_1)\cp(\sigma+\id_1)\cp(\id_1+h)\cp \sigma\cp g$, in the free prop on generators $G=\{f\colon 1 \to 1, g\colon 2 \to 2, h\colon 2 \to
1\}$:
\[
\begin{tikzpicture}[oriented WD, bb port length=0pt, bbx=.7cm, bb port sep=1pt]
	\node[bb={1}{1}] (f) {$f$};
	\node[bb={2}{1}, below right=1 and 1.5 of f] (h) {$h$};
	\node[bb={2}{2}, above right=0 and 1.5 of h] (g) {$g$};
	\node[bb={0}{0}, fit=(f) (g) (h)] (outer) {};
	\coordinate (outer_in1) at (outer.west|-f_in1);
	\coordinate (outer_in2) at (outer.west|-h_in1);
	\coordinate (outer_in3) at (outer.west|-h_in2);
	\coordinate (outer_out1) at (outer.east|-g_out1);
	\coordinate (outer_out2) at (outer.east|-g_out2);
	\coordinate (A1) at ($(f_out1)+(0.5,0)$);
	\coordinate (A2) at (A1|-outer_in2);
	\coordinate (A3) at (A1|-outer_in3);
	\coordinate (B1) at ($(A1)+(0.5,0)$);
	\coordinate (B2) at (B1|-A2);
	\coordinate (B3) at (B1|-A3);
	\coordinate (C2) at ($(h_out1)+(0.5,0)$);
	\coordinate (C1) at (C2|-A1);
	\coordinate (D1) at ($(g_in1)-(0.5,0)$);
	\coordinate (D2) at ($(g_in2)-(0.5,0)$);
	\draw (outer_in1) -- (f_in1);
	\draw (f_out1) -- (A1) -- (B2) -- (h_in1);
	\draw (outer_in2) -- (A2) -- (B1) -- (C1) -- (D2) -- (g_in2);
	\draw (outer_in3) -- (A3) -- (B3) -- (h_in2);
	\draw (h_out1) -- (C2) -- (D1) -- (g_in1);
	\draw (g_out1) -- (outer_out1);
	\draw (g_out2) -- (outer_out2);
\end{tikzpicture}
\]
}

\sol{exc.same_free_prop}{
Is it the case that the free prop on generators $(G,s,t)$, defined in \cref{def.free_prop}, is the same thing as the prop presented by $(G,s,t,\varnothing)$, having no relations, as defined in \cref{rdef.presentation_prop}? Or is there a subtle difference somehow?
}{
The free prop on generators $(G,s,t)$, defined in \cref{def.free_prop}, is---for all intents and purposes---the same thing as the prop presented by $(G,s,t,\varnothing)$, having no relations. The only possible ``subtle difference'' we might have to admit is if someone said that a set $S$ is ``subtly different'' than its quotient by the trivial equivalence relation. In the latter, the elements are the singleton subsets of $S$. So for example the quotient of $S=\{1,2,3\}$ by the trivial equivalence relation is the set $\{\{1\}, \{2\}, \{3\}\}$. It is subtly different than $S$, but the two are naturally isomorphic, and category-theoretically, the difference will never make a difference.
}

\sol{exc.rigs_mats}{
\begin{enumerate}
	\item We said in \cref{ex.mat_rig} that for any rig $R$, the set $\Set{Mat}_n(R)$ forms a rig. What is its multiplicative identity $1\in\Set{Mat}_n(R)$?
	\item We also said that $\Set{Mat}_n(R)$ is generally not commutative. Pick an $n$ and show that that $\Set{Mat}_n(\NN)$ is not commutative, where $\NN$ is as in \cref{ex.rig_nat}.
\end{enumerate}
}{
\begin{enumerate}
	\item If $(R,0,+,1,*)$ is a rig, then the multiplicative identity $1\in\Set{Mat}_n(R)$ is the usual $n$-by-$n$ identity matrix: 1's on the diagonal and 0's everywhere else (where by `1' and `0', we mean those elements of $R$). So for $n=4$ it is:
  \[
  \left(
\begin{tabular}{cccc}
  	1&0&0&0\\
  	0&1&0&0\\
  	0&0&1&0\\
  	0&0&0&1	
  \end{tabular}
  \right).
  \]
	\item We choose $n=2$ and hence need to find two elements $A,B\in\Set{Mat}_2(\nn)$ such that $A*B\neq B*A$.
	\[
  	A*B=\left(
    	\begin{array}{cc}
    		0&1\\
    		0&0
    	\end{array}
  	\right)
  	*
  	\left(
    	\begin{array}{cc}
    		0&1\\
    		1&0
    	\end{array}	
  	\right)
  	\neq
  	\left(
    	\begin{array}{cc}
    		0&1\\
    		1&0
    	\end{array}	
  	\right)
    *
  	\left(
    	\begin{array}{cc}
    		0&1\\
    		0&0
    	\end{array}
  	\right)
	=B*A
	\]
	One can calculate from the multiplication formula (recalled in \cref{ex.mat_rig}) says $(A*B)(1,1)=0*0+1*1=1$ and $(B*A)(1,1)=0*0+0*0=0$, which are not equal.
\end{enumerate}
}

\sol{exc.a_signal_flow_graph}{
  The flow graph  
\[
%
\]
  takes in two natural numbers, $x$ and $y$, and produces two
  output signals. What are they?
}{
  Semantically, if we apply the flow graph below to the input signal $(x,y)$ 
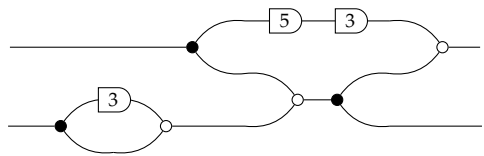
\[
%
\]
the resulting output signal is $(16x+4y, x+4y)$.
}

\sol{exc.monoidal_mat_prod}{
Let $A$ and $B$ be the following matrices with values in $\NN$:
\[
A=\left(
\begin{array}{ccc}
	3&3&1\\
	2&0&4
\end{array}
\right)
\hspace{1in}
B=\left(
\begin{array}{cccc}
	2&5&6&1\\
\end{array}
\right)
\]
What is the monoidal product matrix $A+B$?
}{
The monoidal product of 
$A=\left(
\begin{array}{ccc}
	3&3&1\\
	2&0&4
\end{array}
\right)
$
and 
$B=\left(
\begin{array}{cccc}
	2&5&6&1\\
\end{array}
\right)
$
is
\[
  A+B=\left(
  \begin{array}{ccccccc}
  	3&3&1&0&0&0&0\\
  	2&0&4&0&0&0&0\\
  	0&0&0&2&5&6&1
  \end{array}
  \right)
\]
}

\sol{exc.sig_flow_mats}{
\begin{enumerate}
	\item  What matrices does the signal flow graph
  \[
\begin{aligned}
%
\end{aligned}
  \]
  represent?
  \item What about the signal flow graph
  \[
\begin{aligned}
%
\end{aligned}
    ?
  \]
  \item Are they equal?
\end{enumerate}
}{
\begin{enumerate}
	\item  The signal flow graph on the left represents the matrix on the right:
  \[
    \parbox{1in}{
$
\begin{aligned}

\end{aligned}
$
}
    \hspace{1in}
    \parbox{1in}{$\left(\begin{array}{ccc}1&1&1\end{array}\right)$}
  \]
	\item  The signal flow graph on the left represents the matrix on the right:
  \[
    \parbox{1in}{
    $
\begin{aligned}

\end{aligned}$
}
    \hspace{1in}
    \parbox{1in}{$\left(\begin{array}{ccc}1&1&1\end{array}\right)$}
  \]
  \item They are equal.
\end{enumerate}
}

\sol{ex.drawsfgs}{
  Draw signal flow graphs that represent the following matrices:\\
  \begin{enumerate*}[itemjoin=\hspace{1in}]
  \item[] \hspace{-.2in}
    \item $
      %
\end{aligned}
\]
\end{enumerate}
}

\sol{exc.general_case_S_full}{
Write down a detailed proof of \cref{prop.Sfull}. Suppose $M$ is an $m \times
n$-matrix. Follow the idea of the $(2\times 2)$-case in \cref{eqn.two_by_two},
and construct the signal flow graph $g$---having $m$ inputs and $n$ outputs---as
the composite of four layers, respectively comprising (i) copy maps, (ii)
scalars, (iii) swaps and identities, (iv) addition maps.}
{
\begin{itemize}
  \item For the first layer $g_1$, take the monoidal product of $m$ copies of $c_n$,
    \[
      g_1\coloneq c_n +\dots+ c_n\colon m \to (m \times n),
    \]
    where $c_n$ is the signal flow diagram that makes $n$ copies of a single input:
    \[
      c_n\coloneq
      \comult{1em}\cp (1+\comult{1em})\cp (1+1+\comult{1em})\cp \dots\cp (1+\dots+1+\comult{1em})\colon
      1 \to n
    \]
  \item Next, define
    \begin{align*}
      g_2\coloneq &\quad s_{M(1,1)}+\cdots+s_{M(1,n)} \\
      & +s_{M(2,1)}+\cdots+s_{M(2,n)} \\
      &+\cdots \\
      &+s_{M(m,1)}+\cdots+s_{M(m,n)}\colon (m\times n) \to (m\times n),
    \end{align*}
    where $s_a\colon 1 \to 1$ is the signal flow
    graph generator ``scalar multiplication by $a$.'' This layer amplifies each copy of the input signal by the
    relevant rig element.
  \item The third layer rearranges wires. We will not write this down
    explicitly, but simply say it is the signal flow graph $g_3\colon m \times
    n \to m \times n$, that is the
    composite and monoidal product of swap and identity maps, such that the
    $(i-1)m+j$th input is sent to the $(j-1)n+i$th output, for all $1 \le i \le n$
    and $1 \le j \le m$.
  \item Finally, the fourth layer is similar to the first, but instead adds the
    amplified input signals. We define 
    \[
      g_4\coloneq a_m+\dots+a_m \colon (m \times n)\to n,
    \]
    where $a_m$ is the signal flow graph that adds $m$ inputs to produce a single output:
    \[
      a_m\coloneq
      (1+\dots+1+\add{1em})\cp \cdots\cp(1+1+\add{1em})\cp(1+\add{1em})\cp\add{1em}\colon
      m \to 1
    \]
\end{itemize}
Using \cref{prop.Simages}, it is a straightforward but tedious calculation to show that $g=g_1\cp g_2\cp g_3\cp g_4\colon m \to n$ has the property that $S(g)=M$.
}

\sol{exc.rep_mats}{
  \begin{enumerate}
  	\item For each matrix in \cref{ex.drawsfgs}, draw another signal flow graph that
  represents that matrix.
  	\item Using the above equations and the prop axioms, prove
  that the two signal flow graphs represent the same matrix.
  
  \end{enumerate}
}{
\begin{enumerate}
\item The matrices in \cref{ex.drawsfgs} may also be drawn as the following
signal flow graphs:
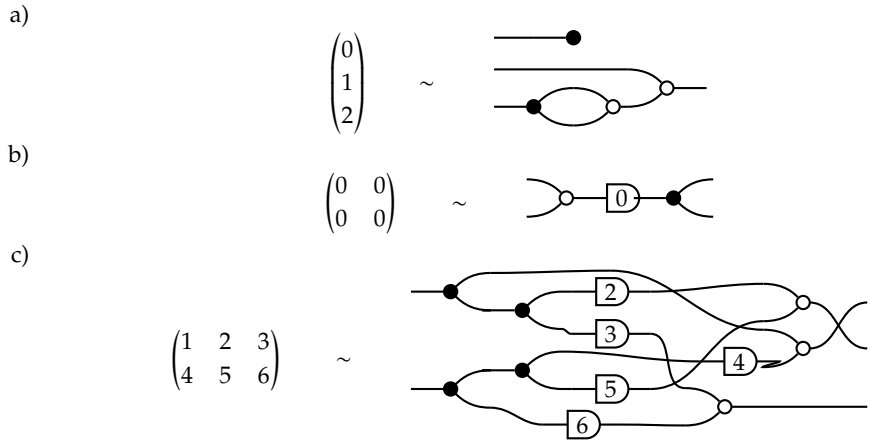
\begin{enumerate}
\item
  \[

\end{aligned}
\]
\end{enumerate}
\item Here are graphical proofs that the representations we chose in our
solution to \cref{ex.drawsfgs} agree with those chosen in Part 1 above.
\begin{enumerate}
  \item
\[
      \begin{aligned}
%
\end{aligned}
\end{align*}
\end{enumerate}
\end{enumerate}
}

\sol{exc.proof_using_diag_lang}{
  Consider the signal flow graphs
  \begin{equation}\label{eqn.exc_two_sfgs}
  \begin{aligned}
    %
\end{aligned}
  \end{equation}
  \begin{enumerate}
  	\item Let $R=(\NN,0,+,1,*)$. By examining the presentation of $\mat(R)$ in \cref{thm.presentation_mat}, and without computing the
  matrices that they represent, conjecture whether they represent the same matrix. How could you prove your conjecture?
  	\item Now suppose the rig is $R=\NN/3$; if you do not know what this means, just replace all 3's with 0's in the right-hand diagram of \cref{eqn.exc_two_sfgs}. Find what you would call a minimal representation of this diagram, using the presentation in \cref{thm.presentation_mat}.
\end{enumerate}
}{
\begin{enumerate}
	\item The signal flow graphs 
  \[
    \begin{aligned}

  \end{aligned}
  \]
cannot represent the same morphism because one has a path from a vertex on the left to one on the right, and the other does not. To prove this, observe that the only graphical equation in \cref{thm.presentation_mat} that breaks a path from left to right is the equation
\[
\begin{aligned}
%
\end{aligned}
\]
So a $0$ scalar must within a path from left to right before we could rewrite the diagram to break that path. No such $0$ scalar can appear, however, because the diagram does not contain any, and the sum and product of any two nonzero natural numbers is always nonzero.
\item Replacing each of the $3$s with $0$ allows us to rewrite the diagram to
\[
%
\]
\end{enumerate}
}

\sol{exc.two_monoids_on_reals}{
Consider the set $\RR$ of real numbers.
	\begin{enumerate}
		\item Show that if $\mu\colon\RR\times\RR\to\RR$ is defined by $\mu(a,b)=a*b$ and if $\eta\in\RR$ is defined to be $\eta=1$, then $(\RR,*,1)$ satisfies all three conditions of \cref{def.monoid_object}.
		\item Show that if $\mu\colon\RR\times\RR\to\RR$ is defined by $\mu(a,b)=a+b$ and if $\eta\in\RR$ is defined to be $\eta=0$, then $(\RR,+,0)$ satisfies all three conditions of \cref{def.monoid_object}.
\end{enumerate}
}{
The three conditions of \cref{def.monoid_object} are
\begin{enumerate}[label=(\alph*)]
  \item $(\mu \otimes \id)\cp\mu = (\id \otimes \mu)\cp\mu$,
  \item $(\eta \otimes \id)\cp\mu = \id = (\id \otimes \eta)\cp\mu$, and
  \item $\sigma_{M,M}\cp\mu = \mu$.
\end{enumerate}
where $\sigma_{M,M}$ is the swap map on $M$ in $\cat{C}$.
\begin{enumerate}
	\item Suppose $\mu\colon\RR\times\RR\to\RR$ is defined by $\mu(a,b)=a*b$ and  $\eta\in\RR$ is defined to be $\eta=1$. The conditions, written diagrammatically, say that starting in the upper left of each diagram below, the result in the lower right is the same regardless of which path you take:
	\[
	\begin{tikzcd}[ampersand replacement=\&, column sep=30pt]
		(a,b,c)\ar[r, |->, "(\mu\otimes\id)"]\ar[d, |->, "(\id\otimes\mu)"']\&
		(a*b, c)\ar[d, |->, "\mu"]\\
		(a,b*c)\ar[r, |->, "\mu"']\&
		a*b*c
	\end{tikzcd}
	\qquad
	\begin{tikzcd}[ampersand replacement=\&, column sep=30pt]
		a\ar[r, |->, "(\eta\otimes\id)"]\ar[d, |->, "(\id\otimes\eta)"']\ar[dr, |->, "\id_a" description]\&
		(1, a)\ar[d, |->, "\mu"]\\
		(a, 1)\ar[r, |->, "\mu"']\&
		a
	\end{tikzcd}	
	\qquad
	\begin{tikzcd}[ampersand replacement=\&]
		(a,b)\ar[r, |->, "\sigma"]\ar[dr, |->, "\mu"']\&
		(b, a)\ar[d, |->, "\mu"]\\\&
		a*b
	\end{tikzcd}	
	\]
	and this is true for $(\rr,*,1)$.
	\item The same reasoning works for $(\rr,+,0)$, shown below:
	\[
	\begin{tikzcd}[ampersand replacement=\&, column sep=30pt]
		(a,b,c)\ar[r, |->, "(\mu\otimes\id)"]\ar[d, |->, "(\id\otimes\mu)"']\&
		(a{+}b, c)\ar[d, |->, "\mu"]\\
		(a,b{+}c)\ar[r, |->, "\mu"']\&
		a{+}b{+}c
	\end{tikzcd}
	\qquad
	\begin{tikzcd}[ampersand replacement=\&, column sep=30pt]
		a\ar[r, |->, "(\eta\otimes\id)"]\ar[d, |->, "(\id\otimes\eta)"']\ar[dr, |->, "\id_a" description]\&
		(0, a)\ar[d, |->, "\mu"]\\
		(a, 0)\ar[r, |->, "\mu"']\&
		a
	\end{tikzcd}	
	\qquad
	\begin{tikzcd}[ampersand replacement=\&]
		(a,b)\ar[r, |->, "\sigma"]\ar[dr, |->, "\mu"']\&
		(b, a)\ar[d, |->, "\mu"]\\\&
		a{+}b
	\end{tikzcd}	
	\]
	\end{enumerate}
}

\sol{exc.check_monoid_obj}{
Recall that in $\mat(R)$, the monoidal unit is $0$ and the monoidal product is $+$, because it is a prop. Recall also that in (the usual monoidal structure on) $\smset$, the monoidal unit is $\{1\}$, a set with one element, and the monoidal product is $\times$ (see \cref{ex.set_as_mon_cat}).

\begin{enumerate}
	\item Check that the functor $U\colon\mat(R)\to\smset$, defined above, preserves the monoidal unit and the monoidal product.
	\item Find a function $p\colon R^n\times R^n\to R^n$ such that $p(0,v)=v$ and $p(v,0)=v$ for any $v\in R^n$.
	\item Use your $p$ to show that if $(M,\mu,\eta)$ is a monoid object in $\mat(R)$ then $(U(M),U(\mu),U(\eta))$ is a monoid object in $\smset$. (This works for any monoidal functor---which we will define in \cref{roughdef.monoidal_functor}---not just for $U$ in particular.)
	\item In \cref{ex.diagrammatic_monoid_obj}, we said that the triple
  $(1,\add{1em},\zero{1em})$ is a commutative monoid object in the prop $\mat(R)$. If $R=\RR$ is the rig of real numbers, this means that we have a monoid structure on the set $\RR$. But in \cref{exc.two_monoids_on_reals} we gave two such monoid structures. Which one is it?
\end{enumerate}
}{
The functor $U\colon\mat(R)\to\smset$ is given on objects by sending $n$ to the set $R^n$, and on morphisms by matrix-vector multiplication. Here $R^n$ means the set of $n$-tuples or $n$-dimensional vectors in $R$. In particular, $R^0=\{()\}$ consists of a single vector of dimension 0.
\begin{enumerate}
	\item $U$ preserves the monoidal unit because $0$ is the monoidal unit of any prop ($\mat(R)$ is a prop), $\{1\}$ is the monoidal unit of $\smset$, and $R^0$ is canonically isomorphic to $\{1\}$. $U$ also preserves the monoidal product because there is a canonical isomorphism $R^m\times R^n\cong R^{m+n}$.
	\item A monoid object in $\mat(R)$ is a tuple $(m,\mu,\eta)$ where $m\in\nn$, $\mu\colon m+m\to m$, and $\eta\colon 0\to m$ satisfy the properties $\mu(\eta,x)=x=\mu(x,\eta)$ and $\mu(x,\mu(y,z))=\mu(\mu(x,y),z)$. Note that there is only one morphism $0\to m$ in $\mat(R)$ for any $m$. It is not hard to show that for any $m\in\nn$ there is only one monoid structure. For example, when $m=2$, $\mu$ must be the following matrix
	\[
	\mu\coloneqq
  	\left(
  	\begin{array}{cc}
  		1&0\\
  		0&1\\
  		1&0\\
  		0&1
  	\end{array}
  	\right)
	\]	
	Anyway, for any monoid $(m,\mu,\eta)$, the morphism $U(\eta)\colon R^0\to R^m$ is given by $U(\eta)(1)\coloneqq (0,\ldots,0)$, and the morphism $U(\mu)\colon R^m\times R^m\to R^m$ is given by
	\[
		U(\mu)((a_1,\ldots,a_m),(b_1,\ldots,b_m))\coloneqq (a_1+b_1,\ldots,a_m+b_m).
	\]
	These give $R^m$ the structure of a monoid.
	\item The triple $(1,\add{1em},\zero{1em})$ corresponds to the additive monoid structure on $\rr$, e.g.\ with $(5,3)\mapsto 8$.
\end{enumerate}
}

\sol{exc.understand_reversed_icons}{
\begin{enumerate}
	\item What is the behavior $\beh(\coadd{1.3em})$ of the reversed addition icon $\coadd{1.3em}\colon 1 \to 2$?
	\item What is the behavior $\beh(\mult{1.3em})$ of the reversed copy icon, $\mult{1.3em}\colon 2\to 1$?
\end{enumerate}
}{
\begin{enumerate}
	\item The behavior $\beh(\coadd{1.3em})$ of the reversed addition icon $\coadd{1.3em}\colon 1 \to 2$ is the relation $\{(x,y,z)\in R^3\mid x=y+z\}$.
	\item The behavior $\beh(\mult{1.3em})$ of the reversed copy icon, $\mult{1.3em}\colon 2\to 1$ is the relation $\{(x,y,z)\in R^3\mid x=y=z\}$.
\end{enumerate}
}

\sol{exc.monoidal_prod_+}{
In \cref{def.rel_prop} we went quickly through monoidal products $+$ in the prop $\rel_R$. If $B\ss R^m\times R^n$ and $C\ss R^p\times R^q$ are morphisms in $\rel_R$, write down $B+C$ in set-notation.}
{
If $B\ss R^m\times R^n$ and $C\ss R^p\times R^q$ are morphisms in $\rel_R$, then take $B+C\ss R^{m+p}\times R^{n+q}$ to be the set
\[
	B+C\coloneqq\{(w,y,x,z)\in R^{m+p}\times R^{n+q}\mid (w,x)\in B\text{ and }(y,z)\in C\}.
\]
}

\sol{exc.SFG_composite_behavior}{
Let $g\colon m \to n$, $h\colon \ell \to n$ be signal flow graphs. Note that
$h\op\colon n \to \ell$ is a signal flow graph, and we can form the
composite $g\cp h\op$:
\[
  \begin{tikzpicture}[oriented WD, bb port length=0pt, bb port sep=1pt]
    \node[bb={4}{4}, minimum width = 2cm] (X) 
    {$\begin{array}{c} \longrightarrow \\ g \end{array}$};
	\node[bb={4}{4}, right= 1 of X, minimum width = 2cm] (Y)
    {$\begin{array}{c} \longleftarrow \\ h\op \end{array}$};
	\draw ($(X_in1)-(.2,0)$) to (X_in1);
	\draw ($(X_in2)-(.2,0)$) to (X_in2);
	\draw ($(X_in4)-(.2,0)$) to (X_in4);
	\draw (X_out1) to (Y_in1);
	\draw (X_out2) to (Y_in2);
	\draw (X_out4) to (Y_in4);
	\draw (Y_out1) to ($(Y_out1)+(.2,0)$);
	\draw (Y_out2) to ($(Y_out2)+(.2,0)$);
	\draw (Y_out4) to ($(Y_out4)+(.2,0)$);
	\draw[label]
		node at ($.5*(X_out3)+.5*(Y_in3)$) {$\vdots$}
		node[left=3pt of X_in3] {$\vdots$}
		node[right=3pt of Y_out3] {$\vdots$}
	;	
\end{tikzpicture}
\]
Show that the behavior of $g\cp(h\op)$ is equal to 
\[
  \beh(g\cp(h\op))=\{(x,y)\,\mid\, S(g)(x)=S(h)(y)\}\ss R^m\times R^\ell.
\]
}{
The behavior of $g\colon m\to n$ and $h\op\colon n\to\ell$ are respectively
\begin{align*}
	\beh(g)&=\{(x,z)\in R^m\times R^n\,\mid\, S(g)(x)=z\}\\
	\beh(h\op)&=\{(z,y)\in R^n\times R^\ell\,\mid\,z=S(h)(y)\}
\end{align*}
and by \cref{eqn.comp_rule_rels}, the composite $\beh(g\cp (h\op))=\beh(g)\cp\beh(h\op)$ is:
\[
  \{(x,y)\mid\text{ there exists } z\in R^n\text{ such that } S(g)(x)=z\text{ and }z=S(h)(y)\}.
\]
Since $S(g)$ and $S(h)$ are functions, the above immediately reduces to the desired formula:
\[  \beh(g\cp(h\op))=\{(x,y)\,\mid\, S(g)(x)=S(h)(y)\}.\]
}

\sol{exc.sfg_behavior_again}{
Let $g\colon m \to n$, $h\colon m \to p$ be signal flow graphs. Note that
$g\op\colon n \to m$ is a signal flow graph, and we can form the
composite $g\op\cp h$
\[
  \begin{tikzpicture}[oriented WD, bb port length=0pt, bb port sep=1pt]
    \node[bb={4}{4}, minimum width = 2cm] (X) 
    {$\begin{array}{c} \longleftarrow \\ g\op \end{array}$};
	\node[bb={4}{4}, right= 1 of X, minimum width = 2cm] (Y)
    {$\begin{array}{c} \longrightarrow \\ h \end{array}$};
	\draw ($(X_in1)-(.2,0)$) to (X_in1);
	\draw ($(X_in2)-(.2,0)$) to (X_in2);
	\draw ($(X_in4)-(.2,0)$) to (X_in4);
	\draw (X_out1) to (Y_in1);
	\draw (X_out2) to (Y_in2);
	\draw (X_out4) to (Y_in4);
	\draw (Y_out1) to ($(Y_out1)+(.5,0)$);
	\draw (Y_out2) to ($(Y_out2)+(.5,0)$);
	\draw (Y_out4) to ($(Y_out4)+(.5,0)$);
	\draw[label]
		node at ($.5*(X_out3)+.5*(Y_in3)$) {$\scriptstyle\vdots$}
		node[above left=-4pt and 3pt of X_in3] {$\vdots$}
		node[above right=-4pt and 3pt of Y_out3] {$\vdots$}
	;	
\end{tikzpicture}
\]
Show that the behavior of $g\op\cp h$ is equal to 
\[
  \{(S(g)(x),S(h)(x))\,\mid\, x \in R^m\}.
\]
}{
The behavior of $g\op\colon n\to m$ and $h\colon m\to p$ are respectively
\begin{align*}
	\beh(g\op)&=\{(y,x)\in R^n\times R^m\,\mid\, y=S(g)(x)\}\\
	\beh(h)&=\{(x,z)\in R^m\times R^p\,\mid\,S(h)(x)=z\}
\end{align*}
and by \cref{eqn.comp_rule_rels}, the composite $\beh((g\op)\cp h)=\beh(g\op)\cp\beh(h)$ is:
\[
  \{(y,z)\mid\text{ there exists } x\in R^m\text{ such that } y=S(g)(x)\text{ and }S(h)(x)=z\}.
\]
This immediately reduces to the desired formula:
\[
  \beh((g\op)\cp h)=\{(S(g)(x),S(h)(x))\,\mid\, x \in R^m\}.
\]
}

\sol{exc.linear_relations}{
Here is an exercise for those that know linear algebra. Let $R$ be a field, and $g\colon
m \to n$ a signal flow graph and let $S(g)\in\mat(R)$ be the associated $(m\times n)$-matrix (see \cref{thm.sfg_to_mat}).
\begin{enumerate}
  \item Show that the composite of $g$ with $0$-reverses, shown here
  \[
    \begin{tikzpicture}[oriented WD, bb port length=0pt, bb port sep=1pt]
      \node[bb={4}{4}, minimum width = 2cm] (X) 
      {$\begin{array}{c} \longrightarrow \\ g \end{array}$};
  	\draw ($(X_in1)-(.5,0)$) to (X_in1);
  	\draw ($(X_in2)-(.5,0)$) to (X_in2);
  	\draw ($(X_in4)-(.5,0)$) to (X_in4);
  	\draw (X_out1) to ($(X_out1)+(.5,0)$);
  	\draw (X_out2) to ($(X_out2)+(.5,0)$);
  	\draw (X_out4) to ($(X_out4)+(.5,0)$);
  	\draw[label]
  		node[above left=-4pt and 3pt of X_in3] {$\vdots$}
  		node[above right=-4pt and 3pt of X_out3] {$\vdots$}
  		node[wdot] at ($(X_out1)+(.5,0)$) {}
  		node[wdot] at ($(X_out2)+(.5,0)$) {}
  		node[wdot] at ($(X_out4)+(.5,0)$) {}
  	;	
  \end{tikzpicture}
  \]
  is equal to the kernel of the matrix $S(g)$.
  \item Show that the composite of discard-reverses with $g$, shown here 
  \[
    \begin{tikzpicture}[oriented WD, bb port length=0pt, bb port sep=1pt]
      \node[bb={4}{4}, minimum width = 2cm] (X) 
      {$\begin{array}{c} \longrightarrow \\ g \end{array}$};
  	\draw ($(X_in1)-(.5,0)$) to (X_in1);
  	\draw ($(X_in2)-(.5,0)$) to (X_in2);
  	\draw ($(X_in4)-(.5,0)$) to (X_in4);
  	\draw (X_out1) to ($(X_out1)+(.2,0)$);
  	\draw (X_out2) to ($(X_out2)+(.2,0)$);
  	\draw (X_out4) to ($(X_out4)+(.2,0)$);
  	\draw[label]
  		node[above left=-4pt and 3pt of X_in3] {$\vdots$}
  		node[above left=-4pt and 3pt of X_out3] {$\vdots$}
  		node[bdot] at ($(X_in1)-(.5,0)$) {}
  		node[bdot] at ($(X_in2)-(.5,0)$) {}
  		node[bdot] at ($(X_in4)-(.5,0)$) {}
  	;	
  \end{tikzpicture}
  \]
  is equal to the image of the matrix $S(g)$.
  \item Show that for any signal flow graph $g$, the subset
  $\beh(g)\subseteq R^m \times R^n$ is a linear subspace. That is, if $b_1,b_2\in \beh(g)$ then so are $b_1+b_2$ and $r*b_1$, for any $r\in R$.
\end{enumerate}
}{
\begin{enumerate}
	\item The behavior of the 0-reverse $\cozero{1.5em}$ is the subset $\{y\in R\mid y=0\}$, and its $n$-fold tensor is similarly $\{y\in R^n\mid y=0\}$. Composing this relation with $S(g)\ss R^m\times R^n$ gives $\{x\in R^m\mid S(g)=0\}$, which is the kernel of $S(g)$.
	\item The behavior of the discard-inverse $\unit{1.5em}$ is the subset $\{x\in R\}$, i.e.\ the largest subset of $R$, and similarly its $m$-fold tensor is $R^n\ss R^n$. Composing this relation with $S(g)\ss R^m\times R^n$ gives $\{y\in R^n\mid \text{ there exists }x\in R^m\text{ such that }S(g)(x)=y\}$, which is exactly the image of $S(g)$.
	\item For any $g\colon m\to n$, we first claim that the behavior $\beh(g)=\{(x,y)\mid S(g)(x)=y\}$ is linear, i.e.\ it is closed under addition and scalar multiplication. Indeed, $S(g)$ is multiplication by a matrix, so if $S(g)(x)=y$ then $S(g)(rx)=ry$ and $S(g)(x_1+x_2)=S(g)(x_1)+S(g)(x_2)$. Thus we conclude that $(x,y)\in\beh(g)$ implies $(rx,ry)\in\beh(g)$, so it's closed under scalar multiplication, and $(x_1,y_1),(x_2,y_2)\in\beh(g)$ implies $(x_1+x_2,y_1+y_2)\in\beh(g)$ so it's closed under addition. Similarly, the behavior $\beh(g\op)$ is also linear; the proof is similar.\\
	
	Finally, we need to show that the composite of any two linear relations is linear. Suppose that $B\ss R^m\times R^n$ and $C\ss R^n\times R^p$ are linear. Take $(x_1,z_1),(x_2,z_2)\in B\cp C$ and take $r\in R$. By definition, there exist $y_1,y_2\in R^n$ such that $(x_1,y_1),(x_2,y_2)\in B$ and $(y_1,z_1),(y_2,z_2)\in C$. Since $B$ and $C$ are linear, $(rx_1,ry_1)\in B$ and $(ry_1,rz_1)\in C$, and also $(x_1+x_2,y_1+y_2)\in B$ and $(y_1+y_2,z_1+z_2)\in C$. Hence $(rx_1,rz_1)\in (B\cp C)$ and $(x_1+x_2,z_1+z_2)\in (B\cp C)$, as desired.
\end{enumerate}
}

\sol{exc.linrel_prop}{
One might want to show that linear relations on $R$ form a prop, $\Cat{LinRel}_R$. That is, one might want to show that there is a sub-prop of the prop $\rel_R$ from \cref{def.rel_prop}, where the morphisms $m\to n$ are the subsets $B\ss R^m\times R^n$ such that $B$ is linear. In other words, where for any $(x,y)\in B$ and $r\in R$, the element $(r*x,r*y)\in R^m\times R^n$ is in $B$, and for any $(x',y')\in B$, the element $(x+x',y+y')$ is in $B$.

This is certainly doable, but for this exercise, we only ask that you prove that the composite of two linear relations is linear.
}
{
Suppose that $B\ss R^m\times R^n$ and $C\ss R^n\times R^p$ are linear. Their composite is the relation $(B\cp C)\ss R^m\times R^p$ consisting of all $(x,z)$ for which there exists $y\in R^n$ with $(x,y)\in B$ and $(y,z)\in C$. We want to show that the set $(B\cp C)$ is linear, i.e.\ closed under scalar multiplication and addition.

For scalar multiplication, take an $(x,z)\in (B\cp C)$ and any $r\in R$. Since $B$ is linear, we have $(r*x,r*y)\in B$ and since $C$ is linear we have $(r*y,r*z)\in C$, so then $(r*x,r*z)\in (B\cp C)$. For addition, if we also have $(x',z')\in(B\cp C)$ then there is some $y'\in R^n$ with $(x',y')\in B$ and $(y',z')\in C$, so since $B$ and $C$ are linear we have $(x+x',y+y')\in B$ and $(y+y',z+z')\in C$, hence $(x+x',z+z')\in (B\cp C)$.
}
\finishSolutionChapter
\section[Solutions for Chapter 6]{Solutions for \cref{chap.hypergraph_cats}.}

\sol{exc.initial_ob_practice}{
Consider the set $A=\{a,b\}$. Find a preorder relation $\leq$ on $A$ such that
\begin{enumerate}
	\item $(A,\leq)$ has no initial object.
	\item $(A,\leq)$ has exactly one initial object.
	\item $(A,\leq)$ has two initial objects.
\end{enumerate}
}
{
Let $A=\{a,b\}$, and consider the preorders shown here: \fbox{$\LMO{a}\quad\LMO{b}$}\,, \quad \fbox{$\LMO{a}\to\LMO{b}$}\,, \quad \fbox{$\LMO{a}\leftrightarrows\LMO{b}$}\,.
\begin{enumerate}
	\item The left-most (the discrete preorder on $A$) has no initial object, because $a\not\leq b$ and $b\not\leq a$.\index{preorder!discrete}
	\item The middle one has one initial object, namely $a$.
	\item The right-most (the co-discrete preorder on $A$) has two initial objects.\index{preorder!codiscrete}
\end{enumerate}
}

\sol{exc.initial_object}{
For each of the graphs below, consider the free category on that graph, and say whether it has an initial object.\\
\begin{enumerate*}[itemjoin=\hspace{.8in}]
	\item \fbox{$\LMO{a}$}
	\item \fbox{$\LMO{a}\to\LMO{b}\to\LMO{c}$}
	\item \fbox{$\LMO{a}\qquad\LMO{b}$}
	\item \fbox{$\begin{tikzcd}\LMO{a}\ar[loop right]\end{tikzcd}$}	
\end{enumerate*}
}{
Recall that the objects of a free category on a graph are the vertices of the
graph, and the morphisms are paths. Thus the free category on a graph $G$ has an
initial object if there exists a vertex $v$ that has a unique path to every
object. In 1.\ and 2., the vertex $a$ has this property, so the free categories
on graphs 1.\ and 2.\ have initial objects. In graph 3., neither $a$ nor $b$
have a path to each other, and so there is no initial object. In graph 4., the
vertex $a$ has many paths to itself, and hence its free category does not have
an initial object either.
}

\sol{exc.rig_initial_obj}{
Recall the notion of rig from \cref{chap.SFGs}. A \emph{rig homomorphism} from $(R,0_R,+_R,1_R,*_R)$ to $(S,0_S,+_S,,1_S, *_S)$ is a function $f\colon R\to S$ such that $f(0_R)=0_S$, $f(r_1+_Rr_2)=f(r_1)+_Sf(r_2)$, etc.
\begin{enumerate}
	\item We said ``etc.'' Guess the remaining conditions for $f$ to be a rig homomorphism.
	\item Let $\Cat{Rig}$ denote the category whose objects are rigs and whose morphisms are rig homomorphisms. We claim $\Cat{Rig}$ has an initial object. What is it?
\end{enumerate}	
}{
\begin{enumerate}
\item The remaining conditions are that $f(1_R)=1_S$, and that $f(r_1\ast_R r_2)
= f(r_1) \ast_S f(r_2)$.
\item The initial object in the category $\Cat{Rig}$ is the natural numbers rig
$(\nn,0,+,1,\ast)$. The fact that is initial means that for any other rig
$R=(R,0_R,+_R,1_R,\ast_R)$, there is a unique rig homomorphism $f\colon \nn \to
R$. 

What is this homomorphism? Well, to be a rig homomorphism, $f$ must send
$0$ to $0_R$, $1$ to $1_R$. Furthermore, we must also have $f(m + n) =
f(m) +_Rf(n)$, and hence 
\[
f(m) = f(\underbrace{1+1+\dots+1}_{m \text{ summands}}) =
\underbrace{f(1)+f(1)+\dots+f(1)}_{m \text{ summands}} = \underbrace{1_R+1_R+
\dots+1_R}_{m \text{ summands}}.
\]
So if there is a rig homomorphism $f\colon\nn\to R$, it must be given by the above formula. But does this formula work correctly for multiplication?

It remains to check $f(m \ast n) = f(m) \ast_R f(n)$, and this will follow from
distributivity. Noting that $f(m\ast n)$ is equal to the sum of $mn$ copies of
$1_R$, we have
\begin{align*}
f(m) \ast_R f(n) &= (\underbrace{1_R+\dots+1_R}_{m \text{ summands}}) \ast_R
(\underbrace{1_R+\dots+1_R}_{n \text{ summands}}) \\
&= \underbrace{1_R \ast (\underbrace{1_R+\dots+1_R}_{n \text{ summands}}) + \dots
+1_R\ast (\underbrace{1_R+\dots+1_R}_{n \text{ summands}})}_{m \text{ summands}} \\
&= \underbrace{1_R+\dots+1_R}_{mn \text{ summands}} = f(m\ast n).
\end{align*}
Thus $(\nn,0,+,1,\ast)$ is the initial object in $\Cat{Rig}$.
\end{enumerate}
}

\sol{exc.universality}{
Explain the statement ``the hallmark of universality is the existence of a
unique map to any other comparable object,'' in the context of
\cref{def.initial_obj}. In particular, what is universal? What is the
``comparable object''?
}{
In \cref{def.initial_obj}, it is the initial object $\varnothing\in\cat{C}$ that is
universal. In this case, all objects $c\in\cat{C}$ are `comparable objects'.
So the universal property of the initial object is that to any \emph{object} $c\in\cat{C}$,
there is a unique map $\varnothing\to c$ coming from the initial object.
}

\sol{exc.initials_are_isomorphic}{
Let $\cat{C}$ be a category, and suppose that $c_1$ and $c_2$ are initial objects. Find an isomorphism between them, using the universal property from \cref{def.initial_obj}.
}
{
If $c_1$ is initial then by the universal property, for any $c$ there is a unique morphism $c_1\to c$; in particular, there is a unique morphism $c_1\to c_2$, call it $f$. Similarly, if $c_2$ is initial then there is a unique morphism $c_2\to c_1$, call it $g$. But how do we know that $f$ and $g$ are mutually inverse? Well since $c_1$ is initial there is a unique morphism $c_1\to c_1$. But we can think of two: $\id_{c_1}$ and $f\cp g$. Thus they must be equal. Similarly for $c_2$, so we have $f\cp g=\id_{c_1}$ and $g\cp f=\id_{c_2}$, which is the definition of $f$ and $g$ being mutually inverse.
}

\sol{exc.join_as_coproduct}{
  Explain why coproducts in a preorder are the same as joins.
}{
Let $(P,\le)$ be a preorder, and $p,q \in P$. Recall that a preorder is a
category with at most one morphism, denoted $\le$, between any two objects. 
Also recall that all diagrams in a preorder commute, since this means any two
morphisms with the same domain and codomain are equal. 

Translating \cref{def.coproduct} to this case, a coproduct $p+q$ is $P$ is an
element of $P$ such that $p \le p+q$ and $q\le p+q$, and such that for all elements $x
\in P$ with maps $p \le x$ and $q \le x$, we have $p+q \le x$. But this says
exactly that $p+q$ is a join: it is a least element above both $p$ and
$q$. Thus coproducts in preorders are exactly the same as joins.
}

\sol{exc.copairing}{
	Suppose $T=\{a,b,c,\ldots,z\}$ is the set of letters in the alphabet, and let $A$ and $B$ be the sets from \cref{eqn.apples_oranges_again}. Consider the function $f\colon A\to T$ sending each element of $A$ to the first letter of its label, e.g.\ $f(\mathrm{apple})=a$. Let $g\colon B\to T$ be the function sending each element of $B$ to the last letter of its label, e.g.\ $g(\mathrm{apple})=e$. Write down the function $\copair{f,g}(x)$ for all elements of $A\sqcup B$.
}{
The function $\copair{f,g}$ is defined by
\begin{align*}
\copair{f,g}\colon A \sqcup B &\longrightarrow T \\
\mathrm{apple1} &\longmapsto a \\
\mathrm{banana1} &\longmapsto b \\
\mathrm{pear1} &\longmapsto p \\
\mathrm{cherry1} &\longmapsto c \\
\mathrm{orange1} &\longmapsto o \\
\mathrm{apple2} &\longmapsto e \\
\mathrm{tomato2} &\longmapsto o \\
\mathrm{mango2} &\longmapsto o.
\end{align*}
}

\sol{exc.coprod_properties}{
  Let $f \colon A \to C$, $g \colon B \to C$, and $h\colon C \to D$ be morphisms
  in a category $\cat{C}$ with coproducts. Show that
  \begin{enumerate}
  \item $\iota_A \cp \copair{f,g} = f$.
  \item $\iota_B \cp \copair{f,g} = g$.
  \item $\copair{f,g}\cp h = \copair{f \cp h,g\cp h}$. 
  \item $\copair{\iota_A,\iota_B} = \id_{A+B}$.
  \end{enumerate}
}{
  \begin{enumerate}
  \item The equation $\iota_A \cp \copair{f,g} = f$ is the commutativity of the
  left hand triangle in the commutative diagram
  \eqref{eqn.universal_prop_coprod} defining $\copair{f,g}$.
  \item The equation $\iota_B \cp \copair{f,g} = g$ is the commutativity of the
  right hand triangle in the commutative diagram
  \eqref{eqn.universal_prop_coprod} defining $\copair{f,g}$.
  \item The equation $\copair{f,g}\cp h = \copair{f \cp h,g\cp h}$ follows from
  the universal property of the coproduct. Indeed, the diagram
  \[
    \begin{tikzcd}[ampersand replacement=\&, column sep=large, cramped]
      A \ar[r,"\iota_A"] \ar[dr,"f"' near end]  \ar[ddr, "f\cp h"']
      \& A+B \ar[d, "\copair{f,g}" description] 
      \& B \ar[l,"\iota_B"'] \ar[dl, "g" near end] \ar[ddl, "g\cp h"] \\[10pt]
      \& C \ar[d, "h" description] \\[10pt]
      \& D
    \end{tikzcd}
  \]
  commutes, and the universal property says there is a unique map $\copair{f\cp
  h, g \cp h}\colon A+B \to D$ for which this occurs. Hence we must have $[f,g]
  \cp h = \copair{f \cp h, g \cp h}$. 
  \item Similarly, to show $\copair{\iota_A,\iota_B} = \id_{A+B}$, observe that
  the diagram 
  \[
    \begin{tikzcd}[ampersand replacement=\&, cramped]
      A \ar[r,"\iota_A"] \ar[dr,"\iota_A"'] \& A+B \ar[d,equal, "\id_{A+B}" description] \& B
      \ar[l,"\iota_B"'] \ar[dl, "\iota_B"] \\[15pt]
      \& A+B
    \end{tikzcd}
  \]
  trivially commutes. Hence by the uniqueness in
  \eqref{eqn.universal_prop_coprod}, $\copair{\iota_{A}, \iota_{B}}=\id_{A+B}$.
  \end{enumerate}
}

\sol{exc.coproducts_give_monoidal_structure}{
Suppose a category $\cat C$ has coproducts, denoted $+$, and an initial object,
denoted $\varnothing$. Then we can show $(\cat C,+,\varnothing)$ is a symmetric
monoidal category (recall \cref{rdef.sym_mon_cat}). In this exercise we develop
the data.
\begin{enumerate}
\item Show that $+$ extends to a functor $\cat C \times \cat C \to \cat C$. In particular, how does it act on morphisms in $\cat{C}\times\cat{C}$?
\item Using the universal properties of the initial object and coproduct, show
that there are isomorphisms $A+\varnothing \to A$ and $\varnothing +A \to A$.
\item Using the universal property of the coproduct, write down morphisms
\begin{enumerate}
\item $(A+B)+C \to A+(B+C)$.
\item $A+B \to B+A$.
\end{enumerate}
If you like, check that these are isomorphisms.
\end{enumerate}
It can then be checked that this data obeys the axioms of a symmetric monoidal
category, but we'll end the exercise here.
}{
This exercise is about showing that coproducts and an initial object give a
symmetric monoidal category. Since all we have are coproducts and an initial
object, and since these are defined by their universal properties, the solution
is to use these universal properties over and over, to prove that all the data
of \cref{rdef.sym_mon_cat} can be constructed.
\begin{enumerate}
\item To define a functor $+\colon \cat{C} \times \cat{C} \to \cat{C}$ we must
define its action on objects and morphisms. In both cases, we just take the
coproduct. If $(A,B)$ is an object of $\cat{C} \times \cat{C}$, its image $A+B$
is, as usual, the coproduct of the two objects of $\cat{C}$. If $(f,g) \colon
(A,B)\to (C,D)$ is a morphism, then we can form a morphism $f+g = \copair{f \cp
\iota_C, g \cp \iota_D}\colon A+B \to C+D$, where $\iota_C \colon C \to C+D$ and
$\iota_D\colon D \to C+D$ are the canonical morphisms given by the definition of
the coproduct $A+B$.

Note that this construction sends identity morphisms to identity morphisms,
since by \cref{exc.coprod_properties} 4 we have 
\[
\id_A + \id_B = \copair{\id_A\cp \iota_{A}, \id_B \cp \iota_{B}} =\copair{\iota_{A}, \iota_{B}}=\id_{A+B}.
\]

To show that $+$ is a functor, we need to also show it preserves composition.
Suppose we also have a morphism $(h,k)\colon (C,D) \to (E,F)$ in $\cat{C} \times
\cat{C}$. We need to show that $(f+g) \cp (h+k) = (f \cp h) + (g \cp k)$. This
is a slightly more complicated version of the argument in
\cref{exc.coprod_properties} 3. It follows from the fact the diagram below
commutes:
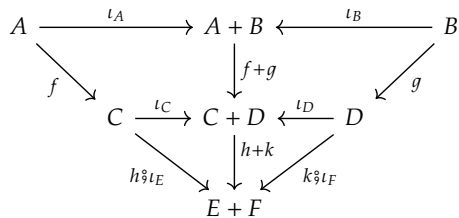
\[
    \begin{tikzcd}[ampersand replacement=\&]
      A \ar[rr,"\iota_A"] \ar[dr,"f"'] \& \& A+B \ar[d, "f+g"] \&\& B
      \ar[ll,"\iota_B"'] \ar[dl, "g"] \\[5pt]
      \& C \ar[r, "\iota_C"] \ar[dr,"h\cp \iota_E"'] \& C+D \ar[d,"h+k" near
      start] \& D
      \ar[l,"\iota_D"'] \ar[dl, "k \cp \iota_F"] \\[5pt]
      \&\& E+F
    \end{tikzcd}
\]
Indeed, we again use the uniqueness of the copairing in
\eqref{eqn.universal_prop_coprod}, this time to show that $(f \cp h) + (g \cp
k)=\copair{f\cp h \cp \iota_E, g\cp k \cp \iota_F} = (f+g) \cp (h+k)$, as
required.

\item Recall the universal property of the initial object gives a unique map
$!_A\colon \varnothing \to A$. Then the copairing $\copair{\id_A,!_A}$ is a map
$A+\varnothing \to A$. Moreover, it is an isomorphism with inverse $\iota_A
\colon A \to A +\varnothing$. Indeed, using the properties in
\cref{exc.coprod_properties} and the universal property of the initial object,
we have $\iota_A \cp \copair{\id_A,!_A} = \id_A$, and 
\[
\copair{\id_A,!_A} \cp \iota_A = \copair{\id_A \cp \iota_A, !_A \cp \iota_A} =
\copair{\iota_A,!_{A+\varnothing}} = \copair{\iota_A,\iota_\varnothing} =
\id_{A+\varnothing}.
\]

An analogous argument shows $\copair{!_A,\id_A} \colon \varnothing +A \to A$ is
an isomorphism.

\item We'll just write down the maps and their inverses; we leave it to you, if
you like, to check that they indeed are inverses.
\begin{enumerate}
\item The map $\copair{\id_A +\iota_B,\iota_C} = \copair{\copair{\iota_A,\iota_B \cp
\iota_{B+C}},\iota_C\cp\iota_{B+C}} \colon (A+B)+C \to A+(B+C)$ is an isomorphism,
with inverse $\copair{\iota_A, \iota_B+\id_C} \colon A+(B+C) \to (A+B)+C$.
\item The map $\copair{\iota_A,\iota_B}\colon A+B \to B+A$ is an isomorphism.
Note our notation here is slightly confusing: there are two maps named
$\iota_A$, (i) $\iota_A \colon A \to A+B$, and (ii) $\iota_A\colon A \to B+A$, and
similarly for $\iota_B$. In the above we mean the map (ii). It has inverse
$\copair{\iota_A,\iota_B}\colon B+A \to A+B$, where in this case we mean the
map (i).
\end{enumerate}
\end{enumerate}
}

\sol{exc.disc_cats_have_pushouts}{
For any set $S$, we have the discrete category $\Cat{Disc}_S$, with $S$ as objects and only identity morphisms.
\begin{enumerate}
	\item Show that all pushouts exist in $\Cat{Disc}_S$, for any set $S$.
	\item For what sets $S$ does $\Cat{Disc}_S$ have an initial object?
\end{enumerate}
}
{
\begin{enumerate}
	\item Suppose given an arbitrary diagram of the form $B\from A\to C$ in $\Cat{Disc}_S$; we need to show that it has a pushout. The only morphisms in $\Cat{Disc}_S$ are identities, so in particular $A=B=C$, and the square consisting of all identities is its pushout.
	\item Suppose $\Cat{Disc}_S$ has an initial object $s$. Then $S$ cannot be empty! But it also cannot have more than one object, because if $s'$ is another object then there is a morphism $s\to s'$, but the only morphisms in $S$ are identities so $s=s'$. Hence the set $S$ must consist of exactly one element.
\end{enumerate}
}
\sol{exc.pushout}{
  What is the pushout of the functions $f\colon \ord{4} \to \ord{5}$ and $g\colon \ord{4} \to
  \ord{3}$ pictured below?
  \[ 

  \]
 Check your answer using the abstract description from \cref{ex.pushouts}.
}
{
The pushout is the set $\ord{4}$, as depicted in the top right in the diagram
below, equipped also with the depicted functions:
  \begin{equation}\label{eqn.pushout_4534}
    %
    \end{equation}

We want to see that this checks out with the description from \cref{ex.pushouts}, i.e.\ that it is the set of equivalence classes in $\ord{5}\sqcup\ord{3}$ generated by the relation $\{f(a)\sim g(a)\mid a\in\ord{4}\}$. If we denote elements of $\ord{5}$ as $\{1,\ldots,5\}$ and those of $\ord{3}$ as $\{1',2',3'\}$, we can redraw the functions $f,g$:
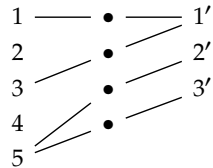
\[
\begin{tikzcd}[ampersand replacement=\&, row sep=0]
	1\& \bullet\ar[l, -]\ar[r, -]\&1'\\
	2\& \bullet\ar[dl, -]\ar[ur, -]\&2'\\
	3\& \bullet\ar[ddl, -]\ar[ur,-]\&3'\\
	4\& \bullet\ar[dl, -]\ar[ur, -]\\
	5
\end{tikzcd}
\]
which says we take the equivalence relation on $\ord{5}\sqcup\ord{3}$ generated by: $1\sim 1'$, , $3\sim 1'$, $5\sim 2'$, and $5\sim 3'$. The equivalence classes are $\{1,1',3\}$, $\{2\}$, $\{4\}$, and $\{5,2',3'\}$. These four are exactly the four elements in the set labeled `pushout' in \cref{eqn.pushout_4534}.
}

\sol{exc.pushout_initial}{
In \cref{ex.pushout_initial} we asked ``why?'' three times.
\begin{enumerate}
	\item Give a justification for ``why?$^1$.''
	\item Give a justification for ``why?$^2$.''
	\item Give a justification for ``why?$^3$.''	
\end{enumerate}
}{
\begin{enumerate}
	\item The diagram to the left commutes because $\varnothing$ is initial,
	and so has a unique map $\varnothing \to X + Y$. This implies we
	must have $f\cp \iota_X = g \cp \iota_Y$.
	\item There is a unique map $X + Y \to T$ making the diagram in
	\eqref{eqn.univ_prop_pushout} commute simply by the universal property
	of the coproduct \eqref{eqn.universal_prop_coprod} applied to the maps
	$x\colon X \to T$ and $y\colon Y \to T$.
	\item Suppose $X+_\varnothing Y$ exists. By the universal property of
	$\varnothing$, given any pair of arrows $x\colon X \to T$ and $y\colon Y
	\to T$, the diagram	
	
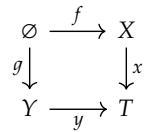
\[
	\begin{tikzcd}[ampersand replacement =\&]
	  \varnothing \ar[r,"f"] \ar[d,"g"']\&X\ar[d,"x"]\\
	  Y\ar[r,"y"'] \& T
	\end{tikzcd}
	\]
	commutes. This means, by the universal property of the pushout
	$X+_\varnothing Y$, there exists a unique map $t\colon X+_\varnothing Y \to T$
	such that $\iota_X \cp t = x$ and $\iota_Y \cp t =y$. Thus
	$X+_\varnothing Y$ is the coproduct $X+Y$.
\end{enumerate}
}

\sol{exc.W_colim}{
Check that the pushout of pushouts from \cref{ex.W_colim} satisfies the universal property of the colimit for the original diagram, \cref{eqn.W_colim}.
}
{
We have to check that the colimit of the diagram shown left really is given by taking three pushouts as shown right:
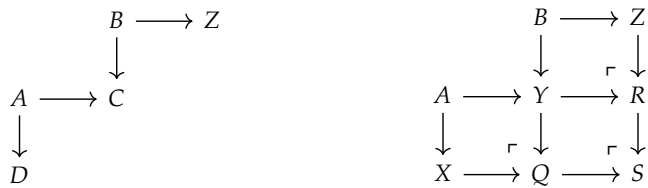
\[
\begin{tikzcd}[ampersand replacement=\&]
	\&
	B\ar[r]\ar[d]\&
	Z\\
	A\ar[r]\ar[d]\&
	C\\
	D
\end{tikzcd}
\hspace{1in}
\begin{tikzcd}[ampersand replacement=\&]
	\&
	B\ar[r]\ar[d]\&
	Z\ar[d]\\
	A\ar[r]\ar[d]\&
	Y\ar[r]\ar[d]\&
	R\ar[ul, phantom, very near start, "\ulcorner"]\ar[d]\\
	X\ar[r]\&
	Q\ar[ul, phantom, very near start, "\ulcorner"]\ar[r]\&
	S\ar[ul, phantom, very near start, "\ulcorner"]
\end{tikzcd}
\]
That is, we need to show that $S$, together with the maps from $A$, $B$, $X$, $Y$, and $Z$, has the required universal property. So suppose given an object $T$ with two commuting diagrams as shown:
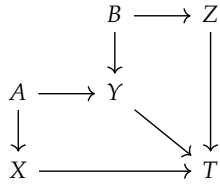
\[
\begin{tikzcd}[ampersand replacement=\&]
	\&
	B\ar[r]\ar[d]\&
	Z\ar[dd]\\
	A\ar[r]\ar[d]\&
	Y\ar[rd]\&
	\\
	X\ar[rr]\&
	\&
	T
\end{tikzcd}
\]
We need to show there is a unique map $S\to T$ making everything commute. Since $Q$ is a pushout of $X\from A\to Y$, there is a unique map $Q\to T$ making a commutative triangle with $Y$, and since $R$ is the pushout of $Y\from B\to Z$, there is a unique map $R\to T$ making a commutative triangle with $Y$. This implies that there is a commuting $(Y,Q,R,T)$ square, and hence a unique map $S\to T$ from its pushout making everything commute. This is what we wanted to show.
}

\sol{exc.pushout_formula}{
Use the formula to show that pushouts---colimits on a diagram $X \xleftarrow{f}
N \xrightarrow{g} Y$---agree with the description we gave in \cref{ex.pushouts}.
}{
The formula in \cref{thm.colims_in_set} says that the pushout $X+_NY$ is given
by the set of equivalence classes of $X \sqcup N \sqcup Y$ under the equivalence
relation generated by $x \sim n$ if $x=f(n)$, and $y \sim n$ if $y = g(n)$,
where $x \in X$, $y\in Y$, $n \in N$.  Since for every $n \in N$ there exists an
$x \in X$ such that $x=f(n)$, this set is equal to the set of equivalence
classes of $X \sqcup Y$ under the equivalence relation generated by $x \sim y$
if there exists $n$ such that $x=f(n)$ and $y = g(n)$. This is exactly the
description of \cref{ex.pushouts}.
}

\sol{exc.cospan_tensor}{
In \cref{eq.cospan_comp} we showed morphisms $A\to B$ and $B\to C$ in
$\cospan{\finset}$. Draw their monoidal product as a morphism $A+B\to
B+C$ in $\cospan{\finset}$.
}{
The monoidal product is 
\[

\]
}

\sol{exc.cospan_comp}{
  Depicting the composite of cospans in \cref{eq.cospan_comp} with the wire
  notation gives
  \[
    \begin{aligned}
      %
    \end{aligned}
  \]
  Comparing \cref{eq.cospan_comp} and \cref{eq.cospan_comp_wires}, describe the
  composition rule in $\cospan\finset$ in terms of wires and connected
  components.
}{
Let $x$ and $y$ be composable cospans in $\cospan\finset$. In terms of wires and
connected components, the composition rule in $\cospan\finset$ says that (i) the
composite cospan has a unique element in the apex for every connected component
of the concatenation of the wire diagrams $x$ and $y$, and (ii) in the wire
diagram for $x\cp y$, each element of the feet is connected by a wire to the
element representing the connected component to which it belongs.
}

\sol{exc.spider}{
  Which morphisms in the following list are equal?
  \begin{enumerate}
    \item 
      \[
\begin{tikzpicture}[spider diagram]
  \node[spider={2}{3}, fill=black] (a) {};
\end{tikzpicture}
\]
    \item 
      \[
\begin{tikzpicture}[spider diagram]
  \node[spider={2}{0}, fill=black] (a) {};
  \node[spider={0}{3}, fill=black, right=.5 of a] (b) {};
\end{tikzpicture}
\]
    \item
      \[
\begin{tikzpicture}[spider diagram]
  \node[spider={1}{2}, fill=black] (a) {};
  \node[spider={2}{2}, fill=black, right=1 of a] (b) {};
  \node[spider={1}{1}, fill=black, below=.5 of a] (c) {};
  \node[spider={1}{1}, fill=black, right=1 of c] (d) {};
  \begin{scope}
    \draw (a_out1) to (b_in1);
    \draw (a_out2) to (b_in2);
    \draw (c_out1) to (d_in1);
  \end{scope}
\end{tikzpicture}
\]
    \item
      \[
\begin{tikzpicture}[spider diagram]
  \node[spider={2}{2}, fill=black] (a) {};
  \node[spider={2}{3}, fill=black, right=1 of a] (b) {};
  \begin{scope}
    \draw (a_out1) to (b_in1);
    \draw (a_out2) to (b_in2);
  \end{scope}
\end{tikzpicture}
\]
    \item
      \[
\begin{tikzpicture}[spider diagram]
  \node[spider={1}{2}, fill=black] (a) {};
  \node[spider={2}{2}, fill=black, right=1 of a] (b) {};
  \node[spider={1}{2}, fill=black, below=.5 of a] (c) {};
  \node[spider={2}{1}, fill=black, right=1 of c] (d) {};
  \begin{scope}
    \draw (a_out1) to (b_in1);
    \draw (a_out2) to (b_in2);
    \draw (c_out1) to (d_in1);
    \draw (c_out2) to (d_in2);
  \end{scope}
\end{tikzpicture}
\]
    \item
      \[
\begin{tikzpicture}[spider diagram]
  \node[spider={1}{2}, fill=black] (a) {};
  \node[spider={1}{2}, fill=black, below=.8 of a] (b) {};
  \node[spider={2}{2}, fill=black, below right=.36 and .9 of a] (c) {};
  \node[spider={2}{1}, fill=black, right=1.9 of a] (d) {};
  \begin{scope}
    \draw (a_out1) to (d_in1);
    \draw (a_out2) to (c_in1);
    \draw (b_out1) to (c_in2);
  \draw (b_out2) to (d_out1|-b_out2);
    \draw (c_out1) to (d_in2);
    \draw (c_out2) to (d_out1|-c_out2);
  \end{scope}
\end{tikzpicture}
\]  
\end{enumerate}
}{
Morphisms 1, 4, and 6 are equal, and morphisms 3 and 5 are equal. Morphism 3 is
not equal to any other depicted morphism. This is an immediate consequence of
\cref{thm.spider}.
}

\sol{exc.suppressed_labels}{
\begin{enumerate}
	\item What label should be on the input to $h$?
	\item What label should be on the output of $g$?
	\item What label should be on the fourth output wire of the composite?
\end{enumerate}
}{
\begin{enumerate}
	\item The input to $h$ should be labelled $B$.
	\item The output of $g$ should be labelled $D$, since we know from the
	labels in the top right that $h$ is a morphism $B \to D \otimes D$.
	\item The fourth output wire of the composite should be labelled $D$
	too!
\end{enumerate}
}

\sol{exc.frob_cospan}{
  By \cref{ex.cospan_hypergraph}, the category $\cospan\finset$ is a hypergraph
  category. (In fact it is a hypergraph prop.) Draw the Frobenius maps for the
  object $\ord{1}$ in $\cospan\finset$ using both the function and wiring
  depictions as in \cref{ex.cospan_finset}.
}{
We draw the function depictions above, and the wiring depictions below. Note
that we depict the empty set with blank space.
  \[
    \begin{aligned}

    \end{aligned}
  \]
}

\sol{exc.frob_cospan2}{
  Using your knowledge of colimits, show that the maps defined in
  \cref{ex.cospan_hypergraph} do indeed obey the special law (see
  \cref{def.spec_comm_frob_mon}).
}{
%

  The special law says that the composite of cospans
  \[
  \spec{.1\textwidth} = X \xrightarrow{\id} X \xleftarrow{\copair{\id,\id}} X+X
  \xrightarrow{\copair{\id,\id}} X \xleftarrow{\id} X
  \]
  is the identity. This comes down to checking that the square
  \begin{equation} \label{eqn.pushout_questionmark}
  \begin{tikzcd}[ampersand replacement=\&]
  X+X \ar[d,"\copair{\id,\id}"'] \ar[r,"\copair{\id,\id}"] \&
  X\ar[d,"\id"]  \\
  X \ar[r,"\id"'] \& X 
  \end{tikzcd}
  \end{equation}
  is a pushout square. It is trivial to see that the square commutes. Suppose
  now that we have maps $f\colon X \to Y$ and $g\colon X \to Y$ such that 
  \[
  \begin{tikzcd}[ampersand replacement=\&]
  X+X \ar[d,"\copair{\id,\id}"'] \ar[r,"\copair{\id,\id}"] \&
  X\ar[d,"f"]  \\
  X \ar[r,"g"'] \& T
  \end{tikzcd}
  \]
  Write $\iota_1\colon X \to X+X$ for the map into the first copy of $X$ in
  $X+X$, given by the definition of coproduct. Then, using the fact that
  $\iota_1 \cp \copair{\id,\id} = \id$ from \cref{exc.coprod_properties} 1, and
  the commutativity of the above square, we have $f=
  \iota_1\cp\copair{\id,\id}\cp f =\iota_1\cp\copair{\id,\id}\cp g = g$. This
  means that $f\colon X \to T$ is the unique map such that
  \[
    \begin{tikzcd}[ampersand replacement=\&]
      X \ar[r,"\copair{\id,\id}"] \ar[d,"\copair{\id,\id}"'] \& X \ar[d,"\id"]
      \ar[ddr,"f", bend left] \\
      X \ar[r,"\id"'] \ar[drr,"g=f"', bend right=20pt] \& X \ar[dr,dashed, "f"] \\[-7pt]
      \&\&[-15pt]T
    \end{tikzcd}
  \]
  commutes, and so \eqref{eqn.pushout_questionmark} is a pushout square.

}

%
%

\sol{exc.fill_in_diagram}{
Fill in the missing diagram in the proof of \cref{prop.hyp_cat_comp_closed}.
}{
The missing diagram is
\[
  \begin{aligned}
    \begin{tikzpicture}[yscale=.7]
      \begin{pgfonlayer}{nodelayer}
	\node [style=none] (left) at (-1.5, 1) {};
	\node [style=none] (helpl) at (-0.5, 1) {};
	\node [style=none] (unith) at (-0.5, 0) {};
	\node [style=bdot] (unit) at (-0.75, 0) {};
	\node [style=bdot] (mult) at (0, .5) {};
	\node [style=bdot] (comult) at (.5, .5) {};
	\node [style=none] (counith) at (1, 1) {};
	\node [style=bdot] (counit) at (1.25, 1) {};
	\node [style=none] (helpr) at (1, 0) {};
	\node [style=none] (right) at (2, 0) {};
      \end{pgfonlayer}
      \begin{pgfonlayer}{edgelayer}
	\draw (left.center) to (helpl.center);
	\draw (unit.center) to (unith.center);
	\draw [bend left, looseness=1.00] (helpl.center) to (mult.center);
	\draw [bend right, looseness=1.00] (unith.center) to (mult.center);
	\draw (mult) to (comult);
	\draw [bend left, looseness=1.00] (comult.center) to (counith.center);
	\draw [bend right, looseness=1.00] (comult.center) to (helpr.center);
	\draw (counith.center) to (counit.center);
	\draw (helpr.center) to (right.center);
      \end{pgfonlayer}
    \end{tikzpicture}
  \end{aligned}
\]
}

\sol{exc.powset_mon_coherence}{
  Check that the maps $\varphi_{S,T}$ defined in \cref{ex.powset} are natural in
  $S$ and $T$. In other words, given $f\colon S\to S'$ and $g\colon T\to T'$, show that the diagram below commutes:
  \[
  \begin{tikzcd}[column sep=large,ampersand replacement=\&]
  	\powset(S)\times\powset(T)\ar[r,"\varphi_{S,T}"] \ar[d,"\im_f\times \im_g"']\&
		\powset(S\times T)\ar[d, "\im_{f\times g}"]\\
		\powset(S')\times\powset(T')\ar[r, "\varphi_{S',T'}"']\&
		\powset(S'\times T')
  \end{tikzcd}  
  \]
}{
Let $A \subseteq S$ and $B \subseteq T$. Then 
\begin{align*}
\varphi_{S',T'}\left((\im_f\times \im_g)(A \times B)\right) 
&= \varphi_{S',T'}(\{f(a) \mid a \in A\} \times \{g(b) \mid b \in B\}) \\
&= \{(f(a),g(b)) \mid a \in A, \, b \in B\} \\
&= \im_{f\times g}(A \times B) \\
&= \im_{f\times g}(\varphi_{S,T}(A,B)).
\end{align*}
Thus the required square commutes.
}

\sol{exc.cospan_as_fcospan}{
	Suppose you're worried that the notation $\cospan{\cat{C}}$ looks like the notation $\cospan{F}$, even though they're very different. An expert tells you ``they're not so different; one is a special case of the other. Just use the constant functor $F(c)\coloneqq\{*\}$.'' What do they mean?
}{
They mean that every category $\cospan{\cat{C}}$ is equal to a category
$\cospan{F}$, for some well-chosen $F$. They also tell you how to choose this
$F$: take the functor $F\colon \cat{C} \to \smset$ that sends every object of
$\cat{C}$ to the set $\{\ast\}$, and every morphism of $\cat{C}$ to the identity
function on $\{\ast\}$. Of course, you will have to check this functor is a lax
symmetric monoidal functor, but in fact this is not hard to do.

To check that $\cospan{\cat{C}}$ is equal to $\cospan{F}$, first observe that
they have the same objects: the objects of $\cat{C}$. Next, observe that a
morphism in $\cospan{F}$ is a cospan $X \leftarrow N \rightarrow Y$ in $\cat{C}$
together with an element of $FN = \{\ast\}$. But $FN$ also has a unique element,
$\ast$! So there's no choice here, and we can consider morphisms of $\cospan{F}$
just to be cospans in $\cat{C}$. Moreover, composition of morphisms in
$\cospan{F}$ is simply the usual composition of cospans via pushout, so
$\cospan{F} =\cospan{\cat{C}}$.

(More technically, we might say that $\cospan{\cat{C}}$ and $\cospan{F}$ are
isomorphic, where the isomorphism is the identity-on-objects functor
$\cospan{\cat{C}} \to \cospan{F}$ that simply decorates each cospan with $\ast$,
and its inverse is the one that forgets this $\ast$. But this is close enough to
equal that many category theorists, us included, don't mind saying equal in this
case.)
}

\sol{exc.circuit_tuple}{
  Write a tuple $(V,A,s,t,\ell)$ that represents the circuit in
  \cref{eq.circuit}.
}{
We can represent the circuit in \cref{eq.circuit} by the tuple
$(V,A,s,t,\ell)$ where $V=\{\textrm{ul},\textrm{ur},\textrm{dl},\textrm{dr}\}$,
$A=\{\textrm{r1},\textrm{r2},\textrm{r3},\textrm{c1},\textrm{i1}\}$, and $s$, $t$, and $\ell$ are
defined by the table
\[
\begin{tabular}{c | c c c c c}
    & r1 & r2 & r3 & c1 & i1 \\ \hline
$s(-)$ & dl & ul & ur & ul & dl \\ 
$t(-)$ & ul & ur & dr & ur & dr \\
$\ell(-)$ & $1\Omega$ & $2\Omega$ & $1\Omega$ & $3F$ & $1H$ 
\end{tabular}
\]
}

\sol{exc.pushforwardcircuit}{
  To understand this functor better, let $c\in F(\ord{4})$ be the circuit
\[
  %
\]
and let $f\colon \ord{4} \to \ord{3}$ be the function
\[
  %
\]
What is the circuit $\elec(f)(c)$?
}{
The circuit $\elec(f)(c)$ is
\[
  %
\]
}

\sol{exc.parallelcirc}{
  Suppose we have circuits  
\[
\begin{aligned}
  %
\end{aligned}
\]
in $\elec(\ord{2})$. Use the definition of $\psi_{V,V'}$ from
\eqref{eqn.Flaxator} to figure out what $4$-vertex circuit $\psi_{\ord{2},\ord{2}}(b,s) \in
\elec(\ord{2} + \ord{2}) = \elec(\ord{4})$ should be.
}
{
The circuit $\psi_{\ord{2},\ord{2}}(b,s)$ is the disjoint union of the two
labelled graphs $b$ and $s$:
\[
\begin{aligned}
  %
\end{aligned}
\]
}

\sol{exc.namethedecoration}{
Morphisms of $\cospan\elec$ are $\elec$-decorated cospans, as defined in
\cref{def.decorated_cospan}. This means \eqref{eqn.decorated_cospan} depicts a
cospan together with a \emph{decoration}, which is some $C$-circuit
$(V,A,s,t,\ell) \in \elec(\ord{2})$. What is it?
}{
The cospan is the cospan $\ord{1} \To{f} \ord{2} \From{g} \ord{1}$, where
$f(1)=1$ and $g(1)=2$. The decoration is the $C$-circuit
$(\ord{2},\{a\},s,t,\ell)$, where $s(a)=1$, $t(a)=2$ and $\ell(a)=\battery$. 
}

\sol{exc.composefcospans}{
Refer back to the example at the beginning of \cref{sec.deccospans}. In
particular, consider the composition of circuits in
\cref{eq.circuitcomposition}. Express the two circuits in this diagram as
morphisms in $\cospan\elec$, and compute their composite. Does it match the
picture in \cref{eqn.circuitcomposed}?
}{
Recall the circuit $C\coloneqq (V,A,s,t,\ell)$ from the solution to
\cref{exc.circuit_tuple}. Then the first decorated cospan is given by the cospan
$\ord{1} \To{f} V \From{g} \ord{2}$, $f(1) =\textrm{ul}$, $g(1) =\textrm{ur}$,
and $g(2) =\textrm{ur}$, decorated by circuit $C$.
The second decorated cospan is given by the cospan $\ord{2} \To{f'} V' \From{g'}
\ord{2}$ and the circuit $C'\coloneqq(V',A',s',t',\ell')$, where $V'=\{l,r,d\}$,
$A'=\{\textrm{r}1',\textrm{r}2'\}$, and the functions are given by the tables
\[
\begin{tabular}{c | c c}
    & 1 & 2 \\ \hline
$f'(-)$ &  l & d\\ 
$g'(-)$ & r & r\\
\end{tabular}
\hspace{2cm}
\begin{tabular}{c | c c}
    & r1' & r2' \\ \hline
$s(-)$ &  l & r\\ 
$t(-)$ & r & d\\
$\ell(-)$ & $5\Omega$ & $8\Omega$
\end{tabular}
\]

To compose these, we first take the pushout of $V \From{g} \ord{2} \To{f'} V'$.
This gives the a new apex $V''
=\{\textrm{ul},\textrm{dl},\textrm{dr},\textrm{m},\textrm{r}\}$ with five
elements, and composite cospan $\ord{1} \To{h} V'' \From{k} \ord{2}$ given by
$h(1) = \textrm{ul}$, $k(1) = \textrm{r}$ and $k(2) =\textrm{m}$. The new
circuit is given by $(V,''A+A',s,''t,''\ell'')$ where the functions are given by
\[
\begin{tabular}{c | c c c c c c c}
    & r1 & r2 & r3 & c1 & i1 & r1' & r2' \\ \hline
$s''(-)$ & dl & ul & m & ul & dl & m & r \\ 
$t''(-)$ & ul & m & dr & m & dr & r & m \\
$\ell''(-)$ & $1\Omega$ & $2\Omega$ & $1\Omega$ & $3F$ & $1H$ & $5\Omega$ &
$8\Omega$ 
\end{tabular}
\]
This is exactly what is depicted in \cref{eqn.circuitcomposed}.
}

\sol{exc.buildcircuit}{
Write $x$ for the open circuit in \eqref{eqn.circuit}. Also define cospans $\eta\colon 0\to 2$ and $\eta\colon 2\to 0$ as follows: 
\[
\eta \coloneqq
\quad
\begin{aligned}
%
\end{aligned}
\quad
=\colon\epsilon
\]
where each of these are decorated by the empty circuit
$(\ord{1},\varnothing,!,!,!) \in \elec(\ord{1})$.%
\tablefootnote{As usual $!$ denotes the unique function, in this case from the empty set to the relevant codomain.}

Compute the composite $\eta \cp x \cp \epsilon$ in $\cospan\elec$. This is a
morphism $\ord{0} \to \ord{0}$; we call such things \emph{closed circuits}.
}{
Composing $\eta$ and $x$ we have
\begin{align*}
\eta \cp x 
&=
\begin{aligned}
%
\end{aligned}
\end{align*}
and composing the result of $\epsilon$ gives
\begin{align*}
\eta \cp x \cp \epsilon
&=
\begin{aligned}
%
\end{aligned}
\end{align*}
}

\sol{exc.wd_drawing_practice}{
\begin{enumerate}
	\item Consider the following cospan $f\in\oprdcospan(2,2; 2)$:
	\[
	%
	\]
	Draw it as a wiring diagram with two inner circles, each with two ports, and one outer circle with two ports.
	\item Draw the wiring diagram corresponding to the following cospan $g\in\oprdcospan(2,2,2;0)$:
	\[
	%
	\]	
	\item Compute the cospan $g\circ_1 f$. What is its arity?
	\item Draw the cospan $g\circ_1 f$. Do you see it as substitution?
\end{enumerate}
}
{
\begin{enumerate}
	\item The cospan shown left corresponds to the wiring diagram shown right:
	\[
	%
	\]
  It has two inner circles, each with two ports. One port of the first is wired to a port of the second. One port of the first is wired to the outside circle, and one port of the second is wired to the outside circle. This is exactly what the cospan says to do.
	\item The cospan shown left corresponds to the wiring diagram shown right:
		\[
	%
	\]	
	\item The composite $g\circ_1 f$ has arity $(2,2,2,2;0)$; there is a depiction on the left:
			\[
	%
	\]
	\item The associated wiring diagram is shown on the right above. One can see that one diagram has been substituted in to a circle of the other.
\end{enumerate}
}

\finishSolutionChapter
\section[Solutions for Chapter 7]{Solutions for \cref{chap.temporal_topos}.}

\sol{exc.pullback_pasting}{
Prove \cref{prop.pullback_pasting} using the definition of limit from \cref{subsec.adjoints_lims_colims}.
}{
In the commutative diagram below, suppose the $(B,C,B',C')$ square is a pullback:
\[
\begin{tikzcd}[ampersand replacement=\&]
	A\ar[r, "f"]\ar[d, "h_1"]\&B\ar[r, "g"]\ar[d, "h_2"]\&C\ar[d, "h_3"]\\
	A'\ar[r, "f'"']\&B'\ar[r, "g'"']\&C'\ar[ul, phantom, very near end, "\lrcorner"]
\end{tikzcd}
\]
We need to show that the $(A,B,A',B')$ square is a pullback iff the $(A,C,A',C')$ rectangle is a pullback.

Suppose first that $(A,B,A',B')$ is a pullback, and take any $(X,p,q)$ as in the following diagram:
\[
\begin{tikzcd}[ampersand replacement=\&]
	X\ar[rrrd, bend left=15pt, "p"]\ar[rdd, bend right=15pt, "q"']\&[-15pt]\\[-5pt]
	\&A\ar[r, "f"]\ar[d, "h_1"]\&B\ar[r, pos=.4, "g"]\ar[d, "h_2"']\&C\ar[d, "h_3"']\\
	\&A'\ar[r, "f'"']\&B'\ar[r, "g'"']\&C'\ar[ul, phantom, very near end, "\lrcorner"]
\end{tikzcd}
\]
where $q\cp f'\cp g'=p\cp h_3$. Then by the universal property of the $(B,C,B',C')$ pullback, we get a unique dotted arrow $r$ making the left-hand diagram below commute:
\[
\begin{tikzcd}[ampersand replacement=\&]
	X\ar[rrrd, bend left=15pt, "p"]\ar[rrdd, bend right=15pt, "q\cp f'"']\ar[rrd, dashed, "r"']\&[-15pt]\\[-5pt]
	\&\&B\ar[r, pos=.4, "g"]\ar[d, "h_2"']\&C\ar[d, "h_3"']\\
	\&\&B'\ar[r, "g'"']\&C'\ar[ul, phantom, very near end, "\lrcorner"]
\end{tikzcd}
\hspace{1in}
\begin{tikzcd}[ampersand replacement=\&]
	X\ar[rrd, bend left=15pt, "r"]\ar[rdd, bend right=15pt, "q"']\ar[rd, dashed, "r'"]\&[-15pt]\\[-5pt]
	\&A\ar[r, "f"]\ar[d, "h_1"]\&B\ar[r, pos=.4, "g"]\ar[d, "h_2"']\&C\ar[d, "h_3"']\\
	\&A'\ar[r, "f'"']\&B'\ar[r, "g'"']\&C'\ar[ul, phantom, very near end, "\lrcorner"]
\end{tikzcd}
\]
In other words $r\cp h_2=g\cp f'$ and $r\cp g=p$. Then by the universal property of the $(A,B,A',B')$ pullback, we get a unique dotted arrow $r'\colon X\to A$ making the right-hand diagram commute, i.e.\ $r'\cp f=r$ and $r'\cp h_1=q$. This gives the existence of an $r$ with the required property, $r'\cp f=r$ and $r'\cp f\cp g=r\cp g=p$. To see uniqueness, suppose given another morphisms $r_0$ such that $r_0\cp f\cp g=p$ and $r_0\cp h_1=q$:
\[
\begin{tikzcd}[ampersand replacement=\&]
	X\ar[dr, "r_0"]\ar[rrrd, bend left=15pt, "p"]\ar[rdd, bend right=20pt, "q"']\&[-15pt]\\[-5pt]
	\&A\ar[r, "f"]\ar[d, "h_1"]\&B\ar[r, pos=.4, "g"]\ar[d, "h_2"']\&C\ar[d, "h_3"']\\
	\&A'\ar[r, "f'"']\&B'\ar[r, "g'"']\&C'\ar[ul, phantom, very near end, "\lrcorner"]
\end{tikzcd}
\]
Then by the uniqueness of $r$, we must have $r_0\cp f=r$, and then by the uniqueness of $r'$, we must have $r_0=r'$. This proves the first result.\\

The second is similar. Suppose that $(A,C,A',C')$ and $(B,C,B',C')$ are pullbacks and suppose given a commutative diagram of the following form:
\[
\begin{tikzcd}[ampersand replacement=\&]
	X\ar[rrd, bend left=15pt, "r"]\ar[rdd, bend right=15pt, "q"']\&[-15pt]\\[-5pt]
	\&A\ar[r, "f"]\ar[d, "h_1"]\&B\ar[r, pos=.4, "g"]\ar[d, "h_2"']\&C\ar[d, "h_3"']\\
	\&A'\ar[r, "f'"']\&B'\ar[r, "g'"']\&C'\ar[ul, phantom, very near end, "\lrcorner"]
\end{tikzcd}
\]
i.e.\ where $r\cp h_2=q\cp f'$. Then letting $p\coloneqq r\cp g$, we have
\[p\cp h_3=r\cp g\cp h_3=r\cp h_2\cp g'=q\cp f'\cp g'\]
so by the universal property of the $(A,C,A',C')$ pullback, there is a unique morphism $r'\colon X\to A$ such that $r'\cp f\cp g=p$ and $r_0\cp h_1=q$, as shown:
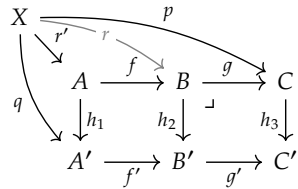
\[
\begin{tikzcd}[ampersand replacement=\&]
	X\ar[dr, "r'"]\ar[rrd, gray, bend left=15pt, "r" description]\ar[rrrd, bend left=15pt, "p"]\ar[rdd, bend right=20pt, "q"']\&[-15pt]\\[-5pt]
	\&A\ar[r, "f"]\ar[d, "h_1"]\&B\ar[r, pos=.4, "g"]\ar[d, "h_2"']\&C\ar[d, "h_3"']\\
	\&A'\ar[r, "f'"']\&B'\ar[r, "g'"']\&C'\ar[ul, phantom, very near end, "\lrcorner"]
\end{tikzcd}
\]
But now let $r_0\coloneqq r'\cp f$. It satisfies $r_0\cp g=p$ and $r_0\cp h_2=q\cp f'$, and $r$ satisfies the same equations: $r\cp g=p$ and $r\cp h_2=q\cp f'$. Hence by the universal property of the $(B,C,B',C')$ pullback $r_0=r'$. It follows that $r'$ is a pullback of the $(A,B,A',B')$ square, as desired.
}

\sol{exc.mono_inj}{
Show that in $\smset$, monomorphisms are just injections:
\begin{enumerate}
	\item Show that if $f$ is a monomorphism then it is injective.
	\item Show that if $f\colon A\to B$ is injective then it is a monomorphism.
\end{enumerate}
}{
A function $f\colon A\to B$ is injective iff for all $a_1,a_2\in A$, if $f(a_1)=f(a_2)$ then $a_1=a_2$. It is a monomorphism iff for all sets $X$ and functions $g_1,g_2\colon X\to A$, if $g_1\cp f=g_2\cp f$ then $g_1=g_2$. Indeed, this comes directly from the universal property of the pullback from \cref{def.mono_epi},
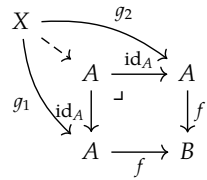
\[
\begin{tikzcd}[ampersand replacement=\&]
	X\ar[ddr, bend right=20pt, "g_1"']\ar[drr, bend left=20pt, "g_2"]\ar[dr, dashed]
	\&[-10pt]\\[-10pt]
	\&A\ar[r, "\id_A"]\ar[d, "\id_A"']\&A\ar[d, "f"]\\
	\&A\ar[r, "f"']\&B\ar[ul, phantom, very near end, "\lrcorner"]
\end{tikzcd}
\]
because the dashed arrow is forced to equal both $g_1$ and $g_2$, thus forcing $g_1=g_2$.
\begin{enumerate}
	\item Suppose $f$ is a monomorphism, let $a_1,a_2\in A$ be elements, and suppose $f(a_1)=f(a_2)$. Let $X=\{*\}$ be a one element set, and let $g_1,g_2\colon X\to A$ be given by $g_1(*)\coloneqq a_1$ and $g_2(*)\coloneqq a_2$. Then $g_1\cp f=g_2\cp f$, so $g_1=g_2$, so $a_1=a_2$.
	\item Suppose that $f$ is an injection, let $X$ be any set, and let $g_1,g_2\colon X\to A$ be such that $g_1\cp f=g_2\cp f$. We will have $g_1=g_2$ if we can show that $g_1(x)=g_2(x)$ for every $x\in X$. So take any $x\in X$; since $f(g_1(x))=f(g_2(x))$ and $f$ is injective, we have $g_1(x)=g_2(x)$ as desired.
\end{enumerate}
}

\sol{exc.pullback_iso_iso}{
\begin{enumerate}
	\item	Show that the pullback of an isomorphism along any morphism is an isomorphism. That is, suppose that $i\colon B'\to B$ is an isomorphism and $f\colon A\to B$ is any morphism. Show that $i'$ is an isomorphism, in the following diagram:
	\[
	\begin{tikzcd}[ampersand replacement=\&]
		A'\ar[r, "f'"]\ar[d, pos=.6, "i'"', "\cong"]\&B'\ar[d, pos=.6, "i", "\cong"']\\
		A\ar[r, "f"']\&B\ar[ul, phantom, very near end, "\lrcorner"]
	\end{tikzcd}
	\]	
	\item Show that for any map $f\colon A\to B$, the square shown is a pullback:
	\[
	\begin{tikzcd}
		A\ar[r, "f"]\ar[d, equal]&
		A\ar[d, equal]\\
		B\ar[r, "f"']&
		B\ar[ul, phantom, very near end, "\lrcorner"]
	\end{tikzcd}
	\]
\end{enumerate}
}{
\begin{enumerate}
	\item Suppose we have a pullback as shown, where $i$ is an isomorphism:
  	
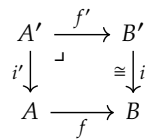
\[
  	\begin{tikzcd}[ampersand replacement=\&]
  		A'\ar[r, "f'"]\ar[d, "i'"']\&B'\ar[d, "i", "\cong"']\\
  		A\ar[r, "f"']\&B\ar[ul, phantom, very near end, "\lrcorner"]
  	\end{tikzcd}
  	\]
  	Let $j\coloneqq i\inv$ be the inverse of $i$, and consider $g\coloneqq (f\cp j)\colon A\to B'$. Then $g\cp i=f$, so by the existence part of the universal property, there is a map $j'\colon A\to A'$ such that $j'\cp i'=\id_A$ and $j'\cp f'=f\cp j$. We will be done if we can show $i'\cp j'=\id_{A'}$. One checks that $(i'\cp j')\cp i'=i'$ and that $(i'\cp j')\cp f'=i'\cp f\cp j=f'\cp i\cp j=f'$. But $\id_{A'}$ also satisfies those properties: $\id_{A'}\cp i'=i'$ and $\id_{A'}\cp f'=f'$, so by the uniqueness part of the universal property, $(i'\cp j')=\id_{A'}$.
	\item We need to show that the following diagram is a pullback:
		
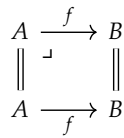
\[
	\begin{tikzcd}[ampersand replacement=\&]
		A\ar[r, "f"]\ar[d, equal]\&
		B\ar[d, equal]\\
		A\ar[r, "f"']\&
		B\ar[ul, phantom, very near end, "\lrcorner"]
	\end{tikzcd}
	\]
	So take any object $X$ and morphisms $g\colon X\to A$ and $h\colon X\to B$ such that $g\cp f=h\cp\id_B$. We need to show there is a unique morphism $r\colon X\to A$ such that $r\cp\id_A=g$ and $r\cp f=h$. That's easy: the first requirement forces $r=g$ and the second requirement is then fulfilled.
\end{enumerate}
}

\sol{exc.monos_pb_pasting}{
Let $\cat{C}$ be a category and suppose the following diagram is a pullback in $\cat{C}$:
\[
\begin{tikzcd}[ampersand replacement=\&]
	A'\ar[r]\ar[d, "f'"']\&A\ar[d, tail, "f"]\\
	B'\ar[r]\&B\ar[ul, phantom, very near end, "\lrcorner"]
\end{tikzcd}
\]
Use \cref{prop.pullback_pasting,exc.pullback_iso_iso} to show that if $f$ is a monomorphism, then so is $f'$.
}{
Consider the diagram shown left, in which all three squares are pullbacks:
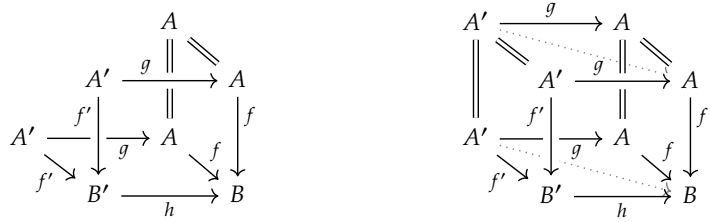
\[
\begin{tikzcd}[ampersand replacement=\&, sep=small]
	\&\&
	A\ar[rd, equal]\ar[dd, equal]\\
	\&A'\ar[rr, crossing over, pos=.25, "g"]\&\&
	A\ar[dd, pos=.25, "f"]\\
	A'\ar[rr, pos=.75, "g"']\ar[dr, "f'"']\&\&
	A\ar[rd, "f"]\\
	\&B'\ar[from=uu, crossing over, pos=.25, "f'"']\ar[rr, "h"']\&\&
	B
\end{tikzcd}
\hspace{1in}
\begin{tikzcd}[ampersand replacement=\&, sep=small]
	A'\ar[rr, "g"]\ar[dd, equal]\ar[rd, equal]\ar[drrr, gray, dotted]\&\&
	A\ar[rd, equal]\ar[dd, equal]\\
	\&A'\ar[rr, crossing over, pos=.25, "g"]\&\&
	A\ar[dd, pos=.25, "f"]\\
	A'\ar[rr, pos=.75, "g"']\ar[dr, "f'"']\ar[drrr, gray, dotted]\&\&
	A\ar[rd, "f"]\\
	\&B'\ar[from=uu, crossing over, pos=.25, "f'"']\ar[rr, "h"']\&\&
	B
\end{tikzcd}
\]
The front and bottom squares are the same---the assumed pullback---and the right-hand square is a pullback because $f$ is assumed monic. We can complete it to the commutative diagram shown right, where the back square and top square are pullbacks by \cref{exc.pullback_iso_iso}. Our goal is to show that the left-hand square is a pullback.

To do this, we use two applications of the pasting lemma, \cref{exc.pullback_pasting}. 
Since the right-hand face is a pullback and the back face is a pullback, the diagonal rectangle (lightly drawn) is also a pullback. Since the front face is a pullback, the left-hand face is also a pullback.
}

\sol{exc.epi_mono_practice}{
Factor the following function $f\colon \ord{3}\to \ord{3}$ as an epimorphism followed by a monomorphism.
\[
  \begin{tikzpicture}
		\foreach \x in {0,...,2} 
			{\draw (0,.4-.4*\x) node (X0\x) {$\bullet$};}
		\node[draw, ellipse, inner sep=0pt, fit=(X00) (X02)] (X0) {};
		\foreach \x in {0,...,2} 
			{\draw (2,.4-.4*\x) node (Y0\x) {$\bullet$};}
		\node[draw, ellipse, inner sep=0pt, fit=(Y00) (Y02)] (Y0) {};
		\draw[mapsto] (X00.center) -- (Y01.center);
		\draw[mapsto] (X01.center) -- (Y01.center);
		\draw[mapsto] (X02.center) -- (Y02.center);
  \end{tikzpicture}
\]
}
{
The following is an epi-mono factorization of $f$:
\[
  \begin{tikzpicture}
		\foreach \x in {0,...,2} 
			{\draw (0,.4-.4*\x) node (X0\x) {$\bullet$};}
		\node[draw, ellipse, inner sep=0pt, fit=(X00) (X02)] (X0) {};
		\foreach \x in {1,...,2} 
			{\draw (2,.4-.4*\x) node (Y0\x) {$\bullet$};}
		\node[draw, ellipse, inner sep=0pt, fit=(Y01) (Y02)] (Y0) {};
		\foreach \x in {0,...,2} 
			{\draw (4,.4-.4*\x) node (Z0\x) {$\bullet$};}
		\node[draw, ellipse, inner sep=0pt, fit=(Z00) (Z02)] (Y0) {};
		\begin{scope}[short=-2pt, mapsto]
  		\draw (X00) -- (Y01);
  		\draw (X01) -- (Y01);
  		\draw (X02) -- (Y02);
			\draw (Y01) -- (Z01);
			\draw (Y02) -- (Z02);
		\end{scope}
  \end{tikzpicture}
\]
}

\sol{ex.ccposet_quantale}{
Let $\cat{V}=(V,\leq,I,\otimes,\multimap)$ be a (unital, commutative) quantale---see
\cref{def.quantale}---and suppose it satisfies the following for all $v,w,x\in V$:
\begin{itemize}
	\item $v\leq I$,
	\item $v\otimes w\leq v$ and $v\otimes w\leq w$
	\item $if $x\leq v$ and $x\leq w$ then $x\leq v\otimes w$.
\end{enumerate}
\begin{enumerate}
	\item Show that $\cat{V}$ is a cartesian closed category, in fact a cartesian closed preorder.
	\item Can every cartesian closed preorder be obtained in this way?
\qedhere
\end{enumerate}
}{
\begin{enumerate}
  \item If $\cat{V}$ is a quantale with the stated properties, then
  \begin{itemize}
  	\item $I$ serves as a top element: $v\leq I$ for all $v\in V$. 
		\item $v\otimes w$ serves as a meet operation, i.e.\ it satisfies the same universal property as $\wedge$, namely $v\otimes w$ is a greatest lower bound for $v$ and $w$.
	\end{itemize}
	Now the $\multimap$ operation satisfies the same universal property as exponentiation (hom-object) does, namely $v\leq (w\multimap x)$ iff $v\otimes W\leq x$. So $\cat{V}$ is a cartesian closed category, and of course it is a preorder.
	\item Not every cartesian closed preorder comes from a quantale with the stated properties, because quantales have all joins and cartesian closed preorders need not. Finding a counterexample---a cartesian closed preorder that is missing some joins---takes some ingenuity, but it can be done. Here's one we came up with:
	
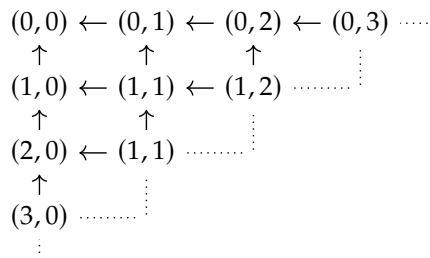
\[
	\begin{tikzcd}[ampersand replacement=\&, sep=small, font=\small]
		(0,0)\ar[from=r]\ar[from=d]\&
		(0,1)\ar[from=r]\ar[from=d]\&
		(0,2)\ar[from=r]\ar[from=d]\&
		(0,3)\ar[from=r, -, dotted]\ar[from=d, -, dotted]
		\&{}\\
		(1,0)\ar[from=r]\ar[from=d]\&
		(1,1)\ar[from=r]\ar[from=d]\&
		(1,2)\ar[from=r, -, dotted]\ar[from=d, -, dotted]
		\&{}\\
		(2,0)\ar[from=r]\ar[from=d]\&
		(1,1)\ar[from=r, -, dotted]\ar[from=d, -, dotted]
		\&{}\\
		(3,0)\ar[from=r, -, dotted]\ar[from=d, -, dotted]
		\&{}\\
		{}
	\end{tikzcd}
	\]
	This is the product preorder $\nn\op\times\nn\op$: its objects are pairs $(a,b)\in\nn\times\nn$ with $(a,b)\leq(a',b')$ iff, in the usual ordering on $\nn$ we have $a'\leq a$ and $b'\leq b$. But you can just look at the diagram.\\
	
	It has a top element, $(0,0)$, and it has binary meets, $(a,b)\wedge(a',b')=(\max(a,a'),\max(b,b'))$. But it has no bottom element, so it has no empty join. Thus we will be done if we can show that for each $x,y$, the hom-object $x\multimap y$ exists. The formula for it is $x\multimap y=\bigvee\{w\mid w\wedge x\leq y\}$, i.e.\ we need these particular joins to exist. Since $y\wedge x\leq y$, we have $y\leq x\multimap y$. So we can replace the formula with $x\multimap y=\bigvee\{w\mid y\leq w\text{ and }w\wedge x\leq y\}$. But the set of elements in $\nn\op\times\nn\op$ that are bigger than $y$ is finite and nonempty.%
	\footnote{If $y=(a,b)$ then there are exactly $(a+1)*(b+1)$ elements $y'$ for which $y\leq y'$.}
	So this is a finite nonempty join, and $\nn\op\times\nn\op$ has all finite nonempty joins: they are given by $\inf$.
\end{enumerate}
}

\sol{exc.characteristic_practice}{
  Let $X=\NN=\{0,1,2,\ldots\}$ and $Y=\ZZ=\{\ldots,-1,0,1,2,\ldots\}$; we have
  $X\ss Y$, so consider it as a monomorphism $m\colon X\to Y$. It has a
  characteristic function $\corners{m}\colon Y\to\BB$, as in
  \cref{def.subobject_classifier}.
  \begin{enumerate}
    \item What is $\corners{m}(-5)\in\BB$?
    \item What is $\corners{m}(0)\in\BB$?  
	\end{enumerate}
}{
Let $m\colon \ZZ\to\BB$ be the characteristic function of the inclusion $\nn\ss\zz$.

\begin{enumerate*}[itemjoin=\hspace{1in}]
	\item $\corners{m}(-5)=\false$.
	\item $\corners{m}(0)=\true.$
\end{enumerate*}
}

\sol{exc.simple_char_funs}{
  \begin{enumerate}
    \item Consider the identity function $\id_\NN\colon \NN\to \NN$. It is an
      injection, so it has a characteristic function $\corners{\id_\NN}\colon
      \NN\to\BB$.  Give a concrete description of $\corners{\id_\NN}$, i.e.\ its
      exact value for each natural number $n\in\NN$.
    \item Consider the unique function $!_\NN\colon\varnothing\to\NN$ from the
      empty set. Give a concrete description of $\corners{!_\NN}$.  
\end{enumerate}
}{
\begin{enumerate}
	\item The characteristic function $\corners{\id_\NN}\colon\NN\to\BB$ sends each $n\in\nn$ to $\true$.
	\item Let $!_\NN\colon\varnothing\to\nn$ be the inclusion of the empty set. The characteristic function $\corners{!_\NN}\colon\NN\to\BB$ sends each $n\in\nn$ to $\false$.
\end{enumerate}
}

\sol{exc.neg_char}{
Every boolean has a negation, $\neg\false=\true$ and $\neg\true=\false$. The function $\neg\colon\BB\to\BB$ is the characteristic function of some thing, (*?*).
\begin{enumerate}
	\item What sort of thing should (*?*) be? For example, should $\neg$ be the characteristic function of an object? A topos? A morphism? A subobject? A pullback diagram?
	\item Now that you know the sort of thing (*?*) is, which thing of that sort is it?
\end{enumerate}
}{
\begin{enumerate}
	\item The sort of thing (*?*) we're looking for is a subobject of $\BB$, say $A\ss\BB$. This would have a characteristic function, and we're trying to find the $A$ for which the characteristic function is $\neg\colon\BB\to\BB$.
	\item The question now asks ``what is $A$?'' The answer is $\{\false\}\ss\BB$.
\end{enumerate}
}

\sol{exc.implies_char}{
Given two booleans $P,Q$, define $P\imp Q$ to mean $P=(P\wedge Q)$.
\begin{enumerate}
	\item Write down the truth table for the statement $P=(P\wedge Q)$:
	\[
	\begin{array}{cc||c|c}
		P&Q&P\wedge Q&P=(P\wedge Q)\\
		\true&\true&\?&\?\\
		\true&\false&\?&\?\\
		\false&\true&\?&\?\\
		\false&\false&\?&\?\\
	\end{array}
	\]
	\item If you already have an idea what $P\imp Q$ should mean, does it agree with the last column of table above?
	\item What is the characteristic function $m\colon \BB\times\BB\to\BB$ for $P\imp Q$?
	\item What subobject does $m$ classify?
\end{enumerate}
}{
\begin{enumerate}
	\item Here is the truth table for $P=(P\wedge Q)$:
	\begin{equation}\label{eqn.truth_table_imp}
	\begin{array}{cc||c|c}
		P&Q&P\wedge Q&P=(P\wedge Q)\\
		\true&\true&\true&\true\\
		\true&\false&\false&\false\\
		\false&\true&\false&\true\\
		\false&\false&\false&\true\\
	\end{array}
	\end{equation}
	\item Yes!
	\item The characteristic function for $P\imp Q$ is the function $\corners{\imp}\colon\BB\times\BB\to\BB$ given by the first, second, and fourth column of \cref{eqn.truth_table_imp}.
	\item It classifies the subset $\{(\true,\true),(\false,\true),(\false,\false)\ss\BB\times\BB$.
\end{enumerate}
}

\sol{exc.even_prime_10}{
Consider the sets $E\coloneqq\{n\in\NN\mid n\text{ is even}\}$, $P\coloneqq\{n\in\NN\mid n\text{ is prime}\}$, and $T\coloneqq\{n\in\NN\mid n\geq 10\}$. Each is a subset of $\NN$, so defines a function $\NN\to\BB$.
\begin{enumerate}
	\item What is $\corners{E}(17)$?
	\item What is $\corners{P}(17)$?
	\item What is $\corners{T}(17)$?
	\item Name the smallest three elements in the set classified by $(\corners{E}\wedge\corners{P})\vee\corners{T}$.
\end{enumerate}
}{
Say that $\corners{E},\corners{P},\corners{T}\colon\nn\to\bb$ classify respectively the subsets $E\coloneqq\{n\in\NN\mid n\text{ is even}\}$, $P\coloneqq\{n\in\NN\mid n\text{ is prime}\}$, and $T\coloneqq\{n\in\NN\mid n\geq 10\}$ of $\nn$.
\begin{enumerate}
	\item $\corners{E}(17)=\false$ because $17$ is not even.
	\item $\corners{P}(17)=\true$ because $17$ is prime.
	\item $\corners{T}(17)=\true$ because $17\geq 10$.
	\item The set classified by $(\corners{E}\wedge\corners{P})\vee\corners{T}$ is that of all natural numbers that are either above 10 or an even prime. The smallest three elements of this set are $2, 10, 11$.
\end{enumerate}
}

\sol{ex.usual_top_R}{
\begin{enumerate}
	\item What is the 1-dimensional analogue of $\epsilon$-balls as found in \cref{ex.usual_R}? That is, for each $x\in\RR$, define $B(x,\epsilon)$.
	\item When is an arbitrary subset $U\ss\RR$ called open, in analogy with \cref{ex.usual_R}?
	\item Find three open sets $U_1$, $U_2$, and $U$ in $\RR$, such that $(U_i)_{i\in\{1,2\}}$ covers $U$.
	\item Find an open set $U$ and a collection $(U_i)_{i\in I}$ of opens sets where $I$ is infinite, such that $(U_i)_{i\in I}$ covers $U$.
\end{enumerate}
}{
\begin{enumerate}
	\item The 1-dimensional analogue of an $\epsilon$-ball around a point $x\in\RR$ is $B(x,\epsilon)\coloneqq\{x'\in\RR\mid |x-x'|<\epsilon\}$, i.e.\ the set of all points within $\epsilon$ of $x$.
	\item A subset $U\ss\RR$ is open if, for every $x\in U$ there is some $\epsilon>0$ such that $B(x,\epsilon)\ss U$.
	\item Let $U_1\coloneqq\{x\in\RR\mid 0<x<2\}$ and $U_2\coloneqq\{x\in\RR\mid 1<x<3\}$. Then $U\coloneqq U_1\cup U_2=\{x\in\RR\mid 0<x<3\}$.
	\item Let $I=\{1,2,3,4,\ldots\}$ and for each $i\in I$ let $U_i\coloneqq\{x\in\RR\mid \frac{1}{i}<x<1\}$, so we have $U_1\ss U_2\ss U_3\ss\cdots$. Their union is $U\coloneqq\bigcup_{i\in I}U_i=\{x\in\RR\mid 0<x<1\}$.
\end{enumerate}
}

\sol{exc.course_fine}{
\begin{enumerate}
	\item Verify that for any set $X$, what we called $\Op_{\mathrm{crse}}$ in \cref{ex.coarse_fine} really is a topology, i.e.\ satisfies the conditions of \cref{def.topological_space}.
	\item Verify also that $\Op_{\mathrm{fine}}$ really is a topology.
	\item Show that if $(X,\powset(X))$ is discrete and $(Y,\Op_Y)$ is any topological space, then every function $X\to Y$ is continuous.
\end{enumerate}
}{
\begin{enumerate}
	\item The coarse topology on $X$ is the one whose only open sets are $X\ss X$ and $\varnothing\ss X$. This is a topology because it contains the top and bottom subsets, it is closed under finite intersection (the intersection $A\cap B$ is $\varnothing$ iff one or the other is $\varnothing$), and it is closed under arbitrary union (the union $\bigcup_{i\in I}A_i$ is $X$ iff $A_i=X$ for some $i\in I$).
	\item The fine topology on $X$ is the one where every subset $A\ss X$ is considered open. All the conditions on a topology say ``if such-and-such then such-and-such is open,'' but these are all satisfied because everything is open!
	\item If $(X,\powset(X))$ is discrete, $(Y,\Op_Y)$ is any topological space, and $f\colon X\to Y$ is any function then it is continuous. Indeed, this just means that for any open set $U\ss Y$ the preimage $f\inv(U)\ss X$ is open, and everything in $X$ is open.
\end{enumerate}
}

\sol{exc.opens_Sierp}{
Recall the Sierpinski space, say $(X,\Op_1)$ from \cref{ex.Sierpinski}.
\begin{enumerate}
	\item Write down the Hasse diagram for its preorder of opens.
	\item Write down all the coverings.
\end{enumerate}
}{
\begin{enumerate}
	\item The Hasse diagram for the Sierpinsky topology is \fbox{$\varnothing\to\{1\}\to\{1,2\}$}.
	\item A set $(U_i)_{i\in I}$ covers $U$ iff either
	\begin{itemize}
		\item $I=\varnothing$ and $U=\varnothing$; or
		\item $U_i=U$ for some $i\in I$.
	\end{itemize}
	In other words, the only way that some collection of these sets could cover another set $U$ is if that collection contains $U$ or if $U$ is empty and the collection is also empty.
\end{enumerate}
}

\sol{exc.subspace_topology}{
  Given any topological space $(X,\Op)$, any subset $Y\subseteq X$ can be given the
  \emph{subspace topology}, call it $\Op_{?\cap Y}$. This topology defines any $A \subseteq Y$ to be open, $A\in\Op_{?\cap Y}$,
if there is an open set $B\in\Op$ such that $A = B \cap Y$.
\begin{enumerate}
	\item Find a $B\in\Op$ that shows that the whole set $Y$ is open, i.e.\ $Y\in\Op_{?\cap Y}$.
	\item Show that $\Op_{?\cap Y}$ is a topology in the sense of \cref{def.topological_space}.%
	\footnote{Hint 1: for any set $I$, collection of sets $(U_i)_{i\in I}$ with $U_i\ss X$, and set $V\ss X$, one has $\left(\bigcup_{i\in I}U_i\right)\cap V=\bigcup_{i\in I}(U_i\cap V)$. Hint 2: for any $U, V, W\ss X$, one has $(U\cap W)\cap (V\cap W)=(U\cap V)\cap W$.}
	\item Show that the inclusion function $Y \hookrightarrow X$ is a
	  continuous function.
\end{enumerate}
}{
Let $(X,\Op)$ be a topological space, suppose that $Y\ss X$ is a subset, and consider the subspace topology $\Op_{?\cap Y}$.
\begin{enumerate}
	\item We want to show that $Y\in\Op_{?\cap Y}$. We need to find $B\in\Op$ such that $Y=B\cap Y$; this is easy, you could take $B=Y$ or $B=X$, or anything in between.
	\item We still need to show that $\Op_{?\cap Y}$ contains $\varnothing$ and is closed under finite intersection and arbitrary union. $\varnothing=\varnothing\cap Y$, so according to the formula, $\varnothing\in\Op_{?\cap Y}$. Suppose that $A_1,A_2\in\Op_{?\cap Y}$. Then there exist $B_1,B_2\in\Op$ with $A_1=B_1\cap Y$ and $A_2=B_2\cap Y$. But then $A_1\cap A_2=(B_1\cap Y)\cap (B_2\cap Y)=(B_1\cap B_2)\cap Y$, so it is in $\Op_{?\cap Y}$ since $B_1\cap B_2\in\Op$. The same idea works for arbitrary unions: given a set $I$ and $A_i$ for each $i\in I$, we have $A_i=B_i\cap Y$ for some $B_i\in\Op$, and
	\[\bigcup_{i\in I}A_i=\bigcup_{i\in I}(B_i\cap Y)=\left(\bigcup_{i\in i}B_i\right)\cap Y\in\Op_{?\cap Y}.\]
\end{enumerate}
}

\sol{exc.top_sp_quantale}{
In \cref{subsec.Lawv_metric_spaces,subsec.preorders_Bool_enriched} we discussed how $\Bool$-categories are preorders and $\Cost$-categories are Lawvere metric spaces, and in \cref{subsec.variations_quantale} we imagined interpretations of $\cat{V}$-categories for other quantales $\cat{V}$.

If $(X,\Op)$ is a topological space and $\cat{V}$ the corresponding quantale as in \cref{rem.top_sp_quantale}, how might we imagine a $\cat{V}$-category? 
}{
Let's imagine a $\cat{V}$-category $\cat{C}$, where $\cat{V}$ is the quantale corresponding to the open sets of a topological space $(X,\Op)$. Its Hasse diagram would be a set of dots and some arrows between them, each labeled by an open set $U\ss\Op$. It might look something like this:
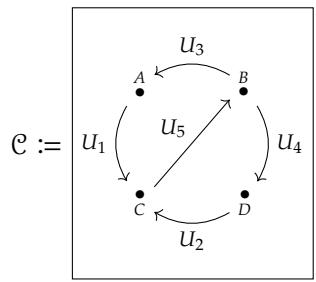
\[
\begin{tikzpicture}[font=\scriptsize, x=1cm]
	\node (a) {$\LMO{A}$};
	\node[right=1 of a] (b) {$\LMO{B}$};
	\node[below=1 of a] (c) {$\LMO[under]{C}$};
	\node[right=1 of c] (d) {$\LMO[under]{D}$};
	\draw[->] (c) to node[above left=-1pt and -1pt] {$U_5$} (b);
	\draw[bend right,->] (a) to node[left] {$U_1$} (c);
	\draw[bend left,->] (d) to node[below] {$U_2$} (c);
	\draw[bend right,->] (b) to node[above] {$U_3$} (a);
	\draw[bend left,->] (b) to node[right] {$U_4$} (d);
	\node[draw, inner sep=20pt, fit=(a) (b) (c) (d)] (X) {};
	\node[left=0 of X, font=\normalsize] {$\cat{C}\coloneqq$};
\end{tikzpicture}
\] 
Recall from \cref{sec.enrichment} that the `distance' between two points is computed by taking the join, over all paths between them, of the monoidal product of distances along that path. For example, $\cat{C}(B,C)=(U_3\wedge U_1)\vee(U_4\wedge U_2)$, because $\wedge$ is the monoidal product in $\cat{V}$.\\

In general, we can thus imagine the open set $\cat{C}(a,b)$ as a kind of `size restriction' for getting from $a$ to $b$, like bridges that your truck needs to pass under. The size restriction for getting from $a$ to itself is $X$: no restriction. In general, to go on any given route (path) from $a$ to $b$, you have to fit under every bridge in the path, so we take their meet. But we can go along any path, so we take the join over all paths.
}

\sol{exc.fiber_practice}{
Consider the function $f\colon X\to Y$ shown in \cref{eqn.sections_of_function}.
\begin{enumerate}
	\item What is the fiber of $f$ over $a$?
	\item What is the fiber of $f$ over $c$?
	\item What is the fiber of $f$ over $d$?
	\item Gave an example of a function $f'\colon X\to Y$ for which every fiber has either one or two elements.
\end{enumerate}
}{
\begin{equation}\label{eqn.sections_of_function_redraw}

\end{equation}
\begin{enumerate}
	\item The fiber of $f$ over $a$ is $\{a_1,a_2\}$.
	\item The fiber of $f$ over $c$ is $\{c_1\}$.
	\item The fiber of $f$ over $d$ is $\varnothing$.
	\item A function $f'\colon X\to Y$ for which every fiber has either one or two elements is shown below.
\end{enumerate}
\[
%
\]
}

\sol{exc.presheaf_ex_cont}{
Refer to \cref{eqn.sections_of_function}.
\begin{enumerate}
	\item Let $V_1=\{a,b,c\}$. Draw all the sections over it, i.e.\ all elements of $\Fun{Sec}_f(V_1)$, as we did in \cref{eqn.all_six_sections}.
	\item Let $V_2=\{a,b,c,d\}$. Again draw all the sections, $\Fun{Sec}_f(V_2)$.
	\item Let $V_3=\{a,b,d,e\}$. How many sections (elements of $\Fun{Sec}_f(V_3)$) are there?
\end{enumerate}
}{
Refer to \cref{eqn.sections_of_function_redraw}.
\begin{enumerate}
	\item Here is a drawing of all six sections over $V_1=\{a,b,c\}$:
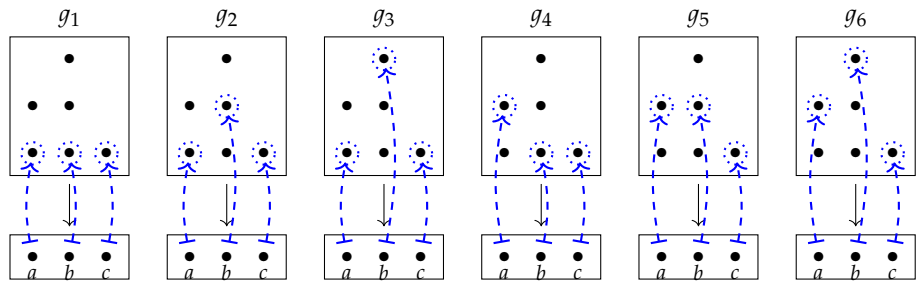
	\[
  	%
\]
	\item When $V_2=\{a,b,c,d\}$, there are no sections: $\Fun{Sec}_f(V_2)=\varnothing$.
	\item When $V_3=\{a,b,d,e\}$, the set $\Fun{Sec}_f(V_3)$) has $2*3*1*2=12$ elements.
\end{enumerate}
}

\sol{exc.section_practice}{
\begin{enumerate}
	\item Write out the sets of sections $\Fun{Sec}_f(\{a,b,c\})$ and $\Fun{Sec}_f(\{a,c\})$.
	\item Draw lines from the first to the second to indicate the restriction map.
\end{enumerate}
}{
$\Fun{Sec}_f(\{a,b,c\})$ and $\Fun{Sec}_f(\{a,c\})$ are drawn as the top row (six-element set) and bottom row (two-element set) below, and the restriction map is also shown:
\[
\begin{tikzcd}[ampersand replacement=\&, column sep=8pt]
	(a_1,b_1,c_1)\ar[dr, |->]\&
	(a_1,b_2,c_1)\ar[d, |->]\&
	(a_1,b_3,c_1)\ar[dl, |->]\&
	(a_2,b_1,c_1)\ar[dr, |->]\&
	(a_2,b_2,c_1)\ar[d, |->]\&
	(a_2,b_3,c_1)\ar[dl, |->]\\
	\&(a_1,c_1)\&\&\&
	(a_2,c_1)
\end{tikzcd}
\]
}

\sol{exc.sections_agree_overlap}{
Again let $U_1=\{a,b\}$ and $U_2=\{b,e\}$, so the overlap is $U_1\cap U_2=\{b\}$.
\begin{enumerate}
	\item Find a section $g_1\in\Fun{Sec}_f(U_1)$ and a section $g_2\in\Fun{Sec}_f(U_2)$ that \emph{do not} agree on the overlap.
	\item For your answer ($g_1,g_2)$ in part 1, can you find a section $g\in\Fun{Sec}_f(U_1\cup U_2)$ such that $\restrict{g}{U_1}=g_1$ and $\restrict{g}{U_2}=g_2$?
	\item Find a section $h_1\in\Fun{Sec}_f(U_1)$ and a section $h_2\in\Fun{Sec}_f(U_2)$ that \emph{do} agree on the overlap, but which are different than our choice in \cref{eqn.one_section}.
	\item Can you find a section $h\in\Fun{Sec}_f(U_1\cup U_2)$ such that $\restrict{h}{U_1}=h_1$ and $\restrict{h}{U_2}=h_2$?
\end{enumerate}
}{
\begin{enumerate}
	\item Let $g_1\coloneqq(a_1,b_1)$ and $g_2\coloneqq(b_2,e_1)$; these do not agree on the overlap.
	\item No, there's no section $g\in\Fun{Sec}_f(U_1\cup U_2)$ for which $\restrict{g}{U_1}=g_1$ and $\restrict{g}{U_2}=g_2$.
\end{enumerate}
\[

\]
}

\sol{exc.whats_the_sheaf}{
If $M$ is a sphere as in \cref{ex.tangent_bundle}, we know from \cref{def.sheaf} that we can consider the category $\Shv(M)$ of sheaves on $M$; in fact, such categories are toposes and these are what we're getting to.

But are the sheaves on $M$ the vector fields? That is, is there a one-to-one correspondence between sheaves on $M$ and vector fields on $M$? If so, why? If not, how are sheaves on $M$ and vector fields on $M$ related?
}
{
No, there is not a one-to-one correspondence between sheaves on $M$ and vector fields on $M$. The relationship between sheaves on $M$ and vector fields on $M$ is that the \emph{set of all} vector fields on $M$ corresponds to \emph{one} sheaf, namely $\Fun{Sec}_\pi$, where $\pi\colon TM\to M$ is the tangent bundle as described in \cref{ex.tangent_bundle}. There are so many sheaves on $M$ that they don't even form a set (it's just a `collection'); again, one member of this gigantic collection is the sheaf $\Fun{Sec}_\pi$ of all possible vector fields on $M$.
}

\sol{exc.sierpinski}{
Consider the Sierpinski space $(\{1,2\},\Op_1)$ from \cref{ex.Sierpinski}.
\begin{enumerate}
	\item What is the category $\Op$ for this space? (You may have already figured this out in \cref{exc.opens_Sierp}; if not, do so now.)
	\item What does a presheaf on $\Op$ consist of?
	\item What is the sheaf condition for $\Op$?
	\item How do we identify a sheaf on $\Op$ with a function?
\end{enumerate}
}{
\begin{enumerate}
	\item The Hasse diagram for the Sierpinsky topology is \fbox{$\varnothing\to\{1\}\to\{1,2\}$}\,.
	\item A presheaf $F$ on $\Op$ consists of any three sets and any two functions $F(\{1,2\})\to F(\{1\})\to F(\varnothing)$ between them.
	\item Recall from \cref{exc.opens_Sierp} that the only non-trivial covering (a covering of $U$ is \emph{non-trivial} if it does not contain $U$) occurs when $U=\varnothing$ in which case the empty family over $U$ is a cover.
	\item As explained in \cref{ex.empty_cover}, $F$ will be a sheaf iff $F(\varnothing)\cong\{1\}$. Thus we the category of sheaves is equivalent to that of just two sets and one function $F(\{1,2\})\to F(\{1\})$.
\end{enumerate}
}

\sol{exc.booleans_as_subspace_1}{
Let $X=\Cat{1}$ be the one point space. We said above that its subobject classifier is the set $\bb$ of booleans, but how does that align with the definition of $\Omega$ given in \cref{eqn.omega_as_opens}?
}
{
The one-point space $X=\{1\}$ has two open sets, $\varnothing$ and $\{1\}$, and every sheaf $S\in\Shv(X)$ assigns $S(\varnothing)=\{()\}$ by the sheaf condition (see \cref{ex.empty_cover}). So the only data in a sheaf $S\in\Shv(X)$ is the set $S(\{1\})$. This is how we get the correspondence between sets and sheaves on the one point space.

According to \cref{eqn.omega_as_opens}, the subobject classifier $\Omega\colon\Op(X)\op\to\smset$ in $\Shv(X)$ should be the functor where $\Omega(\{1\})$ is the set of open sets of $\{1\}$. So we're hoping to see that there is a one-to-one correspondence between the set $\Op(\{1\})$ and the set $\bb=\{\true,\false\}$ of booleans. Indeed there is: there are two open sets of $\{1\}$, as we said, $\varnothing$ and $\{1\}$, and these correspond to $\false$ and $\true$ respectively.
}

\sol{exc.Omega_functorial}{
\begin{enumerate}
	\item Show that the definition of $\Omega$ given above in \cref{eqn.omega_as_opens,eqn.omega_open_rest} is functorial, i.e., that whenever $W\ss V\ss U$, the restriction map $\Omega(U)\to\Omega(V)$ followed by the restriction map $\Omega(V)\to\Omega(W)$ is the same as the restriction map $\Omega(U)\to\Omega(W)$.
	\item Is that all that's necessary to conclude that $\Omega$ is a presheaf?
\end{enumerate}
}{
By \cref{eqn.omega_as_opens,eqn.omega_open_rest} the definition of $\Omega(U)$ is $\Omega(U)\coloneqq\{U'\in\Op\mid U'\ss U\}$, and the definition of the restriction map for $V\ss U$ is $U'\mapsto U'\cap V$.
\begin{enumerate}
	\item It is functorial: given $W\ss V\ss U$ and $U'\ss U$, we indeed have $(U'\cap V)\cap W=U'\cap W$, since $W\ss V$. For functoriality, we also need preservation of identities, and this amounts to $U'\cap U=U'$ for all $U'\ss U$.
	\item	Yes, a presheaf is just a functor; the above check is enough.
\end{enumerate}
}

\sol{exc.classify_subgraph}{
Consider the subgraph $G'\ss G$ shown here:
\[
\boxCD{
\begin{tikzcd}[ampersand replacement=\&]
	\LMO{A}\ar[r]\&\LMO{B}\&\LMO{C}
\end{tikzcd}
}
\quad\ss\quad
\boxCD{
\begin{tikzcd}[ampersand replacement=\&]
	\LMO{A}\ar[r, shift left, "f"]\&\LMO{B}\ar[l, shift left, "g"]\ar[r, "h"]\&\LMO{C}\ar[r, "i"]\&\LMO{D}
\end{tikzcd}
}
\]
Find the graph homomorphism $\corners{G'}\colon G\to\Omega$ classifying it. See \cref{ex.subobject_classifier_graphs}.
}{
We need a graph homomorphism of the following form:
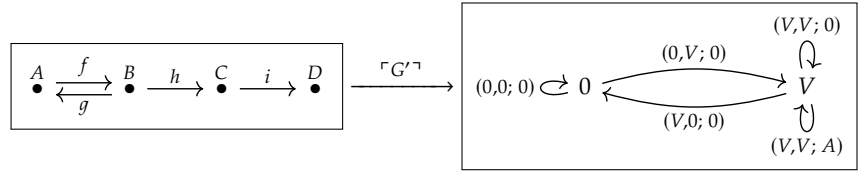
\[
\boxCD{
\begin{tikzcd}[ampersand replacement=\&]
	\LMO{A}\ar[r, shift left, "f"]\&\LMO{B}\ar[l, shift left, "g"]\ar[r, "h"]\&\LMO{C}\ar[r, "i"]\&\LMO{D}
\end{tikzcd}
}
\To{\quad\corners{G'}\quad}
\boxCD{
\begin{tikzcd}[column sep=70pt, ampersand replacement=\&]
	0
		\ar[loop left, "{(0,0;\ 0)}"]
		\ar[r, bend left=15pt, "{(0, V;\ 0)}"]
\&
	V
		\ar[loop above, "{(V, V;\ 0)}"]
		\ar[loop below, "{(V, V;\ A)}"]
		\ar[l, bend left=15pt, "{(V, 0;\ 0)}"]
\end{tikzcd}
}
\]
There is only one that classifies $G'$, and here it is. Let's write $\gamma\coloneqq\corners{G'}$.
\begin{itemize}
	\item Since $D$ is missing from $G'$, we have $\gamma(D)=0$ (vertex: missing).
	\item Since vertices $A,B,C$ are present in $G'$ we have $\gamma(A)=\gamma(B)=\gamma(C)=V$ (vertex: present). 
	\item The above forces $\gamma(i)=(V,0;0)$ (arrow from present vertex to missing vertex: missing).
	\item Since the arrow $f$ is in $G'$, we have $\gamma(f)=(V,V; A)$ (arrow from present vertex to present vertex: present).
	\item Since the arrows $g$ and $h$ are missing in $G'$, we have $\gamma(g)=\gamma(h)=(V,V; 0)$ (arrow from present vertex to present vertex: missing).
\end{itemize}
}

\sol{exc.real_line_logic}{
Consider the real line $\RR$ as a topological space, and consider the open subset $U=\RR-\{0\}$.
\begin{enumerate}
	\item What open subset is $\neg U$?
	\item What open subset is $\neg\neg U$?
	\item Is it true that $U\ss\neg\neg U$?
	\item Is it true that $\neg\neg U\ss U$?
\end{enumerate}	
}{
With $U=\RR-\{0\}\ss\RR$, we have:
\begin{enumerate}
	\item The complement of $U$ is $\RR-U=\{0\}$ and $\neg U$ is its interior, which is $\neg U=\varnothing$.
	\item The complement of $\neg U$ is $\RR-\varnothing=\RR$, and this is open, so $\neg\neg U=\RR$.
	\item It is true that $U\ss\neg\neg U$.
	\item It is false that $\neg\neg U\ss^?U$.
\end{enumerate}
}

\sol{exc.top_bot_practice}{
Let $(X,\Op)$ be a topological space.
\begin{enumerate}
	\item Suppose the symbol $\top$ corresponds to an open set such that for any open set $V\in\Op$, we have $(\top\wedge V)=V$. Which open set is it?
	\item Other things we should expect from $\top$ include $(\top\vee V)=\top$ and $(V\imp \top)=\top)$ and $(\top\imp V)=V$. Do these hold for your answer to 1?
	\item The symbol $\bot$ corresponds to an open set $U\in\Op$ such that for any open set $V\in\Op$, we have $(\bot\vee V)=V$. Which open set is it?
	\item Other things we should expect from $\bot$ include $(\bot\wedge V)=\bot$ and $(\bot\imp V)=\top)$. Do these hold for your answer to 1?
\end{enumerate} 
}{
\begin{enumerate}
	\item If for any $V\in\Op$ we have $\top\wedge V=V$ then when $V=X$ we have $\top\wedge X\coloneqq\top\cap X=X$, but anything intersected with $X$ is itself, so $\top=\top\cap X=X$.
	\item $(\top\vee V)\coloneqq (X\cup V)=X$ holds and $(V\imp X)=\bigcup_{\{R\in\Op\mid R\cap V\ss X\}}R=X$ holds because $(X\cap V)\ss X$.
	\item If for any set $V\in\Op$ we have $(\bot\vee V)=V$, then when $V=\emptyset$ we have $(\bot\vee\varnothing)=(\bot\cup\varnothing)=\varnothing$, but anything unioned with $\varnothing$ is itself, so $\bot=\bot\cup\varnothing=\varnothing$.
	\item $(\bot\wedge V)=(\varnothing\cap V)=\varnothing$ holds, and $(\bot\imp V)=\bigcup_{\{R\in\Op\mid R\cap\varnothing\ss V\}}R=X$ holds because $(X\cap\varnothing)\ss V$.
\end{enumerate}
}

\sol{exc.weather_bob}{
Just now we described how a predicate $p\colon S\to\Omega$, such as ``\dots likes the weather,'' acts on sections $s\in S(U)$, say $s=\mathrm{Bob}$. But by \cref{def.subobject_classifier}, any predicate $p\colon S\to\Omega$ to the subobject classifier also defines a subobject of $\{S\mid p\}\ss S$. Describe the sections of this subsheaf.
}{
$S$ is the sheaf of people, the set of which changes over time: a section in $S$ over any interval of time is a person who is alive throughout that interval. A section in the subobject $\{S\mid p\}$ over  any interval of time is a person who is alive \emph{and likes the weather} throughout that interval of time.
}

\sol{exc.vdash}{
Give an example of a space $X$, a sheaf $S\in\Shv(X)$, and two predicates $p,q\colon S\to\Omega$ for which $p(s)\vdash_{s:S} q(s)$ holds. You do not have to be formal.
}
{
We need an example of a space $X$, a sheaf $S\in\Shv(X)$, and two predicates $p,q\colon S\to\Omega$ for which $p(s)\vdash_{s:S} q(s)$ holds. Take $X$ to be the one-point space, take $S$ to be the sheaf corresponding to the set $S=\nn$, let $p(s)$ be the predicate ``$24\leq s\leq 28$,'' and let $q(s)$ be the predicate ``$s$ is not prime.'' Then $p(s)\vdash_{s:S} q(s)$ holds.\\

As an informal example, take $X$ to be the surface of the earth, take $S$ to be the sheaf of vector fields as in \cref{ex.tangent_bundle} thought of in terms of wind-blowing. Let $p$ be the predicate ``the wind is blowing due east at somewhere between 2 and 5 kilometers per hour'' and let $q$ be the predicate ``the wind is blowing at somewhere between 1 and 5 kilometers per hour.'' Then $p(s)\vdash_{s:S} q(s)$ holds. This means that for any open set $U$, if the wind is blowing due east at somewhere between 2 and 5 kilometers per hour throughout $U$, then the wind is blowing at somewhere between 1 and 5 kilometers per hour throughout $U$ as well.
}
\sol{exc.predicate_practice}{
In the topos $\Cat{Set}$, where $\Omega=\BB$, consider the predicate $p\colon\NN\times\ZZ\to\BB$ given by
\[
  p(n,z)=
  \begin{cases}
    \true&\tn{ if }n\leq|z|\\
    \false&\tn{ if }n>|z|.
  \end{cases}
\]
\begin{enumerate}
	\item What is the set of $n\in\NN$ for which the predicate $\forall(z:\ZZ)\ldotp p(n,z)$ holds?
	\item What is the set of $n\in\NN$ for which the predicate $\exists(z:\ZZ)\ldotp p(n,z)$ holds?
	\item What is the set of $z\in\ZZ$ for which the predicate $\forall(n:\NN)\ldotp p(n,z)$ holds?
	\item What is the set of $z\in\ZZ$ for which the predicate $\exists(n:\NN)\ldotp p(n,z)$ holds?	
\end{enumerate}
}{
We have the predicate $p\colon\NN\times\ZZ\to\BB$ given by $p(n,z)$ iff $n\leq|z|$.
\begin{enumerate}
	\item The predicate $\forall(z:\ZZ)\ldotp p(n,z)$ holds for $\{0\}\ss\NN$.
	\item The predicate $\exists(z:\ZZ)\ldotp p(n,z)$ holds for $\NN\ss\NN$.
	\item The predicate $\forall(n:\NN)\ldotp p(n,z)$ holds for $\varnothing\ss\ZZ$.
	\item The predicate $\exists(n:\NN)\ldotp p(n,z)$ holds for $\ZZ\ss\ZZ$.
\end{enumerate}
}

\sol{exc.worrying_news_universal}{
Suppose $s$ is a person alive throughout the interval $U$. Apply the above definition to the example $p(s,t)=$ ``person $s$ is worried about news $t$'' from above. Here, $T(V)$ is the set of items that are in the news throughout the interval $V$.
\begin{enumerate}
  \item What open subset of $U$ is $\forall(t:T)\ldotp p(s,t)$ for a person $s$?
  \item Does it have the semantic meaning you'd expect, given the less formal description in \cref{subsec.quantification}?
\end{enumerate}
}{
Suppose $s$ is a person alive throughout the interval $U$. Apply the above definition to the example $p(s,t)=$ ``person $s$ is worried about news $t$'' from above.
\begin{enumerate}
	\item The formula says that $\forall(t:T)\ldotp p(s,t)$ ``returns the largest open set $V\ss U$ for which $p(\restrict{s}{V},t)=V$ for all $t\in T(V)$.'' Note that $T(V)$ is the set of items that are in the news throughout the interval $V$. Substituting, this becomes ``the largest interval of time $V\ss U$ over which person $s$ is worried about news $t$ for every item $t$ that is in the news throughout $V$.'' In other words, for $V$ to be nonempty, the person $s$ would have to be worried about \emph{every single item of news} throughout $V$. My guess is that there's a festival happening or a happy kitten somewhere that person $s$ is not worried about, but maybe I'm assuming that person $s$ is sufficiently mentally ``normal.'' There may be people who are sometimes worried about literally everything in the news; we ask you to please be kind to them. 
	\item Yes, it is exactly the same description.
\end{enumerate}
}

\sol{exc.worrying_news_existential}{
Apply the above definition to the ``person $s$ is worried about news $t$'' example above. \begin{enumerate}
  \item What open set is $\exists(t:T)\ldotp p(s,t)$ for a person $s$?
  \item Does it have the semantic meaning you'd expect?
\end{enumerate}
}{
Suppose $s$ is a person alive throughout the interval $U$. Apply the above definition to the example $p(s,t)=$ ``person $s$ is worried about news $t$'' from above.
\begin{enumerate}
	\item The formula says that $\exists(t:T)\ldotp p(s,t)$ ``returns the union $V=\bigcup_iV_i$ of all the open sets $V_i$ for which there exists some $t_i\in T(V_i)$ satisfying $p(\restrict{s}{V_i},t_i)=V_i$.'' Substituting, this becomes ``the union of all time intervals $V_i$ for which there is some item $t_i$ in the news about which $s$ is worried throughout $V_i$.'' In other words it is all the time that $s$ is worried about at least one thing in the news. Perhaps when $s$ is sleeping or concentrating on something, she is not worried about anything, in which case intervals of sleeping or concentrating would not be subsets of $V$. But if $s$ said ``there's been such a string of bad news this past year, it's like I'm always worried about something!,'' she is saying that it's like $V=$``this past year.''
	\item This seems like a good thing for ``there exists a piece of news that worries $s$'' to mean: the news itself is allowed to change as long as the person's worry remains. Someone might disagree and think that the predicate should mean ``there is one piece of news that worries $s$ throughout the whole interval $V$.'' In that case, perhaps this person is working within a different topos, e.g.\ one where the site has fewer coverings. Indeed, it is the notion of covering that makes existential quantification work the way it does.
\end{enumerate}
}

\sol{exc.two_defs_of_closure}{
Suppose $j\colon\Omega\to\Omega$ is a morphism of sheaves on $X$, such that $p\leq j(p)$ holds for all $U\ss X$ and $p\in\Omega(U)$. Show that for all $q\in\Omega(U)$ we have $j(j(q))\leq j(q)$ iff $j(j(q))=j(q)$.
}
{
It is clear that if $j(j(q))=j(q)$ then $j(j(q))\leq j(q)$ by reflexivity. On the other hand, assume the hypothesis, that $p\leq j(p)$ for all $U\ss X$ and $p\in\Omega(U)$. If $j(j(q))\leq j(q)$, then letting $p\coloneqq j(q)$ we have both $j(p)\leq p$ and $p\leq j(p)$. This means $p\cong j(p)$, but $\Omega$ is a poset (not just a preorder) so $p=j(p)$, i.e.\ $j(j(q))= j(q)$ as desired.
}

\sol{exc.check_modality}{
Let $S$ be the sheaf of people as in \cref{subsec.predicates}, and let $j\colon\Omega\to\Omega$ be ``assuming Bob is in San Diego...'' 
\begin{enumerate}
	\item Name any predicate $p\colon S\to\Omega$, such as ``likes the weather.''
	\item Choose a time interval $U$. For an arbitrary person $s\in S(U)$, what sort of thing is $p(s)$, and what does it mean?
	\item What sort of thing is $j(p(s))$ and what does it mean?
	\item Is it true that $p(s)\leq j(p(s))$? Explain briefly.
	\item Is it true that $j(j(p(s))=j(p(s)$? Explain briefly.
	\item Choose another predicate $q\colon S\to\Omega$. Is it true that $j(p\wedge q)=j(p)\wedge j(q)$? Explain briefly.
\end{enumerate}
}
{
Let $S$ be the sheaf of people and $j$ be ``assuming Bob is in San Diego...''
\begin{enumerate}
	\item Take $p(s)$ to be ``$s$ likes the weather.''
	\item Let $U$ be the interval 2019/01/01 -- 2019/02/01. For an arbitrary person $s\in S(U)$, $p(s)$ is a subset of $U$, and it means the subset of $U$ throughout which $s$ likes the weather.
	\item Similarly $j(p(s))$ is a subset of $U$, and it means the subset of $U$ throughout which, assuming Bob is in San Diego, $s$ liked the weather. In other words, $j(p(s))$ is true whenever Bob is not in San Diego, and it is true whenever $s$ likes the weather.
	\item It is true that $p(s)\leq j(p(s))$, by the `in other words' above.
	\item It is true that $j(j(p(s))=j(p(s)$, because suppose given a time during which ``if Bob is in San Diego then if Bob is in San Diego then $s$ likes the weather.'' Then if Bob is in San Diego during this time then $s$ likes the weather. But that is exactly what $j(p(s))$ means.
	\item Take $q(s)$ to be ``$s$ is happy.'' Suppose ``if Bob is in San Diego then both $s$ likes the weather and $s$ is happy.'' Then both ``if Bob is in San Diego then $s$ likes the weather'' and ``if Bob is in San Diego then $s$ is happy'' are true too. The converse is equally clear.
\end{enumerate}
}

\sol{exc.explain_intervals}{
\begin{enumerate}
	\item Explain why $[2,6]\in o_{[0,8]}$.
	\item Explain why $[2,6]\not\in o_{[0,5]}\cup o_{[4,8]}$.
\end{enumerate}
}{
We have $o_{[a,b]}\coloneqq\{[d,u]\in\IR\mid a<d\leq u<b\}$.
\begin{enumerate}
	\item Since $0\leq 2\leq 6\leq 8$, we have $[2,6]\in o_{[0,8]}$ by the above formula.
	\item In order to have $[2,6]\in^? o_{[0,5]}\cup o_{[4,8]}$, we would need to have either $[2,6]\in^? o_{[0,5]}$ or $[2,6]\in^? o_{[4,8]}$. But the formula does not hold in either case.
\end{enumerate}
}

\sol{exc.R_subsp_IR}{
Show that a subset $U\ss\RR$ is open in the subspace topology of $\RR\ss\IR$ iff $U\cap\RR$ is open in the usual topology on $\RR$ defined in \cref{ex.usual_R}.
}{
A subset $U\ss\RR$ is open in the subspace topology of $\RR\ss\IR$ iff there is an open set $U'\ss\IR$ with $U=U'\cap\RR$. We want to show that this is the case iff $U$ is open in the usual topology.

Suppose that $U$ is open in the subspace topology. Then $U=U'\cap\RR$, where $U'\ss\IR$ is the union of some basic opens, $U'=\bigcup_{i\in I}o_{[a_i,b_i]}$, where $o_{[a_i,b_i]}=\{[d,u]\in\IR\mid a_i<d<u<b_i\}$. Since $\RR=\{[x,x]\in\IR\}$, the intersection $U'\cap\RR$ will then be
\[U=\bigcup_{i\in I}\{x\in\RR\mid a_i<x<b_i\}\]
and this is just the union of open balls $B(m_i,r_i)$ where $m_i\coloneqq\frac{a_i+b_i}{2}$ is the midpoint and $r_i\coloneqq\frac{b_i-a_i}{2}$ is the radius of the interval $(a_i,b_i)$. The open balls $B(m_i,r_i)$ are open in the usual topology on $\RR$ and the union of opens is open, so $U$ is open in the usual topology.

Suppose that $U$ is open in the usual topology. Then $U=\bigcup_{j\in J}B(m_j,\epsilon_j)$ for some set $J$. Let $a_j\coloneqq m_j-\epsilon_j$ and $b_j\coloneqq m_j+\epsilon_j$. Then \[U=\bigcup_{j\in J}\{x\in\RR\mid a_j<x<b_j\}=\bigcup_{j\in J}(o_{[a_j,b_j]}\cap\RR)=\left(\bigcup_{j\in J}o_{[a_j,b_j]}\right)\cap\RR\]
which is open in the subspace topology.
}

\sol{exc.interval_domain_top}{
Fix any topological space $(X,\Op_X)$ and any subset $R\ss\IR$ of the interval domain. Define $H_X(U)\coloneqq\{f\colon U\cap R\to X\mid f\text{ is continuous}\}$.
\begin{enumerate}
	\item Is $H_X$ a presheaf? If not, why not; if so, what are the restriction maps?
	\item Is $H_X$ a sheaf? Why or why not?
\end{enumerate}
}{
Fix any topological space $(X,\Op_X)$ and any subset $R\ss\IR$ of the interval domain. Define $H_X(U)\coloneqq\{f\colon U\cap R\to X\mid f\text{ is continuous}\}$.
\begin{enumerate}
	\item $H_X$ is a presheaf: given $V\ss U$ the restriction map sends the continuous function $f\colon U\cap R\to X$ to its restriction along the subset $V\cap R\ss U\cap R$.
	\item It is a sheaf: given any family $U_i$ of open sets with $U=\bigcup_iU_i$ and a continuous function $f_i\colon U_i\cap R\to X$ for each $i$, agreeing on overlaps, they can be glued together to give a continuous function on all of $U\cap R$, since $U\cap R=(\bigcup_iU_i)\cap R=\bigcup_i(U_i\cap R)$.
\end{enumerate}
}

\finishSolutionChapter

\endgroup

\backmatter

\printbibliography
\printindex

\end{document}